\documentclass[a4paper,12pt]{article}
\usepackage{amsfonts,amsmath,amssymb}
\newtheorem{theorem}{Theorem}[section]
\newtheorem{lemma}[theorem]{Lemma}
\newtheorem{proposition}[theorem]{Proposition}
\newtheorem{corollary}[theorem]{Corollary}
\newtheorem{remark}[theorem]{Remark}
\newtheorem{definition}[theorem]{Definition}
\numberwithin{equation}{section}
\newenvironment{proof}{{\bf Proof\ }}{\QED\\}
\newcommand{\QED}{\hspace*{\fill}\rule{2.5mm}{2.5mm}}

\newcommand{\be}{\begin{equation}}
\newcommand{\ee}{\end{equation}}
\newcommand{\ba}{\begin{array}}
\newcommand{\ea}{\end{array}}
\newcommand{\bea}{\begin{eqnarray}}
\newcommand{\eea}{\end{eqnarray}}
\newcommand{\bee}{\begin{eqnarray*}}
\newcommand{\eee}{\end{eqnarray*}}

\newcommand{\lab}{\label}
\newcommand{\und}{\underline}

\newcommand{\ds}{\displaystyle}
\newcommand{\nn}{\nonumber}
\providecommand{\norm}[1]{\lVert#1\rVert}
\providecommand{\normm}[1]{\left\lVert#1\right\rVert}
\renewcommand{\c}{\cdot}
\newcommand{\les}{\lesssim}

\newcommand{\R}{\mathbb{R}}

\newcommand{\rf}{\mathcal{R}}
\newcommand{\II}{\mathcal{I}_0}
\newcommand{\Ij}{\mathcal{I}_{0,j}}
\newcommand{\no}{\mathcal{N}_1}
\newcommand{\noo}{\mathcal{N}_2}
\newcommand{\dcal}{\mathcal{D}_1}
\newcommand{\dcall}{\mathcal{D}_2}
\newcommand{\dcalll}{{}^*\mathcal{D}_1}
\newcommand{\dcallll}{{}^*\mathcal{D}_2}
\newcommand{\BB}{\mathcal{B}}
\newcommand{\PP}{\mathcal{P}}

\newcommand{\rr}{{\bf R}}
\renewcommand{\gg}{{\bf g}}
\newcommand{\dd}{{\bf D}}

\newcommand{\ep}{\varepsilon}
\newcommand{\kep}{\epsilon}
\newcommand{\La}{\Lambda}
\newcommand{\Si}{\Sigma_0}
\newcommand{\Sit}{\Sigma_t}
\newcommand{\pr}{\partial}
\newcommand{\nab}{\nabla}
\newcommand{\lb}{\underline{L}}
\newcommand{\dmt}{d\mu_{t,u}}

\newcommand{\ptu}{P_{t,u}}
\newcommand{\pou}{P_{0,u}}

\newcommand{\half}{\frac{1}{2}}

\renewcommand{\o}{\omega}
\renewcommand{\S}{\mathbb{S}^2}
\newcommand{\po}{\partial_{\omega}}

\renewcommand{\H}{\mathcal{H}}
\newcommand{\li}[2]{L^{#1}_uL^{#2}(\mathcal{H}_u)}
\newcommand{\lh}[1]{L^{#1}(\mathcal{H}_u)}
\newcommand{\tx}[2]{L^{#1}_tL^{#2}_{x'}}
\newcommand{\xt}[2]{L^{#1}_{x'}L^{#2}_t}
\newcommand{\lpt}[1]{L^{#1}(P_{t,u})}
\newcommand{\lpo}[1]{L^{#1}(P_{0,u})}
\newcommand{\lsit}[2]{L^{#1}_tL^{#2}(\Sit)}

\newcommand{\nabb}{\mbox{$\nabla \mkern-13mu /$\,}}
\newcommand{\ddb}{\mbox{$\nabla \mkern-13mu /$\,}}
\newcommand{\lap}{\mbox{$\Delta \mkern-13mu /$\,}}
\newcommand{\divb}{\mbox{$\textrm{div} \mkern-13mu /$\,}}
\newcommand{\curlb}{\mbox{$\textrm{curl} \mkern-13mu /$\,}}

\renewcommand{\a}{\alpha}
\renewcommand{\b}{\beta}
\newcommand{\ab}{\underline{\alpha}}
\newcommand{\bb}{\underline{\beta}}
\newcommand{\ga}{\gamma}
\renewcommand{\d}{\delta}
\newcommand{\s}{\sigma}
\renewcommand{\r}{\rho}
\newcommand{\kb}{\overline{k}}
\newcommand{\etah}{\widehat{\eta}}
\newcommand{\kepb}{\overline{\epsilon}}
\newcommand{\db}{\overline{\delta}}

\newcommand{\z}{\zeta}
\newcommand{\zb}{\underline{\zeta}}
\newcommand{\trc}{\textrm{tr}\chi}
\newcommand{\hch}{\widehat{\chi}}
\newcommand{\chb}{\underline{\chi}}
\newcommand{\trchb}{\textrm{tr}{\underline{\chi}}}
\newcommand{\hchb}{\widehat{\underline{\chi}}}
\newcommand{\xib}{\underline{\xi}}
\newcommand{\het}{\widehat{\eta}}
\renewcommand{\th}{\theta}
\newcommand{\trt}{\textrm{tr}\theta}
\newcommand{\hth}{\widehat{\theta}}

\newcommand{\bsit}{\widehat{\mathcal{B}}}

\begin{document}

\begin{center}
\Large{\bf Parametrix for wave equations on a rough background III: space-time regularity of the phase}
\end{center}

\vspace{0.5cm}

\begin{center}
\large{J\'er\'emie Szeftel}
\end{center}

\vspace{0.2cm}

\begin{center}
\large{DMA, Ecole Normale Sup\'erieure,\\
45 rue d'Ulm, 75005 Paris,\\
jeremie.szeftel@ens.fr}
\end{center}

\vspace{0.5cm}

{\bf Abstract.} This is the third of a sequence of four papers  \cite{param1}, \cite{param2}, \cite{param3}, \cite{param4} dedicated to the construction and the control of a parametrix to the homogeneous wave equation $\square_{\bf g} \phi=0$, where ${\bf g}$ is a rough metric satisfying the Einstein vacuum equations. Controlling such a parametrix as well as its error term when one only assumes $L^2$ bounds on the curvature tensor ${\bf R}$ of ${\bf g}$ is a major step of the proof of the bounded $L^2$ curvature conjecture proposed in \cite{Kl:2000}, and solved by S. Klainerman, I. Rodnianski and the author in \cite{boundedl2}. On a more general level, this sequence of papers deals with the control of the eikonal equation on a rough background, and with the  derivation of $L^2$ bounds for Fourier integral operators on manifolds with rough phases and symbols, and as such is also of independent interest.

\vspace{0.2cm}

\section{Introduction}

We consider the Einstein vacuum equations,
\be\lab{eq:I1}
{\bf R}_{\alpha\beta}=0
\end{equation}
where ${\bf R}_{\alpha\beta}$
denotes the  Ricci curvature tensor  of  a four dimensional Lorentzian space time  $(\mathcal{M},\,  {\bf g})$. The Cauchy problem consists in finding a metric ${\bf g}$ satisfying \eqref{eq:I1} such that the metric induced by ${\bf g}$ on a given space-like hypersurface $\Sigma_0$ and the second fundamental form of $\Sigma_0$ are prescribed. The initial data then consists of a Riemannian three dimensional metric $g_{ij}$ and a symmetric tensor $k_{ij}$ on the space-like hypersurface $\Sigma_0=\{t=0\}$. Now, \eqref{eq:I1} is an overdetermined system and the initial data set $(\Sigma_0,g,k)$ must satisfy the constraint equations
\be\lab{const}
\left\{\begin{array}{l}
\nabla^j k_{ij}-\nabla_i \textrm{Tr}k=0,\\
 R-|k|^2+(\textrm{Tr}k)^2=0, 
\end{array}\right.
\ee
where the covariant derivative $\nabla$ is defined with respect to the metric $g$, $R$ is the scalar curvature of $g$, and $\textrm{Tr}k$ is the trace of $k$ with respect to the metric $g$.

The fundamental problem in general relativity is to 
 study the long term regularity and asymptotic 
properties
 of the Cauchy developments of general, asymptotically flat,  
initial data sets $(\Sigma_0, g, k)$. As far as local regularity is concerned it 
is natural to ask what are the minimal regularity properties of
the initial data which guarantee the existence and 
uniqueness of local developments. In \cite{boundedl2}, we obtain the 
following result which solves bounded $L^2$ curvature conjecture proposed 
in \cite{Kl:2000}:

\begin{theorem}[Theorem 1.10 in \cite{boundedl2}]\lab{th:mainbl2}
Let $(\mathcal{M}, {\bf g})$ an asymptotically flat solution to the Einstein vacuum equations \eqref{eq:I1} together with a maximal foliation by space-like hypersurfaces $\Sigma_t$ defined as level hypersurfaces of a time function $t$. Let $r_{vol}(\Sigma_t,1)$ the volume radius on scales $\leq 1$ of $\Sigma_t$\footnote{See Remark \ref{rem:volrad} below for a  definition}. Assume that the initial slice $(\Sigma_0,g,k)$ is such that:
$$\norm{R}_{L^2(\Sigma_0)}\leq \ep,\,\norm{k}_{L^2(\Sigma_0)}+\norm{\nabla k}_{L^2(\Sigma_0)}\leq \ep\textrm{ and }r_{vol}(\Sigma_0,1)\geq \frac{1}{2}.$$
Then, there exists a small universal constant $\ep_0>0$ such that if $0<\ep<\ep_0$, then the following control holds on $0\leq t\leq 1$:
$$\norm{\R}_{L^\infty_{[0,1]}L^2(\Sigma_t)}\lesssim \ep,\,\norm{k}_{L^\infty_{[0,1]}L^2(\Sigma_t)}+\norm{\nabla k}_{L^\infty_{[0,1]}L^2(\Sigma_t)}\lesssim \ep\textrm{ and }\inf_{0\leq t\leq 1}r_{vol}(\Sigma_t,1)\geq \frac{1}{4}.$$
\end{theorem}

\begin{remark}
While  the first   nontrivial improvements  for well posedness for quasilinear  hyperbolic systems (in spacetime dimensions greater than $1+1$), based on Strichartz estimates,  were obtained in   \cite{Ba-Ch1}, \cite{Ba-Ch2}, \cite{Ta1}, \cite{Ta}, \cite{KR:Duke}, \cite{KR:Annals}, \cite{SmTa}, Theorem \ref{th:mainbl2}, is the first  result in which the  full nonlinear structure of the quasilinear system, not just its principal part,  plays  a  crucial  role.  We note that  though  the result is not  optimal with respect to the  standard  scaling  of the Einstein equations, it is  nevertheless critical   with respect to its causal geometry,  i.e. $L^2 $ bounds on the curvature is the minimum requirement necessary to obtain lower bounds on the radius of injectivity of null hypersurfaces. We refer the reader to section 1 in \cite{boundedl2} for more motivations and historical perspectives concerning Theorem \ref{th:mainbl2}. 
\end{remark}

\begin{remark}
The regularity assumptions on $\Sigma_0$ in Theorem \ref{th:mainbl2} - i.e. $R$ and $\nabla k$ bounded in $L^2(\Sigma_0)$ - correspond to an initial data set $(g,\, k )\in H^2_{loc}(\Sigma_0)\times H^1_{loc}(\Sigma_0)$.
\end{remark}

\begin{remark}\lab{rem:reducsmallisok}
In \cite{boundedl2}, our main result is stated for corresponding large data. We then reduce the proof to the  small data statement of Theorem \ref{th:mainbl2} relying on a truncation and rescaling procedure, the control of the harmonic radius of $\Sigma_0$ based on Cheeger-Gromov convergence of Riemannian manifolds together with the assumption on the lower bound of the volume radius of $\Sigma_0$, and the gluing procedure in \cite{Co}, \cite{CoSc}. We refer the reader to section 2.3 in \cite{boundedl2} for the details.
\end{remark}

\begin{remark}\lab{rem:volrad}
We recall for the convenience of the reader the definition of the volume radius of the Riemannian manifold $\Sigma_t$. Let $B_r(p)$ denote the geodesic ball of center $p$ and radius $r$. The volume radius $r_{vol}(p,r)$ at a point $p\in \Sigma_t$ and scales $\leq r$ is defined by
$$r_{vol}(p,r)=\inf_{r'\leq r}\frac{|B_{r'}(p)|}{r^3},$$
with $|B_r|$ the volume of $B_r$ relative to the metric $g_t$ on $\Sigma_t$. The volume radius $r_{vol}(\Sigma_t,r)$ of $\Sigma_t$ on scales $\leq r$ is the infimum of $r_{vol}(p,r)$ over all points $p\in \Sigma_t$.
\end{remark}

The proof of Theorem \ref{th:mainbl2}, obtained in the sequence of papers \cite{boundedl2}, \cite{param1}, \cite{param2}, \cite{param3}, \cite{param4}, \cite{bil2}, relies on the following ingredients\footnote{We also need trilinear estimates and an $L^4(\mathcal{M})$ Strichartz estimate (see the introduction in \cite{boundedl2})}: 
{\em\begin{enumerate}
\item[{\bf A}] Provide  a system of coordinates relative to which \eqref{eq:I1} exhibits a null structure.

\item[{\bf B}] Prove  appropriate bilinear estimates for solutions to $\square_{\bf g} \phi=0$, on
 a fixed Einstein vacuum  background\footnote{Note that the first bilinear estimate of this type was obtained in \cite{BIL}}.

\item[{\bf C}] Construct a parametrix for solutions to the homogeneous wave equations $\square_{\bf g} \phi=0$ on a fixed Einstein vacuum  background, and obtain control of the parametrix and of its error term only using the fact that the curvature tensor is bounded in $L^2$. 
\end{enumerate}
}

Steps {\bf A} and {\bf B} are carried out in \cite{boundedl2}. In particular, the proof of the bilinear estimates rests on a representation formula for the solutions of the wave equation using the following plane wave parametrix\footnote{\eqref{param} actually corresponds to a half-wave parametrix. The full parametrix corresponds to the sum of two half-parametrix. See \cite{param2} for the construction of the full parametrix}:
\be\lab{param}
Sf(t,x)=\int_{\S}\int_{0}^{+\infty}e^{i\lambda u(t,x,\o)}f(\lambda\o)\lambda^2 d\lambda d\o,\,(t,x)\in\mathcal{M} 
\ee
where $u(.,.,\o)$ is a solution to the eikonal equation ${\bf g}^{\alpha\beta}\partial_\alpha u\partial_\beta u=0$ on $\mathcal{M}$ such that $u(0,x,\o)\sim x.\o$ when $|x|\rightarrow +\infty$ on $\Sigma_0$\footnote{The asymptotic behavior for $u(0,x,\o)$ when $|x|\rightarrow +\infty$ will be used in \cite{param2} to generate with the parametrix any initial data set for the wave equation}. Therefore, in order to complete the proof of the bounded $L^2$ curvature conjecture, we need to carry out step {\bf C} with the parametrix defined in \eqref{param}. 

\begin{remark}
Note that the parametrix \eqref{param} is invariantly defined\footnote{Our choice is reminiscent of the one used in \cite{SmTa} in the context of $H^{2+\epsilon}$ solutions of quasilinear wave equations. Note however that the construction in that paper is coordinate dependent}, i.e. without reference to any coordinate system. This is crucial since coordinate systems consistent with $L^2$ bounds on the curvature would not be regular enough to control a parametrix. 
\end{remark}

\begin{remark}
In addition to their relevance to the resolution of the bounded $L^2$ curvature conjecture, the methods and results  of step {\bf C} are  also of independent interest. Indeed, they deal on the one hand with the control of the eikonal equation ${\bf g}^{\alpha\beta}\partial_\alpha u\partial_\beta u=0$ at a critical level\footnote{As we will see in this paper, we need at least $L^2$ bounds on the curvature to obtain a lower bound on the radius of injectivity of the null level hypersurfaces of the solution $u$ of the eikonal equation, which in turn is necessary to control the local regularity of $u$}, and on the other hand with the derivation of $L^2$ bounds for Fourier integral operators with significantly lower differentiability  assumptions both for the corresponding phase and symbol compared to classical methods (see for example \cite{stein} and references therein). 
\end{remark}

In view of the energy estimates for the wave equation, it suffices to control the parametrix at $t=0$ (i.e. restricted to $\Sigma_0$)
\be\lab{parami}
Sf(0,x)=\int_{\S}\int_{0}^{+\infty}e^{i\lambda u(0,x,\o)}f(\lambda\o)\lambda^2 d\lambda d\o,\,x\in\Sigma_0 
\ee
 and the error term
\be\lab{err} 
Ef(t,x)=\square_{\bf g}Sf(t,x)=\int_{\S}\int_{0}^{+\infty}e^{i\lambda u(t,x,\o)}\square_{\bf g}u(t,x,\o)f(\lambda\o)\lambda^3 d\lambda d\o,\,(t,x)\in\mathcal{M}. 
\ee
This requires the following ingredients, the two first being related to the control of the parametrix restricted to $\Sigma_0$ \eqref{parami}, and the two others being related to the control of the error term \eqref{err}:
{\em\begin{enumerate}
\item[{\bf C1}] Make an appropriate choice for the equation satisfied by $u(0,x,\o)$ on $\Sigma_0$, and control the geometry of the foliation generated by the level surfaces of $u(0,x,\o)$ on $\Sigma_0$.

\item[{\bf C2}] Prove that the parametrix at $t=0$ given by \eqref{parami} is bounded in $\mathcal{L}(L^2(\mathbb{R}^3),L^2(\Sigma_0))$ using the estimates for $u(0,x,\o)$ obtained in {\bf C1}.

\item[{\bf C3}] Control the geometry of the foliation generated by the level hypersurfaces of $u$ on $\mathcal{M}$.

\item[{\bf C4}] Prove that the error term \eqref{err} satisfies the estimate $\norm{Ef}_{L^2(\mathcal{M})}\leq C\norm{\lambda f}_{L^2(\mathbb{R}^3)}$ using the estimates for $u$ and $\square_{\gg}u$ proved in {\bf C3}.
\end{enumerate}
}

Step {\bf C1} has been carried out in \cite{param1} and step {\bf C2} has been carried out in \cite{param2}. In the present paper, we focus on step {\bf C3}. This step was initiated in the sequence of papers \cite{FLUX}, \cite{LP},  \cite{STT} where the authors prove in particular the estimate $\square_{\gg}u\in L^{\infty}(\mathcal{M})$ using a geodesic foliation. In view of achieving step {\bf C4}, we actually need to work in a time foliation. 
We start by reproving the estimates obtained in \cite{FLUX}, \cite{LP}, \cite{STT} in the case of a time foliation. We also obtain new estimates which will be crucial for the proof of step {\bf C4}. Let us mention in particular
\begin{itemize} 
\item a lower bound for the radius of injectivity of the null level hypersurfaces of $u$,

\item the control of the second fundamental form $k$,

\item the control of the null lapse associated to $u$,

\item a second order derivative of $\square_{\gg}u$ requires an estimate,

\item the control of the regularity of the $u$-foliation on $\mathcal{M}$ with respect to the parameter $\o\in\S$, which requires estimates for first and second order derivatives with respect to $\o$ of various geometric quantities related to $u$.
\end{itemize} 
The difficulty will be to obtain the aforementioned estimates when assuming only $L^2$ bounds on the curvature tensor ${\bf R}$. Indeed, this level of regularity for ${\bf R}$ is critical for the control of the eikonal equation. In turn, at numerous  places in this paper, we will encounter log-divergences which have to be tackled by ad-hoc techniques taking full advantage of the structure of the Einstein equations. More precisely, we will use the regularity obtained in Step C1, together with null transport equations tied to the eikonal equation, elliptic systems of Hodge type, the geometric Littlewood-Paley theory of \cite{LP}, sharp trace theorems, and an extensive use of the structure of the Einstein equations, to propagate the regularity on $\Sigma_0$ to the space-time, thus achieving Step {\bf C3}.\\

The rest of the paper is as follows. In section \ref{sec:mainresultpres}, we state our main result. In section \ref{sec:calcineqgeneral}, we derive embeddings with respect to the foliation generated by $t$ and $u$ on $\mathcal{M}$ which are consistent with the level of regularity we are considering. In section \ref{sec:regxproof}, we investigate the regularity with respect to $(t, x)$ of the foliation generated by $u$ on $\mathcal{M}$. In section \ref{sec:secondderlb}, we derive estimates for certain second order derivatives of the $u$-foliation on $\mathcal{M}$. In section \ref{sec:firstderivomega}, we derive estimates for first order derivatives with respect to $\o$ of the $u$-foliation on $\mathcal{M}$. In section \ref{sec:secondderivomega}, we derive estimates for second order derivatives with respect to $\o$ of the $u$-foliation on $\mathcal{M}$. In section \ref{sec:depnormonomega}, we investigate the dependence in $\o$ of certain norms associated to the $u$-foliation on $\mathcal{M}$. Finally, additional estimates are derived in section \ref{sec:commutatorest}.\\

\noindent{\bf Acknowledgments.} The author wishes to express his deepest gratitude to Sergiu Klainerman and Igor Rodnianski for stimulating discussions and constant encouragements during the long years where this work has matured. He also would like to stress that  the basic strategy of the construction of the parametrix and how it fits  into the whole proof of the bounded $L^2$ curvature conjecture has been done in collaboration with them. The author is supported by ANR jeunes chercheurs SWAP.\\

\section{Main results}\lab{sec:mainresultpres}

\subsection{Maximal foliation on $\mathcal{M}$}

We foliate the space-time $\mathcal{M}$ by space-like hypersurfaces $\Sit$ defined as level hypersurfaces of a time function $t$. Denoting by 
$T$ the unit, future oriented, normal to $\Sigma_t$ and $k$
the second fundamental form 
\be\lab{def:k}
k_{ij}=-<\dd_iT, \pr_j> 
\ee
we find,
$$k_{ij}= -\frac{1}{2}\mathcal{L}_T \gg_{\,ij}$$
with $\mathcal{L}_X$ denoting the Lie derivative with respect to the 
vectorfield $X$. Let Tr$(k)=g^{ij}k_{ij}$ where $g$ is the induced metric on $\Sit$ and Tr is the trace. In order to be consistent with the statement of Theorem \ref{th:mainbl2}, we impose a maximal foliation 
\be\lab{maxfoliation}
\textrm{Tr}(k)=0.
\ee
We also define the lapse $n$ as 
\be\lab{lapsen}
n^{-1}=T(t).
\ee
We have:
\be\lab{3.2}
\dd_TT=n^{-1}\nab n,
\end{equation}
where $\nab$ denotes the gradient with respect to the induced metric 
on $\Sigma_t$. To check \eqref{3.2} observe that $\pr_t=nT$ and therefore,
for an arbitrary  vectorfield $X$ tangent to  $\Sigma_t$, we easily
calculate,
$<\dd_T T, X>=n^{-2} X^i<\dd_{\pr_t} \pr_t, \pr_i>=
-n^{-2}X^i< \pr_t, \dd_{\pr_t}\pr_i>
=-n^{-2}X^i<\pr_t,\dd_{\pr_i}\pr_t> 
=-n^{-2} X^i\half \pr_i<\pr_t,\pr_t>= n^{-2} X^i\half
\pr_i(n^2)=n^{-1}X(n)$.

Finally, the lapse $n$ satisfies the following elliptic equation on $\Sit$ (see \cite{ChKl} p. 13):
\be\lab{lapsen1}
\Delta n=|k|^2n,
\ee
where one uses \eqref{def:k}, \eqref{3.2}, Einstein vacuum equations \eqref{eq:I1} and the fact that the foliation generated by $t$ on $\mathcal{M}$ is maximal \eqref{maxfoliation}.

\subsection{Geometry of the foliation generated by $u$ on $\mathcal{M}$}

Remember that $u$ is a solution to the eikonal equation $\gg^{\alpha\beta}\partial_\alpha u\partial_\beta u=0$ on $\mathcal{M}$ depending on a extra parameter $\o\in \S$.  The level hypersufaces $u(t,x,\o)=u$  of the optical function $u$ are denoted by  $\H_u$. Let $L'$ denote the space-time gradient of $u$, i.e.:
\be\lab{def:L'0}
L'=-\gg^{\a\b}\pr_\b u \pr_\a.
\ee
Using the fact that $u$ satisfies the eikonal equation, we obtain:
\be\lab{def:L'1}
\dd_{L'}L'=0,
\ee
which implies that $L'$ is the geodesic null generator of $\H_u$.

We have: 
$$T(u)=\pm |\nab u|$$
where $|\nab u|^2=\sum_{i=1}^3|e_i(u)|^2$ relative to an orthonormal frame $e_i$ on $\Sigma_t$. Since the sign of $T(u)$ is irrelevant, we choose by convention:
\be\lab{it1'}
T(u)=|\nab u|.
\end{equation}
We denote by $P_{t,u}$  the surfaces of intersection
between $\Sigma_t$ and  $\H_u$. They play a fundamental role
in our discussion.
\begin{definition}[\textit{Canonical null pair}]
 \be\lab{it2}
L=bL'=T+N, \qquad \lb=2T-L=T-N
\end{equation}
where $L'$ is the space-time gradient of $u$ \eqref{def:L'0}, $b$  is  the  \textit{lapse of the null foliation} (or shortly null lapse)
\be\lab{it3}
b^{-1}=-<L', T>=T(u),
\end{equation} 
and $N$ is a unit normal, along $\Sigma_t$, to the surfaces $P_{t,u}$. Since $u$ satisfies the 
eikonal equation $\gg^{\alpha\beta}\partial_\alpha u\partial_\beta u=0$ on $\mathcal{M}$, 
this yields $L'(u)=0$ and thus $L(u)=0$. In view of the definition of $L$ and \eqref{it1'}, 
we obtain:
\be\lab{it3bis}
N=-\frac{\nabla u}{|\nabla u|}.
\end{equation} 
\label{def:nulllapse}
\end{definition}

\begin{remark}\lab{defnullgeod}
$u$ is prescribed on $\Si$ as in step {\bf C1}. For any $(0,x)$ on $\Si$, $L$ is defined as $L=T+N$ where $T$ is the unit normal to $\Si$ at $(0,x)$ and $N=-\nabla u/|\nabla u|$ at $(0,x)$, and $b$ is defined as $b^{-1}=|\nabla u|$. Let $\kappa_x(t)$ denote the null geodesic parametrized by $t$ and such that $\kappa_x(0)=(0,x)$ and $\kappa_x'(0)=b^{-1}L$. Then, we claim that 
\be\lab{defnullgeod1}
\kappa_x'(t)=b(\kappa_x(t))^{-1}L_{\kappa_x(t)}\textrm{ for all }t. 
\ee
Indeed, $L'=b^{-1}L$ is the geodesic null generator of $\H_u$ (see \eqref{def:L'1}).
\end{remark}

\begin{definition} A  null frame   $e_1,e_2,e_3,e_4$ at a point $p\in P_{t,u}$ 
  consists,
in addition to the null pair     $e_3=\lb,
e_4=L$, of {\sl  arbitrary  orthonormal}  vectors  $e_1,e_2$ tangent
to $P_{t,u}$. 
\end{definition}
\begin{definition}[\textit{Ricci coefficients}]

 Let  $e_1,e_2,e_3,e_4$ be a null frame on
$P_{t,u}$ as above.   The following tensors on  $S_{t,u}$ 
\begin{alignat}{2}
&\chi_{AB}=<\dd_A e_4,e_B>, &\quad 
&\chb_{AB}=<\dd_A e_3,e_B>,\label{chi}\\
&\z_{A}=\half <\dd_{3} e_4,e_A>,&\quad
&\zb_{A}=\half <\dd_{4} e_3,e_A>,\nn\\
&\xib_{A}=\half <\dd_{3} e_3,e_A>.\nn
\end{alignat}
are called the Ricci coefficients associated to our canonical null pair.

We decompose $\chi$ and $\chb$ into
their  trace and traceless components.
\begin{alignat}{2}
&\trc = \gg^{AB}\chi_{AB},&\quad &\trchb = \gg^{AB}\chb_{AB},
\label{trchi}\\
&\hch_{AB}=\chi_{AB}-\half \trc \gg_{AB},&\quad 
&\hchb_{AB}=\chb_{AB}-\half \trchb \gg_{AB},
\label{chih} 
\end{alignat}
\end{definition}

\begin{definition}\label{def:nullcurv}
The null components of the curvature tensor
$\rr$ of the space-time metric $\gg$ are given
by:
\bea
\a_{ab}&=&\rr(L,e_a,L, e_b)\,,\qquad \b_a=\half \rr(e_a, L,\lb, L) ,\\ \r &=&\frac{1}{4}\rr(\lb, L, \lb,
L)\,,
\qquad
\s=\frac{1}{4}\, ^{\star} \rr(\lb, L, \lb, L)\\
\bb_a&=&\half \rr(e_a,\lb,\lb, L)\,,\qquad \underline{\a}_{ab}=\rr(\lb, e_a,\lb, e_b)
\eea
where $^\star \rr$ denotes the Hodge dual of $\rr$. The null decomposition 
of $^\star \rr$ can be related to that of $\rr$ according to the formulas, see \cite{ChKl} :
\bea
\a(^\star \rr)&=&-^\star\a(\rr), \quad \b(^\star \rr)=-^\star \b(\rr), \quad \r(^\star \rr)=\s(\rr)\\
\s(^\star \rr)&=&-\rho(\rr), \quad \bb(^\star \rr)=-^\star \bb(\rr), \quad \underline{\a}(^\star
\rr)=^\star\underline{\a}(\rr)
\eea
\end{definition}
Observe that all tensors defined above are $\ptu$-tangent.

\begin{definition}\lab{def:decompositionk}
We decompose the symmetric traceless 2 tensor $k$ into the scalar $\d$, the $\ptu$-tangent 1-form $\epsilon$, and the $\ptu$-tangent symmetric 2-tensor $\eta$ as follows:
\be\lab{decompositionk}
\left\{\begin{array}{l}
k_{NN}=\d\\
k_{AN}=\kep_A\\
k_{AB}=\eta_{AB}.
\end{array}\right.
\ee
Note that Tr$(k)=$tr$(\eta)+\d$ which together with the maximal foliation assumption \eqref{maxfoliation} 
yields:
\be\lab{eq:traceeta}
\textrm{tr}(\eta)=-\d.
\ee
\end{definition}

The following  \textit{Ricci equations} can be easily derived from the properties 
of $T$ \eqref{def:k} \eqref{3.2}, the fact that $L'$ is geodesic \eqref{def:L'1}, 
and the definition \eqref{chi} of the Ricci coefficients (see \cite{ChKl} p. 171): 
\begin{alignat}{2}
&\dd_A e_4=\chi_{AB} e_B - \kep_{A} e_4, &\quad 
&\dd_A e_3=\chb_{AB} e_B + \kep_{A} e_3,\nn\\
&\dd_{4} e_4 = -\db   e_4, &\quad 
&\dd_{4} e_3= 2\zb_{A} e_A + \db   e_3, 
  \label{ricciform} \\
 &\dd_{3} e_4 = 2\z_{A}e_A +
(\d +n^{-1}\nab_Nn)  e_4,&\quad &\dd_{3} e_3 = 2\xib_{A}e_A -
(\d +n^{-1}\nab_Nn) e_3,\nn\\ &\dd_{4} e_A = \ddb_{4} e_A +
\zb_{A} e_4,&\quad &\dd_{3} e_A = \ddb_{3} e_A+\z_A e_3
+ \xib_A e_4,
\nn\\
&\dd_{B}e_A = \nabb_B e_A +\half \chi_{AB}\, e_3 +\half
 \chb_{AB}\, e_4\nn
\end{alignat}
where, $\ddb_{ 3}$, $\ddb_{ 4}$ denote the 
projection on $P_{t,u}$ of $\dd_3$ and $\dd_4$, $\nabb$
denotes the induced covariant derivative on $P_{t,u}$
and $\db, \kepb$ are defined by: 
\be\lab{newk}
\db=\d-n^{-1}N(n),\,\kepb_A=\kep_A-n^{-1}\nab_A n.
\end{equation}
Also,
\begin{align}
&\chb_{AB}=-\chi_{AB}-2k_{AB},\nn\\
&\zb_A = -\kepb_A,\label{etab}\\
&\xib_A = \kep_{A}+n^{-1} \nabb_A n-\z_A\nn.
\end{align}

\begin{remark}
We also have the identity:
\begin{equation}
\z_A = b^{-1}\nabb_A b + \kep_{A}.
\label{etaa}
\end{equation}
Indeed, recall from the definition of $b$ \eqref{it3} that $b^{-1}\nabb b=-b\nabb T(u)$, 
which together with the fact that $e_A(u)=0$ implies:
$$b^{-1}\nabb_Ab=-b\nabb_AT(u)=-b[e_A,T](u)=(\dd_{e_A}T-\dd_Te_A)(u).$$
Now, using the ricci equations \eqref{ricciform} for $\dd_{e_A}T$ and $\dd_T e_A$ and the 
fact that $L(u)=e_A(u)=0$ and $T(u)=b^{-1}$ yields \eqref{etaa}.
\end{remark}

\subsection{Null structure equations}\lab{sec:nullstructeq}

Below we  write down our main structure equations. 
\begin{proposition}\lab{prop:nulltransporteq}
The components $\trc,\, \hch,\,\z$ and the lapse
$b$ verify the following equations\footnote{which can be interpreted as
transport equations along the  null geodesics 
generated by $L$. Indeed  observe that if an $P$ tangent tensorfield
 $\Pi$ satisfies the homogeneous equation  $\ddb_4\Pi=0$ then $\Pi$
is parallel transported along null geodesics. }:
\begin{align}
&L(b) = - b\, \db  , \label{D4a}\\
&L(\trc) + \half (\trc)^2 = - |\hch|^2 -\db   \trc,
\label{D4trchi}\\
&\ddb_{4} \hch + \trc \hch = 
-\db   \hch - \a,
\label{D4chih}\\
&\ddb_{4}\z_A + \half (\trc)\z_A = -(\kepb_B +\z_B)\hch_{AB} -
\half \trc \kepb_A - \b_A,\label{D4eta}.
\end{align}
\label{proptransp}
\end{proposition}
\begin{remark}
Equation \eqref{D4trchi} is known as the Raychaudhuri equation in the relativity
literature, see e.g. \cite{HaEl}.
\end{remark}

\begin{proof}
The proof is derived  from  the  formulas \eqref{ricciform} above (see \cite{ChKl} chapter 7). We briefly sketch the proof for convenience. We start with \eqref{D4a}. Using the fact that $L'$ is geodesic \eqref{def:L'1} and the fact that $L=bL'$ by \eqref{it2}, we obtain:
$$\dd_{L}L=b^{-1}L(b)L$$
which together with the Ricci equations \eqref{ricciform} for $\dd_LL$ yields \eqref{D4a}.

To obtain \eqref{D4trchi} and \eqref{D4chih}, we compute:
\bee
\ddb_L\chi_{AB}&=&L(\chi_{AB})-\chi(\ddb_Le_A,e_B)-\chi(e_A,\ddb_Le_B)\\
&=&\gg(\dd_L\dd_{e_A}L,e_B)-\chi(\ddb_Le_A,e_B)+\gg(\dd_{e_A}L,\dd_Le_B)-\chi(e_A,\ddb_Le_B)\\
&=&\gg(\dd_{e_A}\dd_LL,e_B)+\gg(\dd_{[L,e_A]}L,e_B)-\chi(\ddb_Le_A,e_B)+\rr_{LALB}\\
&&+\gg(\dd_{e_A}L,\dd_Le_B-\ddb_Le_B)\\
&=&\gg(\dd_{e_A}\dd_LL,e_B)+\gg(\dd_{\dd_Le_A-\ddb_Le_A-\dd_{e_A}L}L,e_B)+\a_{AB}\\
&& +\gg(\dd_{e_A}L,\dd_Le_B-\ddb_Le_B)
\eee
which together with the Ricci equations \eqref{ricciform} yields:
\be\lab{D4chi}
\ddb_L\chi_{AB}=-\chi_{AC}\chi_{CB}-\db\chi_{AB}-\a_{AB}.
\ee
Decomposing \eqref{D4chi} into its trace and traceless part yields respectively \eqref{D4trchi} and \eqref{D4chih}.

Finally, we derive \eqref{D4eta}. We compute:
\bee
\ddb_L\z_{A}&=&L(\z_{A})-\z(\ddb_Le_A)\\
&=&\half\gg(\dd_L\dd_{\lb}L,e_A)+\half\gg(\dd_{\lb}L,\dd_Le_A)-\z(\ddb_Le_A)\\
&=&\half\gg(\dd_{\lb}\dd_LL,e_A)+\half\gg(\dd_{[L,\lb]}L,e_A)+\half\rr_{L\lb AL}+\half\gg(\dd_{\lb}L,\dd_Le_A-\ddb_Le_A)\\
&=&\half\gg(\dd_{\lb}\dd_LL,e_A)+\half\gg(\dd_{\dd_L\lb-\dd_{\lb}L}L,e_A)-\b_{A}+\half\gg(\dd_{\lb}L,\dd_Le_A-\ddb_Le_A),
\eee
which together with the Ricci equations \eqref{ricciform} yields \eqref{D4eta}.
\end{proof}

To obtain estimates for $\chi$, we may use the transport equations \eqref{D4trchi} \eqref{D4chih}. However, this does not allow us to get enough regularity. Instead, we follow \cite{KR:Annals} \cite{FLUX} and consider \eqref{D4trchi} for $\trc$ together with an elliptic system of Hodge type for $\hch$.
\begin{proposition}
The expression $(\divb\hch)_A=\nabb^B \hch_{AB}$ 
verifies the following equation:
\be\label{Codaz}
(\divb \hch)_A + \hch_{AB}\kep_{B}=\half (\nabb_A \trc + \kep_{A} \trc) - 
\b_A.
\ee
\end{proposition}

\begin{proof}
The proof is derived  from  the  formulas \eqref{ricciform} (see \cite{ChKl} chapter 7). We briefly sketch the proof for convenience. We compute:
\bee
\ddb_C\chi_{AB}&=&e_C(\chi_{AB})-\chi(\ddb_{e_C}e_A,e_B)-\chi(e_A,\ddb_{e_C}e_B)\\
&=&\gg(\dd_{e_C}\dd_{e_A}L,e_B)-\chi(\ddb_{e_C}e_A,e_B)+\gg(\dd_{e_A}L,\dd_{e_C}e_B)-\chi(e_A,\ddb_{e_C}e_B)\\
&=&\gg(\dd_{e_A}\dd_{e_C}L,e_B)+\gg(\dd_{[e_C,e_A]}L,e_B)-\chi(\ddb_{e_C}e_A,e_B)+\rr_{CBLA}\\
&&+\gg(\dd_{e_A}L,\dd_{e_C}e_B-\ddb_{e_C}e_B)\\
&=&\ddb_A\chi_{CB}-\gg(\dd_{e_C}L,\dd_{e_A}e_B-\ddb_{e_A}e_B)+\gg(\dd_{\dd_{e_C}e_A-\ddb_{e_C}e_A-\dd_{e_A}e_C+\ddb{e_A}e_C}L,e_B)\\
&&+\rr_{CABL}+\gg(\dd_{e_A}L,\dd_Le_B-\ddb_Le_B)
\eee
which together with the Ricci equations \eqref{ricciform} yields:
$$\ddb_C\chi_{AB}=\ddb_{B}\chi_{AC}-\chi_{AB}\kep_{C}+\rr_{CBLA}+\chi_{AC}\kep_B.$$
Contracting in the previous equality yields \eqref{Codaz}.
\end{proof}

Finally, we consider the control of $\z$ and $\lb\trc$. To this end, we follow again \cite{KR:Annals} \cite{FLUX}: we derive an elliptic system of Hodge type for $\z$ and a transport equation for $\lb\trc$.
\begin{proposition}
We have:
\bea
\lab{D3trc}\lb(\trc)+\half\trchb\trc&=&2\divb\z+(\d +n^{-1}\nab_Nn)\trc-\hch\c\hchb+2\z\c\z+2\r,\\
\lab{D3chih}
\ddb_3\hch+\half\trchb\hch&=&\nabb\widehat{\otimes}\z+(\d +n^{-1}\nab_Nn)\hch-\half 
\trc\hchb+\z\widehat{\otimes} \z,
\eea
where for $F, G$ $\ptu$-tangent 1 forms, we denote by $\nabb\widehat{\otimes}F$ the traceless part of the symmetrized covariant derivative of $F$, i.e. $\nabb\widehat{\otimes}F_{AB}=\nabb_AF_B+\nabb_BF_A-\divb F\d_{AB}$ and by $F\widehat{\otimes} G$ the traceless symmetric 2-tensor $F\widehat{\otimes}G_{AB}=F_AG_B+F_BG_A-2F_CG_C\d_{AB}$. Also, the expressions  $\divb
\z=\nabb^B\z_B$ and \mbox{$\curlb \z=\in^{AB}\nabb_A\z_B$} 
verify the following equations:
\begin{align}
&\divb\,\z = \half\bigg(\mu +\half\trc\trchb +\hch\c\hchb -2|\z|^2\bigg)
 - \rho,\label{diveta}\\
&\curlb\,\z =  -\half\hch \wedge  \hchb+\s,
\label{curleta}   
\end{align}
where for $F, G$ symmetric traceless $\ptu$-tangent 2-tensors, we denote by $F\wedge G$ the tensor $F\wedge G_{AB}=\in_{AB}F_{AC}G_{BC}$. Furthermore, we have the Gauss equation,
\be\lab{gauss}
K=\half\hch\c\hchb-\frac{1}{4} \trc \trchb -\rho.
\end{equation}
Finally, setting, 
\be\lab{eqmu}
\mu=\lb(\trc) - 
\big (\d +n^{-1}\nab_N n\big )\trc
\end{equation}
we find 
\begin{equation}
\begin{split}
L(\mu) + \trc \mu &=2(\zb-\z)\c\nabb\trc
-2\hch\c\Bigl (\nabb\widehat{\otimes}\z + \z\widehat{\otimes}\z  -\d\hch\Bigr )
\\& -\trc 
\bigg (2\divb\z+2\z\c\z+4 (\kep - \z)\c n^{-1} \nabb n -2\db(\d+n^{-1}\nab_Nn)+
4\rho \\&
-\half\trc\trchb+2|\kep|^2+3|\hch|^2+4\hch\c\etah-2|n^{-1}N(n)|^2\bigg).
\end{split}
\label{D4tmu}
\end{equation}
\end{proposition}

\begin{proof}
To obtain \eqref{D3trc}, \eqref{D3chih}, \eqref{diveta} and \eqref{curleta}, we compute:
\bee
\ddb_{\lb}\chi_{AB}&=&\lb(\chi_{AB})-\chi(\ddb_{\lb}e_A,e_B)-\chi(e_A,\ddb_{\lb}e_B)\\
&=&\gg(\dd_{\lb}\dd_{e_A}L,e_B)-\chi(\ddb_{\lb}e_A,e_B)+\gg(\dd_{e_A}L,\dd_{\lb}e_B)-\chi(e_A,\ddb_{\lb}e_B)\\
&=&\gg(\dd_{e_A}\dd_{\lb}L,e_B)+\gg(\dd_{[\lb,e_A]}L,e_B)-\chi(\ddb_{\lb}e_A,e_B)+\rr_{\lb ALB}\\
&&+\gg(\dd_{e_A}L,\dd_{\lb}e_B-\ddb_{\lb}e_B)\\
&=& 2\nabb_A\z_B-\gg(\dd_{\lb}L,\dd_{e_A}e_B-\ddb_{e_A}e_B)+\gg(\dd_{\dd_{\lb}e_A-\ddb_{\lb}e_A-\dd_{e_A}\lb}L,e_B)\\
&&+\r\d_{AB}-\s\in_{AB}+\gg(\dd_{e_A}L,\dd_{\lb}e_B-\ddb_{\lb}e_B)
\eee
which together with the Ricci equations \eqref{ricciform} yields:
\be\lab{D3chi}
\ddb_{\lb}\chi_{AB}=2\nabb_A\z_B+\chi_{AB}(\d+n^{-1}\nab_Nn)+2\z_A\z_B-\chb_{AC}\chi_{CB}+\r\d_{AB}-\s\in_{AB}.
\ee
Taking the symmetric part of \eqref{D3chi}, and decomposing into its trace and traceless part yields respectively \eqref{D3trc} and \eqref{D3chih}. \eqref{diveta} follows from \eqref{D3trc}. Finally, taking the antisymmetric part of \eqref{D3chi} yields \eqref{curleta}.

We now focus on obtaining \eqref{D4tmu}. Differentiating the Raychaudhuri equation with respect to $\lb$ yields:
\bea
\lab{D4tmu1}L(\lb\trc)&=& [L,\lb]\trc+\lb(L\trc)\\
\nn&=&\db\lb(\trc)-(\d+n^{-1}\nab_Nn)L(\trc)-2(\z-\zb)\c\nabb\trc-\lb(\trc)\trc-2\ddb_{\lb}(\hch)\c\hch\\
\nn&&-\lb(\db)\trc-\db\lb(\trc)\\
\nn&=& -(\d+n^{-1}\nab_Nn)L(\trc)-2(\z-\zb)\c\nabb\trc\\
\nn&&-\trc\left(-\half\trc\trchb+2\divb\z+(\d +n^{-1}\nab_Nn)\trc-\hch\c\hchb+2\z\c\z+2\r\right)\\
\nn&&-2\hch\c\left(-\half\trchb\hch+\nabb\widehat{\otimes}\z+(\d +n^{-1}\nab_Nn)\hch-\half 
\trc\hchb+\z\widehat{\otimes} \z\right)-\lb(\db)\trc,
\eea
where we used \eqref{D3trc} and \eqref{D3chih} in the last equality.

In view of the last term in the right-hand side of \eqref{D4tmu1}, we need to compute $\lb(\db)$. We first compute $T(\d)$. We have:
\bee
T(\d)&=&-\gg(\dd_T\dd_NT,N)-\gg(\dd_NT,\dd_TN)\\
&=& -\gg(\dd_N\dd_TT,N)-\gg(\dd_{[T,N]}T,N)+\rr_{NTNT}-\gg(\dd_NT,\dd_TN)\\
&=& -\gg(\dd_N\dd_TT,N)-\gg(\dd_{\dd_TN-\dd_NT}T,N)+\r-\gg(\dd_NT,\dd_TN),
\eee
which together with the Ricci equations \eqref{ricciform} yields:
\be\lab{D4tmu2}
T(\d)=-n^{-1}\nab^2_Nn+\rho+|\kep|^2+\d^2+2\kep\c(\z-n^{-1}\nabb n).
\ee
Now, since $L=T+N$, $\lb=T-N$ and $\db=\d-n^{-1}\nab_Nn$, we have:
$$T(\d)=\half \lb(\db)+\half L(\d+n^{-1}\nab_Nn)-n^{-1}\nab^2_Nn+|n^{-1}N(n)|^2,$$
which together with \eqref{D4tmu2} yields:
\be\lab{D4tmu3}
\lb(\db)=-L(\d+n^{-1}\nab_Nn)+2\rho+2|\kep|^2+2\d^2+4\kep\c(\z-n^{-1}\nabb n)-2|n^{-1}N(n)|^2.
\ee
Therefore taking  $ \mu =\lb(\trc)- (\d  + n^{-1} N (n))\trc$, and plugging \eqref{D4tmu3} in \eqref{D4tmu1}, 
we derive the desired transport equation \eqref{D4tmu}.
\end{proof}

\subsection{Commutation formulas}

We have the following four useful commutation formulas (see \cite{ChKl} p. 159):
\begin{lemma}
Let $\Pi_{\und{A}}$ be an m-covariant tensor tangent to the surfaces $P_{t,u}$.
Then,
\bea
\lab{comm1}\nabb_B \ddb_{4} \Pi_{\und{A}} - \ddb_{4}\nabb_B \Pi_{\und{A}} &=&
\chi_{BC} \nabb_C \Pi_{\und{A}} - n^{-1} \nabb_B n \ddb_{4} \Pi_{\und{A}}\\
 &+& \sum_i (\chi_{A_i B} \kepb_C -
\chi_{BC}\kepb_{A_i} -\in_{A_i C}{}^*\b_B) \Pi_{A_1..\Check{C}..A_m},\nn
\eea
\bea
\lab{comm2}\nabb_B \ddb_{3} \Pi_{\und{A}} - \ddb_{3}\nabb_B \Pi_{\und{A}} &=&
\chb_{BC} \nabb_C \Pi_{\und{A}} - \xib_B \ddb_{4} \Pi_{\und{A}}- b^{-1}\nab_Bb \ddb_{3}\Pi_{\und{A}}\\ 
\nn&+& \sum_i (-\chi_{A_i B} \xib_{C} +
\chi_{BC}\xib_{A_i}  -\chb_{A_i B} \z_{C} +
\chb_{BC}\z_{A_i } \\
&&+\in_{A_i C}{}^*\bb_B) \Pi_{A_1..\Check{C}..A_m},\nn
\eea
\bea\lab{comm3}
\nn\ddb_{3}\ddb_{4} \Pi_{\und{A}} - \ddb_{4}\ddb_{3} \Pi_{\und{A}} &=&
 -\db  \ddb_{3} \Pi_{\und{A}}
+ (\d + n^{-1} \nab_Nn)\ddb_{4} \Pi_{\und{A}} +2(\z_{B}-\zb_B) \nabb_B \Pi_{\und{A}} \\ 
&+& 2\sum_i (\zb_{A_i} \z_{C} -
\zb_{C}\z_{A_i} +\in_{A_i C}{}^*\s) \Pi_{A_1..\Check{C}..A_m}.
\eea
Finally, \eqref{comm1}, \eqref{comm2} together with the fact that $N=\half(L-\lb)$ yield:
\bea
\lab{comm4}\nabb_B \nabb_{N} \Pi_{\und{A}} - \nabb_{N}\nabb_B \Pi_{\und{A}} &=&
(\chi_{BC}+k_{BC}) \nabb_C \Pi_{\und{A}} - b^{-1} \nabb_B b \nabb_{N} \Pi_{\und{A}}\\ 
&+& \half\sum_i (\chi_{A_i B} (\kepb_C+\xib_C) -
\chi_{BC}(\kepb_{A_i}+
\xib_{A_i})+\chb_{A_i B} \z_C - \chb_{BC}\z_{A_i}\nn\\ & &
-\in_{A_i C}{}^*(\b_B+\bb_B)) \Pi_{A_1..\Check{C}..A_m}.\nn
\eea
\end{lemma}

For some applications we have in mind, we would like to get rid of the term containing a $\ddb_4$ derivative in the RHS of \eqref{comm1}. This is achieved by considering the commutator $[\nabb,\ddb_{nL}]$ instead of $[\nabb,\ddb_4]$:
\bea
\lab{comm5}\nabb_B \ddb_{nL} \Pi_{\und{A}} - \ddb_{nL}\nabb_B \Pi_{\und{A}} &=&
n\chi_{BC} \nabb_C \Pi_{\und{A}}\\
&+& \sum_i (n\chi_{A_i B} \kepb_C -
n\chi_{BC}\kepb_{A_i} -\in_{A_i C}n{}^*\b_B) \Pi_{A_1..\Check{C}..A_m}.\nn
\eea
\eqref{comm5} yields for any scalar function $f$:
\be\lab{comm6}
[nL,\lap]f=-2n\chi\nabb^2f+n(2\hch_{AB}\kepb_B-\kepb_{A}\trc -n^{-1}\nabb_An\trc+\nabb\trc)\nabb_Af.
\ee
Also, we would like to get rid of the term containing a $\nabb_N$ derivative in the RHS of \eqref{comm4}. This is achieved by considering the commutator $[\nabb,\nabb_{bN}]$ instead of $[\nabb,\nabb_N]$:
\bea
\nabb_B \nabb_{bN} \Pi_{\und{A}} - \nabb_{bN}\nabb_B \Pi_{\und{A}} &=&
b(\chi_{BC}+k_{BC}) \nabb_C \Pi_{\und{A}} 
\label{comm7}\\ &+& \frac{b}{2}\sum_i (\chi_{A_i B} (\kepb_C+\xib_C) -
\chi_{BC}(\kepb_{A_i N}+
\xib_{A_i})+\chb_{A_i B} \z_C - \chb_{BC}\z_{A_i}\nn\\ & &
-\in_{A_i C}{}^*(\b_B+\bb_B)) \Pi_{A_1..\Check{C}..A_m}.\nn
\eea

\subsection{Bianchi identities}

In view of the formulas on p. 161 of \cite{ChKl}, the Bianchi equations for $\a, \b, \rho, \s, \bb$ are:
\bea
\ddb_L(\b)&=&\divb\a-\db  \b+(2\kep-\kepb)\c\a\lab{bianc1}\\
\ddb_{\lb}(\b)&=&\nabb\r+(\nabb\s)^*+2\hch\c\bb+(\d+n^{-1}\nabla_Nn)\b+\xib\c\a+3(\z\r+{}^*\z\s)\lab{bianc1bis}\\
L(\rho)&=&\divb\b-\half\hchb\c\a+(\kep-2\kepb)\c\b\lab{bianc2}\\
\lb(\rho)&=&-\divb\bb-\half\hch\c\ab+2\xib\c\b+(\kep-2\z)\c\bb\lab{bianc3}\\
L(\s)&=&-\curlb\b+\half\hchb\,{}^*\a+(-\kep+2\kepb)\,{}^*\b\lab{bianc4}\\
\lb(\s)&=&-\curlb\bb-\half\hch\,{}^*\ab-2\xib\,{}^*\b+(\kep-2\z)\,{}^*\bb\lab{bianc5}\\
\ddb_L(\bb)&=&-\nabb\r+(\nabb\s)^*+2\hchb\c\b+\db\bb-3(\zb\r-{}^*\zb\s)\lab{bianc6}
\eea

\subsection{Assumptions on $\rr$ and $u_{|_{\Sigma_0}}$}

\subsubsection{Assumptions on $\rr$}

We introduce the $L^2$ curvature flux $\rf$ relative to the time foliation: 
\be\lab{curvflux}
\rf=\left(\norm{\a}^2_{\lh{2}}+\norm{\b}_{\lh{2}}^2+\norm{\r}_{\lh{2}}^2+\norm{\s}_{\lh{2}}^2+\norm{\bb}_{\lh{2}}^2\right)^{\half}.
\ee
In view of the statement of Theorem \ref{th:mainbl2}, the goal of this paper is to control the geometry of the null hypersurfaces $\H_u$ of $u$ up to time $t=1$ when only assuming smallness on $\norm{\rr}_{\lsit{\infty}{2}}$ and $\rf$. In the rest of the paper, we still denote by $\H_u$ the portion of the hypersurface of $u$ between $t=0$ and $t=1$, and we assume for some small $\ep>0$:
\be\lab{curvflux1}
\norm{\rr}_{\lsit{\infty}{2}}\leq\ep\textrm{ and }\sup_{\o, u}\rf\leq\ep,
\ee
where the supremum is taken over all possible values $u\in\mathbb{R}$ of $u(t,x, \o)$ and over all possible $\o$ in $\S$, with  $u$ solution to the eikonal equation $\gg^{\alpha\beta}\partial_\alpha u\partial_\beta u=0$ on $\mathcal{M}$, and depending on a extra parameter $\o\in \S$. Note that \eqref{curvflux1} corresponds to a bootstrap assumption\footnote{There should be a large enough universal bootstrap constant in front of $\ep$ in the right-hand side of \eqref{curvflux1}, which we omit for the simplicity of the exposition} in the proof of Theorem \ref{th:mainbl2}  in \cite{boundedl2} under which steps {\bf C3}  and {\bf C4} must be achieved\footnote{Recall that step {\bf C3} corresponds to the control of the $u$-foliation on $\mathcal{M}$, while step {\bf C4} corresponds to the control of the error term \eqref{err}}. We refer to section 5.3 in \cite{boundedl2} for the bootstrap assumption corresponding to \eqref{curvflux1} in the proof of the bounded $L^2$ curvature conjecture.

\begin{remark}
Note that in \eqref{curvflux1}, all components of $\rr$ are controlled in $\lsit{\infty}{2}$, while all components but $\ab$ are controlled in $\li{\infty}{2}$. Thus, it will be crucial to avoid $\ab$ in our estimates in order to obtain suitable control on $\H_u$. This will be possible due to the specific form of the null structure equations of the $u$-foliation on  $\mathcal{M}$ (see section \ref{sec:nullstructeq})\footnote{The only exception is the transport equation \eqref{lb1} satisfied by $\lb\lb(b)$ which contains an $\ab$ term, and leads to the weak estimate \eqref{estlblbb}}.
\end{remark}

\begin{remark}\lab{rem:noprobleminf}
As a byproduct of the reduction to small initial data outlined in Remark \ref{rem:reducsmallisok} and performed in section 2.3 of \cite{boundedl2}, we may choose $(\Sigma_0, g, k)$ to be smooth, small and  asymptotically flat outside a compact set $U$ of $\Sigma_0$ of diameter of order 1 (see section 2.3 in \cite{boundedl2} for details). In turn, using the finite speed of propagation, we may assume that $(\mathcal{M}, \gg)$ to be smooth, small and  asymptotically flat  outside of compact set $\widetilde{U}$ of $\mathcal{M}\cap\{0\leq t\leq 1\}$ of diameter of order 1. This allows us to avoid issues about decay at infinity, and to solely concentrate on establishing regularity of the $u$-foliation on the compact set $\widetilde{U}$ of $\mathcal{M}\cap\{0\leq t\leq 1\}$.
\end{remark}

\subsubsection{Assumptions on $u_{|_{\Sigma_0}}$}\lab{sec:choiceinitdata}

Recall that $u$ is a solution to the eikonal equation $\gg^{\alpha\beta}\partial_\alpha u\partial_\beta u=0$ on $\mathcal{M}$ depending on a extra parameter $\o\in \S$. Now, for $u$ to be uniquely defined, we need to prescribe it on $\Si$ (i.e. at $t=0$). This issue has been settled in Step {\bf C1} (see \cite{param1}). In that step, the choice of $u(0,x,\o)$ is such that $u(0,x,\o)$ has enough regularity to achieve step {\bf C2}. At  the same time, it is also such that $u$ is regular enough for $t>0$ to achieve step {\bf C3}. More precisely, the regularity of $u$ for $t>0$ will involve transport equations - see for instance Proposition \ref{prop:nulltransporteq} - and will therefore require the same regularity at $t=0$. We denote this regularity at $t=0$ by the quantities $\II$ and $\Ij$, $j\in\mathbb{N}$, which are defined by
\bea\lab{defI0}
\II&=&\norm{b(0,.)-1}_{L^\infty(\Sigma_0)}+\norm{\nabla b(0,.)}_{L^\infty_u L^2(P_{0,u})}+\norm{\nabla b(0,.)}_{L^\infty_u L^4(P_{0,u})}\\
\nn&&+\norm{\trc(0,.)}_{L^\infty(\Sigma_0)}+\norm{\nabla\trc(0,.)}_{L^\infty_u L^2(P_{0,u})}+\norm{\po N(0,.)}_{L^\infty(\Sigma_0)}\\
\nn&&+\norm{\nabla\po N(0,.)}_{L^\infty_u L^2(P_{0,u})}+\norm{\po b(0,.)}_{L^\infty(\Sigma_0)}+\norm{\nabb\po b(0,.)}_{L^\infty_u L^2(P_{0,u})}\\
\nn&&+\norm{\po\chi(0,.)}_{L^\infty_u L^2(P_{0,u})}+\norm{\po\z(0,.)}_{L^\infty_u L^2(P_{0,u})}+\norm{\po^2N(0,.)}_{L^\infty_u L^2(P_{0,u})}\\
\nn&&+\norm{\po^2b(0,.)}_{L^\infty_u L^2(P_{0,u})},
\eea
and 
\bea\lab{defI0j}
\nn\Ij &=& \norm{P_j(NN(\trc))(0, .)}_{L^2(\Sigma_0)}+\norm{P_j(NN(b))(0, .)}_{L^2(\Sigma_0)}+\norm{P_j(\nabb_N\Pi(\po\chi))(0, .)}_{L^2(\Sigma_0)}\\
&&+\norm{P_j(\nabb_N(\Pi(\po^2N)))(0, .)}_{L^2(\Sigma_0)}+\norm{P_j(\Pi(\po^2\z))(0, .)}_{L^2(\Sigma_0)},
\eea
where $P_j$ denotes the geometric Littlewood-Paley projections $P_j$ which have been constructed in \cite{LP} using the heat flow on the surfaces $P_{0,u}$ (see section \ref{sec:geompal}). This regularity $\II$ and $\Ij$required for $u(0,x,\o)$ is consistent with the estimates derived in step {\bf C1}, where the following estimate for the initial data quantities $\II$ and $\Ij$ has been derived under the curvature bound assumption \eqref{curvflux1} (see \cite{param1}):
\be\lab{estinit}
\II\les\ep,
\ee
and 
\be\lab{estinit:bis}
\Ij\les\ep 2^{\frac{j}{2}},\,\,\forall j\geq 0.
\ee
From now on, we assume that $u$ is the solution to the eikonal equation $\gg^{\alpha\beta}\partial_\alpha u\partial_\beta u=0$ on $\mathcal{M}$ which is prescribed on $\Si$ as in step {\bf C1}, and such that it satisfies on $\Si$ the smallness assumption \eqref{estinit} and \eqref{estinit:bis}.

\subsection{Main results}

We define some norms on $\H_u$. For any $1\leq p\leq +\infty$ and for any tensor $F$ on $\H_u$, we have:
$$\norm{F}_{\lh{p}}=\left(\int_0^1dt\int_{P_{t,u}}|F|^p\dmt\right)^{\frac{1}{p}},$$
where $\dmt$ denotes the area element of $P_{t,u}$.
We also introduce the following norms:
$$\no(F)=\norm{F}_{\lh{2}}+\norm{\nabb F}_{\lh{2}} +\norm{\ddb_LF}_{\lh{2}},$$
$$\noo(F)=\no(F)+\norm{\nabb^2 F}_{\lh{2}} +\norm{\nabb\ddb_LF}_{\lh{2}}.$$
Let $x'$ a coordinate system on $\pou$. By transporting this coordinate system along the null geodesics generated by $L$, we obtain a coordinate system $(t,x')$ of $\H_u$. We define the following norms:
$$\norm{F}_{\xt{\infty}{2}}=\sup_{x'\in P_{0,u}}\left(\int_0^1 |F(t,x')|^2dt\right)^{\half},$$
$$\norm{F}_{\xt{2}{\infty}}=\normm{\sup_{0\leq t\leq 1}|F(t,x'))|}_{L^2(\pou)}.$$

\begin{remark}
In the rest of the paper, all inequalities hold for any $\o\in\S$ with the constant in the right-hand side being independent of $\o$. Thus, one may take the supremum in $\o$ everywhere. To ease the notations, we do not explicitly write down this supremum. 
\end{remark}

\begin{remark}
Let a function $f$ depending on $u\in\mathbb{R}$. In the rest of the paper, all estimates on $\H_u$ will be either of the following types
\be\lab{esttype1}
|f(u)|\les \ep,
\ee
or
\be\lab{esttype2}
|f(u)|\les2^j\ep+2^{\frac{j}{2}}\ep\gamma(u),
\ee
where $\gamma$ is a function of $L^2(\mathbb{R})$ satisfying $\norm{\gamma}_{L^2(\mathbb{R})}\leq 1$. For instance, the inequalities \eqref{estn}-\eqref{estzeta}, \eqref{estNomega}, \eqref{estricciomega}, and \eqref{dechch1}-\eqref{estpo2b} below are of the first type, while the inequalities \eqref{estlblbtrc}-\eqref{estlblbb}, \eqref{estricciomegabis} and \eqref{estricciomega3} below are of the second type. All inequalities of the first type hold for any $u$ with the constant in the right-hand side being independent of $u$. Thus, one may take the supremum in $u$ in these inequalities. To ease the notations, we do not explicitly write down the supremum in $u$ for all estimates of the type \eqref{esttype1}. 
\end{remark}

\begin{remark}
The contribution $2^{\frac{j}{2}}\gamma(u)$ to \eqref{esttype2} will always come from the initial data term of a transport equation estimate which is controlled using \eqref{estinit:bis}. In the particular case of the estimate \eqref{estlblbb} below, it will also come from the presence of an term involving $\ab$ in the transport equation satisfied by $\lb\lb(b)$ (see \eqref{lb1}).
\end{remark}

The following theorem investigates the regularity of $u$ with respect to $(t,x)$:
\begin{theorem}\lab{thregx}
Assume that $u$ is the solution to the eikonal equation $\gg^{\alpha\beta}\partial_\alpha u\partial_\beta u=0$ on $\mathcal{M}$ such that $u$ is prescribed on $\Si$ as in section \ref{sec:choiceinitdata} where it satisfies in particular \eqref{estinit}. Assume also that the estimate \eqref{curvflux1} is satisfied. Then, null geodesics generating $\H_u$ do not have conjugate points (i.e. there are no caustics) and distinct null geodesics do not intersect. Furthermore, the following estimates are satisfied:
\be\lab{estn}
\norm{n-1}_{\lh{\infty}}+\norm{\nab n}_{\lh{\infty}}+\norm{\nabla n}_{\tx{\infty}{2}}+\norm{\nabla^2n}_{\tx{\infty}{2}}+\norm{\nabla T(n)}_{\tx{\infty}{2}}\lesssim\ep,
\ee
\be\lab{estk}
\no(k)+\norm{\ddb_{\lb}\kep}_{\lh{2}}+\norm{\lb(\d)}_{\lh{2}}+\norm{\kepb}_{\xt{\infty}{2}}+\norm{\db}_{\xt{\infty}{2}}\lesssim\ep,
\ee
\be\lab{estb}
\norm{b-1}_{\lh{\infty}}+\noo(b)+\norm{\lb(b)}_{\xt{2}{\infty}}+\norm{\lb(b)}_{\tx{\infty}{4}}\lesssim\ep,
\ee
\be\lab{esttrc}
\norm{\trc}_{\lh{\infty}}+\norm{\nabb\trc}_{\xt{2}{\infty}}+\norm{\lb\trc}_{\xt{2}{\infty}}
\lesssim\ep,
\ee
\be\lab{esthch}
\norm{\hch}_{\xt{2}{\infty}}+\no(\hch)+\norm{\ddb_{\lb}\hch}_{\lh{2}}
\lesssim\ep,
\ee
\be\lab{estzeta}
\norm{\z}_{\xt{2}{\infty}}+\no(\z)\lesssim\ep.
\ee
\end{theorem} 

We introduce the family of intrinsic Littlewood-Paley projections 
$P_j$ which have been constructed in \cite{LP} using the heat flow 
on the surfaces $P_{t,u}$ (see section \ref{sec:geompal}). This allows us to state our second theorem 
which investigates the regularity of $\lb\lb\trc$, $\nabb_{\lb}\z$ and $\lb\lb b$.
\begin{theorem}\lab{thregx1}
Assume that $u$ is the solution to the eikonal equation $\gg^{\alpha\beta}\partial_\alpha u\partial_\beta u=0$ on $\mathcal{M}$ such that $u$ is prescribed on $\Si$ as in section \ref{sec:choiceinitdata} where it satisfies in particular \eqref{estinit} and \eqref{estinit:bis}. Assume also that the estimate \eqref{curvflux1} is satisfied. Then, there exists a function $\gamma$ in $L^2(\mathbb{R})$ satisfying $\norm{\gamma}_{L^2(\mathbb{R})}\leq 1$, such that for all $j\geq 0$, we have:
\be\lab{estlblbtrc}
\norm{P_j\lb\lb\trc}_{\lh{2}}\lesssim 2^j\ep+2^{\frac{j}{2}}\ep\gamma(u),
\ee
\be\label{estlbzeta}
\norm{P_j\ddb_{\lb}(\z)}_{\lh{2}}\lesssim \ep+2^{-\frac{j}{2}}\ep\gamma(u),
\ee
and
\be\lab{estlblbb}
\norm{P_j\lb\lb b}_{\tx{\infty}{2}}\lesssim 2^j\ep+2^{\frac{j}{2}}\ep\gamma(u).
\ee
\end{theorem}


The following theorem investigates the regularity with respect to the parameter $\o\in\S$.
\begin{theorem}\lab{thregomega}
Assume that $u$ is the solution to the eikonal equation $\gg^{\alpha\beta}\partial_\alpha u\partial_\beta u=0$ on $\mathcal{M}$ such that $u$ is prescribed on $\Si$ as in section \ref{sec:choiceinitdata} where it satisfies in particular \eqref{estinit} and \eqref{estinit:bis}. Assume also that the estimate \eqref{curvflux1} is satisfied. Then, we have the following estimates:
\be\lab{estNomega}
\norm{\po N}_{\lh{\infty}}\lesssim 1,
\ee
\be\lab{estricciomega}
\norm{\dd\po N}_{\xt{2}{\infty}}+\norm{\po b}_{\lh{\infty}}+\norm{\nabb\po b}_{\xt{2}{\infty}}+\norm{\po\chi}_{\xt{2}{\infty}}+\norm{\po\z}_{\xt{2}{\infty}}\lesssim \ep,
\ee
and
\be\lab{estricciomegabis}
\norm{P_j\ddb_{\lb}\Pi(\po\chi)}_{\tx{p}{2}}\lesssim 2^j\ep+2^{\frac{j}{2}}\ep\gamma(u),
\ee
where $p$ is any real number such that $2\leq p<+\infty$, and where $\gamma$ is a function of $L^2(\mathbb{R})$ satisfying $\norm{\gamma}_{L^2(\mathbb{R})}\leq 1$.\\

Also, we have the following decomposition for $\hch$:
\be\lab{dechch}
\hch=\chi_1+\chi_2,
\ee
where $\chi_1$ and $\chi_2$ are two symmetric traceless $\ptu$-tangent 2-tensors satisfying: 
\be\lab{dechch1}
\norm{\po\chi_1}_{\tx{\infty}{2}}+\no(\chi_2)+\norm{\ddb_{\lb}\chi_2}_{\lh{2}}+\norm{\chi_2}_{\xt{\infty}{2}}+\norm{\po\chi_2}_{\tx{\infty}{2}}\lesssim \ep
\ee
and for any $2\leq p<+\infty$, we have:
\be\lab{dechch2}
\norm{\nabb\chi_1}_{\tx{\infty}{2}}+\norm{\chi_1}_{\tx{p}{\infty}}+\norm{\po\chi_2}_{\tx{p}{4_-}}+\norm{\po\chi_2}_{\lh{6_-}}+\norm{\nabb\po\chi_2}_{\lh{2}}\lesssim \ep.
\ee
Furthermore, for any $2\leq q<4$, we have:
\be\lab{dechch3}
\norm{\ddb_L\chi_1}_{\tx{\infty}{2}+\tx{2}{q}}+\norm{\ddb_{\lb}\chi_1}_{\tx{\infty}{2}+\tx{2}{q}}\lesssim \ep.
\ee

Finally, let $\o$ and $\o'$ in $\S$. Then, there holds the following lower bound
\be\lab{ad1}
|N(., \o)-N(., \o')|\gtrsim |\o-\o'|.
\ee
\end{theorem}

\begin{remark}
Notice from \eqref{dechch1} that $\chi_1$ and $\chi_2$ have at least the same regularity as $\hch$. Now, the point of the decomposition \eqref{dechch} is that both $\chi_1$ and $\chi_2$ have better regularity properties than $\hch$. Indeed, in view of \eqref{dechch2}, $\chi_1$ has better regularity with respect to $(t,x)$ while $\chi_2$ has better regularity with respect to $\o$.
\end{remark}

\begin{remark}
Let $\o$ and $\o'$ in $\S$. The estimate \eqref{estNomega} for $N$ yields the following upper bound for $N(\c, \o)-N(\c, \o')$:
$$|N(., \o)-N(., \o')|\les |\o-\o'|.$$
Note that \eqref{ad1} establishes the corresponding lower bound. 
\end{remark}

Finally, the following theorem contains estimates for second order derivatives with respect to $\o$.
\begin{theorem}\lab{thregomega2}
Assume that $u$ is the solution to the eikonal equation $\gg^{\alpha\beta}\partial_\alpha u\partial_\beta u=0$ on $\mathcal{M}$ such that $u$ is prescribed on $\Si$ as in section \ref{sec:choiceinitdata} where it satisfies in particular \eqref{estinit} and \eqref{estinit:bis}. Assume also that the estimate \eqref{curvflux1} is satisfied. Then, we have the following estimates:
\be\lab{estNomega2}
\norm{\po^2 N}_{\xt{2}{\infty}}\lesssim 1,
\ee
\be\lab{estNomega2bis}
\norm{\ddb_{L}\po^2 N}_{\lh{2}}\lesssim \ep.
\ee
\be\lab{estpo2b}
\norm{\po^2b}_{\tx{\infty}{2}}\les\ep,
\ee
\be\lab{estricciomega2}
\norm{P_j\Pi(\po^2\chi)}_{\tx{\infty}{2}}\lesssim 2^j\ep,
\ee
and
\be\lab{estricciomega3}
\norm{P_j\ddb_{\lb}\Pi(\po^2 N)}_{\tx{p}{2}}+\norm{P_j\Pi(\po^2\z)}_{\tx{p}{2}}\lesssim 2^j\ep+2^{\frac{j}{2}}\ep\gamma(u),
\ee
where $p$ is any real number such that $2\leq p<+\infty$, and where $\gamma$ is a function of $L^2(\mathbb{R})$ satisfying $\norm{\gamma}_{L^2(\mathbb{R})}\leq 1$.
\end{theorem}

\begin{remark}
Our assumption on curvature \eqref{curvflux1} is critical with respect to the control of the Eikonal equation as can be seen throughout the paper where numerous log-losses are barely overcome. In order to prove Theorem \ref{thregx}, Theorem \ref{thregx1}, Theorem \ref{thregomega}, and Theorem \ref{thregomega2} we will rely in particular on the null transport equations and the elliptic systems of Hodge type on $\ptu$ of section \ref{sec:nullstructeq}, the geometric Littlewood-Paley theory of \cite{LP}, sharp trace theorems, and an extensive use of the crucial structure of the Bianchi identities \eqref{bianc1}-\eqref{bianc6}. 
\end{remark}

\begin{remark}
The regularity with respect to $(t,x)$ for $u$ is clearly limited as a consequence of the fact that we only assume $L^2$ bounds on $\rr$. On the other hand, $\rr$ is independent of the parameter $\o$, and one might infer that $u$ is smooth with respect to $\o$. Surprisingly, this is not at all the case. Indeed, we are even not able to go beyond estimates for the second order derivatives with respect to $\o$ which are given in Theorem \ref{thregomega2}. This is due to the fact that we rely in a fundamental way on the null transport equations of Proposition  \ref{prop:nulltransporteq}. Now, the commutator between $L$ and $\o$ gives rise to a tangential derivative with respect to $\ptu$ (see \eqref{commo1}) for which we have less control. This leads to a loss of one derivative for each derivative taken with respect to $\o$ for all quantities estimated through transport equations. This is best seen by comparing the estimate \eqref{esttrc} \eqref{esthch} for $\chi$, the estimate \eqref{estricciomega} for $\po\chi$ and the estimate \eqref{estricciomega2} for $\po^2\chi$.
\end{remark}

\subsection{Dependance of the norm $L^\infty_u\lh{2}$ on $\o\in\S$}\lab{sec:obadidonc}

Let $\o$ and $\o'$ in $\S$ such that 
$$|\o-\o'|\les 2^{-\frac{j}{2}}.$$ 
Let $u=u(., \o)$ and $u'=u(., \o')$. In this section, we compare the norm in $\li{\infty}{2}$ with the norm in $L^\infty_{u'}L^2(\H_{u'})$ for various scalars and tensors, relying on the estimates of the previous section. Let us first stress the difficulty by considering the decomposition for $\trc$ in Proposition \ref{cor:so} below. A naive approach consists in writing the following decomposition 
$$\trc(t,x,\o)=\trc(t,x,\o')+(\trc(t,x,\o)-\trc(t,x,\o'))=f_1^j+f_2^j.$$
$f^j_1$ does not depend on $\o$ and satisfies, in view of the estimate \eqref{esttrc}
$$\norm{f^j_1}_{L^\infty}\les \norm{\trc(.,\o')}_{L^\infty}\les\ep.$$
Also, we have
$$f^j_2=(\o-\o')\int_0^1\po\trc(t,x,\o_\sigma)d\sigma,$$
which together with the fact that $|\o-\o'|\les 2^{-\frac{j}{2}}$ yields
$$\norm{f^j_2}_{\li{\infty}{2}}\les 2^{-\frac{j}{2}}\normm{\int_0^1\po\trc(t,x,\o_\sigma)d\sigma}_{\li{\infty}{2}}.$$
Unfortunately, we can not obtain the desired estimate for $f^j_2$ since we have $\po\trc(.,\o_\sigma)\in L^\infty_{u_\sigma}L^2(\H_{u_\sigma})$, and $\li{\infty}{2}$ and $L^\infty_{u_\sigma}L^2(\H_{u_\sigma})$ are not directly comparable. Nevertheless, relying on the geometric Littlewood-Paley projections of \cite{LP}, on well-suited coordinate systems, and on various commutator estimates, we are able to improve on this naive approach in order to obtain the decompositions below.
 
\begin{proposition}\lab{cor:xx1bis}
Let $\o$ and $\o'$ in $\S$ such that $|\o-\o'|\les 2^{-\frac{j}{2}}$. Let $N=N(.,\o)$ and $N'=N(.,\o')$. For any $j\geq 0$, we have the following decomposition for $N-N'$:
$$N-N'=(F^j_1+F^j_2)(\o-\o')$$
where $F^j_1$ only depends on $\o'$ and satisfies:
$$\norm{F^j_1}_{L^\infty}\les 1,$$
and where $F^j_2$ satisfies:
$$\norm{F^j_2}_{L^\infty_u\lh{2}}\les 2^{-\frac{j}{2}}.$$
\end{proposition}

\begin{proposition}\lab{cor:so}
Let $\o$ and $\o'$ in $\S$ such that $|\o-\o'|\les 2^{-\frac{j}{2}}$. For any $j\geq 0$, we have the following decomposition for $\trc(.,\o)$:
$$\trc(.,\o)=f^j_1+f^j_2$$
where $f^j_1$ does not depend on $\o$ and satisfies:
$$\norm{f^j_1}_{L^\infty}\les \ep,$$
and where $f^j_2$ satisfies:
$$\norm{f^j_2}_{L^\infty_u\lh{2}}\les 2^{-\frac{j}{2}}\ep.$$
\end{proposition}

\begin{proposition}\lab{cor:so1}
Let $\o$ and $\o'$ in $\S$. Let $p\in\mathbb{Z}$. For any $j\geq 0$, we have the following estimate for $b^p(.,\o)-b^p(,\o')$:
$$\norm{b^p(.,\o)-b^p(.,\o')}_{L^\infty_u\lh{2}}\les |\o-\o'|\ep.$$
\end{proposition}

\begin{lemma}\lab{cor:so2}
Let $\o$ and $\o'$ in $\S$. For any $j\geq 0$, we have the following estimate for $\chi_2(.\o)-\chi_2(.,\o')$:
$$\norm{\chi_2(.\o)-\chi_2(.,\o')}_{L^\infty_u\lh{4_-}}\les |\o-\o'|\ep.$$
\end{lemma}

\begin{proposition}\lab{cor:so3}
Let $\o$ and $\o'$ in $\S$ such that $|\o-\o'|\les 2^{-\frac{j}{2}}$. For any $j\geq 0$, we have the following decomposition for $\chi(.,\o)$ and $\hch(.,\o)$:
$$\chi(.,\o),\, \hch(.,\o)=\chi_2(.,\o')+F^j_1+F^j_2$$
where $F^j_1$ does not depend on $\o$ and satisfies for any $2\leq p<+\infty$:
$$\norm{F^j_1}_{L^\infty_{u_{\o'}}L^p_tL^\infty(P_{t, u_{\o'}})}\les \ep,$$
and where $F^j_2$ satisfies:
$$\norm{F^j_2}_{L^\infty_u\lh{2}}\les 2^{-\frac{j}{2}}\ep.$$
\end{proposition}

\begin{proposition}\lab{cor:so4}
Let $\o$ and $\o'$ in $\S$ such that $|\o-\o'|\les 2^{-\frac{j}{2}}$. For any $j\geq 0$, we have the following decomposition for $\chi(.,\o)$ and $\hch(.,\o)$:
$$\chi(.,\o),\, \hch(.,\o)=F^j_1+F^j_2$$
where $F^j_1$ does not depend on $\o$ and satisfies:
$$\norm{F^j_1}_{L^\infty_{u_{\o'}}L^\infty(P_{t, u_{\o'}})L^2_t}\les \ep,$$
and where $F^j_2$ satisfies:
$$\norm{F^j_2}_{L^\infty_u\lh{2}}\les 2^{-\frac{j}{2}}\ep.$$
\end{proposition}

\begin{proposition}\lab{cor:so7}
Let $\o$ and $\o'$ in $\S$ such that $|\o-\o'|\les 2^{-\frac{j}{2}}$. For any $j\geq 0$, we have the following decomposition for $\hch(.,\o)^2$:
$$\hch(.,\o)^2=\chi_2(.,\o')^2+\chi_2(.,\o')F^j_1+\chi_2(.,\o')F^j_2+F^j_3+F^j_4+F^j_5$$
where $F^j_1$ and $F^j_3$ do not depend on $\o$ and satisfy:
$$\norm{F^j_1}_{L^\infty_{u_{\o'}}L^2_tL^\infty(P_{t, u_{\o'}})}+\norm{F^j_3}_{L^\infty_{u_{\o'}}L^2_tL^\infty(P_{t, u_{\o'}})}\les \ep,$$
where $F^j_2$ and $F^j_4$ satisfy:
$$\norm{F^j_2}_{L^\infty_u\lh{2}}+\norm{F^j_4}_{L^\infty_u\lh{2}}\les 2^{-\frac{j}{2}}\ep,$$
and where $F^j_5$ satisfies
$$\norm{F^j_5}_{L^2(\mathcal{M})}\les \ep 2^{-j}.$$
\end{proposition}

\begin{proposition}\lab{cor:so9}
Let $\o$ and $\o'$ in $\S$ such that $|\o-\o'|\les 2^{-\frac{j}{2}}$. For any $j\geq 0$, we have the following decomposition for $\hch(.,\o)^3$:
\bee
\hch(.,\o)^3&=&\chi_2(.,\o')^3+\chi_2(.,\o')^2F^j_1+\chi_2(.,\o')^2F^j_2+\chi_2(.,\o')F^j_3+\chi_2(.,\o')F^j_4\\
&&+\chi_2(.,\o')F^j_5+F^j_6+F^j_7+F^j_8+F^j_9
\eee
where $F^j_1$, $F^j_3$ and $F^j_6$ do not depend on $\o$ and satisfy:
$$\norm{F^j_1}_{L^\infty_{u_{\o'}}L^2_tL^\infty(P_{t, u_{\o'}})}+\norm{F^j_3}_{L^\infty_{u_{\o'}}L^2_tL^\infty(P_{t, u_{\o'}})}+\norm{F^j_6}_{L^\infty_{u_{\o'}}L^2_tL^\infty(P_{t, u_{\o'}})}\les \ep,$$
where $F^j_2$, $F^j_4$ and $F^j_7$ satisfy:
$$\norm{F^j_2}_{L^\infty_u\lh{2}}+\norm{F^j_4}_{L^\infty_u\lh{2}}+\norm{F^j_7}_{L^\infty_u\lh{2}}\les 2^{-\frac{j}{2}}\ep,$$
where $F^j_5$ and $F^j_8$ satisfy
$$\norm{F^j_5}_{L^2(\mathcal{M})}+\norm{F^j_8}_{L^2(\mathcal{M})}\les \ep 2^{-j}.$$
and where $F^j_9$ satisfies
$$\norm{F^j_9}_{L^{2_-}(\mathcal{M})}\les \ep 2^{-\frac{3j}{2}}.$$
\end{proposition}

\begin{proposition}\lab{cor:atlast}
Let $\o$ and $\o'$ in $\S$ such that $|\o-\o'|\les 2^{-\frac{j}{2}}$. For any $j\geq 0$, we have the following decomposition for $\z(.,\o)$ and $\nabb b(.,\o)$:
$$\z(.,\o),\, \nabb b(.,\o)=F^j_1+F^j_2$$
where $F^j_1$ does not depend on $\o$ and satisfies for any $2\leq p<+\infty$:
$$\norm{F^j_1}_{L^\infty_{u_{\o'}}L^2_tL^p(P_{t, u_{\o'}})}\les \ep,$$
and where $F^j_2$ satisfies:
$$\norm{F^j_2}_{L^\infty_u\lh{2}}\les 2^{-\frac{j}{4}}\ep.$$
\end{proposition}

\begin{proposition}\lab{cor:atlast1}
Let $\o$ and $\o'$ in $\S$ such that $|\o-\o'|\les 2^{-\frac{j}{2}}$. For any $j\geq 0$, we have the following decomposition for $b(.,\o)-b(.,\o')$:
$$b(.,\o)-b(.,\o')=(f^j_1+f^j_2)(\o-\o')$$
where $f^j_1$ does not depend on $\o$ and satisfies:
$$\norm{f^j_1}_{L^\infty}\les \ep,$$
and where $f^j_2$ satisfies:
$$\norm{f^j_2}_{L^\infty_u\lh{2}}\les 2^{-\frac{j}{4}}\ep.$$
\end{proposition}

\subsection{Additional estimates for $\trc$}

In this section, we state estimates for $\trc$ involving the geometric Littlewood-Paley projections $P_j$ on $\ptu$ constructed in  \cite{LP}, that are not direct consequences of the estimate \eqref{esttrc} for $\trc$ and basic properties of $P_j$.
\begin{proposition}\lab{prop:zz37}
$\trc$ satisfies the following estimates
\be\lab{zz37}
\norm{P_j\trc}_{\xt{2}{\infty}}\les 2^{-j}\ep,
\ee
and
\be\lab{zz38}
\norm{P_j(nL\trc)}_{\xt{2}{1}}\les 2^{-j}\ep.
\ee
\end{proposition}

\begin{proposition}\lab{prop:zz39}
$\trc$ satisfies the following estimates
\be\lab{zz39}
\norm{\nabb P_{\leq j}\trc}_{\xt{2}{\infty}}\les \ep,
\ee
and
\be\lab{zz40}
\norm{\nabb P_{\leq j}(nL\trc)}_{\xt{2}{1}}\les \ep.
\ee
\end{proposition}

The rest of the paper is as follows. In section \ref{sec:calcineqgeneral}, we derive several embeddings on $\ptu$, $\H_u$ and $\Sigma_t$ which are compatible with the regularity stated in Theorem \ref{thregx}. We also discuss the Littlewood-Paley projections of \cite{LP} as well as several elliptic systems of Hodge type on $\ptu$. In section \ref{sec:regxproof}, we prove Theorem \ref{thregx}. In section \ref{sec:secondderlb}, we prove Theorem \ref{thregx1}. In section \ref{sec:firstderivomega}, we prove Theorem \ref{thregomega}. In section \ref{sec:secondderivomega} we  prove Theorem \ref{thregomega2}. In section \ref{sec:depnormonomega}, we derive the various decompositions of section \ref{sec:obadidonc}. Finally, we prove Proposition \ref{prop:zz37} and Proposition \ref{prop:zz39} in section \ref{sec:commutatorest}. 

\section{Calculus inequalities on $\ptu$, $\H_u$ and $\Sit$}\lab{sec:calcineqgeneral}

In this section, we first recall some calculus inequalities from \cite{LP} on the 2-surfaces $\ptu$. We then  discuss the Littlewood-Paley projections of \cite{LP} as well as several elliptic systems of Hodge type on $\ptu$. We  establish calculus inequalities on $\H_u$. Finally, we establish calculus inequalities on $\Sit$, and we construct geometric Littlewood-Paley projections on $\Sigma_t$ in the spirit of \cite{LP}.

\subsection{Calculus inequalities on $\ptu$}\lab{sec:calcineq}

We denote by $\gamma$ the metric induced by $\gg$ on $\ptu$. 
A coordinate chart $U\subset \ptu$  with coordinates $x^1, x^2$ is admissible if,
relative to  these coordinates, there exists  a constant $c>0$ such that,
\be\lab{eq:coordchart}
c^{-1}|\xi|^2\le \gamma_{AB}(p)\xi^A\xi^B\le c|\xi|^2, \qquad \mbox{uniformly for  all }
\,\, p\in U.
\ee
We assume that $\ptu$ can be covered by a finite number
of admissible coordinate charts, i.e., charts satisfying 
the conditions \eqref{eq:coordchart}. Furthermore, we assume that the 
constant $c$ in \eqref{eq:coordchart} and the number of charts is independent of $t$ and $u$. 
\begin{remark}
The existence of a covering of $\ptu$ by coordinate charts satisfying \eqref{eq:coordchart} with a constant $c>0$ and the number of charts independent of $t$ and $u$ will be shown in section \ref{sec:coord}.
\end{remark}
Under these assumptions, the following calculus inequality has been proved in \cite{LP}:
\begin{proposition} Let $f$ be a  real scalar function. Then, 
\be\lab{eq:isoperimetric}
\|f\|_{L^2(\ptu)}\lesssim\|\nabb f\|_{L^1(\ptu)}+\|f\|_{L^1(\ptu)}.
\ee
\end{proposition}
As a corollary of the estimate \eqref{eq:isoperimetric}, 
 the following Gagliardo-Nirenberg inequality is derived in \cite{LP}:
\begin{corollary}
Given an arbitrary tensorfield  $F$ on $\ptu$ and any $2\le p<\infty$, we have:
\be\lab{eq:GNirenberg}
\|F\|_{L^p(\ptu)}\lesssim \|\nabb F\|_{L^2(\ptu)}^{1-\frac{2}{p}}\|F\|_{L^2(\ptu)}^{\frac 2 p}+\|F\|_{L^2(\ptu)}.
\end{equation}
\end{corollary}
As a corollary of \eqref{eq:isoperimetric}, it also classical to derive the following inequality (for a proof, see for example \cite{GT} page 157):
\begin{corollary}
For any tensorfield $F$ on $\ptu$ and any $p>2$:
\be\lab{sobinftyptu}
\norm{F}_{L^\infty(\ptu)}\lesssim \norm{\nabla F}_{L^p(\ptu)}+\norm{F}_{L^p(\ptu)}.
\ee
\end{corollary}

We recall  the Bochner identity on $\ptu$ (which has dimension 2). This allows us  to 
control  the $L^{2}$ norm of the second derivatives of a tensorfield 
in terms of the $L^{2}$ norm of the laplacian and 
geometric quantities associated with $\ptu$ (see for example \cite{LP} for a proof).
\begin{proposition}\label{prop:Bochner}
Let  $K$ denote the  Gauss curvature of $\ptu$. Then
\noindent 
{\bf i})\quad For a scalar function $f$:
\begin{equation}
\label{sboch}
\int_{\ptu} |\nabb^{2} f|^{2}\dmt = \int_{\ptu} |\lap f|^{2}\dmt - 
\int_{\ptu} K |\nabb f|^{2}\dmt.
\end{equation}

\noindent
{\bf ii)}\quad For a vectorfield $F_{a}$:
\bea\label{vboch}
\nn\int_{\ptu} |\nabb^{2} F|^{2}\dmt &=& \int_{\ptu} |\lap F|^{2}\dmt -
\int_{\ptu} K (2\,|\nabb F|^{2}-|\divb F|^{2}-|\curlb F|^2)\dmt \\
&&+ \int_{\ptu} K^{2} |F|^{2}\dmt,
\eea
where $\divb F=\gamma^{ab}\nabb_b F_a$, $\curlb F=\divb( ^*F)=\in_{ab}\nabb_a F_b$.
\end{proposition}

Using \eqref{eq:GNirenberg} and \eqref{vboch}, the following Bochner inequality is derived in \cite{LP} for a tensor $F$. For all $2\leq p<+\infty$, we have:
\bea
\lab{vbochineq}\norm{\nabb^2F}_{\lpt{2}}&\lesssim& \norm{\lap F}_{\lpt{2}}+(\norm{K}_{\lpt{2}}+\norm{K}^{\frac{1}{2}}_{\lpt{2}})\norm{\nabb F}_{\lpt{2}}\\
&&\nn +\norm{K}_{\lpt{2}}^{\frac{p}{p-1}}(\norm{\nabb F}_{\lpt{2}}^{\frac{p-2}{p-1}}\norm{F}_{\lpt{2}}^{\frac{1}{p-1}}+\norm{F}_{\lpt{2}}).
\eea

\subsection{Geometric Littlewood Paley theory on $\ptu$}\lab{sec:geompal}

We recall the properties of the heat  equation  for arbitrary 
tensorfields $F$ on $\ptu$.
$$\partial_\tau U(\tau)F -\lap U(\tau) F=0, \,\, U(0)F=F.$$
The following $L^2$ estimates for the operator
$U(\tau)$ are proved in \cite{LP}.
\begin{proposition}
We have the following estimates for the operator $U(\tau)$:
\begin{align}
&\|U(\tau) F\|^2_{\lpt{2}}+\int_0^\tau\norm{\nabb U(\tau')F}^2_{\lpt{2}}d\tau'\lesssim \|F\|^2_{\lpt{2}},\label{eq:l2heat1}\\
&\|\nabb U(\tau) F\|^2_{\lpt{2}}+\int_0^\tau\norm{\lap U(\tau')F}^2_{\lpt{2}}d\tau'\lesssim \|\nabb F\|^2_{\lpt{2}},\label{eq:l2heatnab}\\
& \tau\|\nabb U(\tau) F\|^2_{\lpt{2}}+\int_0^\tau{\tau'}\norm{\lap U(\tau')F}^2_{\lpt{2}}d\tau'\lesssim \|F\|^2_{\lpt{2}}.\label{eq:l2heat2}
\end{align}
\label{le:L2heat}
\end{proposition}
We also introduce the nonhomogeneous heat equation:
$$\partial_\tau V(\tau) -\lap V(\tau)=F(\tau), \,\, V(0)=0,$$
for which we easily derive the following estimates:
\begin{proposition}
We have the following estimates for the operator $V(\tau)$:
\begin{align}
&\norm{\nabb V(\tau)}^2_{\lpt{2}}+\int_0^\tau\norm{\lap V(\tau')}^2_{\lpt{2}}d\tau'\lesssim \int_0^\tau\norm{F(\tau')}^2_{\lpt{2}}d\tau',\label{heatF1}\\
&\norm{V(\tau)}^2_{\lpt{2}}+\int_0^\tau\norm{\nabb V(\tau')}^2_{\lpt{2}}d\tau'\lesssim \int_0^\tau\int_{\ptu}V(\tau')F(\tau')\dmt d\tau'.\label{heatF2}
\end{align}
\label{le:L2heatbis}
\end{proposition}

We now recall the definition of the geometric Littlewood-Paley projections $P_j$ constructed in \cite{LP}:
\begin{definition}\label{defLP}
Consider a smooth function $m$ on $[0,\infty)$,
vanishing sufficiently fast at $\infty$,
verifying the  vanishing  moments property:
\be\lab{eq:moments}
\int_0^\infty \tau^{k_1}\partial_\tau^{k_2} m(\tau) d\tau=0, \,\,\,\,
|k_1|+|k_2|\le N. 
\end{equation}
We set,
  $m_j(\tau)=2^{2j}m(2^{2j}\tau)$ 
and  define the geometric Littlewood -Paley (LP) 
projections $P_j$, for   arbitrary tensorfields  $F$ on $\ptu$
to be 
\be\lab{eq:LP}P_j F=\int_0^\infty m_j(\tau) U(\tau) F d\tau.
\end{equation}
Given an interval $I\subset \Bbb Z$ we define $$P_I=\sum_{j\in I} P_j F.$$
In particular we shall use the notation $P_{<k}, P_{\le k}, P_{>k}, P_{\ge k}$.
\end{definition}
Observe that $P_j$ are selfadjoint, i.e., $P_j=P_j^*$, in the sense,
$$<P_jF, G>=<F,P_j G>,$$
where, for any given $m$-tensors $F,G$ 
$$<F,G>=\int_{\ptu}\ga^{i_1j_1}\ldots\ga^{i_mj_m}
F_{i_1\ldots i_m}G_{j_1\ldots j_m}d\mu_u    $$ 
denotes the usual $L^2$ scalar product. Recall also 
from \cite{LP} that there exists a function $m$ satisfying \eqref{eq:moments} 
such that the LP-projections associated to $m$ verify:
\be\lab{eq:partition}
\sum_jP_j=I.
\end{equation}


The following properties of the LP-projections $P_j$ have been proved in \cite{LP}:
\begin{theorem}\label{thm:LP}
 The LP-projections $P_j$ verify the following
 properties:

i)\quad {\sl $L^p$-boundedness} \quad For any $1\le
p\le \infty$, and any interval $I\subset \Bbb Z$,
\be\lab{eq:pdf1}
\|P_IF\|_{\lpt{p}}\lesssim \|F\|_{\lpt{p}}
\end{equation}

ii) \quad  {\sl Bessel inequality} 
$$\sum_j\|P_j F\|_{\lpt{2}}^2\lesssim \|F\|_{\lpt{2}}^2$$

iii)\quad {\sl Finite band property}\quad For any $1\le p\le \infty$.
\begin{equation}
\begin{array}{lll}
\|\lap P_j F\|_{\lpt{p}}&\lesssim & 2^{2j} \|F\|_{\lpt{p}}\\
\|P_jF\|_{\lpt{p}} &\lesssim & 2^{-2j} \|\lap F \|_{\lpt{p}}.
\end{array}
\end{equation}

In addition, the $L^2$ estimates
\begin{equation}
\begin{array}{lll}
\|\nabb P_j F\|_{\lpt{2}}&\lesssim & 2^{j} \|F\|_{\lpt{2}}\\
\|P_jF\|_{\lpt{2}} &\lesssim & 2^{-j} \|\nabb F  \|_{\lpt{2}}
\end{array}
\end{equation}
hold together with the dual estimate
$$\| P_j \nabb F\|_{\lpt{2}}\lesssim 2^j \|F\|_{\lpt{2}}$$

iv) \quad{\sl Weak Bernstein inequality}\quad For any $2\le p<\infty$
\begin{align*}
&\|P_j F\|_{\lpt{p}}\lesssim (2^{(1-\frac 2p)j}+1) \|F\|_{\lpt{2}},\\
&\|P_{<0} F\|_{\lpt{p}}\lesssim \|F\|_{\lpt{2}}
\end{align*}
together with the dual estimates 
\begin{align*}
&\|P_j F\|_{\lpt{2}}\lesssim (2^{(1-\frac 2p)j}+1) \|F\|_{\lpt{p'}},\\
&\|P_{<0} F\|_{\lpt{2}}\lesssim \|F\|_{\lpt{p'}}
\end{align*}
\end{theorem}

We also recall the definition of the negative  fractional powers of $\La^2=I-\lap$ on any 
 smooth tensorfield $F$ on $\ptu$ used in \cite{LP}.
\be\lab{eq:defineLaa}
\La^{\a}F=\frac{1}{\Gamma(-\a/2)}\int_0^\infty \tau^{-\frac{\a}{2}-1}
e^{-\tau}U(\tau)F d\tau
\end{equation}
 where $\a$ is an arbitrary complex number with $\Re(\a)< 0$ and $\Gamma$ denotes the Gamma function. We extend the definition of fractional powers of $\La$ to the range of $\a$ with $\Re(\a)>0$, on smooth tensorfields $F$, by defining first 
 $$
 \La^\a F = \La^{\a-2} \cdot (I-\lap) F 
 $$
 for $0< \Re (\a) \le 2$ and then, in general, for 
 $0< \Re(\a) \le 2n$, with an arbitrary positive integer $n$, according 
 to the formula
 $$
 \La^\a F = \La^{\a-2n} \cdot (I-\lap)^n F.
 $$
With this definition, $\La^\a$ is symmetric and verifies the group property
$\La^{\a}\La^{\b} =\La^{\a+\b}$. We also have by standard complex interpolation the following inequality:
\begin{equation}\label{interpolLa}
\norm{\La^{\mu\a+(1-\mu)\b}F}_{\lpt{2}}\lesssim\norm{\La^{\a}F}^{\mu}_{\lpt{2}}\norm{\La^{\b}F}^{1-\mu}_{\lpt{2}}.
\end{equation}

We now investigate the boundedness of $\La^{-a}$ on $\lpt{p}$ spaces for $0\leq a\leq 1$. For any tensor $F$ on $\ptu$ and any $a\in\R$, integrating by parts and using the definition of $\La$, we get:
\begin{equation}\label{La1}
\begin{array}{ll}
\ds \norm{\La^a F}^2_{\lpt{2}}+\norm{\nabb\La^a F}^2_{\lpt{2}} & \ds =\int_{\ptu}\La^{a}F\c\La^{a}F\dmt+\int_{\ptu}\nabb\La^{a}F\c\nabb\La^{a}F\dmt\\
& \ds =\int_{\ptu}(1-\lap)\La^{a}F\c\La^{a}F\dmt=\int_{\ptu}\La^2\La^{a}F\c\La^{a}F\dmt\\
& \ds =\norm{\La^{a+1}F}^2_{\lpt{2}}.
\end{array}
\end{equation}
Taking $a=-1$ in \eqref{La1}, we obtain:
\begin{equation}\label{La2}
\norm{\nabb\La^{-1}F}_{\lpt{2}}\lesssim\norm{F}_{\lpt{2}}.
\end{equation}
Below, we deduce several estimates from \eqref{La2}. 
Taking the adjoint of \eqref{La2}, we obtain for any tensor $F$:
\begin{equation}\label{La3}
\norm{\La^{-1}\nabb F}_{\lpt{2}}\lesssim\norm{F}_{\lpt{2}}.
\end{equation}
Also, \eqref{eq:GNirenberg} and \eqref{La2} imply for any tensor $F$ on $\ptu$:
\begin{equation}\label{La4}
\norm{\La^{-1}F}_{\lpt{p}}\lesssim\norm{F}_{\lpt{2}}\textrm{ for all }2\leq p<+\infty.
\end{equation}
Taking the adjoint of \eqref{La4} yields:
\begin{equation}\label{La5}
\norm{\La^{-1}F}_{\lpt{2}}\lesssim\norm{F}_{\lpt{p}}\textrm{ for all }1<p\leq 2.
\end{equation}
Interpolating between the identity and $\La^{-1}$, we deduce form \eqref{La5}:
\begin{equation}\label{La6}
\norm{\La^{-a}F}_{\lpt{2}}\lesssim\norm{F}_{\lpt{p}}\textrm{ for all }0<a<1,\,\frac{2}{1+a}<p\leq 2.
\end{equation}
The proposition below completes the estimates for the heat flow recalled at the beginning of the section:
\begin{proposition}
Let $a\in\R$ and $d>0$. We have the following estimates for the nonhomogeneous heat equation:
\begin{align}
&\tau\norm{U(\tau)F}^2_{\lpt{2}}+\int_0^\tau{\tau}'\norm{\nabb U(\tau')F}^2_{\lpt{2}}d\tau'\lesssim \norm{\La^{-1}F}^2_{\lpt{2}},\label{eq:l2lambda1}\\
&\norm{\La^a V(\tau)}^2_{\lpt{2}}+\int_0^\tau\norm{\nabb\La^a V(\tau')}^2_{\lpt{2}}d\tau'\lesssim \int_0^\tau\int_{\ptu}\La^{2a} V(\tau')F(\tau')d\mu_ud\tau',\label{eq:l2heat1bis}\\
&\tau^{2d}\norm{V(\tau)}^2_{\lpt{2}}+\int_0^\tau{\tau'}^{2d}\norm{\nabb V(\tau')}^2_{\lpt{2}}d\tau'\lesssim \int_0^\tau\int_{\ptu}{\tau'}^{2d}V(\tau')F(\tau')d\mu_ud\tau'\nonumber \\
& \hspace{8.5cm} +\int_0^\tau{\tau'}^{2d-1}\norm{V(\tau')}^2_{\lpt{2}}d\tau'.\label{eq:l2heat2bis}
\end{align}
\end{proposition}

Finally, we conclude this section by recalling the sharp Bernstein inequality for scalars obtained in 
\cite{LP}. It is derived under the additional assumption that the Christoffel symbols $\Gamma^A_{BC}$ of the coordinate system \eqref{eq:coordchart} on $\ptu$ verify:
\be\lab{eq:gammaL2}
\sum_{A,B,C}\int_U|\Gamma^A_{BC}|^2 dx^1dx^2\le c^{-1},
\end{equation}
with a constant $c>0$ independent of $u$ and where $U$ is a coordinate chart. 
\begin{remark}
The existence of a covering of $\ptu$ by coordinate charts satisfying \eqref{eq:coordchart} and \eqref{eq:gammaL2} with a constant $c>0$ and the number of charts independent of $u$ will be shown in section \ref{sec:coord}.
\end{remark}
Let $0\leq\ga<1$, and let $K_\ga$ be defined by:
\begin{equation}\label{La7}
K_\ga := \norm{\La^{-\ga}K}_{\lpt{2}}.
\end{equation}
Then, we have the following sharp Bernstein inequality for any scalar function $f$ on $\ptu$,  $0\le \ga<1$,
 any $j\ge 0$, and an arbitrary $2\le p<\infty$ (see \cite{LP}):
\bea
\|P_j f\|_{\lpt{\infty}}&\lesssim & 2^j\big(1+ 2^{-\frac{j}{p}} 
\big (K_\ga^{\frac{1}{p(1-\ga)}} + K_{\ga}^{\frac{1}{2p}}\big ) +
1\big)\|f\|_{\lpt{2}}\label{eq:strongbernscalar}\\
\|P_{<0} f\|_{\lpt{\infty}}&\lesssim &  
\big (1  +K_\ga^{\frac{2}{p(1-\ga)}} + K_{\ga}^{\frac{1}{2p}}\big)
\|f\|_{\lpt{2}}.\label{eq:strong-Bern-0}
\eea
Also, the Bochner identity \eqref{sboch} together with the properties of $\La$ implies the 
following inequality (see \cite{LP}):
\bea
\int_{\ptu} |\nabb^2 f|^2&\lesssim& \int_{\ptu} |\lap f|^2  + \big (K_\ga^{\frac{2}{1-\ga}} + K_{\ga}\big )\int_{\ptu} |\nabb f|^2.\label{eq:Bochconseq}
\eea
Thus, we need to bound $K_{\ga}$ in order to be able to use \eqref{eq:strongbernscalar}, \eqref{eq:strong-Bern-0}, and \eqref{eq:Bochconseq}. For $\Re(\a)<0$, we will use the fact that for any tensor $F$ on $\ptu$:
\begin{equation}\label{La8}
\norm{\La^{-\a}F}^2_{\lpt{2}}\lesssim\norm{P_{<0}F}^2_{\lpt{2}}+\sum_{j=0}^{+\infty}2^{-2\a j}\norm{P_jF}^2_{\lpt{2}}.
\end{equation}
which follows from the methods in \cite{LP}.

\begin{remark}
The starting point for the proof of the estimates \eqref{eq:strongbernscalar}-\eqref{eq:Bochconseq} in \cite{LP} is the following estimate for the $L^\infty$ norm of any tensor $F$ on $\ptu$:
\be\lab{linftynormtensor}
\norm{F}_{L^\infty}\les\norm{\nabb^2F}^{\frac{1}{p}}_{\lpt{2}}(\norm{\nabb F}^{\frac{p-2}{p}}_{\lpt{2}}\norm{F}^{\frac{1}{p}}_{\lpt{2}}+\norm{F}^{\frac{p-1}{p}}_{\lpt{2}})+\norm{\nabb F}_{\lpt{2}}
\ee
which is valid for any $2\leq p<+\infty$. This estimate requires the assumption \eqref{eq:gammaL2}.
\end{remark}

\subsection{Hodge systems}\lab{sec:hodgesystem}

We consider the following Hodge operators acting on  $2$ surface $\ptu$:
\begin{enumerate}
\item The operator $\dcal $ takes any $1$-form $  F  $ into the pairs of
functions $(\divb  F  \,,\, \curlb  F  )$.
\item The operator $\dcall$ takes any $\ptu$ tangent symmetric, traceless
tensor $  F  $ into the $\ptu$ tangent one form $\divb  F  $.
\item The operator $\dcalll$ takes the pair of scalar functions 
$(\rho, \s)$ into the $\ptu$-tangent 1-form $-\nabb \rho+(\nabb \s)^\star$.
\item The operator $\dcallll$ takes 1-forms $F$  on $\ptu$ into  the 2-covariant, symmetric,
traceless tensors $-\half  \widehat{\mathcal{L}_F\ga }$ with $\mathcal{L}_F\ga$ the traceless part of the Lie derivative of the metric $\ga$ relative to $F$, i.e.
$$\widehat{(\mathcal{L}_F\ga)}_{ab}=\nabb_b F_a+\nabb_a F_b-(\divb F)\ga_{ab}.$$
\end{enumerate}
Observe that $\dcalll$, resp. $\dcallll$ are  the $L^2$ adjoints of
 $\dcal$, respectively $\dcall$. \\
 
 We record the following simple identities,
\bea
\dcalll\cdot\dcal&=&-\lap+K,\qquad \dcal\cdot\dcalll=-\lap,\label{eq:dcalident}\\
\dcallll\cdot\dcall&=&-\half\lap+K,\qquad \dcall\cdot\dcallll=-\half(\lap+K).\label{eq:dcallident}
\eea
Using integration by parts, this immediately yields the following identities for Hodge systems:
\begin{proposition}\lab{prop:hodgeident} 
Let $(\ptu,\ga)$ be a two dimensional manifold with Gauss curvature $K$.\\

{\bf i.)}\quad The following identity holds for vectorfields  $  F  $ 
 on $\ptu$:
\be\lab{eq:hodgeident1}
\int_{\ptu}\big(|\nabb   F  |^2+K|  F  |^2\big)=\int_{\ptu}\big( |\divb   F  |^2+|\curlb  F  |^2\big)=\int_{\ptu}|\dcal  F  |^2
\end{equation} 

{\bf ii.)}\quad The following identity holds for symmetric, traceless,
 2-tensorfields   $  F  $ 
 on $\ptu$:
\be\lab{eq:hodgeident2}
\int_{\ptu}\big(|\nabb   F  |^2+2K|  F  |^2\big)=2\int_{\ptu} |\divb   F  |^2=2\int_{\ptu} |\dcall   F  |^2
\end{equation} 

{\bf iii.)}\quad The following identity holds for pairs of functions $(\rho,\s)$
 on $\ptu$:
\be\lab{eq:hodgeident3}
\int_{\ptu}\big(|\nabb \rho|^2+|\nabb\s|^2\big)=\int_{\ptu} |-\nabb\rho+(\nabb\s)^\star |^2=\int_{\ptu}|\dcalll
(\rho,\s)|^2
\end{equation} 

{\bf iv.)}\quad  The following identity holds for vectors $  F  $ on $\ptu$,
\be\lab{eq:hodgeident3*}
\int_{\ptu} \big(|\nabb  F  |^2-K|  F  |^2\big)=2\int_{\ptu}|\dcallll   F  |^2
\end{equation}
\end{proposition}

We recall the following estimate from \cite{LP}. Let $0\leq \gamma<1$ and let $F$ a $\ptu$-tangent tensor. Then, we have
$$\int_{\ptu}K|F|^2\les \norm{\La^{-\gamma}K}_{\tx{\infty}{2}}\norm{\nabb F}^{1+\gamma}_{\lpt{2}}\norm{F}^{1-\gamma}_{\lpt{2}}.$$
Together with Proposition \ref{prop:hodgeident}, we immediately obtain the following corollary.

\begin{corollary}\lab{prop:hodgeestimateslol}
Assume that $\norm{\La^{-\gamma}K}_{\tx{\infty}{2}}\lesssim\ep$ for some $0\leq \gamma<1$. The following estimates hold on  an arbitrary  2-surface $\ptu$:\\

{\bf i.)}\quad Let a $\ptu$-tangent 1-form $H$, and let the pair of scalars $F=(\rho,\s)$ such that $\divb  H  =\rho,\,\, \curlb  H  =\s$. Then, we formally write $H=\dcal^{-1}F$, and we have the following estimate
\be\lab{eq:estimdcal-1lol}
\|\nabb\cdot\dcal^{-1}F\|_{L^2(\ptu)}\lesssim\|F\|_{L^2(\ptu)}+\ep\|\dcal^{-1}F\|_{L^2(\ptu)}
\end{equation}

{\bf ii.)}\quad Let a $\ptu$-tangent symmetric, traceless, 2-tensorfields $F$, and let the $\ptu$ tangent 1-forms $H$ such that $\divb  F =H$. Then, we formally write $F=\dcall^{-1}H$, and we have the following estimate
\be\lab{eq:estimdcall-1lol}
\|\nabb\cdot\dcall^{-1}F\|_{L^2(\ptu)}\lesssim\|F\|_{L^2(\ptu)}+\ep\|\dcall^{-1}F\|_{L^2(\ptu)}
\end{equation}

{\bf iii.)}\quad Let $(\r, \s)$ a pair of scalars on $\ptu$, and let the $\ptu$-tangent $L^2$  1-forms $F$ such that $-\nabb\rho+(\nabb\s)^\star=F$. Then, we formally write $(\r, \s)=\dcalll^{-1}F$, and we have the following estimate
\be\lab{eq:estimdcalll-1lol}
\|\nabb\cdot \dcalll^{-1}F \|_{L^2(\ptu)}\lesssim\|F\|_{L^2(\ptu)}
\end{equation}

{\bf iv.)}\quad Let a $\ptu$ tangent 1-form $H$, and let $F$ the $\ptu$-tangent 2-forms such that $\dcallll   H =F$. Then, we formally write $H=\dcallll^{-1}F$, and we have the following estimate
\be\lab{eq:estimdcallll-1lol}
\|\nabb\cdot \dcallll^{-1}F\|_{L^2(\ptu)}\lesssim\|F\|_{L^2(\ptu)}+\ep\|\dcallll^{-1}F \|_{L^2(\ptu)}.
\end{equation}
\label{prop:Hodgeestimates}
\end{corollary}

In view of \eqref{eq:estimdcal-1lol}, \eqref{eq:estimdcall-1lol}, \eqref{eq:estimdcalll-1lol} and \eqref{eq:estimdcallll-1lol}, we have schematically
\be\lab{forfaitbambin}
\|\nabb\cdot \mathcal{D}^{-1}F\|_{L^2(\ptu)}\lesssim\|F\|_{L^2(\ptu)}+\ep\|\mathcal{D}^{-1}F \|_{L^2(\ptu)},
\ee
where $\mathcal{D}=\dcal$, $\dcall$, $\dcalll$ or $\dcallll$. Note that $\ptu$ is a non compact two dimensional surface, so that $\|\mathcal{D}^{-1}F \|_{L^2(\ptu)}$ is not controlled by $\|F\|_{L^2(\ptu)}$. However, recall from Remark \ref{rem:noprobleminf} that there is a compact set $\widetilde{U}$ of $\mathcal{M}\cap\{0\leq t\leq 1\}$ of diameter of order 1, such that $(\mathcal{M}, \gg)$ is smooth, small and  asymptotically flat  outside of $\widetilde{U}$. Then, relying on the coordinate charts on $\ptu$ satisfying \eqref{eq:coordchart}, we easily obtain for any scalar $f$ on $\ptu$
$$\norm{f}_{L^2(\ptu\cap \widetilde{U})}\les\norm{\nabb f}_{\lpt{2}}.$$
Choosing $f=|F|$, we deduce for any tensor $F$
$$\norm{F}_{L^2(\ptu\cap \widetilde{U})}\les\norm{\nabb F}_{\lpt{2}}.$$
In view of \eqref{forfaitbambin}, this yields, schematically
\be\lab{forfaitbambin1}
\|\nabb\cdot \mathcal{D}^{-1}F\|_{L^2(\ptu)}+\|\mathcal{D}^{-1}F \|_{L^2(\ptu\cap \widetilde{U})}\lesssim\|F\|_{L^2(\ptu)}+\ep\|\mathcal{D}^{-1}F \|_{L^2(\ptu\setminus \ptu\cap\widetilde{U})},
\ee
where $\mathcal{D}=\dcal$, $\dcall$, $\dcalll$ or $\dcallll$. Due to the fact that $(\mathcal{M}, \gg)$ is smooth, small and  asymptotically flat  outside of $\widetilde{U}$ as recalled above, all scalars and tensors estimated in this paper will be sufficiently smooth and decaying in outside of $\widetilde{U}$ so that the last term in the right-hand side will always be harmless. For the simplicity of the exposition, we omit this term. Thus, by a slight abuse of notation, we will use the following estimate in the rest of the paper
\be\lab{eq:estimdcal-1}
\|\nabb\cdot \mathcal{D}^{-1}F\|_{L^2(\ptu)}+\|\mathcal{D}^{-1}F\|_{L^2(\ptu)}\lesssim\|F\|_{L^2(\ptu)},
\ee
where $\mathcal{D}=\dcal$, $\dcall$, $\dcalll$ or $\dcallll$.

\begin{remark}\lab{rem:Dcont}
The estimate \eqref{eq:estimdcal-1} together with the Gagliardo-Nirenberg inequality \eqref{eq:GNirenberg} yields for any $2\leq p<+\infty$:
$$\norm{\mathcal{D}^{-1}F}_{\lpt{p}}\lesssim \norm{F}_{\lpt{2}}$$
where $F$ is a $\ptu$-tangent tensor and $\mathcal{D}^{-1}$ denotes one of the operators $\mathcal{D}^{-1}_1$, $\mathcal{D}^{-1}_2$, ${}^*\mathcal{D}^{-1}_1$, ${}^*\mathcal{D}^{-1}$. We also obtain the dual inequality for any $1<p\leq 2$:
$$\norm{\mathcal{D}^{-1}F}_{\lpt{2}}\lesssim \norm{F}_{\lpt{p}}.$$
\end{remark}

The following lemma generalizes Remark \ref{rem:Dcont}.
\begin{lemma}\lab{lemma:lbt6}
For all $1<p\leq 2\leq q<+\infty$ such that $\frac{1}{p}<\frac{1}{q}+\half$, we have:
$$\norm{\mathcal{D}^{-1}F}_{\lpt{q}}\les \norm{F}_{\lpt{p}},$$
where $F$ is a $\ptu$-tangent tensor and $\mathcal{D}^{-1}$ denotes one of the operators $\mathcal{D}^{-1}_1$, $\mathcal{D}^{-1}_2$, ${}^*\mathcal{D}^{-1}_1$, ${}^*\mathcal{D}^{-1}$.
\end{lemma}

\begin{proof}
Let $F, p, q$ as in the statement of Lemma \ref{lemma:lbt6}. We decompose $\norm{\mathcal{D}^{-1}F}_{\lpt{q}}$ using the property \eqref{eq:partition} of the geometric Littlewood-Paley projections:
\be\lab{minou1}
\norm{\mathcal{D}^{-1}F}_{\lpt{q}}\les\norm{P_{<0}\mathcal{D}^{-1}F}_{\lpt{q}}+\sum_{l\geq 0}\norm{P_l\mathcal{D}^{-1}F}_{\lpt{q}}.
\ee
We focus on the second term in the right-hand side of \eqref{minou1}, the other being easier to handle. 
Since $2\leq q<+\infty$, we may use the weak Bernstein inequality for $P_l$:
\bea
\lab{minou2}\norm{P_l\mathcal{D}^{-1}F}_{\lpt{q}}&\les& 2^{l(1-\frac{2}{q})}\norm{P_l\mathcal{D}^{-1}F}_{\lpt{2}}\\
\nn&\les& 2^{l(1-\frac{2}{q})}\norm{P_l\mathcal{D}^{-1}}_{\mathcal{L}(\lpt{p},\lpt{2})}\norm{F}_{\lpt{p}}\\
\nn&\les& 2^{l(1-\frac{2}{q})}\norm{\mathcal{D}^{-1}P_l}_{\mathcal{L}(\lpt{2},\lpt{p'})}\norm{F}_{\lpt{p}}
\eea 
where $p'$ is the conjugate exponent of $p$, i.e. $\frac{1}{p}+\frac{1}{p'}=1$, and where we used the fact that $\mathcal{D}_1^{-1}$ is the adjoint of $\mathcal{D}^{-1}$. 

Next, we evaluate $\norm{\mathcal{D}_1^{-1}P_l}_{\mathcal{L}(\lpt{2},\lpt{p'})}$. Using the Gagliardo-Nirenberg inequality \eqref{eq:GNirenberg}, we have for any scalar function $f$ on $\ptu$:
\bee
\norm{\mathcal{D}^{-1}P_lf}_{\lpt{p'}}&\les& \norm{\nabb\mathcal{D}^{-1}P_lf}^{1-\frac{2}{p'}}_{\lpt{2}}\norm{\mathcal{D}^{-1}P_lf}^{\frac{2}{p'}}_{\lpt{p'}}\\
&\les& \norm{\mathcal{D}^{-1}P_l}_{\mathcal{L}(\lpt{2})}^{\frac{2}{p'}}\norm{f}_{\lpt{2}}\\
&\les& \norm{P_l\mathcal{D}^{-1}}_{\mathcal{L}(\lpt{2})}^{\frac{2}{p'}}\norm{f}_{\lpt{2}}\\
&\les &2^{-\frac{2l}{p'}}\norm{f}_{\lpt{2}},
\eee
where we used the $L^2$ boundedness and the finite band property for $P_l$, the estimate \eqref{eq:estimdcal-1} for $\mathcal{D}^{-1}$ and the estimate \eqref{eq:estimdcal-1} for $\mathcal{D}^{-1}$. This yields:
$$\norm{\mathcal{D}_1^{-1}P_l}_{\mathcal{L}(\lpt{2},\lpt{p'})}\les 2^{-\frac{2l}{p'}}$$
which together with \eqref{minou1} and \eqref{minou2} implies:
\bee
\norm{\mathcal{D}^{-1}F}_{\lpt{q}}&\les& \left(1+\sum_{l\geq 0}2^{l(1-\frac{2}{q}-\frac{2}{p'})}\right)\norm{F}_{\lpt{p}}\\
&\les& \left(1+\sum_{l\geq 0}2^{l(-1-\frac{2}{q}+\frac{2}{p})}\right)\norm{F}_{\lpt{p}}\\
&\les& \norm{F}_{\lpt{p}},
\eee
where we used the fact that $\frac{1}{p}<\frac{1}{q}+\half$ in the last inequality. This concludes the proof of Lemma \ref{lemma:lbt6}.
\end{proof}

We end this section with an algebraic expression for the commutators between $L$ and $\dcal^{-1}$, $\dcall^{-1}$, $\dcalll^{-1}$. 
\begin{lemma} Let ${\cal D}^{-1}$ be any of the operators $\dcal^{-1}$, $\dcall^{-1}$, $\dcalll^{-1}$.
Then,
\be\lab{eq:commLdcal-1}
[L,{\cal D}^{-1}]={\cal D}^{-1} [{\cal D} , L]{\cal D}^{-1}
\end{equation}
\label{le:commutationL}
\end{lemma}

\subsection{Calculus inequalities on $\H_u$}

For all integrable function $f$ on $\H_u$, the coarea formula implies:
\begin{equation}\label{coarea}
\int_{\H} fd\H =\int_{0}^{1}\int_{\ptu} f b\dmt dt.
\end{equation}
It is also well-known that for a scalar function $f$:
\begin{equation}\label{du}
\frac{d}{dt}\left(\int_{\ptu}f\dmt\right) =\int_{\ptu}\left(nL(f)+n\trc f\right)\dmt.
\end{equation}

We have the classical Sobolev inequality on $\H$:
\begin{lemma}
For any tensor $F$ on $\H$, we have:
\be\lab{sobineq}
\norm{F}_{\lh{6}}\lesssim \no(F),
\ee
and
\be\lab{sobineq1}
\norm{F}_{\tx{\infty}{4}}\lesssim \no(F).
\ee
\end{lemma}

\begin{proof}
Using \eqref{eq:isoperimetric}, we have:
\bee
\norm{F(t,.)}^6_{\lpt{6}}&=&\norm{|F(t,.)|^3}^2_{\lpt{2}}\\
&\les&  \norm{\nabb F(t,.)\c F(t,.)|F(t,.)|}^2_{\lpt{1}}\\
&\les &\norm{\nabb F(t,.)}^2_{\lpt{2}}\norm{F(t,.)}_{\lpt{4}}^4,
\eee
which yields:
\be\lab{misosoup}
\norm{F}_{\lh{6}}\les \norm{\nabb F}^{\frac{1}{3}}_{\lh{2}}\norm{F}^{\frac{2}{3}}_{\tx{\infty}{4}}.
\ee

Using \eqref{du} and \eqref{sobineq}, we have:
\begin{equation}\label{p1e1}
\begin{array}{lll}
\ds\norm{F(t,.)}_{\lpt{4}}^4&=&\norm{F(0,.)}_{\lpo{4}}^4\\
& &\ds +4\int_0^t\int_{P_{\tau,u}}n\dd_LF(\tau,x')\c F(\tau,x')|F(\tau,x')|^2d\tau d\mu_{\tau,u}\\
& &\ds+\int_0^t\int_{P_{\tau,u}}\trc |F(\tau,x')|^4d\tau d\mu_{\tau,u}\\
&\lesssim & \ds\norm{F(0,.)}_{\lpo{4}}^4+\norm{\dd_LF}_{\lh{2}}\norm{F}^3_{\lh{6}}\\
& & \ds +\norm{\trc}_{\lh{\infty}}\norm{F}^4_{\lh{4}}\\
&\lesssim & \ds\norm{F(0,.)}_{\lh{4}}^4+\no(F)^4+\norm{F}^4_{\lh{6}}.
\end{array}
\end{equation}
Replacing $F$ with $\varphi(t)F$ where $\varphi$ is a smooth function such that $\varphi(0)=1$ and $\varphi(1)=0$, and proceeding as in \eqref{p1e1}, we obtain:
\begin{displaymath}
\begin{array}{lll}
\ds\norm{F(0,.)}_{\lpo{4}}^4&=&\ds -4\int_0^1\int_{P_{\tau,u}}n\varphi(\tau)^4\dd_LF(\tau,x')\c F(\tau,x')|F(\tau,x')|^2d\tau d\mu_{\tau,u}\\
& & \ds -4\int_0^1\int_{P_{\tau,u}}\varphi'(\tau)\varphi(\tau)^3|F(\tau,x')|^4d\tau d\mu_{\tau,u}\\
& &\ds -\int_0^t\int_{P_{\tau,u}}\trc \varphi(\tau)^4|F(\tau,x')|^4d\tau d\mu_{\tau,u}\\
&\lesssim & \ds\no(F)^4+\norm{F}^4_{\lh{6}},
\end{array}
\end{displaymath}
which together with \eqref{p1e1} yields:
$$\norm{F(t,.)}_{\lpt{4}}\lesssim \no(F)+\norm{F}_{\lh{6}}.$$
Taking the supremum in $t$ yields 
\begin{equation}\label{p1e3}
\norm{F}_{\tx{\infty}{4}}\lesssim \no(F)+\norm{F}_{\lh{6}}.
\end{equation}
Finally, \eqref{misosoup} and \eqref{p1e3} imply \eqref{sobineq} and \eqref{sobineq1}. This concludes the proof.
\end{proof}

\begin{lemma}
For any tensor $F$:
\be\lab{sobineq2}
\norm{F}_{\tx{\infty}{2}}\lesssim \norm{\dd_LF}_{\lh{2}}^{\half}\norm{F}_{\lh{2}}^{\half}+\norm{F}_{\lh{2}}.
\ee
Furthermore, if $F(0,.)$ belongs to $L^2(\pou)$, we have:
\be\lab{sobineq2bis}
\norm{F}_{\tx{\infty}{2}}\lesssim \norm{F(0,.)}_{L^2(\pou)}+\norm{\dd_LF}_{\lh{2}}.
\ee
\end{lemma}

\begin{proof}
Using \eqref{du}, we have:
\be\lab{p2e1}
\begin{array}{lll}
\norm{F(t,.)}^2_{\lpt{2}}&=&\norm{F(0,.)}_{\lpo{2}}^2\\
& &\ds +2\int_0^t\int_{P_{\tau,u}}n\dd_LF(\tau,x')\c F(\tau,x')d\tau d\mu_{\tau,u}\\
& &\ds+\int_0^t\int_{P_{\tau,u}}\trc |F(\tau,x')|^2d\tau d\mu_{\tau,u}\\
&\lesssim & \ds\norm{F(0,.)}_{\lpo{2}}^2+\norm{\dd_LF}_{\lh{2}}\norm{F}_{\lh{2}}\\
&& +\norm{\trc}_{\lh{\infty}}\norm{F}^2_{\lh{2}}\\
&\lesssim & \ds\norm{F(0,.)}_{\lpo{2}}^2+\norm{\dd_LF}_{\lh{2}}\norm{F}_{\lh{2}}+\norm{F}^2_{\lh{2}}.
\end{array}
\ee
Replacing $F$ with $\varphi(t)F$ where $\varphi$ is a smooth function such that $\varphi(0)=1$ and $\varphi(1)=0$, and proceeding as in \eqref{p2e1}, we obtain:
\begin{displaymath}
\begin{array}{lll}
\norm{F(0,.)}^2_{\lpo{2}}&=&\ds -2\int_0^1\int_{P_{\tau,u}}\varphi(\tau)^2n\dd_LF(\tau,x')\c F(\tau,x')d\tau d\mu_{\tau,u}\\
& &\ds-2\int_0^1\int_{P_{\tau,u}} \varphi(\tau)'\varphi(\tau)|F(\tau,x')|^2d\tau d\mu_{\tau,u}\\
& &\ds-\int_0^1\int_{P_{\tau,u}}\trc \varphi(\tau)^2|F(\tau,x')|^2d\tau d\mu_{\tau,u}\\
&\lesssim & \ds\norm{\dd_LF}_{\lh{2}}\norm{F}_{\lh{2}}+\norm{F}^2_{\lh{2}},
\end{array}
\end{displaymath}
which together with \eqref{p2e1} yields:
\begin{equation}
\norm{F(t,.)}^2_{\lpt{2}}\lesssim \norm{\dd_LF}_{\lh{2}}\norm{F}_{\lh{2}}+\norm{F}^2_{\lh{2}}.
\end{equation}
Taking the supremum in $t$ yields \eqref{sobineq2}. 

To obtain \eqref{sobineq2bis}, we combine \eqref{p2e1} with Gronwall's lemma. This concludes the  proof.
\end{proof}

The following lemma will be useful to estimate the various transport equations arising in the null structure equations. Its proof is immediate.
\begin{lemma}
Let $W$ and $F$ two $\ptu$-tangent tensors such that $\ddb_LW=F$. Then, for any $p\geq 1$, we have:
\be\lab{estimtransport1}
\norm{W}_{\xt{p}{\infty}}\lesssim \norm{W(0)}_{L^p(\pou)}+\norm{F}_{\xt{p}{1}}.
\ee
\end{lemma}

\subsection{Calculus inequalities on $\Sit$}

Recall that $g$ is the metric induced by $\gg$ on $\Sit$.  
A coordinate chart $U\subset \Sit$  with coordinates $x=(x_1, x_2,x_3)$ is admissible if,
relative to  these coordinates, there exists  a constant $c>0$ such that,
\be\lab{eq:coordchartsit00}
c^{-1}|\xi|^2\le g_{ij}(p)\xi^i\xi^j\le c|\xi|^2, \qquad \mbox{uniformly for  all }
\,\, p\in U.
\ee
We assume that $\Sit$ can be covered by a global admissible coordinates system,
 i.e., a chart satisfying the conditions \eqref{eq:coordchartsit00} with $U=\Sit$. 
 Furthermore, we assume that the 
constant $c$ in \eqref{eq:coordchartsit00} is independent of $t$. 
\begin{remark}
The existence of a global coordinate system $\Sit$ satisfying \eqref{eq:coordchartsit00} with a constant $c>0$ independent of $t$ will be shown in section \ref{sec:globalcoordsit}.
\end{remark}

\begin{lemma}\lab{sitl1}
Let $f$ a real scalar function on $\Sit$. Then:
\be\lab{eq:sitl1}
\norm{f}_{L^{\frac{3}{2}}(\Sit)}\lesssim \norm{\nabla f}_{L^1(\Sit)}.
\ee 
\end{lemma}

\begin{proof}
We may assume that $f$ has compact support in $\Sit$. In the global coordinate system $x=(x_1,x_2,x_3)$ on $\Sit$ satisfying \eqref{eq:coordchartsit00}, we have:
\bee
|f(x_1,x_2,x_3)|^{\frac{3}{2}}=\left|\int_{-\infty}^{x_1}\partial_1f(y,x_2,x_3)dy\int_{-\infty}^{x_2}\partial_2f(x_1,y,x_3)dy\int_{-\infty}^{x_3}\partial_3f(x_1,x_2,y)dy\right|^{\frac{1}{2}}\\
\lesssim\left(\int_{\mathbb{R}}|\partial_1f(y,x_2,x_3)|dy\right)^{\frac{1}{2}}\left(\int_{\mathbb{R}}|\partial_2f(x_1,y,x_3)|dy\right)^{\frac{1}{2}}\left(\int_{\mathbb{R}}|\partial_3f(x_1,x_2,y)|dy\right)^{\frac{1}{2}}.
\eee
Hence,
\bee
&&\int_{\mathbb{R}^3}|f(x_1,x_2,x_3)|^{\frac{3}{2}}dx_1dx_2dx_3\\
&\lesssim&\left(\int_{\mathbb{R}^3}|\partial_1f(x_1,x_2,x_3)|dx_1dx_2dx_3\right)^{\frac{1}{2}}\left(\int_{\mathbb{R}^3}|\partial_2f(x_1,x_2,x_3)|dx_1dx_2dx_3\right)^{\frac{1}{2}}\\
&&\left(\int_{\mathbb{R}^3}|\partial_3f(x_1,x_2,x_3)|dx_1dx_2dx_3\right)^{\frac{1}{2}}\\
&\lesssim& \left(\int_{\mathbb{R}^3}|\nabla f(x_1,x_2,x_3)|dx_1dx_2dx_3\right)^{\frac{3}{2}}.
\eee
Now in view of the bootstrap assumption \eqref{boot1} \eqref{boot4}, and the coordinates system properties \eqref{eq:volumega} and  \eqref{eq:coordchartsit0},  we have $\frac{1}{5}\leq\sqrt{|g_t|}\leq 5$ which together with the previous estimate yields:
$$\left(\int_{\mathbb{R}^3}|f(x)|^{\frac{3}{2}}\sqrt{|g_t|}dx_1dx_2dx_3\right)^{\frac{2}{3}}\lesssim \int_{\mathbb{R}^3}|\nabla f(x)|\sqrt{|g_t|}dx_1dx_2dx_3$$
as desired.
\end{proof}

As a corollary of the estimate \eqref{eq:sitl1}, we may derive the following Sobolev embeddings.
\begin{corollary}
Given an arbitrary tensorfield $F$ on $\Sit$, we have:
\be\lab{eq:gnirenbergsit}
\norm{F}_{L^{3}(\Sit)}\lesssim \norm{\nabla F}_{L^{\frac{3}{2}}(\Sit)}
\ee 
and
\be\lab{sobineqsit}
\norm{F}_{L^{6}(\Sit)}\lesssim \norm{\nabla F}_{L^{2}(\Sit)}.
\ee 
\end{corollary}

\begin{proof}
We use \eqref{eq:sitl1} with $f=|F|^2$:
$$\norm{F}^2_{L^{3}(\Sit)}=\norm{|F|^2}_{L^{\frac{3}{2}}(\Sit)}\lesssim \norm{F\nabla F}_{L^{1}(\Sit)}\lesssim \norm{\nabla F}_{L^{\frac{3}{2}}(\Sit)}\norm{F}_{L^{3}(\Sit)}$$
which yields \eqref{eq:gnirenbergsit}. To obtain \eqref{sobineqsit}, we use \eqref{eq:sitl1} with $f=|F|^4$:
$$\norm{F}^4_{L^{6}(\Sit)}=\norm{|F|^4}_{L^{\frac{3}{2}}(\Sit)}\lesssim \norm{|F|^2F\nabla F}_{L^{1}(\Sit)}\lesssim \norm{\nabla F}_{L^{2}(\Sit)}\norm{F}^3_{L^{6}(\Sit)}$$
which yields \eqref{sobineqsit}.
\end{proof}

As a corollary of \eqref{eq:sitl1}, it is classical to derive the following inequality (for a proof, see for example \cite{GT} page 157):
\begin{corollary}
\be\lab{sobinftysit}
\norm{F}_{L^\infty(\Sit)}\lesssim \norm{\nabla F}_{L^p(\Sit)}+\norm{F}_{L^p(\Sit)},
\ee
where $p$ is any real number $p>3$.
\end{corollary}

As a corollary of \eqref{sobineqsit} and \eqref{sobinftysit}, we immediately obtain:
\be\lab{sobinftysitbis}
\norm{F}_{L^\infty(\Sit)}\lesssim \norm{\nabla^2F}_{L^2(\Sit)}+\norm{F}_{L^2(\Sit)}.
\ee

\begin{lemma}
For any tensor $F$ on $\mathcal{M}$:
\be\lab{hehehehe0bis}
\norm{F}_{L^\infty_u\lpt{2}}\lesssim \norm{\nabla F}_{\lsit{\infty}{2}}+\norm{F}_{\lsit{\infty}{2}}.
\ee
and
\be\lab{hehehehe0}
\norm{F}_{L^{\infty}_u\lpt{2}}\lesssim \norm{\nabla F}_{\lsit{\infty}{\frac{3}{2}}}+\norm{F}_{\lsit{\infty}{2}}.
\ee
\end{lemma}

\begin{proof}
We first recall the analogous formula to \eqref{coarea} \eqref{du}. For all integrable functions on $\Sit$, the coarea formula implies:
\be\lab{coareasit}
\int_{\Sit}fd\Sit =\int_u\int_{\ptu}fb\dmt du
\ee
Also, we have for all integrable scalar functions $f$:
\be\lab{dusit}
\frac{d}{du}\left(\int_{\ptu}f\dmt \right)=\int_{\ptu}b(\nabla_Nf+\textrm{tr}\theta f)\dmt
\ee
where $\theta$ is the second fundamental form of $\ptu$ in $\Sit$, i.e. $\theta_{ij}=\nabla_iN_j$. Note that from the definition of $k$, $\chi$ and $\theta$, we have:
\be\lab{hehehehe1}
\chi_{AB}=<\dd_AL,e_B>=<\nabla_AT,e_B>+<\nabla_AN,e_B>=-k_{AB}+\theta_{AB}.
\ee
The proof of \eqref{hehehehe0bis} is easier, so we focus on \eqref{hehehehe0}. Using \eqref{coareasit}-\eqref{hehehehe1}, we obtain:
\bea
\nn |F|^2_{L^\infty_uL^2_{x'}}&\lesssim& \int_u\int_{\ptu}(2F\c \nabla_NF+\textrm{tr}\theta|F|^2)\dmt dx'\\
\nn &\lesssim& |\nabla_N F|_{\lsit{\infty}{\frac{3}{2}}}|F|_{\lsit{\infty}{3}}+|\trc|_{L^\infty}|F|^2_{\lsit{\infty}{2}}\\
\nn && +|\textrm{tr}k|_{\lsit{\infty}{6}}|F|_{\lsit{\infty}{3}}|F|_{\lsit{\infty}{2}}\\
\lab{hehehehe2} &\lesssim & |\nabla F|^2_{\lsit{\infty}{\frac{3}{2}}}+|F|^2_{\lsit{\infty}{2}} 
\eea
where we have used the bootstrap assumptions \eqref{boot3} \eqref{boot4} and the Gagliardo-Nirenberg inequality \eqref{eq:gnirenbergsit} in the last inequality. Since the order in which we take the supremum over $t$ and $u$ does not matter, we obtain \eqref{hehehehe0} by taking the supremum over $t$ in \eqref{hehehehe2}.
\end{proof}

We have the following corollary of the estimate \eqref{hehehehe0}:
\begin{corollary}
For any tensor $F$ on $\mathcal{M}$, we have
\be\lab{boarding}
\norm{F}_{\tx{\infty}{4}}\les \norm{\nabla F}_{L^\infty_t L^2(\Sigma_t)}+\norm{F}_{L^\infty_t L^2(\Sigma_t)}.
\ee
\end{corollary}

\begin{proof}
Using \eqref{hehehehe0} with $F$ replaced by $|F|^2$, we obtain 
\bee
\norm{F}_{L^{\infty}_u\lpt{4}}^2&\lesssim& \norm{F\c\nabla F}_{\lsit{\infty}{\frac{3}{2}}}+\norm{F}^2_{\lsit{\infty}{4}}\\
&\les& \norm{F}_{\lsit{\infty}{6}}\norm{\nabla F}_{\lsit{\infty}{2}}+\norm{F}_{\lsit{\infty}{6}}^\frac{3}{4}\norm{F}_{\lsit{\infty}{2}}^\frac{1}{4}\\
&\les& \norm{\nabla F}_{L^\infty_t L^2(\Sigma_t)}+\norm{F}_{L^\infty_t L^2(\Sigma_t)},
\eee
where we used in the last inequality the Sobolev embedding \eqref{sobineqsit}. This concludes the proof of the corollary.
\end{proof}

\begin{proposition}
For any tensor $F$ on $\Sit$, we have the following inequality:
\be\lab{prop:bochsit}
\norm{\nabla^2F}_{L^2(\Sit)}\lesssim \norm{\Delta F}_{L^2(\Sit)}+\norm{F}_{L^2(\Sit)}.
\ee
\end{proposition}

\begin{proof}
We recall the Bochner identity on the 3 dimensional manifold $\Sit$ for a tensor $F$:
\bea
\lab{prop:bochsit1}\int_{\Sit}|\nabla^2F|^2d\Sit &=&\int_{\Sit}|\Delta F|^2d\Sit -\int_{\Sit}(R_t)_{ij}\nabla_iF_l\nabla_jF_ld\Sit\\
&&\nn+\int_{\Sit}(R_t)_{ijlm}\nabla_mF_j\nabla_lF_id\Sit-\int_{\Sit}(R_t)_{ijlm}(R_t)_{inml}F_mF_nd\Sit
\eea
where $R_t$ is the curvature tensor of the induced metric on $\Sit$. The bound \eqref{curvriccisit1} on $R_t$ together with the Sobolev inequality \eqref{sobineqsit} and \eqref{prop:bochsit1} implies:
\bee
\int_{\Sit}|\nabla^2F|^2d\Sit &\lesssim& \norm{\Delta f}_{L^2(\Sit)}^2 +\norm{R_t}_{L^2(\Sit)}\norm{\nabla  F}_{L^4(\Sit)}^2+\norm{R_t}^2_{L^2(\Sit)}\norm{F}_{L^{\infty}(\Sit)}^2\\
&\lesssim& \norm{\Delta F}_{L^2(\Sit)}^2 +\ep\norm{\nabla F}^{\frac{1}{2}}_{L^2(\Sit)}\norm{\nabla F}_{L^6(\Sit)}^{\frac{3}{2}}+\ep^2(\norm{\nabla^2F}^2_{L^2(\Sit)}+\norm{F}_{L^2(\Sit)}^2)\\
&\lesssim& \norm{\Delta F}_{L^2(\Sit)}^2 +\norm{\nabla F}^{\frac{1}{2}}_{L^2(\Sit)}\norm{\nabla^2F}_{L^2(\Sit)}^{\frac{3}{2}}
\eee
which yields \eqref{prop:bochsit}.
\end{proof}

\begin{proposition}
For any tensor $F$ on $\Sit$, we have the following inequality:
\be\lab{prop:linftysitptu}
\norm{F}_{L^\infty(\Sit)}\les\norm{F}_{L^2(\Sit)}+\norm{\nabla F}_{L^2(\Sit)}+\norm{\nabb\nabla F}_{L^2(\Sit)}.
\ee
\end{proposition}

\begin{proof}
Using \eqref{sobinftyptu} with $p=4$, we have:
\be\lab{eul1}
\norm{F}_{L^\infty(\Sit)}\les \norm{\nabb F}_{L^\infty_u\lpt{4}}+\norm{F}_{L^\infty_u\lpt{4}}.
\ee
Pick any real number $u_0$. Now, using the coarea formula \eqref{coareasit} and \eqref{dusit}, as well as the Sobolev embedding \eqref{sobineqsit}, we have:
\bea
\lab{eul2}&&\norm{\nabb F(u,.)}^4_{L^\infty_u\lpt{4}}+\norm{F(u,.)}^4_{L^\infty_u\lpt{4}}\\
\nn&\les& \norm{\nabb F(u_0,.)}^4_{L^\infty_u\lpt{4}}+\norm{F(u_0,.)}^4_{L^\infty_u\lpt{4}}+
\int_{\Sit}\nabla_N\nabb F\c \nabb F|\nabb F|^2\\
\nn&&+\int_{\Sit}|\nabb F|^4(\trt+b^{-1}\nabla_Nb)+\int_{\Sit}\nabla_NF\c F|F|^2+\int_{\Sit}|F|^4(\trt+b^{-1}\nabla_Nb)\\
\nn&\les& \norm{\nabb F(u_0,.)}^4_{L^\infty_u\lpt{4}}+\norm{F(u_0,.)}^4_{L^\infty_u\lpt{4}}+
\norm{\nabb_N\nabb F}_{L^2(\Sit)}\norm{\nabb F}^3_{L^6(\Sit)}\\
\nn&&+\norm{\nabb_NF}_{L^2(\Sit)}\norm{F}^3_{L^6(\Sit)}+(\norm{\nabb F}^4_{L^6(\Sit)}+\norm{F}_{L^6(\Sit)})(\norm{\trt}^2_{L^3(\Sit)}+\norm{b^{-1}\nabla_Nb}_{L^3(\Sit)})\\
\nn&\les& \norm{\nabb F(u_0,.)}^4_{L^\infty_u\lpt{4}}+\norm{F(u_0,.)}^4_{L^\infty_u\lpt{4}}+
\norm{[\nabb_N,\nabb]F}^4_{L^2(\Sit)}+\norm{\nabb\nabla F}^4_{L^2(\Sit)}\\
\nn&&+\norm{\nabla F}_{L^2(\Sit)},
\eea
where we used in the last inequality the estimates \eqref{estk}-\eqref{esttrc} for $b$ and $\trt$.

In view of \eqref{eul1} and \eqref{eul2}, we need to estimate $\norm{[\nabla_N,\nabb]F}_{L^2(\Sit)}$. 
Using the commutator formula \eqref{comm4}, we have:
\bea
\nn\norm{[\nabla_N,\nabb]F}_{L^2(\Sit)}&\les&\norm{\th}_{L^\infty_u\lpt{4}}\norm{\nabb F}_{L^2_u\lpt{4}}
 +\norm{b^{-1} \nabb b}_{L^\infty_u\lpt{4}}\norm{\nabb_{N}F}_{L^2_u\lpt{4}}\\ 
\nn&+& (\norm{\chi}_{L^4(\Sit)}(\norm{\kepb}_{L^4(\Sit)}+\norm{\xib}_{L^4(\Sit)})+\norm{\chb}_{L^4(\Sit)}\norm{\z}_{L^4(\Sit)}\\
\nn&&+\norm{\b}_{L^2(\Sit)}+\norm{\bb}_{L^2(\Sit)})\norm{F}_{L^\infty(\Sit)}\\
\nn&\les& D\ep\norm{\nabb^2F}_{L^2(\Sit)}+D\ep\norm{\nabb\nabla_NF}_{L^2(\Sit)}+D\ep\norm{F}_{L^\infty(\Sit)}\\
\lab{eul3}&\les& D\ep\norm{\nabb\nabla F}_{L^2(\Sit)}+D\ep\norm{F}_{L^\infty(\Sit)},
\eea
where we have used the curvature bound \eqref{curvflux1} for $\b$ and $\bb$, the bootstrap assumptions \eqref{boot1}-\eqref{boot6} for $b, \th, \chi, \chb, \z$ and $\xib$, and the estimate:
$$\norm{H}_{L^\infty_u\lpt{4}}=\norm{H}_{\tx{\infty}{4}}\les \no(H),$$
which is valid for any tensor $H$ and follows from \eqref{p1e3}.

Now, in view of \eqref{eul1}-\eqref{eul3}, we have for any real number $u_0$:
\bee
&&\norm{F}_{L^\infty(\Sit)}+ \norm{\nabb F}_{L^\infty_u\lpt{4}}+\norm{F}_{L^\infty_u\lpt{4}}\\
\nn&\les& \norm{\nabb F(u_0,.)}^4_{L^\infty_u\lpt{4}}+\norm{F(u_0,.)}^4_{L^\infty_u\lpt{4}}+ \norm{\nabb\nabla F}_{L^2(\Sit)}+\norm{\nabla F}_{L^2(\Sit)}\\
\nn&&+D\ep\norm{F}_{L^\infty(\Sit)},
\eee
which yields:
\bea
\lab{eul4}&&\norm{F}_{L^\infty(\Sit)}+ \norm{\nabb F}_{L^\infty_u\lpt{4}}+\norm{F}_{L^\infty_u\lpt{4}}\\
\nn&\les& \norm{\nabb F(u_0,.)}^4_{L^\infty_u\lpt{4}}+\norm{F(u_0,.)}^4_{L^\infty_u\lpt{4}}+ \norm{\nabb\nabla F}_{L^2(\Sit)}+\norm{\nabla F}_{L^2(\Sit)}.
\eea
Let $\varphi$ a smooth compactly supported scalar function on $\Sit$. Applying \eqref{eul4} respectively to $\varphi F$ with $u_0$ outside of the support of $\varphi$, and then to $F$ with $u_0$ inside the support of $\varphi$ finally yields \eqref{prop:linftysitptu}. This concludes the proof of the proposition.
\end{proof}

For the following proposition, we assume that for each $\d>0$, there exists a constant $C(\d)>0$ and a finite covering of $\Sit$ by charts $U$ with coordinates systems relative to which we have
\be\label{coorharmth1bis00}
(1+\d)^{-1}|\xi|^2\leq g_{ij}(p)\xi^i\xi^j\leq (1+\d)|\xi|^2,\,p\in U
\ee
and
\be\label{coorharmth2bis00}
\int_{U}|\partial^2g_{ij}|^2\sqrt{|g|}dx\leq C(\d).
\ee

\begin{remark}
The existence of a finite covering of $\Sit$ by coordinates systems relative to which we have \eqref{coorharmth1bis00} \eqref{coorharmth2bis00} with $C(\d)$ and the number of charts being independent of $t$ will be shown in section \ref{sec:coordharm}.
\end{remark}

\begin{proposition}
Assume that for each $\d>0$, there is a finite covering of $\Sit$ by coordinates systems relative to which we have \eqref{coorharmth1bis00} \eqref{coorharmth2bis00}. For an arbitrary tensorfield $F$ on $\Sit$, we have the following inequality:
\be\lab{prop:bochsit2}
\norm{\nabla^2F}_{L^{\frac{3}{2}}(\Sit)}\lesssim \norm{\Delta F}_{L^{\frac{3}{2}}(\Sit)}+\norm{\nabla F}_{L^2(\Sit)}.
\ee
\end{proposition}

\begin{proof}
\eqref{prop:bochsit2} may be reduced by partition of unity to the case where $F$ has compact support in a coordinate chart $U$. Let $x=(x_1,x_2,x_3)$ a coordinate system on $U$ satisfying \eqref{coorharmth1bis00} \eqref{coorharmth2bis00}. We have:
\bee
&&\norm{\nabla^2F-\partial^2F}_{L^{\frac{3}{2}}(U)}\\
\nn &\lesssim&\norm{(g_{ij}-\d_{ij})\partial^2F}_{L^{\frac{3}{2}}(U)}+\norm{\partial g\partial F}_{L^{\frac{3}{2}}(U)}+\norm{\partial^2g F}_{L^{\frac{3}{2}}(U)}\\
\nn&\lesssim&\norm{g_{ij}-\d_{ij}}_{L^{\infty}(U)}\norm{\partial^2F}_{L^{\frac{3}{2}}(U)}+\norm{\partial g}_{L^{6}(U)}\norm{\nabla F}_{L^{2}(U)}\\
&&\nn +\norm{\partial^2g}_{L^{2}(U)}(1+\norm{\partial g}_{L^{3}(U)})\norm{F}_{L^{6}(U)}\\
\nn&\lesssim&\d\norm{\partial^2F}_{L^{\frac{3}{2}}(U)}+C(\d)\norm{\nabla F}_{L^{2}(U)},
\eee
where we have used the Sobolev embedding \eqref{sobineqsit} in the last inequality. Thus, we now fix $\d>0$ small enough such that for a constant $C>0$, we have:
\be\lab{prop:bochsit3}
\norm{\nabla^2F-\partial^2F}_{L^{\frac{3}{2}}(U)}\leq\frac{1}{2}\norm{\partial^2F}_{L^{\frac{3}{2}}(U)}+C\norm{\nabla F}_{L^{2}(U)}.
\ee
Note that $C=C(\d)>0$ is now a fixed number. Similarly, we also have:
\be\lab{prop:bochsit4}
\normm{\Delta F-\sum_{j=1}^3\partial^2_jF}_{L^{\frac{3}{2}}(U)}\leq\frac{1}{2}\normm{\sum_{j=1}^3\partial^2_jF}_{L^{\frac{3}{2}}(U)}+C\norm{\nabla F}_{L^{2}(U)},
\ee
where $\sum_{j=1}^3\partial^2_j$ is the usual Laplacian in $\mathbb{R}^3$. Now, from usual Calderon-Zygmund theory, we have:
$$\norm{\partial^2F}_{L^{\frac{3}{2}}(U)}\lesssim\normm{\sum_{j=1}^3\partial^2_jF}_{L^{\frac{3}{2}}(U)}$$
which together with \eqref{prop:bochsit3} and \eqref{prop:bochsit4} yields \eqref{prop:bochsit2}.
\end{proof}

Finally, we have the following useful commutation formula for any scalar function $f$ on $\mathcal{M}$:
\be\lab{commsit}
 [\Delta,\dd_T]f=-2k\nabla^2f+2n^{-1}\nabla n\nabla T(f)+n^{-1}\Delta n T(f) -\nabla k\nabla f-2n^{-1}k\nabla n\nabla f
\ee
where we used the fact that we are on a maximal foliation (see \eqref{maxfoliation}), so that the term Tr$(k)\Delta f$ vanishes. We also used the fact that the Einstein equations \eqref{eq:I1} are satisfied, so that the term of type $\rr \nabla f$ vanishes as well. We also provide commutation formulas with tensors. Let $\Pi_{\und{A}}$ be an m-covariant tensor tangent to $\Sigma_t$. Then, we have:
\bea
\lab{commsitbis}\nabla_j \dd_T \Pi_{\und{A}} - \dd_T\nabla_j \Pi_{\und{A}} &=&
k_{jl} \nabla_l \Pi_{\und{A}} +n^{-1} \nabla_j n \dd_T \Pi_{\und{A}}+ \sum_i (n^{-1}k_{A_i j} \nabla_ln \\
\nn&& -n^{-1}k_{jl}\nabla_{A_i}n +\rr_{TA_i}(g_t)_j^l-\rr_{Tl}(g_t)_j^{A_i}) \Pi_{A_1..\Check{l}..A_m}.\nn
\eea
For some applications we have in mind, we would like to get rid of the term containing a $\dd_T$ derivative in the right-hand side of \eqref{commsitbis}. This is achieved by considering the commutator $[\nabla,\dd_{nT}]$ instead of $[\nabla,\dd_T]$:
\bea
\lab{commsitter}\nabla_j \dd_{nT} \Pi_{\und{A}} - \dd_{nT}\nabla_j \Pi_{\und{A}} &=&
nk_{jl} \nabla_l \Pi_{\und{A}} \\
 &+& \sum_i (k_{A_i j} \nabla_ln -
k_{jl}\nabla_{A_i}n +n\rr_{TA_i}(g_t)_j^l-n\rr_{Tl}(g_t)_j^{A_i}) \Pi_{A_1..\Check{l}..A_m}\nn.
\eea

\subsection{Geometric Littlewood-Paley theory on $\Sit$}\lab{sec:LPSit}

\subsubsection{The Gagliardo-Nirenberg inequality on $\Sigma_t$}

We first consider the case of $L^p(\Sigma_t)$ with $2\leq p\leq 6$. Using the Sobolev inequality \eqref{sobineqsit} and interpolation implies for any tensor $F$ on $\Sigma_t$
\be\lab{GNsit}
\norm{F}_{L^p(\Sigma_t)}\les \norm{\nabla F}^{3(\frac{1}{2}-\frac{1}{p})}_{L^2(\Sigma_t)} \norm{F}^{-\frac{1}{2}+\frac{3}{p}}_{L^2(\Sigma_t)}\,\,\forall 2\leq p\leq 6.
\ee

Next, we derive the following analog of Lemma \ref{sitl1}
\begin{lemma}\lab{sitl1bis}
Assume that for $\d=\frac{1}{2}$, there is a finite covering of $\Sit$ by coordinates systems relative to which we have \eqref{coorharmth1bis00} \eqref{coorharmth2bis00}. Let $f$ a real scalar function on $\Sit$. Then:
\be\lab{eq:sitl1bis}
\norm{f}_{L^\infty(\Sit)}\lesssim \norm{\nabla^2f}^{\frac{1}{2}}_{L^2(\Sit)}\norm{\nabla f}^{\frac{1}{2}}_{L^2(\Sit)}+\norm{\nabla f}_{L^2(\Sigma_t)}.
\ee 
\end{lemma}

\begin{proof}
The statement may be reduced to the case where $f$ has compact support in an admissible local chart $U$ of $\Sigma_t$ satisfying \eqref{coorharmth1bis00} \eqref{coorharmth2bis00} with $\d=\frac{1}{2}$. Let $x=(x_1,x_2,x_3)$ denote the corresponding coordinate system. We start by proving the following estimate on $\R^3$
\be\lab{patriots1}
\norm{f^2}_{L^\infty(\R^3)}\les \norm{\pr f}_{L^2(\R^3)}\norm{\pr^2f}_{L^2(\R^3)}.
\ee
To this end, we introduce a standard Littlewood-Paley decomposition on $\R^3$. Let $\varphi$ a positive function in $C^{\infty}_0(\mathbb{R}^3)$ equal to 1 for $|\xi|\leq 1/2$ and to 0 for $|\xi|\geq 1$. For all integer $p$ we define the Littlewood-Paley projection $\Delta_p$ by $\widehat{\Delta_pf}(\xi)=\psi(2^{-p}\xi)\hat{f}(\xi)$ where $\psi(\xi)=\varphi(\xi/2)-\varphi(\xi)$. We also define $\Delta_{-1}$ by $\widehat{\Delta_{-1}f}(\xi)=\varphi(\xi)\hat{f}(\xi)$. The Littlewood-Paley decomposition of $f$ is:
$$f=\displaystyle\sum_{p\geq -1}\Delta_p(f).$$
Using the Littlewood-Paley decomposition for $f^2$, we have:
\bea\lab{patriots}
\norm{f^2}_{L^\infty(\R^3)}&\les& \sum_{j\geq -1}\norm{\Delta_j(f^2)}_{L^\infty(\R^3)}\\
\nn&\les& \sum_{j\geq -1}2^{\frac{3j}{2}}\norm{\Delta_j(f^2)}_{L^2(\R^3)}\\
\nn&\les& \sum_{j, l, m\geq -1}2^{\frac{3j}{2}}\norm{\Delta_j(\Delta_lf\Delta_mf)}_{L^2(\R^3)}.
\eea
The expression being symmetric in $(l, m)$, we may assume $m\leq l$. We consider the two cases $l\leq j$ and $j<l$ separately. If $j<l$, we use the boundedness of $\Delta_j$ on $L^2(\R^3)$ and the Bernstein inequality for $\Delta_m$ to obtain
\be\lab{patriots2}
\norm{\Delta_j(\Delta_lf\Delta_mf)}_{L^2(\R^3)}\les 2^{\frac{3m}{2}}\norm{\Delta_lf}_{L^2(\R^3)}\norm{\Delta_mf}_{L^2(\R^3)}.
\ee
If $l\leq j$, we use the finite band property for $\Delta_j$, $\Delta_l$ and $\Delta_m$, and the Bernstein inequality for $\Delta_m$ to obtain
\bea\label{patriots3}
\norm{\Delta_j(\Delta_lf\Delta_mf)}_{L^2(\R^3)}&\les& 2^{-2j}\norm{\Delta(\Delta_lf\Delta_mf)}_{L^2(\R^3)}\\
\nn&\les& 2^{-2j}\norm{\Delta(\Delta_lf)\Delta_mf}_{L^2(\R^3)}+2^{-2j}\norm{\Delta_lf\Delta(\Delta_mf)}_{L^2(\R^3)}\\
\nn&&+2^{-2j}\norm{\nabla(\Delta_lf)\nabla(\Delta_mf)}_{L^2(\R^3)}\\
\nn&\les& 2^{-2j}(2^{2l+\frac{3m}{2}}+2^{l+\frac{5m}{2}}+2^{\frac{7m}{2}})\norm{\Delta_lf}_{L^2(\R^3)}\norm{\Delta_mf}_{L^2(\R^3)}\\
\nn&\les& 2^{-2j+2l+\frac{3m}{2}}\norm{\Delta_lf}_{L^2(\R^3)}\norm{\Delta_mf}_{L^2(\R^3)},
\eea
where we used the fact that $m\leq l$ in the last inequality. Now, \eqref{patriots2} and \eqref{patriots3} imply
$$2^{\frac{3j}{2}}\norm{\Delta_j(\Delta_lf\Delta_mf)}_{L^2(\R^3)}\les 2^{-\frac{|j-l|}{4}-\frac{|j-m|}{4}}(2^{2l}\norm{\Delta_lf}_{L^2(\R^3)})(2^m\norm{\Delta_mf}_{L^2(\R^3)}).$$
Together with \eqref{patriots}, we infer
$$\norm{f^2}_{L^\infty(\R^3)}\les \left(\sum_{l\geq -1}(2^{2l}\norm{\Delta_lf}_{L^2(\R^3)})^2\right)\left(\sum_{m\geq -1}(2^m\norm{\Delta_mf}_{L^2(\R^3)})^2\right)\les\norm{\pr f}_{L^2(\R^3)}\norm{\pr^2f}_{L^2(\R^3)},$$
which is \eqref{patriots1}. Now in view of the assumptions \eqref{coorharmth1bis00} \eqref{coorharmth2bis00} with $\d=\frac{1}{2}$, we have $\frac{1}{8}\leq\sqrt{|g_t|}\leq 8$ and the estimate $\norm{\Gamma}_{L^3(U)}\les 1$ where $\Gamma$ is the corresponding Christoffel symbol, which together with \eqref{patriots1} yields:
\bee
&&\norm{f}_{L^\infty(\R^3)}\\
&\lesssim& \left(\int_{\mathbb{R}^3}(|\nabla^2 f(x)|+|\Gamma(x)| |\nabla f(x)|)^2\sqrt{|g_t|}dx\right)^{\frac{1}{2}}\left(\int_{\mathbb{R}^3}|\nabla f(x)|^2\sqrt{|g_t|}dx\right)^{\frac{1}{2}}\\
&\lesssim& \left(\left(\int_{\mathbb{R}^3}|\nabla^2 f(x)|^2\sqrt{|g_t|}dx\right)^{\frac{1}{2}}+\left(\int_{\mathbb{R}^3}|\nabla f(x)|^6\sqrt{|g_t|}dx\right)^{\frac{1}{6}}\right)\left(\int_{\mathbb{R}^3}|\nabla f(x)|^2\sqrt{|g_t|}dx\right)^{\frac{1}{2}}.
\eee
Coming back to $\Sigma_t$, we obtain
$$\norm{f}_{L^\infty(\Sit)}\lesssim (\norm{\nabla^2f}^{\frac{1}{2}}_{L^2(\Sit)}+\norm{\nabla f}^{\frac{1}{2}}_{L^6(\Sit)})\norm{\nabla f}^{\frac{1}{2}}_{L^2(\Sit)}+\norm{\nabla f}_{L^2(\Sigma_t)},$$
which together with the Sobolev embedding \eqref{sobineqsit} yields \eqref{eq:sitl1bis}.
\end{proof}

Let $F$ a tensor on $\Sigma_t$. Then \eqref{eq:sitl1bis} with the choice $f=|F|^2$ yields
\bee
\norm{F}^2_{L^\infty(\Sigma_t)}&\les& \norm{F\c\nabla^2F+|\nabla F|^2}^{\frac{1}{2}}_{L^2(\Sit)}\norm{F\c\nabla F}^{\frac{1}{2}}_{L^2(\Sit)}+\norm{F\c\nabla F}_{L^2(\Sigma_t)}\\
&\les& (\norm{F}_{L^\infty(\Sigma_t)}\norm{\nabla^2F}_{L^2(\Sit)}+\norm{\nabla F}^2_{L^4(\Sigma_t)})^{\frac{1}{2}}(\norm{F}_{L^\infty(\Sigma_t)}\norm{\nabla F}_{L^2(\Sit)})^{\frac{1}{2}}\\
&&+\norm{F}_{L^\infty(\Sigma_t)}\norm{\nabla F}_{L^2(\Sit)}.
\eee
Using the Gagliardo-Nirenberg inequality \eqref{GNsit} to evaluate $\norm{\nabla F}_{L^4(\Sigma_t)}$, we deduce
\bee
\norm{F}^2_{L^\infty(\Sigma_t)}&\les& \norm{F}_{L^\infty(\Sigma_t)}\norm{\nabla^2F}_{L^2(\Sit)}^{\frac{1}{2}}\norm{\nabla F}_{L^2(\Sit)}^{\frac{1}{2}}+\norm{\nabla^2F}_{L^2(\Sigma_t)}^{\frac{3}{4}}\norm{\nabla F}_{L^2(\Sigma_t)}^{\frac{3}{4}}\norm{F}_{L^\infty(\Sigma_t)}^{\frac{1}{2}}\\
&&+\norm{F}_{L^\infty(\Sigma_t)}\norm{\nabla F}_{L^2(\Sit)}.
\eee
Thus, we finally obtain for any tensor $F$ on $\Sigma_t$
$$\norm{F}_{L^\infty(\Sigma_t)}\les \norm{\nabla^2F}_{L^2(\Sit)}^{\frac{1}{2}}\norm{\nabla F}_{L^2(\Sit)}^{\frac{1}{2}}+\norm{\nabla F}_{L^2(\Sit)}.$$
Interpolating with the Sobolev embedding \eqref{sobineqsit} on $\Sigma_t$, we finally obtain the following Gagliardo-Nirenberg inequality on $\Sigma_t$
\be\lab{GNsit1}
\norm{F}_{L^p(\Sigma_t)}\les \norm{\nabla^2F}^{\frac{1}{2}-\frac{3}{p}}_{L^2(\Sigma_t)}\norm{\nabla F}^{\frac{1}{2}+\frac{3}{p}}_{L^2(\Sigma_t)}\,\,\forall 6\leq p\leq +\infty.
\ee

\subsubsection{Heat equation on $\Sigma_t$}

In this section we study the properties of the heat  equation  for arbitrary
tensorfields $F$ on $\Sigma_t$.
$$\pr_\tau \mathcal{U}(\tau)F -\Delta \mathcal{U}(\tau) F=0, \,\, \mathcal{U}(0)F=F.$$
Observe that  the operators $\mathcal{U}(\tau)$ are
selfadjoint and form a semigroup for $\tau>0$. In other words  for all, real valued, smooth 
tensorfields $F, G$,
\be\lab{sit:eq:seladj-semigroup}
\int_{\Sigma_t}\mathcal{U}(\tau)F\c G =\int_{\Sigma_t}F\c \mathcal{U}(\tau) G, \qquad \mathcal{U}(\tau_1)\mathcal{U}(\tau_2)=\mathcal{U}(\tau_1+\tau_2).
\end{equation}

We have the following $L^2(\Sigma_t)$ estimates for the operator $\mathcal{U}(\tau)$.
\begin{align}
&\|\mathcal{U}(\tau) F\|_{L^2(\Sigma_t)}\leq \|F\|_{L^2(\Sigma_t)},\label{sit:eq:l2heat1}\\
&\|\nab \mathcal{U}(\tau) F\|_{L^2(\Sigma_t)}\leq \|\nab F\|_{L^2(\Sigma_t)}.\label{sit:eq:l2heatnab}
\end{align} 
They are obtained after multiplication of the Heat equation satisfied by $\mathcal{U}(\tau)F$ respectively with 
$\mathcal{U}(\tau)F$ and $\Delta\mathcal{U}(\tau)F$, and then integration over $\Sigma_t$. 

In the next proposition we establish a simple $L^p(\Sigma_t)$ estimate for $\mathcal{U}(\tau)$.
\begin{proposition} For every $2\leq p\le\infty$, we have
  $$ \|\mathcal{U}(\tau)F\|_{L^p(\Sigma_t)}\leq \|F\|_{L^p(\Sigma_t)}.$$
\end{proposition}
 
\begin{proof}:\quad  The proof is identical to the one in \cite{LP} on compact 2-surfaces. We reproduce it here for the convenience of the reader. We shall first prove
the Lemma for scalar functions $f$.
 We multiply the equation $\pr_\tau \mathcal{U}(\tau)f- \Delta
\mathcal{U}(\tau)f=0$ by $\big(\mathcal{U}(\tau)f\big)^{2p-1}$ and integrate by parts.
We get,
$$\frac{1}{2p}\frac{d}{d\tau} \|\mathcal{U}(\tau)f\|_{L^{2p}(\Sigma_t)}^{2p}+
(2p-1)\int |\nab \mathcal{U}(\tau) f|^2|\mathcal{U}(\tau)f|^{2p-2}=0.$$
Therefore,
$$\|\mathcal{U}(\tau) f\|_{L^{2p}(\Sigma_t)}\leq \|f\|_{L^{2p}(\Sigma_t)}.$$
The case when  $F$ is a tensorfield can be treated in the same manner 
with  multiplier   $\big(|\mathcal{U}(\tau)F|^2\big)^{p-1}\mathcal{U}(\tau) F$.
\end{proof}

\subsubsection{ Invariant Littlewood-Paley theory on $\Sigma_t$}

In this section we shall develop  an invariant, fully tensorial, Littlewood-Paley theory on $\Sigma_t$. We follow the analog construction in \cite{LP} for two dimensional compact  manifolds. Now, the results essentially rely on the properties of 
the heat flow discussed in the previous section. Since these properties are true for manifolds of arbitrary dimensions,  both compact  and noncompact, the results in \cite{LP} extend in a straightforward fashion. Thus, we recall below the main objects introduced in \cite{LP}, and we refer to \cite{LP} for the proofs. 

\begin{definition}
Consider  the class $\mathfrak{M}$ of smooth functions $m$ on $[0,\infty)$,
vanishing sufficiently fast at $\infty$,
verifying the  vanishing  moments property:
\be\lab{sit:eq:moments}
\int_0^\infty \tau^{k_1}\pr_\tau^{k_2} m(\tau) d\tau=0, \,\,\,\,
|k_1|+|k_2|\leq N 
\end{equation}

 We set,
  $m_k(\tau)=2^{2k}m(2^{2k}\tau)$ 
and  define the geometric Littlewood -Paley (LP) 
projections $Q_k$, associated to the LP- representative 
function $m\in \mathfrak{M}$, for   arbitrary tensorfields  $F$ on $\Sigma_t$
to be 
\be\lab{sit:eq:LP}Q_k F=\int_0^\infty m_k(\tau) \mathcal{U}(\tau) F d\tau
\end{equation}
Given an interval $I\subset \Bbb Z$ we define $$Q_I=\sum_{k\in I} Q_k F.$$
In particular we shall use the notation $Q_{<k}, Q_{\leq k}, Q_{>k}, Q_{\ge k}$.
\end{definition}
Observe that $Q_k$ are selfadjoint, i.e., $Q_k=Q_k^*$, in the sense,
$$<Q_kF, G>=<F,Q_k G>,$$
where, for any given $m$-tensors $F,G$ 
$$<F,G>=\int_{\Sigma_t}g^{i_1j_1}\ldots g^{i_mj_m}
F_{i_1\ldots i_m}G_{j_1\ldots j_m}d\mbox{vol}_g    $$ 
denotes the usual $L^2(\Sigma_t)$ scalar product. 

We have the following lemma (see \cite{LP} for the proof)
\begin{lemma}
If $a,b\in \mathfrak{M}$  so does $a\star b$ defined by
\be\lab{sit:eq:funnyconv}
a\star b(\tau)=\int_{0}^\tau d\tau_1 \,\,a(\tau_1)b(\tau-\tau_1).
\end{equation}
Also, $(a\star b)_k=a_k\star b_k$. In particular
if we denote by $^{(a)}Q_k$ and $^{(b)}Q_k$ the LP projections
associated to  $a,b$ then,
$$^{\,(a)}Q_k \c ^{\,(b)}Q_k=^{\,(a\star b)}Q_k$$
\label{sit:le:convolution}
\end{lemma}

Motivated by this Lemma
we define:
\begin{definition} Given a positive integer $\ell$ we 
 define the class $\mathfrak{M}_{\ell}\subset\mathfrak{M}$ of LP-
representatives 
 to consist of functions of the form 
$$\bar{m}=m\star m\star\ldots\star m=(m\star)^\ell,$$
for some $m\in \mathfrak{M}$.\label{sit:def:Ml}
\end{definition}

We have a Littlewood-Paley decomposition thanks to the following lemma (see \cite{LP} for the proof)
\begin{lemma} For any $\ell\ge 1$  there exists an element $\bar m\in \mathfrak{M}_{\ell}$ 
such that the LP-projections associated to $\bar m$ verify:
\be\lab{sit:eq:partition}
\sum_kQ_k=I.
\end{equation}
Moreover, the functions $\bar m=(\star m)^\ell$ and $m$ can 
be chosen to have compact support on the open interval $(0,\infty)$.
\end{lemma}

Finally, the following theorem summarizes the main properties of the Littlewood-Paley decompositions $Q_k$. 

\begin{theorem} The LP-projections $Q_k$ associated
to an arbitrary $m\in {\cal M}$ verify the following
 properties:

i)\quad {\sl $L^p(\Sigma_t)$-boundedness} \quad For any $1\le
p\leq \infty$, and any interval $I\subset \Bbb Z$,
\be\lab{sit:eq:pdf1}
\|Q_IF\|_{L^p(\Sigma_t)}\les \|F\|_{L^p(\Sigma_t)}
\end{equation}

ii) \quad {\sl $L^p(\Sigma_t)$- Almost Orthogonality}\quad  Consider two families
of LP-projections $Q_k, \tilde Q_k$ associated to $m$ and  respectively 
$\tilde m$, both in ${\cal M}$. For any  
$1\leq p\le
\infty$:
\be\lab{sit:eq:pdf2}
\|Q_k\tilde Q_{k'}F\|_{L^p(\Sigma_t)}\les 2^{-2|k-k'|} \|F\|_{L^p(\Sigma_t)}
\end{equation}

iii) \quad  {\sl Bessel inequality} 
$$\sum_k\|Q_k F\|_{L^2(\Sigma_t)}^2\les \|F\|_{L^2(\Sigma_t)}^2$$
iv)\quad {\sl Reproducing Property} \quad  Given any integer $\ell\ge 2$ 
 and  $\bar m\in \mathfrak{M}_{\ell}$ there exists $m\in \mathfrak{M}$ such that 
 such that $\bar m=  m\star m$. Thus,
$$^{(\bar m)}Q_k =^{(m)}Q_k\c ^{(m)}Q_k.$$
Whenever there is no danger of confusion we shall simply write $Q_k=Q_k\c Q_k$.

v)\quad {\sl Finite band property}\quad For any $1\leq p\leq \infty$.
\bee
\|\Delta Q_k F\|_{L^p(\Sigma_t)}&\les& 2^{2k} \|F\|_{L^p(\Sigma_t)}\\
\|Q_kF\|_{L^p(\Sigma_t)} &\les& 2^{-2k} \|\Delta F \|_{L^p(\Sigma_t)}
\eee
Moreover given $m\in \mathfrak{M}$ we can find $\bar{m}\in \mathfrak{M}$
such that $\Delta Q_k =2^{2k}\bar{P}_k$ with $\bar{P}_k$
the LP projections associated to $\bar{m}$.

In addition, the $L^2(\Sigma_t)$ estimates
\bee
\|\nab Q_k F\|_{L^2(\Sigma_t)}&\les& 2^{k} \|F\|_{L^2(\Sigma_t)}\\
\|Q_kF\|_{L^2(\Sigma_t)} &\les& 2^{-k} \|\nab F  \|_{L^2(\Sigma_t)}
\eee
hold together with the dual estimate
$$\| Q_k \nab F\|_{L^2(\Sigma_t)}\les 2^k \|F\|_{L^2(\Sigma_t)}$$

vi) \quad{\sl Bernstein inequality}\quad For any $2\leq p\leq+\infty$
\begin{align*}
&\|Q_k F\|_{L^p(\Sigma_t)}\les (2^{3(\frac{1}{2}-\frac{1}{p})k}+1) \|F\|_{L^2(\Sigma_t)},\\
&\|Q_{<0} F\|_{L^p(\Sigma_t)}\les \|F\|_{L^2(\Sigma_t)}
\end{align*}
together with the dual estimates 
\begin{align*}
&\|Q_k F\|_{L^2(\Sigma_t)}\les (2^{3(\frac{1}{2}-\frac{1}{p})k}+1) \|F\|_{L^{p'}(\Sigma_t)},\\
&\|Q_{<0} F\|_{L^2(\Sigma_t)}\les \|F\|_{L^{p'}(\Sigma_t)}
\end{align*}
\label{sit:thm:LP}
\end{theorem}

\begin{proof}
We refer to \cite{LP} for the proof of i)-v). Next, we turn to the proof of  vi). In the case $2\leq p\leq 6$, it is an easy consequence of the Gagliardo-Nirenberg  inequality \eqref{GNsit}:
\bee
\norm{Q_kF}_{L^p(\Sigma_t)}&\les& \norm{\nabla Q_kF}^{3(\frac{1}{2}-\frac{1}{p})}_{L^2(\Sigma_t)} \norm{Q_kF}^{-\frac{1}{2}+\frac{3}{p}}_{L^2(\Sigma_t)}\\
&\les& 2^{3(\frac{1}{2}-\frac{1}{p})k}\norm{F}_{L^2(\Sigma_t)},
\eee
where we used the finite band property and the boundedness on $L^2(\Sigma_t)$ for $Q_k$. Next, we consider the case 
$6<p\leq +\infty$. Using the Gagliardo-Nirenberg  inequality \eqref{GNsit1}, we have
\bee
\norm{Q_kF}_{L^p(\Sigma_t)}&\les& \norm{\nabla^2Q_kF}^{\frac{1}{2}-\frac{3}{p}}_{L^2(\Sigma_t)}\norm{\nabla Q_kF}^{\frac{1}{2}+\frac{3}{p}}_{L^2(\Sigma_t)}\\
&\les& (\norm{\Delta Q_kF}_{L^2(\Sigma_t)}+\norm{Q_kF}_{L^2(\Sigma_t)})^{\frac{1}{2}-\frac{3}{p}}2^{k(\frac{1}{2}+\frac{3}{p})}\norm{F}^{\frac{1}{2}+\frac{3}{p}}_{L^2(\Sigma_t)}\\
&\les& (2^{3(\frac{1}{2}-\frac{1}{p})k}+1) \|F\|_{L^2(\Sigma_t)}
\eee
where we used the Bochner inequality \eqref{prop:bochsit}, and the finite band property and the boundedness on $L^2(\Sigma_t)$ for $Q_k$. This concludes the proof of vi), and of the theorem. 
\end{proof}

\subsubsection{Besov spaces on $\Sigma_t$}

Using the Littlewood-Paley projections of the previous section, we introduce Besov spaces on $\Sigma_t$. 

\begin{definition}
Let $a\geq 0$. We define the Besov norms
$$\norm{F}_{\bsit^a}=\sum_{j\geq 0}2^{aj}\norm{Q_jF}_{L^2(\Sigma_t)}+\norm{F}_{L^2(\Sigma_t)},$$
where $F$ is an arbitrary tensor on $\Sigma_t$.
\end{definition}

In view of the definition of $\bsit^{\frac{3}{2}}$ and the Bernstein inequality for $Q_j$, we immediately obtain the following embedding
\be\lab{bsit}
\norm{F}_{L^\infty(\Sigma_t)}\les \norm{F}_{\bsit^{\frac{3}{2}}(\Sigma_t)}
\ee
where $F$ is an arbitrary tensor on $\Sigma_t$.

Next, we consider the action of $\nabla$ on $\bsit^{\frac{5}{2}}$.
\begin{lemma}
Let $f$ a scalar function on $\Sigma_t$. Then, we have the following estimate
\be\lab{bsit1}
\norm{\nabla f}_{\bsit^{\frac{3}{2}}}\les\norm{f}_{\bsit^{\frac{5}{2}}}.
\ee
\end{lemma}

\begin{proof}
We have
\be\lab{bsit2}
\norm{Q_j\nab f}_{L^2(\Sit)}\les \norm{Q_j\nab Q_{<0}f}_{L^2(\Sit)}+\sum_{l\geq 0}\norm{Q_j\nab Q_lf}_{L^2(\Sit)}.
\ee
Next, we estimate the right-hand side of \eqref{bsit2}. We start with the case $j<l$. Using the finite band property for $Q_j$, we have
\be\lab{bsit3}
\norm{Q_j\nab Q_l f}_{L^2(\Sit)}\les 2^j\norm{Q_lf}_{L^2(\Sit)}.
\ee

Next, we consider the case $l\leq j$. Using the finite band property for $Q_j$, we have
\be\lab{bsit3bis}
\norm{Q_j\nab Q_l f}_{L^2(\Sit)}\les 2^{-2j}\norm{\Delta\nab Q_l f}_{L^2(\Sit)}.
\ee
Furthermore, we have $[\Delta, \nab]h=R_t\nab h$ for any scalar $h$ on $\Sit$, where $R_t$ is the curvature tensor of the induced metric on $\Sigma_t$. Thus, we obtain
\bea\lab{bsit3ter}
\norm{\Delta\nab Q_l f}_{L^2(\Sit)}&\les& \norm{\nab\Delta Q_l f}_{L^2(\Sit)}+\norm{[\Delta,\nab]Q_l f}_{L^2(\Sit)}\\
\nn&\les& \norm{\nab\Delta Q_l f}_{L^2(\Sit)}+\norm{R_t\nab Q_l f}_{L^2(\Sit)}\\
\nn&\les& \norm{\nab\Delta Q_l f}_{L^2(\Sit)}+\norm{R_t}_{L^2(\Sit)}\norm{\nab Q_l f}_{L^\infty(\Sit)}\\
\nn&\les& 2^{3l}\norm{Q_l f}_{L^2(\Sit)}+\norm{\nab Q_l f}_{L^\infty(\Sit)},
\eea
where we used in the last inequality the finite band property for $P_l$ and the bound \eqref{curvriccisit1} for $R_t$. 
Next, we evaluate the second term in the right-hand side of \eqref{bsit3}. Using the Gagliardo-Nirenberg inequality \eqref{GNsit1} with $p=+\infty$, we have
$$\norm{\nab Q_l f}_{L^\infty(\Sit)}\les \norm{\nab^3 Q_l f}^{\frac{1}{2}}_{L^\infty(\Sit)}\norm{\nab^2 Q_l f}^{\frac{1}{2}}_{L^\infty(\Sit)}+\norm{\nab^2 Q_l f}_{L^\infty(\Sit)}.$$
Together with the Bochner inequality \eqref{prop:bochsit} on $\Sit$, we obtain
\bea\lab{bsit4}
\norm{\nab Q_l f}_{L^\infty(\Sit)}&\les& (\norm{\Delta\nab Q_lf}_{L^2(\Sit)}+\norm{\Delta Q_lf}_{L^2(\Sit)}+\norm{\nab Q_lf}_{L^2(\Sit)})^{\frac{1}{2}}\\
\nn&&\times(\norm{\Delta Q_lf}_{L^2(\Sit)}+\norm{\nab Q_lf}_{L^2(\Sit)})^{\frac{1}{2}}+\norm{\Delta Q_lf}_{L^2(\Sit)}+\norm{\nab Q_lf}_{L^2(\Sit)}\\
\nn&\les&\norm{\Delta\nab Q_lf}_{L^2(\Sit)}^{\frac{1}{2}}2^l\norm{Q_lf}^{\frac{1}{2}}_{L^2(\Sit)}
+2^{2l}\norm{Q_lf}_{L^2(\Sit)}
\eea
where we used the finite band property for $Q_l$ in the last inequality. \eqref{bsit3ter} and \eqref{bsit4} imply
$$\norm{\Delta\nab Q_l f}_{L^2(\Sit)}\les \norm{\Delta\nab Q_lf}_{L^2(\Sit)}^{\frac{1}{2}}2^l\norm{Q_lf}^{\frac{1}{2}}_{L^2(\Sit)}+2^{3l}\norm{Q_l f}_{L^2(\Sit)},$$
which yields
$$\norm{\Delta\nab Q_l f}_{L^2(\Sit)}\les 2^{3l}\norm{Q_l f}_{L^2(\Sit)}.$$
Together with \eqref{bsit3bis}, we obtain
\be\lab{bsit5}
\norm{Q_j\nab Q_l f}_{L^2(\Sit)}\les 2^{-2j+3l}\norm{Q_l f}_{L^2(\Sit)}.
\ee

Finally, using \eqref{bsit3} for $j<l$ and \eqref{bsit5} for $j\geq l$, we obtain
$$2^{\frac{3j}{2}}\norm{Q_j\nab Q_l f}_{L^2(\Sit)}\les 2^{-\frac{|j-l|}{2}}(2^{\frac{5l}{2}}\norm{Q_l f}_{L^2(\Sit)}),$$
which together with \eqref{bsit2} and the definition of $\bsit^{\frac{3}{2}}$ implies \eqref{bsit1}. This concludes the proof of the lemma. 
\end{proof}

We conclude this section with two estimates for the product in the Besov space $\bsit^{\frac{1}{2}}$. 
\begin{lemma}
We have
\be\lab{bsit6}
\norm{|F|^2}_{\bsit^{\frac{1}{2}}}\les (\norm{\nab F}_{L^2(\Sit)}+\norm{F}_{L^2(\Sit)})^2
\ee
for any tensor $F$ on $\Sit$.
\end{lemma}

\begin{proof}
We have
\be\lab{bsit7}
\norm{Q_j(|F|^2)}_{L^2(\Sit)}\les \sum_{l, m\geq 0}\norm{Q_j(Q_lF\c Q_mF)}_{L^2(\Sit)},
\ee
where we dropped the terms involving $Q_{<0}$ which are easier to handle. Next, we estimate the right-hand side of \eqref{bsit7}. By symmetry, we may assume $m\leq l$. We start with the case $j<m$. Using the dual Bernstein inequality for $Q_j$, we have
\be\lab{bsit8}
\norm{Q_j(Q_l F\c Q_m F)}_{L^2(\Sit)}\les 2^{\frac{3j}{2}}\norm{Q_lF\c Q_mF}_{L^1(\Sit)}\les 2^{\frac{3j}{2}}\norm{Q_lF}_{L^2(\Sit)}\norm{Q_mF}_{L^2(\Sit)}.
\ee

Next, we consider the case $m\leq j<l$. Using the boundedness on $L^2(\Sit)$ of $Q_j$ and the Bernstein inequality for $Q_m$, we have
\be\lab{bsit9}
\norm{Q_j(Q_l F\c Q_m F)}_{L^2(\Sit)}\les \norm{Q_lF}_{L^2(\Sit)}\norm{Q_mF}_{L^\infty(\Sit)}\les 2^{\frac{3m}{2}}\norm{Q_lF}_{L^2(\Sit)}\norm{Q_mF}_{L^2(\Sit)}.
\ee

Finally, we consider the case $l\leq j$. Using the finite band property for $Q_j$, $Q_l$ and $Q_m$, we have
\bee
&&\norm{Q_j(Q_l F\c Q_m F)}_{L^2(\Sit)}\\
&\les& 2^{-2j}\norm{\Delta(Q_l F\c Q_m F)}_{L^2(\Sit)}\\
&\les& 2^{-2j}\norm{\Delta(Q_l F)\c Q_m F}_{L^2(\Sit)}+2^{-2j}\norm{\nab(Q_l F)\c \nab(Q_m F)}_{L^2(\Sit)}\\
&&+ 2^{-2j}\norm{Q_l F\c \Delta(Q_m F)}_{L^2(\Sit)}\\
&\les& (2^{-2j+2l}+2^{-2j+2m})\norm{Q_l F}_{L^2(\Sit)}\norm{Q_m F}_{L^\infty(\Sit)}+2^{-2j}\norm{\nab Q_l F}_{L^4(\Sit)}\norm{\nab Q_m F}_{L^4(\Sit)}.
\eee
Together with the Bernstein inequality for $Q_m$, the Gagliardo-Nirenberg inequality \eqref{GNsit}, and the fact that $m\leq l$, we obtain
\bee
&&\norm{Q_j(Q_l F\c Q_m F)}_{L^2(\Sit)}\\
&\les& 2^{-2j+2l+\frac{3m}{2}}\norm{Q_l F}_{L^2(\Sit)}\norm{Q_m F}_{L^2(\Sit)}\\
&&+2^{-2j}\norm{\nab^2 Q_l F}^{\frac{3}{4}}_{L^2(\Sit)}\norm{\nab Q_l F}^{\frac{1}{4}}_{L^2(\Sit)}\norm{\nab^2 Q_m F}^{\frac{3}{4}}_{L^2(\Sit)}\norm{\nab Q_m F}^{\frac{1}{4}}_{L^2(\Sit)}.
\eee
Using the finite band property for $Q_l$ and $Q_m$, the Bochner inequality \eqref{prop:bochsit} on $\Sit$ and the fact that $m\leq l$, we obtain
\be\lab{bsit10}
\norm{Q_j(Q_l F\c Q_m F)}_{L^2(\Sit)}\les 2^{-2j+2l+\frac{3m}{2}}\norm{Q_l F}_{L^2(\Sit)}\norm{Q_m F}_{L^2(\Sit)}.
\ee
In the end, \eqref{bsit8}, \eqref{bsit9} and \eqref{bsit10} imply
$$2^{\frac{j}{2}}\norm{Q_j(Q_l F\c Q_m F)}_{L^2(\Sit)}\les 2^{-\frac{|l-j|}{4}-\frac{|m-j|}{4}}(2^l\norm{Q_l F}_{L^2(\Sit)})(2^m\norm{Q_m F}_{L^2(\Sit)}),$$
which together with \eqref{bsit7} and the definition of $\bsit^{\frac{1}{2}}$ implies \eqref{bsit6}. This concludes the proof of the lemma. 
\end{proof}

\begin{lemma}
We have
\be\lab{bsit11}
\norm{fh}_{\bsit^{\frac{1}{2}}}\les (\norm{f}_{L^\infty(\Sit)}+\norm{\nab h}_{L^3(\Sit)})\norm{h}_{\bsit^{\frac{1}{2}}}
\ee
for any scalars $f$ and $h$ on $\Sit$.
\end{lemma}

\begin{proof}
We have
\be\lab{bsit12}
\norm{Q_j(fh)}_{L^2(\Sit)}\les \sum_{l\geq 0}\norm{Q_j(fQ_lh)}_{L^2(\Sit)},
\ee
where we dropped the term involving $Q_{<0}$ which is easier to handle. Next, we estimate the right-hand side of \eqref{bsit12}. We start with the case $j<l$. Using the boundedness on $L^2(\Sit)$ of $Q_j$, we have
\be\lab{bsit13}
\norm{Q_j(fQ_l h)}_{L^2(\Sit)}\les \norm{fQ_lh}_{L^2(\Sit)}\les \norm{f}_{L^\infty(\Sit)}\norm{Q_lh}_{L^2(\Sit)}.
\ee

Next, we consider the case $l\leq j$. Using the finite band property for $Q_j$, we have
\bee
\norm{Q_j(fQ_l h)}_{L^2(\Sit)}&\les& 2^{-j}\norm{\nab(fQ_lh)}_{L^2(\Sit)}\\
&\les& 2^{-j}\norm{\nab f}_{L^3(\Sit)}\norm{Q_lh}_{L^6(\Sit)}+2^{-j}\norm{f}_{L^\infty(\Sit)}\norm{\nab Q_lh}_{L^2(\Sit)}.
\eee
Using the Bernstein inequality and the finite band property for $Q_l$, this yields
\be\lab{bsit14}
\norm{Q_j(fQ_l h)}_{L^2(\Sit)}\les 2^{-j+l}(\norm{\nab f}_{L^3(\Sit)}+\norm{f}_{L^\infty(\Sit)})\norm{Q_lh}_{L^2(\Sit)}.
\ee
Finally, \eqref{bsit13} and \eqref{bsit14} imply
$$2^{\frac{j}{2}}\norm{Q_j(fQ_l h)}_{L^2(\Sit)}\les 2^{-\frac{|l-j|}{2}}(2^{\frac{l}{2}}\norm{Q_lh}_{L^2(\Sit)}),$$
which together with \eqref{bsit12} and the definition of $\bsit^{\frac{1}{2}}$ implies \eqref{bsit11}. This concludes the proof of the lemma. 
\end{proof}

\section{Regularity with respect to $(t,x)$}\lab{sec:regxproof}

This section is devoted to the proof of Theorem \ref{thregx}. We assume the following bootstrap assumptions:
\be\lab{boot1}
\norm{n-1}_{\lh{\infty}}+\norm{b-1}_{\lh{\infty}}\leq\frac{1}{10},
\ee
\be\lab{boot2}
\norm{\nabla n}_{\tx{\infty}{2}}+\norm{\nabla^2n}_{\tx{\infty}{2}}+\norm{\nabla \dd_Tn}_{\tx{\infty}{2}}+\noo(b)+\norm{\lb(b)}_{\xt{2}{\infty}}\leq D\ep,
\ee
\be\lab{boot3}
\no(k)+\norm{\ddb_{\lb}\kep}_{\lh{2}}+\norm{\dd_{\lb}\d}_{\lh{2}}+\norm{\kepb}_{\xt{\infty}{2}}+\norm{\db}_{\xt{\infty}{2}}\leq D\ep,
\ee
\be\lab{boot4}
\norm{\trc}_{\lh{\infty}}+\norm{\nabb\trc}_{\xt{2}{\infty}}+\norm{\lb\trc}_{\xt{2}{\infty}}
\leq D\ep,
\ee
\be\lab{boot5}
\norm{\hch}_{\xt{2}{\infty}}+\no(\hch)+\norm{\ddb_{\lb}\hch}_{\lh{2}}
\leq D\ep,
\ee
\be\lab{boot6}
\norm{\z}_{\xt{2}{\infty}}+\no(\z)\leq D\ep,
\ee
where $D>0$ is a large enough constant. We will improve on these estimates. To this end, we show in section \ref{sec:nointersec} that the null hypersurfaces $\H_u$ are well-behaved for $0\leq t\leq 1$, in the sense that there are neither caustics nor intersection of distinct null geodesics generating $\H_u$. In section \ref{sec:coordptusit}, we construct coordinate systems on $\ptu$ and $\Sigma_t$ needed for the validity of the estimates derived in section \ref{sec:calcineqgeneral}. In section \ref{sec:boundgaussK}, we derive an estimate for the Gauss curvature $K$ necessary to obtain a useful strong Bernstein inequality, as well as a useful Bochner inequality on $\ptu$. In sections \ref{sec:impn} and \ref{sec:imponk}, we improve on the bootstrap bounds for $n$ and $k$ in \eqref{boot1}-\eqref{boot3}, with the exception of the trace estimates for $\kepb$ and $\db$ in \eqref{boot3}. In section \ref{sec:timefolvsgeodfol}, we show how to infer estimates in the time foliation from corresponding estimates in the geodesic foliation obtained in \cite{FLUX}. This allows us to improve on the $L^\infty$ bound for $\trc$ and the trace bounds on $\hch$ and $\z$ in the bootstrap bounds \eqref{boot4} and \eqref{boot6}. In section \ref{sec:tracenormk}, we improve on the bootstrap bound \eqref{boot3} for the trace estimates of $\db$ and $\kepb$. In section \ref{sec:impbootb}, we improve on the bootstrap bounds for $b$ in \eqref{boot1} \eqref{boot2}, and we also derive an estimate for $b$ in $\tx{\infty}{4}$. Finally, we improve on the remaining bootstrap bounds in \eqref{boot4}-\eqref{boot6} in section \ref{sec:remainest}.

\begin{remark}
This section concerns the regularity of the foliation generated by $u$ on $\mathcal{M}$ with respect to $(t,x)$. Thus, the dependance in the angle $\o\in\S$ plays no role in this section. 
\end{remark}

\subsection{Lower bound on the injectivity radius on $\H_u$}\lab{sec:nointersec}

The control we obtain on the geometric quantities associated to our foliation is only valid as long as no caustic form and null geodesics do not intersect on $\H_u$. The goal of  this section is to prove the absence of caustic and that null geodesics do not intersect at least until $t=1$, i.e. the null radius of injectivity of $\H_u$ is at least 1. In addition to the bound \eqref{curvflux1} on the curvature tensor $\rr$ of $\gg$, we make the following regularity assumption on $\gg$. There exists a coordinate chart on $\mathcal{M}$ such that
\be\lab{grossereg}
\norm{\gg}_{C^2(\mathcal{M})}\leq M,
\ee
where $M$ is a very large constant. 

\begin{remark}
The assumption \eqref{grossereg} is only used to prove the absence of caustic and that null geodesics do not intersect at least until $t=1$, which is a qualitative property. On the other hand, we only rely on the bound \eqref{curvflux1} on $\rr$ to prove the various quantitative bounds of Theorems \ref{thregx}, \ref{thregx1}, \ref{thregomega} and \ref{thregomega2}.
\end{remark}

For $(0,x)$ in $\Si$, recall the definition in Remark \ref{defnullgeod} of the null geodesic $\kappa_x(t)$. For all $0\leq t\leq 1$, let $\Phi_t:\Si\rightarrow\Sit$ defined by $\Phi_t(0,x)=\kappa_x(t)$. We have $\Phi_0(0,x)=(0,x)$ on $\Si$ which together with \eqref{grossereg} and the global inversion theorem shows that $\Phi_t$ is a $C^1$ global diffeomorphism from $\Si$ to $\Sit$ for $0\leq t\leq \frac{1}{100M}$. We define $t_0\geq 0$ as the supremum of $0\leq t\leq 1$ such that $\Phi_t$ is bijective from $\Si$ to $\Sit$. Our goal is to show that we have in fact $t_0=1$. We will first show the absence of caustic which is a consequence of the fact that $\Phi_t$ is locally injective. We will then show that $\Sit$ is covered by the $u$-foliation which is equivalent to the surjectivity of $\Phi_t$. Finally, we will show the nonintersection of distinct null geodesics which is equivalent to the global injectivity of $\Phi_t$.

\begin{remark}
As long as $0\leq t<t_0$, there are no caustics and no distinct null geodesic intersections. Thus, we may assume that the u-foliation exists and satisfies the bounds \eqref{boot1}-\eqref{boot6} given by the bootstrap assumptions. Furthermore, we may assume the identity \eqref{defnullgeod1} for the null geodesics $\kappa_x(t)$.
\end{remark}

\subsubsection{Absence of caustic}\lab{sec:nocaustic}

The absence of caustic is equivalent to the absence of conjugate points and is a consequence of the fact  that $\Phi_t$ is locally injective. Since $\Phi_t$ preserves the u-foliation, it is enough to show that $\Phi_t$ is locally injective as a map from $\pou$ to $\ptu$. We will actually show that $\Phi_t$ as a map from $\pou$ to $\ptu$ is a local $C^1$ diffeomorphism. \\

Let $(0,x)$ a point in $\pou$. From \eqref{defnullgeod1}, we have $\kappa_x'(t)=b^{-1}L_{\kappa_x(t)}$ for all $0\leq t< t_0$. Since $\Phi_t(0,x)=\kappa_x(t)$, we obtain the following differential equation for the Jacobian matrix $D\Phi_t$ of $\Phi_t$:
$$\frac{d}{dt}\left(D\Phi_t\right)=b^{-1}\chi(D\Phi_t,e_b)e_b$$
which together with the fact that $\chi$ is symmetric yields:
$$\frac{d}{dt}\left(\det\left(D\Phi_t(D\Phi_t)^T\right)\right)=2b^{-1}\trc\det\left(D\Phi_t(D\Phi_t)^T\right)$$
and after integration in time:
\be\label{diffeoloc}
\det\left(D\Phi_t(D\Phi_t)^T\right)=\exp\left(2\int_0^tb^{-1}\trc d\tau\right)\sim 1
\ee
where we used the bootstrap assumption \eqref{boot4}. In particular, the local inversion Theorem together with \eqref{diffeoloc} and \eqref{grossereg} yields the fact that $\Phi_t$ as a map from $\pou$ to $\ptu$ is a local $C^1$ diffeomorphism. In particular, no caustic form for all $0\leq t< t_0$. 

\subsubsection{Covering of $\Sit$ by the $u$-foliation}\lab{sec:surjective}

We will prove that for all $0\leq t<t_0$, $\Sit$ is covered by the $u$-foliation, i.e.:
$$\Sit=\cup_u\ptu$$
which is equivalent to the surjectivity of $\Phi_t$ as a map from $\Si$ to $\Sit$.\\

Let $A=\{t\,/\,\Sit=\cup_u\ptu\}$. We start by showing that $A$ is closed in $0\leq t<t_0$. Consider a sequence of times $t_p\rightarrow \underline{t}$ such that $t_p$ belongs to $A$ for all $p$. Let $(\underline{t},x)$ an arbitrary point in $\Sigma_{\underline{t}}$. There exists a sequence $(t_p,x_p)$ in $\Sigma_{t_p}$ such that $(t_p,x_p)$ converges to $(\underline{t},x)$. Since $(t_p,x_p)$ is in $\Sigma_{t_p}$ and $t_p$ belongs to $A$, $(t_p,x_p)$ belong to $\cup\Sigma_{t_p}$ and therefore there is $(0,x_p^0)$ in $\Si$ such that $(t_p,x_p)=\kappa_{x_p^0}(t_p)$. Now, the bound \eqref{grossereg} together with the fact that $(t_p,x_p)$ is a bounded sequence implies that $(0,x_p^0)$ is a bounded sequence in $\Si$. Thus, up to a subsequence, it converges to $(0,x_0)$ in $\Si$. Finally, using again the bound \eqref{grossereg} together with the fact that $t_p$ converges to $\underline{t}$ and $(0,x^0_p)$ converges to $(0,x_0)$ implies that $\kappa_{x_p^0}(t_p)$ converges to $\kappa_{x_0}(\underline{t})$. Thus, $(\underline{t},x)=\kappa_{x_0}(\underline{t})$ which shows that $(\underline{t},x)$ belongs to $\cup_u\ptu$. Therefore, $\underline{t}$ belongs to $A$ which implies that $A$ is closed.\\

Let us now prove that $A$ is open in $0\leq t<t_0$. Let $\underline{t}\in A$ and consider times $t$ satisfying $|t-\underline{t}|<\frac{1}{100M}$ where $M$ is the constant appearing in \eqref{grossereg}. Let $(t,x_0)$ an arbitrary point in $\Sit$. We may assume $t>\underline{t}$ since the case $t<\underline{t}$ is treated in the exact same way. Let $C^-$ denote the backward null cone with vertex $(t,x_0)$ (we would consider the forward null cone in the case $t<\underline{t}$). Let $S^-$ denote the intersection of the backward null cone $C^-$ with $\Sigma_{\underline{t}}$. Then, the assumption $|t-\underline{t}|<\frac{1}{100M}$ together with the bound \eqref{grossereg} implies that $S^-$ is a $C^1$ compact orientable surface in $\Sigma_{\underline{t}}$. In particular, since any compact set of $\Sigma_{\underline{t}}$ is included in $\{-B<u<B\}$ for a large enough constant $B$, the set $\{u\,/\,\ptu\cap S^-\neq\emptyset\}$ is a bounded subset of $\R$. Using the fact that $S^-$ is compact, $\ptu$ is closed in $\Sigma_{\underline{t}}$, and $\underline{t}\in A$, we obtain the existence of $u_0$ such that $P_{t,u_0}\cap S^-\neq\emptyset$ and:
$$u_0=\min\{u\,/\,\ptu\cap S^-\neq\emptyset\}.$$ 
Let $(\underline{t},x_1)$ a point in $P_{t,u_0}\cap S^-$. Then, by definition of $u_0$ we have $P_{t,u}\cap S^-=\emptyset$ for all $u<u_0$ which implies that $N=-N_{S^-}$ at $(\underline{t},x_1)$ where $N=\nabla u/|\nabla u|$ is the normal to $\ptu$ and $N_{S^-}$ is the outward normal to $S^-$. In turn, this implies that $L$ coincides with the null generator of the backward null cone $C^-$ at $(\underline{t},x_1)$. From \eqref{defnullgeod}, let $(0,x_2)$ on $\Si$ such that $b^{-1}L=\kappa'_{x_2}(\underline{t})$. Since $\kappa'_{x_2}(\underline{t})$  coincides with the null generator of the backward null cone $C^-$ at $(\underline{t},x_1)$, we obtain $\kappa_{x_2}(t)=(t,x_0)$. Therefore, $(t,x_0)$ belongs to $P_{t,u_1}$ where $u_1=u(0,x_2)$. This implies that $\Sit=\cup_u\ptu$ for all $|t-\underline{t}|<\frac{1}{100M}$ so that $A$ is open.\\

Finally, $A$ is closed and open in $0\leq t<t_0$. Furthermore, $\Si=\cup_u\pou$ from the construction of $u$ on $\Si$ in \cite{param1}. Therefore, $A=\{0\leq t<t_0\}$, i.e. $\Sit=\cup_u\ptu$ for all $0\leq t<t_0$.

\subsubsection{Nonintersection of distinct null geodesics}

We would like to show that $t_0\geq 1$. Assume by contradiction that $0<t_0<1$. \\

Let us first show that there exist two distinct null geodesics intersecting at $t=t_0$. Assume by contradiction that there exists $\d>0$ such that no distinct null geodesics intersect on $0\leq t< t_0+\d$. Then, $u$ exists on $0\leq t< t_0+\d$ unless a caustic forms at a time $0<t_1<t_0+\d$. Assume that $t_1$ is the first such time. Then, $u$ exists on $0\leq t<t_1$ and $b$ and $\trc$ satisfy the bootstrap assumptions \eqref{boot1} \eqref{boot4} on $0\leq t<t_1$ so that \eqref{diffeoloc} holds on $0\leq t<t_1$. Now, since $\Phi_t$ is $C^1$ from the assumption \eqref{grossereg}, this implies that:
$$\det\left(D\Phi_t(D\Phi_t)^T\right)\sim 1,\,0\leq t<t_1+\d_1$$
for some $\d_1>0$. In turn, this yields the absence of caustic for $0\leq t<t_1+\d_1$ contradicting the definition of $t_1$. In particular, we obtain the absence of caustic for $0\leq t< t_0+\d$, the existence of $u$ on the same time interval, and in turn  $\Sit=\cup_u\ptu$ from section \ref{sec:surjective}. Finally, on $0\leq t<t_0+\d$, no distinct null geodesic intersect and $\Sit=\cup_u\ptu$ so that $\Phi_t$ which is both injective and surjective. This contradicts the definition of $t_0$. We conclude that there exist two distinct null geodesics that intersect at $t_0$.\\

From the previous paragraph, $u$ exists on the time interval $0\leq t<t_0$ where it satisfies $\Sit=\cup_u\ptu$ and the bootstrap assumptions \eqref{boot1}-\eqref{boot6}. Furthermore, two distinct null geodesics intersect at $t_0$. Let $(0,x_1)\neq (0,x_2)$ two points in $\Si$ such that $\kappa_{x_1}(t_0)=\kappa_{x_2}(t_0)=(t_0,x_0)$. In view of \eqref{grossereg}, there exists a coordinate chart $U\subset\mathcal{M}$ which is a neighborhood of $(t_0,x_0)$ such that relative to this coordinate system, we have:
\be\lab{coordregc2}
\norm{\gg_{\a\b}}_{C^2(U)}\lesssim M,\,\forall \a,\b=0,\dots, 3.
\ee
Now, from the Ricci equations \eqref{ricciform} we have:
\bea\lab{estDL}
&&\norm{\dd_LL}_{\li{\infty}{6}}+\norm{\dd_{\lb}L}_{\li{\infty}{6}}+\norm{\nabb L}_{\li{\infty}{6}}\\
\nn&\lesssim& \norm{\chi}_{\li{\infty}{6}}+\norm{\z}_{\li{\infty}{6}}+\norm{k}_{\li{\infty}{6}}+\norm{\nabla n}_{\li{\infty}{6}},
\eea
which together with the Sobolev embedding \eqref{sobineq} and the bootstrap assumptions \eqref{boot1}-\eqref{boot6} yields:
\be\lab{estDL1}
\ds\norm{\dd L}_{\li{\infty}{6}}\lesssim 1.
\ee
From the bootstrap assumption \eqref{boot2} and the Sobolev embedding \eqref{sobineq}, we have:
\be\lab{estDL2}
\norm{L(b)}_{\li{\infty}{6}}+\norm{\nabb b}_{\li{\infty}{6}}\lesssim 1.
\ee
We now estimate $\lb(b)$. Using the transport equation satisfied by $b$ \eqref{D4a}, the computation of $\lb(\db  )$ \eqref{D4tmu3} and the commutation formula \eqref{comm3}, we obtain the following transport equation:
\bea
\lab{estDL3}L(\lb(b)-b(\d +n^{-1}\nabla_N n))=-2b(k_{AN}-\xi_A)n^{-1}\nabla_An+2b|n^{-1}N(n)|^2\\
\nonumber -2bk_{Nm}k_N^m
+2(\xi_B-\xib_B)\nabb_Bb-2b\rho.
\eea
\eqref{estDL3} together with the Gagliado-Nirenberg inequality \eqref{eq:GNirenberg} and the bootstrap assumption \eqref{boot1}-\eqref{boot6} yields:
\be\lab{estDL4}
\begin{array}{lll}
&&\norm{\lb(b)-b(\d +n^{-1}\nabla_N n)}_{\tx{6}{\infty}}\\
&\lesssim &\norm{-2b(k_{AN}-\xi_A)n^{-1}\nabla_An+2b|n^{-1}N(n)|^2}_{\xt{6}{1}}\\
&&+\norm{-2bk_{Nm}k_N^m
+2(\xi_B-\xib_B)\nabb_Bb-2b\rho}_{\xt{6}{1}}\\
&\lesssim&\norm{k}^2_{\tx{2}{6}}+\norm{\xi}^2_{\tx{2}{6}}+\norm{\xib}^2_{\tx{2}{6}}+\norm{\nabb b}^2_{\tx{2}{6}}+\norm{\nabla n}^2_{\tx{2}{6}}+\norm{\rho}_{\tx{2}{6}}\\
&\lesssim& 1+\norm{\rho}_{\tx{2}{6}}
\end{array}
\ee
which again using the bootstrap assumptions implies:
\be\lab{estDL5}
\norm{\lb(b)}_{\li{\infty}{6}}\lesssim 1+\norm{\rho}_{\li{\infty}{6}}.
\ee
\eqref{estDL5} together with the bound \eqref{grossereg} implies:
\be\lab{estDL6}
\norm{\lb(b)}_{\li{\infty}{6}}\lesssim M.
\ee
Finally, \eqref{estDL1}, \eqref{estDL2} and \eqref{estDL6} yield:
\be\lab{estDL7}
\norm{\dd (b^{-1}L)}_{\li{\infty}{6}}\lesssim M.
\ee
In particular, the same bound holds in $L^6(\mathcal{M}\cap\{0\leq t<t_0\})$ which together with \eqref{coordregc2} implies in the coordinate chart $U$:
$$\ds\norm{\partial (b^{-1}L)}_{L^6(U\cap\{0\leq t<t_0\})}\lesssim M.$$
Together with the usual Sobolev embedding in dimension 4, this yields, in the coordinate chart $U$:
\be\lab{estDL8}
\ds\norm{b^{-1}L}_{C^{\frac{1}{12}}(U\cap\{0\leq t<t_0\})}\lesssim M.
\ee
Now, using the fact that $\kappa_{x_1}(t_0)=\kappa_{x_2}(t_0)$, and the fact that $\kappa_{x}(t)$ is continuous in $t$ from \eqref{grossereg}, we have 
$$\lim_{t\rightarrow {t_0}_-}\textrm{dist}(\kappa_{x_1}(t),\kappa_{x_2}(t))=0$$
where dist denotes the geodesic distance in $\Sit$. Together with 
\eqref{estDL8}, this implies that the distance between $b^{-1}L_{\kappa_{x_1}(t)}$ 
and $b^{-1}L_{\kappa_{x_2}(t)}$ as vectors of $\mathbb{R}^4$ in the 
coordinate chart $U$ converges to 0 as $t\rightarrow {t_0}_-$. As 
$b^{-1}L_{\kappa_{x_j}(t)}=\kappa'_{x_j}(t)$ for $0\leq t<t_0$ and $j=1, 2$ by \eqref{defnullgeod1}, and since 
$\kappa'_{x}(t)$ is continuous in $t$ from \eqref{grossereg}, we conclude 
that $\kappa_{x_1}'(t_0)=\kappa_{x_2}'(t_0)$. From the classical uniqueness result for ODEs, we deduce that $\kappa_{x_1}(t)=\kappa_{x_2}(t)$ for all $t$. 
In particular, taking $t=0$, we obtain $(0,x_1)=(0,x_2)$ which yields a contradiction. \\

Finally, we have proved that $t_0\geq 1$. In particular, we have:
\be\lab{conclusionfoliation}
\begin{array}{l}
\textrm{On }0\leq t\leq 1, \textrm{ there are no caustics and no intersection of distinct null geodesics.} \\
\textrm{In particular, }u\textrm{ exists on }0\leq t\leq 1\textrm{ and the bootstrap assumption \eqref{boot1}-\eqref{boot6} hold.}\\ 
\textrm{Furthermore, }\Sit=\cup_u\ptu\textrm{ for all }0\leq t\leq 1. 
\end{array}
\ee

\subsection{Coordinate systems on $\Sit$ and $\ptu$}\lab{sec:coordptusit}

\subsubsection{A global coordinate system on $\ptu$}\lab{sec:coord}

\begin{lemma}\lab{lemm:coord}
There exists a global coordinate system $x'$ on $\ptu$ satisfying:
\be\lab{eq:coordchartbis}
(1-O(\ep))|\xi|^2\leq \gamma_{AB}(p)\xi^A\xi^B\leq (1+O(\ep))|\xi|^2, \qquad \mbox{uniformly for  all }
\,\, p\in\ptu,
\ee
and the Christoffel symbols $\Gamma^A_{BC}$ of the coordinate system verify:
\be\lab{eq:gammaL2bis}
\sum_{A,B,C}\int_{\ptu}|\Gamma^A_{BC}|^2 dx^1dx^2\lesssim\ep.
\end{equation}
\end{lemma}

\begin{remark}
Lemma \ref{lemm:coord} provides the existence of a global coordinate system on $\ptu$ satisfying assumptions \eqref{eq:coordchart} and \eqref{eq:gammaL2}. Thus, we may use the results of sections \ref{sec:calcineq} and \ref{sec:geompal} in the rest of the paper.
\end{remark}

\begin{proof}:\quad In {\bf step B1}, we have constructed a global coordinate system $x'=(x^1,x^2)$ on $\pou$ (see \cite{param1}). By transporting this coordinate system along the null geodesics generated by $L$, we obtain a coordinate system $x'$ of $\ptu$, and a coordinate system $(t,x')$ of $\H$. Let $\ga_t$ denote the restriction of $\gg$ to $\ptu$. We claim that relative to the coordinates $(t,x')$ on $\H$,  the metric $\ga_t$ verifies:
\be\lab{eq:transportga}
\frac{d}{dt}\ga_{AB}=2n\chi_{AB}.
\end{equation}

 Indeed  relative to the coordinates
$t,x'$ on $\H$  we have $nL=\frac{\partial}{\partial t}$ and  since
 $[\frac{\partial}{\partial t},\frac{\partial}{\partial x^A}]=0 $ we infer from
$\nab_{nL}\ga=0$, and $\ga_{AB}=\ga(\frac{\partial}{\partial x^A}, \frac{\partial}{\partial {x}^B})$,
\bee
0&=&(\nab_{nL}\ga)(\frac{\partial}{\partial {x}^A}, \frac{\partial}{\partial {x}^B})=
\frac{d}{dt}\ga_{AB}-n\ga(\nab_{\frac{\partial}{\partial t}}\frac{\partial}{\partial {x}^A}, 
\frac{\partial}{\partial {x}^B})-n\ga(\frac{\partial}{\partial {x}^A}, 
\nab_{\frac{\partial}{\partial t}}\frac{\partial}{\partial {x}^B})\\
&=&\frac{d}{dt}\ga_{AB}-n\ga(\nab_{\frac{\partial}{\partial {x}^A}} L, 
\frac{\partial}{\partial {x}^B})-n\ga(\frac{\partial}{\partial {x}^A}, 
\nab_{\frac{\partial}{\partial {x}^B}}L )\\
&=&\frac{d}{dt}\ga_{AB}-2n\chi_{AB}
\eee
as desired.

Now, using the bootstrap assumptions \eqref{boot1} and \eqref{boot4} \eqref{boot5}, we have $|n-1|\leq \half$ and $\norm{\chi}_{\xt{\infty}{2}}\leq D\ep$. Together with \eqref{eq:transportga} and the fact that \eqref{eq:coordchartbis} is satisfied on $\pou$, this yields \eqref{eq:coordchartbis}.

Differentiating \eqref{eq:transportga} and using the fact that derivatives $\frac{\partial}{\pr x^A}$ commute with $\frac{d}{dt}$, we obtain:
\bee
\frac{d}{dt}\pr_C\ga_{AB}& =& 2\nab_C(n)\chi_{AB}+2n\pr_C\chi_{AB}
\\
&=& 2\nab_C(n)\chi_{AB}+2n\nab_C\chi_{AB}+(\pr\ga)\cdot\chi
\eee
with $(\pr\ga)\cdot\chi$ denoting sum of  terms involving
only products between derivatives of the metric coefficients
and components of $\chi$.
Therefore, using the bootstrap assumptions \eqref{boot1} and \eqref{boot4} \eqref{boot5}, we obtain:
\bee
\norm{\pr\ga}_{\tx{\infty}{2}}&\lesssim& \norm{\nabb n}_{\lh{4}}\norm{\chi}_{\lh{4}}+\norm{n}_{\lh{\infty}}\norm{\nabb\chi}_{\lh{2}}+\norm{\chi}_{\xt{\infty}{2}}\norm{\pr\ga}_{\lh{2}}\\
&\lesssim & \ep+\ep\norm{\pr\ga}_{\tx{\infty}{2}},
\eee
which yields \eqref{eq:gammaL2bis}. This concludes the proof of lemma \ref{lemm:coord}.
\end{proof}

\begin{remark}\label{rem:volumega}
Denoting $|\ga|=\det(\ga_{AB})$, we obtain from \eqref{eq:transportga}:
$$\frac{d}{dt}\log |\ga|=\ga^{AB}\frac{d}{dt}\ga_{AB}=2n\trc\quad$$
or,
\be\lab{eq:volumega}
\frac{d}{dt}\sqrt{|\ga|}=\trc\sqrt{|\ga|}.
\end{equation}
Now,  relative to the coordinates $t, x^1, x^2$, 
$\int_{\ptu} f\dmt=\int \int f\sqrt{|\ga|}\,  dx^1dx^2$, therefore,
$$\frac{d}{dt}\int_{\ptu} f\dmt= 
\int \int \frac{d}{dt}(f\sqrt{|\ga|})\,  dx^1dx^2=
\int_{\ptu}\big(\frac{d}{dt} f +n\trc f\big)\dmt
$$
which proves \eqref{du}.
\end{remark}

\begin{remark}
Since the global coordinate system $x'$ on $\ptu$ is obtained by transporting the coordinate system on $\pou$ along the null geodesics generated by $L$, it requires in particular that null geodesics generating $\H$ have no conjugate points, and that two distinct null geodesics do not intersect. This fact has been proved in section \ref{sec:nointersec} (see \eqref{conclusionfoliation}).
\end{remark}

\subsubsection{A global coordinate system on $\Sit$}\lab{sec:globalcoordsit}

Recall that we have constructed a global coordinate system on $\ptu$ in section \ref{sec:coord}. Let us denote $x'$ such a coordinate system. We obtain a global coordinate system on $\Sit$ as follows. First, recall from \eqref{conclusionfoliation} that $\Sit=\cup\ptu$ so that $u$ is defined on $\Sit$. To any  $p\in\Sit$, we associate the coordinates $(u(p),x'(p))$ where $u(p)$ is the value of the optical function $u$ at $p$, and $x'(p)$ are the coordinate of $p$ in the coordinate system of $\ptu$ constructed in section \ref{sec:coord}. In this coordinate system, the metric $g_t$ on $\Sit$ (i.e. the restriction of $\gg$ on $\Sit$) takes the following form:
\be\label{eq:coordchartsit0}
g_t=\left(\begin{array}{ll}
b^{-2} & 0\\
0& \ga
\end{array}\right),
\ee
where $\ga$ is the induced metric on $\ptu$. Together with the estimate \eqref{boot1} for $b$ and \eqref{eq:coordchartbis} for $\ga$, we obtain
$$\left(\frac{10}{11}+O(\ep)\right)|\xi|^2\leq (g_t)_{ij}(p)\xi^i\xi^j\leq \left(\frac{11}{10}+O(\ep)\right)|\xi|^2,$$
and thus, for $\ep>0$ small enough, we deduce
\be\label{eq:coordchartsit}
\frac{5}{6}|\xi|^2\leq (g_t)_{ij}(p)\xi^i\xi^j\leq \frac{6}{5}|\xi|^2.
\ee
This coordinate system allows us in particular to get a lower bound on the volume radius of the Riemannian manifold $\Sit$. We recall below the definition of the volume radius on a general Riemannian manifold $M$.

\begin{definition}
Let $B_r(p)$ denote the geodesic ball of center $p$ and radius $r$. The volume radius $r_{vol}(p,r)$ at a point $p\in M$ and scales $\leq r$ is defined by
$$r_{vol}(p,r)=\inf_{r'\leq r}\frac{|B_{r'}(p)|}{r^3},$$
with $|B_r|$ the volume of $B_r$ relative to the metric on $M$. The volume radius $r_{vol}(M,r)$ of $M$ on scales $\leq r$ is the infimum of $r_{vol}(p,r)$ over all points $p\in M$.
\end{definition}

Denote by $B^c_r(p)$ the euclidean ball of center $p$ and radius $r$ in the coordinate system \eqref{eq:coordchartsit} of $\Sit$. Then, clearly $B^c_{\frac{5r}{6}}(p)\subset B_r(p)$. Thus, we obtain a lower bound for any $p\in\Sit$:
$$|B_r(p)|\geq \left|B^c_{\frac{5r}{6}}(p)\right| =\int_{B^c_{\frac{5r}{6}}(p)}\sqrt{|g_t|}dudx'\geq \frac{5}{6}\left|B^c_{\frac{5r}{6}}(p)\right|\geq \left(\frac{5}{6}\right)^4r^3,$$
which yields the following lower bound on the volume radius of $\Sit$ at scales $\leq 1$:
\be\label{volradbound}
r_{vol}(\Sit,1)\geq \left(\frac{5}{6}\right)^4. 
\ee

\subsubsection{Harmonic coordinates on $\Sit$}\lab{sec:coordharm}

We will need a second coordinate system on $\Sit$ since the coordinate system in \eqref{eq:coordchartsit} is not regular enough for some of the applications we have in mind. Indeed,  
we only control some Christoffel symbols in this coordinate system (see for example \eqref{eq:gammaL2bis}), but no second order derivative of the metric coefficients. The second coordinate system we have in mind are the harmonic coordinates. To obtain an appropriate covering of $\Sit$ by harmonic coordinates, we rely on the following general result based on Cheeger-Gromov convergence of Riemannian manifolds, see \cite{An} or Theorem 5.4 in \cite{Pe}. 

\begin{theorem}\label{th:coordharm}
Given $c_1>0, c_2>0, c_3>0$, there exists $r_0>0$ such that any 3-dimensional, complete, Riemannian manifold $(M,g)$ with $\norm{R}_{L^2(M)}\leq c_1$ and volume radius at scales $\leq 1$ bounded from below by $c_2$, i.e. $r_{vol}(M,1)\geq c_2$, verifies the following property:

Every geodesic ball $B_r(p)$ with $p\in M$ and $r\leq r_0$ admits a system of harmonic coordinates $x=(x_1,x_2,x_3)$ relative to which we have
\be\label{coorharmth1}
(1+c_3)^{-1}\d_{ij}\leq g_{ij}\leq (1+c_3)\d_{ij},
\ee
and
\be\label{coorharmth2}
r\int_{B_r(p)}|\partial^2g_{ij}|^2\sqrt{|g|}dx\leq c_3.
\ee
\end{theorem}

To apply Theorem \ref{th:coordharm}, we need to bound the curvature tensor $R_t$ on $\Sit$ in $\lsit{\infty}{2}$. Since $\Sit$ has dimension 3, it is enough to bound its Ricci tensor. Now, we have the following formula relating the Ricci tensor on $\Sit$ to the curvature tensor $\rr$ on $\mathcal M$ and $k$:
$$(R_t)_{ij}=k_{il}k^l_j+\rr_{iTjT}$$
which yields:
\be\label{curvriccisit}
\norm{R_t}_{\lsit{\infty}{2}}\leq \norm{\rr}_{\lsit{\infty}{2}}+\no(k)^2.
\ee
The curvature bound \eqref{curvflux1}, the bootstrap assumption \eqref{boot3} and \eqref{curvriccisit} imply:
\be\label{curvriccisit1}
\norm{R_t}_{\lsit{\infty}{2}}\lesssim \ep.
\ee
Let $\d>0$. \eqref{curvriccisit1} together with the volume lower bound \eqref{volradbound} and Theorem \eqref{th:coordharm} yields the existence of $r_0(\d)>0$ and a finite covering of $\Sit$ by geodesic balls  of radius $r_0(\d)$ such that each geodesic ball in the covering admits a system of harmonic coordinates $x=(x_1,x_2,x_3)$ relative to which we have
\be\label{coorharmth1bis}
(1+\d)^{-1}\d_{ij}\leq g_{ij}\leq (1+\d)\d_{ij},
\ee
and
\be\label{coorharmth2bis}
r_0(\d)\int_{B_{r_0}(p)}|\partial^2g_{ij}|^2\sqrt{|g|}dx\leq \d.
\ee

\begin{remark}
$\Sit$ is asymptotically flat and therefore admits a harmonic coordinates system in a neighborhood of infinity. Therefore, the problem of covering $\Sit$ with harmonic coordinates charts is reduced to a compact region which explains why we may chose finitely many harmonic coordinates charts covering $\Sit$ and satisfying \eqref{coorharmth1bis} \eqref{coorharmth2bis}. 
\end{remark}

\subsection{Bound on the Gauss curvature $K$}\lab{sec:boundgaussK}

The following proposition will be crucial to obtain useful strong Bernstein and Bochner inequalities.
\begin{proposition}\label{propK} 
Let $K$ the gauss curvature on $\ptu$. Then, $K$ satisfies the following bounds: 
\be\lab{estgauss1}
\norm{K}_{\lh{2}}\lesssim\ep
\ee 
and 
\be\lab{estgauss2}
\norm{\Lambda^{-\half}K}_{\tx{\infty}{2}}\lesssim\ep.
\ee
\end{proposition}

The proof of Proposition \ref{propK} is postponed to section \ref{sec:propK}. The following consequence of Proposition \ref{propK} will be useful in the sequel.  
Proposition \ref{propK} and \eqref{La8} with the choice $a=1/2$ imply:
\begin{equation}\label{ad65}
\norm{K_{\frac{1}{2}}}_{L^\infty(0,1)}=\norm{\La^{-\frac{1}{2}}K}_{\tx{\infty}{2}}\lesssim\ep,
\end{equation}
where $K_{1/2}$ has been defined in \eqref{La7}. Together with \eqref{eq:strongbernscalar} and \eqref{eq:strong-Bern-0}  with the choice $\gamma=1/2$, we obtain for any scalar function $f$ on $\ptu$ and any $j\ge 0$ the following sharp Bernstein inequality:
\bea
\|P_j f\|_{\lpt{\infty}}&\lesssim & 2^j\|f\|_{\lpt{2}},\label{eq:strongbernscalarbis}\\
\|P_{<0} f\|_{\lpt{\infty}}&\lesssim &  
 \|f\|_{\lpt{2}}.\label{eq:strong-Bern-0bis}
\eea

Also, \eqref{ad65} and \eqref{eq:Bochconseq} with the choice $\gamma=1/2$ imply the following Bochner inequality:
\begin{equation}\label{eq:Bochconseqbis}
\int_{\ptu} |\nabb^2 f|^2\lesssim \int_{\ptu} |\lap f|^2  + \ep\int_{\ptu} |\nabb f|^2.
\end{equation}

Finally, using the Gagliardo-Nirenberg inequality \eqref{eq:GNirenberg} and \eqref{eq:Bochconseqbis}, 
we obtain for any $2\leq p<+\infty$, any $j\geq 0$, and any scalar function $f$:
\bea
\lab{chat1}\norm{\nabb P_jf}_{\lpt{p}}&\les& \norm{\nabb^2P_jf}^{1-\frac{2}{p}}_{\lpt{2}}\norm{\nabb P_jf}^{\frac{2}{p}}_{\lpt{2}}\\
\nn&\les & (\norm{\lap P_jf}_{\lpt{2}}+\norm{\nabb P_jf}_{\lpt{2}})^{1-\frac{2}{p}}2^{\frac{2j}{p}}\norm{f}_{\lpt{2}}^{\frac{2}{p}}\\
\nn&\les & 2^{2(1-\frac{1}{p})j}\norm{f}_{\lpt{2}}.
\eea
Taking the dual of \eqref{chat1}, we obtain for any $1<p\leq 2$, any $j\geq 0$, and any $\ptu$-tangent 1-form $F$:
\be\lab{chat2}
\norm{P_j\divb(F)}_{\lpt{p}}\les 2^{\frac{2j}{p}}\norm{f}_{\lpt{2}}.
\ee

\begin{remark}
\eqref{eq:strongbernscalarbis} and \eqref{eq:strong-Bern-0bis} only hold for scalar functions $f$ on $\ptu$. For tensors $F$ on $\ptu$, and for arbitrary $2\leq p<+\infty$, we have the following sharp Bernstein inequality (see \cite{LP} for the proof):
\bea
\|P_j F\|_{\lpt{\infty}}&\lesssim & 2^j(1+2^{-\frac{j}{p}}\norm{K}^{\frac{1}{p}}_{\lpt{2}}+2^{-\frac{j}{p-1}}\norm{K}^{\frac{1}{p-1}}_{\lpt{2}})\|F\|_{\lpt{2}},\label{eq:strongberntensor}\\
\|P_{<0} F\|_{\lpt{\infty}}&\lesssim &  
 (1+\norm{K}^{\frac{1}{p}}_{\lpt{2}}+\norm{K}^{\frac{1}{p-1}}_{\lpt{2}})\|F\|_{\lpt{2}}.\label{eq:strong-Bern-0tensor}
\eea
\end{remark}

\subsection{Estimates for the lapse $n$}\lab{sec:impn}

The goal of this section is to improve the estimate for $n$ in the bootstrap assumptions \eqref{boot1} \eqref{boot2}.

\subsubsection{Estimates for $k$ on $\Sit$}

We recall the definition of $E$ and $H$ used in the standard electric-magnetic decomposition of the tensor $\rr$ (see \cite{ChKl} chapter 7). We have:
\be\lab{electricmagnetic}
E_{\a\b}=\rr_{\mu\a\nu\b}T^\mu T^\nu,\,H_{\a\b}={}^*\rr_{\mu\a\nu\b}T^\mu T^\nu.
\ee
Then, $k$ satisfies the following symmetric Hodge system on $\Sit$:
\be\lab{eqksit}
\left\{\begin{array}{l}
\textrm{curl} k_{ij}=H_{ij},\\
\nabla^jk_{ij}=0,\\
\textrm{Tr}k=0,
\end{array}\right.
\ee
where curl$k_{ij}=\frac{1}{2}(\in_i^{lm}\nabla_lk_{mj}+\in_j^{lm}\nabla_lk_{mi})$. The solution $k$ of the symmetric Hodge system \eqref{eqksit} in 3 dimensions satisfies the following estimate (see \cite{ChKl} chapter 4):
\be\lab{eqksit1}
\int_{\Sit}\left(|\nabla k|^2+3(R_t)_{jl}k^{ij}k_i^l-\frac{1}{2}R_t|k|^2\right)d\Sit=\int_{\Sit}|H|^2d\Sit.
\ee
The bound \eqref{curvflux1} on $\rr$, the bound \eqref{curvriccisit} on $R_t$, the definition of $H$ \eqref{electricmagnetic} and \eqref{eqksit} yield:
\be\lab{eqksit2}
\norm{\nabla k}^2_{L^2(\Sit)}\lesssim \ep\norm{k}_{L^6(\Sit)}^2+\ep^2
\ee
which together with the Sobolev embedding \eqref{sobineqsit} implies:
\be\lab{eqksit3}
\norm{\nabla k}_{\lsit{\infty}{2}}\lesssim \ep.
\ee

\begin{remark}\lab{decayasymptot}
To obtain \eqref{eqksit2} from \eqref{eqksit1}, we rely on $L^2(\Sit)$ bounds for $R_t$ and $\rr$. This is enough on compacts, but not at infinity. Fortunately, $\Sit$ is asymptotically flat so that $R_t$ and $\rr$ decay at least like $r^{-3}$ at infinity which is fast enough to obtain \eqref{eqksit2}. Furthermore, the fact that $\Sit$ is asymptotically flat also implies that $k$ decays at least like $r^{-2}$ at infinity which together with the Sobolev embedding \eqref{sobineqsit} and the estimate \eqref{eqksit3} yields:
\be\lab{eqksit4}
\norm{k}_{\lsit{\infty}{2}}\lesssim \ep.
\ee 
\end{remark}

\subsubsection{Improvement of the bootstrap assumptions on $n$}

We first improve the $L^\infty$ bound for $n-1$. Using the Sobolev embedding \eqref{sobineqsit}, \eqref{sobinftysit} and the consequence of the Bochner identity \eqref{prop:bochsit}, we have:
\bee
\norm{n-1}_{L^\infty(\mathcal{M})}&\lesssim& \norm{\nabla n}_{\lsit{\infty}{6}}+ \norm{n-1}_{\lsit{\infty}{6}}\\
&\lesssim& \norm{\nabla^2n}_{\lsit{\infty}{2}}+ \norm{\nabla n}_{\lsit{\infty}{2}}\\
&\lesssim& \norm{\Delta n}_{\lsit{\infty}{2}}+ \norm{\nabla n}_{\lsit{\infty}{2}}.
\eee
Together with the equation of the lapse \eqref{lapsen1} on $\Sit$, the bootstrap assumption \eqref{boot1}  , the Sobolev embedding \eqref{sobineqsit}, and the estimates \eqref{eqksit3} \eqref{eqksit4}, we obtain:
\bea
\lab{lapn1}\norm{n-1}_{L^\infty(\mathcal{M})}&\lesssim& \norm{n|k|^2}_{\lsit{\infty}{2}}+ \norm{\nabla n}_{\lsit{\infty}{2}}\\
\nn&\lesssim& \norm{n}_{L^\infty(\mathcal{M})}\norm{k}^2_{\lsit{\infty}{4}}+ \norm{\nabla n}_{\lsit{\infty}{2}}\\
\nn&\lesssim& \ep^2+ \norm{\nabla n}_{\lsit{\infty}{2}}
\eea
Multiplying the equation for the lapse \eqref{lapsen1} by $n-1$ on $\Sit$, integrating by parts yields:
\bee
\norm{\nabla n}_{L^2(\Sit)}^2 &=&\int_{\Sit}|k|^2n(n-1)d\Sit\lesssim \norm{k}^2_{L^2(\Sit)}\norm{n}_{L^\infty(\Sit)}\norm{n-1}_{L^\infty(\Sit)}\\
&\lesssim& D^2\ep^2\norm{n-1}_{L^\infty(\Sit)}
\eee
where we used the bootstrap assumption \eqref{boot1} and \eqref{boot3}. Together with \eqref{lapn1}, this yields:
\be\lab{lapn2}
\norm{n-1}_{L^\infty(\mathcal{M})}+\norm{\nabla n}_{\lsit{\infty}{2}} \lesssim \ep.
\ee
Furthermore, the equation for the lapse \eqref{lapsen1}, the Bochner identity \eqref{prop:bochsit},  together with the estimates \eqref{eqksit3} \eqref{eqksit4} and \eqref{lapn2} yields:
\bea
\lab{lapn3}\norm{\nabla^2 n}_{\lsit{\infty}{2}} &\lesssim&\norm{\Delta n}_{\lsit{\infty}{2}}+\norm{\nabla n}_{\lsit{\infty}{2}} \\
\nn&\lesssim&\norm{n|k|^2}_{\lsit{\infty}{2}}+\ep\\
\nn&\lesssim& \norm{n}_{L^\infty(\mathcal{M})}\norm{k}^2_{\lsit{\infty}{4}}+\ep\\
\nn&\lesssim& \ep. 
\eea
Using \eqref{hehehehe0bis}, \eqref{lapn2} and \eqref{lapn3}, we also obtain:
\be\lab{lapn3bis}
\norm{\nabla n}_{\tx{\infty}{2}} \lesssim \ep.
\ee

We differentiate the equation of the lapse \eqref{lapsen1} with respect to $\nabla$. We obtain:
\be\lab{lapsen2}
\Delta(\nabla n)=\nabla(n|k|^2)+[\Delta,\nabla]n= |k|^2\nabla n+2nk\nabla k+R_t\nabla n,
\ee
which together with the bound \eqref{curvriccisit1} on $R_t$, the Sobolev embedding \eqref{sobineqsit}, and the estimates \eqref{eqksit3}, \eqref{eqksit4}, \eqref{lapn2} and \eqref{lapn3}, yields:
\bea
\lab{lapn4}&&\norm{\Delta(\nabla n)}_{\lsit{\infty}{\frac{3}{2}}}\\
\nn&\lesssim& \norm{k}_{\lsit{\infty}{3}}^2\norm{\nabla n}_{\lsit{\infty}{2}}+\norm{n}_{L^\infty(\mathcal{M})}\norm{k}_{\lsit{\infty}{6}}\norm{\nabla k}_{\lsit{\infty}{2}}\\
\nn&&+\norm{R_t}_{\lsit{\infty}{2}}\norm{\nabla n}_{\lsit{\infty}{6}}\\
\nn&\lesssim& \ep.
\eea
\eqref{prop:bochsit2}, \eqref{lapn3} and \eqref{lapn4} imply:
\be\lab{lapn5}
\norm{\nabla^3n}_{\lsit{\infty}{\frac{3}{2}}}\lesssim \norm{\Delta(\nabla n)}_{\lsit{\infty}{\frac{3}{2}}}+\norm{\nabla^2n}_{\lsit{\infty}{2}}\lesssim\ep.
\ee

We now differentiate the equation of the lapse \eqref{lapsen1} with respect to $\dd_T$. Using the commutation formula \eqref{commsit}, we obtain:
\bea
\label{lapn6}\Delta(T(n))&=&T(\Delta n)+ [\Delta,\dd_T]n\\
\nn&=&|k|^2T(n)+nk\dd_Tk-2k\nabla^2n+2n^{-1}\nabla n\nabla T(n)+n^{-1}\Delta n T(n)\\
\nn&&-\nabla k\nabla n-2n^{-1}k\nabla n\nabla n.
\eea
We need an estimate for $\dd_Tk$. We have the following identity (see \cite{ChKl} chapter 11):
$$\dd_Tk_{ij}=-n\nabla^2n_{ij}+2n^{-1}\nabla_in\nabla_jn+(R_t)_{ij}$$
which together with the bound \eqref{curvriccisit1}, \eqref{lapn2} and \eqref{lapn3} yields:
\be\lab{lapn7}
\norm{\dd_Tk}_{\lsit{\infty}{2}}\lesssim \norm{n\nabla^2n}_{\lsit{\infty}{2}}+\norm{\nabla n}^2_{\lsit{\infty}{4}}+\norm{R_t}_{\lsit{\infty}{2}}\lesssim\ep.
\ee
We multiply \eqref{lapn6} by $T(n)$ and integrate by parts, which yields:
\bea
\label{lapn8}&&\norm{\nabla(T(n))}^2_{\lsit{\infty}{2}}\\
\nn&\lesssim&\bigg(\norm{k}^2_{\lsit{\infty}{3}}\norm{T(n)}_{\lsit{\infty}{6}}+\norm{n}_{L^{\infty}}\norm{k}_{\lsit{\infty}{3}}\norm{\dd_Tk}_{\lsit{\infty}{2}}\\
\nn&&+\norm{k}_{\lsit{\infty}{3}}\norm{\nabla^2n}_{\lsit{\infty}{2}}
+\norm{n^{-1}\nabla n}_{\lsit{\infty}{3}}\norm{\nabla T(n)}_{\lsit{\infty}{2}}\\
&&\nn +\norm{n^{-1}\Delta n}_{\lsit{\infty}{\frac{3}{2}}}\norm{T(n)}_{\lsit{\infty}{6}}+\norm{\nabla k}_{\lsit{\infty}{2}}\norm{\nabla n}_{\lsit{\infty}{3}}\\
&&\nn +\norm{n^{-1}k}_{\lsit{\infty}{2}}\norm{\nabla n}^2_{\lsit{\infty}{6}}\bigg)\norm{T(n)}_{\lsit{\infty}{6}}\\
\nn&\lesssim&\bigg(\ep+\ep\norm{T(n)}_{\lsit{\infty}{6}}+\ep\norm{\nabla T(n)}_{\lsit{\infty}{2}}\bigg)\norm{T(n)}_{\lsit{\infty}{6}}
\eea
where we used \eqref{eqksit3}, \eqref{eqksit4}, \eqref{lapn2}, \eqref{lapn3} and \eqref{lapn7} in the last inequality. \eqref{lapn8} and the Sobolev embedding \eqref{sobineqsit} imply:
\be\label{lapn9}
\norm{\nabla(T(n))}_{\lsit{\infty}{2}}\lesssim\ep.
\ee
We now estimate \eqref{lapn6} in $\lsit{\infty}{\frac{3}{2}}$:
\bea
\label{lapn10}&&\norm{\Delta(T(n))}_{\lsit{\infty}{\frac{3}{2}}}\\
\nn&\lesssim& \norm{k}^2_{\lsit{\infty}{4}}\norm{T(n)}_{\lsit{\infty}{6}}+\norm{nk}_{\lsit{\infty}{6}}\norm{\dd_Tk}_{\lsit{\infty}{2}}\\
&&\nn+\norm{k}_{\lsit{\infty}{6}}\norm{\nabla^2n}_{\lsit{\infty}{2}}+\norm{n^{-1}\nabla n}_{\lsit{\infty}{6}}\norm{\nabla T(n)}_{\lsit{\infty}{2}}\\
&&\nn+\norm{n^{-1}\Delta n}_{\lsit{\infty}{2}}\norm{T(n)}_{\lsit{\infty}{6}}+\norm{\nabla k}_{\lsit{\infty}{6}}\norm{\nabla n}_{\lsit{\infty}{2}}\\
&&\nn+\norm{n^{-1}k}_{\lsit{\infty}{3}}\norm{\nabla n}_{\lsit{\infty}{6}}^2\\
&\lesssim &\nn \ep
\eea
where we used \eqref{eqksit3}, \eqref{eqksit4}, \eqref{lapn2}, \eqref{lapn3}, \eqref{lapn7} and \eqref{lapn9} in the last inequality. \eqref{prop:bochsit2}, \eqref{lapn9} and \eqref{lapn10} imply:
\be\lab{lapn11}
\norm{\nabla^2T(n)}_{\lsit{\infty}{\frac{3}{2}}}\lesssim \norm{\Delta(T(n))}_{\lsit{\infty}{\frac{3}{2}}}+\norm{\nabla T(n)}_{\lsit{\infty}{2}}\lesssim\ep.
\ee
Finally, \eqref{hehehehe0}, \eqref{lapn3}, \eqref{lapn5}, \eqref{lapn9} and \eqref{lapn11} imply:
\be\lab{lapn12}
\norm{\nabla^2n}_{\tx{\infty}{2}}+\norm{\nabla T(n)}_{\tx{\infty}{2}}\lesssim \ep.
\ee
Note that \eqref{lapn2}, \eqref{lapn3bis} and \eqref{lapn12} improve the estimates for $n$ in the bootstrap assumptions \eqref{boot1} \eqref{boot2}.

\subsubsection{An $L^\infty(\mathcal{M})$ estimate for $\nab n$}

In view of the embedding \eqref{bsit}, we have
$$\norm{\nab n}_{L^\infty(\Sit)}\les \norm{\nab n}_{\bsit^{\frac{3}{2}}}.$$
Together with the estimate \eqref{bsit1}, this yields
\be\lab{jesaisplusquoimettre}
\norm{\nab n}_{L^\infty(\Sit)}\les \norm{n-1}_{\bsit^{\frac{5}{2}}}.
\ee
Now, in view of the definition of the Besov spaces $\bsit^a$ and the finite band property for $Q_j$, we have
$$\norm{n-1}_{\bsit^{\frac{5}{2}}}\les \norm{\Delta n}_{\bsit^{\frac{1}{2}}}.$$
Injecting the equation for the lapse \eqref{lapsen1} in the right-hand side, we obtain
$$\norm{n-1}_{\bsit^{\frac{5}{2}}}\les \norm{n|k|^2}_{\bsit^{\frac{1}{2}}}.$$
Using the product estimates \eqref{bsit6} and \eqref{bsit11}, we deduce
\bee
\norm{n-1}_{\bsit^{\frac{5}{2}}}&\les& (\norm{n}_{L^{\infty}(\mathcal{M})}+\norm{\nab n}_{\lsit{\infty}{3}})\norm{|k|^2}_{\bsit^{\frac{1}{2}}}\\
&\les& (\norm{n}_{L^{\infty}(\mathcal{M})}+\norm{\nab n}_{\lsit{\infty}{3}})(\norm{\nabla k}_{L^2(\Sit)}+\norm{k}_{L^2(\Sit)})^2.
\eee
Together with the estimates \eqref{eqksit3} and \eqref{eqksit4} for $k$, the estimates \eqref{lapn2} and \eqref{lapn3} for $n$, and the Sobolev embedding \eqref{sobineqsit}, we deduce
$$\norm{n-1}_{\bsit^{\frac{5}{2}}}\les \ep.$$
In view of \eqref{jesaisplusquoimettre}, we finally obtain
$$\norm{\nab n}_{L^\infty(\Sit)}\les\ep.$$

\subsection{Estimates for $k$ on $\mathcal{H}_u$}\lab{sec:imponk}

The goal of this section is to improve the estimate for $\no(k)$, $\norm{\ddb_{\lb}\kep}_{\lh{2}}$ and $\norm{\ddb_{\lb}\d}_{\lh{2}}$ given by the bootstrap assumption \eqref{boot3}. The improvement of $\norm{\kepb}_{\xt{\infty}{2}}$ and $\norm{\db}_{\xt{\infty}{2}}$ is postponed to section \ref{sec:tracenormk}. Note that the bootstrap assumption \eqref{boot3} yields:
\be\lab{bootconseqk}
\no(\d)+\no(\kep)+\no(\eta)\leq D\ep. 
\ee

\subsubsection{A Hodge type system on $\H_u$}

The first step is to derive a Hodge type system analog to \eqref{eqksit} on $\H$. We first recall the formula p. 106/107 in \cite{ChKl} relating the derivatives of $k$ to the derivatives of $\eta, \kep, \d$:
\bee
\nn\nabla_Nk_{NN} &=&\nabla_N\d+2b^{-1}\nabb b\c\kep\\
\nn\nabla_Bk_{NA}&=&\nabb_B\kep_A+\frac{3}{4}\d\trt\ga_{AB}-\half\trt\het_{AB}-\het_{AC}\hth_{CB}+\frac{3}{2}\d\hth_{AB}\\
\nn\nabla_Ck_{AB}&=&\nabb_C\eta_{AB}+\th_{AC}\kep_B+\theta_{BC}\kep_A\\
\nn\nabla_Ak_{NN} &=&\nabla_A\d-2\th_{AB}\kep_B\\
\nn\nabla_Nk_{AN}&=&\nabb_N\kep_A-\frac{3}{2}\d b^{-1}\nabb_Ab+b^{-1}\nabb_Bb\het_{AB}\\
\nn\nabla_Nk_{AB}&=&\nabb_N\eta_{AB}-b^{-1}\nabb_Ab\kep_B-b^{-1}\nabb_Bb\kep_A
\eee
where $\th$ is the second fundamental form of $\ptu$ in $\Sit$. Since $L=T+N$, $\th$ is connected to the second fundamental form $k$ of $\Sit$ and the null second fundamental form $\chi$ of $\ptu$ through the formula:
\be\lab{def:theta}
\theta_{AB}=\chi_{AB}+\eta_{AB}.
\ee
Together with the Hodge system \eqref{eqksit}, we obtain:
\be\lab{hodgkh}
\begin{array}{lll}
\divb\eta_A+\nabb_N\kep_A&=&-\th_{AB}\kep_B-\trt\kep_A+\frac{3}{2}\d b^{-1}\nabb_Ab-b^{-1}\nabla_Bb\het_{AB}\\
\divb\kep+\nabla_N\d&=&-\frac{3}{2}\d\trt+\het\hth-2b^{-1}\nabb_A b\kep_A\\
\nabb_C\eta_{AB}-\nabb_B\eta_{AC}&=&\rr_{ATBC}-\th_{AC}\kep_B+\th_{AB}\kep_C\\
\nabb_N\eta_{AB}-\nabb_B\kep_A&=&\rr_{ATBN}+b^{-1}\nabb_Ab\kep_B+b^{-1}\nabb_Bb\kep_A+\frac{3}{4}\d\trt\ga_{AB}\\
&&-\half\trt\het_{AB}-\het_{AC}\hth_{CB}+\frac{3}{2}\d\hth_{AB}\\
\nabb_N\kep_A-\nabla_A\d&=&\rr_{NTAN}+\frac{3}{2}\d b^{-1}\nabb_Ab-b^{-1}\nabb_Bb\het_{AB}-2\th_{AB}\kep_B
\end{array}
\ee
In order to obtain a Hodge system on $\H$, we need to replace the derivatives in the $N$-direction with derivatives in the $L$-direction in \eqref{hodgkh}. We use the following formula for $\dd_T\d, \ddb_T\ep, \ddb_T\eta$ (see \cite{ChKl} p. 337):
\be\lab{kinTderivatives}
\begin{array}{lll}
\dd_T\d&=&-n^{-1}\nabla^2_Nn+\rho+\d^2-\z\zb+\z\kep-\zb\kep\\
\ddb_T\kep&=&-n^{-1}\nabb\nabla_Nn+\half(\bb+\b)+b^{-1}\nabb bn^{-1}\nabla_Nn-\frac{3}{2}(\z-n^{-1}\nabb n)\d\\
&&+(\z-n^{-1}\nabb n+\kep)\het+\half\d\kep\\
\ddb_T\eta&=&-n^{-1}\nabb^2n+\frac{1}{4}(\ab+\a)-\d\eta+\kep\kep-(\z-n^{-1}\nabb n)\kep
\end{array}
\ee
Now, \eqref{hodgkh} and \eqref{kinTderivatives} yield:
\be\lab{hodgkh1}
\begin{array}{lll}
\divb\eta_A+\ddb_L\kep_A&=&\half(\bb+\b)-n^{-1}\nabb_A\nabla_Nn+b^{-1}\nabb_A bn^{-1}\nabla_Nn\\
&&-\frac{3}{2}(\z_A-n^{-1}\nabb_A n)\d-\th_{AB}\kep_B-\trt\kep_A+\frac{3}{2}\d b^{-1}\nabb_Ab-b^{-1}\nabla_Bb\het_{AB}\\
\divb\kep+\dd_L\d&=& \rho-n^{-1}\nabla^2_Nn+\d^2-\z\zb+\z\kep-\zb\kep-\frac{3}{2}\d\trt+\het\hth-2b^{-1}\nabb_A b\kep_A\\
\nabb_C\eta_{AB}-\nabb_B\eta_{AC}&=&-\in_{BC}{}^*\b_A+\in_{BC}{}^*\bb_A-\th_{AC}\kep_B+\th_{AB}\kep_C\\
\ddb_L\eta_{AB}-\nabb_B\kep_A&=&\frac{1}{2}\a_{AB}-n^{-1}\nabb^2_{AB}n-\half\d\eta_{AB}+2\kep_A\kep_B-(\z_A-n^{-1}\nabb_A n)\kep_B\\
&&-(\z_B-n^{-1}\nabb_Bn)\kep_A+b^{-1}\nabb_Ab\kep_B+b^{-1}\nabb_Bb\kep_A+\frac{3}{4}\d\trt\ga_{AB}\\
&&-\half\trt\het_{AB}-\het_{AC}\hth_{CB}+\frac{3}{2}\d\hth_{AB}\\
\ddb_L\kep_A-\nabla_A\d&=&\b_A-n^{-1}\nabb_A\nabla_Nn+b^{-1}\nabb_A bn^{-1}\nabla_Nn-\frac{3}{2}(\z_A-n^{-1}\nabb_A n-b^{-1}\nabb_Ab)\d\\
&&-b^{-1}\nabb_Bb\het_{AB}-2\th_{AB}\kep_B
\end{array}
\ee
Using the curvature bound \eqref{curvflux1}, the bootstrap assumptions \eqref{boot1}-\eqref{boot6}, the bound \eqref{bootconseqk} on $\eta, \kep, \d$ together with \eqref{hodgkh1}, we obtain:
\be\lab{hodgkh2}
\norm{\divb\eta_A+\ddb_L\kep_A}_{\li{\infty}{2}}+\norm{\divb\kep+\dd_L\d}_{\li{\infty}{2}}\lesssim\kep+D^2\ep^2\lesssim\ep
\ee
and
\bea
&&\nn\norm{\nabb_C\eta_{AB}-\nabb_B\eta_{AC}}_{\li{\infty}{2}}+\norm{\ddb_L\eta_{AB}-\nabb_B\kep_A}_{\li{\infty}{2}}+\norm{\ddb_L\kep_A-\nabla_A\d}_{\li{\infty}{2}}\\
\lab{hodgkh3}&\lesssim&\kep+D^2\ep^2\lesssim\ep.
\eea

\subsubsection{Estimates for $\eta, \kep, \d$}

We start with the following identity:
\bea
\lab{hodgk4}&&\int_{\H}|\nabb_C\eta_{AB}-\nabb_B\eta_{AC}|^2+2|\ddb_L\eta_{AB}-\nabb_B\kep_A|^2+2|\ddb_L\kep_A-\nabla_A\d|^2\\
\nn&=&2\int_{\H}|\nabb\eta|^2+|\ddb_L\eta|^2+|\nabb\kep|^2+|\divb\kep|^2+2|\ddb_L\kep|+|\nabb\d|^2+|\dd_L\d|^2\\
\nn&&-2\int_{\H}\nabb_C\eta_{AB}\nabb_B\eta_{AC}-4\int_{\H}\ddb_L\eta_{AB}\nabb_B\kep_A-4\int_{\H}\ddb_L\kep_A\nabla_A\d\\
\nn&&-2\int_{\H}|\divb\kep|^2-2\int_{\H}|\ddb_L\kep|^2-2\int_{\H}|\dd_L\d|^2.
\eea
We compute the last terms in the right hand side of \eqref{hodgk4} using integration by parts and the coarea formula \eqref{coarea} on $\H$:
\bee
&&\nn-2\int_{\H}\nabb_C\eta_{AB}\nabb_B\eta_{AC}-4\int_{\H}\ddb_L\eta_{AB}\nabb_B\kep_A-4\int_{\H}\ddb_L\kep_A\nabla_A\d\\
&&\nn-2\int_{\H}|\divb\kep|^2-2\int_{\H}|\ddb_L\kep|^2-2\int_{\H}|\dd_L\d|^2\\
&=& 2\int_{\H}\eta^{AB}(\nabb_B\divb\eta_{A}+[\nabb_C,\nabb_B]\eta_{AC})
+4\int_{\H}\eta_{AB}(\nabb_B\ddb_L\kep_A+[\ddb_L,\nabb_B]\kep_A)\\
&&-4\int_{\H}(-n^{-1}L(n)-\db  +\trc)\eta_{AB}\nabb_B\kep_A-4\int_{\ptu}\eta_{AB}\nabb_B\kep_A\\
&&+4\int_{\H}\d(\ddb_L\divb\kep+[\nabb_A,\ddb_L]\kep_A)
\nn-2\int_{\H}|\divb\kep|^2-2\int_{\H}|\ddb_L\kep|^2-2\int_{\H}|\dd_L\d|^2\\
&=& -2\int_{\H}|\divb\eta+\ddb_L\kep|^2-2\int_{\H}|\divb\kep+\dd_L\d|^2
+2\int_{\H}\eta_{AB}\rr_{AB}{}^{CD}\eta_{CD}\\
&&+4\int_{\H}\eta_{AB}(\chi\nabb\eta-n^{-1}\nabb n\ddb_L\eta+(\chi\kepb+{}^*\b)\eta)
+4\int_{\H}\d(\chi\nabb\kep-n^{-1}\nabb n\ddb_L\kep+(\chi\kepb+{}^*\b)\kep)\\
&&-4\int_{\H}(-n^{-1}L(n)-\db  +\trc)\eta_{AB}\nabb_B\kep_A-4\int_{\ptu}\eta_{AB}\nabb_B\kep_A\\
&&-4\int_{\H}(-n^{-1}L(n)-\db  +\trc)\d\divb\kep+4\int_{\ptu}\d\divb\kep
\eee
Together with the curvature bound \eqref{curvflux1} and the bootstrap assumptions \eqref{boot1}-\eqref{boot6}, we obtain:
\bea
\lab{hodgk5}&&\nn-2\int_{\H}\nabb_C\eta_{AB}\nabb_B\eta_{AC}-4\int_{\H}\ddb_L\eta_{AB}\nabb_B\kep_A-4\int_{\H}\ddb_L\kep_A\nabla_A\d\\
\nn&&\nn-2\int_{\H}|\divb\kep|^2-2\int_{\H}|\ddb_L\kep|^2-2\int_{\H}|\dd_L\d|^2\\
\nn&=& -2\int_{\H}|\divb\eta+\ddb_L\kep|^2-2\int_{\H}|\divb\kep+\dd_L\d|^2-4\int_{\ptu}\eta_{AB}\nabb_B\kep_A+4\int_{\ptu}\d\divb\kep+O(D^3\ep^3).
\eea
The bounds \eqref{hodgkh2} \eqref{hodgkh3} together with \eqref{hodgk4} and \eqref{hodgk5} yield:
\be\lab{hodgk6}
\int_{\H}|\nabb\eta|^2+|\nabb\kep|^2+|\ddb_L\kep|^2+|\nabb\d|^2+|\dd_L\d|^2\lesssim \left|\int_{\ptu}\eta_{AB}\nabb_B\kep_A\right|+\left|\int_{\ptu}\d\divb\kep\right|+\ep^2.
\ee

We now state a lemma which will allow us to control the integrals over $\ptu$ in \eqref{hodgk6}.

\begin{lemma}\lab{ineqtraceptu}
Let $F$ and $G$ tensors on $\Sit$ such that $F\c\nabb G$ is a scalar. Then, we have:
\be\lab{ineqtraceptu1}
\left|\int_{\ptu}F\c\nabb G\right|\lesssim \norm{F}_{H^1(\Sit)}\norm{G}_{H^1(\Sit)}. 
\ee
\end{lemma}

The proof of Lemma \ref{ineqtraceptu} is postponed to section \ref{sec:gowinda}. We now use Lemma \ref{ineqtraceptu} to obtain estimates for $\eta, \kep, \d$. The bounds \eqref{eqksit3} \eqref{eqksit4} for $k$  on $\Sit$ together with \eqref{ineqtraceptu1} yield the following estimate:
$$\left|\int_{\ptu}\eta_{AB}\nabb_B\kep_A\right|+\left|\int_{\ptu}\d\divb\kep\right|\lesssim\norm{k}_{H^1(\Sit)}^2\lesssim \ep^2.$$
Together with \eqref{hodgk6}, this implies:
\be\lab{hodgk7}
\int_{\H}|\nabb\eta|^2+|\nabb\kep|^2+|\ddb_L\kep|^2+|\nabb\d|^2+|\dd_L\d|^2\lesssim \ep^2.
\ee
Using also \eqref{sobineq2bis}, we finally obtain:
\be\lab{hodgk8}
\no(\eta)+\no(\kep)+\no(\d)\lesssim \ep.
\ee
Now, in view of \eqref{kinTderivatives}, we have:
\be\lab{hodgk9}
\norm{\dd_T\d}_{\li{\infty}{2}}+\norm{\ddb_T\kep}_{\li{\infty}{2}}\lesssim \ep+D^2\ep^2\lesssim \ep
\ee
where we have used the curvature bound \eqref{curvflux1} and the bootstrap assumptions \eqref{boot1}-\eqref{boot6}. \eqref{hodgk8} and \eqref{hodgk9} yield:
\be\lab{hodgk10}
\norm{\dd_{\lb}\d}_{\li{\infty}{2}}+\norm{\ddb_{\lb}\kep}_{\li{\infty}{2}}\lesssim \ep.
\ee
\eqref{hodgk8} and \eqref{hodgk10} improve the estimate for $\no(k)$, $\norm{\ddb_{\lb}\kep}_{\li{\infty}{2}}$ and $\norm{\dd_{\lb}\d}_{\li{\infty}{2}}$ given by the bootstrap assumption \eqref{boot3}.

\subsection{Time foliation versus geodesic foliation}\lab{sec:timefolvsgeodfol}

While we work with a time foliation, we recall that the estimates corresponding to the bootstrap assumptions on $\chi$ and $\z$ have already been proved in the context of a geodesic foliation in the sequence of papers \cite{FLUX} \cite{LP} \cite{STT}. One may reprove these estimates by adapting the proofs to the context of a time foliation. However, this would be rather lengthy and we suggest here a more elegant solution which consists in translating certain estimates from the geodesic foliation to the time foliation, and in obtaining directly the rest of the estimates. More precisely, we wish to obtain the $L^\infty$ bound from $\trc$, and the trace bounds for $\hch$ and $\z$ by exploiting the corresponding estimates in the geodesic foliation. We will also obtain the trace bounds for $\d$ and $\kep$ by reducing to estimates in the geodesic foliation in section \ref{sec:tracenormk}. Finally, these trace bounds and the null structure equations will allow us to get all the remaining estimates in sections \ref{sec:impbootb} and \ref{sec:remainest}. We start by recalling some of the results obtained in the context of the geodesic foliation in the sequence of papers \cite{FLUX} \cite{LP} \cite{STT}.

\subsubsection{The case of the geodesic foliation}\lab{sec:defgeodfol}

Remember that $u$ is a solution to the eikonal equation $\gg^{\alpha\beta}\partial_\alpha u\partial_\beta u=0$ on $\mathcal{M}$.  The level hypersufaces $u(t,x,\o)=u$  of the optical function $u$ are denoted by  $\H_u$. $L'=-\gg^{\a\b}\pr_\b u \pr_\a$ is the geodesic null generator of $\H_u$. In particular, we have:
$$\dd_{L'}L'=0,\,\gg(L',L')=0.$$
Let $s$ denote its affine parameter, i.e. $L'(s)=1$. We denote by $P'_{s,u}$  the level surfaces of $s$ in $\H_u$. 

\begin{definition} A  null frame   $e'_1,e'_2,e'_3,e'_4$ at a point $p\in P'_{s,u}$ consists,
in addition to $e'_4=L'$, of {\sl  arbitrary  orthonormal}  vectors  $e'_1,e'_2$ tangent
to $P'_{s,u}$ and the unique vectorfield $e'_3=\lb'$ satisfying the relations:
$$\gg(e'_3,e'_4)=-2,\,\gg(e'_3,e'_3)=0,\,\gg(e'_3,e'_1)=0,\,\gg(e'_3,e'_2)=0.$$
\end{definition}

\begin{definition}[\textit{Ricci coefficients in the geodesic foliation}]
 Let  $e'_1,e'_2,e'_3,e'_4$ be a null frame on
$P'_{s,u}$ as above.   The following tensors on  $P'_{s,u}$ 
\be\label{chip}
\begin{array}{ll}
\chi'_{AB}=<\dd_{e'_A} e'_4,e'_B>, & \chb'_{AB}=<\dd_{e'_A} e'_3,e'_B>,\\
\z'_{A}=\half <\dd_{e'_A} e'_4,e'_3> &
\end{array}
\ee
are called the Ricci coefficients associated to the geodesic foliation.

We decompose $\chi'$ and $\chb'$ into
their  trace and traceless components.
\begin{alignat}{2}
&\trc' = \gg^{AB}\chi'_{AB},&\quad &\trchb' = \gg^{AB}\chb'_{AB},
\label{trchip}\\
&\hch'_{AB}=\chi'_{AB}-\half \trc' \gg_{AB},&\quad 
&\hchb'_{AB}=\chb'_{AB}-\half \trchb' \gg_{AB}.
\label{chihp} 
\end{alignat}
\end{definition}

\begin{definition}
The null components of the curvature tensor
$\rr$ of the space-time metric $\gg$ in the geodesic foliation are given
by:
\bea
\a'_{ab}&=&\rr(L',e'_a,L', e'_b)\,,\qquad \b'_a=\half \rr(e'_a, L',\lb', L') ,\\ \r' &=&\frac{1}{4}\rr(\lb', L', \lb',
L')\,,
\qquad
\s'=\frac{1}{4}\, ^{\star} \rr(\lb', L', \lb', L')\\
\bb'_a&=&\half \rr(e'_a,\lb',\lb', L')\,,\qquad \underline{\a'}_{ab}=\rr(\lb', e'_a,\lb', e'_b)
\eea
where $^\star \rr$ denotes the Hodge dual of $\rr$. 
\end{definition}

The following  \textit{Ricci equations} can be easily derived (see \cite{FLUX}): 
\begin{alignat}{2}
&\dd_{e'_A}e'_4=\chi'_{AB} e'_B - \z'_A e'_4, &\quad 
&\dd_{e'_A} e'_3=\chb'_{AB} e'_B + \z'_A e'_3,\nn\\
&\dd_{e'_4} e'_4 = 0, &\quad 
&\dd_{e'_4} e'_3= -2\z'_{A} e'_A, 
  \label{ricciformp} \\
&\dd_{e'_4} e'_A = \ddb_{e'_4} e'_A -
\z'_{A} e'_4,&\quad &\dd_{e'_B}e'_A = \nabb'_{e'_B} e'_A +\half \chi'_{AB}\, e'_3 +\half
 \chb'_{AB}\, e'_4\nn
\end{alignat}
where, $\ddb_{ e'_3}$, $\ddb_{e'_ 4}$ denote the 
projection on $P'_{s,u}$ of $\dd_{e'_3}$ and $\dd_{e'_4}$, and $\nabb'$
denotes the induced covariant derivative on $P'_{s,u}$.

We now recall the main estimates obtained in the sequence of papers \cite{FLUX} \cite{LP} \cite{STT}. We have:
\be\lab{estgeodesicfol}
\norm{\trc'}_{\lh{\infty}}+\norm{\hch'}_{L^{2}_{x'}L^{\infty}_s}+\norm{\z'}_{L^{2}_{x'}L^{\infty}_s}\lesssim\ep
\ee
and 
\be\lab{estgeodesicfolbis}
\norm{\chb'}_{L^2_{x'}L^\infty_s}+\no'(\chi')+\no'(\z')\lesssim\ep,
\ee
where the norm $\no'$ is given by 
$$\no'(F)=\norm{F}_{\lh{2}}+\norm{\nabb' F}_{\lh{2}} +\norm{\ddb_{L'}F}_{\lh{2}}.$$

\begin{remark}\lab{remark:equivnorm}
Note that the norm $\lh{\infty}$ does not depend on the particular foliation. Now, this is also the case for  the trace norm $L^{2}_{x'}L^{\infty}_s$. Indeed, recall the definition of the null geodesic $\kappa_x$ in Remark \ref{defnullgeod}. Then, we have:
\bee
\norm{F}_{\xt{\infty}{2}}^2=\sup_{(0,x)\in\Si}\int_0^1|F(\kappa_x(t))|^2dt=\sup_{(0,x)\in\Si}\int_0^1|F(\kappa_x(s))|^2n^{-1}b^{-1}ds\sim \norm{F}_{L^{\infty}_{x'}L^{2}_s}^2
\eee
where we used the fact that $\frac{dt}{ds}=n^{-1}b^{-1}$ and the fact that $nb\sim 1$ by the bootstrap assumption \eqref{boot1}.
\end{remark}

In the next section, we will obtain the estimates corresponding to \eqref{estgeodesicfol} in the time foliation. For now, we conclude this section by recalling the definition and properties of the Besov spaces constructed in 
the sequence of papers \cite{FLUX} \cite{LP} \cite{STT}. For $P'_{s,u}$-tangent tensors $F$ on $\H_u$, $0\leq a\leq 1$,  we introduce the Besov norms:
\bea
\|F\|_{{\BB'}^a}&=&\sum_{j\ge 0} 2^{ja} \sup_{0\le s\le 1}\|
P'_jF\|_{L^2(P'_{s,u})} +  \sup_{0\le s\le 1}\|
P'_{<0}F\|_{L^2(P'_{s,u})}    ,\label{eq:Besovnorm}\\
\|F\|_{{\PP'}^a}&=&\sum_{j\ge 0} 2^{ja} \| P'_jF\|_{L^2(\H_u)} +
\| P'_{<0}F\|_{L^2(\H_u)}\label{eq:Penrosenorm}
\eea
where $P'_j$ are the geometric Littlewood-Paley projections on the 2-surfaces $P'_{s,u}$. Using the definition of these Besov spaces, we recall another estimate obtained in the sequence of papers \cite{FLUX} \cite{LP} \cite{STT}. We have:
\be\lab{estgeodesicfolter}
\norm{\chb'}_{{\BB'}^0}\lesssim\ep.
\ee

We now recall some properties of these Besov spaces obtained in the sequence of papers \cite{FLUX} \cite{LP} \cite{STT}. We have for scalar functions on  $\H_u$ (see \cite{FLUX} section 5):
\be\lab{eq:Besov-Sobolev-BB0}
\|f\|_{\lh{\infty}}\lesssim \|f\|_{{\BB'}^1}\lesssim \|f\|_{L_s^\infty L_{x'}^2}+\|\nabb' f\|_{{\BB'}^0}.
\end{equation}
Furthermore, for any $P'_{s,u}$-tangent tensors $F, G$ on $\H_u$, we have:
\be\lab{eq:secondbilBesov}
\| F\cdot G\|_{{\BB'}^0}\lesssim \big(\|\nabb' F\|_{L_s^\infty L_x^2} +
\|F\|_{L^\infty}\big)\|G\|_{{\BB'}^0}
\end{equation}
and 
\be\lab{eq:secondbilBesovbis}
\| F\cdot G\|_{{\PP'}^0}\lesssim \big(\|\nabb' F\|_{L_s^\infty L_x^2} +
\|F\|_{L^\infty}\big)\|G\|_{{\PP'}^0}. 
\end{equation}
To bound Besov norms, we sometime use the following non sharp embedding estimate. For 
any $0\leq a<\half$, we have:
\be\lab{eq:GoodestimBthetanotGd}
\|F\|_{{\BB'}^a}\lesssim \no'(F).
\end{equation}
We also have the following non sharp product estimate:
\be\lab{eq:nonsharpproduct}
\|FG\|_{{\PP'}^a}\lesssim \no'(F)\no'(G).
\end{equation}

The following proposition is the key tool used in \cite{FLUX} to control the transport equations 
appearing in the null structure.
\begin{proposition}
Assume that the scalar function $U$  satisfies $U(0)=0$ and the following transport
 equation along  $\H_u$:
$$\frac{d}{ds}U+a\trc U =F_1\cdot\ddb_{L'} P+F_2\cdot W,$$
where $a$ is some positive number. Then,
\be\label{eq:mainleprop1}
\|U\|_{{\BB'}^0}\lesssim\big(\no'(F_1)+\|F_1\|_{L_{x'}^\infty L_s^2}\big)\cdot\no'(P)
\big(\no'(F_2)+\|F_2\|_{L_{x'}^\infty L_s^2}\big)\cdot\|W\|_{{\PP'}^0}.
\ee
\label{prop:mainlemmaapplication}
\end{proposition}

Finally, using the previous proposition, we may prove the following version
of the sharp classical trace theorem which is a slight generalization of Corollary 5.10 in \cite{FLUX}.
\begin{corollary} Assume $F$ is an $P'_{s,u}$-tangent tensor which
admits a decomposition
of the form, $\nabb' F=A\ddb_{L'}P+E$.
Then,
\be\lab{eq:funny-classical-trace}
\|F\|_{L_{x'}^\infty L_s^2}\lesssim \no'(F)+\no'(P)(\norm{A}_{L^\infty}+\norm{\nabb'A}_{L^2_{x'}L^\infty_s}+\norm{\ddb_{L'}A}_{L^2_{x'}L^\infty_s})
+\|E\|_{{\PP'}^0}.
\end{equation}
\label{corr:funny-classical-trace}
\end{corollary}
\begin{proof} 
The scalar function  $f(t)=\int_0^t|F|^2$ verifies the transport
equation,
$$L f=|F|^2,\qquad f(0)=0.$$
Recall the following commutator formula in the geodesic foliation:
$$[\ddb_{L'},\nabb'_A]f=-\chi'_{AB}\nabb'_Bf.$$
Differentiating and applying the commutator formula, we derive,
\bee
\ddb_{L'} (\nabb' f)&=&2 F \cdot\nabb' F-\chi' \cdot(\nabb' f)\\
&=&F\cdot A\ddb_{L'}P +F\cdot E-\chi' \cdot(\nabb' f)
\eee
Applying  \eqref{eq:mainleprop1}, we deduce:
 \bee
\|\nabb' f\|_{\BB^0}&\lesssim&
\big(\no'(FA)+\|FA\|_{L_{x'}^\infty L_s^2}\big)\no'(P)\big(\no'(F)+\|F\|_{L_{x'}^\infty L_s^2}\big)\|E\|_{{\PP'}^0}\nn\\
&&+\big(\no'(\chi')+\|\chi'\|_{L_{x'}^\infty L_s^2}\big)\|\nabb' f\|_{{\PP'}^0},
\eee
which together with the estimates \eqref{estgeodesicfol} and \eqref{estgeodesicfolbis} for $\chi'$ and the fact that $\ep$ is small yields:
\be\lab{eq:interm-classical-trace}
\|\nabb' f\|_{{\BB'}^0}\lesssim\big(\no'(FA)+\|FA\|_{L_{x'}^\infty L_s^2}\big)\no'(P)+(\no'(F)+\|F\|_{L_{x'}^\infty L_s^2}\big)\|E\|_{{\PP'}^0}.
\end{equation}
We have:
\bee
\no'(FA)+\|FA\|_{L_{x'}^\infty L_s^2}&\lesssim& \|A\nabb'F\|_{\lh{2}}+\|A\ddb_{L'}F\|_{\lh{2}}+\|F\nabb'A\|_{\lh{2}}\\
&&+\|F\ddb_{L'}A\|_{\lh{2}}+\|AF\|_{\lh{2}}+\|F\|_{L_{x'}^\infty L_s^2}\norm{A}_{L^\infty}\\
&\lesssim& (\|F\|_{L_{x'}^\infty L_s^2}+\no'(F))(\norm{A}_{L^\infty}+\norm{\nabb'A}_{L^2_{x'}L^\infty_s}+\norm{\ddb_{L'}A}_{L^2_{x'}L^\infty_s}),
\eee
which together with \eqref{eq:interm-classical-trace} yields:
$$\|\nabb' f\|_{{\BB'}^0}\lesssim\big(\no'(F)+\|F\|_{L_{x'}^\infty L_s^2}\big)\big((\norm{A}_{L^\infty}+\norm{\nabb'A}_{L^2_{x'}L^\infty_s}+\norm{\ddb_{L'}A}_{L^2_{x'}L^\infty_s})\no'(P)+\|E\|_{{\PP'}^0}\big).$$
Now, in view of  estimate \eqref{eq:Besov-Sobolev-BB0}, we infer that,
\bee
\|f\|_{L^\infty}&\lesssim&  \|f\|_{L_s^\infty L_{x'}^2}+ \|\nabb'f\|_{{\BB'}^0}\\
&\lesssim&\big(\no'(F)+\|F\|_{L_{x'}^\infty
L_s^2}\big)\big((\norm{A}_{L^\infty}+\norm{\nabb'A}_{L^2_{x'}L^\infty_s}+\norm{\ddb_{L'}A}_{L^2_{x'}L^\infty_s})\no'(P)\\
&&+\|E\|_{{\PP'}^0}\big) +\|F\|_{L_s^2L_{x'}^4}^2.
\eee
Thus, recalling the definition of $f=\int_0^t |F|^2$,
and the estimate  $\|F\|_{L_s^2L_{x'}^4}\lesssim \no'(F)$, 
 we obtain:
\bee
\|F\|_{L_{x'}^\infty L_s^2}^2&\lesssim& \big(\no'(F)+\|F\|_{L_{x'}^\infty
L_s^2}\big)\big((\norm{A}_{L^\infty}+\norm{\nabb'A}_{L^2_{x'}L^\infty_s}+\norm{\ddb_{L'}A}_{L^2_{x'}L^\infty_s})\no'(P)\\
&&+\|E\|_{{\PP'}^0}\big)+\no'(F)^2
\eee
which yields the desired estimate \eqref{eq:funny-classical-trace}.
\end{proof}

\subsubsection{Estimates in the time foliation}

In this section, we obtain the $L^\infty$ bound from $\trc$, and the trace bounds for $\hch$ and $\z$ by exploiting the corresponding estimates in the geodesic foliation \eqref{estgeodesicfol}. We start by establishing the relation between the Ricci coefficients in the time and geodesic foliation. Recall first from the definition of $L$ and $L'$ \eqref{it2} that $L=bL'$. Since $(e_1, e_2)$ and $(e'_1, e'_2)$ are both orthonormal vectors in the tangent space of $\H_u$ which are both orthogonal to $L$, we may chose these vectors such that there is a tensor $F'$ on $P'_{s,u}$ satisfying:
$$e_A=e'_A+F'_AL',\, A=1, 2.$$
Also, writing $\lb$ in the frame $e'_1,e'_2, L',\lb'$, and using the fact that $g(L,\lb)=-2$, $g(\lb,\lb)=0$ and  $g(\lb,e_A)=0,\,A=1,2$, we obtain:
$$\lb=b^{-1}\lb'+2b^{-1}F'_Ae'_A+b^{-1}|F'|^2L'.$$
Finally, we have established the following relations:
\be\lab{comptg}
\begin{array}{l}
L=bL',\\
e_A=e'_A+F'_AL',\, A=1, 2,\\
\lb=b^{-1}\lb'+2b^{-1}F'_Ae'_A+b^{-1}|F'|^2L'.
\end{array}
\ee

We now use the definition \eqref{chi} and \eqref{chip} of the Ricci coefficients respectively in the time and geodesic foliation. We first establish the relation between $\chi$ and $\chi'$. Using the definition \eqref{chi} of $\chi$ and \eqref{chip} of $\chi'$, we have:
$$\chi_{AB}=\gg(\dd_{e_A}L,e_B)=\gg(\dd_{e'_A+F'_AL'}(bL'),e'_B+F'_BL')=b\chi'_{AB}$$
where we used the Ricci equations \eqref{ricciformp} and the identities $\gg(L',L')=\gg(L',e'_A)=0,\, A=1, 2$. In particular, we obtain:
\be\lab{comptg0}
\chi=b\chi',\,\trc=b\trc',\,\hch=b\hch'.
\ee
\eqref{comptg0} together with the bootstrap assumption \eqref{boot1} and the estimate \eqref{estgeodesicfol} yields:
\be\lab{comptg1}
\begin{array}{l}
\norm{\trc}_{\lh{\infty}}\leq\norm{b}_{\lh{\infty}}\norm{\trc'}_{\lh{\infty}}\lesssim\ep,\\
\norm{\hch}_{\xt{\infty}{2}}\leq\norm{b}_{\lh{\infty}}\norm{\hch'}_{L^{2}_{x'}L^{\infty}_s}\lesssim\ep,
\end{array}
\ee
where we have used the fact that the trace norms $L^{2}_{x'}L^{\infty}_t$ and $L^{2}_{x'}L^{\infty}_s$ are equivalent by Remark \ref{remark:equivnorm}. Note that \eqref{comptg1} is an improvement of the corresponding estimates in the bootstrap assumptions \eqref{boot4} \eqref{boot5}.

Next, we establish the relation between $\chb$ and $\chb'$. Using the definition \eqref{chi} of $\chb$ and \eqref{chip} of $\chb'$, we have:
\bea
\nn\chb_{AB}&=&\gg(\dd_{e_A}\lb,e_B)=\gg(\dd_{e'_A+F'_AL'}(b^{-1}\lb'+2b^{-1}F'_Ce'_C+b^{-1}|F'|^2L'),e'_B+F'_BL')\\
\nn&=& b^{-1}\big(\gg(\dd_{e'_A+F'_AL'}\lb',e'_B+F'_BL')+2\gg(\dd_{e'_A+F'_AL'}F',e'_B+F'_BL')\\
\nn&&+|F'|^2\gg(\dd_{e'_A+F'_AL'}L',e'_B+F'_BL')\big)\\
\nn&=&b^{-1}\big(\chb'_{AB}-2F'_B\z'_A-2F'_A\z'_B+2\gg(\dd_{e'_A}F',e'_B)-2F'_B\chi'_{AC}F'_C+2F'_A\gg(\dd_{L'}F',e'_B)\\
\lab{comptg2}&&+|F'|^2\chi'_{AB}\big)
\eea
where we used the Ricci equations \eqref{ricciformp}.

We establish the relation between $\z$ and $\z'$. Using the definition \eqref{chi} of $\z$ and \eqref{chip} of $\z'$, we have:
\bea
\nn\z_A&=&\half\gg(\dd_{\lb}L,e_A)=\half\gg(\dd_{b^{-1}\lb'+2b^{-1}F'_Ce'_C+b^{-1}|F'|^2L'}(bL'),e'_A+F'_AL')\\
\lab{comptg3}&=& \half\gg(\dd_{\lb'}L',e'_A)+\chi'_{AC}F'_C.
\eea
Now, we have:
\bea
\lab{comptg4}\half\gg(\dd_{\lb'}L',e'_A)&=&-\half\gg(L',\dd_{\lb'}e'_A)\\
\nn&=&-\half\gg(L',[\lb',e'_A])-\half\gg(L',\dd_{e'_A}\lb')\\
\nn&=&\z'_A-\half\gg(L',[\lb',e'_A]),
\eea
where we used the definition of $\z'$ \eqref{chip} in the last equality. The last term in \eqref{comptg4} is given by:
\be\lab{comptg5}
-\half\gg(L',[\lb',e'_A])=-\half[\lb',e'_A](u)=-\half \lb'(e'_A(u))+\half e'_A(\lb'(u))=0
\ee
where we used the fact that $e'_A(u)=0$ and $\lb'(u)=-2$. Finally, \eqref{comptg3}-\eqref{comptg5} yield:
\be\lab{comptg6}
\nn\z_A=\z'_A+\chi'_{AC}F'_C,
\ee
which together with the estimate \eqref{estgeodesicfol} and Remark \ref{remark:equivnorm} implies:
\be\lab{comptg7}
\norm{\z}_{\xt{\infty}{2}}\lesssim \norm{\z'}_{L^\infty_{x'}L^2_s}+\norm{\chi'}_{L^\infty_{x'}L^2_s}\norm{F'}_{L^\infty}\lesssim\ep+\ep\norm{F'}_{L^\infty}.
\ee

In view of \eqref{comptg7}, we need to estimate $\norm{F'}_{L^\infty}$. We make the bootstrap assumption:
\be\lab{bootF}
\norm{F'}_{L^\infty}\leq D^2\ep
\ee
where $D$ is the large constant appearing in the bootstrap assumptions \eqref{boot1}-\eqref{boot6}. Our goal is to improve on the constant in the right-hand side of \eqref{bootF}. 
We first estimate $\dd_{L'}F$. In view of the Ricci equations \eqref{ricciform}, we have:
\bea
\nn\kepb_A&=&-\half\gg(\dd_{L}\lb,e_A)=-\half\gg(\dd_{bL'}(b^{-1}\lb'+2b^{-1}F'_Ce'_C+b^{-1}|F'|^2L'),e'_A+F'_AL')\\
\nn&=&-\half\gg(\dd_{L'}\lb',e'_A)-\gg(\dd_{L'}(F),e'_A)\\
\lab{comptg8}&=&\z'_A-\gg(\dd_{L'}(F),e'_A),
\eea
where we used the Ricci equations \eqref{ricciformp} to obtain the last equality. \eqref{comptg8} implies:
\be\lab{comptg9}
\norm{\ddb_{L'}F'}_{L^\infty_{x'}L^2_s}+\no'(\ddb_{L'}F')\lesssim \norm{\z'}_{L^\infty_{x'}L^2_s}+\norm{\kepb}_{L^\infty_{x'}L^2_s}+\no'(\z')+\no'(\kepb).
\ee
Now, in view of the definition of $\no$ and $\no'$, and the relation \eqref{comptg} between $\nabb$ and $\nabb'$, we have for any tensor $G$:
\be\lab{comptg10}
\no(G)\lesssim (1+\norm{F'}_{L^\infty})\no'(G)\lesssim \no'(G)\textrm{ and }\no'(G)\lesssim (1+\norm{F'}_{L^\infty})\no(G)\lesssim \no(G)
\ee
where we used the bootstrap assumption \eqref{bootF}. Remark \ref{remark:equivnorm}, the estimates \eqref{estgeodesicfol}, \eqref{estgeodesicfolbis}, \eqref{comptg9}, \eqref{comptg10} and the bootstrap assumptions \eqref{boot2} \eqref{boot3} imply:
\be\lab{comptg11}
\norm{\ddb_{L'}F'}_{L^\infty_{x'}L^2_s}+\no'(\ddb_{L'}F')\lesssim D\ep.
\ee

We now estimate $\nabb'F'$. In view of \eqref{comptg2}, we have:
\bea
\nn \gg(\dd_{e'_A}F',e'_B)
\nn&=& \half b\chb_{AB} -\half\chb'_{AB}+F'_B\z'_A+F'_A\z'_B+F'_B\chi'_{AC}F'_C-F'_A\gg(\dd_{L'}F',e'_B)\\
\lab{comptg12}&&-\half |F'|^2\chi'_{AB}
\eea
which yields:
\bea
\lab{comptg13}\norm{\nabb'F'}_{{\BB'}^0}&\lesssim& \norm{b\chb}_{{\BB'}^0}+\norm{\chb'}_{{\BB'}^0}+\norm{F'\z'}_{{\BB'}^0}+\norm{F'^2\chi'}_{{\BB'}^0}+\norm{F'\ddb_{L'}F'}_{{\BB'}^0}\\
\nn&\lesssim&\no'(b\chb)+\norm{\chb'}_{{\BB'}^0}+(\norm{\nabb'F'}_{L^\infty_sL^2_{x'}}+\norm{F'}_{L^\infty})\norm{\z'}_{{\BB'}^0}\\
\nn&&+(\norm{F'\nabb'F'}_{L^\infty_sL^2_{x'}}+\norm{F'}^2_{L^\infty})\norm{\chi'}_{{\BB'}^0}
+(\norm{\nabb'F'}_{L^\infty_sL^2_{x'}}+\norm{F'}_{L^\infty})\norm{\ddb_{L'}F'}_{{\BB'}^0}\\
\nn&\lesssim&\no(b\chb)+\norm{\chb'}_{{\BB'}^0}+(\norm{\nabb'F'}_{L^\infty_sL^2_{x'}}+\norm{F'}_{L^\infty})\no'(\z')\\
\nn&&+\norm{F'}_{L^\infty}(\norm{\nabb'F'}_{L^\infty_sL^2_{x'}}+\norm{F'}_{L^\infty})\no'(\chi')\\
\nn&&+(\norm{\nabb'F'}_{L^\infty_sL^2_{x'}}+\norm{F'}_{L^\infty})\no'(\ddb_{L'}F')
\eea
where we used several times the inequalities \eqref{eq:secondbilBesov} and 
\eqref{eq:GoodestimBthetanotGd} as well as \eqref{comptg10}. The estimates \eqref{estgeodesicfolbis} and \eqref{estgeodesicfolter} together with the bootstrap assumptions \eqref{boot1}-\eqref{boot5}, \eqref{comptg11} and \eqref{comptg13} yield:
\be\lab{comptg14}
\norm{\nabb'F'}_{{\BB'}^0}\lesssim D\ep+D\ep(1+\norm{F'}_{L^\infty})(\norm{\nabb'F'}_{{\BB'}^0}+\norm{F'}_{L^\infty}).
\ee
Finally, the bootstrap assumption \eqref{bootF} together with \eqref{comptg14} and the fact that $\ep$ is small yields:
$$\norm{\nabb'F'}_{{\BB'}^0}\lesssim D\ep$$
which together with \eqref{eq:Besov-Sobolev-BB0} implies:
\be\lab{comptg15}
\norm{F'}_{L^\infty}\lesssim D\ep.
\ee
\eqref{comptg15} is an improvement of the bootstrap assumption \eqref{bootF} which shows that $F'$ is indeed in $L^\infty$ and satisfies the bound \eqref{comptg15}. In particular, \eqref{comptg7} and \eqref{comptg15} imply:
\be\lab{comptg16}
\norm{\z}_{\xt{\infty}{2}}\lesssim \ep.
\ee
Note that \eqref{comptg1} and \eqref{comptg16} are improvements of the corresponding estimates in the bootstrap assumptions \eqref{boot4}-\eqref{boot6}.

\begin{remark}\lab{remark:compt17}
We also have an estimate for $\nabb'F'$ in $L^2_{x'}L^\infty_s$. Indeed, \eqref{comptg12} yields:
\bea
\nn\norm{\nabb'F'}_{L^2_{x'}L^\infty_s}&\lesssim& \norm{b\chb}_{L^2_{x'}L^\infty_s}+\norm{\chb'}_{L^2_{x'}L^\infty_s}+\norm{F'\z'}_{L^2_{x'}L^\infty_s}+\norm{F'^2\chi'}_{L^2_{x'}L^\infty_s}+\norm{F'\ddb_{L'}F'}_{L^2_{x'}L^\infty_s}\\
\nn&\lesssim&\norm{b}_{L^\infty}\no'(\chb)+\norm{\chb'}_{L^2_{x'}L^\infty_s}+\norm{F'}_{L^\infty}\norm{\z'}_{L^2_{x'}L^\infty_s}\\
\nn&&+\norm{F'}^2_{L^\infty}\norm{\chi'}_{L^2_{x'}L^\infty_s}+\norm{F'}_{L^\infty}\norm{\ddb_{L'}F'}_{L^2_{x'}L^\infty_s}\\
\nn&\lesssim&\norm{b}_{L^\infty}\no(\chb)+\ep+D\ep\no'(\z')+D^2\ep^2\no'(\chi')+D\ep\no'(\ddb_{L'}F')\\
\lab{comptg18}&\lesssim& D\ep
\eea
where we used \eqref{comptg10}, \eqref{comptg11}, \eqref{comptg15}, the bootstrap assumptions \eqref{boot1}-\eqref{boot6} on $b$ and $\chb$, and the estimates \eqref{estgeodesicfolbis}.
\end{remark}

\begin{remark}
\eqref{comptg12} implies the following estimate for $\nabb'F'+\half\chb'$:
\bea
\nn \no'\left(\nabb'F'+\half\chb'\right)&\lesssim& \no'(b\chb)+\no'(F'\z')+\no'(F'^2\chi')+\no'(F'\ddb_{L'}F')\\
\nn&\lesssim& \no(b\chb)+(\norm{F'}_{L^\infty}+\norm{\nabb'F'}_{L^2_{x'}L^\infty_s})(\no'(\z')+\norm{\z'}_{L^\infty_{x'}L^2_s}\\
\nn&&+\norm{F'}_{L^\infty}(\no'(\chi')+\norm{\chi'}_{L^\infty_{x'}L^2_s})+\no'(\ddb_{L'}F')+\norm{\ddb_{L'}F'}_{L^\infty_{x'}L^2_s})\\
\lab{comptg19} &\lesssim & D\ep,
\eea
where we used \eqref{comptg10}, the estimates \eqref{estgeodesicfol} and \eqref{estgeodesicfolbis}, the estimates \eqref{comptg11} \eqref{comptg15} and \eqref{comptg18} for $F'$, and the bootstrap assumptions \eqref{boot1}-\eqref{boot6} for $b$ and $\chb$. 
\end{remark}

\subsection{Trace norm bounds for $\db  $ and $\kepb$}\lab{sec:tracenormk}

The goal of this section is to improve the estimate for $\norm{\db}_{\xt{\infty}{2}}$ and $\norm{\kepb}_{\xt{\infty}{2}}$ given by the bootstrap assumption \eqref{boot3}. Let us first define $k_{LL}$ and $k_{LA}$:
\be\lab{tracek}
k_{LL}=-\gg(\dd_LT,L),\,k_{LA}=-\gg(\dd_LT,e_A),\,A=1,2.
\ee
Then, using the definition of $\kb$ \eqref{newk} and the computation of $\dd_TT$ \eqref{3.2}, we have:
\bea
\nn\db  &=&\d -n^{-1}\nabla_Nn=-\gg(\dd_NT,N)-\gg(\dd_TT,N)=-\gg(\dd_LT,N)=-\gg(\dd_LT,L)\\
\lab{tracek1}&=&k_{LL}.
\eea
and
\bea
\lab{tracek2}\kepb_{A}&=&\kep_{A}-n^{-1}\nabla_An=-\gg(\dd_NT,e_A)-\gg(\dd_TT,e_A)=-\gg(\dd_LT,e_A)\\
\nn&=&k_{LA}.
\eea
We also define $k_{L'L'}$ and $k_{L'A}$:
\be\lab{tracek3}
k_{L'L'}=-\gg(\dd_{L'}T,L'),\,k_{L'A}=-\gg(\dd_{L'}T,e'_A),\,A=1,2.
\ee
Then, the relations \eqref{comptg} between $L, e_1, e_2$ and $L', e'_1, e'_2$ together with the definitions \eqref{tracek} and \eqref{tracek3} yield:
\be\lab{tracek4}
k_{LL}=b^2k_{L'L'}\textrm{ and }k_{LA}=bk_{L'A}+bF'_Ak_{L'L'}.
\ee
Thus, \eqref{tracek1}, \eqref{tracek2} and \eqref{tracek4} imply:
\bea
\lab{tracek5}\norm{\db  }_{\xt{\infty}{2}}&\lesssim& \norm{bk_{L'L'}}_{L^\infty_{x'}L^2_s}\lesssim \norm{k_{L'L'}}_{L^\infty_{x'}L^2_s}\\
\nn\norm{\kepb}_{\xt{\infty}{2}}&\lesssim& \norm{bk_{L'A}}_{L^\infty_{x'}L^2_s}+\norm{bF'_Ak_{L'L'}}_{L^\infty_{x'}L^2_s}\lesssim \norm{k_{L'L'}}_{L^\infty_{x'}L^2_s}+\norm{k_{L'A}}_{L^\infty_{x'}L^2_s}
\eea
where we used the bootstrap assumption \eqref{boot1}, the $L^\infty$ bound for $F'$ \eqref{comptg15} and Remark \ref{remark:equivnorm}.

In view of \eqref{tracek5}, it is enough to bound the trace norms $\norm{k_{L'L'}}_{L^\infty_{x'}L^2_s}$ and $\norm{k_{L'A}}_{L^\infty_{x'}L^2_s}$. First, note that the bootstrap assumptions \eqref{boot1} \eqref{boot3} together with the $L^\infty$ bound for $F'$ \eqref{comptg15} and the identity \eqref{tracek4} yield:
\be\lab{tracek6}
\norm{k_{L'L'}}_{L^\infty_{x'}L^2_s}+\norm{k_{L'A}}_{L^\infty_{x'}L^2_s}\lesssim D\ep.
\ee
Our goal in this section is to improve the constant $D$ in the right-hand side of \eqref{tracek6}. We will rely on the trace estimate \eqref{eq:funny-classical-trace}. The improved estimates for $n$ \eqref{lapn3bis} \eqref{lapn12} and the improved estimate for $\d, \kep$ \eqref{hodgk8} imply:
\be\lab{tracek7}
\no(\db  )+\no(\kepb)\lesssim \ep.
\ee
\eqref{tracek1}, \eqref{tracek4}, \eqref{tracek7}, \eqref{comptg10} and the bootstrap assumptions \eqref{boot1} \eqref{boot2} for $b$ yield:
\be\lab{tracek8}
\no'(k_{L'L'})\lesssim \no(b^{-2}\db  )\lesssim\ep.
\ee
\eqref{tracek6} and \eqref{tracek8} yield:
\bea
\nn\no'(F'k_{L'L'})&\lesssim& \norm{F'}_{L^\infty}\no'(k_{L'L'})+(\norm{\nabb'F'}_{L^2_{x'}L^\infty_s}+\norm{\ddb_{L'}F'}_{L^2_{x'}L^\infty_s})\norm{k_{L'L'}}_{L^\infty_{x'}L^2_s}\\
\nn&\lesssim& D\ep^2+D^2\ep^2+D\ep\no'(\ddb_{L'}F')\\
\lab{tracek9}&\lesssim& \ep
\eea
where we used \eqref{comptg11}, \eqref{comptg15} and \eqref{comptg18}. Finally, \eqref{tracek2}, \eqref{tracek4}, \eqref{tracek7}, \eqref{tracek9}, \eqref{comptg10} and the bootstrap assumptions \eqref{boot1} \eqref{boot2} for $b$ yield:
\be\lab{tracek10}
\no'(k_{L'A})\lesssim \no(b^{-1}\kepb)+\no'(F'_Ak_{L'L'})\lesssim\ep.
\ee

In order to apply  the trace estimate \eqref{eq:funny-classical-trace}, we need to show that $\nabb'k_{L'L'}$ and $\nabb'k_{L'A}$ admit a decomposition of the form, $A\ddb_{L'}P+E$. We start with $k_{L'L'}$. We have:
\bea
\nn\nabb'_{e'_A}k_{L'L'}&=&-\dd_{e'_A}\gg(\dd_{L'}T,L')=-\gg(\dd_{e'_A}\dd_{L'}T,L')-\gg(\dd_{L'}T,D_{e'_A}L')\\
\nn&=&-\gg(\dd_{L'}\dd_{e'_A}T,L')-\rr_{e'_AL'TL'}-\gg(\dd_{[e'_A,L']}T,L')\\
\nn&&-b^{-1}\gg(-\db  N+\zb_Ce_C,\chi'_{AB}e'_B-\z'_AL')\\
\nn&=&-\ddb_{L'}[\gg(\dd_{.}T,L')]_{A}-b^{-1}F'_B\a'_{AB}-\frac{b^{-1}}{2}\b'_A\\
\lab{tracek11}&&+b^{-1}\chi'_{AB}(k_{BN}-\zb_B)-b^{-2}\db  (2\chi'_{AB}F'_B+\z'_A).
\eea
Relying on the Bianchi identities, the following decomposition for $\a', \b'$ were obtained in \cite{FLUX}:
\be\lab{tracek12}
\a'=\ddb_{L'}(P_1)+E_1,\, \b'=\ddb_{L'}(P_2)+E_2,
 \ee
where $P_1=\mathcal{D'}_2^{-1}\b'$, $P_2=\mathcal{D'}_1^{-1}(\r',-\s')$, and 
\be\lab{tracek13}
\no'(P_1)+\no'(P_2)+\norm{E_1}_{{\PP'}^0}+\norm{E_2}_{{\PP'}^0}\lesssim\ep.
\ee
We define the tensors $A_1, A_2, P_3, E_3$ as:
\be\lab{tracek14}
\begin{array}{l}
{P_3}_{A}=-\gg(\dd_{e'_A}T,L'),\, A_1=-b^{-1}F',\,A_2=-\frac{b^{-1}}{2},\\
E_3=b^{-1}\chi'_{AB}(k_{BN}-\zb_B)-b^{-2}\db  (2\chi'_{AB}F'_B+\z'_A),
\end{array}
\ee
which together with \eqref{tracek11} and \eqref{tracek12} yields:
\be\lab{tracek15}
\nabb'k_{L'L'}=A_1\ddb_{L'}P_1+A_2\ddb_{L'}P_2+\ddb_{L'}P_3+E_1+E_2+E_3.
\ee
Now, we have:
$${P_3}_A=-\gg(\dd_{e'_A}T,L')=b^{-1}k_{AN}-F'_Ab^{-2}\db  $$
which yields:
\bea
\lab{tracek16}\no'(P_3)&\lesssim& \no(b^{-1}k_{AN})\\
\nn&&+(\norm{F'}_{L^\infty}+\norm{\nabb'F}_{L^2_{x'}L^\infty_s}+\norm{\ddb_{L'}'F}_{L^2_{x'}L^\infty_s})(\no(b^{-2}\db  )+\norm{b^{-2}\db  }_{\xt{\infty}{2}})\\
\nn&\lesssim& \ep+D^2\ep^2\\
\nn&\lesssim &\ep
\eea
where we used the bootstrap assumptions \eqref{boot1} \eqref{boot2} for $b$, the bootstrap assumption \eqref{boot3} for $\db  $, the improved estimate \eqref{hodgk8} for $k_{AN}$, the estimates \eqref{comptg11} \eqref{comptg15} and \eqref{comptg18} for $F$ together with Remark \ref{remark:equivnorm} and \eqref{comptg10}. Using the bootstrap assumptions \eqref{boot1} \eqref{boot2} for $b$ and the estimates \eqref{comptg11} \eqref{comptg15} and \eqref{comptg18} for $F$, we also have:
\bea
\nn&&\norm{A_1}_{L^\infty}+\norm{\nabb' A_1}_{L^2_{x'}L_s^\infty}+\norm{\ddb_{L'}A_1}_{L^2_{x'}L_s^\infty}
+ \norm{A_2}_{L^\infty}+\norm{\nabb' A_2}_{L^2_{x'}L_s^\infty}+\norm{\ddb_{L'}A_2}_{L^2_{x'}L_s^\infty}\\
\nn&\lesssim& (\norm{b}_{L^\infty}+\noo(b))(1+\norm{F'}_{L^\infty}+\norm{\nabb'F'}_{L^2_{x'}L_s^\infty}+\norm{\ddb_{L'}F'}_{L^2_{x'}L_s^\infty})\\
\lab{tracek17}&\lesssim& D\ep.
\eea
The functional inequalities \eqref{eq:secondbilBesovbis} and \eqref{eq:nonsharpproduct} yield:
\bea
\nn\norm{E_3}_{{\PP'}^0}&\lesssim& (\norm{b}_{L^\infty}+\norm{\nabb'b}_{L_s^\infty L^2_{x'}})(1+\norm{F'}_{L^\infty}+\norm{\nabb'F'}_{L_s^\infty L^2_{x'}})\\
\nn&&\times(\no'(\chi')^2+\no'(\z')^2+\no'(\db  )^2+\no'(\kep)^2+\no'(\zb)^2)\\
\nn&\lesssim&D^2\ep^2\\
\lab{tracek18}&\lesssim&\ep,
\eea
where we used the bootstrap assumptions \eqref{boot1}-\eqref{boot6} for $b$, $\db  $, $k_{BN}$ and $\zb$, the estimates \eqref{comptg11} \eqref{comptg15} and \eqref{comptg18} for $F$, and the estimate 
\eqref{estgeodesicfolbis} for $\chi'$ and $\z'$. Finally, the sharp trace estimate \eqref{eq:funny-classical-trace} together with \eqref{tracek15} and the estimates \eqref{tracek8} \eqref{tracek13}, \eqref{tracek16}, \eqref{tracek17} and \eqref{tracek18} yields:
\bea
\lab{tracek19}\norm{k_{L'L'}}_{L^\infty_{x'}L^2_s}&\lesssim& \no'(k_{L'L'})+(\norm{A_1}_{L^\infty}+\norm{\nabb' A_1}_{L^2_{x'}L_s^\infty})\no'(P_1)+(\norm{A_2}_{L^\infty}\\
\nn&&+\norm{\nabb' A_2}_{L^2_{x'}L_s^\infty})\no'(P_2)+\no'(P_3)+\norm{E_1}_{{\PP'}^0}+\norm{E_2}_{{\PP'}^0}+\norm{E_3}_{{\PP'}^0}\\
\nn &\lesssim&\ep,
\eea
which is an improvement of \eqref{tracek6} for $k_{L'L'}$.

Next, we show that $\nabb'k_{L'A}$ admit a decomposition of the form, $A\ddb_{L'}P+E$. We have:
\bea
\nn\nabb'_{e'_B}[k_{L'.}]_{e'_A}&=&-e'_B[\gg(\dd_{L'}T,e'_A)]+\gg(\dd_{L'}T,\ddb_{e'_B}e'_A)\\
\nn&=&-\gg(\dd_{e'_B}\dd_{L'}T,e'_A)-\gg(\dd_{L'}T,D_{e'_B}e'_A-\ddb_{e'_B}e'_A)\\
\nn&=&-\gg(\dd_{L'}\dd_{e'_B}T,e'_A)-\rr_{e'_BL'Te'_A}-\gg(\dd_{[e'_B,L']}T,e'_A)\\
\nn&&-b^{-1}\gg(-\db  N+\zb_Ce_C,\half\chi'_{AB}\lb'+\half\chb'_{AB}L')\\
\nn&=&-\ddb_{L'}[\gg(\dd_{.}T,.)]_{AB}-\rr_{e'_BL'Te'_A}-\gg(\dd_{\dd_{e'_B}L'}T,e'_A)\\
\nn&&+\gg(\dd_{\dd_{L'}e'_B-\ddb_{L'}e'_B}T,e'_A)+\gg(\dd_{e'_B}T,\dd_{L'}e'_A-\ddb_{L'}e'_A)\\
\nn&&-b^{-1}\gg(-\db  N+\zb_Ce_C,\half\chi'_{AB}\lb'+\half\chb'_{AB}L')\\
\nn&=&-\ddb_{L'}[\gg(\dd_{.}T,.)]_{AB}-\half b^{-1}\a'_{AB}\\
\nn&&+\half b^{-1}(\r'\d_{AB}-\s'\in_{AB}+2F'_C\in_{AC}{}^*\b'_B-|F'|^2\a'_{AB})\\
\nn&&+\chi'_{BC}(k_{CA}-F'_Ab^{-1}k_{CN}-F'_Cb^{-1}\kepb_{A}+F'_CF'_Ab^{-2}\db  )\\
\nn&&+b^{-1}\z'_A(k_{BN}-b^{-1}F'_B\db  )+\half\chi'_{AB}(-\db  +b^{-2}\db  |F'|^2\\
\lab{tracek21}&&+2b^{-1}\zb_CF'_C)+\half b^{-2}\chb'_{AB}\db  .
\eea
Define as in \cite{FLUX} $\check{\r}', \check{\s}'$ as:
\be\lab{tracek21bis}
\check{\r}'=\r'-\half\hch'\cdot\hchb',\,\check{\s}=\s-\half\hch'\wedge\hchb'.
\ee
Relying on the Bianchi identities, the following decomposition for $\check{\r}', \check{\s}'$ was obtained in \cite{FLUX}:
\be\lab{tracek22}
(\check{\r}',\check{\s}')=\ddb_{L'}(P_4)+E_4,
 \ee
where $P_4={}^*\mathcal{D'}_1^{-1}\bb'$, and 
\be\lab{tracek23}
\no'(P_4)+\norm{E_4}_{{\PP'}^0}\lesssim\ep.
\ee
We have 
$$\gg(\dd_{e'_B}T,e'_A)=-k_{AB}+F'_Bb^{-1}\kepb_A+b^{-1}F'_Ak_{BN}-b^{-2}F'_AF'_B\db  $$
which yields:
\be\lab{tracek23bis}
-\gg(\dd_{e'_B}T,e'_A)=\ddb_{L'}(P_5)+A_6\ddb_{L'}(P_6)+E_6
\ee
where $P_5, A_6, P_6, E_6$ are defined by:
\be\lab{tracek23ter}
\begin{array}{l}
{P_5}_{AB}=-k_{AB}+F'_Bb^{-1}\kepb_A-b^{-2}F'_AF'_B\db  ,\\
{A_6}_{B}=b^{-1}F'_B,\,{P_6}_B=k_{BN},\,{E_6}_{AB}=\ddb_{L}'[b^{-1}F']_Ak_{BN}.
\end{array}
\ee
We define the tensors $A_4, A_7, A_8, E_7$ as:
\be\lab{tracek24}
\begin{array}{l}
A_4=\half b^{-1},\,A_7=b^{-1}F',\,A_8=-b^{-1}|F'|^2\\
{E_7}_{AB}=\chi'_{BC}(k_{CA}-F'_Ab^{-1}k_{CN}-F'_Cb^{-1}\kepb_{A}+F'_CF'_Ab^{-2}\db  )\\
+b^{-1}\z'_A(k_{BN}-b^{-1}F'_B\db  )+\half\chi'_{AB}(-\db  +b^{-2}\db  |F'|^2+2b^{-1}\zb_CF'_C),
\end{array}
\ee
which together with \eqref{tracek12}, \eqref{tracek21}, \eqref{tracek22} and \eqref{tracek23bis} yields:
\bea\lab{tracek25}
\nn\nabb'k_{L'A}&=&A_7\ddb_{L'}P_2+A_8\ddb_{L'}P_1+A_4\ddb_{L'}P_4+\ddb_{L'}P_5+A_6\ddb_{L'}P_6+E_1+E_2+E_4\\
&&+E_6+E_7+\frac{1}{4}b^{-1}(\hch'\cdot\hchb'\d_{AB}-\hch'\wedge\hchb'\in_{AB})+\half b^{-2}\chb'_{AB}\db  
\eea
In view of \eqref{comptg19}, we define $E_8$ as:
\be\lab{tracek25bis}
E_8=\frac{1}{2}b^{-1}(\hch'\cdot(\nabb'F'+\half\hchb')\d_{AB}-\hch'\wedge(\nabb'F'+\half\hchb')\in_{AB})+ b^{-2}(\nabb'_BF'_A+\half\chb'_{AB})\db  .
\ee
Note that the non sharp product estimate \eqref{eq:nonsharpproduct} together with the bootstrap assumptions \eqref{boot1}-\eqref{boot3} for $b$ and $\db  $, the estimate \eqref{estgeodesicfolbis} for $\chi'$ and the estimate \eqref{comptg19} for $\nabb'F'+\half\chb'$ yields:
\be\lab{tracek25bisbis}
\norm{E_8}_{{\PP'}^0}\lesssim D^2\ep^2\lesssim \ep.
\ee
Now, we recall the following result from \cite{FLUX} section 7:
$$\nabb'\chi'=\ddb_{L'}(P_{10})+E_{10},$$
with $\no'(P_{10})+\norm{E_{10}}_{{\PP'}^0}\lesssim\ep$ which together with \eqref{tracek1}, \eqref{tracek15} and \eqref{tracek25bis} yields:
\bea
&&\lab{tracek25ter}\frac{1}{4}b^{-1}(\hch'\cdot\hchb'\d_{AB}-\hch'\wedge\hchb'\in_{AB})+\half b^{-2}\chb'_{AB}\db  \\
\nn&=&A_{11}\ddb_{L'}(P_{11})+E_{11}-\nabb'(\frac{1}{2}b^{-1}(\hch'\cdot F'\d_{AB}-\hch'\wedge F'\in_{AB})+ b^{-2}F'_A\db  ).
\eea
Using \eqref{tracek25bisbis}, the fact that $\no'(P_{10})+\norm{E_{10}}_{{\PP'}^0}\lesssim\ep$ from \cite{FLUX}, the estimate \eqref{estgeodesicfolbis} for $\chi'$, the bootstrap assumptions \eqref{boot1}-\eqref{boot3} for $b$ and $\db  $, and the estimates \eqref{comptg11} \eqref{comptg15} and \eqref{comptg18} for $F$ implies that $A_{11}, P_{11}, E_{11}$ satisfy:
\be\lab{tracek25quatre}
\norm{A_{11}}_{L^\infty}+\norm{\nabb' A_{11}}_{L^2_{x'}L_s^\infty}+\norm{\ddb_{L'}A_{11}}_{L^2_{x'}L_s^\infty}+\no'(P_{11})+\norm{E_{11}}_{{\PP'}^0}\lesssim\ep.
\ee
Now, \eqref{tracek25} and \eqref{tracek25ter} yield:
\bea
&&\nn\nabb'[k_{L'A}+\frac{1}{2}b^{-1}(\hch'\cdot F'\d_{AB}-\hch'\wedge F'\in_{AB})+ b^{-2}F'_A\db  ]\\
\nn&=&A_7\ddb_{L'}P_2+A_8\ddb_{L'}P_1+A_4\ddb_{L'}P_4+\ddb_{L'}P_5+A_6\ddb_{L'}P_6+A_{11}\ddb_{L'}P_{11}\\
&&\lab{tracek25six}+E_1+E_2+E_4+E_6+E_7+E_{11}.
\eea
\eqref{tracek23ter}, \eqref{tracek24}, the bootstrap assumptions \eqref{boot1} \eqref{boot2} for $b$, the bootstrap assumption \eqref{boot3} for $\kepb, \db  $, the improved estimate \eqref{hodgk8} for $k_{AN}$, the estimates \eqref{comptg11} \eqref{comptg15} and \eqref{comptg18} for $F'$ together with Remark \ref{remark:equivnorm} and \eqref{comptg10}, the estimate \eqref{estgeodesicfolbis} for $\chi'$ and $\z'$, the trace estimate \eqref{estgeodesicfol}, the inequality \eqref{eq:secondbilBesovbis} and the non sharp product estimate \eqref{eq:nonsharpproduct} yield:
\be\lab{tracek27}
\begin{array}{l}
\norm{A_j}_{L^\infty}+\norm{\nabb' A_j}_{L^2_{x'}L_s^\infty}+\norm{\ddb_{L'}A_j}_{L^2_{x'}L_s^\infty}\lesssim D\ep\textrm{ for }j=4, 6, 7, 8,\\
\no'(P_j)\lesssim \ep\textrm{ for }j=5, 6,\\
\norm{E_j}_{{\PP'}^0}\lesssim\ep\textrm{ for }j=6, 7.
\end{array}
\ee
Note also that \eqref{tracek10}, the bootstrap assumptions \eqref{boot1} \eqref{boot2} for $b$, the bootstrap assumption \eqref{boot3} for $\db  $, the estimates \eqref{comptg11} \eqref{comptg15} and \eqref{comptg18} for $F'$ together with the inequality \eqref{comptg10}, the estimate \eqref{estgeodesicfolbis} for $\chi'$ and the trace estimate \eqref{estgeodesicfol} imply:
\be\lab{tracek27bis}
\no'\bigg(k_{L'A}+\frac{1}{2}b^{-1}(\hch'\cdot F'\d_{AB}-\hch'\wedge F'\in_{AB})+ b^{-2}F'_A\db  \bigg)\lesssim \ep+D^2\ep^2\lesssim\ep.
\ee
Finally, the sharp trace estimate \eqref{eq:funny-classical-trace} together with \eqref{tracek25six} and the estimates \eqref{tracek13}, \eqref{tracek23}, \eqref{tracek25quatre}, \eqref{tracek27} and \eqref{tracek27bis} yields:
\be\lab{tracek29}
\normm{k_{L'A}+\frac{1}{2}b^{-1}(\hch'\cdot F'\d_{AB}-\hch'\wedge F'\in_{AB})+ b^{-2}F'_A\db  }_{L^\infty_{x'}L^2_s}\lesssim\ep+D^2\ep^2\lesssim\ep.
\ee
\eqref{tracek29}, the bootstrap assumption \eqref{boot1} for $b$, the bootstrap assumption \eqref{boot3} for $\db  $, the estimate \eqref{comptg15} for $F'$ and the trace estimate \eqref{estgeodesicfol} for $\chi'$ imply:
\be\lab{tracek30}
\norm{k_{L'A}}_{L^\infty_{x'}L^2_s}\lesssim\ep,
\ee
which is an improvement of \eqref{tracek6} for $k_{L'A}$. \eqref{tracek5}, \eqref{tracek19} and \eqref{tracek30} yield:
\be\lab{tracek31}
\norm{\db  }_{\xt{\infty}{2}}+\norm{\kepb}_{\xt{\infty}{2}}\lesssim\ep,
\ee
which improves the trace estimates for $\db  $ and $\kepb$ given by the bootstrap assumption \eqref{boot3}. 
 
\subsection{Estimates for $b$}\lab{sec:impbootb}

The goal of this section is to improve the bootstrap assumption for $b$ given by \eqref{boot1} and \eqref{boot2}, and to derive an estimate for $\lb(b)$ in $\tx{\infty}{4}$. Using the transport equation for $b$ \eqref{D4a} and the estimate for transport equations \eqref{estimtransport1},  we obtain:
\bea
\lab{remain9}\norm{b-1}_{L^\infty}&\lesssim&\norm{b(0)-1}_{L^\infty}+\norm{b\db}_{\xt{\infty}{1}}\\
\nn&\lesssim&\ep+(1+\norm{b-1}_{L^\infty})\norm{\db}_{\xt{\infty}{2}}\\
\nn&\lesssim&\ep+D\ep\norm{b-1}_{L^\infty}
\eea
where we used the bootstrap assumption \eqref{boot3} for $\db$ in the last inequality. \eqref{remain9} yields:
\be\lab{remain10}
\norm{b-1}_{L^\infty}\lesssim\ep
\ee
which improves the estimate for $b$ given by the bootstrap assumption \eqref{boot1}. Using \eqref{D4a} and \eqref{etaa}, we obtain:
\bea
\lab{remain11}\noo(b)&\lesssim& \no(L(b))+\no(\nabb b)\\
\nn&\lesssim & \no(b\db)+\no(b\z)+\no(b\kep)\\
\nn&\lesssim & (\norm{b}_{L^\infty}+\noo(b))(\no(\db)+\no(\z)+\no(\kep))\\
\nn&\lesssim & \ep+\ep\noo(b)
\eea
where we used \eqref{remain10} and the improved estimates \eqref{hodgk8} for $\db$ and $\kep$, and  \eqref{remain8} for $\z$. \eqref{remain11} yields:
\be\lab{remain12}
\noo(b)\lesssim\ep.
\ee
We also derive an estimate for $\lb(b)$. Differentiating the transport equation for $b$ \eqref{D4a} with respect to $\lb$, we obtain:
\bea\label{josh}
L(\lb(b))&=&[L,\lb](b)-\lb(b)\db-b\lb(\db)\\
\nn&=& (\d+n^{-1}\nabla_Nn)\db b-2(\z-\zb)\c\nabb b-b\lb(\db),
\eea
where we used the commutator formula \eqref{comm3} in the last equality. This yields:
\bea
\lab{josh1}&&\norm{L(\lb(b))}_{\lh{2}}\\
\nn&\les&(1+\norm{b}_{L^\infty})\bigg(\norm{\d+n^{-1}\nabla_Nn}_{\lh{4}}\norm{\db}_{\lh{4}}+\norm{\z-\zb}_{\lh{4}}\norm{\nabb b}_{\lh{4}}\\
\nn&&+\norm{\lb(\db)}_{\lh{2}}\bigg)\\
\nn&\les& (1+\norm{b}_{L^\infty})\bigg(\no(\d)^2+\no(n^{-1}\nabla n)^2+\no(\z)^2+\no(\zb)^2+\no(\nabb b)^2+\norm{\lb(\db)}_{\lh{2}}\bigg)\\
\nn&\les& \ep+D^2\ep^2\\
\nn&\les&\ep,
\eea
where we used the bootstrap assumptions \eqref{boot1}-\eqref{boot6}. Together with the estimate for transport equations \eqref{estimtransport1},  we obtain:
\be\lab{remain13}
\norm{\lb(b)}_{\xt{2}{\infty}}\les\norm{L(\lb(b))}_{\lh{2}}\les \ep.
\ee
\eqref{remain12} and \eqref{remain13} improve the estimate for $b$ given by the bootstrap assumption \eqref{boot2}. 
 
 Finally, we derive an estimate for $\lb(b)$ in $\tx{\infty}{4}$. In view of \eqref{D4tmu3}, we have
\bee
b\lb(\db)&=&-L(b(\d+n^{-1}\nab_Nn))+L(b)(\d+n^{-1}\nab_Nn)+2b\rho+2b|\kep|^2+2b\d^2\\
&&+4b\kep\c(\z-n^{-1}\nabb n)-2b|n^{-1}N(n)|^2,
\eee
which together with \eqref{josh} implies
\be\label{estblq}
L\left(\lb(b)-b(\d+n^{-1}\nab_Nn)\right)= -2b\rho+h_1,
\ee
where $h$ is given by
\bee
h_1 &=& (\d+n^{-1}\nabla_Nn)\db b-2(\z-\zb)\c\nabb b-L(b)(\d+n^{-1}\nab_Nn)-2b|\kep|^2-2b\d^2\\
&&-4b\kep\c(\z-n^{-1}\nabb n)+2b|n^{-1}N(n)|^2.
\eee
In view of the bootstrap assumptions \eqref{boot1}-\eqref{boot6}, we have the following estimate for $h_1$
\bea
\nn\norm{h_1}_{\tx{1}{4}} &\les& (\no(\d)+\no(n^{-1}\nabla n)+\noo(b)+\no(\kep)+\no(\z)+\no(\zb))^2(1+\norm{b}_{L^\infty})\\
\nn&\les& D^2\ep^2\\
\lab{estblq1}&\les& \ep.
\eea
Next, we decompose the term involving $\rho$ in the right-hand side of \eqref{estblq}. In view of the Bianchi identity \eqref{bianc6}, we have:
$$(n\r,n\s)={}^*\mathcal{D}_1^{-1}\left(\ddb_{nL}(\bb)-\nabb(n)\r+\nabb(n)\s-2n\hchb\c\b-n\db\bb+3n(\zb\r-{}^*\zb\s)\right)$$
which yields
\be\lab{estblq2}
(\rho, \sigma)=L({}^*\mathcal{D}_1^{-1}(\bb))+n^{-1}[{}^*\mathcal{D}_1^{-1},\ddb_{nL}]\bb+h_2,
\ee
where $h_2$ is given by
$$h_2=n^{-1}{}^*\mathcal{D}_1^{-1}\left(-\nabb(n)\r+\nabb(n)\s-2n\hchb\c\b-n\db\bb+3n(\zb\r-{}^*\zb\s)\right).$$
In view of the bootstrap assumptions \eqref{boot1}-\eqref{boot6} and Lemma \ref{lemma:lbt6}, we have the following estimate for $h_2$
\bea\lab{estblq3}
&&\norm{h_2}_{\tx{1}{4}}+\norm{h_2}_{\tx{2}{3}}\\
\nn&\les& \norm{n^{-1}}_{L^\infty}\normm{-\nabb(n)\r+\nabb(n)\s-2n\hchb\c\b-n\db\bb+3n(\zb\r-{}^*\zb\s)}_{\tx{1}{\frac{3}{2}}}\\
\nn&&+\norm{n^{-1}}_{L^\infty}\normm{-\nabb(n)\r+\nabb(n)\s-2n\hchb\c\b-n\db\bb+3n(\zb\r-{}^*\zb\s)}_{\tx{2}{\frac{4}{3}}}\\
\nn&\les& (\no(\nabb n)+\no(\hchb)+\no(\db)+\no(\zb))\norm{(\b, \r, \s, \bb)}_{\li{\infty}{2}}\norm{n^{-1}}_{L^\infty}(1+\norm{n}_{L^\infty})\\
\nn&\les& D\ep^2\\
\nn&\les& \ep,
\eea
where we used the bootstrap assumptions \eqref{boot1}-\eqref{boot6} and the curvature bound \eqref{curvflux1}. Next, we estimate the commutator term in the right-hand side of \eqref{estblq2}. This is done in the following lemma.
\begin{lemma}\lab{lemma:lbt8}
$[{}^*\mathcal{D}_1^{-1},\ddb_{nL}](\bb)$ satisfies the following estimate:
$$\norm{[{}^*\mathcal{D}_1^{-1},\ddb_{nL}](\bb)}_{\tx{2}{3}}+\norm{[{}^*\mathcal{D}_1^{-1},\ddb_{nL}](\bb)}_{\tx{1}{4}}\les\ep.$$
\end{lemma}

We postpone the proof of Lemma \ref{lemma:lbt8} to section \ref{sec:gowinda1} and conclude the estimate for $\lb(b)$ in $\tx{\infty}{4}$. In view of \eqref{estblq} and \eqref{estblq2}, we have
\bee
&&L\left(\lb(b)-b(\d+n^{-1}\nab_Nn)+2b\pi_1({}^*\mathcal{D}_1^{-1}(\bb))\right)\\
&=& 2L(b)\pi_1({}^*\mathcal{D}_1^{-1}(\bb))-2b\pi_1([{}^*\mathcal{D}_1^{-1},\ddb_{nL}](\bb))-2b \pi_1(h_2)+h_1,
\eee
where $\pi_1$ denote the projection in $\mathbb{R}^2$ on the first coordinate. Together with the estimate for transport equations \eqref{estimtransport1}, we obtain
\bee
&&\normm{\lb(b)-b(\d+n^{-1}\nab_Nn)+2b\pi_1({}^*\mathcal{D}_1^{-1}(\bb))}_{\tx{\infty}{4}}\\
&\les& \norm{L(b)}_{\tx{2}{8}}\norm{{}^*\mathcal{D}_1^{-1}(\bb)}_{\tx{2}{8}}+\norm{b}_{L^\infty}(\norm{[{}^*\mathcal{D}_1^{-1},\ddb_{nL}](\bb)}_{\tx{1}{4}}+\norm{h_2}_{\tx{1}{4}})+\norm{h_1}_{\tx{1}{4}}\\
&\les& \noo(b)\norm{\nabb{}^*\mathcal{D}_1^{-1}(\bb)}_{\lh{2}}+\ep,
\eee
where we used in the last inequality the Gagliardo-Nirenberg inequality \eqref{eq:GNirenberg}, the bootstrap assumption \eqref{boot1} for $b$, Lemma \ref{lemma:lbt8}, the estimate \eqref{estblq3} for $h_2$ and the estimate \eqref{estblq1} for $h_1$. Together with the bootstrap assumption \eqref{boot2} for $b$, the estimate \eqref{eq:estimdcal-1} for ${}^*\mathcal{D}^{-1}_1$, and the curvature bound \eqref{curvflux1} for $\bb$, we deduce
\be
\normm{\lb(b)-b(\d+n^{-1}\nab_Nn)+2b\pi_1({}^*\mathcal{D}_1^{-1}(\bb))}_{\tx{\infty}{4}}\les\ep.
\ee
This yields
\bea\lab{estblq4}
\normm{\lb(b)}_{\tx{\infty}{4}}&\les&\normm{-b(\d+n^{-1}\nab_Nn)+2b\pi_1({}^*\mathcal{D}_1^{-1}(\bb))}_{\tx{\infty}{4}}+\ep\\
\nn&\les&\norm{b}_{L^\infty}(\no(\d)+\no(n^{-1}\nab_Nn))+\norm{b}_{L^\infty}\no({}^*\mathcal{D}_1^{-1}(\bb))+\ep\\
\nn&\les& D\ep+\no({}^*\mathcal{D}_1^{-1}(\bb)).
\eea

Next, we estimate the right-hand side of \eqref{estblq4}. In view of the estimate \eqref{eq:estimdcal-1} for ${}^*\mathcal{D}^{-1}_1$ and the curvature bound \eqref{curvflux1} for $\bb$, we have
\be\lab{estblq5}
\norm{\nabb{}^*\mathcal{D}_1^{-1}(\bb)}_{\lh{2}}\les \norm{\bb}_{\lh{2}}\les\ep.
\ee
Also, in view of \eqref{estblq2}, we have
\bea\lab{estblq6}
\nn&&\norm{L({}^*\mathcal{D}_1^{-1}(\bb)}_{\lh{2}}\les \norm{(\rho, \sigma)}_{\lh{2}}+\norm{n^{-1}}_{L^\infty}\norm{[{}^*\mathcal{D}_1^{-1},\ddb_{nL}]\bb}_{\lh{2}}+\norm{h_2}_{\lh{2}}\\
&\les&\ep,
\eea
where we used in the last inequality the curvature bound \eqref{curvflux1} for $(\r, \s)$, the bootstrap assumption \eqref{boot1} for $n$, the commutator estimate of Lemma \ref{lemma:lbt8}, and the estimate \eqref{estblq3} for $h_2$.  Finally, \eqref{estblq5} and \eqref{estblq6} imply
$$\no({}^*\mathcal{D}_1^{-1}(\bb))\les\ep$$
which together with \eqref{estblq4} yields the following estimate for $\lb(b)$
\be\lab{estblq7}
\normm{\lb(b)}_{\tx{\infty}{4}}\les D\ep.
\ee

\begin{remark}
The estimate \eqref{estblq7} contains the bootstrap constant $D$ in its right-hand side. This is not an issue since such an estimate is not part of our bootstrap assumptions \eqref{boot1}-\eqref{boot6}. 
\end{remark}
 
\subsection{Remaining estimates for $\trc$, $\hch$ and $\z$}\lab{sec:remainest}

We first estimate $\nabb\trc$. Differentiating the Raychaudhuri equation \eqref{D4trchi} and using the commutation formula \eqref{comm1}, we obtain:
\be\lab{remain1}
\ddb_{L}\nabb\trc=-\left(\frac{3}{2}\trc+\hch+\db  \right)\nabb\trc-2\hch\nabb\hch+n^{-1}\nabb nL(\trc)-\nabb(\db  )\trc,
\ee
which together with the bootstrap assumptions \eqref{boot1}-\eqref{boot6} and the estimate for transport equations \eqref{estimtransport1} yields:
\bea
\nn\norm{\nabb\trc}_{\xt{2}{\infty}}&\lesssim& \normm{\frac{3}{2}\trc+\hch+\db  }_{\xt{2}{\infty}}\norm{\nabb\trc}_{\lh{2}}+\norm{\hch}_{\xt{\infty}{2}}\norm{\nabb\hch}_{\lh{2}}\\
\nn&&+\norm{n^{-1}\nabb n}_{\tx{2}{8}}\norm{L(\trc)}_{\tx{2}{\frac{8}{3}}}+\norm{\nabb(\db  )}_{\lh{2}}\norm{\trc}_{L^\infty}\\
\nn&\lesssim& D^2\ep^2\\
\lab{remain2}&\lesssim&\ep
\eea
where we used the Raychaudhuri equation \eqref{D4trchi}, the embeddings \eqref{eq:GNirenberg} and \eqref{sobineq1}, and the bootstrap assumption to bound $L(\trc)$:
$$\norm{L(\trc)}_{\tx{2}{\frac{8}{3}}}\lesssim\norm{\chi}_{\tx{2}{8}}\norm{\chi}_{\tx{\infty}{4}}+\norm{\chi}_{\tx{2}{8}}\norm{\db  }_{\tx{\infty}{4}}\lesssim \no(\chi)^2+\no(\db  )^2\lesssim \ep.$$
Note that \eqref{remain2}  improves the estimate for $\nabb\trc$ given by the bootstrap assumption \eqref{boot4}. 

We now estimate $\no(\chi)$. Using the transport equation for $\hch$ \eqref{D4chih}, we obtain:
\be\lab{remain2bis}
\norm{\ddb_L\hch}_{\lh{2}}\lesssim \norm{\trc}_{L^\infty}\norm{\hch}_{\lh{2}}+\norm{\db}_{\lh{4}}\norm{\hch}_{\lh{4}}+\norm{\a}_{\lh{2}}\lesssim\ep+D^2\ep^2\lesssim\ep
\ee
where we have used the curvature bound \eqref{curvflux1} for $\a$, and the bootstrap assumptions \eqref{boot2}-\eqref{boot5} for $\chi$ and $\db$. Next, using the codazzi equation \eqref{Codaz} for $\hch$, we obtain:
\be\lab{remain3}
\norm{\mathcal{D}_2\hch}_{\lh{2}}\lesssim \norm{\nabb\trc}_{\lh{2}}+\norm{\b}_{\lh{2}}+\norm{\chi}_{\lh{4}}\norm{\kep}_{\lh{4}}\lesssim\ep+D^2\ep^2\lesssim\ep
\ee
where we have used \eqref{remain2}, the curvature bound \eqref{curvflux1} for $\b$, and the bootstrap assumptions \eqref{boot3}-\eqref{boot5} for $\chi$ and $\kep$. The Hodge estimate \eqref{eq:estimdcal-1} together with \eqref{remain3} yields:
\be\lab{remain4}
\norm{\nabb\hch}_{\lh{2}}\lesssim \ep.
\ee
\eqref{remain2bis} and \eqref{remain4} imply:
\be\lab{remain4bis}
\no(\hch)\lesssim \ep.
\ee
Note that \eqref{remain4bis}  improves the estimate for $\no(\hch)$ given by the bootstrap assumption \eqref{boot5}. 

We now estimate $\lb\trc$. Using the transport equation for $\mu$ \eqref{D4tmu} and the estimate for transport equations \eqref{estimtransport1},  we obtain:
\bea
\nn\norm{\mu}_{\xt{2}{\infty}} &\lesssim& \bigg\|-\trc \mu +2(\zb_A-\z_A)\nabb_A(\trc)
-2\hch_{AB}\Bigl (2\nabb_A \z_B +  2\z_A\z_B   \\
\nn&&+ (\d +n^{-1}N(n))\hch_{AB} 
-\half\trchb\hch_{AB} -\half\trc\hchb_{AB} \Bigr )
\\
\nn&&
+  (\d  +n^{-1}\nab_N n))\big (\half (\trc)^2 + |\hch|^2 +
\db   \trc \big ) \\ 
\nn&& +\trc 
\bigg (2 (k_{AN} - \z_A) n^{-1} \nab_A n + 2 |n^{-1} N(n)|^2 -
2\rho \\
\nn&&- 2 k_{Nm} k^{m}_N+2|n^{-1}N(n)|^2 - 2n^{-1}N(n)\trc
\bigg)\bigg\|_{\xt{2}{1}}\\
\nn&\lesssim & \norm{\trc}_{L^\infty}\norm{\mu}_{\xt{2}{\infty}} +(\norm{\zb}_{\xt{\infty}{2}} +\norm{\z}_{\xt{\infty}{2}})\norm{\nabb\trc}_{\lh{2}}\\
\nn&&+ \norm{\hch}_{\xt{\infty}{2}} \Bigl (\norm{\nabb\z}_{\lh{2}}+ \norm{\z}^2_{\lh{4}}  \\
 \nn&&+ \norm{(\d +n^{-1}N(n))\hch}_{\lh{2}} +\norm{\trchb\hch}_{\lh{2}} +\norm{\trc\hchb}_{\lh{2}} \Bigr )\\
\nn&&+  \normm{(\d  +n^{-1}\nab_N n))\big (\half (\trc)^2 + |\hch|^2 +
\db   \trc \big )}_{\lh{2}}  \\ 
\nn&& +\norm{\trc}_{L^\infty} 
\bigg\|2 (k_{AN} - \z_A) n^{-1} \nab_A n + 2 |n^{-1} N(n)|^2 -
2\rho \\
\nn&&- 2 k_{Nm} k^{m}_N+2|n^{-1}N(n)|^2 - 2n^{-1}N(n)\trc
\bigg\|_{\lh{2}}\\
\lab{remain5}&\lesssim& D\ep\norm{\mu}_{\xt{2}{\infty}} +D^2\ep^2
\eea
where we used the curvature bound \eqref{curvflux1}, the bootstrap assumptions \eqref{boot1}-\eqref{boot6} and the Sobolev inequality \eqref{sobineq}. \eqref{remain5} yields:
$$\norm{\mu}_{\xt{2}{\infty}}\lesssim\ep$$
which together with the bootstrap assumptions \eqref{boot1}-\eqref{boot6} and the definition \eqref{eqmu} of $\mu$ implies:
\be\lab{remain6}
\norm{\lb\trc}_{\xt{2}{\infty}}\lesssim
\normm{\lb(\mu)+\half (\trc)^2 + \big (\d +n^{-1}\nab_N n\big )\trc}_{\xt{2}{\infty}}
\lesssim\ep+D^2\ep^2\lesssim \ep.
\ee
Note that \eqref{remain6}  improves the estimate for $\lb\trc$ given by the bootstrap assumption \eqref{boot4}. 

We now estimate $\no(\z)$. Using the transport equation for $\z$ \eqref{D4eta}, we obtain:
\be\lab{remain6bis}
\norm{\ddb_L\z}_{\lh{2}}\lesssim \norm{\chi}_{\lh{4}}(\norm{\z}_{\lh{4}}+\norm{\kepb}_{\lh{4}})+\norm{\b}_{\lh{2}}\lesssim\ep+D^2\ep^2\lesssim\ep
\ee
where we have used the curvature bound \eqref{curvflux1} for $\b$, and the bootstrap assumptions \eqref{boot2}-\eqref{boot5} for $\chi$ and $\kepb$. Next, using the div-curl system of equations \eqref{diveta} \eqref{curleta} for $\z$, we obtain:
\bea
\nn\norm{\mathcal{D}_1\z}_{\lh{2}}&\lesssim& \norm{\mu}_{\lh{2}}+\norm{\r}_{\lh{2}}+\norm{\s}_{\lh{2}}+\norm{\chi}^2_{\lh{4}}+\norm{k}^2_{\lh{4}}+\norm{\z}^2_{\lh{4}}\\
\nn&\lesssim&\ep+D^2\ep^2\\
\lab{remain7}&\lesssim&\ep
\eea
where we have used \eqref{remain6}, the curvature bound \eqref{curvflux1} for $\r$ and $\s$, and the bootstrap assumptions \eqref{boot3}-\eqref{boot6} for $\chi$, $k$ and $\z$. The Hodge estimate \eqref{eq:estimdcal-1} together with \eqref{remain7} yields:
\be\lab{remain8}
\norm{\nabb\z}_{\lh{2}}\lesssim \ep.
\ee
\eqref{remain6bis} and \eqref{remain8} imply:
\be\lab{remain8bis}
\no(\z)\lesssim \ep.
\ee
Note that \eqref{remain8bis}  improves the estimate for $\no(\z)$ given by the bootstrap assumption \eqref{boot6}. 

We now estimate $\ddb_{\lb}\hch$. Using the null structure equation \eqref{D3chih}, we obtain:
\bea
&&\lab{remain8ter}\norm{\ddb_{\lb}\hch}_{\lh{2}}\\
\nn&\lesssim& \norm{\nabb\z}_{\lh{2}}+\norm{\chi}_{\lh{4}}(\norm{\chb}_{\lh{4}}+\norm{\d}_{\lh{4}}+\norm{n^{-1}N(n)}_{\lh{4}})+\norm{\z}_{\lh{4}}^2\\
\nn&\lesssim& \ep+D^2\ep^2\\
\nn&\lesssim&\ep
\eea
where we have used \eqref{remain8} and the bootstrap assumptions \eqref{boot1}-\eqref{boot6} for $n$, $\chi$, $\chb$, $\d$ and $\z$. Note that \eqref{remain8ter}  improves the estimate for $\ddb_{\lb}\hch$ given by the bootstrap assumption \eqref{boot5}. 

Finally, \eqref{lapn2}, \eqref{lapn3bis}, \eqref{lapn12}, \eqref{hodgk8}, \eqref{hodgk10}, \eqref{comptg1}, \eqref{comptg16}, \eqref{tracek31}, \eqref{remain2}, \eqref{remain4bis}, \eqref{remain6}, \eqref{remain8bis}, \eqref{remain8ter}, \eqref{remain10}, \eqref{remain12} and \eqref{remain13} improve the bootstrap assumptions \eqref{boot1}-\eqref{boot6}. Thus, there exists a universal constant $D>0$ such that \eqref{boot1}-\eqref{boot6} and \eqref{estblq7} hold. This yields \eqref{estn}-\eqref{estzeta} which concludes the proof of Theorem \ref{thregx}.

\section{Estimates for $\lb\lb\trc$, $\ddb_{\lb}(\z)$ and $\lb\lb(b)$}\lab{sec:secondderlb}

This section is devoted to the proof of Theorem \ref{thregx1}. We assume the following bootstrap assumptions. There exists a function $\gamma$ in $L^2(\R)$ with $\norm{\gamma}_{L^2(\R)}\leq 1$ such that for all $j\geq 0$, we have:
\be\lab{boot7}
\norm{P_j\lb\lb\trc}_{\lh{2}}\lesssim 2^jD\ep+2^{\frac{j}{2}}D\ep\gamma(u),
\ee
\be\label{boot8}
\norm{P_j(\ddb_{\lb}(\z))}_{\lh{2}}\lesssim D^2\ep+2^{-\frac{j}{2}}D^2\ep\gamma(u),
\ee
where $D>0$ is a large enough constant. We will improve on these estimates. Using the estimates obtained in Theorem \ref{thregx}, in particular for $\trc$ and $\hch$, would yield an upper bound for $\lb\lb\trc$ of the following  type  
\be\lab{jfkwaiting}
\norm{P_j\lb\lb\trc}_{\lh{2}}\lesssim 2^{\frac{j}{2}}\ep\gamma(u)+\sum_{l,q}2^j2^{-\frac{|q-l|}{2}}\gamma^{(1)}_q\gamma^{(2)}_l,\textrm{ where }\gamma^{(1)}_q\in\ell^2(\mathbb{N})\textrm{ and }\gamma^{(2)}_l\in\ell^\infty(\mathbb{N})
\ee
which is not summable. This forces us to rely on a Besov improvement for $\trc$, as well as a suitable  decomposition for $\nabb\hch$ (see \eqref{impbes7}). This is done in section \ref{sec:besovtrc}. Then, we derive a system of equations for $\lb\lb\trc$ and $\ddb_{\lb}(\z)$ in section \ref{sec:fautpasdeconner}. This allows us to improve on the bootstrap assumption \eqref{boot8} in section \ref{sec:lbz}, and \eqref{boot7} in section \ref{sec:lbt}. 
Finally, the estimate for $\lb\lb(b)$ is then derived in section \ref{sec:fautpasdeconner1}.

\subsection{Besov improvement for $\trc$ in the time foliation}\label{sec:besovtrc}

In this section, we first define Besov spaces, and then explain how to adapt the ideas in the sequence of papers \cite{FLUX} \cite{LP} \cite{STT} to obtain the Besov improvement for $\trc$.

\subsubsection{Definition of the Besov spaces and first properties}

Following \cite{LP} \cite{FLUX}, we introduce for $0\leq a\leq 1$ and for tensors $F$ on $\ptu$ the Besov norm:
\be\label{besovptu}
\|F\|_{B^a_{2,1}(\ptu)}=\sum_{j\ge 0} 2^{ja}\|P_jF\|_{L^2(\ptu)} +  \| P_{<0}F\|_{L^2(\ptu)},  
\ee
where $P_j$ are the geometric Littlewood-Paley projections on the 2-surfaces $\ptu$. Furthermore, for $\ptu$-tangent tensors $F$ on $\H_u$, $0\leq a\leq 1$,  we introduce the Besov norms:
\bea
\|F\|_{{\BB}^a}&=&\sum_{j\ge 0} 2^{ja} \sup_{0\le t\le 1}\|
P_jF\|_{L^2(\ptu)} +  \sup_{0\le t\le 1}\|
P_{<0}F\|_{L^2(\ptu)}    ,\label{eq:Besovnormt}\\
\|F\|_{{\PP}^a}&=&\sum_{j\ge 0} 2^{ja} \| P_jF\|_{L^2(\H_u)} +
\| P_{<0}F\|_{L^2(\H_u)}.\label{eq:Penrosenormt}
\eea
Note that these Besov spaces in the time foliation correspond to the Besov spaces in the geodesic foliation defined by the norms \eqref{eq:Besovnorm} \eqref{eq:Penrosenorm}. The goal of section \ref{sec:besovtrc} is to prove the following estimates for $\trc$ and $\hch$:
\be\lab{esttrcbesov}
\norm{\nabb\trc}_{\BB^0}\lesssim\ep,
\ee
and
\be\lab{esthchbesov}
\nabb\hch=\ddb_{L}(P)+E\textrm{ with }\no(P)\lesssim\ep\textrm{ and }\norm{E}_{\PP^0}\lesssim\ep.
\ee
Note that the corresponding estimates in the geodesic foliation have been proved in the sequence of papers \cite{FLUX} \cite{LP} \cite{STT}. One may reprove these estimates by adapting the proofs to the context of a time foliation. However, this would be rather lengthy and we suggest here a more elegant solution which consists in identifying the key structure in \cite{FLUX} \cite{LP} \cite{STT} and showing that the analog structure exists in the time foliation. This will be done in the next section.

We conclude this section with several functional inequalities satisfied by the Besov spaces $\BB^a, \PP^a$. Note that properties of the Besov spaces on 2-surfaces derived in \cite{LP} apply to the Besov spaces $B^a_{2,1}$. Indeed, these properties only depend on the fact that $\ptu$ is a 2-surface satisfying the coordinate system assumption \eqref{eq:coordchart} and the assumption \eqref{estgauss2}   on the gauss curvature $K$. In particular, the following estimates are immediate consequences on the estimates in \cite{LP} for $B^a_{2,1}$ (see also section 5 in \cite{FLUX}):
\be\lab{linftybound}
\norm{f}_{L^\infty}\lesssim\norm{f}_{\BB^1}\lesssim \norm{f}_{\tx{\infty}{2}}+\norm{\nabb f}_{\BB^0},
\ee
where $f$ is a scalar function on $\H_u$, 
\be\lab{eq:secondbilBesov:bis}
\| F\cdot H\|_{{\BB}^0}\lesssim \big(\|\nabb F\|_{L_t^\infty L_x^2} +
\|F\|_{L^\infty}\big)\|H\|_{{\BB}^0},
\end{equation}
where $F$ and $H$ are $\ptu$-tangent tensors, and
\be\lab{DinvsurBesov}
\norm{\nabb\cdot\mathcal{D}^{-1}F}_{\PP^a}\lesssim\norm{F}_{\PP^a}
\ee
where $0\leq a<1$, $F$ is a $\ptu$ tangent tensor on $\H_u$, and $\mathcal{D}^{-1}$ denotes one of the operators $\mathcal{D}^{-1}_1$, $\mathcal{D}^{-1}_2$, ${}^*\mathcal{D}^{-1}_1$, ${}^*\mathcal{D}^{-1}$.  Also, for $0\leq a<\half$ and $\frac{2}{2-a}<p\leq 2$, we have:
\be\lab{DinvsurBesovbis}
\norm{\mathcal{D}^{-1}F}_{\PP^a}\lesssim\norm{F}_{\tx{2}{p}}.
\ee
Finally, we shall make use of the following non sharp embedding and product estimates. For any $\ptu$-tangent tensors $F, G$, and for any $0\leq a<\half$, we have:
\bea
\lab{nonsharpembedding}\norm{F}_{\BB^a}\les \no(F)\\
\lab{nonsharpprod1}\norm{F\cdot G}_{\PP^a}\les \noo(F)\c\norm{G}_{\PP^a}\\
\lab{nonsharpprod2}\norm{F\c G}_{\PP^a}\les \no(F)\c (\norm{G}_{\lh{2}}+\norm{\nabb G}_{\lh{2}}).
\eea

\subsubsection{Structure of the commutators in the time foliation}

As noted at the end of the previous section, the results from the paper \cite{LP} on 2-surfaces immediately apply to $\ptu$. We shall now show that results from the paper \cite{ STT} true in the geodesic foliation apply also to the time foliation due to a similar structure of commutators.

Let $A$ denote $A=n\chi$. Then, the estimates \eqref{estn} for $n$, \eqref{esttrc} for $\trc$ and \eqref{esthch} for $\hch$ of Theorem \ref{thregx} proved in section \ref{sec:regxproof} imply:
\be\lab{estA}
\norm{A}_{\xt{\infty}{2}}+\no(A)\lesssim\ep.
\ee
In view of \eqref{estA} and the commutator identities \eqref{comm5} and \eqref{comm6}, we have:
\be\lab{commA1}
[\ddb_{nL},\nabb]f=A\cdot \nabb f,
\ee
\be\lab{commA2}
[\ddb_{nL},\lap]f=A\cdot\nabb^2f+\nabb A\cdot\nabb f+A\cdot A\cdot \nabb f
\ee
where $f$ is a scalar function on $\H_u$ and:
\be\lab{commA3}
[\ddb_{nL},\nabb]F=A\cdot \nabb F+n\b\cdot F+A\cdot A\cdot F,
\ee
\be\lab{commA4}
[\ddb_{nL},\lap]F=A\cdot\nabb^2F+\nabb A\cdot\nabb F+A\cdot A\cdot \nabb F+n\b\cdot\nabb F+\nabb(n\b\cdot\nabb F+A\cdot A\cdot F)
\ee
where $F$ is a $\ptu$-tangent tensor on $\H_u$. Note that the structure of the commutators \eqref{commA1}-\eqref{commA4} together with the estimate \eqref{estA} for $A$ is the same structure as in the case of a geodesic foliation with the correspondence:
\be\lab{correspondancegt}
L'\rightarrow nL,\,\chi'\rightarrow n\chi\textrm{ and }\b'\rightarrow n\b
\ee
where $L'$, $\chi'$ and $\b'$ have been defined in section \ref{sec:defgeodfol}. 

The proofs of the sharp trace theorems in the paper \cite{STT} rely on the following assumptions (see section 3 of \cite{STT}) where we translate for the time foliation using the correspondence \eqref{correspondancegt}:
\begin{itemize}
\item[{\bf S1}] The two surfaces $\ptu$ satisfy the coordinates system assumption \eqref{eq:coordchart}, the calculus inequalities of section \ref{sec:calcineq} and the geometric Littlewood-Paley theory of section \ref{sec:geompal}

\item[\bf{S2}] The Gauss curvature $K$ of $\ptu$ satisfies the bound \eqref{estgauss1} \eqref{estgauss2}

\item[\bf{S3}] There is $A$ satisfying \eqref{estA} such that we have the commutator structure \eqref{commA1}-\eqref{commA4}

\item[\bf{S4}] $n\b$ satisfies the curvature flux bound $\norm{n\b}_{\lh{2}}\les\ep$ (which follows from the curvature bound \eqref{curvflux1} and the estimate \eqref{estn} for $n$)
\end{itemize}
Since the proof of the sharp trace theorems in \cite{STT} only rely on the structural assumptions {\bf S1}-{\bf S4}, they immediately extend to the case of a time foliation. In particular, we obtain the following analog of the sharp trace theorems in \cite{STT} (see section 4 of \cite{STT}):
\begin{proposition}\label{prop:mainlemmaapplicationt}
Assume that the $\ptu$-tangent tensor $U$  satisfies $U(0)=0$ and the following transport
 equation along  $\H_u$:
$$\ddb_{nL}U+a\trc U =F_1\cdot\ddb_{nL} P+F_2\cdot W,$$
where $a$ is some positive number. Then,
\be\label{eq:mainleprop1t}
\|U\|_{{\BB}^0}\lesssim\big(\no(F_1)+\|F_1\|_{L_{x'}^\infty L_t^2}\big)\cdot\no(P)
+\big(\no(F_2)+\|F_2\|_{L_{x'}^\infty L_t^2}\big)\cdot\|W\|_{{\PP}^0}.
\ee
\end{proposition}



We also obtain the following useful commutator estimates:
\begin{lemma}\lab{lemma:commutdivb}
For a given 1-form $F$, let $w$ the solution of the scalar transport equation
$$nL(w)=\divb(F),\, w=0\textrm{ on }\pou,$$
and let $W$ be a solution of the equation
$$\ddb_{nL}W-n\chi\c W=F,\, W=0\textrm{ on }\pou.$$
Then, for any $1\leq p\leq 2$,
$$\norm{\divb(W)-w}_{\xt{p}{\infty}}\les\ep\norm{F}_{\xt{\frac{2p}{2-p}}{1}}.$$
\end{lemma}

\begin{lemma}
For any $\ptu$-tangent tensor $F$ and all $1\leq q<2$, we have:
\be\lab{supercommut1}
\norm{[P_j,\ddb_{nL}]F}_{\tx{q}{2}}+2^{-j}\norm{\nabb [P_j,\ddb_{nL}]F}_{\tx{q}{2}}\les 2^{-\frac{j}{2}_+}\no(F)
\ee
(by $2^{-\frac{j}{2}_+}$ we mean $2^{-aj}$ for $a<\frac{1}{2}$ arbitrarily close to $\frac{1}{2}$), while for $q=1$,
\be\lab{supercommut2}
\norm{[P_j,\ddb_{nL}]F}_{\tx{1}{2}}+2^{-j}\norm{\nabb [P_j,\ddb_{nL}]F}_{\tx{1}{2}}\les 2^{-j_+}\no(F).
\ee 
\end{lemma}
Finally, using Proposition \ref{prop:mainlemmaapplicationt}, 
we may prove the following version of the sharp classical trace theorem. 
\begin{corollary}\label{corr:funny-classical-tracet} 
Assume $F$ is an $\ptu$-tangent tensor which
admits a decomposition
of the form, $\nabb F=B\ddb_{nL}P+E$.
Then,
\be\lab{eq:funny-classical-tracet}
\|F\|_{L_{x'}^\infty L_t^2}\lesssim \no(F)+\no(P)(\norm{B}_{L^\infty}+\norm{\nabb B}_{L^2_{x'}L^\infty_t}+\norm{\ddb_{L}B}_{L^2_{x'}L^\infty_t})+\|E\|_{{\PP}^0}.
\end{equation}
\end{corollary}
The proof of \eqref{eq:funny-classical-tracet} is the analog of the proof of 
the estimate \eqref{eq:funny-classical-trace} so we skip it.

\subsubsection{Structure of the Bianchi identities in the time foliation}

In this section, we will show that results from the paper \cite{FLUX} true in the geodesic foliation apply also to the time foliation due to a similar structure of the Bianchi identities. We first enlarge the correspondence \eqref{correspondancegt} with the general philosophy that $L'$ should correspond to $nL$ and $\lb'$ to $n^{-1}\lb$:
\be\lab{correspondancegtbis}
\begin{array}{l}
L'\rightarrow nL,\,\lb'\rightarrow n^{-1}\lb,\,e'_A\rightarrow e_A\\
\chi'\rightarrow n\chi,\,\chb'\rightarrow n^{-1}\chb\\
\b'\rightarrow n\b,\,\r'\rightarrow \r,\,\s'\rightarrow \s
\end{array}
\ee
where $L'$, $\lb'$, $\chi'$, $\chb'$, $\b'$, $\r'$ and $\s'$ have been defined in section \ref{sec:defgeodfol}. Following \cite{FLUX}, we define $\check{\r}, \check{\s}$ as:
\be\lab{def:checkrs}
\check{\r}=\r-\half\hch\cdot\hchb,\,\check{\s}=\s-\half\hch\wedge\hchb.
\ee
Multiplying the Bianchi identities \eqref{bianc2} and \eqref{bianc4} by $n$ together with the null structure equations for $\chi$ and $\chb$ yields:
\be\lab{bianccheck}
\begin{array}{lll}
\ds nL(\check{\r})&=&\ds\divb(n\b)-\kepb\c (n\b)+\half(n\hch)\c\bigg(\nabb\widehat{\otimes}\kepb-\kepb\widehat{\otimes}\kepb+(n\trc)\c(n^{-1}\hchb)\\
&& \ds +\frac{1}{2}(n^{-1}\trchb)\c(n\hch)\bigg),\\
\ds nL(\check{\s})&=&\ds -\curlb(n\b)+\kepb\wedge (n\b)+\half(n\hch)\wedge\bigg(\nabb\widehat{\otimes}\kepb-\kepb\widehat{\otimes}\kepb+(n\trc)\c(n^{-1}\hchb)\bigg).
\end{array}
\ee
We now denote $A=(n\chi,\kepb)$ which together with the estimates \eqref{estn} for $n$, the estimates \eqref{estk} for $\kepb$, the estimates \eqref{esttrc} for $\trc$ and \eqref{esthch} for $\hch$ of Theorem \ref{thregx} proved in section \ref{sec:regxproof} still imply the estimate \eqref{estA} for $A$. We also denote $\underline{A}=n\chb$ which in view of the estimates \eqref{estn} for $n$, the estimates \eqref{estk} for $k$, the estimates \eqref{esttrc} for $\trc$ and \eqref{esthch} for $\hch$ of Theorem \ref{thregx} proved in section \ref{sec:regxproof} satisfies the following estimate:
\be\lab{estAb}
\no(\underline{A})\les\ep.
\ee
In view of the definition of $A$ and $\underline{A}$ together with \eqref{bianccheck}, we have:
\be\lab{bianccheckstruct}
nL(\check{\r},-\check{\s})=\mathcal{D}_1(n\b)+A\c (n\b+\nabb A+A\c\underline{A}).
\ee

We now consider a decomposition for $\nabb\mathcal{D}_2^{-1}\mathcal{D}_1^{-1}L(\check{\r},\check{\s})$ which is the analog of the one derived in section 6 of the paper \cite{FLUX}. It relies on the assumptions {\bf S1-S4} together with the following additional  assumptions where we translate for the time foliation using the correspondence \eqref{correspondancegtbis}:
\begin{itemize}
\item[\bf{S5}] $(\check{\r},\check{\s})$ satisfies the curvature flux bound $\norm{\check{\r}}_{\lh{2}}+\norm{\check{\s}}_{\lh{2}}\les\ep$ (which follows from the curvature bound \eqref{curvflux1}, the estimate \eqref{estk} for $k$, and the estimates \eqref{esttrc} and \eqref{esthch} for $\chi$)

\item[\bf{S6}] $nL(\check{\r},-\check{\s})$ has the structure \eqref{bianccheckstruct}

\item[\bf{S7}] The functional inequalities \eqref{DinvsurBesov}, \eqref{DinvsurBesovbis}, \eqref{nonsharpprod1} and \eqref{nonsharpprod2} are satisfied
\end{itemize}
Since the proof of the estimate derived in section 6 of the paper \cite{FLUX} only rely on the structural assumptions {\bf S1}-{\bf S7}, they immediately extend to the case of a time foliation. In particular, we obtain the following analogs of the decompositions derived in section 6 of the paper \cite{FLUX}:
\be\lab{estfluxcorres1:0}
n\b=\ddb_{nL}P_1+E_1
\ee
and
\be\lab{estfluxcorres1}
\nabb\mathcal{D}_2^{-1}\mathcal{D}_1^{-1}nL(\check{\r},\check{\s})=\ddb_{nL}P_2+E_2
\ee
where $P_1, P_2, E_1$ and $E$ satisfy the bounds:
\be\lab{estfluxcorres2}
\no(P_1)+\no(P_2)+\norm{E_1}_{\PP^0}+\norm{E_2}_{\PP^0}\les\ep.
\ee

\subsubsection{Decomposition of $\nabb(n\hch)$}

We now in position to prove the decomposition \eqref{esthchbesov} for $\nabb\hch$. We first derive an equation for $n\hch$. Multiplying the Codazzi type  equation \eqref{Codaz} for $\hch$ by $n$, we obtain:
\be\lab{impbes1}
\mathcal{D}_2(n\hch)+\kepb\c (n\hch)=\half(\nabb(n\trc)+\kepb(n\trc))-n\b,
\ee
which yields:
\be\lab{impbes2}
\nabb(n\hch)=\nabb\mathcal{D}_2^{-1}\bigg(-\kepb\c (n\hch)+\half(\nabb(n\trc)+\kepb(n\trc))-n\b\bigg).
\ee
Now, in view of \eqref{bianccheckstruct}, we have:
\be\lab{impbes3}
n\b=\mathcal{D}_1^{-1}\bigg(nL(\check{\r},-\check{\s})-A\c (n\b+\nabb A+A\c\underline{A})\bigg),
\ee
where $A$ satisfies \eqref{estA} and $\underline{A}$ satisfies \eqref{estAb}. Injecting \eqref{impbes3} in \eqref{impbes2} yields:
\bea
\lab{impbes4}\nabb(n\hch)&=&-\nabb\mathcal{D}_2^{-1}\mathcal{D}_1^{-1}nL(\check{\r},-\check{\s})+\nabb\mathcal{D}_2^{-1}\mathcal{D}_1^{-1}\bigg(A\c (n\b+\nabb A+A\c\underline{A})\bigg)\\
\nn&&+\nabb\mathcal{D}_2^{-1}\bigg(-\kepb\c (n\hch)+\half(\nabb(n\trc)+\kepb(n\trc))\bigg).
\eea

We estimate the second term in the right-hand side of \eqref{impbes4}. Using the embedding \eqref{sobineq1}, the estimate \eqref{DinvsurBesov} with $a=0$, and the estimate \eqref{DinvsurBesovbis} with $a=0$ and $p=\frac{4}{3}$, we have:
\bea
\lab{impbes5}&&\normm{\nabb\mathcal{D}_2^{-1}\mathcal{D}_1^{-1}\bigg(A\c (n\b+\nabb A+A\c\underline{A})\bigg)}_{\PP^0}\\
\nn&\les &\normm{A\c (n\b+\nabb A+A\c\underline{A})}_{\tx{2}{\frac{4}{3}}}\\
\nn&\les &\norm{A}_{\tx{\infty}{4}}(\norm{n}_{L^\infty}\norm{\b}_{\lh{2}}+\norm{\nabb A}_{\lh{2}}+\norm{A}_{\tx{\infty}{4}}\norm{\underline{A}}_{\tx{\infty}{4}})\\
\nn&\les &\no(A)(\norm{n}_{L^\infty}\norm{\b}_{\lh{2}}+\no(A)(1+\no(\underline{A})))\\
\nn&\les& \ep,
\eea
where we used the curvature bound \eqref{curvflux1} for $\b$, the estimate \eqref{estn} for $n$, the estimate \eqref{estA} for $A$ and the estimate \eqref{estAb} for $\underline{A}$.

We estimate the last term in the right-hand side of \eqref{impbes4}. Using the estimate \eqref{DinvsurBesov} and the non sharp product estimates estimates \eqref{nonsharpprod1} and \eqref{nonsharpprod2}, we have:
\bea
\lab{impbes6}&&\normm{\nabb\mathcal{D}_2^{-1}\bigg(-\kepb\c (n\hch)+\half(\nabb(n\trc)+\kepb(n\trc))\bigg)}_{\PP^0}\\
\nn&\les &\normm{-\kepb\c (n\hch)+\half(\nabb(n\trc)+\kepb(n\trc))}_{\PP^0}\\
\nn&\les & \norm{\trc\nabb n}_{\PP^0}+\noo(n)(\norm{\nabb\trc}_{\PP^0}+\norm{\kepb\c\hch}_{\PP^0}+\norm{\kepb\trc}_{\PP^0})\\
\nn&\les & \noo(n)(\no(\trc)+\norm{\nabb\trc}_{\PP^0}+\no(\kepb)\no(\chi))\\
\nn&\les &\ep+\norm{\nabb\trc}_{\PP^0},
\eea
where we used the estimate \eqref{estn} for $n$, the estimate \eqref{estk} for $\kepb$ and the estimates \eqref{esttrc} and \eqref{esthch} for $\chi$.

Finally, the decomposition \eqref{estfluxcorres1} for $\nabb\mathcal{D}_2^{-1}\mathcal{D}_1^{-1}nL(\check{\r},-\check{\s})$ together with the estimate \eqref{estfluxcorres2} and\eqref{impbes4}-\eqref{impbes6} yields the following decomposition for $\nabb(n\hch)$:
\be\lab{impbes7}
\nabb(n\hch)=\ddb_{nL}P+E,
\ee
where $P$ and $E$ satisfy the following estimate:
\be\lab{impbes8}
\no(P)+\norm{E}_{\PP^0}\les \ep+\norm{\nabb\trc}_{\PP^0}.
\ee

\subsubsection{Decomposition of $\nabb(n\db)$}

In order to obtain a Besov improvement for $\trc$, we need to derive for $\nabb(n\db)$ the analog of the decomposition for $\nabb'k_{L'L'}$ derived in \eqref{tracek15}. Recall from \eqref{tracek1} that $\db=k_{LL}$ with $k_{LL}=-\gg(\dd_LT,L)$. Thus, we have:
\bee
\nn\nabb_{e_A}\db&=&-\dd_{e_A}\gg(\dd_{L}T,L)=-\gg(\dd_{e_A}\dd_{L}T,L)-\gg(\dd_{L}T,D_{e_A}L)\\
\nn&=&-\gg(\dd_{L}\dd_{e_A}T,L)-\rr_{e_ALTL}-\gg(\dd_{[e_A,L]}T,L)-\gg(-\db  N,\chi_{AB}e_B-\kep_AL)\\
\nn&=&-\ddb_{L}\kep_{A}-\frac{1}{2}\b_A+\chi_{AB}(\kep_B+\kepb_B)-n^{-1}\nabb_An\db,
\eee
which after multiplication by $n$ yields:
\be\lab{impbes9}
\nabb(n\db)=-\ddb_{nL}\kep-\frac{1}{2}\b+\chi\c(\kep+\kepb).
\ee
The estimates \eqref{estn} and \eqref{estk} for $\kep$ and $\kepb$, the estimates \eqref{esttrc} and \eqref{esthch} for $\chi$, and the non sharp product estimate \eqref{nonsharpprod2} yield:
\be\lab{impbes10}
\no(\kep)+\norm{\chi\c(\kep+\kepb)}_{\PP^0}\les\ep+\no(\chi)(\no(\kep)+\no(\kepb))\les\ep.
\ee
Finally, \eqref{impbes9}, \eqref{impbes10} and the decomposition of $\b$ given by \eqref{estfluxcorres1:0} \eqref{estfluxcorres2} yield:
\be\lab{impbes11}
\nabb(n\db)=\ddb_{nL}P+E,
\ee
where $P$ and $E$ satisfy:
\be\lab{impbes12}
\no(P)+\norm{E}_{\PP^0}\les\ep.
\ee

\subsubsection{Besov improvement for $\trc$}

In view of \eqref{impbes8}, we need an estimate for $\norm{\nabb\trc}_{\PP^0}$. We multiply the transport equation \eqref{remain1} satisfied by $\nabb\trc$. We obtain:
\be\lab{impbes13}
\ddb_{nL}\nabb\trc=-n\left(\frac{3}{2}\trc+\hch+\db  \right)\nabb\trc-2\hch\nabb(n\hch)+\nabb n(2|\hch|^2+L(\trc)-\kepb\trc)-\nabb(n\db  )\trc.
\ee
Using the decomposition \eqref{impbes7} for $\nabb(n\hch)$ and the decomposition \eqref{impbes11}  for $\nabb(n\db)$, we obtain:
\be\lab{impbes14}
-2\hch\nabb(n\hch)-\nabb(n\db)\trc=F_1\ddb_{L}P+F_2W
\ee
where in view of \eqref{impbes8}, \eqref{impbes12} and the estimates \eqref{esttrc} \eqref{esthch} for $\chi$, we have:
\be\lab{impbes15}
\no(F_1)+\norm{F_1}_{\xt{\infty}{2}}+\no(F_2)+\norm{F_2}_{\xt{\infty}{2}}+\no(P)+\norm{W}_{\PP^0}\les\ep+\|\nabb\trc\|_{{\PP}^0}.
\ee
Also, using the Raychaudhuri equation \eqref{D4trchi}, we may rewrite the third term in the right-hand side of \eqref{impbes13} as:
\be\lab{impbes16}
\nabb n(2|\hch|^2+L(\trc)-\kepb\trc)=\chi W_1
\ee
where in view of the estimate \eqref{estn} for $n$, the estimate \eqref{estk} for $\db$ and $\kepb$, the estimates \eqref{esttrc} \eqref{esthch} for $\chi$, and the non sharp product estimate \eqref{nonsharpprod2}, $W_1=\nabb n\c(\chi+\db+\kepb)$ satisfies:
\be\lab{impbes17}
\norm{W_1}_{\PP^0}\les \no(\nabb n)(\no(\chi)+\no(\kepb)+\no(\db))\les\ep.
\ee
Using the estimate \eqref{estn} for $n$, the estimate \eqref{estk} for $\kepb$, the estimates \eqref{esttrc} \eqref{esthch} for $\chi$, we also have:
\bea
&&\nn\no\left(n\left(\frac{3}{2}\trc+\hch+\db  \right)\right)+\normm{n\left(\frac{3}{2}\trc+\hch+\db  \right)}_{\xt{\infty}{2}}\\
\nn&\les& \noo(n)(\no(\chi)+\no(\db)+\norm{\chi}_{\xt{\infty}{2}}+\norm{\db}_{\xt{\infty}{2}})\\
\lab{impbes18}&\les&\ep.
\eea
Finally, \eqref{impbes13}-\eqref{impbes18} yield:
\be\lab{impbes19}
\ddb_{nL}\nabb\trc=F_1\ddb_{L}P+F_2W+F_3\nabb\trc
\ee
where $F_1, F_2, F_3, P$ satisfy:
\be\lab{impbes20}
\no(F_1)+\norm{F_1}_{\xt{\infty}{2}}+\no(F_2)+\norm{F_2}_{\xt{\infty}{2}}+\no(F_3)+\norm{F_3}_{\xt{\infty}{2}}+\no(P)+\norm{W}_{\PP^0}\les\ep.
\ee
We now apply Proposition \ref{prop:mainlemmaapplicationt} and obtain from \eqref{impbes19} \eqref{impbes20} the following Besov improvement for $\nabb\trc$:
$$\|\nabb\trc\|_{{\BB}^0}\lesssim \ep\|\nabb\trc\|_{{\PP}^0}+\ep,$$
and the smallness of $\ep$ finally yields:
\be\label{impbes21}
\|\nabb\trc\|_{{\BB}^0}\lesssim \ep.
\ee
Coming back to the decomposition \eqref{impbes7} \eqref{impbes8} of $\nabb(n\hch)$ and using \eqref{impbes21}, we obtain:
\be\lab{impbes22}
\nabb(n\hch)=\ddb_{nL}P+E\textrm{ with }\no(P)+\norm{E}_{\PP^0}\les \ep.
\ee
\eqref{impbes21} and \eqref{impbes22} yield the desired estimates \eqref{esttrcbesov} and \eqref{esthchbesov}.

\subsection{Structure equations for $\lb\lb\trc$ and $\ddb_{\lb}(\z)$}\lab{sec:fautpasdeconner}

The goal of this section is to prove the following proposition.
\begin{proposition}
Let $\mu_1=b\lb(\mu)$. Then, $\mu_1$ satisfies the following transport equation:
\bea
\nn L(\mu_1)+\trc\mu_1&=&-2b\ddb_{\lb}(\z)\c\nabb\trc
-2b\hch\c\bigg(\nabb\widehat{\otimes}\ddb_{\lb}(\z)+b^{-1}\nabb b\ddb_{\lb}(\z)+2\ddb_{\lb}\z\widehat{\otimes}\z\bigg)\\
\lab{estlbmu1}&&+2\trc bn^{-1}\nabb n\c\ddb_{\lb}(\z)+\divb(F_1)+f_2
\eea
where the $\ptu$-tangent vectorfield $F_1$ and the scalar function $f_2$ satisfy the estimates:
\be\lab{estlbmu1bis}
\norm{F_1}_{\lh{2}}+\norm{f_2}_{\lh{1}}\lesssim\ep.
\ee
Furthermore, $\ddb_{\lb}\z$ satisfies the following Hodge system:
\be\lab{estlbmu2}
\begin{array}{lll}
\ds\divb(\ddb_{\lb}\z)&=&\ds\frac{b^{-1}}{2}\mu_1-b^{-1}\nabb b\c\ddb_{\lb}(\z)-2\z\c\ddb_{\lb}(\z)+\divb(\bb)+h_1,\\
\ds\curlb(\ddb_{\lb}\z)&=&\ds -b^{-1}\nabb b\wedge \ddb_{\lb}(\z)-\curlb(\bb)+h_2,
\end{array}
\ee
where the scalar functions $h_1, h_2$ satisfy the estimates:
\be\lab{estlbmu2bis}
\norm{h_1}_{\lh{1}}+\norm{h_2}_{\lh{1}}\lesssim\ep.
\ee
\end{proposition}

\begin{proof}
We start with the proof of \eqref{estlbmu1} \eqref{estlbmu1bis}. We differentiate the transport equation \eqref{D4tmu} satisfied by $\mu$ with respect to $\lb$. We have:
\bea
\lab{lbmu1}\lb(L(\mu))  &=& -\trc\lb(\mu)-2\ddb_{\lb}(\z)\c\nabb\trc+2(\zb-\z)\c\ddb_{\lb}(\nabb\trc)\\
\nn&& -2\hch\c\Bigl (\ddb_{\lb}(\nabb\widehat{\otimes}\z) + 2\z\widehat{\otimes}\ddb_{\lb}(\z)\Bigr )\\
\nn&& -\trc 
\bigg (2\lb(\divb\z)+4\z\c\ddb_{\lb}(\z)-4  \ddb_{\lb}(\z)\c n^{-1} \nabb n+4\lb(\rho)+4\hch\c\ddb_{\lb}(\etah)\bigg)+f_2^1,
\eea
where $f_2^1$ is given by:
\bea
\lab{lbmu2}f_2^1  &=& -\lb(\trc) \mu+2\ddb_{\lb}(\zb)\c\nabb\trc-2\ddb_{\lb}(\hch)\c\Bigl (\nabb\widehat{\otimes}\z + \z\widehat{\otimes}\z  -\d\hch\Bigr )\\
\nn&&+2\hch\c(\lb(\d)\hch+\d\ddb_{\lb}(\hch))
-\lb(\trc) 
\bigg (2\divb\z+2\z\c\z+4 (\kep - \z)\c n^{-1} \nabb n \\
\nn&&-2\db(\d+n^{-1}\nab_Nn)+
4\rho-\half\trc\trchb+2|\kep|^2+3|\hch|^2+4\hch\c\etah-2|n^{-1}N(n)|^2\bigg)\\
\nn&& -\trc 
\bigg (4 \ddb_{\lb}(\kep)\c n^{-1} \nabb n+4 (\kep - \z)\c \ddb_{\lb}(n^{-1} \nabb n)-2\lb(\db)(\d+n^{-1}\nab_Nn)\\
\nn&&-2\db(\lb(\d)+\lb(n^{-1}\nab_Nn)) -\half\lb(\trc)\trchb-\half\trc\lb(\trchb)+4\kep\ddb_{\lb}(\kep)+6\hch\ddb_{\lb}(\hch)\\
\nn&&+4\ddb_{\lb}(\hch)\c\etah-4n^{-1}N(n)\ddb_{\lb}(n^{-1}N(n))\bigg).
\eea
The curvature bound \eqref{curvflux1} for $\r$ and the estimates \eqref{estn}-\eqref{estzeta} obtained in Theorem \ref{thregx} yield:
\bea
\nn\norm{f^1_2}_{\lh{1}}&\lesssim&\norm{\ddb_{\lb}(\chi)}^2_{\lh{2}}+\no(\chi)^2+\no(\z)^2+\no(k)^2+\no(\nab n)^2+\norm{\r}_{\lh{2}}^2\\
\nn&&+\norm{\nab^2n}_{\lh{2}}^2+\norm{\ddb_{\lb}(\kep)}_{\lh{2}}^2+\norm{\lb(\d)}_{\lh{2}}^2+\norm{n-1}_{\lh{\infty}}\\
\lab{lbmu3}&\les&\ep.
\eea
We now estimate various terms in \eqref{lbmu1}. Note first from the commutator formula \eqref{comm3} that we have:
\bea
\lab{lbmu4}\lb(L(\mu))&=&L(\lb(\mu))+[\lb,L](\mu)\\
\nn&=& L(\lb(\mu))-\db\lb(\mu)+(\d+n^{-1}\nab_Nn)L(\mu)+2(\z-\zb)\c\nabb\mu\\
\nn&=& L(\lb(\mu))-\db\lb(\mu)+(\d+n^{-1}\nab_Nn)L(\mu)+2\divb\left((\z-\zb)\mu\right)\\
\nn&&-2(\divb(\z)-\divb(\zb))\mu.
\eea
Using the commutator formula \eqref{comm2}, we have:
\bea
\lab{lbmu5}
(\zb-\z)\c\ddb_{\lb}(\nabb\trc)&=&(\zb-\z)\c\nabb(\lb(\trc))+(\zb-\z)\c[\ddb_{\lb},\nabb](\trc)\\
\nn&=&\divb((\zb-\z)\lb(\trc))-(\divb(\zb)-\divb(\z))\lb(\trc)\\
\nn&&+(\zb-\z)\c(-\chb\nabb\trc+\xib L(\trc)+b^{-1}\nabb b\lb(\trc)),
\eea
\bea
\lab{lbmu6}
\ddb_{\lb}(\nabb\widehat{\otimes}\z)&=&\nabb\widehat{\otimes}(\ddb_{\lb}\z)+[\ddb_{\lb},\nabb]
\widehat{\otimes}\z\\
\nn&=& \nabb\widehat{\otimes}(\ddb_{\lb}\z)-\chb\nabb\z+\xib \ddb_L(\z)+b^{-1}\nabb b\ddb_{\lb}(\z)+(\chi\xib+\chb\z+\bb)\z,
\eea
and
\bea
\lab{lbmu7}
\lb(\divb(\z))&=&\divb(\ddb_{\lb}(\z))+[\ddb_{\lb},\divb]\z\\
\nn&=& \divb(\ddb_{\lb}(\z))-\chb\c\nabb\z+\xib\c\ddb_L(\z)+b^{-1}\nabb b\c\ddb_{\lb}(\z)
+(\chi\xib+\chb\z+\bb)\z.
\eea
Also, using the Bianchi identity \eqref{bianc3}, we have:
\bea
\lab{lbmu8}\trc\lb(\rho)&=&-\trc\divb\bb-\half\trc\hch\c\ab+2\trc\xib\c\b+\trc(\kep-2\z)\c\bb\\
\nn&=&-\divb(\trc\bb)+\nabb\trc\c\bb-\half\trc\hch\c\ab+2\trc\xib\c\b+\trc(\kep-2\z)\c\bb.
\eea

We now consider the term $\trc\hch\c\ddb_{\lb}(\etah)$ in the right-hand side of \eqref{lbmu1}. We start by computing $\ddb_{\lb}\eta$. We have:
\bee
\ddb_{\lb}(\eta)_{AB}&=&\lb(k_{AB})-\eta(\ddb_{\lb}e_A,e_B)-\eta(e_A,\ddb_{\lb}e_B)\\
\nn&=&-\gg(\dd_{\lb}\dd_{e_A}T,e_B)+\gg(\dd_{\ddb_{\lb}e_A}T,e_B)-\gg(\dd_{e_A}T,\dd_{\lb}e_B-\ddb_{\lb}e_B)\\
\nn&=&-\gg(\dd_{e_A}\dd_{\lb}T,e_B)-\gg(\dd_{[\lb,e_A]}T,e_B)+\rr_{\lb ATB}+\gg(\dd_{\ddb_{\lb}e_A}T,e_B)\\
\nn&&-\gg(\dd_{e_A}T,\dd_{\lb}e_B-\ddb_{\lb}e_B)\\
\nn&=&-\nabb_A\kep_B-n^{-1}\nabb_A\nabb_Bn+n^{-2}\nabb_An\nabb_Bn+\gg(\dd_{\lb}T,\dd_{e_A}e_B-\ddb_{e_A}e_B)\\
\nn&&-\gg(\dd_{\dd_{\lb}e_A-\ddb_{\lb}e_A-\dd_{e_A}\lb}T,e_B)+\half\ab_{AB}-\half\r\d_{AB}+\half\s\in_{AB}\\
\nn&&-\gg(\dd_{e_A}T,\dd_{\lb}e_B-\ddb_{\lb}e_B)
\eee
which together with the Ricci equations \eqref{ricciform} yields:
\bea
\lab{lbmu9}
\ddb_{\lb}(\eta)_{AB}&=&-n^{-1}\nabb_A\nabb_Bn+n^{-2}\nabb_An\nabb_Bn+\half\ab_{AB}-\half\r\d_{AB}\\
\nn&&-\nabb_A\kep_B-\chb_{AC}\eta_{CB}+\xib_A\kepb_B+(\kep_A-\z_A)(n^{-1}\nabb_Bn+\kep_B)+(\xib_B-\z_B)\kep_A.
\eea
In view of \eqref{lbmu9}, we have:
\bea
\lab{lbmu10}
\trc\hch\c\ddb_{\lb}(\etah)&=& \trc\hch\c\bigg(-n^{-1}\nabb^2n+n^{-2}\nabb n\nabb n+\half\ab\\
\nn&&-\nabb\kep-\chb\eta+\xib\kepb+(\kep-\z)(n^{-1}\nabb n+\kep)+(\xib-\z)\kep\bigg).
\eea

Now, \eqref{lbmu1} together with \eqref{lbmu4}-\eqref{lbmu8} and \eqref{lbmu10} yields:
\bea
\lab{lbmu11}L(\lb(\mu))  &=& \db\lb(\mu)-\trc\lb(\mu)-2\ddb_{\lb}(\z)\c\nabb\trc\\
\nn&&-2\hch\c\Bigl (\nabb\widehat{\otimes}\ddb_{\lb}(\z)+b^{-1}\nabb b\ddb_{\lb}(\z) + 2\z\widehat{\otimes}\ddb_{\lb}(\z)\Bigr )\\
\nn&& -\trc 
\bigg (2\divb\ddb_{\lb}(\z)+4\z\c\ddb_{\lb}(\z)-4  \ddb_{\lb}(\z)\c n^{-1} \nabb n\bigg)\\
\nn&&+\divb\bigg(-2(\z-\zb)\mu+2(\zb-\z)\lb(\trc)+4\trc\bb\bigg)+f_2^1+f_2^2,
\eea
where $f_2^2$ is given by:
\bee
f_2^2&=&-(\d+n^{-1}\nab_Nn)L(\mu)-2(\divb(\z)-\divb(\zb))\mu-2(\divb(\zb)-\divb(\z))\lb(\trc)\\
&&+2(\zb-\z)\c(-\chb\nabb\trc+\xib L(\trc)+b^{-1}\nabb b\lb(\trc))\\
&&-2\hch\c(-\chb\nabb\z+\xib \ddb_L(\z)+(\chi\xib+\chb\z+\bb)\z)\\
&&-2\trc(-\chb\c\nabb\z+\xib\c\ddb_L(\z)+(\chi\xib+\chb\z+\bb)\z)\\
&&-4( \nabb\trc\c\bb+2\trc\xib\c\b+\trc(\kep-2\z)\c\bb)\\
&&-4\trc\hch\c\bigg(-n^{-1}\nabb^2n+n^{-2}\nabb n\nabb n-\nabb\kep-\chb\eta+\xib\kepb\\
&&+(\kep-\z)(n^{-1}\nabb n+\kep)+(\xib-\z)\kep\bigg).
\eee
The curvature bound \eqref{curvflux1} for $\b, \bb$ and the estimates \eqref{estn}-\eqref{estzeta} obtained in Theorem \ref{thregx} yield:
\bea
\nn\norm{f^2_2}_{\lh{1}}&\lesssim&\norm{\mu}_{\lh{2}}^2+\norm{L(\mu)}_{L^2_{x'}L^1_t}^2+
\norm{\ddb_{\lb}(\chi)}^2_{\lh{2}}+\no(\chi)^2+\no(\z)^2\\
&&\nn +\no(k)^2+\no(\nab n)^2+\norm{\b}_{\lh{2}}^2+\norm{\bb}_{\lh{2}}^2+\norm{\nab^2n}_{\lh{2}}^2\\
&&\nn +\norm{n-1}_{\lh{\infty}}\\
\lab{lbmu12}&\les&\norm{\mu}_{\lh{2}}^2+\norm{L(\mu)}_{L^2_{x'}L^1_t}^2+\ep.
\eea
Using the definition of $\mu$ \eqref{eqmu}, the formula for $L(\mu)$ given by \eqref{D4tmu}, the curvature bound \eqref{curvflux1} for $\r$ and the estimates \eqref{estn}-\eqref{estzeta} obtained in Theorem \ref{thregx}, we obtain:
$$\norm{\mu}_{\lh{2}}^2+\norm{L(\mu)}_{L^2_{x'}L^1_t}^2\les \ep$$
which together with \eqref{lbmu12} yields:
\be\lab{lbmu13}
\norm{f^2_2}_{\lh{1}}\les\ep.
\ee

Since $\mu_1=b\lb(\mu)$, we have:
$$L(\mu_1)=L(b)\lb(\mu)+bL(\lb(\mu))=-b\db \lb(\mu)+bL(\lb(\mu))$$
where we used the transport equation \eqref{D4a} satisfied by $b$. Together with \eqref{lbmu11}, this yields:
\bea
\nn L(\mu_1)  &=& -\trc\mu_1-2b\ddb_{\lb}(\z)\c\nabb\trc-2b\hch\c\Bigl (\nabb\widehat{\otimes}\ddb_{\lb}(\z)+b^{-1}\nabb b\ddb_{\lb}(\z) + 2\z\widehat{\otimes}\ddb_{\lb}(\z)\Bigr )\\
\nn&& -b\trc 
\bigg (2\divb\ddb_{\lb}(\z)+4\z\c\ddb_{\lb}(\z)-4  \ddb_{\lb}(\z)\c n^{-1} \nabb n\bigg)\\
\nn&&+\divb\bigg(b(-2(\z-\zb)\mu+2(\zb-\z)\lb(\trc)+4\trc\bb)\bigg)\\
\lab{lbmu14}&&-\nabb b\c(-2(\z-\zb)\mu+2(\zb-\z)\lb(\trc)+4\trc\bb)+bf_2^1+bf_2^2,
\eea
which is the desired transport equation \eqref{estlbmu1} for $\mu_1$ with $F_1$ given by:
$$F_1=b(-2(\z-\zb)\mu+2(\zb-\z)\lb(\trc)+4\trc\bb)$$
and $f_2$ given by:
$$f_2=-\nabb b\c(-2(\z-\zb)\mu+2(\zb-\z)\lb(\trc)+4\trc\bb)+bf_2^1+bf_2^2.$$
Using the curvature bound \eqref{curvflux1} for $\bb$ and the estimates \eqref{estn}-\eqref{estzeta} obtained in Theorem \ref{thregx}, we obtain:
\bee
\norm{F_1}_{\lh{2}}&\les& \norm{b}_{L^\infty}(\norm{\z}_{\xt{\infty}{2}}^2+\norm{\zb}_{\xt{\infty}{2}}^2
+\norm{\lb(\trc)}_{\xt{2}{\infty}}+\norm{\mu}_{\xt{2}{\infty}}\\
&&+\norm{\trc}_{L^\infty}\norm{\bb}_{\lh{2}})\\
&\les&\ep,
\eee
and:
$$\norm{f_2}_{\lh{1}}\les \norm{b^{-1}\nabb b}_{\lh{2}}\norm{F_1}_{\lh{2}}+\norm{b}_{L^\infty}(\norm{f^1_2}_{\lh{1}}+\norm{f^2_2}_{\lh{1}})\les\ep,$$
where we used the estimate \eqref{lbmu3} for $f^1_2$ and the estimate \eqref{lbmu13} for $f^2_2$. This concludes the proof of the estimate \eqref{estlbmu1bis} for $F_1$ and $f_2$.

We now turn to the Hodge system satisfied by $\ddb_{\lb}(\z)$. We differentiate the equation \eqref{diveta} giving $\divb(\z)$ with respect to $\lb$:
\be\lab{lbmu15}
\lb(\divb(\z)) = \half\bigg(\lb(\mu) -2\hch\c\ddb_{\lb}(\etah) -4\z\ddb_{\lb}(\z)\bigg)- \lb(\rho)+h_1^1,
\ee
where $h^1_1$ is given by:
$$h^1_1=\frac{1}{4}\lb(\trc)\trchb+\frac{1}{4}\trc\lb(\trchb) -\hch\c\ddb_{\lb}(\hch)-\ddb_{\lb}(\hch)\c\etah.$$
The estimates \eqref{estn}-\eqref{estzeta} obtained in Theorem \ref{thregx} yield:
\be\lab{lbmu16}
\norm{h^1_1}_{\lh{1}}\les \norm{\ddb_{\lb}(\chi)}_{\lh{2}}^2+\no(\chi)^2+\no(k)^2\les \ep.
\ee
\eqref{lbmu7}, \eqref{lbmu9}, \eqref{lbmu15} and the Bianchi identity \eqref{bianc3} yield:
\be\lab{lbmu17}
\divb(\ddb_{\lb}(\z)) = -b^{-1}\nabb b\c\ddb_{\lb}(\z)+\half\bigg(b^{-1}\mu_1 -4\z\ddb_{\lb}(\z)\bigg)+\divb\bb+h_1^1+h_1^2,
\ee
where $h_1^2$ is given by:
\bee
h_1^2&=&-\chb\c\nabb\z+\xib\c\ddb_L(\z)+(\chi\xib+\chb\z+\bb)\z-2\xib\c\b-(\kep-2\z)\c\bb\\
&&-\hch\c\bigg(-n^{-1}\nabb^2n+n^{-2}\nabb n\nabb n-\nabb\kep-\chb\eta+\xib\kepb+(\kep-\z)(n^{-1}\nabb n+\kep)+(\xib-\z)\kep\bigg).
\eee
The curvature bound \eqref{curvflux1} for $\bb$ and the estimates \eqref{estn}-\eqref{estzeta} obtained in Theorem \ref{thregx} yield:
\be\lab{lbmu18}
\norm{h^2_1}_{\lh{1}}\les \norm{\bb}_{\lh{2}}^2+\no(\z)^2+\no(\chi)^2+\no(k)^2+\no(\nabb n)\les \ep.
\ee

Next, we differentiate the equation \eqref{curleta} giving $\curlb(\z)$ with respect to $\lb$:
\be\lab{lbmu19}
\lb(\curlb(\z)) =  \ddb_{\lb}(\hch) \wedge  \etah+\hch\wedge\ddb_{\lb}(\etah)+\lb(\s).
\ee
The commutator formula \eqref{comm2}, \eqref{lbmu9}, \eqref{lbmu19} and the Bianchi identity \eqref{bianc5} yield:
\be\lab{lbmu20}
\curlb(\ddb_{\lb}(\z)) = -b^{-1}\nabb b\wedge\ddb_{\lb}(\z)-\curlb(\bb)+h_2,
\ee
where $h_2$ is given by:
\bee
h_2&=& \in_{AB}\chb_{AC}\nabb_C\z_B-\xib\wedge\ddb_L(\z)+(\chi\xib+\chb\z+\bb)\z
 +\ddb_{\lb}(\hch) \wedge  \etah-2\xib{}^*\b+(\kep-2\z){}^*\bb\\
&& +\hch\wedge\bigg(-n^{-1}\nabb^2n+n^{-2}\nabb n\nabb n-\nabb\kep-\chb\eta+\xib\kepb+(\kep-\z)(n^{-1}\nabb n+\kep)+(\xib-\z)\kep\bigg).
\eee
The curvature bound \eqref{curvflux1} for $\b, \bb$ and the estimates \eqref{estn}-\eqref{estzeta} obtained in Theorem \ref{thregx} yield:
\be\lab{lbmu21}
\norm{h_2}_{\lh{1}}\les \norm{\b}_{\lh{2}}^2+\norm{\bb}_{\lh{2}}^2+\no(\z)^2+\no(\chi)^2+\no(k)^2+\no(\nabb n)\les \ep.
\ee
Finally, \eqref{lbmu16}-\eqref{lbmu21} yield \eqref{estlbmu2} and \eqref{estlbmu2bis} which concludes the proof of the proposition.
\end{proof}

\subsection{Estimates for $\ddb_{\lb}(\z)$}\lab{sec:lbz}

The goal of this section is to obtain an improvement of the bootstrap assumption \eqref{boot8} for $\ddb_{\lb}(\z)$. We will use the following three lemmas.

\begin{lemma}\lab{lemma:lbz2}
Let $F$ a $\ptu$-tangent vectorfield on $\H_u$. Assume there exists two constants $C_1, C_2>0$ possibly depending on $u$ such that for all $j\geq 0$, we have: 
\be\lab{lbz3}
\norm{P_jF}_{\lh{2}}\leq C_1+2^{-\frac{j}{2}}C_2.
\ee
Let $H$ a $\ptu$-tangent vectorfield of the same type. Then, for all $j\geq 0$, we have:
\be\lab{lbz4}
\norm{P_j(H\c F)}_{\lh{2}}\lesssim \no(H)(2^jC_1+2^{\frac{j}{2}}C_2).
\ee
\end{lemma}

\begin{lemma}\lab{lemma:lbz1}
Let $f$ and $h$ two scalar functions on $\H_u$. Let $2\leq p\leq +\infty$. Assume there exists two constants $C_1, C_2>0$ possibly depending on $u$ such that for all $j\geq 0$, we have: 
\be\lab{lbz1}
\norm{P_jf}_{\tx{p}{2}}\leq 2^jC_1+2^{\frac{j}{2}}C_2.
\ee
Then, for all $j\geq 0$, we have:
\be\lab{lbz2}
\norm{P_j(hf)}_{\tx{p}{2}}\lesssim (\norm{h}_{L^\infty}+\norm{\nabb h}_{\BB^0})(2^jC_1+2^{\frac{j}{2}}C_2).
\ee
\end{lemma}

\begin{lemma}\lab{lemma:lbz3}
Let $F$ a $\ptu$-tangent 1-form on $\H_u$. Assume there exists two constants $C_1, C_2>0$ such that for all $j\geq 0$, we have: 
\be\lab{lbz5}
\norm{P_j\mathcal{D}_1(F)}_{\lh{2}}\leq 2^jC_1+2^{\frac{j}{2}}C_2.
\ee
Then, for all $j\geq 0$, we have:
\be\lab{lbz6}
\norm{P_jF}_{\lh{2}}\lesssim C_1+2^{-\frac{j}{2}}C_2.
\ee
\end{lemma}

We also state the following lemmas which will be used in the proof of Lemma \ref{lemma:lbz2} as well as several places in the paper.
\begin{lemma}\lab{lemma:lbz4}
For any $\ptu$-tangent tensor $F$ on $\H_u$, and for all $j\geq 0$, we have:
\be\lab{lbz14}
\sum_{j\geq 0}2^j\norm{P_jF}^2_{\tx{\infty}{2}}+2^{-j}\norm{\nabb P_jF}^2_{\tx{\infty}{2}}\les \no(F)^2.
\ee
\end{lemma}

\begin{lemma}\lab{lemma:lbz5}
For any 1-form $F$ on $\ptu$, for any $1<p\leq 2$ and for all $j\geq 0$, we have:
\be\lab{lbz14bis}
\norm{P_j\divb(F)}_{\lpt{2}}\les 2^{\frac{2}{p}j}\norm{F}_{\lpt{p}}.
\ee
\end{lemma}

We postpone the proof of Lemma \ref{lemma:lbz2} to section \ref{sec:lbz1}, the proof of Lemma \ref{lemma:lbz1} to section \ref{sec:lbz0}, the proof of Lemma \ref{lemma:lbz3} to section \ref{sec:lbz2}, the proof of Lemma \ref{lemma:lbz4} to sections \ref{sec:lbz3}, and the proof of Lemma \ref{lemma:lbz5} to section \ref{sec:lbz4}. We show how they improve the bootstrap assumption \eqref{boot8}. The bootstrap assumption \eqref{boot7} together with the definition of $\mu_1$ and $\mu$ yields for all $j\geq 0$:
\be\lab{lbz7}
\norm{P_j(b^{-1}\mu_1)}_{\lh{2}}\lesssim 2^jD\ep+2^{\frac{j}{2}}D\ep\gamma(u).
\ee
Lemma \ref{lemma:lbz2} implies:
\bea
\nn&&\normm{P_j\left(\frac{b^{-1}}{2}\mu_1-b^{-1}\nabb b\c\ddb_{\lb}(\z)-2\z\c\ddb_{\lb}(\z)\right)}_{\lh{2}}+\norm{P_j(-b^{-1}\nabb b\wedge \ddb_{\lb}(\z))}_{\lh{2}}\\
\nn&\lesssim& 2^jD\ep+2^{\frac{j}{2}}D\gamma(u)+(\no(b^{-1}\nabb b)+\no(\z))(2^jD^2\ep+2^{\frac{j}{2}}D^2\ep\gamma(u))\\
\nn&\lesssim& (1+D\ep)(2^jD\ep+2^{\frac{j}{2}}D\ep\gamma(u))\\
\lab{lbz8} &\lesssim& 2^jD\ep+2^{\frac{j}{2}}D\ep\gamma(u),
\eea
where we used the bootstrap assumptions \eqref{boot8} for $\ddb_{\lb}(\z)$, the estimate \eqref{lbz7} for $\mu_1$ and the estimates \eqref{estb} and \eqref{estzeta} for $b$ and $\z$ obtained in Theorem \ref{thregx}. Using the Littlewood-Paley property iii) of Theorem \ref{thm:LP}, and the dual of the sharp Bernstein inequality \eqref{eq:strongbernscalarbis} for scalars, we obtain:
\bea
\lab{lbz9}
&&\norm{P_j(\divb(\bb)+h_1)}_{\lh{2}}+\norm{P_j(-\curlb(\bb)+h_2)}_{\lh{2}}\\
\nn&\lesssim& 2^j\norm{\bb}_{\lh{2}}+2^j\norm{h_1}_{\lh{1}}+2^j\norm{h_2}_{\lh{1}}\\
\nn&\les&2^j\ep,
\eea
where we used the curvature bound \eqref{curvflux1} for $\bb$ and the estimate \eqref{estlbmu2bis} for $h_1, h_2$. Using the Hodge system \eqref{estlbmu2} satisfied by $\ddb_{\lb}(\z)$ and the estimates  \eqref{lbz8} and \eqref{lbz9}, we obtain:
\bee
\norm{P_j\mathcal{D}_1(\ddb_{\lb}(\z))}_{\lh{2}}&\les &\norm{P_j(\divb(\ddb_{\lb}(z)))}_{\lh{2}}+\norm{P_j(\curlb(\ddb_{\lb}(\z)))}_{\lh{2}}\\
\nn&\lesssim& 2^jD\ep+2^{\frac{j}{2}}D\ep\gamma(u).
\eee
which together with Lemma \ref{lemma:lbz3} yields:
\be\lab{lbz10}
\norm{P_j(\ddb_{\lb}(\z))}_{\lh{2}}\les  D\ep+2^{-\frac{j}{2}}D\ep\gamma(u).
\ee
Note that \eqref{lbz10} is an improvement of the bootstrap assumption \eqref{boot8} for $\ddb_{\lb}(\z)$.

\subsection{Estimates for $\lb\lb\trc$}\lab{sec:lbt}

The goal of this section is to obtain an improvement of the bootstrap assumption \eqref{boot7} for $\lb\lb\trc$. Note first that the bootstrap assumption \eqref{boot7} together with Lemma \ref{lemma:lbz1} with the choice $h=b$ and the definition of $\mu_1$ and $\mu$ yields for all $j\geq 0$:
\be\lab{lbt1}
\norm{P_j(\mu_1)}_{\lh{2}}\lesssim 2^jD\ep+2^{\frac{j}{2}}D\ep\gamma(u).
\ee
Another application of Lemma \ref{lemma:lbz1} this time with the choice $h=b^{-1}$ shows that improving on the bootstrap assumption \eqref{boot7} is equivalent to improving \eqref{lbt1}. We now focus on improving \eqref{lbt1}. After multiplying the transport  equation \eqref{estlbmu1} satisfied by $\mu_1$ by $n$, we have:
\bee
nL(\mu_1)+n\trc\mu_1&=&-2bn\ddb_{\lb}(\z)\c\nabb\trc
-2bn\hch\c\bigg(\nabb\widehat{\otimes}\ddb_{\lb}(\z)+b^{-1}\nabb b\ddb_{\lb}(\z)+2\ddb_{\lb}\z\widehat{\otimes}\z\bigg)\\
&&+2n\trc bn^{-1}\nabb n\c\ddb_{\lb}(\z)+n\divb(F_1)+nf_2.
\eee
which yields:
\bea
&&\lab{lbt2}\norm{P_j(\mu_1)}_{\lh{2}}\\
\nn&\les& 2^{\frac{j}{2}}\gamma(u)\ep+\normm{P_j\left(\int_0^t(n\trc\mu_1)d\tau\right)}_{\lh{2}}+ \normm{P_j\left(\int_0^t(bn\ddb_{\lb}(\z)\c\nabb\trc)d\tau\right)}_{\lh{2}}\\
\nn&&+\normm{P_j\left(\int_0^t(bn\hch\c(\nabb\widehat{\otimes}\ddb_{\lb}(\z))d\tau\right)}_{\lh{2}}+\normm{P_j\left(\int_0^t(bn\hch\c(b^{-1}\nabb b\ddb_{\lb}(\z))d\tau\right)}_{\lh{2}}\\
\nn&&+\normm{P_j\left(\int_0^t(bn\hch\c(\ddb_{\lb}\z\widehat{\otimes}\z\bigg)d\tau\right)}_{\lh{2}}+ \normm{P_j\left(\int_0^t(n\trc bn^{-1}\nabb n\c\ddb_{\lb}(\z))d\tau\right)}_{\lh{2}}\\
\nn&&+ \normm{P_j\left(\int_0^t(\divb(nF_1))d\tau\right)}_{\lh{2}}+ \normm{P_j\left(\int_0^t(\nabb nF_1))d\tau\right)}_{\lh{2}}\\
\nn&&+ \normm{P_j\left(\int_0^t(nf_2)d\tau\right)}_{\lh{2}}
\eea
where we used the following lemma with $f=\mu_1$:

\begin{lemma}\lab{lemma:lbtinit}
Let $f$ a scalar function solution of the following transport equation:
$$L(f)=0,\,f=f_0\textrm{ on }\pou.$$
Assume there is a constant $C>0$ possibly depending on $u$ such that for all $j\geq 0$:
$$\norm{P_jf_0}_{L^2(\pou)}\les C2^{\frac{j}{2}}.$$
Then, we have the following estimate for $f$:
$$\norm{P_jf}_{\tx{\infty}{2}}\les C2^{\frac{j}{2}}.$$
\end{lemma}

The proof of Lemma \ref{lemma:lbtinit} is postponed to section \ref{sec:lbtinit}. In order to estimate the right-hand side of \eqref{lbt1}, we will use the following three lemmas, which constitute the core of section \ref{sec:secondderlb}.

\begin{lemma}\lab{lemma:lbt1}
Let a scalar function $f$ on $\H_u$ such that:
$$\norm{f}_{L^\infty}+\norm{\nabb f}_{\PP^0}\les\ep.$$
Assume that $\mu_1$ satisfies \eqref{lbt1}. Then, we have for all $j\geq 0$:
$$\normm{P_j\left(\int_0^t(f\mu_1)d\tau\right)}_{\lh{2}}\les 2^jD\ep^2+2^{\frac{j}{2}}D\ep^2\gamma(u).$$
\end{lemma}

\begin{lemma}\lab{lemma:lbt2}
Let a $\ptu$-tangent 2-tensor $F$ on $\H_u$ such that $\nabb F$ admits a decomposition of the form:
$$\nabb F=\ddb_{nL}P+E$$
where $P$, $E$ are $\ptu$-tangent tensors, and $F$, $P$, $E$ satisfy:
$$\no(F)+\norm{F}_{\xt{\infty}{2}}+\no(P)+\norm{E}_{\PP^0}\les\ep.$$
Assume that $\ddb_{\lb}(\z)$ satisfies the estimate \eqref{lbz10}. Then, we have for all $j\geq 0$:
$$\normm{P_j\left(\int_0^t(F\c\nabb\ddb_{\lb}(\z))d\tau\right)}_{\lh{2}}\les 2^j\ep+2^jD\ep^2+2^{\frac{j}{2}}D\ep^2\gamma(u).$$
\end{lemma}

\begin{lemma}\lab{lemma:lbt3}
Let a $\ptu$-tangent 1-form $F$ on $\H_u$ such that:
$$\norm{F}_{\PP^0}\les\ep.$$
Assume that $\ddb_{\lb}(\z)$ satisfies the estimate \eqref{lbz10}. Then, we have for all $j\geq 0$:
$$\normm{P_j\left(\int_0^t(F\c\ddb_{\lb}(\z))d\tau\right)}_{\tx{\infty}{2}}\les 2^jD\ep^2+2^{\frac{j}{2}}D\ep^2\gamma(u).$$
\end{lemma}

We will also need the following three lemmas for the proof of  Lemma \ref{lemma:lbtinit}, Lemma \ref{lemma:lbt1},  \ref{lemma:lbt2} and \ref{lemma:lbt3}, as well as various places in this paper.

\begin{lemma}\lab{lemma:lbt4}
Let $f$ a scalar function on $\H_u$ and $F$ a $\ptu$-tangent  2-tensor. For any $j\geq 0$, we have:
$$\normm{P_j\left(\int_0^t\lap(f)d\tau\right)}_{\lh{2}}\les 2^{2j}\norm{f}_{\xt{2}{1}},$$
and 
$$\normm{P_j\left(\int_0^t\divb\divb(F)d\tau\right)}_{\lh{2}}\les 2^{2j}\norm{F}_{\xt{2}{1}}.$$
\end{lemma}

\begin{lemma}\lab{lemma:lbt5}
Let $F$ a $\ptu$-tangent 1-form. For any $j\geq 0$ and any $1<p\leq 2$, we have:
$$\normm{P_j\left(\int_0^t\divb(F)d\tau\right)}_{\tx{\infty}{2}}\les 2^{\frac{2j}{p}}\norm{F}_{\xt{p}{1}}.$$
\end{lemma}

\begin{lemma}\lab{lemma:lbt7}
The following decomposition holds:
$$\nabb(n\r)+(\nabb(n\s))^*={}^*\mathcal{D}_1\c J\c {}^*\mathcal{D}_1^{-1}(\ddb_{nL}(\bb))+{}^*\mathcal{D}_1(H),$$
where $J$ denotes the involution $(\r,\s)\rightarrow (-\r,\s)$ and $H$ is a scalar function on $\H_u$ satisfying the following estimate:
$$\norm{H}_{\tx{2}{3}}\les\ep.$$
\end{lemma}

We postpone the proof of Lemma \ref{lemma:lbt1} to section \ref{sec:lbt1}, the proof of Lemma \ref{lemma:lbt2} to section \ref{sec:lbt2}, the proof of Lemma \ref{lemma:lbt3} to section \ref{sec:lbt3}, the proof of Lemma \ref{lemma:lbt4} to section \ref{sec:lbt4}, the proof of Lemma \ref{lemma:lbt5} to section \ref{sec:lbt5}, and the proof of Lemma \ref{lemma:lbt7} to section \ref{sec:lbt7}. We show how they improve the estimate \eqref{lbt1}. We estimate each term in the right-hand side of \eqref{lbt2} starting with the first one. 

The scalar function $f=n\trc$ satisfies the following estimate:
\bee
\norm{f}_{L^\infty}+\norm{\nabb f}_{\PP^0}&\les &\norm{n}_{L^\infty}\norm{\trc}_{L^\infty}+\norm{n\nabb\trc}_{\PP^0}+\norm{\trc\nabb n}_{\PP^0}\\
&\les &\ep+(1+\noo(n-1))\norm{\nabb\trc}_{\PP^0}+\no(\trc)\no(n-1)\\
&\les& \ep,
\eee
where we used the estimate \eqref{esttrc} satisfied by $\trc$, the estimate \eqref{estn} satisfied by $n$, and the non sharp product estimates \eqref{nonsharpprod1} and \eqref{nonsharpprod2}. Thus, in view of Lemma \ref{lemma:lbt1}, we obtain:
\be\lab{lbt3}
\normm{P_j\left(\int_0^t(n\trc\mu_1)d\tau\right)}_{\lh{2}}\les 2^jD\ep^2+2^{\frac{j}{2}}D\ep^2\gamma(u).
\ee

We now focus on the third term in the right-hand side of \eqref{lbt2}. We define the 2-tensor $F=bn\hch$. In view of the decomposition \eqref{impbes22} for $\nabb(n\hch)$, we have: 
\bee
\nabb F&=&b\nabb(n\hch)+n\hch\nabb b\\
&=& b(\ddb_{nL}P+E)+n\hch\nabb b\\
&=& \ddb_{nL}(bP)-nL(b)P+bE+n\hch\nabb b\\
&=& \ddb_{nL}(bP)+nb\db P+bE+n\hch\nabb b
\eee
where $P$ and $E$ satisfy:
$$\no(P)+\norm{E}_{\PP^0}\les \ep.$$ 
Thus, we set $P_1=bP$ and $E_1=nb\db P+bE+n\hch\nabb b$ 
and obtain:
$$\nabb F = \ddb_{nL}(P_1)+E_1.$$
Furthermore, we have:
\bee
\no(P_1)+\norm{E_1}_{\PP^0}&\les& (\norm{b}_{L^\infty}+\noo(b))\no(P)
+\noo(b)\noo(n)\no(\db)\no(P)+\noo(b)\norm{E}_{\PP^0}\\
&&+\noo(n)\no(\hch)\no(\nabb b)\\
&\les & \no(P)+\norm{E}_{\PP^0}+\ep\\
&\les &\ep,
\eee
where we used the estimate \eqref{esthch} satisfied by $\hch$, the estimate \eqref{estn} satisfied by $n$, the estimate \eqref{estb} satisfied by $b$, the estimate \eqref{estk} satisfied by $\db$, and the non sharp product estimates \eqref{nonsharpprod1} and \eqref{nonsharpprod2}. Thus, in view of Lemma \ref{lemma:lbt2}, we obtain:
\be\lab{lbt4}
\normm{P_j\left(\int_0^t(bn\hch\c(\nabb\widehat{\otimes}\ddb_{\lb}(\z))d\tau\right)}_{\lh{2}}\les 2^jD\ep^2+2^{\frac{j}{2}}D\ep^2\gamma(u).
\ee

We consider the second, the fourth, the fifth and the sixth term in the right-hand side of \eqref{lbt2}. We define the 1-forms: 
$$F_1=bn\nabb\trc, F_2=n\hch\nabb b, F_3=bn\hch\z\textrm{ and }F_4=b\trc\nabb n.$$
These 1-forms satisfy the following estimate:
\bee
&&\norm{F_1}_{\PP^0}+\norm{F_2}_{\PP^0}+\norm{F_3}_{\PP^0}+\norm{F_4}_{\PP^0}\\
&\les& \noo(n)\noo(b)\norm{\nabb\trc}_{\PP^0}+\noo(n)\no(\hch)\no(\nabb b)+\noo(b)\noo(n)\no(\hch)\no(\z)\\
&&+\noo(b)\no(\trc)\no(\nabb n)\\
&\les & \ep,
\eee
where we used the estimate \eqref{esttrc} satisfied by $\trc$,  the estimate \eqref{esthch} satisfied by $\hch$, the estimate \eqref{estn} satisfied by $n$, the estimate \eqref{estb} satisfied by $b$, the estimate \eqref{estzeta} satisfied by $\z$, and the non sharp product estimates \eqref{nonsharpprod1} and \eqref{nonsharpprod2}. Thus, in view of Lemma \ref{lemma:lbt3}, we obtain:
\bea
\nn&&\normm{P_j\left(\int_0^t(bn\ddb_{\lb}(\z)\c\nabb\trc)d\tau\right)}_{\lh{2}}+\normm{P_j\left(\int_0^t(bn\hch\c(b^{-1}\nabb b\ddb_{\lb}(\z))d\tau\right)}_{\lh{2}}\\
\nn&&+\normm{P_j\left(\int_0^t(bn\hch\c(\ddb_{\lb}\z\widehat{\otimes}\z\bigg)d\tau\right)}_{\lh{2}}+ \normm{P_j\left(\int_0^t(n\trc bn^{-1}\nabb n\c\ddb_{\lb}(\z))d\tau\right)}_{\lh{2}}\\
\lab{lbt5} &\les & 2^jD\ep^2+2^{\frac{j}{2}}D\ep^2\gamma(u).
\eea

We consider the seventh term in the right-hand side of \eqref{lbt2}. We define the scalar function $w$ and the the $\ptu$-tangent 1-form $W$ as the solutions of the following transport equations: 
$$nL(w)=\divb(nF_1),\,w=0\textrm{ on }\pou,\textrm{ and }\ddb_{nL}W-n\chi\c W=nF_1,\,W=0\textrm{ on }\pou.$$
We have:
\bea
\lab{lbt6}\normm{P_j\left(\int_0^t(\divb(nF_1))d\tau\right)}_{\lh{2}}&= &\norm{P_jw}_{\lh{2}}\\
\nn&\les &\norm{P_j(w-\divb(W))}_{\lh{2}}+\norm{P_j\divb(W)}_{\lh{2}}\\
\nn&\les & 2^j\norm{w-\divb(W)}_{\tx{2}{1}}+2^j\norm{W}_{\lh{2}}
\eea
where we used the dual of the sharp Bernstein inequality \eqref{eq:strongbernscalarbis} 
and the finite band property of the Littlewood-Paley projection $P_j$. We estimate the two 
terms in the right-hand side of \eqref{lbt6}. Using Lemma \ref{lemma:commutdivb}, we have:
\be\lab{lbt7}
\norm{w-\divb(W)}_{\tx{2}{1}}\les \ep\norm{nF_1}_{\lh{2}}\les \ep^2
\ee
where we used the estimate \eqref{estlbmu1bis} on $F_1$ and the $L^\infty$ bound for $n$ given 
by \eqref{estn}. Also, using the estimate \eqref{estimtransport1} for transport equations, we have:
\bea
\lab{lbt8}\norm{W}_{\lh{2}}&\les& \norm{\chi W}_{\xt{2}{1}}+\norm{nF_1}_{\lh{2}}\\
\nn&\les & \norm{\chi}_{\xt{\infty}{2}}\norm{W}_{\lh{2}}+\ep\\
\nn&\les & \ep\norm{W}_{\lh{2}}+\ep,
\eea
where we used the estimate \eqref{estlbmu1bis} on $F_1$, the trace bound on $\chi$ given by \eqref{esttrc} \eqref{esthch}, and the $L^\infty$ bound for $n$ given 
by \eqref{estn}. \eqref{lbt8} yields:
$$\norm{W}_{\lh{2}}\les\ep$$
which together with \eqref{lbt6} and \eqref{lbt7} implies:
\be\lab{lbt9}
\normm{P_j\left(\int_0^t(\divb(nF_1))d\tau\right)}_{\lh{2}}\les 2^j\ep.
\ee

Finally, we consider the last two terms in the right-hand side of \eqref{lbt2}. Using 
the dual of the sharp Bernstein inequality \eqref{eq:strongbernscalarbis} and 
the estimate \eqref{estimtransport1} for transport equations, we have:
\bea
\lab{lbt10}
&&\normm{P_j\left(\int_0^t(\nabb nF_1))d\tau\right)}_{\lh{2}}+ \normm{P_j\left(\int_0^t(nf_2)d\tau\right)}_{\lh{2}}\\
\nn&\les & 2^j\normm{\int_0^t(\nabb nF_1))d\tau}_{\tx{2}{1}}+2^j\normm{\int_0^t(nf_2)d\tau}_{\tx{2}{1}}\\
\nn&\les & 2^j\norm{\nabb nF_1}_{\lh{1}}+2^j\norm{nf_2}_{\lh{1}}\\
\nn&\les & 2^j\norm{\nabb n}_{\lh{2}}\norm{F_1}_{\lh{2}}+2^j\norm{n}_{L^\infty}\norm{f_2}_{\lh{1}}\\
\nn&\les &2^j\ep,
\eea
where we used the estimate \eqref{estlbmu1bis} on $F_1$ and $f_2$, and the $L^\infty$ bound for $n$ given by \eqref{estn}.

Finally, \eqref{lbt2}-\eqref{lbt6}, \eqref{lbt9} and \eqref{lbt10} yield:
\be\lab{lbt11}
\norm{P_j(\mu_1)}_{\lh{2}}\les D\ep^22^j+D\ep^2 2^{\frac{j}{2}}\gamma(u)+2^j\ep
\ee
which is an improvement of \eqref{lbt1}. \eqref{lbt11} together with Lemma \ref{lemma:lbz1} with the choice $h=b^{-1}$ and the definition of $\mu_1$ yields for all $j\geq 0$:
$$\norm{P_j(\lb\mu)}_{\lh{2}}\les 2^jD\ep^2+2^{\frac{j}{2}}D\ep^2\gamma(u)+2^j\ep$$
which in view of the definition of $\mu$ implies for all $j\geq 0$:
\be\lab{lbt12}
\norm{P_j(\lb\lb\trc)}_{\lh{2}}\les 2^jD\ep^2+2^{\frac{j}{2}}D\ep^2\gamma(u)+2^j\ep.
\ee
\eqref{lbz10} and \eqref{lbt12} improve the bootstrap assumptions \eqref{boot7} \eqref{boot8}. Thus, there exists a universal constant $D>0$ such that \eqref{boot7} \eqref{boot8} hold. This yields \eqref{estlblbtrc} \eqref{estlbzeta}.

\subsection{Estimates for $\lb\lb b$}\lab{sec:fautpasdeconner1}

The goal of this section is to prove the estimate \eqref{estlblbb} for $\lb\lb b$  and to conclude the proof of Theorem \ref{thregx1}.

\subsubsection{Structure equation for $\lb\lb b$}

The goal of this section is to prove the following proposition.
\begin{proposition}\lab{prop:lb1}
Let $b_1=b\lb\lb b -b^2(\lb(\d)+\lb(n^{-1}\nabla_Nn))$. Then, $b_1$ satisfies the following transport equation:
\be\lab{lb1}
L(b_1)=-(2b\nabb b+4b^2\kep)\c\ddb_{\lb}(\z)+b^2\hch\ab+\divb(F_1)+f_2,
\ee
where the $\ptu$-tangent vectorfield $F_1$ and the scalar function $f_2$ satisfy the estimates:
\be\lab{lb2}
\norm{F_1}_{\lh{2}}+\norm{f_2}_{\lh{1}}\les\ep.
\ee
\end{proposition}

\begin{proof}
We differentiate the transport equation \eqref{josh} satisfied by $\lb b$ with respect to $\lb$. We obtain:
\bea
\lab{lb3}L(\lb\lb(b))&=&[L,\lb](\lb(b))+\lb(L(\lb(b)))\\
\nn&=& \db\lb\lb(b)-(\d+n^{-1}\nabla_Nn)L(\lb(b))-2(\z-\zb)\nabb \lb(b)
+(\lb(\d)+\lb(n^{-1}\nabla_Nn))\db b\\
\nn&&+(\d+n^{-1}\nabla_Nn)\lb(\db) b+(\d+n^{-1}\nabla_Nn)\db \lb(b)
-2(\ddb_{\lb}(\z)-\ddb_{\lb}(\zb))\c\nabb b\\
\nn&&-2(\z-\zb)\c\ddb_{\lb}\nabb b-\lb(b)\lb(\db)-b\lb\lb(\db),
\eea
where we used in the last equality the commutator formula \eqref{comm3}.

In view of \eqref{lb3}, we need to compute $\lb\lb(\db)$. Differentiating the formula \eqref{D4tmu3} for $\lb(\db)$ with respect to $\lb$, we obtain:
\bea
\lab{lb4}\lb\lb(\db)&=&-L(\lb(\d)+\lb(n^{-1}\nab_Nn))-[\lb,L](\d+n^{-1}\nab_Nn)+2\lb(\rho)\\
\nn&&+4\kep\c\ddb_{\lb}(\kep)+4\d\lb(\d)+4\ddb_{\lb}(\kep)\c(\z-n^{-1}\nabb n)+4\kep\c(\ddb_{\lb}(\z)-\ddb_{\lb}(n^{-1}\nabb n))\\
\nn&&-4n^{-1}N(n)\lb(n^{-1}N(n))\\
\nn&=&-L(\lb(\d)+\lb(n^{-1}\nab_Nn))-[\lb,L](\d+n^{-1}\nab_Nn)-2\divb(\bb)-\hch\c\ab\\
\nn&&+4\xib\c\b+2(\kep-2\z)\c\bb+4\kep\c\ddb_{\lb}(\kep)+4\d\lb(\d)+4\ddb_{\lb}(\kep)\c(\z-n^{-1}\nabb n)\\
\nn&&+4\kep\c(\ddb_{\lb}(\z)-\ddb_{\lb}(n^{-1}\nabb n))-4n^{-1}N(n)\lb(n^{-1}N(n)),
\eea
where we used the Bianchi identity \eqref{bianc3} for $\lb(\r)$ in the last equality. 
Now, \eqref{lb3}, \eqref{lb4}, the transport equation \eqref{D4a} satisfied by $b$, and the definition of $b_1$ yield:
\bea
\nn L(b_1)&=& bL(\lb\lb b)+L(b)\lb\lb b-b^2L(\lb(\d)+\lb(n^{-1}\nabla_Nn))-2bL(b)(\lb(\d)+\lb(n^{-1}\nabla_Nn))\\
\nn&=& bL(\lb\lb b)-b\db\lb\lb b-bL(\lb(\d)+\lb(n^{-1}\nabla_Nn))+2b\db(\lb(\d)+\lb(n^{-1}\nabla_Nn))\\
\lab{lb5}&=& -(2b\nabb b+4b^2\kep)\c\ddb_{\lb}(\z)+b^2\hch\ab+\divb(F_1)+f_2,
\eea
where the $\ptu$-tangent vectorfield $F_1$ is given by:
\be\lab{lb6}
F_1=-4b(\z-\zb)\lb(b)+2b^2\bb,
\ee
and the scalar function $f_2$ is given by:
\bea
\lab{lb7}f_2&=& -b(\d+n^{-1}\nabla_Nn)L(\lb(b))+4b(\nabb\z-\nabb\zb)\lb(b)+4\nabb b\c(\z-\zb)\lb(b)\\
\nn&&+b(\lb(\d)+\lb(n^{-1}\nabla_Nn))\db b+b(\d+n^{-1}\nabla_Nn)\lb(\db) b+b(\d+n^{-1}\nabla_Nn)\db \lb(b)\\
\nn&&+2b\ddb_{\lb}(\zb)\c\nabb b-2b(\z-\zb)\c [\ddb_{\lb},\nabb](b)-b\lb(b)\lb(\db)
-2bL(b)(\lb(\d)+\lb(n^{-1}\nab_Nn))\\
\nn&&+b^2[\lb,L](\d+n^{-1}\nab_Nn)-4b\nabb(b)\c\bb-4b^2\xib\c\b-2b^2(\kep-2\z)\c\bb-4b^2\kep\c\ddb_{\lb}(\kep)\\
\nn&&-4b^2\d\lb(\d)-4b^2\ddb_{\lb}(\kep)\c(\z-n^{-1}\nabb n)+4b^2\kep\c\ddb_{\lb}(n^{-1}\nabb n)\\
\nn&&+4b^2n^{-1}N(n)\lb(n^{-1}N(n))+2b\db(\lb(\d)+\lb(n^{-1}\nabla_Nn)).
\eea
In view of the definition \eqref{lb6} of $F_1$, we have:
\bea
\lab{lb8}\norm{F_1}_{\lh{2}}&\les& \norm{b}_{L^\infty}(\norm{\z}_{\xt{\infty}{2}}+\norm{\zb}_{\xt{\infty}{2}})\norm{\lb(b)}_{\xt{2}{\infty}}+\norm{b}^2_{L^\infty}\norm{\bb}_{\lh{2}}\\
\nn&\les&\ep,
\eea
where we used in the last inequality the curvature bound \eqref{curvflux1} for $\bb$, and the estimates \eqref{estn}-\eqref{estzeta} for $b, \zb$ and $\z$.

Next, we estimate $f_2$. In view of \eqref{lb7}, we have:
\bea
\lab{lb9}&&\norm{f_2}_{\lh{1}}\\
\nn&\les& \norm{b}_{L^\infty}\bigg(\norm{\d+n^{-1}\nabla_Nn}_{\lh{2}}\norm{L(\lb(b))}_{\lh{2}}+(\norm{\nabb\z-\nabb\zb}_{\lh{2}}\\
\nn&&+\norm{b^{-1}\nabb b}_{\lh{4}}\norm{\z-\zb}_{\lh{4}})\norm{\lb(b)}_{\lh{2}}+\norm{\ddb_{\lb}(\zb)}_{\lh{2}}\norm{\nabb b}_{\lh{2}}\\
\nn&&+\norm{\z-\zb}_{\lh{4}}\norm{[\ddb_{\lb},\nabb](b)}_{\tx{2}{\frac{4}{3}}}+\norm{\lb(b)}_{\lh{2}}\norm{\lb(\db)}_{\lh{2}}\\
\nn&&+\norm{L(b)}_{\lh{2}}\norm{\lb(\d)+\lb(n^{-1}\nab_Nn)}_{\lh{2}}\bigg)\\
\nn&&+\norm{b}^2_{L^\infty}\bigg(\norm{\lb(\d)+\lb(n^{-1}\nabla_Nn)}_{\lh{2}}\norm{\db}_{\lh{2}}+(\norm{\d}_{\lh{2}}\\
\nn&&+\norm{n^{-1}\nabla_Nn}_{\lh{2}})\norm{\lb(\db)}_{\lh{2}}+(\norm{\d}_{\lh{4}}+\norm{n^{-1}\nabla_Nn}_{\lh{4}})\norm{\db}_{\lh{4}}\\
\nn&&\times\norm{\lb(b)}_{\lh{2}}+\norm{[\lb,L](\d+n^{-1}\nab_Nn)}_{\lh{1}}+\norm{b^{-1}\nabb(b)}_{\lh{2}}\norm{\bb}_{\lh{2}}\\
\nn&&+\norm{\xib}_{\lh{2}}\norm{\b}_{\lh{2}}+\norm{\kep-2\z}_{\lh{2}}\norm{\bb}_{\lh{2}}+b\norm{\kep}_{\lh{2}}\norm{\ddb_{\lb}(\kep)}_{\lh{2}}\\
\nn&&+\norm{\d}_{\lh{2}}\norm{\lb(\d)}_{\lh{2}}+\norm{\ddb_{\lb}(\kep)}_{\lh{2}}\norm{\z-n^{-1}\nabb n}_{\lh{2}}\\
\nn&&+\norm{\kep}_{\lh{2}}\norm{\ddb_{\lb}(n^{-1}\nabb n)}_{\lh{2}}+\norm{n^{-1}N(n)}_{\lh{2}}\norm{\lb(n^{-1}N(n))}_{\lh{2}}\\
\nn&&+\norm{\db}_{\lh{2}}(\norm{\lb(\d)}_{\lh{2}}+\norm{\lb(n^{-1}\nabla_Nn)}_{\lh{2}}\bigg)\\
\nn&\les&\ep+\ep\norm{L(\lb(b))}_{\lh{2}}+\ep\norm{[\ddb_{\lb},\nabb](b)}_{\tx{2}{\frac{4}{3}}}+\norm{[\lb,L](\d+n^{-1}\nab_Nn)}_{\lh{1}},
\eea
where we used in the last inequality the curvature bound \eqref{curvflux1} for $\b$ and $\bb$, and the estimates \eqref{estn}-\eqref{estzeta} for $n, b, \kep, \d, \db, \zb, \xib$ and $\z$. Now, we evaluate the right-hand side of \eqref{lb9}. Using the estimate \eqref{josh1} for $\norm{L(\lb(b))}_{\lh{2}}$ and the commutator formulas \eqref{comm2} and \eqref{comm3}, we have:
\bee
&&\norm{L(\lb(b))}_{\lh{2}}+\norm{[\ddb_{\lb},\nabb](b)}_{\tx{2}{\frac{4}{3}}}+\norm{[\lb,L](\d+n^{-1}\nab_Nn)}_{\lh{1}}\\
&\les& \ep+\norm{\hchb}_{\tx{\infty}{4}}\norm{\nabb b}_{\lh{2}}+\norm{\xib}_{\tx{\infty}{4}}\norm{L(b)}_{\lh{2}}+\norm{b^{-1}\nabb b}_{\tx{\infty}{4}}\norm{\lb(b)}_{\lh{2}}\\
&&+\norm{\db}_{\lh{2}}\norm{\lb(\d+n^{-1}\nabla_Nn)}_{\lh{2}}+\norm{\d+n^{-1}\nabla_Nn}_{\lh{2}}\norm{L(\d+n^{-1}\nabla_Nn)}_{\lh{2}}\\
&&+\norm{\z-\zb}_{\lh{2}}\norm{\nabb(\d+n^{-1}\nabla_Nn)}_{\lh{2}}\\
&\les&\ep,
\eee
where we used in the last inequality the estimates \eqref{estn}-\eqref{estzeta} for $n, b, \d, \db, \zb, \xib, \hchb$ and $\z$. Together with \eqref{lb9}, this yields:
\be\lab{lb10}
\norm{f_2}_{\lh{1}}\les \ep.
\ee
In view of \eqref{lb5}, \eqref{lb8} and \eqref{lb10}, this concludes the proof of Proposition \ref{prop:lb1}.
\end{proof}

\subsubsection{Estimates for $\lb\lb(b)$}

After multiplying the transport equation \eqref{lb1} satisfied by $b_1$ by $n$, we have:
$$nL(b_1)=-(2nb\nabb b+4nb^2\kep)\c\ddb_{\lb}(\z)+nb^2\hch\ab+\divb(nF_1)-\nabb nF_1+f_2,$$
which together with Lemma \ref{lemma:lbtinit} yields:
\bea
\lab{lb11}\norm{P_j(b_1)}_{\tx{\infty}{2}}&\les& 2^{\frac{j}{2}}\gamma(u)\ep+\normm{P_j\left(\int_0^t((2nb\nabb b+4nb^2\kep)\c\ddb_{\lb}(\z))d\tau\right)}_{\tx{\infty}{2}}\\
\nn&&+\normm{P_j\left(\int_0^t(nb^2\hch\c\ab)d\tau\right)}_{\tx{\infty}{2}}+ \normm{P_j\left(\int_0^t(\divb(nF_1))d\tau\right)}_{\tx{\infty}{2}}\\
\nn&&+ \normm{P_j\left(\int_0^t(\nabb nF_1)d\tau\right)}_{\tx{\infty}{2}}+ \normm{P_j\left(\int_0^t(nf_2)d\tau\right)}_{\tx{\infty}{2}}.
\eea

Next, we evaluate the right-hand side of \eqref{lb11}. Using the nonsharp product estimates \eqref{nonsharpprod1} and \eqref{nonsharpprod2}, we have:
$$\norm{2nb\nabb b+4nb^2\kep}_{\PP^0}\les \noo(b)\no(n)(\no(\nabb b)+\noo(b)\no(\kep))\les \ep,$$
where we used in the last inequality the estimates \eqref{estn}-\eqref{estb} for $n, \kep$ and $b$. Together with Lemma \ref{lemma:lbt3}, this yields the following estimate for the second term in the right-hand side of \eqref{lb11}:
\be\lab{lb12}
\normm{P_j\left(\int_0^t((2nb\nabb b+4nb^2\kep)\c\ddb_{\lb}(\z))d\tau\right)}_{\tx{\infty}{2}}\les 2^j\ep+2^{\frac{j}{2}}\ep\gamma(u).
\ee
Using Lemma \ref{lemma:lbt5} with $p=2$, we have the following estimate for the second term in the right-hand side of \eqref{lb11}:
\be\lab{lb13}
\normm{P_j\left(\int_0^t(\divb(nF_1))d\tau\right)}_{\tx{\infty}{2}}\les 2^j\norm{nF_1}_{\lh{2}}\les 2^j\norm{n}_{L^\infty}\norm{F_1}_{\lh{2}}\les 2^j\ep,
\ee
where we used in the last inequality the estimate \eqref{estn} for $n$ and the estimate \eqref{lb2} for $F_1$. Also, using the dual of the sharp Bernstein inequality for scalars \eqref{eq:strongbernscalarbis} and the $L^2$ boundedness of $P_j$, and the estimate for transport equations \eqref{estimtransport1}, we have the following estimate for the remaining terms in the right-hand side of \eqref{lb11}:
\bea
\nn&&\normm{P_j\left(\int_0^t(nb^2\hch\c\ab)d\tau\right)}_{\tx{\infty}{2}}+ \normm{P_j\left(\int_0^t(\nabb nF_1)d\tau\right)}_{\tx{\infty}{2}}\\
\nn&&+ \normm{P_j\left(\int_0^t(nf_2)d\tau\right)}_{\tx{\infty}{2}}\\
\nn&\les&2^{\frac{j}{2}}\normm{nb^2\hch\c\ab}_{\xt{2}{1}}+2^j\normm{\nabb nF_1}_{\lh{1}}+ 2^j\normm{nf_2}_{\lh{1}}\\
\nn&\les& \norm{n}_{L^\infty}\norm{b}^2_{L^\infty}\norm{\hch}_{\xt{\infty}{2}}\norm{\ab}_{\lh{2}}+2^j\norm{\nabb n}_{\tx{\infty}{2}}\norm{F_1}_{\lh{2}}+ 2^j\norm{n}_{L^\infty}\norm{f_2}_{\lh{1}}\\
\lab{lb14}&\les& 2^j\ep+\ep\gamma(u),
\eea
where we used in the last inequality the curvature bound $\norm{\ab}_{\lh{2}}\les\gamma(u)\ep$ provided by \eqref{curvflux1}, the estimates \eqref{estn}-\eqref{esthch} for $n, b$ and $\hch$, and the estimate \eqref{lb2} for $F_1$ and $f_2$.

Finally, in view of \eqref{lb11}-\eqref{lb14}, we have:
\be\lab{lb15}
\norm{P_j(b_1)}_{\tx{\infty}{2}}\les 2^j\ep+2^{\frac{j}{2}}\ep\gamma(u).
\ee
Now, in view of the definition of $b_1$ in Proposition \ref{prop:lb1}, we have:
\bea\lab{gorgonzola}
\norm{P_j(b\lb\lb(b))}_{\tx{\infty}{2}}&\les& \norm{P_j(b_1)}_{\tx{\infty}{2}}+\norm{P_j(b^2(\lb(\d)+\lb(n^{-1}\nabla_Nn)))}_{\tx{\infty}{2}}\\
\nn&\les& \norm{P_j(b_1)}_{\tx{\infty}{2}}+\norm{P_j(b^2(\lb(\d))}_{\tx{\infty}{2}}+\norm{\lb(n^{-1}\nabla_Nn))}_{\tx{\infty}{2}}\\
\nn&\les& 2^j\ep+2^{\frac{j}{2}}\ep\gamma(u)+\norm{P_j(b^2(\lb(\d))}_{\tx{\infty}{2}},
\eea
where we used in the last inequality the estimate \eqref{estn} for $n$, and the estimate \eqref{lb15} for $b_1$. 
Now, we have in view of \eqref{hodgkh} and \eqref{kinTderivatives}:
\be\lab{gorgonzola1}
\lb(\d)= \rho+\divb\kep+h,
\ee
where the scalar $h$ is given by
$$h= -n^{-1}\nabla^2_Nn+\d^2-\z\zb+\z\kep-\zb\kep+\frac{3}{2}\d\trt-\het\hth+2b^{-1}\nabb_A b\kep_A.$$
In view of the definition of $h$, we have
\bea\lab{gorgonzola2}
\nn\norm{h}_{\tx{\infty}{2}}&\les& \norm{n^{-1}\nabla^2_Nn}_{\tx{\infty}{2}}+\Big(\norm{\d}_{\tx{\infty}{4}}+\norm{\z}_{\tx{\infty}{4}}+\norm{\zb}_{\tx{\infty}{4}}+\norm{\kep}_{\tx{\infty}{4}}\\
\nn&&+\norm{\th}_{\tx{\infty}{4}}+\norm{\het}_{\tx{\infty}{4}}+\norm{b^{-1}\nabb b}_{\tx{\infty}{4}}\Big)^2\\
&\les& \ep,
\eea
where we used in the last inequality the estimates \eqref{estn}-\eqref{estzeta} for $n, \d, \z, \zb, \kep, \th, \het$ and $b$. Also, using the finite band property for $P_j$ and the estimate \eqref{estk} for $\kep$, we have
\be\lab{gorgonzola3}
\norm{P_j(\divb\kep)}_{\tx{\infty}{2}}\les 2^j\norm{\kep}_{\tx{\infty}{2}}\les 2^j\ep.
\ee
We will obtain in Lemma \ref{lemma:po5} the following estimate for $\rho$
\be\lab{gorgonzola4}
\norm{P_j\rho}_{\tx{\infty}{2}}\les 2^{\frac{j}{2}}\ep.
\ee
Finally, \eqref{gorgonzola1}-\eqref{gorgonzola4} imply
\be\lab{gorgonzola5}
\norm{P_j(\lb(\d))}_{\tx{\infty}{2}}\les 2^j\ep.
\ee
Together with Lemma \ref{lemma:lbz1} with the choice $h=b^2$, this yields:
$$\norm{P_j(b^2\lb(\d))}_{\tx{\infty}{2}}\les 2^j\ep.$$
Together with \eqref{gorgonzola}, we obtain
$$\norm{P_j(b\lb\lb(b))}_{\tx{\infty}{2}}\les 2^j\ep+2^{\frac{j}{2}}\ep\gamma(u).$$
Together with Lemma \ref{lemma:lbz1} with the choice $h=b^{-1}$, this yields:
$$\norm{P_j(\lb\lb(b))}_{\tx{\infty}{2}}\les 2^j\ep+2^{\frac{j}{2}}\ep\gamma(u),$$
which implies the estimate \eqref{estlblbb} for $\lb\lb(b)$. Together with the estimates \eqref{estlblbtrc} and \eqref{estlbzeta} which were obtained in section \ref{sec:lbz} and section \ref{sec:lbt}, this concludes the proof of Theorem \ref{thregx1}.

\section{First order derivatives with respect to $\o$}\lab{sec:firstderivomega}

The goal of this section is to prove Theorem \ref{thregomega}. In section \ref{sec:commutomega}, we derive commutator formulas involving $\po$. In section \ref{sec:jycroispas}, we prove the estimates \eqref{estNomega} and \eqref{estricciomega} for $\po N$, $\po b$, $\po\chi$ and $\po\z$. In section \ref{sec:jycroispas1}, we prove the estimate \eqref{estricciomegabis} for $\ddb_{\lb}\Pi(\po\chi)$. In section \ref{sec:jycroispas2}, we derive the decomposition \ref{dechch}-\eqref{dechch3} for $\hch$. In section \ref{sec:jycroispas3}, we derive Besov improvements for $\po N$ and $\po\chi$. Finally, we prove the lower bound \eqref{ad1} for $N(.,\o)-N(.,\o')$ in section \ref{sec:jycroispas4}.

\subsection{Commutator formulas}\lab{sec:commutomega}

In this section, we derive several formulas involving commutators with $\po$. We start with some useful identities.
\begin{lemma}
For any 1-form $F$, we have the following identity:
\be\lab{coo2ter}
F_{\po e_A}e_A+F_A\po e_A=-F_N\po N-F_{\po N}N.
\ee
For any symmetric 2-tensor $H$, we have the following identity:
\be\lab{coo2quatre}
H_{A\po e_C}H_{CB}+H_{AC}H_{\po e_C B}=-H_{AN}H_{\po NC}-H_{A\po N}H_{NB}.
\ee
For any 2-tensor $H$ and any 1-form $F$, we have the following identity:
\be\lab{coo2cinq}
F_{\po e_B}H_{BA}+F_BH_{\po e_B A}=-F_NH_{\po NA}-F_{\po N}H_{NA}.
\ee
\end{lemma}

\begin{proof}
The identities \eqref{coo2ter}, \eqref{coo2quatre} and \eqref{coo2cinq} are easy consequences of the identities:
\be\lab{coo2}
\gg(\po e_1,e_1)=0,\,\gg(\po e_2,e_2)=0,\,\gg(\po e_1,e_2)=-\gg(\po e_2,e_1),\,\gg(\po e_A,N)=-\gg(\po N,e_A),
\ee
which follow from the fact that $(e_1,e_2,N)$ is orthonormal.
\end{proof}

We first consider commutators for scalar functions.
\begin{proposition}
Let $f$ a scalar function on $\mathcal{M}$. We have:
\be\lab{commo1}
[\po,L]f=\po N(f),\, [\po,\lb]f=-\po N(f)
\ee
and:
\be\lab{commo2}
[\po,\nabb]f=-\nabb_{\po N}fN-\nabla_Nf\po N.
\ee
\end{proposition}

\begin{proof}
Differentiating $\gg(T,N)=0$ and $\gg(N,N)=1$, and using the fact that $T$ is independent of $\o$, 
we obtain 
$$\gg(T,\po N)=0\textrm{ and }\gg(N,\po N)=0$$
which shows that $\po N$ is tangent to $\ptu$. Furthermore, since $T$ is independent of $\o$, and since $L=T+N$ and $\lb=T-N$, we have 
\be\lab{coo1}
\po L=\po N\textrm{ and }\po\lb=-\po N,
\ee
which immediately yields \eqref{commo1}. Furthermore, we have:
$$\nabb f=\dd f+\half\gg(\dd f,L)\lb+\half\gg(\dd f,\lb)L$$
where $\dd f=-\gg^{\a\b}\partial_\a(f)\partial_\b$ denotes the space-time gradient of $f$. Together with \eqref{commo1} and the fact that $[\po,\dd]=0$, this implies \eqref{commo2}. This concludes the proof of the proposition. 
\end{proof}

Next, we consider commutators for $\ptu$-tangent vectorfields. We introduce the projection $\Pi$ of vectorfields on $\Sit$ onto vectorfields tangent to $\ptu$:
$$\Pi X=X-g(X,N)N.$$
We have the following proposition:
\begin{proposition}
Let $X$ a $\ptu$-tangent vectorfield. We have:
\be\lab{commo3}
\po\ddb_LX-\ddb_L(\Pi(\po X))=\nabb_{\po N}X-\gg(\ddb_L(X),\po N)N+\kepb_X \po N-g(X,\po N)\kepb_Ae_A,
\ee
and:
\bea\lab{commo4}
\po\ddb_{\lb}X-\ddb_{\lb}(\Pi(\po X))&=&-\nabb_{\po N}X-\gg(\ddb_{\lb}(X),\po N)N+(\z_X-\xib_X)\po N\\
\nn&&-g(X,\po N)(\z_A-\xib_A)e_A.
\eea
\end{proposition}

\begin{proof}
We start with $\po\ddb_LX-\ddb_L(\Pi(\po X))$. By the definition of $\ddb_L$, we have: 
$$\ddb_L(X)=\gg(\dd_LX,e_A)e_A.$$
Differentiating with respect to $\o$ and using \eqref{coo1}, we obtain:
$$\po(\ddb_L(X))=\gg(\dd_L(\po X),e_A)e_A+\gg(\dd_{\po N}X,e_A)e_A+\gg(\dd_LX,\po e_A)e_A+\gg(\dd_LX,e_A)\po e_A$$
which together with \eqref{coo2ter} yields:
\be\lab{coo3}
\po\ddb_LX=\gg(\dd_L(\po X),e_A)e_A+\nabb_{\po N}X-\gg(\dd_LX,N)\po N-\gg(\dd_LX,\po N)N.
\ee
Since $X$ is tangent to $\ptu$, we have: 
\be\lab{coo2bis}
\po X=\Pi(\po X)+g(\po X,N)N=\Pi(\po X)-g(X,\po N)N
\ee
which yields:
\be\lab{coo3bis}
\gg(\dd_L(\po X),e_A)e_A=\gg(\dd_L(\Pi(\po X)),e_A)e_A-g(X, \po N)\gg(\dd_LN,e_A)e_A.
\ee
Now, using the Ricci equations \eqref{ricciform} for $\dd_LN$ and the fact that $X$ is $\ptu$-tangent, we have:
$$\gg(\dd_LN,e_A)=\kepb_A\textrm{ and }\gg(\dd_LX,N)=-\gg(X,\dd_LN)=-\kepb_X.$$
Together with \eqref{coo3} and \eqref{coo3bis}, this yields \eqref{commo3}.

Next, we consider $\po\ddb_{\lb}X-\ddb_{\lb}(\Pi(\po X))$. Similarly as before, we obtain the analog of \eqref{coo3}:
\be\lab{coo4}
\po\ddb_{\lb}X=\gg(\dd_{\lb}(\po X),e_A)e_A-\nabb_{\po N}X-\gg(\dd_{\lb}X,N)\po N-\gg(\dd_{\lb}X,\po N)N.
\ee
and the analog of \eqref{coo3bis}:
\be\lab{coo4bis}
\gg(\dd_{\lb}(\po X),e_A)e_A=\gg(\dd_{\lb}(\Pi(\po X)),e_A)e_A-g(X, \po N)\gg(\dd_{\lb}N,e_A)e_A.
\ee
Now, using the Ricci equations \eqref{ricciform} for $\dd_{\lb}N$ and the fact that $X$ is $\ptu$-tangent, we have:
$$\gg(\dd_{\lb}N,e_A)=(\z_A-\xib_A)\textrm{ and }\gg(\dd_{\lb}X,N)=-\gg(X,\dd_{\lb}N)=-(\z_X-\xib_X).$$
Together with \eqref{coo4} and \eqref{coo4bis}, this yields \eqref{commo4}. This concludes the proof of the proposition.
\end{proof}

Next, we consider commutators for $\ptu$-tangent tensors. Let $F$ a m-covariant tensor tangent to the surfaces $P_{t,u}$. Then, $\po F$ is not a tangent to $\ptu$. We denote by $\Pi F$ the $\ptu$-tangent part of $F$. We have the following proposition:
\begin{proposition}
Let $F_{\und{A}}$ be an m-covariant tensor tangent to the surfaces $P_{t,u}$.
Then,
\be\lab{commo5}
\po\ddb_LF_{\und{A}}-\ddb_L\Pi(\po F)_{\und{A}}=\nabb_{\po N}F_{\und{A}} -\kepb_{A_i}(\po N)_CF_{A_1..\Check{C}..A_m}+g(e_{A_i},\po N)\kepb_CF_{A_1..\Check{C}..A_m},
\ee
and:
\bea\lab{commo6}
\po\ddb_{\lb}F_{\und{A}}-\ddb_{\lb}\Pi(\po F)_{\und{A}}&=&-\nabb_{\po N}F_{\und{A}} -(\z_{A_i}-\xib_{A_i})(\po N)_CF_{A_1..\Check{C}..A_m}\\
\nn&&+g(e_{A_i},\po N)(\z_C-\xib_C)F_{A_1..\Check{C}..A_m}.
\eea
\end{proposition}

\begin{proof}
For simplicity give the proof for a $\ptu$-tangent 1-form $F$, the general case being similar. 
We start with $\po\ddb_LF-\ddb_L\Pi(\po F)$. By definition, we have:
$$\ddb_LF_A=L(F_A)-F_{\ddb_Le_A}.$$
Differentiating with respect to $\o$ and using \eqref{coo1}, we obtain:
\bee
\po(\ddb_LF)_A+\ddb_LF_{\po e_A}&=&\po N(F_A)+L(\po F_A)+L(F_{\po e_A})-\po F_{\ddb_L(e_A)}-F_{\po(\ddb_Le_A)}\\
&=&\po N(F_A)+\ddb_L(\Pi(\po F))_A+L(F_{\po e_A})-F_{\po(\ddb_Le_A)},
\eee
which together with \eqref{coo2bis} with $X=e_A$, \eqref{commo3}, and the fact that $F$ and $\ddb_LF$   are $\ptu$-tangent yields:
\bee
\po(\ddb_LF)_A+\ddb_LF_{\Pi(\po e_A)}&=&\po N(F_A)+\ddb_L(\Pi(\po F))_A+L(F_{\Pi(\po e_A)})\\
&&-F_{\ddb_L(\Pi(\po e_A))}-F_{\nabb_{\po N}e_A}-\kepb_AF_{\po N}+g(e_A,\po N)F\c\kepb\\
&=&\nabb_{\po N}F_A+\ddb_L(\Pi(\po F))_A+\ddb_LF_{\Pi(\po e_A)}\\
&&-\kepb_AF_{\po N}+g(e_A,\po N)F\c\kepb.
\eee 
This concludes the proof of \eqref{commo5}. The proof of \eqref{commo6} is similar and left to the reader.
\end{proof}

Finally, we consider the commutator of $\po$ with $\mathcal{D}_2$.
\begin{proposition}
Let $H$ a symmetric $\ptu$-tangent 2-tensor. Then, we have:
\bea
\lab{commo7}\po(\divb(H))_A-\divb(\Pi(\po H))_A&=&-\ddb_NH_{A\po N}+g(\po N,e_A)\th\c H\\
\nn&&+\th_{\po NB}H_{BA}-\th_{AB}H_{B\po N}-\trt H_{A\po N},
\eea
where $\th$ is the second fundamental form of $\ptu$ in $\Sit$ (i.e. $\th_{AB}=g(\nabla_AN,e_B)$).
\end{proposition}

\begin{proof}
We first derive a formula for $\po(\ddb_Be_A)$. We have:
\bee
\po(\ddb_Be_A)&=&\po(g(\dd_Be_A,e_C)e_C)\\
&=&g(\dd_{\po e_B}e_A,e_C)e_C+g(\dd_B(\po e_A),e_C)e_C+g(\dd_Be_A,\po e_C)e_C\\
&&+g(\dd_Be_A,e_C)\po e_C.
\eee
Now, using \eqref{coo2bis} to decompose $e_A$, we have:
$$g(\dd_Be_A,\po e_C)=g(\dd_B(\Pi(e_A)),\po e_C)-g(\po N,e_A)g(\dd_BN,e_C).$$
Furthermore, the analog of \eqref{coo2ter} for 2-tensors yields:
$$g(\dd_Be_A,\po e_C)e_C+g(\dd_Be_A,e_C)\po e_C=-g(\po N,e_C)(g(\dd_Be_A,N)e_C+g(\dd_Be_A,e_C)N).$$
Thus, we obtain:
\bea
\lab{coo5}\po(\ddb_Be_A)&=&\ddb_{\po e_B}(e_A)+\ddb_B(\Pi(\po e_A))-g(\po N,e_A)g(\dd_BN,e_C)e_C\\
\nn&&-g(\po N,e_C)(g(\dd_Be_A,N)e_C+g(\dd_Be_A,e_C)N)\\
\nn&=&\ddb_{\po e_B}(e_A)+\ddb_B(\Pi(\po e_A))-g(\po N,e_A)\th_{BC}e_C+\th_{AB}\po N\\
\nn&&-g(\dd_Be_A,\po N)N.
\eea

We now compute $\po(\nabb_CH_{AB})$. We have:
\bea
\lab{coo6}\po(\nabb_CH_{AB})&=&\po e_C(H_{AB})+e_C(\po H_{AB})+e_C(H_{\po e_AB})+e_C(H_{A\po e_B})\\
\nn&&-\po H_{\ddb_Ce_AB}-\po H_{A\ddb_Ce_B}-H_{\po(\ddb_Ce_A)B}-H_{\ddb_Ce_A\po(e_B)}\\
\nn&&-H_{A\po(\ddb_Ce_B)}-H_{\po(e_A)\ddb_Ce_B}\\
\nn&=&\ddb_C(\Pi(\po H))_{AB}+\po e_C(H_{AB})+e_C(H_{\po e_AB})+e_C(H_{A\po e_B})\\
\nn&&-H_{\po(\ddb_Ce_A)B}-H_{\ddb_Ce_A\po(e_B)}-H_{A\po(\ddb_Ce_B)}-H_{\po(e_A)\ddb_Ce_B}
\eea
Using \eqref{coo5}, we have:
\bee
H_{\po(\ddb_Ce_A)B}+H_{A\po(\ddb_Ce_B)}&=&H_{\ddb_{\po e_C}e_AB}+H_{A\ddb_{\po e_C}e_B}
+H_{\ddb_C(\Pi(\po e_A))B}\\
\nn&&+H_{A\ddb_C(\Pi(\po e_B))}
-g(\po N,e_A)\th_{CD}H_{DB}\\
\nn&&-g(\po N,e_B)\th_{CD}H_{AD}
+\th_{AC}H_{\po NB}+\th_{BC}H_{A\po N},
\eee
which together with \eqref{coo6} yields:
\bea
\lab{coo7}\po(\nabb_CH_{AB})&=&\ddb_C(\Pi(\po H))_{AB}+\ddb_{\po e_C}H_{AB}+\ddb_CH_{\Pi(\po e_A)B}+\ddb_CH_{A\Pi(\po e_B)}\\
\nn&&  +g(\po N,e_A)\th_{CD}H_{DB}+g(\po N,e_B)\th_{CD}H_{AD}
-\th_{AC}H_{\po NB}-\th_{BC}H_{A\po N}.
\eea
Contracting \eqref{coo7}, we obtain:
\bea
\lab{coo8}\po(\divb H_A)&=&\divb(\Pi(\po H))_A+\ddb_{\po e_C}H_{AC}+\ddb_CH_{A\po e_C}+\ddb_CH_{\po e_AC}\\
\nn&&+g(\po N,e_A)\th_{BC}H_{CB}+g(\po N,e_C)\th_{CB}H_{AB}-\th_{AC}H_{\po NC}-\trt H_{A\po N}.
\eea
Now, the analog of \eqref{coo2ter}-\eqref{coo2cinq} yields:
$$\ddb_{\po e_C}H_{AC}+\ddb_CH_{A\po e_C}=-\ddb_NH_{A\po N}$$
which together with \eqref{coo8} implies \eqref{commo7}. This concludes the proof.
\end{proof}

\subsection{Control of $\po N$, $\po b$, $\po\chi$ and $\po\z$}\lab{sec:jycroispas}

\subsubsection{Derivatives of $\po N$ with respect to the null frame}

We first compute the derivatives of $\po N$ with respect to the null frame.
\begin{lemma}
\bea
\lab{popo1}\dd_{L}(\po N)&=&-\chi_{\po NB}e_B-\db\po N+\kepb_{\po N}L,\\
\nn\dd_{\lb}(\po N)&=&2\po\z_Ae_A+\chi_{\po NB}e_B+(\d+n^{-1}\nabla_Nn)\po N\\
\lab{popo2}&&+(2\kep_{\po N}+n^{-1}\nabla_{\po N}n)L-2\z_{\po N}N,\\
\lab{popo3}\dd_A(\po N)&=& \po\chi_{AB}e_B-g(\po N,e_A)\z_Be_B-g(\po N,e_A)\d L-\chi_{A\po N}N.
\eea
\end{lemma}

\begin{proof}
We start with $\dd_L(\po N)$. Using the Ricci equation for $\dd_LL$ and the fact that $\po L=\po N$, we have:
\be\lab{po1}
\dd_L\po N+\dd_{\po N}L=-\po(\db)L-\db\po N.
\ee
Now, we have:
\be\lab{po2}
\po\d=2\kep_{\po N},\,\po\db=2\kep_{\po N}-n^{-1}\nabla_{\po N}n.
\ee
Also, the Ricci equations \eqref{ricciform} and the fact that $\po N$ is $\ptu$-tangent imply:
$$\dd_{\po N}L=\chi_{\po NB}e_B-\kep_{\po N}L$$
which together with \eqref{po1} and \eqref{po2} yield \eqref{popo1}.

Next we consider $\dd_{\lb}(\po N)$. Using the Ricci equation for $\dd_LL$ and the fact that $\po L=\po N$ and $\po\lb=-\po N$, we have:
$$
\dd_{\lb}\po N- \dd_{\po N}L = 2\po\z_{A}e_A +2\z_{\po e_A}e_A+2\z_A\po e_A+(\po\d +n^{-1}\nab_{\po N}n)L+(\d +n^{-1}\nab_Nn)\po N,
$$
which together with the Ricci equations \eqref{ricciform}, \eqref{coo2ter} and \eqref{po2} yields \eqref{popo2}.

Finally, we consider $\dd_A(\po N)$. Using the Ricci equation for $\dd_AL$ and the fact that $\po L=\po N$, we have:
$$\dd_A\po N+\dd_{\po e_A}L=\po\chi_{AB}e_B+\chi_{\po e_AB}e_B+\chi_{A\po e_B}e_B+\chi_{AB}\po e_B-k_{\po NA}L-k_{N\po e_A}L-\kep_A\po N.$$
Using \eqref{coo2bis} with $X=e_A$, we obtain:
\bee
\dd_A\po N-g(\po N,e_A)\dd_{N}L &=& \po\chi_{AB}e_B+\chi_{A\po e_B}e_B+\chi_{AB}\po e_B-k_{\po NA}L\\
\nn&&+g(\po N,e_A)\d L-\kep_A\po N,
\eee
which together with the Ricci equations \eqref{ricciform} and \eqref{coo2ter} yields \eqref{popo3}. This concludes the proof of the lemma.
\end{proof}

\subsubsection{Transport equations for $\po\chi$ and $\po\z$}

\begin{lemma}
$\po\chi$ and $\po\z$ satisfy the following transport equations:
\bea
\lab{popo4}\ddb_L(\Pi(\po\chi))_{AB}&=&-\nabb_{\po N}\chi_{AB}-(\po\chi)_{AC}\chi_{CB}-\chi_{AC}(\po\chi)_{CB}-\db\po\chi_{AB}\\
\nn&&+\kepb_A\chi_{\po NB}+\kepb_B\chi_{A\po N}+(\po N)_A\chi_{CB}\kep_C+(\po N)_B\chi_{AC}\kepb_C\\
\nn&&-(2\kep_{\po N}-n^{-1}\nabla_{\po N}n)\chi_{AB}+(\po N)_C(\in_{AC}{}^*\b_B+\in_{BC}{}^*\b_A),
\eea
\bea\lab{popo5}
\ddb_L(\Pi(\po\z))_A&=&-\nabb_{\po N}\z_A+\kepb_A\z_{\po N}-(\po N)_A\kepb\c\z-(k_{B\po N}+\po\z_B)\chi_{AB}\\
\nn&&-(\kepb_B+\z_B)\po\chi_{AB}-\frac{(\po N)_B}{2}(-\a_{AB}+\r\d_{AB}+3\s\in_{AB}).
\eea
\end{lemma}

\begin{proof}
We start with the proof of \eqref{popo4}. Note first from the definition of $\a$, $\b$ and the fact that $\po L=\po N$:
\be\lab{po3}
\po\a_{AB}=-(\po N)_C(\in_{AC}{}^*\b_B+\in_{BC}{}^*\b_A).
\ee
Now, differentiating the transport equation \eqref{D4chi} with respect to $\o$, we obtain:
$$\po(\ddb_L\chi)_{AB}=-\chi_{A\po e_C}\chi_{CB}-\chi_{AC}\chi_{\po e_CB}-\po(\db)\chi_{AB}-\db\po\chi_{AB}-\po\a_{AB}.$$
Together with \eqref{coo2quatre}, the commutator formula \eqref{commo5}, \eqref{po2} and \eqref{po3}, we obtain \eqref{popo4}. 

Next, we prove \eqref{popo5}. Note first from the definition of $\a, \b, \r, \s$ and the fact that $\po L=\po N, \po\lb=-\po N$:
\be\lab{po4}
\po\b_{A}=\frac{(\po N)_B}{2}(-\a_{AB}+\r\d_{AB}+3\s\in_{AB}).
\ee
Now, differentiating the transport equation \eqref{D4eta} with respect to $\o$, we obtain:
\bee
\po(\ddb_L\z)_A&=&-(\po\kepb_{B}+\po\z_{B})\chi_{AB}-(\kepb_{B}+\z_{B})\po\chi_{AB}-(\kepb_{\po e_B}+\z_{\po e_B})\chi_{AB}\\
&&-(\kepb_{B}+\z_{B})\chi_{A\po e_B}-\po\b_A.
\eee
Together with \eqref{coo2cinq}, the commutator formula \eqref{commo5}, and \eqref{po4}, we obtain \eqref{popo5}. This concludes the proof of the lemma.
\end{proof}

\subsubsection{Estimates for $\po N, \po b, \po\chi$ and $\po\z$}

We first derive the $L^\infty$ bound \eqref{estNomega} for $\po N$. In view of the formula \eqref{popo1}  for $\dd_L(\po N)$, we have: 
\bea
\nn\norm{\dd_L(\po N)}_{\xt{\infty}{2}}&\les& \norm{-\chi_{\po NB}e_B-\db\po N+\kepb_{\po N}L}_{\xt{\infty}{2}}\\
\nn&\les&(\norm{\chi}_{\xt{\infty}{2}}+\norm{\db}_{\xt{\infty}{2}}+\norm{\kepb}_{\xt{\infty}{2}})\norm{\po N}_{L^\infty}\\
\lab{po5}&\les&\ep\norm{\po N}_{L^\infty},
\eea
where we used the estimates \eqref{esttrc} \eqref{esthch} for $\chi$ and the estimate \eqref{estk} for $\db$ and $\kepb$ in the last inequality. The estimate for transport equations \eqref{estimtransport1} and \eqref{po5} yield:
$$\norm{\po N}_{L^\infty}\les \norm{\ddb_L\po N}_{\xt{\infty}{2}}\les \ep\norm{\po N}_{L^\infty}$$
which yields the  $L^\infty$ bound \eqref{estNomega} for $\po N$:
\be\lab{po6}
\norm{\po N}_{L^\infty}\les 1.
\ee

Next, we derive an estimate for $\po\chi$. First, the fact that $\chi$ is a $\ptu$-tangent 2-tensor yields for any vectorfields $X, Y$ on $\Sit$:
$$\po\chi_{XY}=\Pi(\po\chi)_{\Pi(X)\Pi(Y)}-g(N,X)\chi_{\po N\Pi(Y)}-g(N,Y)\chi_{\po N\Pi(X)}$$
which implies:
\be\lab{po6bis}
\norm{\po\chi}_{\xt{2}{\infty}}\les \norm{\Pi(\po\chi)}_{\xt{\infty}{2}}+\norm{\chi}_{\xt{2}{\infty}}\norm{\po N}_{L^\infty}\les \norm{\Pi(\po\chi)}_{\xt{\infty}{2}}+\ep,
\ee
where we used the estimate \eqref{po6} for $\po N$, and the estimates \eqref{esttrc} \eqref{esthch} for $\chi$. In view of \eqref{po6bis}, we have to estimate $\norm{\Pi(\po\chi)}_{\xt{\infty}{2}}$. The formula \eqref{popo4}  for $\dd_L(\Pi(\po\chi))$ implies: 
\bea\lab{po7}
&&\norm{\dd_L(\Pi(\po\chi))}_{\lh{2}}\\
\nn&\les&\norm{\nabb_{\po N}\chi}_{\lh{2}}+\norm{\po(\chi)\chi}_{\lh{2}}+\norm{\db\po\chi}_{\lh{2}}+\norm{\kepb\chi\po N}_{\lh{2}}+\norm{\po N\chi\kep}_{\lh{2}}\\
\nn&&+\norm{(2\kep_{\po N}-n^{-1}\nabla_{\po N}n)\chi}_{\lh{2}}+\norm{\po N(\in_{AC}{}^*\b_B+\in_{BC}{}^*\b_A)}_{\lh{2}}\\
\nn&\les&(1+\norm{\po N}_{L^\infty})\bigg(\norm{\nabb\chi}_{\lh{2}}+\norm{\po(\chi)}_{\xt{2}{\infty}}(\norm{\chi}_{\xt{\infty}{2}}+\norm{\db}_{\xt{\infty}{2}})\\
\nn&&+\no(\chi)(\no(\kep)+\noo(n))+\norm{\b}_{\lh{2}}\bigg)\\
\nn&\les&\ep+\ep\norm{\po(\chi)}_{\xt{2}{\infty}},
\eea
where we used in the last inequality the estimate \eqref{po6} for $\po N$, the curvature bound \eqref{curvflux1} for $\b$, and the estimates \eqref{estn}-\eqref{esthch} for $n, \kep, \kepb, \db$ and $\chi$. The estimate for transport equations \eqref{estimtransport1} and \eqref{po7} yield:
$$\norm{\Pi(\po\chi)}_{\xt{2}{\infty}}\les \norm{\ddb_L\Pi(\po\chi))}_{\lh{2}}\les \ep+\ep\norm{\po(\chi)}_{\xt{2}{\infty}}$$
which together with \eqref{po6bis} yields:
\be\lab{po8}
\norm{\po\chi}_{\xt{2}{\infty}}\les\ep.
\ee

We now derive an estimate for $\po\z$. First, the fact that $\z$ is a $\ptu$-tangent 1-form yields for any vectorfields $X$ on $\Sit$:
$$\po\z_X=\Pi(\po\z)_{\Pi(X)}-g(N,X)\z_{\po N}$$
which implies:
\be\lab{po8bis}
\norm{\po\z}_{\xt{2}{\infty}}\les \norm{\Pi(\po\z)}_{\xt{\infty}{2}}+\norm{\z}_{\xt{2}{\infty}}\norm{\po N}_{L^\infty}\les \norm{\Pi(\po\z)}_{\xt{\infty}{2}}+\ep,
\ee
where we used the estimate \eqref{po6} for $\po N$, and the estimate \eqref{estzeta} for $\z$. In view of \eqref{po8bis}, we have to estimate $\norm{\Pi(\po\z)}_{\xt{\infty}{2}}$. The formula \eqref{popo5}  for $\dd_L(\Pi(\po\z))$ implies: 
\bea\lab{po9}
&&\norm{\dd_L(\Pi(\po\z))}_{\lh{2}}\\
\nn&\les&\norm{\nabb_{\po N}\z}_{\lh{2}}+\norm{\kepb\z_{\po N}}_{\lh{2}}+\norm{\po N\kepb\c\z}_{\lh{2}}+\norm{(k_{B\po N}+\po\z_B)\chi}_{\lh{2}}\\
\nn&&+\norm{(\kepb+\z)\po\chi}_{\lh{2}}+\norm{\po N(-\a_{AB}+\r\d_{AB}+3\s\in_{AB}}_{\lh{2}}\\
\nn&\les& (1+\norm{\po N}_{L^\infty})\bigg(\norm{\nabb\z}_{\lh{2}}+\no(\kepb)\no(\z)+(\norm{k}_{\xt{2}{\infty}}+\norm{\z}_{\xt{2}{\infty}})\norm{\chi}_{\xt{2}{\infty}}\\
\nn&&+(\norm{\kepb}_{\xt{\infty}{2}}+\norm{\z}_{\xt{\infty}{2}})\norm{\po\chi}_{\xt{2}{\infty}}+\norm{\a}_{\lh{2}}+\norm{\r}_{\lh{2}}+\norm{\s}_{\lh{2}}\bigg)\\
\nn&\les&\ep+\ep\norm{\po(\z)}_{\xt{2}{\infty}},
\eea
where we used in the last inequality the estimate \eqref{po6} for $\po N$, the curvature bound \eqref{curvflux1} for $\a, \r$ and $\s$, and the estimates \eqref{estn}-\eqref{estzeta} for $\kepb, k, \chi$ and $\z$. The estimate for transport equations \eqref{estimtransport1} and \eqref{po9} yield:
$$\norm{\Pi(\po\z)}_{\xt{2}{\infty}}\les \norm{\ddb_L\Pi(\po\z))}_{\lh{2}}\les \ep+\ep\norm{\po(\z)}_{\xt{2}{\infty}}$$
which together with \eqref{po8bis} yields:
\be\lab{po10}
\norm{\po\z}_{\xt{2}{\infty}}\les\ep.
\ee

We now estimate $\po b$. Differentiating the transport equation \eqref{D4a} for $b$ with respect to $\o$ and using the commutator formula \eqref{commo1}, we obtain:
$$L(\po b)=-\nabb_{\po N}b-\po(b)\db-\po(\db)b=-\nabb_{\po N}b-\po(b)\db-(2\kep_{\po N}-n^{-1}\nabla_{\po N}n)b,$$
where we used \eqref{po2} in the last equality. Since, $\nabb b=b(\z-\kep)$ from \eqref{etaa}, we obtain:
\be\lab{dw}
L(\po b)=-b\z_{\po N}-\po(b)\db-\kepb_{\po N}b.
\ee
This yields:
\bea\label{po11}
&&\norm{L(\po b)}_{\xt{\infty}{2}}\\
\nn&\les& \norm{b\z_{\po N}}_{\xt{\infty}{2}}+\norm{\po(b)\db}_{\xt{\infty}{2}}+\norm{\kepb_{\po N}b}_{\xt{\infty}{2}}\\
\nn&\les& (1+\norm{\po N}_{L^\infty})\bigg(\norm{b}_{\lh{\infty}}(\norm{\z}_{\xt{\infty}{2}}+\norm{\kepb}_{\xt{\infty}{2}})+\norm{\po b}_{\lh{\infty}}\norm{\db}_{\xt{\infty}{2}}\bigg)\\
\nn&\les& \ep+\ep\norm{\po(b)}_{\lh{\infty}},
\eea
where we used in the last inequality the estimate \eqref{po6} for $\po N$, and the estimates \eqref{estn}-\eqref{estb} for $n, \kep, \db$ and $b$. The estimate for transport equations \eqref{estimtransport1} and \eqref{po11} yield:
$$\norm{\po b}_{\lh{\infty}}\les \norm{L(\po b)}_{\xt{\infty}{2}}\les \ep+\ep\norm{\po b}_{\lh{\infty}}$$
which in turn implies:
\be\lab{po12}
\norm{\po b}_{\lh{\infty}}+\norm{L(\po b)}_{\xt{\infty}{2}}\les\ep.
\ee
Next, we estimate $\nabb\po b$. Recall from \eqref{etaa} that $\nabb b=b(\z-\kep)$. Differentiating with respect to $\o$ and using the commutator formula \eqref{commo2}, we obtain:
$$\nabb_A\po b=\nabb_{\po N}bN+\nabla_Nb\po N+\po b(\z_A-\kep_A)+b(\po\z_A-k_{\po NA})$$
which yields the estimate:
\bea
\nn\norm{\nabb b}_{\xt{2}{\infty}}&\les& (1+\norm{\po N}_{L^\infty})\bigg(\norm{\nabla b}_{\xt{2}{\infty}}+\no(\po b)(\no(\z)+\no(\zb))+\norm{\po\z}_{\xt{2}{\infty}}\\
\nn&&+\norm{k}_{\xt{2}{\infty}}\bigg)\\
\lab{po13}&\les& \ep+\no(\po b)\ep,
\eea
where we used in the last inequality the estimate \eqref{po6} for $\po N$, the estimates \eqref{estn}-\eqref{estzeta} for $k, b, \zb$ and $\z$, and the estimate \eqref{po10} for $\po\z$. Now, \eqref{po12} and \eqref{po13} yield:
\be\lab{po14}
\norm{\nabb\po b}_{\xt{2}{\infty}}\les \ep.
\ee

Finally, we estimate $\dd\po N$. In view of \eqref{po5} and \eqref{po6}, we have:
\be\lab{po15}
\norm{\dd_L(\po N)}_{\xt{\infty}{2}}\les\ep.
\ee
Then, using the formula for $\dd_{\lb}\po N$ and $\dd_A\po N$ given respectively by \eqref{popo2} and \eqref{popo3}, we obtain:
\bea
\nn&&\norm{\dd_{\lb}(\po N)}_{\xt{2}{\infty}}+\norm{\dd_A(\po N)}_{\xt{2}{\infty}}\\
\nn&\les&\norm{\po\z}_{\xt{2}{\infty}}+\norm{\po\chi}_{\xt{2}{\infty}}+\norm{\po N}_{L^\infty}\bigg(\norm{\chi}_{\xt{2}{\infty}}+\norm{\d}_{\xt{2}{\infty}}\\
\nn&&+\norm{n^{-1}\nabla n}_{\xt{2}{\infty}}+\norm{\kep}_{\xt{2}{\infty}}+\norm{\z}_{\xt{2}{\infty}}\bigg)\\
\lab{po16}&\les& \ep,
\eea
where we used the estimate \eqref{po6} for $\po N$, the estimates \eqref{estn}-\eqref{estzeta} for $n, \d, \kep, \chi$ and $\z$, the estimate \eqref{po8} for $\po\chi$, and the estimate \eqref{po10} for $\po\z$. 

Finally, \eqref{po6} yields the desired $L^\infty$ bound \eqref{estNomega} for $\po N$, while \eqref{po8}, \eqref{po10}, \eqref{po12}, \eqref{po14}, \eqref{po15} and \eqref{po16} yields the desired estimate \eqref{estricciomega}.

\subsection{Control of $\ddb_{\lb}\Pi(\po\chi)$}\lab{sec:jycroispas1}

The goal of this section is to prove the estimate \eqref{estricciomegabis} for $\ddb_{\lb}\Pi(\po\chi)$. We will use the following lemmas.

\begin{lemma}\lab{lemma:imj}
$\ddb_{\lb}\Pi(\po\chi)$ satisfies the following transport equation
\be\lab{imj1}
\ddb_L(\ddb_{\lb}\Pi(\po\chi))=-\ddb_{\lb}\Pi(\po\chi)\c\chi-\chi\c\ddb_{\lb}\Pi(\po\chi)+\nabb F_1+F_2,
\ee
where $F_1$ and $F_2$ are $\ptu$-tangent tensors satisfying the following estimate
\be\lab{imj2}
\norm{F_1}_{\li{\infty}{2}}+\norm{F_2}_{\tx{2}{1}}\les\ep.
\ee
\end{lemma}

\begin{lemma}\lab{lemma:poo3}
Recall that $\ga$ denotes the metric induced by $\gg$ on $\ptu$. Let $M$ the $\ptu$-tangent 2-tensor defined as the solution of the following transport equation:
\be\lab{poo35}
\ddb_LM_{AB}=M_{AC}\chi_{CB},\,M_{AB}=\ga_{AB}\textrm{ on }\pou,
\ee
Then, $M_{AB}$ satisfies the following estimate:
\be\lab{poo36}
\norm{M-\ga}_{L^\infty}+\norm{\nabb M}_{\BB^0}\les\ep.
\ee
\end{lemma}

\begin{lemma}\lab{lemma:bispoo3}
Recall that $\ga$ denotes the metric induced by $\gg$ on $\ptu$. Let $\widetilde{M}$ the $\ptu$-tangent 2-tensor defined as the solution of the following transport equation:
\be\lab{bispoo35}
\ddb_L\widetilde{M}_{AB}=\chi_{AC}\widetilde{M}_{CB},\,\widetilde{M}_{AB}=\ga_{AB}\textrm{ on }\pou,
\ee
Then, $\widetilde{M}_{AB}$ satisfies the following estimate:
\be\lab{bispoo36}
\norm{\widetilde{M}-\ga}_{L^\infty}+\norm{\nabb \widetilde{M}}_{\BB^0}\les\ep.
\ee
\end{lemma}

\begin{lemma}\lab{lemma:poo4}
Let $F$ a $\ptu$-tangent tensor. Then, for any $1\leq p<q\leq +\infty$ and for any $j\geq 0$, we have:
\be\lab{poo36bis}
\norm{P_jF}_{\tx{p}{\infty}}\les 2^j\norm{F}_{\tx{q}{2}}.
\ee
Also, taking the dual, for any $1\leq p<q\leq +\infty$ and for any $j\geq 0$, we have 
\be\lab{poo37}
\norm{P_jF}_{\tx{p}{2}}\les 2^j\norm{F}_{\tx{q}{1}}.
\ee
\end{lemma}

\begin{lemma}\lab{lemma:poo5}
Let $F$ a $\ptu$-tangent tensor. Then, for any $j\geq 0$ and for any $2\leq p<+\infty$, we have:
$$\normm{P_j\left(\int_0^t\nabb(F)dt\right)}_{\tx{p}{2}}\les 2^j\norm{F}_{\tx{1}{2}}.$$
\end{lemma}

\begin{lemma}\lab{lemma:poo6}
Let $F$ a $\ptu$-tangent 1-form and $2<p\leq +\infty$ such that for all $j\geq 0$:
$$\norm{P_jF}_{\tx{p}{2}}\les 2^j\ep+2^{\frac{j}{2}}\ep\gamma(u),$$
and let $M$ such that:
$$\norm{M-\ga}_{L^\infty}+\norm{\nabb M}_{\BB^0}\les \ep.$$
Then, we have for any $2\leq q<p$ and all $j\geq 0$:
$$\norm{P_j(M^{-1}F)}_{\tx{q}{2}}\les 2^j\ep+2^{\frac{j}{2}}\ep\gamma(u).$$
\end{lemma}

\begin{lemma}\lab{lemma:bispoo6}
Let $F$ a $\ptu$-tangent 1-form and $2\leq p\leq +\infty$ such that for all $j\geq 0$:
$$\norm{P_jF}_{\tx{p}{2}}\les 2^j\ep+2^{\frac{j}{2}}\ep\gamma(u),$$
and let $\widetilde{M}$ such that:
$$\norm{\widetilde{M}-\ga}_{L^\infty}+\norm{\nabb \widetilde{M}}_{\BB^0}\les \ep.$$
Then, we have for any $2\leq q<p$ and all $j\geq 0$:
$$\norm{P_j(F\widetilde{M}^{-1})}_{\tx{q}{2}}\les 2^j\ep+2^{\frac{j}{2}}\ep\gamma(u).$$
\end{lemma}

The proof of Lemma \ref{lemma:imj} is postponed to section \ref{sec:imj}. The proof of Lemma \ref{lemma:poo3} is postponed to section \ref{sec:proofpoo3}. The proof of Lemma \ref{lemma:bispoo3} is completely analogous to the one of Lemma \ref{lemma:poo3} and left to the reader. The proof of Lemma \ref{lemma:poo4} is postponed to section \ref{sec:proofpoo4}. The proof of Lemma \ref{lemma:poo5} is postponed to section \ref{sec:proofpoo5}. The proof of Lemma \ref{lemma:poo6} is postponed to section \ref{sec:proofpoo6}. Finally, The proof of Lemma \ref{lemma:bispoo6} is completely analogous to the one of Lemma \ref{lemma:poo6}, and left to the reader. We are now in position to derive the estimate for $\ddb_{\lb}\Pi(\po\chi)$. Using the transport equation \eqref{imj1} for $\ddb_{\lb}\Pi(\po\chi)$, the transport equation \eqref{poo35} for $M$ and  the transport equation \eqref{bispoo35} for $\widetilde{M}$ allows us to get rid of the first two terms in the left-hand side of \eqref{imj1}: 
\bee
&&\ddb_L(M\c\ddb_{\lb}\Pi(\po\chi)\c\widetilde{M})\\
&=& \ddb_L(M)\c\ddb_{\lb}\Pi(\po\chi)\c\widetilde{M}+M\c\ddb_L(\ddb_{\lb}\Pi(\po\chi))\c\widetilde{M}+M\c\ddb_{\lb}\Pi(\po\chi)\c\ddb_L(\widetilde{M})\\
&=& M\c(\nabb F_1+F_2)\c\widetilde{M}\\
&=& \nabb(M\c F_1\c\widetilde{M})-\nabb(M)\c F_1\c\widetilde{M}-M\c F_1\c\nabb(\widetilde{M})+M\c F_2\c\widetilde{M}.
\eee
Let $2\leq p<q<+\infty$. This yields:
\bea
\lab{davis}&&\norm{P_j(M\c\ddb_{\lb}\Pi(\po\chi)\c\widetilde{M})}_{\tx{q}{2}}\\
\nn&\les&2^{\frac{j}{2}}\gamma(u)+\normm{P_j\left(\int_0^t\nabb(M\c F_1\c\widetilde{M})\right)}_{\tx{q}{2}}+\normm{P_j\left(\int_0^t\nabb(M)\c F_1\c\widetilde{M}dt\right)}_{\tx{q}{2}}\\
\nn&&+\normm{P_j\left(\int_0^tM\c F_1\c\nabb(\widetilde{M})dt\right)}_{\tx{q}{2}}+\normm{P_j\left(\int_0^tM\c F_2\c\widetilde{M}dt\right)}_{\tx{q}{2}},
\eea
where the term $2^{\frac{j}{2}}\gamma(u)$ comes from the initial data term at $t=0$. Next, we estimate the various terms in the right-hand side of \eqref{davis}. 

We consider the first term in the right-hand side of \eqref{davis}. Using Lemma \ref{lemma:poo5}, we have:
\bea
\lab{davis1}\normm{P_j\left(\int_0^t\nabb(M\c F_1\c\widetilde{M})dt\right)}_{\tx{q}{2}}&\les& 2^j\norm{M\c F_1}_{\lh{2}}\\
\nn&\les& 2^j\norm{M}_{L^\infty}\norm{F_1}_{\lh{2}}\norm{\widetilde{M}}_{L^\infty}\\
\nn&\les&2^j\ep,
\eea
where we used in the last inequality the estimate \eqref{imj2} for $F_1$, the estimate \eqref{poo36} for $M$, and the estimate \eqref{bispoo36} for $\widetilde{M}$.

Next, we consider the last three terms in the right-hand side of \eqref{davis}. 
Using the dual sharp Bernstein inequality  for tensors \eqref{poo37} and the estimate \eqref{estimtransport1} for transport equations, we have:
\bea
\lab{davis2}&&\normm{P_j\left(\int_0^t\nabb(M)\c F_1\c\widetilde{M}dt\right)}_{\tx{q}{2}}+\normm{P_j\left(\int_0^tM\c F_1\c\nabb(\widetilde{M})dt\right)}_{\tx{q}{2}}\\
\nn&&+\normm{P_j\left(\int_0^tM\c F_2\c\widetilde{M}dt\right)}_{\tx{q}{2}}\\
\nn&\les& 2^j\normm{P_j\left(\int_0^t\nabb(M)\c F_1\c\widetilde{M}dt\right)}_{\tx{\infty}{1}}+2^j\normm{P_j\left(\int_0^tM\c F_1\c\nabb(\widetilde{M})dt\right)}_{\tx{\infty}{1}}\\
\nn&&+2^j\normm{P_j\left(\int_0^tM\c F_2\c\widetilde{M}dt\right)}_{\tx{\infty}{1}}\\
\nn&\les& 2^j\norm{\nabb(M)\c F_1\c\widetilde{M}}_{\lh{1}}+2^j\norm{M\c F_1\c\nabb(\widetilde{M})}_{\lh{1}}+2^j\norm{M\c F_2\c\widetilde{M}}_{\lh{1}}\\
\nn&\les& 2^j\norm{\nabb(M)}_{\lh{2}}\norm{F_1}_{\lh{2}}\norm{\widetilde{M}}_{L^\infty}+2^j\norm{M}_{L^\infty}\norm{F_1}_{\lh{2}}\norm{\nabb\widetilde{M}}_{\lh{2}}\\
\nn&&+2^j\norm{M}_{L^\infty}\norm{F_2}_{\lh{1}}\norm{\widetilde{M}}_{L^\infty}\\
\nn&\les& 2^j\ep,
\eea
where we used  in the last inequality the estimate \eqref{imj2} for $F_1$ and $F_2$, the estimate \eqref{poo36} for $M$, and the estimate \eqref{bispoo36} for $\widetilde{M}$. Finally, \eqref{davis}, \eqref{davis1} and \eqref{davis2} imply
\be\lab{davis3}
\norm{P_j(M\c\ddb_{\lb}\Pi(\po\chi)\c\widetilde{M})}_{\li{q}{2}}\lesssim 2^j\ep+2^{\frac{j}{2}}\gamma(u).
\ee
Now, since we have chosen $p<q$, \eqref{davis3} together with Lemma \ref{lemma:poo6} and Lemma \ref{lemma:bispoo6} yields:
$$\norm{P_j\ddb_{\lb}\Pi(\po\chi)}_{\li{p}{2}}\lesssim 2^j\ep+2^{\frac{j}{2}}\gamma(u),$$
for any $2\leq p<+\infty$ which is the desired estimate \eqref{estricciomegabis} for $\ddb_{\lb}\Pi(\po\chi)$. 

\subsection{Proof of the decomposition \eqref{dechch} for $\hch$}\lab{sec:jycroispas2}

To conclude the proof of Theorem \ref{thregomega}, we still need to prove the existence of a decomposition \eqref{dechch} for $\hch$. In view of the Codazzi-type equation \eqref{Codaz} for $\hch$, we have:
$$\hch=\mathcal{D}_2^{-1}\left(\half\nabb\trc-\kep\c\chi-\b\right),$$
and we choose the following decomposition:
\be\lab{po17}
\hch=\chi_1+\chi_2\textrm{ where }\chi_1=\mathcal{D}_2^{-1}\left(\half\nabb\trc-\kep\c\chi\right)\textrm{ and }\chi_2=-\mathcal{D}_2^{-1}\b.
\ee

\subsubsection{Estimates for $\chi_1$}

\paragraph{Estimate for $\norm{\nabb\chi_1}_{\lh{2}}$.}

We start by estimating $\nabb\chi_1$. Using the estimate \eqref{eq:estimdcal-1} satisfied by $\mathcal{D}_2^{-1}$ and the definition \eqref{po17} of $\chi_1$, we have:
\bea
\lab{po18}\norm{\nabb\chi_1}_{\tx{\infty}{2}}&\les & \normm{\half\nabb\trc-\kep\c\chi}_{\tx{\infty}{2}}\\
\nn&\les & \normm{\nabb\trc}_{\tx{\infty}{2}}+\no(\kep)\no(\chi)\\
\nn&\les &\ep,
\eea
where we used the estimate \eqref{estk} for $\kep$, and the estimates \eqref{esttrc} \eqref{esthch} for $\chi$ in the last inequality.

\paragraph{Estimate for $\norm{\ddb_{L}\chi_1}_{\tx{\infty}{2}+\tx{2}{q}}$.}

Next, we estimate $\ddb_L\chi_1$ and $\ddb_{\lb}\chi_1$. Note first that for any vectorfield $X$ on $\mathcal{M}$, we have:
$$[\ddb_X,\mathcal{D}_2^{-1}]=\mathcal{D}_2^{-1}[\ddb_X,\mathcal{D}_2]\mathcal{D}_2^{-1}$$
which together with the definition of $\chi_1$ implies:
\be\lab{po19}
\ddb_X\chi_1=\mathcal{D}_2^{-1}\left(\ddb_X\left(\half\nabb\trc-\kep\c\chi\right)\right)+\mathcal{D}_2^{-1}[\ddb_X,\mathcal{D}_2]\chi_1.
\ee
Let $2\leq q<4$. Applying \eqref{po19} with $X=nL$, we obtain:
\bea\lab{po20}
&&\norm{\ddb_{nL}\chi_1}_{\tx{\infty}{2}+\tx{2}{q}}\\
\nn&\les& \norm{\mathcal{D}_2^{-1}\left(\ddb_{nL}\nabb\trc\right)}_{\tx{\infty}{2}}+\norm{\mathcal{D}_2^{-1}\left(\ddb_{nL}\left(\kep\c\chi\right)\right)}_{\tx{2}{q}}+\norm{\mathcal{D}_2^{-1}[\ddb_{nL},\mathcal{D}_2]\chi_1}_{\tx{2}{q}}.
\eea
We estimate the three terms in the right-hand side of \eqref{po20} starting with the first one. Using the commutator formula \eqref{comm5} for $[\ddb_{nL},\nabb]\trc$, and Remark \ref{rem:Dcont} and the dual of \eqref{eq:estimdcal-1} for $\mathcal{D}_2^{-1}$, we obtain:
\bea
\nn\norm{\mathcal{D}_2^{-1}\left(\ddb_{nL}\nabb\trc\right)}_{\tx{\infty}{2}}&\les& \norm{\mathcal{D}_2^{-1}\left([\ddb_{nL},\nabb]\trc\right)}_{\tx{\infty}{2}}+\norm{\mathcal{D}_2^{-1}\left(\nabb\ddb_{nL}\trc\right)}_{\tx{\infty}{2}}\\
\nn&\les&\norm{[\ddb_{nL},\nabb]\trc}_{\tx{\infty}{\frac{4}{3}}}+\norm{nL(\trc)}_{\tx{\infty}{2}}\\
\nn&\les&\norm{n\chi\nabb\trc}_{\tx{\infty}{\frac{4}{3}}}+\ep\\
\nn&\les&\norm{n}_{L^\infty}\norm{\chi}_{\tx{\infty}{4}}\norm{\nabb\trc}_{\tx{\infty}{2}}+\ep\\
\lab{po21}&\les&\ep,
\eea
where we used the estimate \eqref{estn} for $n$, and the estimates \eqref{esttrc} \eqref{esthch} for $\chi$.

Next, we estimate the second term in the right-hand side of \eqref{po20}. Using Lemma \ref{lemma:lbt6}, and since $2\leq q<4$, we obtain:
\bea
\nn\norm{\mathcal{D}_2^{-1}\left(\ddb_{nL}\left(\kep\c\chi\right)\right)}_{\tx{2}{q}}&\les&\norm{\mathcal{D}_2^{-1}\left(\ddb_{nL}(\kep)\c\chi\right)}_{\tx{2}{q}}+\norm{\mathcal{D}_2^{-1}\left(\kep\c\ddb_{nL}(\chi)\right)}_{\tx{2}{q}}\\
\nn&\les&\norm{\ddb_{nL}(\kep)\c\chi}_{\tx{2}{\frac{4}{3}}}+\norm{\kep\c\ddb_{nL}(\chi)}_{\tx{2}{\frac{4}{3}}}\\
\nn&\les&\norm{\ddb_{nL}(\kep)}_{\lh{2}}\norm{\chi}_{\tx{\infty}{4}}+\norm{\kep}_{\tx{\infty}{4}}\norm{\ddb_{nL}(\chi)}_{\lh{2}}\\
\nn&\les&\norm{n}_{L^\infty}\no(\kep)\no(\chi)\\
\lab{po22}&\les&\ep,
\eea
where we used the estimate \eqref{estn} for $n$, the estimate \eqref{estk} for $\kep$, and the estimates \eqref{esttrc} \eqref{esthch} for $\chi$.

Finally, we estimate the third term in the right-hand side of \eqref{po20}. Using the commutator formula \eqref{comm5} for $[\ddb_{nL},\mathcal{D}_2]\trc$, using Lemma \ref{lemma:lbt6}, and since $2\leq q<4$, we obtain:
\bea
\nn\norm{\mathcal{D}_2^{-1}[\ddb_{nL},\mathcal{D}_2]\chi_1}_{\tx{2}{q}}&\les&\norm{[\ddb_{nL},\mathcal{D}_2]\chi_1}_{\tx{2}{\frac{4}{3}}}\\
\nn&\les&\norm{n\chi\nabb\chi_1}_{\tx{2}{\frac{4}{3}}}+\norm{n\chi\kep\chi_1}_{\tx{2}{\frac{4}{3}}}+\norm{n\b\chi_1}_{\tx{2}{\frac{4}{3}}}\\
\nn&\les&\norm{n}_{L^\infty}\bigg(\norm{\chi}_{\tx{\infty}{4}}\norm{\nabb\chi_1}_{\lh{2}}+\norm{\chi}_{\tx{\infty}{4}}\norm{\kep}_{\lh{4}}\norm{\chi_1}_{\lh{4}}\\
\nn&&+\norm{\b}_{\lh{2}}\norm{\chi_1}_{\tx{\infty}{4}}\bigg)\\
\lab{po23}&\les&\ep+\ep\no(\chi_1),
\eea
where we used the curvature bound \eqref{curvflux1} for $\b$, the estimate \eqref{estn} for $n$, the estimate \eqref{estk} for $\kep$, and the estimates \eqref{esttrc} \eqref{esthch} for $\chi$. Now, \eqref{po20}-\eqref{po23} yield:
$$\norm{\ddb_{nL}\chi_1}_{\tx{\infty}{2}+\tx{2}{q}}\les \ep+\ep\no(\chi_1),$$
which together with the bound \eqref{estn} on $n$ and the bound \eqref{po18} on $\nabb\chi_1$ yields:
\be\lab{po24}
\norm{\ddb_{L}\chi_1}_{\tx{\infty}{2}+\tx{2}{q}}\les \ep.
\ee

\paragraph{Estimate for $\norm{\ddb_{\lb}\chi_1}_{\tx{\infty}{2}+\tx{2}{q}}$.}

Next, we estimate $\ddb_{\lb}\chi_1$. Let $2\leq q<4$. Applying \eqref{po19} with $X=bN$, we obtain:
\bea\lab{po25}
&&\norm{\ddb_{bN}\chi_1}_{\tx{\infty}{2}+\tx{2}{q}}\\
\nn&\les& \norm{\mathcal{D}_2^{-1}\left(\ddb_{bN}\nabb\trc\right)}_{\tx{\infty}{2}}+\norm{\mathcal{D}_2^{-1}\left(\ddb_{bN}\left(\kep\c\chi\right)\right)}_{\tx{2}{q}}+\norm{\mathcal{D}_2^{-1}[\ddb_{bN},\mathcal{D}_2]\chi_1}_{\tx{2}{q}}.
\eea
We estimate the three terms in the right-hand side of \eqref{po25} starting with the first one. 
Using the commutator formula \eqref{comm7} for $[\ddb_{bN},\nabb]\trc$, and Remark \ref{rem:Dcont} and the dual of \eqref{eq:estimdcal-1} for $\mathcal{D}_2^{-1}$, we obtain:
\bea
\nn\norm{\mathcal{D}_2^{-1}\left(\ddb_{bN}\nabb\trc\right)}_{\tx{\infty}{2}}&\les& \norm{\mathcal{D}_2^{-1}\left([\ddb_{bN},\nabb]\trc\right)}_{\tx{\infty}{2}}+\norm{\mathcal{D}_2^{-1}\left(\nabb\ddb_{bN}\trc\right)}_{\tx{\infty}{2}}\\
\nn&\les&\norm{[\ddb_{bN},\nabb]\trc}_{\tx{\infty}{\frac{4}{3}}}+\norm{\ddb_{bN}\trc}_{\tx{\infty}{2}}\\
\nn&\les&\norm{n(\chi+\eta)\nabb\trc}_{\tx{\infty}{\frac{4}{3}}}+\ep\\
\nn&\les&\norm{n}_{L^\infty}(\norm{\chi}_{\tx{\infty}{4}}+\norm{k}_{\tx{\infty}{4}})\norm{\nabb\trc}_{\tx{\infty}{2}}+\ep\\
\lab{po26}&\les&\ep,
\eea
where we used the estimate \eqref{estn} for $n$, the estimate \eqref{estk} for $k$, and the estimates \eqref{esttrc} \eqref{esthch} for $\chi$.

Next, we estimate the second term in the right-hand side of \eqref{po25}. Using Lemma \ref{lemma:lbt6}, and since $2\leq q<4$, we obtain:
\bea
\nn\norm{\mathcal{D}_2^{-1}\left(\ddb_{bN}\left(\kep\c\chi\right)\right)}_{\tx{2}{q}}&\les&\norm{\mathcal{D}_2^{-1}\left(\ddb_{bN}(\kep)\c\chi\right)}_{\tx{2}{q}}+\norm{\mathcal{D}_2^{-1}\left(\kep\c\ddb_{bN}(\chi)\right)}_{\tx{2}{q}}\\
\nn&\les&\norm{\ddb_{bN}(\kep)\c\chi}_{\tx{2}{\frac{4}{3}}}+\norm{\kep\c\ddb_{bN}(\chi)}_{\tx{2}{\frac{4}{3}}}\\
\nn&\les&\norm{\ddb_{bN}(\kep)}_{\lh{2}}\norm{\chi}_{\tx{\infty}{4}}+\norm{\kep}_{\tx{\infty}{4}}\norm{\ddb_{bN}(\chi)}_{\lh{2}}\\
\nn&\les&\norm{b}_{L^\infty}(\no(\kep)\norm{\ddb_{N}(\chi)}_{\lh{2}}+\no(\chi)\norm{\ddb_{N}(\kep)}_{\lh{2}})\\
\lab{po27}&\les&\ep,
\eea
where we used the estimate \eqref{estk} for $\kep$, the estimate \eqref{estb} for $b$, and the estimates \eqref{esttrc} \eqref{esthch} for $\chi$.

Finally, we estimate the third term in the right-hand side of \eqref{po25}. Using the commutator formula \eqref{comm7} for $[\ddb_{bN},\mathcal{D}_2]\trc$, and using Lemma \ref{lemma:lbt6}, and since $2\leq q<4$, we obtain:
\bea\lab{po28}
&&\norm{\mathcal{D}_2^{-1}[\ddb_{bN},\mathcal{D}_2]\chi_1}_{\tx{2}{q}}\\
\nn&\les&\norm{[\ddb_{bN},\mathcal{D}_2]\chi_1}_{\tx{2}{\frac{4}{3}}}\\
\nn&\les&\norm{b(\chi+\eta)\nabb\chi_1}_{\tx{2}{\frac{4}{3}}}+\norm{b\chi(\kepb+\xib)\chi_1}_{\tx{2}{\frac{4}{3}}}+\norm{b\chb\z\chi_1}_{\tx{2}{\frac{4}{3}}}+\norm{b\b\chi_1}_{\tx{2}{\frac{4}{3}}}\\
\nn&&+\norm{b\bb\chi_1}_{\tx{2}{\frac{4}{3}}}\\
\nn&\les&\norm{b}_{L^\infty}\bigg(\norm{\chi+\eta}_{\tx{\infty}{4}}\norm{\nabb\chi_1}_{\lh{2}}+\norm{\chi}_{\tx{\infty}{4}}\norm{\kepb+\xib}_{\lh{4}}\norm{\chi_1}_{\lh{4}}\\
\nn&&+\norm{\chb}_{\tx{\infty}{4}}\norm{\z}_{\lh{4}}\norm{\chi_1}_{\lh{4}}+\norm{\b}_{\lh{2}}\norm{\chi_1}_{\tx{\infty}{4}}+\norm{\bb}_{\lh{2}}\norm{\chi_1}_{\tx{\infty}{4}}\bigg)\\
\nn&\les&\ep+\ep\no(\chi_1),
\eea
where we used the curvature bound \eqref{curvflux1} for $\b$ and $\bb$, and the estimates \eqref{estn}-\eqref{estzeta} for $b, \kepb, \eta, \chi, \xib$ and $\z$. Now, \eqref{po25}-\eqref{po28} yield:
$$\norm{\ddb_{bN}\chi_1}_{\tx{\infty}{2}+\tx{2}{q}}\les \ep+\ep\no(\chi_1),$$
which together with the bound \eqref{estb} on $b$, the fact that $\lb=L-2N$, and the bound \eqref{po18} and \eqref{po24} on $\nabb\chi_1$ yields:
\be\lab{po29}
\norm{\ddb_{\lb}\chi_1}_{\tx{\infty}{2}+\tx{2}{q}}\les \ep.
\ee

\paragraph{Estimate for $\norm{\chi_1}_{\tx{p}{\infty}}$.} Using the property \eqref{eq:partition} of the Littlewood-Paley projections, we have:
\bea\lab{po30}
&&\norm{\chi_1}_{\lpt{\infty}}\\
\nn&\les& \sum_{j,l}\normm{P_l\mathcal{D}_2^{-1}P_j\left(\half\nabb\trc-\kep\c\chi\right)}_{\lpt{\infty}}\\
\nn&\les& \sum_{j,l}2^l(1+2^{-\frac{l}{q}}\norm{K}^{\frac{1}{q}}_{\lpt{2}}+2^{-\frac{l}{q-1}}\norm{K}^{\frac{1}{q-1}}_{\lpt{2}})\normm{P_l\mathcal{D}_2^{-1}P_j\left(\half\nabb\trc-\kep\c\chi\right)}_{\lpt{2}}\\
\nn&\les& (1+\norm{K}^{\frac{1}{q-1}}_{\lpt{2}})\sum_{j,l}2^l\normm{P_l\mathcal{D}_2^{-1}P_j\left(\half\nabb\trc-\kep\c\chi\right)}_{\lpt{2}}
\eea
where $2\leq q<+\infty$ will be chosen later, and where we used the sharp Bernstein inequality \eqref{eq:strongberntensor} for tensors. Next, we estimate the right-hand side of \eqref{po30}. Using the finite band property for $P_j$, and the inequality \eqref{eq:estimdcal-1} for ${}^*\mathcal{D}_2$, we have:
\be\lab{dodane}
\norm{\mathcal{D}_2^{-1}P_j}_{\mathcal{L}(\lpt{2})}=\norm{P_j{}^*\mathcal{D}_2^{-1}}_{\mathcal{L}(\lpt{2})}\les 2^{-j}\norm{\nabb{}^*\mathcal{D}_2^{-1}}_{\mathcal{L}(\lpt{2})}\les 2^{-j}
\ee
which together with the boundedness on $L^2$ of $P_l$ yields:
\bea\lab{po31}\normm{P_l\mathcal{D}_2^{-1}P_j\left(\half\nabb\trc-\kep\c\chi\right)}_{\lpt{2}}&\les & 
\normm{\mathcal{D}_2^{-1}P_j\left(\half\nabb\trc-\kep\c\chi\right)}_{\lpt{2}}\\
\nn&\les & 2^{-j}\normm{P_j\left(\half\nabb\trc-\kep\c\chi\right)}_{\lpt{2}}.
\eea

We now derive second estimate for $\normm{P_l\mathcal{D}_2^{-1}P_j\left(\half\nabb\trc-\kep\c\chi\right)}_{\lpt{2}}$. Using the finite band property for $P_l$, we have:
\be\lab{po32:first}
\normm{P_l\mathcal{D}_2^{-1}P_j\left(\half\nabb\trc-\kep\c\chi\right)}_{\lpt{2}}\les 2^{-2l}\normm{\lap\mathcal{D}_2^{-1}P_j\left(\half\nabb\trc-\kep\c\chi\right)}_{\lpt{2}}.
\ee
Next, we estimate the right-hand side of \eqref{po32:first}. In view of the identity \eqref{eq:dcallident} for $\mathcal{D}_2$, we have:
\bea
\lab{po32}&&\normm{\lap\mathcal{D}_2^{-1}P_j\left(\half\nabb\trc-\kep\c\chi\right)}_{\lpt{2}}\\
\nn&\les& \normm{{}^*\mathcal{D}_2P_j\left(\half\nabb\trc-\kep\c\chi\right)}_{\lpt{2}}+\normm{K\mathcal{D}_2^{-1}P_j\left(\half\nabb\trc-\kep\c\chi\right)}_{\lpt{2}}.
\eea
We now estimate both terms in the right-hand side of \eqref{po32} starting with the first one. Using the $L^2$ boundedness for $P_l$ and the finite band property for $P_j$, we have:
\be
\lab{po33}\normm{{}^*\mathcal{D}_2P_j\left(\half\nabb\trc-\kep\c\chi\right)}_{\lpt{2}}\les 2^j\normm{P_j\left(\half\nabb\trc-\kep\c\chi\right)}_{\lpt{2}}.
\ee
Next, we estimate the second term in the right-hand side of \eqref{po32}. We have:
\be
\lab{po34}\normm{K\mathcal{D}_2^{-1}P_j\left(\half\nabb\trc-\kep\c\chi\right)}_{\lpt{2}}\les  \norm{K}_{\lpt{2}}\normm{\mathcal{D}_2^{-1}P_j\left(\half\nabb\trc-\kep\c\chi\right)}_{\lpt{\infty}}.
\ee
In order to estimate the $\lpt{\infty}$ norm in the right-hand side of \eqref{po34}, we use the estimate \eqref{linftynormtensor}. This yields
\bea\lab{pimpampoum}
&&\normm{\mathcal{D}_2^{-1}P_j\left(\half\nabb\trc-\kep\c\chi\right)}_{\lpt{\infty}}\\
\nn&\les& \normm{\nabb^2\mathcal{D}_2^{-1}P_j\left(\half\nabb\trc-\kep\c\chi\right)}^\frac{1}{2}_{\lpt{2}}\normm{\mathcal{D}_2^{-1}P_j\left(\half\nabb\trc-\kep\c\chi\right)}^\frac{1}{2}_{\lpt{2}}\\
\nn&&+\normm{\nabb\mathcal{D}_2^{-1}P_j\left(\half\nabb\trc-\kep\c\chi\right)}_{\lpt{2}}\\
\nn&\les& 2^{-\frac{j}{2}}\normm{\nabb^2\mathcal{D}_2^{-1}P_j\left(\half\nabb\trc-\kep\c\chi\right)}^\frac{1}{2}_{\lpt{2}}\normm{P_j\left(\half\nabb\trc-\kep\c\chi\right)}^\frac{1}{2}_{\lpt{2}}\\
\nn&&+\normm{P_j\left(\half\nabb\trc-\kep\c\chi\right)}_{\lpt{2}},
\eea
where we used in the last inequality \eqref{dodane}, the estimate \eqref{eq:estimdcal-1} for $\nabb\mathcal{D}_2^{-1}$, and the boundedness of $P_j$ on $\lpt{2}$. In order to estimate the first term in the right-hand side of \eqref{pimpampoum}, we use the Bochner inequality for tensors \eqref{vbochineq}. This yields
\bea\lab{pimpampoum1}
&&\normm{\nabb^2\mathcal{D}_2^{-1}P_j\left(\half\nabb\trc-\kep\c\chi\right)}_{\lpt{2}}\\
\nn&\les& \normm{\lap\mathcal{D}_2^{-1}P_j\left(\half\nabb\trc-\kep\c\chi\right)}_{\lpt{2}}+\norm{K}_{\lpt{2}}\normm{\nabb\mathcal{D}_2^{-1}P_j\left(\half\nabb\trc-\kep\c\chi\right)}_{\lpt{2}}\\
\nn&&+\norm{K}^2_{\lpt{2}}\normm{\mathcal{D}_2^{-1}P_j\left(\half\nabb\trc-\kep\c\chi\right)}_{\lpt{2}}\\
\nn&\les& \normm{\lap\mathcal{D}_2^{-1}P_j\left(\half\nabb\trc-\kep\c\chi\right)}_{\lpt{2}}+\norm{K}_{\lpt{2}}\normm{P_j\left(\half\nabb\trc-\kep\c\chi\right)}_{\lpt{2}}\\
\nn&&+2^{-j}\norm{K}^2_{\lpt{2}}\normm{P_j\left(\half\nabb\trc-\kep\c\chi\right)}_{\lpt{2}}, 
\eea
where we used in the last inequality \eqref{dodane} and the estimate \eqref{eq:estimdcal-1} for $\nabb\mathcal{D}_2^{-1}$. Now, \eqref{po32}, \eqref{po33}, \eqref{po34}, \eqref{pimpampoum} and \eqref{pimpampoum1} imply
\bea\lab{pimpampoum2}
&&\normm{\nabb^2\mathcal{D}_2^{-1}P_j\left(\half\nabb\trc-\kep\c\chi\right)}_{\lpt{2}}\\
\nn&\les& \left(2^j+\norm{K}_{\lpt{2}}+2^{-j}\norm{K}^2_{\lpt{2}}\right)\normm{P_j\left(\half\nabb\trc-\kep\c\chi\right)}_{\lpt{2}}.
\eea
Then, \eqref{po32:first} and \eqref{pimpampoum2} yield
\bea\lab{pimpampoum3}
&&\normm{P_l\mathcal{D}_2^{-1}P_j\left(\half\nabb\trc-\kep\c\chi\right)}_{\lpt{2}}\\
\nn&\les& 2^{-2l}\left(2^j+\norm{K}_{\lpt{2}}+2^{-j}\norm{K}^2_{\lpt{2}}\right)\normm{P_j\left(\half\nabb\trc-\kep\c\chi\right)}_{\lpt{2}}.
\eea
Also, using the finite band property $P_l$ and the estimate \eqref{eq:estimdcal-1} for $\nabb\mathcal{D}_2^{-1}$, we have
$$\normm{P_l\mathcal{D}_2^{-1}P_j\left(\half\nabb\trc-\kep\c\chi\right)}_{\lpt{2}}\les 2^{-l}\normm{P_j\left(\half\nabb\trc-\kep\c\chi\right)}_{\lpt{2}}.$$
Interpolating with \eqref{pimpampoum3}, we obtain for any $0\leq \delta\leq 1$
\bea\lab{pimpampoum4}
&&\normm{P_l\mathcal{D}_2^{-1}P_j\left(\half\nabb\trc-\kep\c\chi\right)}_{\lpt{2}}\\
\nn&\les& 2^{-l}2^{-\delta(l-j)}\left(1+\norm{K}^2_{\lpt{2}}\right)^\delta\normm{P_j\left(\half\nabb\trc-\kep\c\chi\right)}_{\lpt{2}}.
\eea

In view of \eqref{po30}, and using \eqref{po31} for $l\leq j$ and \eqref{pimpampoum4} for $l>j$, we obtain for any $2\leq q<+\infty$ and any $0<\delta\leq 1$
\bee
\norm{\chi_1}_{\lpt{\infty}}&\les& (1+\norm{K}_{\lpt{2}})^{\frac{1}{q-1}+2\delta}\sum_{j,l}2^{-\delta |l-j|}\normm{P_j\left(\half\nabb\trc-\kep\c\chi\right)}_{\lpt{2}}\\
&\les& (1+\norm{K}_{\lpt{2}})^{\frac{1}{q-1}+2\delta}\normm{\half\nabb\trc-\kep\c\chi}_{B^0_{2,1}(\ptu)}
\eee
where we used in the last inequality the fact that $\delta>0$ and the definition \eqref{besovptu} for the Besov space $B^0_{2,1}(\ptu)$. This yields
\bea\lab{po39}
\norm{\chi_1}_{\lpt{\infty}}&\les& (1+\norm{K}_{\lpt{2}})^{\frac{1}{q-1}+2\delta}(\norm{\nabb\trc}_{\BB^0}+\norm{\kep\c\chi}_{B^0_{2,1}(\ptu)})\\
\nn&\les& (1+\norm{K}_{\lpt{2}})^{\frac{1}{q-1}+2\delta}(\ep+\norm{\kep\c\chi}_{B^0_{2,1}(\ptu)})
\eea
where we used the Besov improvement \eqref{impbes21} for $\nabb\trc$. Let $2\leq p<+\infty$. We choose $2\leq q<+\infty$ and $0<\delta\leq 1$ such that 
$$\frac{1}{q-1}+2\delta=\frac{2}{p}.$$
Then, \eqref{po39} implies:
\bea
\lab{po40}\norm{\chi_1}_{\tx{p}{\infty}}&\les& \left(1+\norm{K}^{\frac{2}{p}}_{\lh{2}}\right)(\ep+\norm{\kep\c\chi}_{L^{2p}_tB^0_{2,1}(\ptu)})\\
\nn&\les& \ep+\norm{\kep\c\chi}_{L^{2p}_tB^0_{2,1}(\ptu)},
\eea
where we used the estimate \eqref{estgauss1} for the Gauss curvature $K$. We now conclude using the following lemma:
\begin{lemma}\lab{lemma:po41}
Let $F, H$ two $\ptu$-tangent tensors. For any $2\leq r<+\infty$, we have:
\be\lab{po41}
\norm{F\c H}_{L^r_tB^0_{2,1}(\ptu)}\les \no(F)\no(H).
\ee
\end{lemma}

The proof of Lemma \ref{lemma:po41} is postponed to section \ref{sec:gowinda2}. We now derive the estimate for $\norm{\chi_1}_{\tx{p}{\infty}}$. \eqref{po40}, and \eqref{po41} with $r=2p$, $F=\kep$ and $H=\chi$ yield:
\bea
\lab{po42}\norm{\chi_1}_{\tx{p}{\infty}}&\les& \ep+\no(\kep)\no(\chi)\\
\nn&\les&\ep,
\eea
where we used the estimate \eqref{estk} for $\kep$ and the estimates \eqref{esttrc} \eqref{esthch} for $\chi$. \eqref{po42} is the desired estimate for $\norm{\chi_1}_{\tx{p}{\infty}}$.

\paragraph{Estimate for $\norm{\chi_1}_{L^\infty_uL^p_t\mathcal{B}^0_{2,1}(\ptu)}$.}

We will need later on an estimate for $\chi_1$ in $L^\infty_uL^p_t\mathcal{B}^0_{2,1}(\ptu)$. We proceed as for the estimate of $\chi_1$ in $\tx{p}{\infty}$. In view of the definition \eqref{besovptu} of the Besov space $\mathcal{B}^0_{2,1}(\ptu)$, we have
\bea\lab{bispo30}\norm{\chi_1}_{\mathcal{B}^0_{2,1}(\ptu)}&\les& \sum_{j,l}\normm{P_l\nabb\mathcal{D}_2^{-1}P_j\left(\half\nabb\trc-\kep\c\chi\right)}_{\lpt{2}}.
\eea
Next, we estimate the right-hand side of \eqref{bispo30}. 
The finite band property for $P_l$ together with the estimate \eqref{dodane} yields
\bea\lab{bispo31}\normm{P_l\nabb\mathcal{D}_2^{-1}P_j\left(\half\nabb\trc-\kep\c\chi\right)}_{\lpt{2}}&\les & 
2^l\normm{\mathcal{D}_2^{-1}P_j\left(\half\nabb\trc-\kep\c\chi\right)}_{\lpt{2}}\\
\nn&\les & 2^{l-j}\normm{P_j\left(\half\nabb\trc-\kep\c\chi\right)}_{\lpt{2}}.
\eea

We now derive second estimate for $\normm{P_l\nabb\mathcal{D}_2^{-1}P_j\left(\half\nabb\trc-\kep\c\chi\right)}_{\lpt{2}}$. Using the finite band property for $P_l$ and the estimate \eqref{pimpampoum2}, we have:
\bea\lab{bispimpampoum3}
&&\normm{P_l\nabb\mathcal{D}_2^{-1}P_j\left(\half\nabb\trc-\kep\c\chi\right)}_{\lpt{2}}\\
\nn&\les& 2^{-l}\normm{\nabb^2\mathcal{D}_2^{-1}P_j\left(\half\nabb\trc-\kep\c\chi\right)}_{\lpt{2}}\\
\nn&\les& 2^{-l+j}(1+\norm{K}^2_{\lpt{2}})\normm{P_j\left(\half\nabb\trc-\kep\c\chi\right)}_{\lpt{2}}.
\eea
Also, using the boundedness of $P_l$ on $\lpt{2}$ and the estimate \eqref{eq:estimdcal-1} for $\nabb\mathcal{D}_2^{-1}$, we have
$$\normm{P_l\nabb\mathcal{D}_2^{-1}P_j\left(\half\nabb\trc-\kep\c\chi\right)}_{\lpt{2}}\les \normm{P_j\left(\half\nabb\trc-\kep\c\chi\right)}_{\lpt{2}}.$$
Interpolating with \eqref{pimpampoum3}, we obtain for any $0\leq \delta\leq 1$
\bea\lab{bispimpampoum4}
&&\normm{P_l\nabb\mathcal{D}_2^{-1}P_j\left(\half\nabb\trc-\kep\c\chi\right)}_{\lpt{2}}\\
\nn&\les& 2^{-\delta(l-j)}\left(1+\norm{K}^2_{\lpt{2}}\right)^\delta\normm{P_j\left(\half\nabb\trc-\kep\c\chi\right)}_{\lpt{2}}.
\eea

In view of \eqref{bispo30}, and using \eqref{bispo31} for $l\leq j$ and \eqref{bispimpampoum4} for $l>j$, we obtain for any $0<\delta\leq 1$
\bee
\norm{\chi_1}_{\mathcal{B}^0_{2,1}(\ptu)}&\les& (1+\norm{K}_{\lpt{2}})^{2\delta}\sum_{j,l}2^{-\delta |l-j|}\normm{P_j\left(\half\nabb\trc-\kep\c\chi\right)}_{\lpt{2}}\\
&\les& (1+\norm{K}_{\lpt{2}})^{2\delta}\normm{\half\nabb\trc-\kep\c\chi}_{B^0_{2,1}(\ptu)}
\eee
where we used in the last inequality the fact that $\delta>0$ and the definition \eqref{besovptu} for the Besov space $B^0_{2,1}(\ptu)$. This yields
\bea\lab{bispo39}
\norm{\chi_1}_{\mathcal{B}^0_{2,1}(\ptu)}&\les& (1+\norm{K}_{\lpt{2}})^{2\delta}(\norm{\nabb\trc}_{\BB^0}+\norm{\kep\c\chi}_{B^0_{2,1}(\ptu)})\\
\nn&\les& (1+\norm{K}_{\lpt{2}})^{2\delta}(\ep+\norm{\kep\c\chi}_{B^0_{2,1}(\ptu)})
\eea
where we used the Besov improvement \eqref{impbes21} for $\nabb\trc$. Let $2\leq p<+\infty$. We choose $0<\delta\leq 1$ such that 
$$2\delta=\frac{1}{p}.$$
Then, \eqref{bispo39} implies:
\bee
\norm{\chi_1}_{L^p_t\mathcal{B}^0_{2,1}(\ptu)}&\les& \left(1+\norm{K}^{\frac{1}{p}}_{\lh{2}}\right)(\ep+\norm{\kep\c\chi}_{L^{2p}_tB^0_{2,1}(\ptu)})\\
\nn&\les& \ep+\no(\kep)\no(\chi),
\eee
where we used in the last inequality the estimate \eqref{estgauss1} for the Gauss curvature $K$, and the estimate \eqref{po41}. Together with the estimate \eqref{estk} for $\kep$ and the estimates \eqref{esttrc} \eqref{esthch} for $\chi$, we finally obtain
\be\lab{bispo40}
\norm{\chi_1}_{L^p_t\mathcal{B}^0_{2,1}(\ptu)}\les\ep,
\ee
for any $2\leq p<+\infty$.

\subsubsection{Estimates for $\chi_2$}

In view of the decomposition \eqref{po17}, the estimates \eqref{esttrc} \eqref{esthch} for $\chi$, and the estimates \eqref{po18}, \eqref{po24} and \eqref{po29} for $\chi_1$, we have:
\be\lab{po49}
\no(\chi_2)+\norm{\ddb_{\lb}\chi_2}_{\lh{2}}\les\ep.
\ee
We now compute $\po\chi_2$. We have:
$$[\po,\mathcal{D}_2^{-1}]=\mathcal{D}_2^{-1}[\po,\mathcal{D}_2]\mathcal{D}_2^{-1}$$
which together with the definition of $\chi_2$ implies:
\bea\lab{po50}
\po\chi_2&=&-\mathcal{D}_2^{-1}\Pi(\po\b+\po\kep\c\chi+\kep\c\po\chi)+\mathcal{D}_2^{-1}[\Pi\po,\mathcal{D}_2]\chi_2\\
\nn&=&-\mathcal{D}_2^{-1}\left(\frac{(\po N)_B}{2}(-\a_{AB}+\r\d_{AB}+3\s\in_{AB})+\po\kep\c\chi+\kep\c\po\chi\right)\\
\nn&&+\mathcal{D}_2^{-1}\bigg(-\ddb_N(\chi_2)_{A\po N}+g(\po N,e_A)\th\c \chi_2+\th_{\po NB}(\chi_2)_{BA}-\th_{AB}(\chi_2)_{B\po N}\\
\nn&&-\trt (\chi_2)_{A\po N}\bigg),
\eea
where we used the formula \eqref{po4} for $\po\b$ and the commutator formula \eqref{commo7} for $[\Pi\po,\mathcal{D}_2]$. In particular, using the property \eqref{eq:estimdcal-1} of $\mathcal{D}_2^{-1}$, we have the following estimate for $\nabb\po\chi_2$:
\bea
&&\lab{po50bis}\norm{\nabb\po\chi_2}_{\lh{2}}\\
\nn&\les &\normm{\frac{(\po N)_B}{2}(-\a_{AB}+\r\d_{AB}+3\s\in_{AB})+\po\kep\c\chi+\kep\c\po\chi}_{\lh{2}}+\|-\ddb_N(\chi_2)_{A\po N}\\
\nn&&+g(\po N,e_A)\th\c \chi_2+\th_{\po NB}(\chi_2)_{BA}-\th_{AB}(\chi_2)_{B\po N}-\trt (\chi_2)_{A\po N}\|_{\lh{2}}\\
\nn&\les& \norm{\po N}_{L^\infty}\bigg(\norm{\a}_{\lh{2}}+\norm{\r}_{\lh{2}}+\norm{\s}_{\lh{2}}+\norm{\po\kep}_{\xt{2}{\infty}}\norm{\chi}_{\xt{\infty}{2}}\\
\nn&&+\norm{\po\chi}_{\xt{2}{\infty}}\norm{\kep}_{\xt{\infty}{2}}+\norm{\ddb_N\chi_2}_{\lh{2}}+\no(\th)\no(\chi_2)\bigg)\\
\nn&\les& \ep
\eea
where we used the curvature bound \eqref{curvflux1} for $\a, \r$ and $\s$, the estimates \eqref{estk}  \eqref{esttrc} \eqref{esthch} for $\kep, \chi$ and $\th$, the estimate \eqref{estNomega} for $\po N$, the estimate \eqref{estricciomega} for $\po\chi$ and the estimate \eqref{po49} for $\chi_2$.

Next, we plan to estimate the $\tx{p}{4_-}$-norm of $\po\chi_2$ for $2<p<+\infty$. Our goal will be first to show that the terms involving $\a$ in $\po\chi_2$ cancel each other. Applying \eqref{po19} to $\chi_2$ with the choice $X=bN$ yields:
\be\lab{po51}
\ddb_{bN}\chi_2=-\mathcal{D}_2^{-1}\left(\ddb_{bN}(\chi\c\kep+\b)\right)+\mathcal{D}_2^{-1}[\ddb_{bN},\mathcal{D}_2]\chi_2.
\ee
In view of \eqref{po51}, we need to evaluate $\ddb_N(\chi\c\kep+\b)$. We have:
$$\ddb_N(\chi\c\kep+\b)=\chi\c\ddb_N(\kep)+\half\left(\ddb_L(\chi)\c\kep+\ddb_L\b-\ddb_{\lb}(\chi)\c\kep-\ddb_{\lb}\b\right)$$
which together with the equation \eqref{D4chi} and \eqref{D3chi} for $\chi$, the Bianchi identities \eqref{bianc1} and \eqref{bianc2} for $\b$, and the last equation of \eqref{hodgkh} for $\ddb_N\kep$ yields:
\bea
\nn\ddb_N(\chi\c\kep+\b)&=&\divb\a+b^{-1}\nabb b\c\a-\nabb\r-(\nabb\s)^*+(\chi-2\d)\b-(\kep+3\z)\r+(\kep-3{}^*\z)\s\\
\nn&&-(\chi+2\hchb)\bb+2\nabb\d\c\chi-2\kep\c\nabb\z+3\db^{-1}\nabb b\c\chi-2b^{-1}\nabb b\etah\chi-2\th\kep\chi\\
\lab{po52}&&-\kep\chi(\d+n^{-1}\nabla_Nn)-\kep\z\otimes\z+\kep\chb\chi.
\eea
\eqref{po51} and \eqref{po52} yield:
\bea
\nn\ddb_{N}\chi_2&=& -\a+b^{-1}\mathcal{D}_2^{-1}b\bigg(\nabb\r+(\nabb\s)^*-(\chi-2\d)\b+(\kep+3\z)\r-(\kep-3{}^*\z)\s\\
\nn&&+(\chi+2\hchb)\bb-2\nabb\d\c\chi+2\kep\c\nabb\z-3\db^{-1}\nabb b\c\chi+2b^{-1}\nabb b\etah\chi+2\th\kep\chi\\
\lab{po53}&&+\kep\chi(\d+n^{-1}\nabla_Nn)+\kep\z\otimes\z-\kep\chb\chi\bigg)+b^{-1}\mathcal{D}_2^{-1}[\ddb_{bN},\mathcal{D}_2]\chi_2.
\eea
Now, in view of \eqref{po50} and \eqref{po53}, the terms in $\a$ cancel each other, and we finally obtain:
\bea
\lab{po54}\po\chi_2 &=&-\mathcal{D}_2^{-1}\left(\frac{(\po N)_B}{2}(\r\d_{AB}+3\s\in_{AB})+\po\kep\c\chi+\kep\c\po\chi\right)\\
\nn&&+\mathcal{D}_2^{-1}\Bigg(-b^{-1}\mathcal{D}_2^{-1}b\bigg(\nabb\r+(\nabb\s)^*-(\chi-2\d)\b+(\kep+3\z)\r-(\kep-3{}^*\z)\s\\
\nn&&+(\chi+2\hchb)\bb-2\nabb\d\c\chi+2\kep\c\nabb\z-3\d b^{-1}\nabb b\c\chi+2b^{-1}\nabb b\etah\chi+2\th\kep\chi\\
\nn&&+\kep\chi(\d+n^{-1}\nabla_Nn)+\kep\z\otimes\z-\kep\chb\chi\bigg)_{A\po N}-b^{-1}\mathcal{D}_2^{-1}[\ddb_{bN},\mathcal{D}_2](\chi_2)_{A\po N}\\
\nn&&+g(\po N,e_A)\th\c \chi_2+\th_{\po NB}(\chi_2)_{BA}-\th_{AB}(\chi_2)_{B\po N}-\trt (\chi_2)_{A\po N}\Bigg),
\eea

We will use the following four Lemmas.
\begin{lemma}\lab{lemma:po1}
Let $f$ a scalar function equal either to $b$ or $1$, let $F$ a $\ptu$-tangent tensor and let $H$ denote a curvature term among $(\r, \s, \b, \bb)$. Then, we have the following estimate:
\be\lab{po55}
\norm{\mathcal{D}_2^{-1}(bF\c H)}_{\tx{\infty}{4_-}}\les (\norm{F}_{L^\infty}+\norm{\nabb F}_{\tx{\infty}{2}})\ep.
\ee
\end{lemma}

\begin{lemma}\lab{lemma:po2}
Let $h$ a scalar function which denotes a curvature term among $(\r, \s)$. Then, for any $2\leq p<+\infty$, we have the following estimate:
\be\lab{po56}
\norm{\mathcal{D}_2^{-1}b^{-1}\mathcal{D}_2^{-1}(b\nabb h)}_{\tx{p}{4_-}}\les\ep.
\ee
\end{lemma}

\begin{lemma}\lab{lemma:po3}
Let $F$ a $\ptu$-tangent tensors and let $H$ denote a term among $(\r, \s, \b, \bb)$ and 
$G$ is a $\ptu$-tensor satisfying $\no(G)\les\ep$. Then, for any $2\leq p<+\infty$, we have the following estimate:
\be\lab{po57}
\norm{\mathcal{D}_2^{-1}b(\mathcal{D}_2^{-1}(F\c H))}_{\tx{p}{4_-}}+\norm{\mathcal{D}_2^{-1}b(\mathcal{D}_2^{-1}(F\c\nabb G))}_{\tx{p}{4_-}}\les\no(F)\ep.
\ee
\end{lemma}

\begin{lemma}\lab{lemma:po4}
Let $F, G$ and $H$ three $\ptu$-tangent tensors. Then, we have the following estimate:
\be\lab{po58}
\norm{\mathcal{D}_2^{-1}(FGH)}_{\tx{\infty}{4_-}}\les\no(F)\no(G)\no(H).
\ee
\end{lemma}

We also state the following lemma which will be necessary for the proof of Lemma \ref{lemma:po1} as well as several places in this paper.
\begin{lemma}\lab{lemma:po5}
Let $H$ denote a curvature term among $(\r, \s, \b, \bb)$. Then, for any $j\geq 0$, we have the following estimate:
\be\lab{po63}
\norm{P_jH}_{\tx{\infty}{2}}\les 2^{\frac{j}{2}}\ep.
\ee
\end{lemma}

The proof of Lemma \ref{lemma:po1} is postponed to section \ref{sec:gowinda3}, the proof of Lemma \ref{lemma:po2} to section \ref{sec:gowinda4}, the proof of Lemma \ref{lemma:po3} to section \ref{sec:gowinda5}, the proof of Lemma \ref{lemma:po4} to section \ref{sec:gowinda6}, and the proof of Lemma \ref{lemma:po5} to section \ref{sec:gowinda7}. We now derive the estimate for the $\tx{p}{4_-}$-norm of $\po\chi_2$. We consider the various terms in the right-hand side of \eqref{po54}. Lemma \ref{lemma:po1}  and Lemma \ref{lemma:po3} yield:
\bea
\nn&&\normm{\mathcal{D}_2^{-1}\left(\frac{(\po N)_B}{2}(\r\d_{AB}+3\s\in_{AB})\right)}_{\tx{\infty}{4_-}}+\Bigg\|\mathcal{D}_2^{-1}\Bigg(-b^{-1}\mathcal{D}_2^{-1}b\bigg(-(\chi-2\d)\b\\
\nn&&+(\kep+3\z)\r-(\kep-3{}^*\z)\s+(\chi+2\hchb)\bb\bigg)_{A\po N}\Bigg)\Bigg\|_{\tx{\infty}{4_-}}\\
\nn&\les& \ep(\norm{\po N}_{L^\infty}+\norm{\nabb\po N}_{\tx{\infty}{2}}+\no(\chi)+\no(\d)+\no(\kep)+\no(\z)+\no(\chb))\\
\lab{po59}&\les& \ep,
\eea
where we used the estimates \eqref{estk}-\eqref{estzeta} for $\d, \kep, \chi, \chb$ and $\z$, and the estimates \eqref{estNomega} \eqref{estricciomega} for $\po N$. 

Using the commutator formula \eqref{comm7} together with Remark \ref{rem:Dcont} for $\mathcal{D}_2^{-1}$ and Lemma \ref{lemma:po3} and Lemma \ref{lemma:po4}, we obtain:
\bea
\nn&&\normm{\mathcal{D}_2^{-1}\Bigg(b^{-1}\mathcal{D}_2^{-1}[\ddb_{bN},\mathcal{D}_2](\chi_2)_{A\po N}\Bigg)}_{\tx{p}{4_-}}\\
\nn&\les& \normm{\mathcal{D}_2^{-1}\Bigg(b^{-1}\mathcal{D}_2^{-1}\nabb(b(\chi+\eta)\chi_2)_{A\po N}\Bigg)}_{\tx{p}{4_-}}+\normm{\mathcal{D}_2^{-1}\Bigg(b^{-1}\mathcal{D}_2^{-1}\nabb(b(\chi+\eta))\chi_2)_{A\po N}\Bigg)}_{\tx{p}{4_-}}\\
\nn&&+\normm{\mathcal{D}_2^{-1}\Bigg(b^{-1}\mathcal{D}_2^{-1}((\chi(\kepb+\xib+\chb\z)\chi_2)_{A\po N}\Bigg)}_{\tx{p}{4_-}}+\normm{\mathcal{D}_2^{-1}\Bigg(b^{-1}\mathcal{D}_2^{-1}((\b+\bb)\chi_2)_{A\po N}\Bigg)}_{\tx{p}{4_-}}\\
\nn&\les& (\noo(b)(\no(\chi)+\no(\eta))\no(\chi_2)+\no(\chi)(\no(\kepb)+\no(\xib))\no(\chi_2)\\
\nn&&+\no(\chb)\no(\z)\no(\chi_2)+\ep\no(\chi_2)\\
\lab{po60}&\les&\ep,
\eea
where we used the estimates \eqref{estn}-\eqref{estzeta} for $b, \eta, \kepb, \chi, \chb, \xib$ and $\z$, the estimate \eqref{estNomega} for $\po N$, and the estimate \eqref{po49} for $\chi_2$.

Using Remark \ref{rem:Dcont} for $\mathcal{D}_2^{-1}$ and Lemma \ref{lemma:po3} and Lemma \ref{lemma:po4}, we obtain:
\bea
\nn&&\normm{\mathcal{D}_2^{-1}\Bigg(b^{-1}\mathcal{D}_2^{-1}b\bigg(-2\nabb\d\c\chi+2\kep\c\nabb\z-3\d b^{-1}\nabb b\c\chi+2b^{-1}\nabb b\etah\chi+2\th\kep\chi\bigg)_{A\po N}\Bigg)}_{\tx{p}{4_-}}\\
\nn&&+\normm{\mathcal{D}_2^{-1}\Bigg(b^{-1}\mathcal{D}_2^{-1}b\bigg(\kep\chi(\d+n^{-1}\nabla_Nn)+\kep\z\otimes\z-\kep\chb\chi\bigg)_{A\po N}\Bigg)}_{\tx{\infty}{4_-}}\\
\nn&&+\normm{\mathcal{D}_2^{-1}\Bigg(g(\po N,e_A)\th\c \chi_2+\th_{\po NB}(\chi_2)_{BA}-\th_{AB}(\chi_2)_{B\po N}-\trt (\chi_2)_{A\po N}\Bigg)}_{\tx{\infty}{4_-}}\\
\nn&\les& \noo(b)\big(\no(\d)\no(\chi)+\no(\kep)\no(\z)+\no(\eta)\no(\chi)+(\no(\th)+\no(\d)+\no(n^{-1}\nabla_Nn)\\
\nn&&+\no(\chb))\no(\kep)\no(\chi)+\no(\kep)\no(\z)^2\big)+\norm{\po N}_{L^\infty}\no(\th)\no(\chi_2)\\
\lab{po61}&\les&\ep,
\eea
where we used the estimates \eqref{estn}-\eqref{estzeta} for $n, b, \eta, \kep, \chi, \chb, \th$ and $\z$, and the estimate \eqref{estNomega} for $\po N$, and the estimate \eqref{po49} for $\chi_2$.

Using the analog of Lemma \ref{lemma:lbt6} for $\mathcal{D}_2^{-1}$, we obtain:
\bea
\lab{po62}&&\norm{\mathcal{D}_2^{-1}\left(\po\kep\c\chi+\kep\c\po\chi\right)}_{\tx{\infty}{4_-}}\\
\nn&\les&\norm{\po\kep\c\chi+\kep\c\po\chi}_{\tx{\infty}{\frac{4}{3}}}\\
\nn&\les&\norm{\po\kep}_{\tx{\infty}{2}}\norm{\chi}_{\tx{\infty}{4}}+\norm{\kep}_{\tx{\infty}{4}}\norm{\po\chi}_{\tx{\infty}{2}}\\
\nn&\les&\ep,
\eea
where we used the estimates \eqref{estn}-\eqref{esthch} for $\kep$ and $\chi$, and the estimate \eqref{estricciomega} for $\po \chi$.

Finally, \eqref{po54}, Lemma \ref{lemma:po2}, and \eqref{po59}-\eqref{po62} yield for all $2\leq p<+\infty$:
$$\norm{\po\chi_2}_{\tx{p}{4_-}}\les \ep.$$
Using the Gagliardo Nirenberg inequality \eqref{eq:GNirenberg}, \eqref{po50bis} and interpolation, we obtain:
$$\norm{\po\chi_2}_{\lh{6_-}}\les\ep.$$
Together with the estimates \eqref{po18}, \eqref{po24}, \eqref{po29} and \eqref{po42} for $\chi_1$, and  the estimates \eqref{po49}, \eqref{po50bis} for $\chi_2$, we obtain the desired decomposition \eqref{dechch}-\eqref{dechch3} for $\hch$

\subsection{Besov improvement for $\po N$ and $\po\chi$}\lab{sec:jycroispas3}

The goal of this section in to prove the following proposition.
\begin{proposition}
We have the following estimate:
\be\lab{pope}
\norm{\nabb\po N}_{\BB^0}+\norm{\Pi(\po\chi)}_{\BB^0}\les\ep.
\ee
\end{proposition}

\begin{proof}
The formula \eqref{popo3} for $\dd_A(\po N)$ yields:
$$\nabb_A\po N=\po\chi_{AB}e_B-g(\po N,e_A)\z_Be_B,$$
which together with the estimate \eqref{eq:secondbilBesov:bis} and the non sharp embedding \eqref{nonsharpembedding} yields:
\bea
\lab{pope1}\norm{\nabb\po N}_{\BB^0}&\les& \norm{\Pi(\po\chi)}_{\BB^0}+\norm{\po N\c\z}_{\BB^0}\\
\nn&\les &\norm{\Pi(\po\chi)}_{\BB^0}+(\norm{\nabb\po N}_{\tx{\infty}{2}}+\norm{\po N}_{L^\infty})\norm{\z}_{\BB^0}\\
\nn&\les &\norm{\Pi(\po\chi)}_{\BB^0}+(\norm{\nabb\po N}_{\tx{\infty}{2}}+\norm{\po N}_{L^\infty})\no(\z)\\
\nn&\les &\norm{\Pi(\po\chi)}_{\BB^0}+\ep,
\eea
where we used in the last inequality the estimate \eqref{estzeta} for $\z$ and the estimates \eqref{estNomega} and \eqref{estricciomega} for $\po N$. 

In view of \eqref{pope1}, it remains to estimate $\norm{\Pi(\po\chi)}_{\BB^0}$. We recall the structure of the transport equation \eqref{popo4} satisfied by $\Pi(\po\chi)$:
\be\lab{pope2}\ddb_L(\Pi(\po\chi))=-\nabb_{\po N}\chi-(2\chi+\db)\c\Pi(\po\chi)+(4\kepb-2\kep+n^{-1}\nabla n)\c\chi\c\po N+\po N\c\b.
\ee
Recall from \eqref{estfluxcorres1:0} and \eqref{impbes22} the following decompositions:
$$n\b=\ddb_{nL}P_1+E_1,\,\nabb(n\hch)=\ddb_{nL}P_2+E_2\textrm{ where }\no(P_j)+\norm{E_j}_{\PP^0}\les \ep\textrm{ for }j=1, 2.$$
Together with \eqref{pope2}, this yields:
\be\lab{pope3}\ddb_{nL}(\Pi(\po\chi))=-(2\chi+\db)\c\Pi(\po\chi)+F\c\ddb_{nL}(P)+F\c E,
\ee
where $F$, $P$ and $E$ are given respectively by:
$$F=n\po N,\, P=-P_1+P_2,$$
and
$$E=-E_1+E_2+(4\kepb-2\kep+n^{-1}\nabla n)\c\chi.$$
$F$ satisfies:
\be\lab{pope4}
\no(F)+\norm{F}_{\xt{\infty}{2}}\les (\norm{n}_{L^\infty}+\no(n))(\norm{\po N}_{L^\infty}+\no(\po N))\les \ep,
\ee
where we used in the last inequality the estimate \eqref{estn} for $n$ and the estimates \eqref{estNomega} and \eqref{estricciomega} for $\po N$. $P$ satisfies:
\be\lab{pope5}
\no(P)\les \no(P_1)+\no(P_2)\les \ep.
\ee
Finally, using the non sharp product estimate \eqref{nonsharpprod2}, $E$ satisfies:
\be\lab{pope6}
\norm{E}_{\PP^0}\les \norm{E_1}_{\PP^0}+\norm{E_2}_{\PP^0}+(\no(\kepb)+\no(\kep)+\no(n^{-1}\nabla n))\no(\chi)\les \ep,
\ee
where we used in the last inequality the estimates \eqref{estn}-\eqref{esthch} for $n, \kep, \kepb$ and $\chi$.
Now, \eqref{pope3}-\eqref{pope6} together with the sharp trace theorem estimate \eqref{eq:mainleprop1t} yields:
\bea
\lab{pope7}\|\Pi(\po\chi)\|_{{\BB}^0}&\lesssim&\big(\no(\chi)+\no(\db)+\norm{\chi}_{\xt{\infty}{2}}+\|\db\|_{L_{x'}^\infty L_t^2}\big)\cdot\|\Pi(\po\chi)\|_{{\PP}^0}\\
\nn&&+\big(\no(F)+\|F\|_{L_{x'}^\infty L_t^2}\big)\cdot\no(P)
+\big(\no(F)+\|F\|_{L_{x'}^\infty L_t^2}\big)\cdot\|E\|_{{\PP}^0}\\
\nn&\les& \ep\|\Pi(\po\chi)\|_{{\PP}^0}+\ep,
\eea
where we used the estimate \eqref{estn}-\eqref{esthch} for $\db$ and $\chi$ in the last estimate. 

Finally, \eqref{pope1} and \eqref{pope7} yield \eqref{pope} which concludes the proof of the proposition.
\end{proof}

\subsection{Estimate for $N(\c, \o)-N(\c, \o')$}\lab{sec:jycroispas4}

The goal of this section is to prove \eqref{ad1}. The following lemmas will be useful. 
\begin{lemma}\lab{lemma:ad2} 
We have:
\be\lab{ad2}
\norm{Q_{>1}(N)}_{L^\infty}\les \ep,
\ee
where $Q_j$ is the geometric Littlewood-Paley decomposition on $\Sit$ introduced in section \ref{sec:LPSit}. 
\end{lemma}

\begin{lemma}\lab{lemma:ad3} 
Let $\o$ and $\o'$ in $\S$. Let $N'=N(.,\o')$, and let $\BB^0$ the Besov space defined with respect to $u(.,\o)$. We have:
\be\lab{ad3}
\norm{\nabb Q_{\leq 1}(N')}_{\BB^0}\les \ep.
\ee
\end{lemma}

\begin{lemma}\lab{lemma:ad4} 
Let $\o$ and $\o'$ in $\S$. Let $N'=N(.,\o')$, and let $\lh{2}$ defined with respect to $u(.,\o)$. We have:
\be\lab{ad4}
\norm{\dd_L(N')}_{\lh{2}}\les \ep.
\ee
\end{lemma}

The proof of Lemma \ref{lemma:ad2} is postponed to section \ref{sec:ad2}, the proof of Lemma \ref{lemma:ad3} is postponed to section \ref{sec:ad3}, and the proof of Lemma \ref{lemma:ad4} is postponed to section \ref{sec:ad4}. We now prove \eqref{ad1}. 

Let us define the angle $\o_1\in \S$ as:
$$\o_1=\frac{\o-\o'}{|\o-\o'|},$$
and let $N_1=N(\c,\o_1)$. In view of Lemma \ref{lemma:ad2}, we have:
\bea
\nn\norm{g(\po N,N_1)-1}_{L^\infty}&\les&  \norm{g(\po N,Q_{\leq 1}(N_1))-1}_{L^\infty}+ \norm{g(\po N(\c,\o''),Q_{>1}(N_1)}_{L^\infty}\\
\nn&\les& \norm{g(\po N,Q_{\leq 1}(N_1))-1}_{L^\infty}+ \norm{\po N}_{L^\infty}\norm{Q_{>1}(N_1)}_{L^\infty}\\ 
\lab{ad5}&\les& \norm{g(\po N,Q_{\leq 1}(N_1))-1}_{L^\infty}+\ep,
\eea
where we used the estimate \eqref{estNomega} for $\po N$. 

Since $g(\po N,Q_{\leq 1}(N_1))-1$ is a scalar function, we may estimate its $L^\infty$ norm using \eqref{linftybound}:
\bea
&&\lab{ad6}\norm{g(\po N,Q_{\leq 1}(N_1))-1}_{L^\infty}\\
\nn&\les& \norm{g(\po N,Q_{\leq 1}(N_1))-1}_{\tx{\infty}{2}}+\norm{g(\po N,Q_{\leq 1}(N_1))-1}_{\BB^0}\\
\nn&\les& \norm{g(\po N,N_1)-1}_{\tx{\infty}{2}}+\norm{\po N}_{\tx{\infty}{2}}\norm{Q_{>1}(N_1)}_{L^\infty}+\norm{g(\po N,Q_{\leq 1}(N_1))-1}_{\BB^0}\\
\nn&\les& \norm{g(\po N,N_1)-1}_{\tx{\infty}{2}}+\ep+\norm{\nabb g(\po N,Q_{\leq 1}(N_1))}_{\BB^0},
\eea
where we used the estimate \eqref{estNomega} for $\po N$ and Lemma \ref{lemma:ad2} in the last inequality.

Next, we estimate the right-hand side of \eqref{ad6} starting with the last term. Using the estimate \eqref{eq:secondbilBesov:bis}, we have:
\bea 
\nn\norm{\nabb g(\po N,Q_{\leq 1}(N_1))}_{\BB^0}&\les& (\norm{\nabb Q_{\leq 1}(N_1)}_{\tx{\infty}{2}}+\norm{Q_{\leq 1}(N_1)}_{L^\infty})\norm{\nabb\po N}_{\BB^0}\\
\nn&&+(\norm{\nabb\po N}_{\tx{\infty}{2}}+\norm{\po N}_{L^\infty})\norm{\nabb Q_{\leq 1}(N_1)}_{\BB^0}\\
\lab{ad7}&\les& \ep
\eea
where we used in the last inequality the estimates \eqref{estNomega} \eqref{estricciomega} for $\po N$ and the estimate of Lemma \ref{lemma:ad3} for $Q_{\leq 1}(N_1)$.

We consider the last term in the right-hand side of \eqref{ad6}. Let $\o''\in S$ on the arc joining $\o$ and $\o'$, and let $N=N(.,\o'')$. Then, with our choice for $N_1$, we have at $t=0$ (see \cite{param1}):
$$\norm{g(\po N,N_1)-1}_{L2(\pou)}\les \ep+|\o-\o'|,$$
which together with the estimate \eqref{estimtransport1} for transport equations yields:
\bea
\lab{ad8}&&\norm{g(\po N,N_1)-1}_{\tx{\infty}{2}}\\
\nn&\les& \norm{\ddb_L(g(\po N,N_1))}_{\lh{2}}+\ep+|\o-\o'|\\
\nn&\les& \norm{\dd_L(\po N)}_{\lh{2}}+\norm{\dd_L(N_1)}_{\lh{2}}\norm{\po N}_{L^\infty}+\ep+|\o-\o'|\\
\nn&\les& \ep+|\o-\o'|,
\eea
where we used in the last inequality the estimates \eqref{estNomega} and \eqref{estricciomega} for $\po N$, and Lemma \ref{lemma:ad4} for $N_1$.

Finally, \eqref{ad6}-\eqref{ad8} yield:
$$\norm{g(\po N,N_1)-1}_{L^\infty}\les  \ep+|\o-\o'|,$$
for any $N=N(.,\o'')$ with $\o''\in \S$ on the arc joining $\o$ and $\o'$. This yields:
$$|g(N-N',N_1)-|\o-\o'||\les |\o-\o'|(\ep+|\o-\o'|).$$
Therefore, we have:
$$|N-N'|\geq |g(N-N',N_1)|\geq |\o-\o'|(1-O(\ep)-O(|\o-\o'|))\gtrsim |\o-\o'|,$$
which implies the desired estimate \eqref{ad1}. This concludes the proof of Theorem \ref{thregomega}.

\section{Second order derivatives with respect to $\o$}\lab{sec:secondderivomega}

The goal of this section is to prove Theorem \ref{thregomega2}.

\subsection{Equation for $\dd_L\po^2N, \dd_A\po^2N, \dd_{\lb}\po^2N$, $\po^2\z$ and $\po^2b$}

The following lemma provides the formulas satisfied by $\dd_L\po^2N, \dd_A\po^2N$ and $\dd_{\lb}\po^2N$.
\begin{lemma}\lab{lemma:poo1}
$\po^2N$ satisfies the following formulas:
\bea\lab{poo1}
\nn\dd_L(\po^2N)&=&-2(\po\chi)_{\po NB}e_B-\chi_{\Pi(\po^2N)B}e_B+2\chi_{\po N\po N}N+(|\po N|^2n^{-1}\nabla_Nn\\
&&+\eta_{\po N\po N}+\kepb_{\Pi(\po^2N)})L-\db\po^2N-\kepb_{\po N}\po N+|\po N|^2\z_Be_B,
\eea
\bea
\nn\dd_A(\po^2N)&=&(\po^2\chi)_{AB}e_B-(\po^2N)_A(\z_Be_B+\d L)-\chi_{A\Pi(\po^2N)}N-\po\chi_{A\po N}N-2\chi_{A\po N}\po N\\
\lab{poo2}&&-(\po N)_A\left(2\po\z_Be_B-2\z_{\po N}N+\left(\frac{5}{2}\kep_{\po N}+n^{-1}\nabla_{\po N}n\right)L+2\d\po N\right),
\eea
and
\bea
\lab{poo3}\dd_{\lb}(\po^2N)&=&2\po^2\z_Be_B-4\po\z_{\po N}N-2\z_{\Pi(\po^2N)}N-|\po N|^2\z_Be_B+2\po\chi_{\po NB}e_B\\
\nn&&+\chi_{\Pi(\po^2N)B}e_B-2\chi_{\po N\po N}N+(\d+n^{-1}\nabla_Nn)\po^2N+(4\kep_{\po N}\\
\nn&&+n^{-1}\nabla_{\po N}n)\po N+(-3|\po N|^2\d+2\eta_{\po N\po N}+2\kep_{\Pi(\po^2N)}+n^{-1}\nabla_{\po^2N}n)L.
\eea
\end{lemma}

\begin{proof}
We first derive \eqref{poo1}. We differentiate the equation \eqref{popo1} satisfied by $\po N$ with respect to $\o$. Using the fact that $\po L=\po N$, we obtain:
\bea\lab{poo4}
&&\dd_L\po^2N+\dd_{\po N}\po N\\
\nn&=&-\po\chi_{\po NB}e_B-\chi_{\Pi(\po^2N)B}e_B-\chi_{\po N\po e_B}e_B-\chi_{\po NB}\po e_B-\po(\db)\po N-\db\po^2N\\
\nn&&+(k_{\po N\po N}+k_{N\po^2N}-n^{-1}\nabla_{\po^2N}n)L-\kepb_{\po N}\po N.
\eea
We compute the various term in the right-hand side of \eqref{poo4}. Using \eqref{coo2ter}, we have:
\be\lab{poo5}
\chi_{\po N\po e_B}e_B+\chi_{\po NB}\po e_B=-\chi_{\po N\po N}N.
\ee
Also, the formula \eqref{popo2} for $\dd_A(\po N)$ yields:
\be\lab{poo5bis}
\dd_{\po N}\po N=\po\chi_{\po NB}e_B-|\po N|^2\z_Be_B-|\po N|^2\d L-\chi_{\po N\po N}N.
\ee
Now, differentiating twice $g(N,N)=1$ with respect to $\o$ yields:
\be\lab{poo6}
\po^2N=\Pi(\po^2N)-|\po N|^2N.
\ee
Finally, \eqref{poo4}, \eqref{poo5}, \eqref{poo6} and the formula \eqref{po2} for $\po\db$ yields \eqref{poo1}.

Next, we derive \eqref{poo2}. We differentiate the equation \eqref{popo3} satisfied by $\dd_A(\po N)$ with respect to $\o$. Using the fact that $\po L=\po N$, we obtain:
\bea
&&\lab{poo7}\dd_A(\po^2N)+g(\po e_A,N)\dd_N(\po N)\\
\nn&=&\po^2\chi_{AB}e_B+g(\po e_A,N)\po\chi_{NB}e_B+\po\chi_{A\po e_B}e_B+\po\chi_{AB}\po e_B\\
\nn&&-g(\po^2N,e_A)(\z_Be_B+\d L)-g(\po N,e_A)(\po\z_Be_B+\z_{\po B}e_B+\z_B\po e_B+2\kep_{\po N}L\\
\nn&&+\d\po N)-\po\chi_{A\po N}N-\chi_{A\Pi(\po^2N)}-\chi_{A\po N}\po N.
\eea
We compute the various term in the right-hand side of \eqref{poo7}. Using \eqref{coo2ter}, we have:
\be\lab{poo8}
\chi_{A\po e_B}e_B+\chi_{AB}\po e_B=-\chi_{A\po N}N,
\ee
and 
\be\lab{poo9}
\z_{\po B}e_B+\z_B\po e_B=-\z_{\po N}N.
\ee
Using the equations \eqref{poo1} and \eqref{poo2} respectively for $\dd_L(\po N)$ and $\dd_{\lb}(\po N)$ together with the fact that $N=\half(L-\lb)$ yields:
\be\lab{poo10}
\dd_N(\po N)=-\po\z_Be_B-\chi_{\po NB}e_B-\d\po N-\left(\frac{\kep_{\po N}}{2}+n^{-1}\nabla_{\po N}n\right)L+\z_{\po N}N.
\ee
Finally, \eqref{poo7}-\eqref{poo10} together with the fact that $g(\po e_A,N)=-g(\po N,e_A)$ and $\po\chi_{NB}=-\chi_{\po NB}$ yields \eqref{poo2}.

Last, we derive \eqref{poo3}. We differentiate the equation \eqref{popo2} satisfied by $\dd_{\lb}(\po N)$ with respect to $\o$. Using the fact that $\po L=\po N$ and $\po\lb=-\po N$, we obtain:
\bea
\lab{poo11}&&\dd_{\lb}(\po^2N)-\dd_{\po N}(\po N)\\
\nn&=& 2(\po^2\z)_Ae_A+2\po\z_{\po e_A}e_A+2\po\z_A\po e_A+\po\chi_{\po NA}e_A+\chi_{\Pi(\po^2N)A}e_A+\chi_{\po N\po e_A}e_A\\
\nn&&+\chi_{\po NA}\po e_A+(\d+n^{-1}\nabla_Nn)\po^2N+(\po(\d)+n^{-1}\nabla_{\po N}n)\po N\\
\nn&&+(2\eta_{\po N\po N}+2k_{\po^2NN}+n^{-1}\nabla_{\po^2N}n)L+(2\kep_{\po N}+n^{-1}\nabla_{\po N}n)\po N\\
\nn&&-2\po\z_{\po N}N-2\z_{\Pi(\po^2N)}N-2\z_{\po N}\po N.
\eea
We compute the various term in the right-hand side of \eqref{poo11}. Using \eqref{coo2ter}, we have:
\be\lab{poo13}
\po\z_{\po A}e_A+\po\z_A\po e_A=-\po\z_N\po N-\po\z_{\po N}N=\z_{\po N}\po N-\po\z_{\po N}N,
\ee
where we used the fact that $\po\z_N=-\z_{\po N}$. Also, contracting \eqref{poo8} with $\po N$ yields:
\be\lab{poo12}
\chi_{\po N\po e_A}e_A+\chi_{\po NA}\po e_A=-\chi_{\po N\po N}N.
\ee
Finally, \eqref{poo11}-\eqref{poo12} together with \eqref{po2} for $\po(\d)$, \eqref{poo5bis} and \eqref{poo6} yields \eqref{poo3}. This concludes the proof of Lemma \ref{lemma:poo1}.
\end{proof}

The following lemma provides the transport equation satisfied by $\Pi(\po^2\z)$. 
\begin{lemma}\lab{lemma:poo2}
$\Pi(\po^2\z)$ satisfies the following transport equation:
\bea
\lab{poo14}&&\ddb_L(\Pi(\po^2\z))_A\\
\nn&=&-\chi_{AB}\po^2\z_B-(\kepb_B+\z_B)\po^2\chi_{AB}-\ddb_{\po^2N}\z_A-\frac{(\po^2N)_B}{2}(-\a_{AB}+\r\d_{AB}+3\s\in_{AB})\\
\nn&&+\kepb_A\z_{\Pi(\po^2N)}-\chi_{AB}\eta_{B\Pi(\po^2N)}-(\po^2N)_A\kepb\c\z-2\nabb_{\po N}(\Pi(\po\z))_A\\
\nn&&+(\po N)_A(\ddb_L\z_{\po N}-\chi_{\po NB}\z_B-\db\z_{\po N}-2\kepb\c\po\z-\th_{\po NB}\z_B-\eta_{\po NB}\z_B)\\
\nn&&-2(\eta_{B\po N}+\po\z_B)\po\chi_{AB}+\db\po\chi_{A\po N}-2\kepb_A\po\z_{\po N}+(-3\z_{\po N}+\kep_{\po N}-\kepb_{\po N})\chi_{A\po N}\\
\nn&&+|\po N|^2\kep_B\chi_{AB}+(\th_{A\po N}+\eta_{\po NA}-(\po N)_A\db)\z_{\po N}+\frac{|\po N|^2}{2}\b_A\\
\nn&&+\frac{(\po N)_B}{2}\bigg((\po N)_C(\in_{AC}{}^*\b_B+\in_{BC}{}^*\b_A)-\d_{AB}(\b_{\po N}+\bb_{\po N})\\
\nn&&+\frac{3}{2}\in_{AB}({}^*\b_{\po N}-{}^*\bb_{\po N})\bigg).
\eea
\end{lemma}

\begin{proof}
We differentiate the equation \eqref{popo5} satisfied by $\dd_{\lb}(\po N)$ with respect to $\o$:
\bea
\lab{poo15}&&\po(\ddb_L(\Pi(\po\z)))_A\\
\nn&=&-\po(\nabb_{\po N}\z)_A+(\eta_{\po NA}+g(\po e_A,N)\db)\z_{\po N}+\kepb_A\po\z_{\po N}+\kepb_A\z_{\Pi(\po^2N)}-(\po^2N)_A\kepb\c\z\\
\nn&&-(\po N)_A\po(\kepb\c\z)-\po(k_{B\po N}+\po\z_B)\chi_{AB}-(k_{B\po N}+\po\z_B)(\po\chi_{AB}+\chi_{A\po e_B})\\
\nn&&-\po(\kepb_B+\z_B)\po\chi_{AB}-(\kepb_B+\z_B)(\po^2\chi_{AB}+\po\chi_{A\po e_B})\\
\nn&&-\po\left(\frac{(\po N)_B}{2}(-\a_{\c B}+\r\d_{\c B}+3\s\in_{\c B})\right)_A
\eea
We compute the various term in the right-hand side of \eqref{poo15}. We have:
$$\po^2\z_A=\po(\Pi(\po\z))_A-\z_{\po N}(\po N)_A,$$
which yields:
\bea
\lab{poo16}\ddb_L(\Pi(\po(\Pi(\po\z))))_A&=&\ddb_L(\Pi(\po^2\z))_A-(\ddb_L\z)_{\po N}(\po N)_A\\
\nn&&-\z_{\ddb_L(\po N)}(\po N)_A-\z_{\po N}g(\ddb_L(\po N),e_A)\\
\nn&=&\ddb_L(\Pi(\po^2\z))_A-(\ddb_L\z)_{\po N}(\po N)_A\\
\nn&&+\chi_{\po NB}\z_B(\po N)_A+\db\z_{\po N}(\po N)_A
+\z_{\po N}\chi_{\po NA}+\db\z_{\po N}(\po N)_A,
\eea
where we used the formula \eqref{popo1} for $\dd_L(\po N)$ in the last equality. Using the commutator formula \eqref{commo5} together with \eqref{poo16} yields:
\bea
\lab{poo17}&&\po(\ddb_L(\Pi(\po\z)))_A\\
\nn&=&\ddb_L(\Pi\po(\Pi(\po\z)))_A+\nabb_{\po N}\Pi(\po\z)_A-\kepb_A\po\z_{\po N}+(\po N)_A\kepb\c\po\z\\
\nn&=& \ddb_L(\Pi(\po^2\z))_A-(\ddb_L\z)_{\po N}(\po N)_A+\chi_{\po NB}\z_B(\po N)_A+\db\z_{\po N}(\po N)_A\\
\nn&& +\z_{\po N}\chi_{\po NA}+\db\z_{\po N}(\po N)_A+\nabb_{\po N}\Pi(\po\z)_A-\kepb_A\po\z_{\po N}+(\po N)_A\kepb\c\po\z.
\eea

Next, we compute the term $\po(\nabb_{\po N}\z)_A$. We have:
\bea
\lab{poo18}\po(\nabb_{\po N}\z_A)&=&\po(\dd_{\po N}(\z_A)-\z_{\nabb_{\po N}e_A})\\
\nn&=&\dd_{\po^2N}(\z_A)+\dd_{\po N}(\po\z_A)+\dd_{\po N}(\z_{\Pi(\po e_A)})
-\po\z_{\nabb_{\po N}e_A}-\z_{\po(\nabb_{\po N}e_A)}\\
\nn&=&\dd_{\po^2N}(\z_A)+\nabb_{\po N}(\Pi(\po\z))_A+\dd_{\po N}(\z_{\Pi(\po e_A)})
-\z_{\po(\nabb_{\po N}e_A)}
\eea
Now, \eqref{coo5} implies:
\bee
\po(\nabb_{\po N}e_A)&=&\po((\po N)_B\nabb_Be_A)\\
&=&(\po^2N)_B\nabb_B(e_A)+g(\po N,\po e_B)\nabb_B(e_A)+(\po N)_B(\ddb_{\po e_B}(e_A)\\
&&+\nabb_B(\Pi(\po e_A))-(\po N)_A\th_{BC}e_C+\th_{AB}\po N-(\dd_Be_A,\po N)N)\\
&=&\nabb_{\Pi(\po^2N)}(e_A)-|\po N|^2\ddb_N(e_A)+\nabb_{\po N}(\Pi(\po e_A))-(\po N)_A\th_{\po NC}e_C\\
&&+\th_{A\po N}\po N-(\dd_{\po N}e_A,\po N)N)\\
&=&\nabb_{\po^2N}(e_A)+\nabb_{\po N}(\Pi(\po e_A))-(\po N)_A\th_{\po NC}e_C+\th_{A\po N}\po N\\
&&-(\dd_{\po N}e_A,\po N)N),
\eee
where we used \eqref{poo6} in the last equality. Together with \eqref{poo18}, we obtain:
\bee
\po(\nabb_{\po N}\z_A)&=&\ddb_{\po^2N}(\z)_A+\nabb_{\po N}(\Pi(\po\z))_A+\nabb_{\po N}(\z)_{\Pi(\po e_A)}\\
\nn&&+(\po N)_A\th_{\po NB}\z_B-\th_{A\po N}\z_{\po N},
\eee
which yields:
\be\lab{poo19}
\po(\nabb_{\po N}\z)_A=\ddb_{\po^2N}(\z)_A+\nabb_{\po N}(\Pi(\po\z))_A+(\po N)_A\th_{\po NB}\z_B-\th_{A\po N}\z_{\po N}.
\ee

Next, we compute $\po(\kepb\c\z)$. Using \eqref{coo2ter}, we have:
\bea
\lab{poo20}\po(\kepb\c\z)&=&(\eta_{\po NB}+k_{N\po e_B}-n^{-1}\nabla_{\po e_B}n)\z_B+\kepb_B(\po\z_B+\z_{\po e_B})\\
\nn&=&\eta_{\po Ne_B}\z_B-\db\z_{\po N}+\kepb\c\po\z.
\eea
Using again \eqref{coo2ter}, we also obtain:
\bea
\lab{poo21}&&-\po(k_{B\po N}+\po\z_B)\chi_{AB}-(k_{B\po N}+\po\z_B)(\po\chi_{AB}+\chi_{A\po e_B})\\
\nn&&-\po(\kepb_B+\z_B)\po\chi_{AB}-(\kepb_B+\z_B)(\po^2\chi_{AB}+\po\chi_{A\po e_B})\\
\nn&=&-(k_{B\po^2N}+\po^2\z_B)\chi_{AB}+(k_{N\po N}+\po\z_N)\chi_{A\po N}-(\eta_{B\po N}+\po\z_B)\po\chi_{AB}\\
\nn&&-(k_{\po NB}+\po\z_B)\po\chi_{AB}+(k_{NN}-n^{-1}\nabla_Nn)\po\chi_{A\po N}+(\kepb_{\po N}+\z_{\po N})\po\chi_{AN}\\
\nn&&-(\kepb_B+\z_B)\po^2\chi_{AB}\\
\nn&=&-(\eta_{B\Pi(\po^2N)}-|\po N|^2\kep_B+\po^2\z_B)\chi_{AB}+(\kep_{\po N}-\z_{\po N})\chi_{A\po N}\\
\nn&&-2(\eta_{B\po N}+\po\z_B)\po\chi_{AB}+\db\po\chi_{A\po N}-(\kepb_{\po N}+\z_{\po N})\chi_{A\po N}-(\kepb_B+\z_B)\po^2\chi_{AB},
\eea
where we used the fact that $\po\z_N=-\z_{\po N}$, $\po\chi_{AN}=-\chi_{A\po N}$ and the decomposition of $\po^2N$ \eqref{poo6} in the last inequality.

Finally, we consider the last term in the right-hand side of \eqref{poo15}. From the definition of $\b$, $\r$, $\s$, and the fact that $\po L=\po N$ and $\po\lb=-\po N$, we have:
$$\po\r=-\b_{\po N}-\bb_{\po N},\,\po(\in\s)_{AB}=\half\in_{AB}({}^*\b_{\po N}-{}^*\bb_{\po N}),$$
which together with the formula \eqref{po3} for $\po\a$ yields:
\bea
\lab{poo22}&&\po\left(\frac{(\po N)_B}{2}(-\a_{\c B}+\r\d_{\c B}+3\s\in_{\c B})\right)_A\\
\nn&=&\half(\po^2N)_B(-\a_{AB}+\r\d_{AB}+3\s\in_{AB})-\frac{|\po N|^2}{2}\b_A+\frac{(\po N)_B}{2}\bigg((\po N)_C(\in_{AC}{}^*\b_B\\
\nn&&+\in_{BC}{}^*\b_A)-\d_{AB}(\b_{\po N}+\bb_{\po N})+\frac{3}{2}\in_{AB}({}^*\b_{\po N}-{}^*\bb_{\po N})\bigg).
\eea
Using \eqref{poo15}-\eqref{poo22} yields \eqref{poo14} which concludes the proof of Lemma \ref{lemma:poo2}.
\end{proof}

Finally, the following lemma provides the transport equation satisfied by $\po^2b$. 
\begin{lemma}\lab{lemma:poo2bob}
$\po^2b$ satisfies the following transport equation:
\bea\lab{transportpo2b}
L(\po^2b)&=&-\nabb_{\po N}(\po b)-b\po\z_{\po N}-b\z_{\Pi(\po^2N)}-\po(b)\z_{\po N}-\po^2(b)\db\\
\nn&&-\po(b)(2\kep_{\po N}-n^{-1}\nabb_{\po N}n)-k_{\po N \po N}b-\kepb_{\Pi(\po^2N)}b -\kepb_{\po N}\po b,
\eea
\end{lemma}

\begin{proof}
Recall the transport equation \eqref{dw}  satisfied by $\po b$
$$L(\po b)=-b\z_{\po N}-\po(b)\db-\kepb_{\po N}b.$$
Differentiating with respect to $\o$ yields \eqref{transportpo2b}. This concludes the proof of the Lemma.
\end{proof}

\subsection{Estimates for $\po^2N, \po^2b, \po^2\chi$ and $\po^2\z$}

\subsubsection{Estimates for $\po^2N$}

In view of the formula \eqref{poo1} for $\dd_L(\po^2N)$, we have:
\bea
\lab{poo23}&&\norm{\dd_L(\po^2N)}_{\lh{2}}\\
\nn&\les& \norm{\po\chi}_{\lh{2}}\norm{\po N}_{L^\infty}+(\norm{\chi}_{\xt{\infty}{2}}+\norm{\kepb}_{\xt{\infty}{2}}+\norm{\db}_{\xt{\infty}{2}})\norm{\Pi(\po^2N)}_{\xt{2}{\infty}}\\
\nn&&+(\norm{\chi}_{\lh{2}}+\norm{n^{-1}\nabla n}_{\lh{2}}+\norm{\eta}_{\lh{2}}+\norm{\kepb}_{\lh{2}}+\norm{\z}_{\lh{2}})\norm{\po N}^2_{L^\infty}\\
\nn&\les&\ep+\ep\norm{\Pi(\po^2N)}_{\xt{2}{\infty}},
\eea
where we used in the last inequality the estimates \eqref{estn}-\eqref{estzeta} for $n, \eta, \kepb, \db, \chi$ and $\z$, the estimate \eqref{estNomega} for $\po N$ and the estimate \eqref{estricciomega} for $\po\chi$. Now, the decomposition \eqref{poo6} for $\po^2N$ yields:
$$\dd_L(\po^2N)=\dd_L(\Pi(\po^2N))-|\po N|^2\dd_LN-2g(\po N,\dd_L(\po N))N,$$
which together with \eqref{poo23} and the estimates \eqref{estNomega} \eqref{estricciomega} for $\po N$ yields:
$$\norm{\ddb_L(\Pi(\po^2N))}_{\lh{2}}\les\ep+\ep\norm{\Pi(\po^2N)}_{\xt{2}{\infty}}.$$
Together with the estimate \eqref{estimtransport1} for transport equations, this implies:
$$\norm{\Pi(\po^2N)}_{\xt{2}{\infty}}\les 1,$$
and using again the decomposition \eqref{poo6} for $\po^2N$ and the estimate \eqref{estNomega}  for $\po N$, we obtain:
\be\lab{poo25}
\norm{\po^2N}_{\xt{2}{\infty}}\les 1.
\ee
Finally, \eqref{poo23} and \eqref{poo25} imply:
\be\lab{poo26}
\norm{\dd_L(\po^2N)}_{\lh{2}}\les\ep.
\ee

\subsubsection{estimate for $\po^2b$}

In view of the transport equation \eqref{transportpo2b}, we have
\be\lab{fuite}
L(\po^2b)=f,
\ee
where the scalar $f$ is given by
\bee
f &=& -\nabb_{\po N}(\po b)-b\po\z_{\po N}-b\z_{\Pi(\po^2N)}-\po(b)\z_{\po N}-\po^2(b)\db\\
\nn&&-\po(b)(2\kep_{\po N}-n^{-1}\nabb_{\po N}n)-k_{\po N \po N}b-\kepb_{\Pi(\po^2N)}b -\kepb_{\po N}\po b.
\eee
In view of the definition of $f$, we have
\bea\lab{fuite1}
\nn\norm{f}_{\xt{2}{1}}&\les& (1+\norm{\po N}_{L^\infty}+\norm{b}_{L^\infty}+\norm{\po b}_{L^\infty})^3(1+\norm{\z}_{\xt{\infty}{2}}+\norm{\kep}_{\xt{\infty}{2}})\\
\nn&&\times(\norm{\nabb\po b}_{\lh{2}}+\norm{\po\z}_{\lh{2}}+\norm{\po^2N}_{\lh{2}}+\norm{k}_{\lh{2}}\\
\nn&&+\norm{n^{-1}\nabla n}_{\lh{2}}+\norm{\z}_{\lh{2}})+\norm{\db}_{\xt{\infty}{2}}\norm{\po^2b}_{\lh{2}}
\\
&\les& \ep+\ep\norm{\po^2b}_{\lh{2}},
\eea
where we used in the last inequality the estimates \eqref{estn} and \eqref{estk} for $n, k, \kep$ and $\db$, the estimate \eqref{estb} for $b$, the estimate \eqref{estzeta} for $\z$, the estimate \eqref{estNomega} for $\po N$, the estimate \eqref{estricciomega} for $\po b$ and $\po\z$, and the estimate \eqref{poo25} for $\po^2N$. \eqref{fuite}, \eqref{fuite1} together with the estimate for transport equations \eqref{estimtransport1} yield
$$\norm{\po^2b}_{\xt{\infty}{2}}\les \ep+\ep\norm{\po^2b}_{\lh{2}},$$
which implies
\be\lab{fuite2}
\norm{\po^2b}_{\xt{\infty}{2}}\les \ep.
\ee

\subsubsection{Estimates for $\po^2\chi$}

In view of \eqref{poo2}, we have:
\be\lab{poo27}
g(\dd_A(\po^2N),e_B)=(\po^2\chi)_{AB}+F_{AB},
\ee
where the $\ptu$-tangent 2-tensor $F$ is given by:
$$F_{AB}=-(\po^2N)_A\z_Be_B-2\chi_{A\po N}(\po N)_B-(\po N)_A\left(2\po\z_B+2\d(\po N)_B\right).$$
$F$ satisfies the following estimate:
\bea\lab{poo28}
&&\norm{F}_{\tx{\infty}{\frac{4}{3}}}\\
\nn&\les&\norm{\po^2N}_{\tx{\infty}{2}}\norm{\z}_{\tx{\infty}{4}}+(\norm{\chi}_{\tx{\infty}{2}}+\norm{\d}_{\tx{\infty}{2}})\norm{\po N}^2_{L^\infty}+\norm{\po N}_{L^\infty}\norm{\po\z}_{\tx{\infty}{2}}\\
\nn&\les& \ep,
\eea
where we used in the last inequality the estimates \eqref{estk}-\eqref{estzeta} for $\d, \chi$ and $\z$, the estimates \eqref{estNomega} \eqref{estricciomega} for $\po N$ and $\po\z$, and the estimate \eqref{poo25} for $\po^2N$. 

Using the decomposition \eqref{poo6}, we have: 
$$\dd_A(\po^2N)=\dd_A(\Pi(\po^2N))-2g(\po N,\dd_A\po N)N-|\po N|^2\dd_AN$$
which together with the fact that $\dd_AN=\th_{AB}e_B$ yields:
$$g(\ddb_A(\Pi(\po^2N)),e_B)=g(\dd_A(\po^2N),e_B)+|\po N|^2\th_{AB}.$$
Together with \eqref{poo27}, this yields:
\be\lab{poo29}
\nabb(\Pi(\po^2N))=\Pi(\po^2\chi)+\tilde{F}
\ee
where $\widetilde{F}=F+|\po N|^2\th$. In view of \eqref{poo28} and the estimates \eqref{estk}-\eqref{esthch} for $\th=\chi+\eta$, we have:
\be\lab{poo30}
\norm{\widetilde{F}}_{\tx{\infty}{\frac{4}{3}}}\les\norm{F}_{\tx{\infty}{\frac{4}{3}}}+
\norm{\th}_{\tx{\infty}{2}}\norm{\po N}^2_{L^\infty}\les \ep.
\ee

Using \eqref{poo29} together with the finite band property and the weak Bernstein inequality for $P_j$, we have:
\bea
\lab{poo31}\norm{P_j\Pi(\po^2\chi)}_{\tx{\infty}{2}}&\les& \norm{P_j\nabb(\Pi(\po^2N))}_{\tx{\infty}{2}}+\norm{P_j\tilde{F}}_{\tx{\infty}{2}}\\
\nn&\les& 2^j\norm{\Pi(\po^2N)}_{\tx{\infty}{2}}+2^{\frac{j}{2}}\norm{\tilde{F}}_{\tx{\infty}{\frac{4}{3}}}\\
\nn&\les&2^j\ep,
\eea
where we used the estimate \eqref{poo25} for $\po^2N$, and the estimate \eqref{poo30} for $\widetilde{F}$. \eqref{poo31} is the desired estimate for $\po^2\chi$. 

\begin{remark}
While $\widetilde{F}$ satisfies \eqref{poo30}, we may also derive a second estimate. We have:
\bea
\nn\norm{\widetilde{F}}_{\xt{1}{\infty}}&\les&\norm{\po^2N}_{\xt{2}{\infty}}\norm{\z}_{\xt{2}{\infty}}+(\norm{\chi}_{\xt{2}{\infty}}+\norm{\d}_{\xt{2}{\infty}}+\norm{\th}_{\xt{2}{\infty}})\norm{\po N}^2_{L^\infty}\\
\nn&&+\norm{\po N}_{L^\infty}\norm{\po\z}_{\tx{2}{\infty}}\\
\lab{poo28bis}&\les& \ep,
\eea
where we used in the last inequality the estimates \eqref{estk}-\eqref{estzeta} for $\d, \chi, \th$ and $\z$, the estimates \eqref{estNomega} \eqref{estricciomega} for $\po N$ and $\po\z$, and the estimate \eqref{poo25} for $\po^2N$. 
\end{remark}

\subsubsection{estimate for $\po^2\z$}\lab{sec:imagej}

In view of the formula \eqref{poo14} for $\ddb_L(\Pi(\po^2\z))$, the decomposition \eqref{poo6} for $\po^2N$, and the decomposition \eqref{poo29} for $\po^2\chi$, we have:
\be\lab{poo32}
\ddb_L(\Pi(\po^2\z))=-\chi\c\Pi(\po^2\z)+\nabb(F_1)+F_2-\frac{|\po N|^2}{2}\ddb_{\lb}(\z),
\ee
where the $\ptu$-tangent tensors $F_1$ and $F_2$ are respectively given by:
$$F_1=-(\kepb+\z)\c\Pi(\po^2N)-2\po N\c\Pi(\po \z)$$
and 
\bee
(F_2)_A&=&(\nabb(\kepb)+\nabb(\z))\c\Pi(\po^2N)+(\kepb+\z)\c\widetilde{F}-\ddb_{\Pi(\po^2N)}\z_A+\frac{|\po N|^2}{2}\ddb_L(\z)\\
\nn&&-\frac{(\po^2N)_B}{2}(-\a_{AB}+\r\d_{AB}+3\s\in_{AB})\\
\nn&&+\kepb_A\z_{\Pi(\po^2N)}-\chi_{AB}\eta_{B\Pi(\po^2N)}-(\po^2N)_A\kepb\c\z+2\divb(\po N)(\Pi(\po\z))_A\\
\nn&&+(\po N)_A(\ddb_L\z_{\po N}-\chi_{\po NB}\z_B-\db\z_{\po N}-2\kepb\c\po\z-\th_{\po NB}\z_B-\eta_{\po NB}\z_B)\\
\nn&&-2(\eta_{B\po N}+\po\z_B)\po\chi_{AB}+\db\po\chi_{A\po N}-2\kepb_A\po\z_{\po N}+(-3\z_{\po N}+\kep_{\po N}\\
\nn&&-\kepb_{\po N})\chi_{A\po N}+|\po N|^2\kep_B\chi_{AB}+(\th_{A\po N}+\eta_{\po NA}-(\po N)_A\db)\z_{\po N}+\frac{|\po N|^2}{2}\b_A\\
\nn&&+\frac{(\po N)_B}{2}\bigg((\po N)_C(\in_{AC}{}^*\b_B+\in_{BC}{}^*\b_A)-\d_{AB}(\b_{\po N}+\bb_{\po N})\\
\nn&&+\frac{3}{2}\in_{AB}({}^*\b_{\po N}-{}^*\bb_{\po N})\bigg).
\eee

We estimate $F_1$ and $F_2$. For $F_1$, we have:
\bea
\nn\norm{F_1}_{\lh{2}}&\les&(\norm{\kepb}_{\xt{\infty}{2}}+\norm{\z}_{\xt{\infty}{2}})\norm{\Pi(\po^2N)}_{\xt{2}{\infty}}+\norm{\po N}_{L^\infty}\norm{\Pi(\po\z)}_{\lh{2}}\\
\lab{poo33}&\les&\ep,
\eea
where we used the estimates \eqref{estn}-\eqref{estzeta} for $\kepb$ and $\z$, the estimates \eqref{estNomega} \eqref{estricciomega} for $\po N$ and $\po\z$, and the estimate \eqref{poo25} for $\po^2N$. For $F_2$, we have:
\bea
\lab{poo34}&&\norm{F_2}_{\xt{1}{2}}\\
\nn&\les&(\norm{\nabb\kepb}_{\lh{2}}+(\norm{\kepb}{\xt{\infty}{2}}+\norm{\z}_{\xt{\infty}{2}})\norm{\widetilde{F}}_{\xt{1}{\infty}}+\norm{\po^2N}_{\xt{2}{\infty}}\bigg(\norm{\nabb\z}_{\lh{2}}\\
\nn&&+\norm{\a}_{\lh{2}}+\norm{\r}_{\lh{2}}+\norm{\s}_{\lh{2}}+\no(\kepb)\no(\z)+\no(\chi)\no(\eta)\bigg)\\
\nn&&+\norm{\nabb\po N}_{\xt{2}{\infty}}\norm{\po\z}_{\xt{2}{\infty}}+\norm{\po N}^2_{L^\infty}\bigg(\norm{\ddb_L(\z)}_{\lh{2}}+(\no(\chi)+\no(\d)+\no(\th)\\
\nn&&+\no(\eta)+\no(\th)+\no(\db))\no(\z)+(\no(\kep)+\no(\kepb))\no(\chi)+\norm{\b}_{\lh{2}}+\norm{\bb}_{\lh{2}}\bigg)\\
\nn&&+\norm{\po N}_{L^\infty}\bigg(\norm{\kepb}_{\lh{2}}\norm{\po\z}_{\xt{2}{\infty}}+(\norm{\eta}_{\lh{2}}+\norm{\db}_{\lh{2}})\norm{\po\chi}_{\xt{2}{\infty}}\bigg)\\
\nn&&+\norm{\po\z}_{\xt{2}{\infty}}\norm{\po\chi}_{\xt{2}{\infty}}\\
\nn&\les &\ep,
\eea
where we used in the last inequality the curvature bound \eqref{curvflux1} for $\a, \b, \r, \s, \b$ and $\bb$, 
the estimates \eqref{estn}-\eqref{estzeta} for $\kep, \kepb, \eta, \d, \db, \chi, \th$ and $\z$, the estimates \eqref{estNomega} \eqref{estricciomega} for $\po N, \po\chi$ and $\po\z$, the estimate \eqref{poo25} for $\po^2N$, and the estimate \eqref{poo28bis} for $\widetilde{F}$.

We are now in position to derive the estimate for $\po^2\z$. 
Using the transport equation \eqref{poo32} for $\Pi(\po^2\z)$ and the transport equation \eqref{poo35}, for $M$ allows us to get rid of the troublesome term $\chi\c\Pi(\po^2\z)$:
\bee
\ddb_L(M\c \Pi(\po^2\z))&=&\ddb_L(M)\c \Pi(\po^2\z)+M\c \ddb_L(\Pi(\po^2\z))\\
\nn&=&M\c\nabb(F_1)+M\c F_2-\frac{|\po N|^2}{2}M\c \ddb_{\lb}(\z)\\
\nn&=&\nabb(M\c F_1)-\nabb(M)\c F_1+M\c F_2-\frac{|\po N|^2}{2}M\c \ddb_{\lb}(\z),
\eee
Let $2\leq p<q<+\infty$. This yields:
\bea
\lab{poo38}&&\norm{P_j(M\c \Pi(\po^2\z))}_{\tx{q}{2}}\\
\nn&\les&\normm{P_j\left(\int_0^t\nabb(M\c F_1)dt\right)}_{\tx{q}{2}}+\normm{P_j\left(\int_0^t\nabb(M)\c F_1dt\right)}_{\tx{q}{2}}\\
\nn&&+\normm{P_j\left(\int_0^tM\c F_2dt\right)}_{\tx{q}{2}}+\normm{P_j\left(\int_0^t\frac{|\po N|^2}{2}M\c \ddb_{\lb}(\z)dt\right)}_{\tx{q}{2}},
\eea
Next, we estimate the various terms in the right-hand side of \eqref{poo38}. 

We consider the first term in the right-hand side of \eqref{poo38}. Using Lemma \ref{lemma:poo5}, we have:
\bea
\lab{poo39}\normm{P_j\left(\int_0^t\nabb(M\c F_1)dt\right)}_{\tx{q}{2}}&\les& 2^j\norm{M\c F_1}_{\lh{2}}\\
\nn&\les& 2^j\norm{M}_{L^\infty}\norm{F_1}_{\lh{2}}\\
\nn&\les&2^j\ep,
\eea
where we used in the last inequality the estimate \eqref{poo33} for $F_1$ and the estimate \eqref{poo36} for $M$.

Next, we consider the second and the third term in the right-hand side of \eqref{poo38}. 
Using the dual sharp Bernstein inequality  for tensors \eqref{poo37} and the estimate \eqref{estimtransport1} for transport equations, we have:
\bea
\lab{poo40}&&\normm{P_j\left(\int_0^t\nabb(M)\c F_1dt\right)}_{\tx{q}{2}}+\normm{P_j\left(\int_0^tM\c F_2dt\right)}_{\tx{q}{2}}\\
\nn&\les& 2^j\normm{\int_0^t\nabb(M)\c F_1dt}_{\tx{\infty}{1}}+2^j\normm{\int_0^tM\c F_2dt}_{\tx{\infty}{1}}\\
\nn&\les& 2^j\norm{\nabb(M)\c F_1}_{\lh{1}}+2^j\norm{M\c F_2}_{\lh{1}}\\
\nn&\les& 2^j\norm{\nabb(M)}_{\lh{2}}\norm{F_1}_{\lh{2}}+2^j\norm{M}_{L^\infty}\norm{F_2}_{\lh{1}}\\
\nn&\les& 2^j\ep,
\eea
where we used  in the last inequality the estimate \eqref{poo33} for $F_1$, the estimate \eqref{poo34} for $F_2$, and the estimate \eqref{poo36} for $M$.

Finally, we consider the last term in the right-hand side of \eqref{poo38}. Using Lemma \ref{lemma:lbt3}, we have:
\be\lab{poo41}
\normm{P_j\left(\int_0^t\frac{|\po N|^2}{2}M\c \ddb_{\lb}(\z)dt\right)}_{\tx{\infty}{2}}\les \norm{|\po N|^2M}_{\PP^0}(2^j\ep+2^{\frac{j}{2}}\ep\gamma(u)).
\ee
Now, using the non sharp product estimate \eqref{nonsharpprod2}, we have:
\bee
\norm{|\po N|^2M}_{\PP^0}&\les& \no(\po N)(\norm{M\po N}_{\lh{2}}+\norm{\nabb(M\po N)}_{\lh{2}}\\
&\les& \no(\po N)(\norm{M}_{L^\infty}\no(\po N)+\norm{\po N}_{L^\infty}\norm{\nabb M}_{\lh{2}})\\
&\les& 1,
\eee
where we used in the last inequality the estimates \eqref{estNomega} \eqref{estricciomega} for $\po N$, and the estimate \eqref{poo36} for $M$. Together with \eqref{poo41}, this yields:
$$\normm{P_j\left(\int_0^t\frac{|\po N|^2}{2}M\c \ddb_{\lb}(\z)dt\right)}_{\tx{\infty}{2}}\les 2^j\ep+2^{\frac{j}{2}}\ep\gamma(u),$$
which together with \eqref{poo38}, \eqref{poo39} and \eqref{poo40} implies:
\be\lab{poo42}
\norm{P_j(M\c \Pi(\po^2\z))}_{\tx{q}{2}}\les 2^j\ep+2^{\frac{j}{2}}\ep\gamma(u).
\ee
Now, since we have chosen $p<q$, \eqref{poo42} and Lemma \ref{lemma:poo6} yield:
\be\lab{poo43}
\norm{P_j(\Pi(\po^2\z))}_{\tx{p}{2}}\les 2^j\ep+2^{\frac{j}{2}}\ep\gamma(u),
\ee
for any $2\leq p<+\infty$ which is the desired estimate for $\po^2\z$.

\subsubsection{Estimate for $\ddb_{\lb}(\Pi(\po^2N))$}

In view of the decomposition \eqref{poo6} for $\po^2N$, we have:
$$\dd_{\lb}(\Pi(\po^2N))=\dd_{\lb}(\po^2N)+2g(\po N,\dd_{\lb}(\po N))N+|\po N|^2\dd_{\lb}N$$
which yields:
\bea
\lab{poo44}\ddb_{\lb}(\Pi(\po^2N))&=&\Pi(\dd_{\lb}(\po^2N))+|\po N|^2\Pi(\dd_{\lb}N)\\
\nn&=& \Pi(\dd_{\lb}(\po^2N))+|\po N|^2(\z_A-\xib_A)e_A.
\eea
where we used the Ricci equations \eqref{ricciform} for $\dd_{\lb}N$ in the last equality. The formula 
\eqref{poo3} for $\dd_{\lb}(\po^2N)$ and \eqref{poo44} imply:
\bea\lab{poo45}
\nn\ddb_{\lb}(\Pi(\po^2N))&=&2\po^2\z_Be_B-|\po N|^2\z_Be_B+2\po\chi_{\po NB}e_B+\chi_{\Pi(\po^2N)B}e_B+(\d+n^{-1}\nabla_Nn)\\
&&\times\Pi(\po^2N)+(4\kep_{\po N}+n^{-1}\nabla_{\po N}n)\po N+|\po N|^2(\z_A-\xib_A)e_A.
\eea
Now, let $2\leq p<+\infty$. \eqref{poo45}, the estimate \eqref{poo43} for $\norm{P_j(\Pi(\po^2\z))}_{\tx{p}{2}}$, together with the $L^2$ boundedness and the weak Bernstein inequality for $P_j$, yields:
\bea
\lab{poo46}&&\norm{P_j\ddb_{\lb}(\Pi(\po^2N))}_{\tx{p}{2}}\\
\nn&\les&\norm{P_j\Pi(\po^2\z)}_{\tx{p}{2}}+\norm{P_j(|\po N|^2\z)}_{\tx{\infty}{2}}+\norm{P_j(\po\chi\po N)}_{\tx{\infty}{2}}\\
\nn&&+\norm{P_j(\chi\Pi(\po^2N))}_{\tx{\infty}{2}}+\norm{P_j((\d+n^{-1}\nabla_Nn)\Pi(\po^2N))}_{\tx{\infty}{2}}\\
\nn&&+\norm{P_j((4\kep_{\po N}+n^{-1}\nabla_{\po N}n)\po N}_{\tx{\infty}{2}}+\norm{P_j(|\po N|^2(\z-\xib)}_{\tx{\infty}{2}}\\
\nn&\les&2^j\ep+2^{\frac{j}{2}}\ep\gamma(u)+\norm{|\po N|^2\z}_{\tx{\infty}{2}}+\norm{\po\chi\po N}_{\tx{\infty}{2}}+2^{\frac{j}{2}}\norm{\chi\Pi(\po^2N)}_{\tx{\infty}{\frac{4}{3}}}\\
\nn&&+2^{\frac{j}{2}}\norm{(\d+n^{-1}\nabla_Nn)\Pi(\po^2N)}_{\tx{\infty}{\frac{4}{3}}}+\norm{(4\kep_{\po N}+n^{-1}\nabla_{\po N}n)\po N}_{\tx{\infty}{2}}\\
\nn&&+\norm{|\po N|^2(\z-\xib)}_{\tx{\infty}{2}}\\
\nn&\les&2^j\ep+2^{\frac{j}{2}}\ep\gamma(u)+\norm{\po N}^2_{L^\infty}(\norm{\z}_{\tx{\infty}{2}}+
\norm{\kep}_{\tx{\infty}{2}}+\norm{n^{-1}\nabla n}_{\tx{\infty}{2}}+\norm{\xib}_{\tx{\infty}{2}})\\
\nn&&+\norm{\po N}_{L^\infty}\norm{\po\chi}_{\tx{\infty}{2}}
+2^{\frac{j}{2}}(\norm{\chi}_{\tx{\infty}{4}}+\norm{\d}_{\tx{\infty}{4}}+\norm{n^{-1}\nabla n}_{\tx{\infty}{4}})\norm{\Pi(\po^2N)}_{\tx{\infty}{2}}\\
\nn&\les&2^j\ep+2^{\frac{j}{2}}\ep\gamma(u),
\eea
where we used in the last inequality the estimates \eqref{estn}-\eqref{estzeta} for $n, \d, \kep, \chi, \xib$ and $\z$, the estimates \eqref{estNomega} and \eqref{estricciomega} for $\po N$ and $\po\chi$, and the estimate \eqref{poo25} for $\po^2N$. \eqref{poo46} is the desired estimate for $\ddb_{\lb}(\Pi(\po^2N))$. 

In view of the estimates \eqref{poo25}, \eqref{poo26}, \eqref{fuite2}, \eqref{poo31}, \eqref{poo43} and \eqref{poo46}, this concludes the proof of Theorem \ref{thregomega2}.

\section{Dependance of the norm $L^\infty_u\lh{2}$ on $\o\in\S$}\lab{sec:depnormonomega}

The goal of this section is to derive the various decompositions of section \ref{sec:obadidonc}. In section \ref{sec:ohlalala}, we derive the basic estimates, first for scalars, and then for tensors using a scalarization procedure. In section \ref{sec:ohlalala1}, we obtain the desired decompositions for $\po N$, $\trc$ and $b^p$. In section \ref{sec:ohlalala2}, we provide variants of the results in section \ref{sec:ohlalala}. In section \ref{sec:ohlalala3}, we obtain the desired decompositions for $\chi$, $\hch^2$ and $\hch^3$. In section \ref{sec:ohlalala4}, we provide further variants of the results in section \ref{sec:ohlalala}. Finally, the desired decompositions for $\z$, $\nabb b$ and $\po b$ are derived in section \ref{sec:ohlalala5}. 

\subsection{The basic estimates}\lab{sec:ohlalala}

The goal of this section in to prove the following proposition.
\begin{proposition}\lab{prop:xx1}
Let $f(.,\o)$ a scalar function depending on a parameter $\o\in\S$ such that: 
$$\norm{f}_{L^\infty_u\lh{2}}+\norm{\dd f}_{L^\infty_u\lh{2}}+\norm{\po f}_{L^\infty_u\lh{2}}\les \ep.$$
Assume also that the existence of a function $\gamma$ in $L^2(\R)$ such that for all $j\geq 0$, we have:
$$\norm{P_j(L(\po f))}_{\lh{2}}+\norm{P_j(\lb(\po f))}_{\lh{2}}\les 2^j\ep+2^{\frac{j}{2}}\gamma(u)\ep.$$
Let $\o$ and $\o'$ in $\S$. Let $u=u(t,x,\o)$ and $u'=u(t,x,\o')$. Then, for any $\o''$ in $\S$ on the arc joining $\o$ and $\o'$, and for any $j\geq 0$, we have the following decomposition for $f(.,\o'')$:
$$f(.,\o'')=P_{\leq \frac{j}{2}}(f(.,\o'))+f^j_2$$
and where $f^j_2$ satisfies:
$$\norm{f^j_2}_{L^\infty_u\lh{2}}\les 2^{-\frac{j}{2}}\ep+|\o-\o'|^{\frac{3}{2}}2^{\frac{j}{4}}\ep.$$
\end{proposition}

As a corollary of Proposition \ref{prop:xx1}, we obtain:
\begin{corollary}\lab{cor:xx1}
Let $F(.,\o)$ a tensor depending on a parameter $\o\in\S$ such that:
$$\norm{F}_{L^\infty_u\lh{2}}+\norm{F}_{\tx{2}{\infty}}+\norm{\dd F}_{L^\infty_u\lh{2}}+\norm{\po F}_{L^\infty_u\lh{2}}\les \ep.$$
Assume also that the existence of a function $\gamma$ in $L^2(\R)$ such that for all $j\geq 0$ and for some $2<p\leq+\infty$, we have:
$$\norm{P_j(\ddb_L(\po F))}_{\tx{p}{2}}+\norm{P_j(\ddb_{\lb}(\po F))}_{\tx{p}{2}}\les 2^j\ep+2^{\frac{j}{2}}\gamma(u)\ep.$$
Let $\o$ and $\o'$ in $\S$. Let $u=u(t,x,\o)$ and $u'=u(t,x,\o')$. Then, for any $\o''$ in $\S$ on the arc joining $\o$ and $\o'$, and for any $j\geq 0$, we have the following decomposition for $F(.,\o'')$:
$$F(.,\o'')=F^j_1+F^j_2$$
where $F^j_1$ does not depend on $\o$ and satisfies:
$$\norm{F^j_1}_{L^\infty(P_{t, u_{\o'}})}\les \norm{F}_{L^\infty(P_{t, u_{\o'}})},$$
and where $F^j_2$ satisfies:
$$\norm{F^j_2}_{L^\infty_u\lh{2}}\les 2^{-\frac{j}{2}}\ep+|\o-\o'|^{\frac{3}{2}}2^{\frac{j}{4}}\ep.$$
\end{corollary}

The following lemmas will be useful for the proof of Proposition \ref{prop:xx1} and Corollary \ref{cor:xx1}.
\begin{lemma}\lab{lemma:xx2}
Let $\o$ and $\o'$ in $\S$. Let $u=u(t,x,\o)$ and $u'=u(t,x,\o')$. Then, 
for any tensor $F$, we have:
$$
\norm{F}_{L^\infty_{u'}L^2(\H_{u'})}\les \norm{F}_{L^\infty_uL^2(\H_u)}+|\o-\o'|^{\frac{1}{4}}\norm{F}^{\frac{1}{2}}_{L^\infty_uL^2(\H_u)}\left(\sup_u\left(\int_u^{u+|\o-\o'|}\norm{\dd F}^2_{L^2(\H_{\tau})}d\tau\right)^{\frac{1}{2}}\right)^{\frac{1}{2}}.
$$
\end{lemma}

\begin{lemma}\lab{lemma:xx3}
Let $f$ a scalar function and $\o, \o'$ in $\S$. Then, for any $l\geq 0$, we have:
$$\norm{P_lf}_{L^\infty_{u'}L^2(\H_{u'})}\les (2^{-l}+|\o-\o'|^{\half}2^{-\frac{l}{2}})(\norm{f}_{L^\infty_uL^2(\H_u)}+\norm{\dd f}_{L^\infty_uL^2(\H_u)}).$$
\end{lemma}

\begin{lemma}\lab{lemma:xx4}
Let $f$ a scalar function and $\o, \o'$ in $\S$. Then, for any $l\geq 0$, we have:
\bee
&&\norm{P_{\leq l}f}_{L^\infty_{u'}L^2(\H_{u'})}\\
&\les& (1+|\o-\o'|^{\half}2^{\frac{l}{2}})\norm{f}_{L^\infty_uL^2(\H_u)}+|\o-\o'|^{\frac{1}{4}}\norm{f}^{\half}_{L^\infty_uL^2(\H_u)}\\
&&\times\left(\sup_u\sum_{q\leq l}\left(\int_u^{u+|\o-\o'|}(\norm{P_q(nL(f))}_{L^2(\H_{\tau})}^2+\norm{P_q(bN(f))}_{L^2(\H_{\tau})}^2)\right)^{\frac{1}{2}}d\tau\right)^{\frac{1}{2}}.
\eee
\end{lemma}

\begin{lemma}\lab{lemma:xx5}
Let $f$ a scalar function and $\o, \o'$ in $\S$. Then, for any $l\geq 0$, we have:
$$\norm{[\po,P_{\leq l}]f}_{L^\infty_{u'}L^2(\H_{u'})}\les \norm{\dd f}_{L^\infty_uL^2(\H_u)}.$$
\end{lemma}

\begin{lemma}\lab{lemma:xx6}
We have:
$$\norm{\dd_TQ_{\leq 1}(N)}_{\lsit{\infty}{2}}+\norm{\nabla\dd_TQ_{\leq 1}(N)}_{\lsit{\infty}{2}}\les \ep.$$
\end{lemma}

\begin{lemma}\lab{lemma:xx7}
Let $N_j=N(.,\o_j), j=1, 2, 3$ where $\o_j\in\S$ are given respectively by $\o_1=(1,0,0)$, $\o_2=(0,1,0)$ and $\o_3=(0,0,1)$. Then, $Q_{\leq 1}(N_1), Q_{\leq 1}(N_2)$ and $Q_{\leq 1}(N_3)$ form a basis of the tangent space of $\Sit$.
\end{lemma}

We also state the following lemma which will be used in the proof of Lemma \ref{lemma:xx5}. Note this lemma, together with Lemma \ref{lemma:xx2}, is at the core of all decompositions of section \ref{sec:depnormonomega}.

\begin{lemma}\label{lemma:xx8}
Let $\o$ and $\o'$ in $\S$. Let $u=u(t,x,\o)$ and $u'=u(t,x,\o')$. Then, 
for any tensor $F$ and any $2\leq p<+\infty$, we have:
$$
\norm{F}_{L^\infty_{u'}L^p(\H_{u'})}\les \norm{F}^{1-\frac{1}{p}}_{L^\infty_uL^{2(p-1)}(\H_u)}\norm{\nabb F}^{\frac{1}{p}}_{L^\infty_uL^2(\H_u)}.
$$
\end{lemma}

The proof of Corollary \ref{cor:xx1} is postponed to section \ref{sec:xx1bis}, the proof of Lemma \ref{lemma:xx2} is postponed to section \ref{sec:xx2}, the proof of Lemma \ref{lemma:xx3} is postponed to section \ref{sec:xx3}, the proof of Lemma \ref{lemma:xx4} is postponed to section \ref{sec:xx4}, the proof of Lemma \ref{lemma:xx5} is postponed to section \ref{sec:xx5}, the proof of Lemma \ref{lemma:xx6} is postponed to section \ref{sec:xx6}, the proof of Lemma \ref{lemma:xx7} is postponed to section \ref{sec:xx7}, and the proof of Lemma \ref{lemma:xx8} is postponed to section \ref{sec:xx8}. We now conclude the proof of Proposition \ref{prop:xx1}.

\subsubsection{Proof of Proposition \ref{prop:xx1}}

We decompose $f(.,\o'')$ as:
\bea
\lab{xx1}f(.,\o'')&=&P''_{\leq\frac{j}{2}}(f(.,\o''))+\sum_{l>\frac{j}{2}}P''_l(f(.,\o'')\\
\nn&=&P'_{\leq\frac{j}{2}}(f(.,\o'))+\int_{[\o',\o'']}\po P'''_{\leq\frac{j}{2}}(f(.,\o'''))d\o'''(\o'-\o'')+\sum_{l>\frac{j}{2}}P''_l(f(.,\o''))\\
\nn&=&P'_{\leq\frac{j}{2}}(f(.,\o'))+\int_{[\o',\o'']}(P'''_{\leq\frac{j}{2}}(\po f)(.,\o''')+[\po,P'''_{\leq\frac{j}{2}}]f(.,\o'''))d\o'''(\o'-\o'')\\
\nn&&+\sum_{l>\frac{j}{2}}P''_l(f(.,\o'')),
\eea
where $\o'''$ denotes an angle in $\S$ on the arc joining $\o'$ and $\o''$.

Next, we estimate the last two terms in the right-hand side of \eqref{xx1}. Using Lemma \ref{lemma:xx3}, we have:
\bea
\nn\sum_{l>\frac{j}{2}}\norm{P''_l(f(.,\o''))}_{L^\infty_u\lh{2}}&\les& \sum_{l>\frac{j}{2}}(2^{-l}+|\o''-\o|^{\half}2^{-\frac{l}{2}})(\norm{f}_{L^\infty_u\lh{2}}+\norm{\dd f}_{L^\infty_u\lh{2}})\\
\label{xx3}&\les& (2^{-\frac{j}{2}}+|\o''-\o|^{\half}2^{-\frac{j}{4}})\ep,
\eea
where we used the assumptions on $f$ in the last inequality.

Using Lemma \ref{lemma:xx4}, we have:
\bea
\label{xx4}&&\norm{P'''_{\leq\frac{j}{2}}(\po f)(.,\o''')}_{L^\infty_u\lh{2}}\\
\nn&\les& (1+|\o'''-\o|^{\half}2^{\frac{j}{4}})\norm{\po f}_{L^\infty_u\lh{2}}+|\o'''-\o|^{\frac{1}{4}}\norm{\po f}^{\half}_{L^\infty_uL^2(\H_u)}\\
\nn&&\times\left(\sup_u\sum_{q\leq \frac{j}{2}}\left(\int_u^{u+|\o-\o'''|}(\norm{P_q(nL(\po f))}_{L^2(\H_{\tau})}^2+\norm{P_q(bN(\po f))}_{L^2(\H_{\tau})}^2)d\tau\right)^{\frac{1}{2}}\right)^{\frac{1}{2}}\\
\nn&\les& (1+|\o'''-\o|^{\half}2^{\frac{j}{4}})\ep+|\o-\o'''|^{\frac{1}{4}}\ep^{\half}\\
\nn&&\times\left(\sup_u\sum_{q\leq \frac{j}{2}}\left(\int_u^{u+|\o-\o'''|}(\norm{P_q(nL(\po f))}_{L^2(\H_{\tau})}^2+\norm{P_q(bN(\po f))}_{L^2(\H_{\tau})}^2)d\tau\right)^{\frac{1}{2}}\right)^{\frac{1}{2}},
\eea
where we used the assumptions on $\po f$ in the last inequality. Now, the assumption on $L(\po f)$ and $\lb(\po f)$ together with Lemma \ref{lemma:lbt1} yields:
\bee
&&\norm{P_q(nL(\po f))}_{L^2(\H_{\tau})}^2+\norm{P_q(bN(\po f))}_{L^2(\H_{\tau})}^2\\
&\les& (\norm{n}_{L^\infty}+\norm{\nabb n}_{\PP^0}+\norm{b}_{L^\infty}+\norm{\nabb b}_{\PP^0})^2(2^{2q}\ep^2+2^q\ga(u)^2)\\
&\les& 2^{2q}\ep^2+2^q\ga(u)^2,
\eee
where we used in the last inequality the estimate \eqref{estn} for $n$ and the estimate \eqref{estb} for $b$. Together with \eqref{xx4}, this implies:
\bea
\label{xx5}\norm{P'''_{\leq\frac{j}{2}}(\po f)(.,\o''')}_{L^\infty_u\lh{2}}&\les& (1+|\o'''-\o|^{\frac{1}{4}}2^{\frac{j}{8}}+|\o'''-\o|^{\half}2^{\frac{j}{4}})\ep\\
\nn&\les& (1+|\o'''-\o|^{\half}2^{\frac{j}{4}})\ep.
\eea

Using Lemma \ref{lemma:xx5}, we have:
\be\label{xx6}
\norm{[\po,P'''_{\leq\frac{j}{2}}]f(.,\o'''))}_{L^\infty_u\lh{2}}\les \norm{\dd f}_{L^\infty_u\lh{2}}\les \ep,
\ee
where we used the assumptions on $f$ in the last inequality.

In view of \eqref{xx1}, we have $f(.,\o'')=f^1_j+f_2^j$ where $f_1^j$ is defined as:
\be\label{xx7}
f^1_j=P'_{\leq\frac{j}{2}}(f(.,\o')),
\ee
and $f^2_j$ is defined as:
\be\label{xx8}
f^2_j=\int_{[\o',\o'']}(P'''_{\leq\frac{j}{2}}(\po f)(.,\o''')+[\po,P'''_{\leq\frac{j}{2}}]f(.,\o'''))d\o'''(\o'-\o'')+\sum_{l>\frac{j}{2}}P''_l(f(.,\o'')).
\ee
Using \eqref{xx3}, \eqref{xx5} and \eqref{xx6}, and the fact that $\o''$ is on the arc of $\S$ joining $\o$ and $\o'$, we have the following estimate for $f^2_j$:
\bea
\nn\lab{xx10}\norm{f^2_j}_{L^\infty_u\lh{2}}&\les& \int_{[\o',\o'']} (1+|\o'''-\o|^{\half}2^{\frac{j}{4}})\ep d\o'''|\o'-\o''|+(2^{-\frac{j}{2}}+|\o''-\o|^{\half}2^{-\frac{j}{4}})\ep\\
\nn&\les& (1+|\o'-\o|^{\half}2^{\frac{j}{4}}) |\o'-\o|\ep+(2^{-\frac{j}{2}}+|\o'-\o|^{\half}2^{-\frac{j}{4}})\ep\\
&\les& 2^{-\frac{j}{2}}\ep+|\o'-\o|^{\frac{3}{2}}2^{\frac{j}{4}}\ep.
\eea
This concludes the proof of Proposition \ref{prop:xx1}.

\subsubsection{Proof of Corollary \ref{cor:xx1}}\lab{sec:xx1bis}

Using Lemma \ref{lemma:xx7}, it suffices to prove the decomposition of Corollary \ref{cor:xx1} where $F(\o'', .)$ is replaced by $g(F(.,\o''),Q_{\leq 1}(N_l))$ for $l=1, 2, 3$. Since the proof is identical for $l=1, 2, 3$, we simply take $l=1$. Therefore, it remains to prove that the following decomposition holds $g(F(., \o''),Q_{\leq 1}(N_1))$:
\be\lab{xx18}
g(F(.,\o''),Q_{\leq 1}(N_1))=P_{\leq \frac{j}{2}}(g(F(.,\o'),Q_{\leq 1}(N_1)))+f_2^j,
\ee
where the scalar function $f_2^j$ satisfies:
\be\lab{xx20}
\norm{f_2^j}_{L^\infty_u\lh{2}}\les 2^{-\frac{j}{2}}\ep+|\o-\o'|^{\frac{3}{2}}2^{\frac{j}{4}}\ep.
\ee
In particular, $F_1^j$ is connected to the first term in the right-hand side of \eqref{xx18}, which does not  depend on $\o$ and satisfies the following estimate
$$\normm{P_{\leq \frac{j}{2}}(g(F(.,\o'),Q_{\leq 1}(N_1)))}_{L^\infty(P_{t, u_{\o'}})}\les \norm{F}_{L^\infty(P_{t, u_{\o'}})}\norm{Q_{\leq 1}(N_1)}_{L^\infty}\les \norm{F}_{L^\infty(P_{t, u_{\o'}})},$$
where we used the fact that $P_{\leq \frac{j}{2}}$ is bounded on $L^\infty(P_{t, u_{\o'}})$ and the fact that $Q_{\leq 1}$ is bounded on $L^\infty$. 

Let $f=g(F(.,\o),Q_{\leq 1}(N_1))$. In order to prove the decomposition \eqref{xx18} \eqref{xx20} for $g(F(.,\o''),Q_{\leq 1}(N_1))$, it suffices to show that $f$ satisfies the assumptions of Proposition \ref{prop:xx1}. First, we estimate $\dd f$. We have:
\bee
\norm{\dd f}_{L^\infty_u\lh{2}}&\les& \norm{\dd F}_{L^\infty_u\lh{2}}\norm{Q_{\leq 1}(N_1)}_{L^\infty}+\norm{F}_{\tx{2}{\infty}}\norm{\dd Q_{\leq 1}(N_1)}_{\tx{\infty}{2}}\\
&\les & \ep+\norm{\dd Q_{\leq 1}(N_1)}_{\tx{\infty}{2}}
\eee
where we used in the last inequality the assumptions on $F$, and the fact that $Q_{\leq 1}$ is bounded on $L^\infty$. Using the functional inequality \eqref{hehehehe0bis}, we obtain:
\bea\label{xx22bis}
&&\norm{\dd Q_{\leq 1}(N_1)}_{\tx{\infty}{2}}\\
\nn&\les& \norm{\dd Q_{\leq 1}(N_1)}_{\lsit{\infty}{2}}+\norm{\nabla\dd Q_{\leq 1}(N_1)}_{\lsit{\infty}{2}}\\
\nn&\les &\norm{\nabla Q_{\leq 1}(N_1)}_{\lsit{\infty}{2}}+\norm{\dd_TQ_{\leq 1}(N_1)}_{\lsit{\infty}{2}}+\norm{\nabla^2 Q_{\leq 1}(N_1)}_{\lsit{\infty}{2}}\\
\nn&&+\norm{\nabla\dd_TQ_{\leq 1}(N_1)}_{\lsit{\infty}{2}}\\
\nn&\les &\norm{\nabla Q_{\leq 1}(N_1)}_{\lsit{\infty}{2}}+\norm{\dd_TQ_{\leq 1}(N_1)}_{\lsit{\infty}{2}}+\norm{\nabla\dd_TQ_{\leq 1}(N_1)}_{\lsit{\infty}{2}}\\
\nn&\les & \ep,
\eea
where we used the Bochner identity on $\Sit$ \eqref{prop:bochsit}, the finite band property for $Q_{\leq 1}$, and Lemma \ref{lemma:xx6}. Finally, we obtain:
\be\lab{xx22}
\norm{\dd f}_{L^\infty_u\lh{2}}\les \ep.
\ee

Next, we estimate $\po f$. We have:
$$\po f=g(\po F,Q_{\leq 1}(N_1)),$$
which yields:
\be\lab{xx23}
\norm{\po f}_{L^\infty_u\lh{2}}\les \norm{\po F}_{L^\infty_u\lh{2}}\norm{Q_{\leq 1}(N_1)}_{L^\infty}\les 1
\ee
where we used in the last inequality the assumptions on $\po F$, and the fact that $Q_{\leq 1}$ is bounded on $L^\infty$.

Finally, we estimate $L(\po f)$ and $\lb(\po f)$. The estimate for $L(\po f)$ being similar, we focus on $\lb(\po f)$. We have:
\be\lab{xx24}
\dd_{\lb}(\po f)=g(\dd_{\lb}(\po F) ,Q_{\leq 1}(N_1))+g(\po F,\dd_{\lb}Q_{\leq 1}(N_1)).
\ee
The estimate \eqref{ad3} yields:
$$\norm{\nabb Q_{\leq 1}(N_1)}_{\BB^0}\les\ep$$
which together with Lemma \ref{lemma:poo6} and the assumption for $\dd_{\lb}\po F$ yields:
\be\lab{xx25}
\norm{P_j(g(\dd_{\lb}(\po^2N),Q_{\leq 1}(N_1)))}_{\lh{2}}\les 2^j\ep+2^{\frac{j}{2}}\ep\gamma(u).
\ee
Furthermore, using the dual of the sharp Bernstein inequality \eqref{eq:strongbernscalarbis}, we obtain:
\bea
\label{xx26}\norm{P_j(g(\po F,\dd_{\lb}Q_{\leq 1}(N_1)))}_{\lh{2}}&\les& 2^j\norm{g(\po F,\dd_{\lb}Q_{\leq 1}(N_1))}_{\tx{2}{1}}\\
\nn&\les & 2^j\norm{\po F}_{\li{\infty}{2}}\norm{\dd_{\lb}Q_{\leq 1}(N_1)))}_{\tx{\infty}{2}}\\
\nn&\les & 2^j\ep
\eea
where we used in the last inequality the assumptions for $\po F$, and the estimate \eqref{xx22bis}. Now, \eqref{xx24}-\eqref{xx26} yield:
$$\norm{P_j(\lb(\po f))}_{\lh{2}}\les 2^j\ep+2^{\frac{j}{2}}\ep\gamma(u).$$
The corresponding estimate for $L(\po f)$ may be obtained in the same way and is actually easier. Thus, we obtain:
\be\lab{xx27}
\norm{P_j(L(\po f))}_{\lh{2}}+\norm{P_j(\lb(\po f))}_{\lh{2}}\les 2^j\ep+2^{\frac{j}{2}}\ep\gamma(u).
\ee
In view of \eqref{xx22}, \eqref{xx23}, and \eqref{xx27}, $f$ satisfies the assumptions of Proposition \ref{prop:xx1}, which in turn yields the decomposition \eqref{xx18}-\eqref{xx20} for $g(F(.,\o''),Q_{\leq 1}(N_1))$. This concludes the proof of Corollary \ref{cor:xx1}.

\subsection{Decompositions involving $\po N$, $\trc$ and $b^p$}\lab{sec:ohlalala1}

In this section, we obtain the proof of Proposition \ref{cor:xx1bis}, Proposition \ref{cor:so} and Proposition \ref{cor:so1} as a consequence of Proposition \ref{prop:xx1}, Corollary \ref{cor:xx1}, and Lemma \ref{lemma:xx8}.

\subsubsection{Proof of Proposition \ref{cor:xx1bis}}

We have:
\be\label{xx11bis}
N-N'=\int_{[\o,\o']}\po N(.,\o'')d\o''(\o-\o').
\ee
We denote $\po N''=\po N(.,\o'')$. Now, in view of the estimates \eqref{estNomega} and \eqref{estricciomega} for $\po N$, and \eqref{estNomega2}, \eqref{estNomega2bis} and \eqref{estricciomega2} for $\po^2N$, $\po N$ satisfies the assumptions of Corollary \ref{cor:xx1}. Thus, we have the following decomposition for $\po N''$
 \be\lab{xx12bis}
\po N''=F_1^j+F_2^j,
\ee
where the vectorfield $F_1^j$ only depends on $\o'$ and satisfies:
\be\lab{xx13bis}
\norm{F_1^j}_{L^\infty}\les \norm{\po N'}_{L^\infty}\lesssim 1
\ee
in view of \eqref{estNomega}, and where the vectorfield $F_2^j$ satisfies:
\be\lab{xx14bis}
\norm{F_2^j}_{L^\infty_u\lh{2}}\les 2^{-\frac{j}{2}}\ep.
\ee
Injecting the decomposition \eqref{xx12bis} in \eqref{xx11bis}, and in view of \eqref{xx13bis} \eqref{xx14bis}, 
we obtain the desired decomposition for $N-N'$. This concludes the proof of the proposition. 

\subsubsection{Proof of Proposition \ref{cor:so}}

In view of the estimates \eqref{esttrc}, \eqref{estricciomega} and \eqref{estricciomegabis} for $\trc$, $f=\trc$ satisfies the assumption of Proposition \ref{prop:xx1}. Thus, in view of Proposition \ref{prop:xx1}, $\trc(.,\o)$ satisfies the desired decomposition with 
$$f^j_1=P_{\leq \frac{j}{2}}(\trc(.,\o')).$$
There remains to prove the $L^\infty$ estimate for $f^j_1$ which is an immediate consequence of the estimate \eqref{esttrc} for $\trc$ and the fact that $P_{\leq \frac{j}{2}}$ is bounded on $\lpt{\infty}$. This concludes the proof of the proposition.

\subsubsection{Proof of Proposition \ref{cor:so1}}

We have
\bea\lab{lebronnn}
\nn\norm{b^p(.,\o)-b^p(.,\o')}_{L^\infty_u\lh{2}}&\les&\left(\int_{[\o,\o']}\norm{b^{p-1}(.,\o'')\po b(.,\o'')}_{L^\infty_u\lh{2}}d\o''\right)|\o-\o'|\\
&\les&\left(\int_{[\o,\o']}\norm{\po b(.,\o'')}_{L^\infty_u\lh{2}}d\o''\right)|\o-\o'|,
\eea
where we used in the last inequality the estimate \eqref{estb} for $b$. Now, using Lemma \ref{lemma:xx8} with $p=2$, we have
$$\norm{\po b(.,\o'')}_{L^\infty_u\lh{2}}\les \norm{\po b}_{L^\infty_uL^2(\H_u)}+\norm{\nabb\po b}_{L^\infty_uL^2(\H_u)}\les\ep$$
where we used the estimate \eqref{estricciomega} for $\po b$ in the last inequality. Together with \eqref{lebronnn}, we obtain
$$\norm{b^p(.\o)-b^p(.,\o')}_{L^\infty_u\lh{2}}\les |\o-\o'|\ep,$$
which concludes the proof of the proposition.

\subsection{A first variant of Proposition \ref{prop:xx1}}\lab{sec:ohlalala2}

We start with the following refinement of Lemma \ref{lemma:xx8}:
\begin{corollary}\label{lin}
Let $\o$ and $\o'$ in $\S$. Let $u=u(t,x,\o)$ and $u'=u(t,x,\o')$. Then, 
for any tensor $F$, and for any $2\leq p, q\leq +\infty$, we have:
$$
\norm{F}_{L^\infty_{u'}L^2(\H_{u'})}\les \norm{F}^{\frac{1}{2}}_{\tx{p}{q}}\norm{\nabb F}^{\frac{1}{2}}_{\tx{\frac{p}{p-1}}{\frac{q}{q-1}}}.
$$
\end{corollary}

\begin{proof}
Let $f$ a scalar. Then, using a standard estimate in $\R^2$, we have the analog of \eqref{tolbiac}
\bee
&&\int_{y_1} \sup_{y_2}|f(\Phi^{-1}_{t,\o}(u,y_1,y_2))|^2dy_1\\
\nn&\les& \left(\int_{y}|f(\Phi^{-1}_{t,\o}(u,y_1,y_2))|^qdy_1dy_2\right)^{\frac{1}{q}}\left(\int_{y}|\partial_{y_2}f(\Phi^{-1}_{t,\o}(u,y_1,y_2))|^{\frac{q}{q-1}}dy_1dy_2\right)^{1-\frac{1}{q}}\\
\nn&\les& \left(\int_{\ptu}|f|^q\dmt\right)^{\frac{1}{q}}\left(\int_{\ptu}|\nabb f|^{1-\frac{1}{q}}\dmt\right)^{\half}.
\eee
Together with \eqref{lxx8:13}, this yields:
\bee
\norm{f}^2_{L^2(\H_{u'=u_0})}&\lesssim&  \frac{1}{|\o-\o'|}\left(\int_0^1\int _{u_0-|\o-\o'|}^{u_0+|\o-\o'|}
\left(\int_{\ptu}|f|^q\dmt\right)^{\frac{1}{q}}\left(\int_{\ptu}|\nabb f|^{\frac{q}{q-1}}\dmt\right)^{1-\frac{1}{q}}du dt\right)\\
&\lesssim&  \norm{f}_{\tx{p}{q}}\norm{\nabb f}_{\tx{\frac{p}{p-1}}{\frac{q}{q-1}}}.
\eee
Since this holds for any real number $u_0$, we take the supremum which yields:
$$\norm{f}^2_{L^{\infty}_{u'}L^2(\H_{u'})}\lesssim \norm{f}_{\tx{p}{q}}\norm{\nabb f}_{\tx{\frac{p}{p-1}}{\frac{q}{q-1}}}.$$
Finally, let $F$ a tensor. Applying the previous inequality to $f=|F|$, we obtain
$$\norm{F}^2_{L^{\infty}_{u'}L^2(\H_{u'})}\lesssim   \norm{F}_{\tx{p}{q}}\norm{\nabb F}_{\tx{\frac{p}{p-1}}{\frac{q}{q-1}}}.$$
This concludes the proof of the corollary.
\end{proof}

We will need the following refinement of Corollary \ref{cor:xx1}:
\begin{corollary}\lab{cor:leman}
Let $F(.,\o)$ a tensor depending on a parameter $\o\in\S$ such that for any $2\leq p<+\infty$:
$$\norm{F}_{L^\infty_u\lh{2}}+\norm{F}_{\tx{p}{\infty}}+\norm{\dd F}_{L^\infty_u\lh{2}}+\norm{\po F}_{L^\infty_u\lh{2}}\les \ep.$$
Assume that there exists two tensors $H_1$ and $H_2$ such that
$$\po F=H_1+H_2,$$
such that we have
$$\norm{H_1}_{\li{\infty}{2}}+\norm{H_2}_{\li{\infty}{2}}\les\ep,$$
and there exists a function $\gamma$ in $L^2(\R)$ such that for all $j\geq 0$ and for some $2<q\leq+\infty$, we have:
$$\norm{P_j(\ddb_L(H_1))}_{\tx{q}{2}}+\norm{P_j(\ddb_{\lb}(H_1))}_{\tx{q}{2}}\les 2^j\ep+2^{\frac{j}{2}}\gamma(u)\ep,$$
and such that $H_2$ satisfies for some $2\leq q<+\infty$
$$\norm{H_2}_{\tx{q}{\frac{8}{3}}}+\norm{\nabb H_2}_{\tx{\frac{q}{q-1}}{\frac{8}{5}}}\les\ep.$$
Let $\o$ and $\o'$ in $\S$. Let $u=u(t,x,\o)$ and $u'=u(t,x,\o')$. Then, for any $\o''$ in $\S$ on the arc joining $\o$ and $\o'$, and for any $j\geq 0$, we have the following decomposition for $F(.,\o'')$:
$$F(.,\o'')=F^j_1+F^j_2$$
where $F^j_1$ does not depend on $\o$ and satisfies such that for any $2\leq p<+\infty$:
$$\norm{F^j_1}_{L^\infty_{u{\o'}}L^p_tL^\infty(P_{t, u_{\o'}})}\les\ep,$$
and where $F^j_2$ satisfies:
$$\norm{F^j_2}_{L^\infty_u\lh{2}}\les 2^{-\frac{j}{2}}\ep+|\o-\o'|^{\frac{3}{2}}2^{\frac{j}{4}}\ep.$$
\end{corollary}

\begin{proof}
Using Lemma \ref{lemma:xx7}, it suffices to prove the decomposition of Corollary \ref{cor:leman} where $F(\o'', .)$ is replaced by $g(F(.,\o''),Q_{\leq 1}(N_l))$ for $l=1, 2, 3$. Since the proof is identical for $l=1, 2, 3$, we simply take $l=1$. Therefore, it remains to prove that the following decomposition holds $g(F(., \o''),Q_{\leq 1}(N_1))$:
\be\lab{leman1}
g(F(.,\o''),Q_{\leq 1}(N_1))=f^j_1+f_2^j,
\ee
where $f^j_1$ does not depend on $\o$ and satisfies such that for any $2\leq p<+\infty$:
\be\lab{leman2}
\norm{f^j_1}_{L^\infty_{u_{\o'}}L^p_tL^\infty(P_{t, u_{\o'}})}\les\ep,
\ee
and where the vectorfields $f_2^j$ satisfies:
\be\lab{leman3}
\norm{f_2^j}_{L^\infty_u\lh{2}}\les 2^{-\frac{j}{2}}\ep+|\o-\o'|^{\frac{3}{2}}2^{\frac{j}{4}}\ep.
\ee

Let $f(.,\o)=g(F(.,\o),Q_{\leq 1}(N_1))$. Arguing as in the proof of Corollary \ref{cor:xx1}, and using the assumptions for $F$, we have the analog of \eqref{xx22} and \eqref{xx23}:
\be\lab{leman4}
\norm{\dd f}_{L^\infty_u\lh{2}}\les \ep.
\ee
and 
\be\lab{leman5}
\norm{\po f}_{L^\infty_u\lh{2}}\les \ep.
\ee
Also, in view of the assumptions for $F$ and the fact that $Q_{\leq 1}$ is bounded on $L^\infty$, we have
\be\lab{leman6}
\norm{f(.,\o')}_{L^\infty_{u_{\o'}}L^p_tL^\infty(P_{t, u_{\o'}})}\les\norm{F(.,\o')}_{L^\infty_{u_{\o'}}L^p_tL^\infty(P_{t, u_{\o'}})}\norm{Q_{\leq 1}(N_1)}_{L^\infty}\les\ep.
\ee
In order to prove the decomposition \eqref{leman1} \eqref{leman2} \eqref{leman3} for $g(F(.,\o''),Q_{\leq 1}(N_1))$, we follow the proof of Proposition \ref{prop:xx1}. In particular, we recall the decomposition \eqref{xx1} of $f(.,\o'')$:
\bea\lab{leman7}
\nn f(.,\o'')&=&P'_{\leq\frac{j}{2}}(f(.,\o'))+\int_{[\o',\o'']}(P'''_{\leq\frac{j}{2}}(\po f)(.,\o''')+[\po,P'''_{\leq\frac{j}{2}}]f(.,\o'''))d\o'''(\o'-\o'')\\
&&+\sum_{l>\frac{j}{2}}P''_l(f(.,\o'')),
\eea
where $\o'''$ denotes an angle in $\S$ on the arc joining $\o'$ and $\o''$. Also, in view of the estimate \eqref{leman4}, we have the analog of the estimate \eqref{xx3}
\be\label{leman8}
\sum_{l>\frac{j}{2}}\norm{P''_l(f(.,\o''))}_{L^\infty_u\lh{2}}\les (2^{-\frac{j}{2}}+|\o''-\o|^{\half}2^{-\frac{j}{4}})\ep,
\ee
and the analog of the estimate \eqref{xx6} 
\be\label{leman9}
\norm{[\po,P'''_{\leq\frac{j}{2}}]f(.,\o'''))}_{L^\infty_u\lh{2}}\les \ep.
\ee
Also, using \eqref{leman6} and the fact that $P'_{\leq \frac{j}{2}}$ is bounded on $L^\infty(P_{t, u_{\o'}})$, we have for any $2\leq p<+\infty$:
\be\lab{leman10}
\normm{P'_{\leq\frac{j}{2}}(f(.,\o'))}_{L^\infty_{u_{\o'}}L^p_tL^\infty(P_{t, u_{\o'}})}\les\ep.
\ee

In view of \eqref{leman7}, we have $f(.,\o'')=f^1_j+f_2^j$ where $f_1^j$ is defined as:
$$f^1_j=P'_{\leq\frac{j}{2}}(f(.,\o')),$$
and $f^2_j$ is defined as:
\be\label{leman11}
f^2_j=f^j_{2,1}+f^j_{2,2}
\ee
with 
\be\lab{leman12}
f^j_{2,1}=\int_{[\o',\o'']}(P'''_{\leq\frac{j}{2}}(\po f)(.,\o''')(\o'-\o'')
\ee
and 
$$f^j_{2,2}=\int_{[\o',\o'']}[\po,P'''_{\leq\frac{j}{2}}](\po f)(.,\o'''))d\o'''(\o'-\o'')+\sum_{l>\frac{j}{2}}P''_l(f(.,\o'')).$$
In view of the definition of $f^j_1$, $f^j_{2,2}$ and the estimates \eqref{leman8}, \eqref{leman9} and \eqref{leman10}, $f^j_1$ does not depend on $\o$ and satisfies for any $2\leq p<+\infty$:
\be\lab{leman13}
\norm{f^j_1}_{L^\infty_{u_{\o'}}L^p_tL^\infty(P_{t, u_{\o'}})}\les\ep,
\ee
while $f^j_2$ satisfies:
\be\lab{leman14}
\norm{f^j_{2,2}}_{L^\infty_u\lh{2}}\les 2^{-\frac{j}{2}}\ep+|\o-\o'|^{\frac{3}{2}}2^{\frac{j}{4}}\ep.
\ee

We still need to estimate $f^j_{2,1}$. We have:
$$\po f=g(\po F,Q_{\leq 1}(N_1))$$
and thus
\be\lab{leman15}
\po f=h_1+h_2
\ee
where 
$$h_j=g(H_j,Q_{\leq 1}(N_1)), \, j=1,2,$$
Since the assumptions for $H_1$ in Corollary \ref{cor:leman} are the same as the assumptions for $\po F$ in Corollary \ref{cor:xx1}, we obtain the analog of \eqref{xx23} and \eqref{xx27} for $h_1$:
$$\norm{h_1}_{\li{\infty}{2}}\les\ep,$$
and
$$\norm{P_j(L(h_1))}_{\lh{2}}+\norm{P_j(\lb(h_1))}_{\lh{2}}\les 2^j\ep+2^{\frac{j}{2}}\ep\gamma(u).$$
Thus, the estimates for $h_1$ in Corollary \ref{cor:leman} are the same as the assumptions for $\po f$ in Proposition \ref{prop:xx1}, and we obtain the analog of \eqref{xx5}
\be\label{leman16}
\norm{P'''_{\leq\frac{j}{2}}h_1(.,\o''')}_{L^\infty_u\lh{2}}\les (1+|\o'''-\o|^{\half}2^{\frac{j}{4}})\ep.
\ee
Next, let $2\leq q<+\infty$. We have in view of Corollary \ref{lin} 
\bea\lab{leman17}
\norm{P'''_{\leq\frac{j}{2}}h_2(.,\o''')}_{L^\infty_u\lh{2}}&\les& \norm{P_{\leq\frac{j}{2}}h_2}^{\frac{1}{2}}_{\tx{q}{\frac{8}{3}}}\norm{\nabb P_{\leq\frac{j}{2}}h_2}^{\frac{1}{2}}_{\tx{\frac{q}{q-1}}{\frac{8}{5}}}\\
\nn&\les& \norm{h_2}^{\frac{1}{2}}_{\tx{q}{\frac{8}{3}}}\norm{\nabb h_2}^{\frac{1}{2}}_{\tx{\frac{q}{q-1}}{\frac{8}{5}}},
\eea
where we used in the last inequality the finite band property and the boundedness on $\lpt{2}$ of $P_{\leq \frac{j}{2}}$. Now, in view of the definition of $h_2$, we have
\bea\lab{leman18}
&&\norm{h_2}_{\tx{q}{\frac{8}{3}}}+\norm{\nabb h_2}_{\tx{\frac{q}{q-1}}{\frac{8}{5}}}\\
\nn&\les& \norm{Q_{\leq 1}(N_1)}_{L^\infty}(\norm{H_2}_{\tx{q}{\frac{8}{3}}}+\norm{\nabb H_2}_{\tx{q}{\frac{8}{3}}})+\norm{H_2}_{\tx{q}{\frac{8}{3}}}\norm{\nabb Q_{\leq 1}(N_1)}_{\tx{\infty}{4}}\\
\nn&\les& \ep+\ep\norm{\nabb Q_{\leq 1}(N_1)}_{\tx{\infty}{4}},
\eea
where we used in the last inequality the assumptions on $H_2$ and the fact that $Q_{\leq 1}$ is bounded on $L^\infty$. In order to estimate the right-hand side of \eqref{leman18}, we use the estimate \eqref{boarding}. We obtain
\bea\lab{coucoudeldoo}
\norm{\nabb Q_{\leq 1}(N_1)}_{\tx{\infty}{4}}&\les& \norm{\nabla\nabb Q_{\leq 1}(N_1)}_{L^\infty_t L^2(\Sigma_t)}+\norm{\nabb Q_{\leq 1}(N_1)}_{L^\infty_t L^2(\Sigma_t)}\\
\nn&\les& 1,
\eea
where we used in the last inequality the estimate \eqref{ad20}. Together with \eqref{leman18}, this yields
$$\norm{h_2}_{\tx{q}{\frac{8}{3}}}+\norm{\nabb h_2}_{\tx{\frac{q}{q-1}}{\frac{8}{5}}}\les \ep.$$
In view of \eqref{leman17} we deduce
\be\lab{leman19}
\norm{P'''_{\leq\frac{j}{2}}h_2(.,\o''')}_{L^\infty_u\lh{2}}\les \ep.
\ee
Now, \eqref{leman12}, \eqref{leman15} and \eqref{leman19} imply:
\be\lab{leman20}
\norm{f^j_{2,1}}_{L^\infty_u\lh{2}}\les |\o-\o'|\ep.
\ee
 Finally, \eqref{leman11}, \eqref{leman14} and \eqref{leman20} imply
$$\norm{f^j_2}_{L^\infty_u\lh{2}}\les 2^{-\frac{j}{2}}\ep+|\o-\o'|^{\frac{3}{2}}2^{\frac{j}{4}}\ep,$$
which together with the decomposition $f(.,\o'')=f^1_j+f_2^j$ and the estimate \eqref{leman13} yields the conclusion of the corollary.
\end{proof}

\subsection{Decompositions involving $\chi$}\lab{sec:ohlalala3}

The goal of this section is to prove the decompositions of Lemma \ref{cor:so2}, Proposition \ref{cor:so3}, Proposition \ref{cor:so4}, Proposition \ref{cor:so7} and Proposition \ref{cor:so9}. The proof of Lemma \ref{cor:so2} is given in section \ref{sec:pio1}, the proof of Proposition \ref{cor:so3} is given in section \ref{sec:pio2}, the proof of Proposition \ref{cor:so4} is given in section \ref{sec:pio5}, the proof of Proposition \ref{cor:so7} is given in section \ref{sec:pio6}, and the proof of Proposition \ref{cor:so9} is given in section \ref{sec:pio8}. 

We will need the following product lemma.
\begin{lemma}\lab{lemma:nicosi}
Let $F$ and $H$ $\ptu$-tangent tensors on $\mathcal{H}_u$ such that for any $2\leq r<+\infty$ we have
$$\norm{F}_{\tx{r}{\infty}}+\norm{\nabb F}_{L^r_tB^0_{2,1}(\ptu)}+\norm{H}_{\tx{r}{\infty}}+\norm{\nabb H}_{L^r_tB^0_{2,1}(\ptu)}\lesssim \ep.$$
Then, we have for any $2\leq r<+\infty$ we have
$$\norm{FH}_{\tx{r}{\infty}}+\norm{\nabb (FH)}_{L^r_tB^0_{2,1}(\ptu)}\lesssim\ep.$$
\end{lemma}

We will also need the following consequence of Corollary \ref{cor:leman} and Lemma \ref{lemma:nicosi}.
\begin{corollary}\lab{cor:so5}
Let $\o$ and $\o'$ in $\S$. For any $j\geq 0$ and any integer $l\geq 2$, we have the following decomposition for $\chi_1(.,\o)^l$:
$$\chi_1(.,\o)^l=F^j_1+F^j_2$$
where $F^j_1$ does not depend on $\o$ and satisfies for any $2\leq p<+\infty$:
$$\norm{F^j_1}_{L^\infty_{u_{\o'}}L^p_tL^\infty(P_{t, u_{\o'}})}\les \ep,$$
where $F^j_2$ satisfies:
$$\norm{F^j_2}_{L^\infty_u\lh{2}}\les 2^{-\frac{j}{2}}+|\o-\o'|^{\frac{3}{2}}2^{\frac{j}{4}}\ep.$$
\end{corollary}

We will need the following consequence of Lemma \ref{cor:so2} and Corollary \ref{cor:so5}: 
\begin{corollary}\lab{cor:so6}
Let $\o$ and $\o'$ in $\S$. For any $j\geq 0$ and any integer $l\geq 1$, we have the following decomposition for $\chi_1(.,\o)^l\chi_2(.,\o)$:
$$\chi_1(.,\o)^l\chi_2(.,\o)=\chi_2(.,\o')F^j_1+\chi_2(.,\o')F^j_2+F^j_3$$
where $F^j_1$ does not depend on $\o$ and satisfies for any $2\leq p<+\infty$:
$$\norm{F^j_1}_{L^\infty_{u_{\o'}}L^p_tL^\infty(P_{t, u_{\o'}})}\les \ep,$$
where $F^j_2$ and $F^j_3$ satisfy:
$$\norm{F^j_2}_{L^\infty_u\lh{2}}+\norm{F^j_3}_{L^\infty_u\lh{2}}\les 2^{-\frac{j}{2}}+|\o-\o'|^{\frac{3}{2}}2^{\frac{j}{4}}\ep.$$
\end{corollary}

Finally, we will need the following consequence in particular of Lemma \ref{cor:so2}: 
\begin{corollary}\lab{cor:so8}
Let $\o$ and $\o'$ in $\S$. For any $j\geq 0$, we have the following decomposition for $\chi_1(.,\o)\chi_2(.,\o)^2$:
$$\chi_1(.,\o)\chi_2(.,\o)^2=\chi_2(.,\o')^2F^j_1+\chi_2(.,\o')^2F^j_2+\chi_2(.,\o')F^j_3+F^j_4$$
where $F^j_1$ does not depend on $\o$ and satisfies:
$$\norm{F^j_1}_{L^\infty_{u_{\o'}}L^2_tL^\infty(P_{t, u_{\o'}})}\les \ep,$$
where $F^j_2$ and $F^j_3$ satisfy:
$$\norm{F^j_2}_{L^\infty_u\lh{2}}+\norm{F^j_3}_{L^\infty_u\lh{2}}\les 2^{-\frac{j}{2}}+|\o-\o'|^{\frac{3}{2}}2^{\frac{j}{4}}\ep,$$
and where $F^j_4$ satisfies
$$\norm{F^j_4}_{L^2(\mathcal{M})}\les \ep 2^{-j}.$$
\end{corollary}

The proof of Lemma \ref{lemma:nicosi} is postponed to section \ref{sec:gowinda10}, the proof of Corollary \ref{cor:so5} is postponed to section \ref{sec:pio4}, the proof of Corollary \ref{cor:so6} is postponed to section \ref{sec:pio5} and the proof of Corollary \ref{cor:so8} is postponed to section \ref{sec:pio7}.

\subsubsection{Proof of Lemma \ref{cor:so2}}\lab{sec:pio1}

We have
\be\lab{lebron}
\norm{\chi_2(.\o)-\chi_2(.,\o')}_{L^\infty_u\lh{4_-}}\les\left(\int_{[\o,\o']}\norm{\po\chi_2(.,\o'')}_{L^\infty_u\lh{4_-}}d\o''\right)|\o-\o'|.
\ee
Now, using Lemma \ref{lemma:xx8} with $p=4_-$, we have
$$\norm{\po\chi_2(.,\o'')}_{L^\infty_u\lh{4_-}}\les \norm{\po\chi_2}_{L^\infty_uL^{6_-}(\H_u)}+\norm{\nabb\po\chi_2}_{L^\infty_uL^2(\H_u)}\les\ep$$
where we used the estimate \eqref{dechch2} for $\po\chi_2$ in the last inequality. Together with \eqref{lebron}, we obtain
$$\norm{\chi_2(.\o)-\chi_2(.,\o')}_{L^\infty_u\lh{4_-}}\les |\o-\o'|\ep,$$
which concludes the proof of the lemma.

\subsubsection{Proof of Proposition \ref{cor:so3}}\lab{sec:pio2}

In view of the decomposition \eqref{chih} of $\chi$ in its trace part $\trc$ and traceless part $\hch$, in view of the decomposition \eqref{dechch} of $\hch$ in the sum of $\chi_1$ and $\chi_2$, and in view of the decomposition of 
Corollary \ref{cor:so} for $\trc$, it suffices to obtain the following decomposition for $\chi_1$
 \be\lab{lebron1}
\chi_1(.,\o)=F_1^j+F_2^j,
\ee
where the vectorfield $F_1^j$ only depends on $(t, x, \o')$ and satisfies for any $2\leq p<+\infty$:
\be\lab{lebron2}
\norm{F_1^j}_{L^\infty_{u_{\o'}}L^p_tL^\infty(P_{t, u_{\o'}})}\lesssim \ep
\ee
and where the vectorfield $F_2^j$ satisfies:
\be\lab{lebron3}
\norm{F_2^j}_{L^\infty_u\lh{2}}\les 2^{-\frac{j}{2}}\ep.
\ee
Now, in view of the estimates \eqref{dechch1}, \eqref{dechch2} and \eqref{dechch3} for $\chi_1$, $F=\chi_1$ satisfies for any $2\leq p<+\infty$:
$$\norm{F}_{L^\infty_u\lh{2}}+\norm{F}_{\tx{p}{\infty}}+\norm{\dd F}_{L^\infty_u\lh{2}}+\norm{\po F}_{L^\infty_u\lh{2}}\les \ep.$$
Also, we have
$$\po F=H_1+H_2\textrm{ with }H_1=\po\hch\textrm{ and }H_2=-\po\chi_2,$$
and $H_1$ and $H_2$ satisfy the assumption of Corollary \ref{cor:leman} in view of the estimates \eqref{estricciomega} and \eqref{estricciomegabis} for $\po\hch$ and the estimate \eqref{dechch2} for $\po\chi_2$. Thus, in view of Corollary \ref{cor:leman}, $\chi_1(.,\o)$ satisfies the decomposition \eqref{lebron1} and the estimates \eqref{lebron2} \eqref{lebron3}. This concludes the proof of the proposition.

\subsubsection{Proof of Proposition \ref{cor:so4}}\lab{sec:pio3}

In view of Corollary \ref{cor:so3}, it suffices to prove the decomposition for $chi_2$:
$$\chi_2(.,\o)=F^j_1+F^j_2$$
where $F^j_1$ does not depend on $\o$ and satisfies:
$$\norm{F^j_1}_{L^\infty_{u_{\o'}}L^\infty(P_{t, u_{\o'}})L^2_t}\les \ep,$$
and where $F^j_2$ satisfies:
$$\norm{F^j_2}_{L^\infty_u\lh{2}}\les 2^{-\frac{j}{2}}.$$
We choose
$$F^j_1=\chi_2(.,\o')\textrm{ and }F^j_2=\chi_2(., \o)-\chi_2(.,\o').$$
Then, the estimates for $F^j_1$ and $F^j_2$ follow from the estimate \eqref{dechch1} and the Lemma \ref{cor:so2} for $\chi_2$. This concludes the proof of the proposition.

\subsubsection{Proof of Corollary \ref{cor:so5}}\lab{sec:pio4}

In view of the estimates \eqref{dechch1}, \eqref{dechch2} and \eqref{dechch3} for $\chi_1$, $F=\chi_1^l$ satisfies for any $2\leq p<+\infty$:
$$\norm{F}_{L^\infty_u\lh{2}}+\norm{F}_{\tx{p}{\infty}}+\norm{\dd F}_{L^\infty_u\lh{2}}+\norm{\po F}_{L^\infty_u\lh{2}}\les \ep.$$
Also, we have
$$\po F=H_1+H_2\textrm{ with }H_1=l\chi_1^{l-1}\po\hch\textrm{ and }H_2=-l\chi_1^{l-1}\po\chi_2.$$
Lemma \ref{lemma:nicosi}, together with the estimates \eqref{dechch2} and \eqref{bispo40} for $\chi_1$ yields for any $2\leq r<+\infty$
$$\norm{\chi_1^{l-1}}_{\tx{r}{\infty}}+\norm{\nabb (\chi_1^{l-1})}_{L^r_tB^0_{2,1}(\ptu)}\lesssim\ep.$$
Together with Lemma \ref{lemma:grr} and the estimates \eqref{estricciomega} and \eqref{estricciomegabis} for $\po\hch$, we obtain:
$$\norm{P_j(\ddb_L(H_1))}_{\tx{q}{2}}+\norm{P_j(\ddb_{\lb}(H_1))}_{\tx{q}{2}}\les 2^j\ep+2^{\frac{j}{2}}\gamma(u)\ep.$$
Also, $H_2$ satisfies the following estimate
\bee
&&\norm{H_2}_{\tx{2}{\frac{8}{3}}}+\norm{\nabb H_2}_{\tx{2}{\frac{8}{5}}}\\
&\les& \norm{\chi_1^{l-1}}_{\tx{4}{\infty}}\norm{\po\chi_2}_{\tx{4}{\frac{8}{3}}}+\norm{\chi_1^{l-1}}_{\tx{\infty}{8}}\norm{\nabb\po\chi_2}_{\li{\infty}{2}}\\
&&+\norm{\chi_1^{l-2}}_{\tx{\infty}{16}}\norm{\nabb\chi_1}_{\tx{\infty}{2}}\norm{\po\chi_2}_{\tx{2}{16}}\\
&\les& \ep,
\eee
where we used in the last inequality the estimate \eqref{dechch2} for $\chi_1$ and $\po\chi_2$. 

Finally, we have proved that $F$, $H_1$ and $H_2$ satisfy the assumption of Corollary \ref{cor:leman}. Thus, we may apply Corollary \ref{cor:leman} to obtain the desired decomposition $\chi_1^l(.,\o)$. This concludes the proof of the corollary.

\subsubsection{Proof of Corollary \ref{cor:so6}}\lab{sec:pio5}

We decompose $\chi_1(.,\o)^l\chi_2(.,\o)$ as
\be\lab{okidoki}
\chi_1(.,\o)^l\chi_2(.,\o)=\chi_1(.,\o)^l\chi_2(.,\o')+\chi_1(.,\o)^l(\chi_2(.,\o)-\chi_2(.,\o')).
\ee
In view of Lemma \ref{cor:so2} and the estimate \eqref{dechch2} for $\chi_1$, we have
\bea\lab{okidoki1}
&&\norm{\chi_1(.,\o)^l(\chi_2(.,\o)-\chi_2(.,\o'))}_{\li{\infty}{2}}\\
\nn&\les& \norm{\chi_1(.,\o)}^l_{\li{\infty}{6l}}\norm{\chi_2(.,\o)-\chi_2(.,\o')}_{\li{\infty}{3}}\\
\nn&\les& |\o-\o'|\ep.
\eea
Finally, in view of the decomposition for $\chi_1(.,\o)^l$ provided by Corollary \ref{cor:so5}, \eqref{okidoki} and \eqref{okidoki1}, we obtain the desired decomposition for $\chi_1(.,\o)^l\chi_2(.,\o)$ with $F^j_1$ and $F^j_2$ defined in the statement of Corollary \ref{cor:so5}, and 
$$F^j_3=\chi_1(.,\o)^l(\chi_2(.,\o)-\chi_2(.,\o')).$$
This concludes the proof of the corollary.

\subsubsection{Proof of Proposition \ref{cor:so7}}\lab{sec:pio6}

In view of the decomposition \eqref{dechch1} for $\hch$, we decompose $\hch(.,\o)^2$ as
\be\lab{visitmed}
\hch(.,\o)^2=\chi_1(.,\o)^2+2\chi_1(.,\o)\chi_2(.,\o)+\chi_2(.,\o)^2.
\ee
We have
\be\lab{okidoki5}
\chi_2(.,\o)^2=\chi_2(.,\o')^2+\chi_2(.,\o')(\chi_2(.,\o)-\chi_2(.,\o'))+(\chi_2(.,\o)-\chi_2(.,\o'))^2.
\ee
Now, we have in view of Lemma \ref{cor:so2} and the estimate \eqref{dechch2} for $\chi_2$:
\be\lab{okidoki6}
\norm{\chi_2(.,\o)-\chi_2(.,\o')}_{\li{\infty}{2}}\les |\o-\o'|\ep\textrm{ and }\norm{(\chi_2(.,\o)-\chi_2(.,\o'))^2}_{L^2(\mathcal{M})}\les |\o-\o'|^2\ep.
\ee
Finally, in view of \eqref{visitmed}, Corollary \ref{cor:so5} with $l=2$, Corollary \ref{cor:so6} with $l=1$, \eqref{okidoki5} and \eqref{okidoki6}, we obtain the desired decomposition for $\hch^2$. 

\subsubsection{Proof of Corollary \ref{cor:so8}}\lab{sec:pio7}

We decompose $\chi_1(.,\o)\chi_2(.,\o)^2$ as
\bea\lab{okidoki2}
\chi_1(.,\o)\chi_2(.,\o)^2&=&\chi_1(.,\o)\chi_2(.,\o')^2+\chi_1(.,\o)\chi_2(.,\o')(\chi_2(.,\o)-\chi_2(.,\o'))\\
\nn&&+\chi_1(.,\o)(\chi_2(.,\o)-\chi_2(.,\o'))^2.
\eea
In view of Lemma \ref{cor:so2} and the estimate \eqref{dechch2} for $\chi_1$, we have
\bea\lab{okidoki3}
&&\norm{\chi_1(.,\o)(\chi_2(.,\o)-\chi_2(.,\o'))}_{\li{\infty}{2}}\\
\nn&\les& \norm{\chi_1(.,\o)}_{\li{\infty}{6}}\norm{\chi_2(.,\o)-\chi_2(.,\o')}_{\li{\infty}{3}}\\
\nn&\les& |\o-\o'|\ep.
\eea
Also, in view of the estimate \eqref{dechch2} for $\chi_1$ and $\chi_2$, we have
\bea\lab{okidoki4}
\nn\norm{\chi_1(.,\o)(\chi_2(.,\o)-\chi_2(.,\o'))^2}_{L^2(\mathcal{M})}&\les& \norm{\chi_1(.,\o)}_{L^{10}(\mathcal{M})}\norm{\chi_2(.,\o)-\chi_2(.,\o')}_{L^5(\mathcal{M})}^2\\
\nn&\les& \norm{\po\chi_2}_{L^5(\mathcal{M})}^2|\o-\o'|^2\ep\\
&\les& |\o-\o'|^2\ep.
\eea
Finally, in view of the decomposition for $\chi_1(.,\o)$ provided by \eqref{lebron1} \eqref{lebron2} \eqref{lebron3}, \eqref{okidoki2}, \eqref{okidoki3} and \eqref{okidoki4}, we obtain the desired decomposition for $\chi_1(.,\o)\chi_2(.,\o)^2$ with $F^j_1$ and $F^j_2$ defined in \eqref{lebron1},  
$$F^j_3=\chi_1(.,\o)^l(\chi_2(.,\o)-\chi_2(.,\o')),$$
and 
$$F^j_4=\chi_1(.,\o)(\chi_2(.,\o)-\chi_2(.,\o'))^2.$$
This concludes the proof of the corollary.

\subsubsection{Proof of Proposition \ref{cor:so9}}\lab{sec:pio8}

In view of the decomposition \eqref{dechch1} for $\hch$, we decompose $\hch(.,\o)^3$ as
\be\lab{visitmedbis}
\hch(.,\o)^3=\chi_1(.,\o)^3+3\chi_1(.,\o)\chi_2(.,\o)^2+3\chi_1(.,\o)^2\chi_2(.,\o)+\chi_2(.,\o)^3.
\ee
We have
\bea\lab{okidoki5bis}
\nn\chi_2(.,\o)^3&=&\chi_2(.,\o')^3+3\chi_2(.,\o')^2(\chi_2(.,\o)-\chi_2(.,\o'))+3\chi_2(.,\o')(\chi_2(.,\o)-\chi_2(.,\o'))^2\\
&&+(\chi_2(.,\o)-\chi_2(.,\o'))^3.
\eea
Now, we have in view of Lemma \ref{cor:so2} and the estimate \eqref{dechch2} for $\chi_2$:
\bea\lab{okidoki6bis}
&&\nn\norm{\chi_2(.,\o)-\chi_2(.,\o')}_{\li{\infty}{2}}\les |\o-\o'|\ep,\,\norm{(\chi_2(.,\o)-\chi_2(.,\o'))^2}_{L^2(\mathcal{M})}\les |\o-\o'|^2\ep,\\
&&\textrm{and }\norm{(\chi_2(.,\o)-\chi_2(.,\o'))^3}_{L^{2_-}(\mathcal{M})}\les |\o-\o'|^3\ep.
\eea
Finally, in view of \eqref{visitmedbis}, Corollary \ref{cor:so5} with $l=3$, Corollary \ref{cor:so6} with $l=2$, Corollary \ref{cor:so8}, \eqref{okidoki5bis} and \eqref{okidoki6bis}, we obtain the desired decomposition for $\hch^3$. 

\subsection{A second variant of Proposition \ref{prop:xx1}}\lab{sec:ohlalala4}

We have the following variant of Proposition \ref{prop:xx1}
\begin{proposition}\lab{bis:prop:xx1}
Let $f(.,\o)$ a scalar function depending on a parameter $\o\in\S$ such that: 
$$\norm{f}_{L^\infty_u\lh{2}}+\no(f)+\norm{\Lambda^{-1}(\nabla_{bN} f)}_{L^\infty_u\lh{2}}+\norm{\po f}_{L^\infty_u\lh{2}}\les \ep.$$
Let $\o$ and $\o'$ in $\S$. Let $u=u(t,x,\o)$ and $u'=u(t,x,\o')$. Then, for any $\o''$ in $\S$ on the arc joining $\o$ and $\o'$, and for any $j\geq 0$, we have the following decomposition for $f(.,\o'')$:
$$f(.,\o'')=P_{\leq \frac{j}{2}}(f(.,\o'))+f^j_2$$
and where $f^j_2$ satisfies:
$$\norm{f^j_2}_{\lh{2}}\les 2^{-\frac{j}{4}}\ep+2^{\frac{j}{4}}|\o-\o'|\ep.$$
\end{proposition}

As a corollary of Proposition \ref{bis:prop:xx1}, we obtain:
\begin{corollary}\lab{bis:cor:xx1}
Let $F(.,\o)$ a tensor depending on a parameter $\o\in\S$ such that:
$$\norm{F}_{L^\infty_u\lh{2}}+\no(F)+\norm{\po F}_{L^\infty_u\lh{2}}\les \ep.$$
Also, assume the existence of tensors $H_1$ and $H_2$ such that 
$$\nabb_{bN}F=\nabb H_1+H_2\textrm{ with }\norm{H_1}_{\li{\infty}{2}}+\norm{H_2}_{\tx{2}{\frac{4}{3}}}\les\ep.$$
Let $\o$ and $\o'$ in $\S$. Let $u=u(t,x,\o)$ and $u'=u(t,x,\o')$. Then, for any $\o''$ in $\S$ on the arc joining $\o$ and $\o'$, and for any $j\geq 0$, we have the following decomposition for $F(.,\o'')$:
$$F(.,\o'')=F^j_1+F^j_2$$
where $F^j_1$ does not depend on $\o$ and satisfies for any $2\leq q\leq +\infty$:
$$\norm{F^j_1}_{L^q(P_{t, u_{\o'}})}\les \norm{F}_{L^q(P_{t, u_{\o'}})},$$
and where $F^j_2$ satisfies:
$$\norm{F^j_2}_{\lh{2}}\les 2^{-\frac{j}{4}}\ep+2^{\frac{j}{4}}|\o-\o'|\ep.$$
\end{corollary}

The following lemma will be useful for the proof of Proposition \ref{bis:prop:xx1} and Corollary \ref{bis:cor:xx1}.
\begin{lemma}\lab{bis:lemma:xx5}
Let $f$ a scalar function and $\o, \o'$ in $\S$. Assume that $f$ satisfies
$$\norm{f}_{L^\infty_u\lh{2}}+\no(f)+\norm{\Lambda^{-1}(\nabla_{bN} f)}_{L^\infty_u\lh{2}}\les \ep.$$
Then, for any $l\geq 0$, we have:
$$\norm{[\po,P_{\leq l}]f}_{L^2(\H_{u'})}\les 2^{\frac{l}{2}}\ep.$$
\end{lemma}

The proof of Corollary \ref{bis:cor:xx1} is postponed to section \ref{bis:sec:xx1bis} and the proof of Lemma \ref{bis:lemma:xx5} is postponed to section \ref{bis:sec:xx5}. We now conclude the proof of Proposition \ref{bis:prop:xx1}.

\subsubsection{Proof of Proposition \ref{bis:prop:xx1}}

We decompose $f(.,\o'')$ as:
\bea
\lab{bis:xx1}f(.,\o'')&=&P''_{\leq\frac{j}{2}}(f(.,\o''))+\sum_{l>\frac{j}{2}}P''_l(f(.,\o'')\\
\nn&=&P'_{\leq\frac{j}{2}}(f(.,\o'))+\int_{[\o',\o'']}\po P'''_{\leq\frac{j}{2}}(f(.,\o'''))d\o'''(\o'-\o'')+\sum_{l>\frac{j}{2}}P''_l(f(.,\o''))\\
\nn&=&P'_{\leq\frac{j}{2}}(f(.,\o'))+\int_{[\o',\o'']}(P'''_{\leq\frac{j}{2}}(\po f)(.,\o''')+[\po,P'''_{\leq\frac{j}{2}}]f(.,\o'''))d\o'''(\o'-\o'')\\
\nn&&+\sum_{l>\frac{j}{2}}P''_l(f(.,\o'')),
\eea
where $\o'''$ denotes an angle in $\S$ on the arc joining $\o'$ and $\o''$.

Next, we estimate the last two terms in the right-hand side of \eqref{bis:xx1}. Using Lemma \ref{lemma:xx8} with $p=2$, we have:
\bea
\nn\sum_{l>\frac{j}{2}}\norm{P''_l(f(.,\o''))}_{L^\infty_u\lh{2}}&\les& \sum_{l>\frac{j}{2}}\norm{P_lf}_{\li{\infty}{2}}^{\frac{1}{2}}\norm{\nabb P_lf}_{\li{\infty}{2}}^{\frac{1}{2}}\\
\nn&\les& \left(\sum_{l>\frac{j}{2}}2^{-\frac{l}{2}}\right)\norm{\nabb f}_{\li{\infty}{2}}\\
\label{bis:xx3}&\les& 2^{-\frac{j}{4}}\ep,
\eea
where we used the finite band property for $P_l$ and the assumptions on $f$. Also, using Lemma \ref{lemma:xx8} with $p=2$, we have:
\bea
\label{bis:xx5}\norm{P'''_{\leq\frac{j}{2}}(\po f)(.,\o''')}_{L^\infty_u\lh{2}}&\les& \norm{P_{\leq\frac{j}{2}}\po f}_{\li{\infty}{2}}^{\frac{1}{2}}\norm{\nabb P_{\leq\frac{j}{2}}\po f}_{\li{\infty}{2}}^{\frac{1}{2}}\\
\nn&\les& 2^{\frac{j}{4}}\norm{\po f}_{\li{\infty}{2}}\\
\nn&\les& 2^{\frac{j}{4}}\ep,
\eea
where we used the finite band property for $P_{\leq\frac{j}{2}}$ and the assumptions on $f$.

Using Lemma \ref{bis:lemma:xx5} together with the assumptions on $f$, we have:
\be\label{bis:xx6}
\norm{[\po,P'''_{\leq\frac{j}{2}}]f(.,\o'''))}_{L^\infty_u\lh{2}}\les 2^{\frac{j}{4}}\ep.
\ee

In view of \eqref{bis:xx1}, we have $f(.,\o'')=f^1_j+f_2^j$ where $f_1^j$ is defined as:
\be\label{bis:xx7}
f^1_j=P'_{\leq\frac{j}{2}}(f(.,\o')),
\ee
and $f^2_j$ is defined as:
\be\label{bis:xx8}
f^2_j=\int_{[\o',\o'']}(P'''_{\leq\frac{j}{2}}(\po f)(.,\o''')+[\po,P'''_{\leq\frac{j}{2}}]f(.,\o'''))d\o'''(\o'-\o'')+\sum_{l>\frac{j}{2}}P''_l(f(.,\o'')).
\ee
Using \eqref{bis:xx3}, \eqref{bis:xx5} and \eqref{bis:xx6}, and the fact that $\o''$ is on the arc of $\S$ joining $\o$ and $\o'$, we have the following estimate for $f^2_j$:
\bee
\norm{f^2_j}_{L^\infty_u\lh{2}}&\les& 2^{-\frac{j}{4}}\ep+2^{\frac{j}{4}}|\o-\o'|\ep.
\eee
This concludes the proof of Proposition \ref{bis:prop:xx1}.

\subsubsection{Proof of Corollary \ref{bis:cor:xx1}}\lab{bis:sec:xx1bis}

Using Lemma \ref{lemma:xx7}, it suffices to prove the decomposition of Corollary \ref{bis:cor:xx1} where $F(\o'', .)$ is replaced by $g(F(.,\o''),Q_{\leq 1}(N_l))$ for $l=1, 2, 3$. Since the proof is identical for $l=1, 2, 3$, we simply take $l=1$. Therefore, it remains to prove that the following decomposition holds $g(F(., \o''),Q_{\leq 1}(N_1))$:
\be\lab{bis:xx18}
g(F(.,\o''),Q_{\leq 1}(N_1))=P_{\leq \frac{j}{2}}(g(F(.,\o'),Q_{\leq 1}(N_1)))+f_2^j,
\ee
where the scalar function $f_2^j$ satisfies:
\be\lab{bis:xx20}
\norm{f_2^j}_{\lh{2}}\les 2^{-\frac{j}{4}}\ep+2^{\frac{j}{4}}|\o-\o'|\ep.
\ee
In particular, $F_1^j$ is connected to the first term in the right-hand side of \eqref{bis:xx18}, which does not  depend on $\o$ and satisfies the following estimate for any $2\leq q\leq +\infty$:
$$\normm{P_{\leq \frac{j}{2}}(g(F(.,\o'),Q_{\leq 1}(N_1)))}_{L^q(P_{t, u_{\o'}})}\les \norm{F}_{L^q(P_{t, u_{\o'}})}\norm{Q_{\leq 1}(N_1)}_{L^\infty}\les \norm{F}_{L^q(P_{t, u_{\o'}})},$$
where we used the fact that $P_{\leq \frac{j}{2}}$ is bounded on $L^q(P_{t, u_{\o'}})$ and the fact that $Q_{\leq 1}$ is bounded on $L^\infty$. 

Let $f=g(F(.,\o),Q_{\leq 1}(N_1))$. In order to prove the decomposition \eqref{bis:xx18} \eqref{bis:xx20} for $g(F(.,\o''),Q_{\leq 1}(N_1))$, it suffices to show that $f$ satisfies the assumptions of Proposition \ref{bis:prop:xx1}. This was already done in the proof of Corollary \ref{cor:xx1}, up to the estimate of $\nabla_{bN} f$ which is the only one for which the proof has to be adapted. We have:
\bea\lab{mamoul}
&&\norm{\La^{-1}(\nabla_{bN}f)}_{\li{\infty}{2}}\\
\nn&\les& \norm{\La^{-1}(\gg(Q_{\leq 1}(N_1),\nabb_{bN}F))}_{\li{\infty}{2}}+\norm{\La^{-1}(\gg(\nabb_{bN}Q_{\leq 1}(N_1), F))}_{\li{\infty}{2}}\\
\nn&\les& \norm{\La^{-1}(\gg(Q_{\leq 1}(N_1),\nabb H_1+H_2))}_{\li{\infty}{2}}+\norm{b}_{L^\infty}\norm{\nabb_{N}Q_{\leq 1}(N_1)}_{\tx{\infty}{4}}\norm{F}_{\tx{2}{4}}\\
\nn&\les& \norm{\La^{-1}(\gg(Q_{\leq 1}(N_1),\nabb H_1))}_{\li{\infty}{2}}+\norm{\La^{-1}(\gg(Q_{\leq 1}(N_1),H_2))}_{\li{\infty}{2}}+\ep,
\eea
where we used the fact that $\La^{-1}$ in bounded on $\lpt{2}$, the assumptions on $F$ and in particular the decomposition for $\nabb_NF$, the estimate \eqref{estb} for $b$, and the estimate \eqref{coucoudeldoo} for $\nabb_NQ_{\leq 1}(N_1)$. We consider the first term in the right-hand side of \eqref{mamoul}. We have
$$\gg(Q_{\leq 1}(N_1),\nabb H_1)=\nabb(\gg(Q_{\leq 1}(N_1),H_1))-\gg(\nabb Q_{\leq 1}(N_1),H_1)$$
and thus
\bea\lab{mamoul1}
&&\norm{\La^{-1}(\gg(Q_{\leq 1}(N_1),\nabb H_1))}_{\li{\infty}{2}}\\
\nn&\les& \norm{\La^{-1}\nabb(\gg(Q_{\leq 1}(N_1), H_1))}_{\li{\infty}{2}}+\norm{\La^{-1}(\gg(\nabb Q_{\leq 1}(N_1), H_1))}_{\li{\infty}{2}}\\
\nn&\les& \norm{Q_{\leq 1}(N_1)}_{L^\infty}\norm{H_1}_{\li{\infty}{2}}+\norm{\nabb Q_{\leq 1}(N_1)}_{\tx{\infty}{4}}\norm{H_1}_{\li{\infty}{2}}\\
\nn&\les& \ep,
\eea
where we used the fact that $\La^{-1}\nabb$ in bounded on $\lpt{2}$, the fact that $\La^{-1}$ is bounded from $\lpt{\frac{4}{3}}$ to $\lpt{2}$, the assumption on $H_1$, and the estimate \eqref{coucoudeldoo} for $\nabb_NQ_{\leq 1}(N_1)$. Next, we estimate the first term in the right-hand side of \eqref{mamoul}. We have
\bea\lab{mamoul2}
\norm{\La^{-1}(\gg(Q_{\leq 1}(N_1), H_2))}_{\li{\infty}{2}}&\les& \norm{Q_{\leq 1}(N_1)}_{L^\infty}\norm{H_2}_{\tx{2}{\frac{4}{3}}}\\
\nn&\les& \ep,
\eea
where we used the fact that $\La^{-1}$ in bounded from $\lpt{\frac{4}{3}}$ to $\lpt{2}$ and the assumption on $H_2$.
 In view of \eqref{mamoul}, \eqref{mamoul1} and \eqref{mamoul2}, we finally obtain
$$\norm{\La^{-1}(\nabla_{bN} f)}_{\li{\infty}{2}}\les\ep.$$ 
Together with the other estimates for $f$ which may be derived as in Corollary \ref{cor:xx1}, we obtain that $f$ satisfies the assumptions of Proposition \ref{bis:prop:xx1}, which in turn yields the decomposition \eqref{bis:xx18}-\eqref{bis:xx20} for $g(F(.,\o''),Q_{\leq 1}(N_1))$. This concludes the proof of Corollary \ref{bis:cor:xx1}.
 
\subsection{Decompositions involving $\z$, $\nabb b$ and $\po b$}\lab{sec:ohlalala5}

The goal of this section is to prove Propositions \ref{cor:atlast} and Proposition \ref{cor:atlast1}. The proof of Proposition \ref{cor:atlast} is given in section \ref{sec:atlast}, and the proof of Proposition \ref{cor:atlast1} is given in  section \ref{sec:atlast1}.

We will need the following two lemmas.
\begin{lemma}\lab{lemma:assemblee}
$\ddb_{bN}\nabb b$ and $\ddb_{bN}\z$ satisfy the following decomposition:
$$\ddb_{bN}\nabb b, \, \ddb_{bN}\z=\nabb h_1+H_2,$$
where the scalar $h_1$ and the tensor $H_2$ satisfy
$$\norm{h_1}_{\li{\infty}{2}}+\norm{H_2}_{\tx{2}{\frac{4}{3}}}\les\ep.$$
\end{lemma}

\begin{lemma}\lab{lemma:assemblee1}
There holds the following estimate
$$\norm{\La^{-1}(\nabla_{bN}\po b)}_{\li{\infty}{2}}\les\ep.$$
\end{lemma}

The proof of Lemma \ref{lemma:assemblee} is postponed to section \ref{sec:assemblee}, and the proof of Lemma \ref{lemma:assemblee1} is postponed to section \ref{sec:assemblee1}.

\subsubsection{Proof of Corollary \ref{cor:atlast}}\lab{sec:atlast}

In view of the estimate \eqref{estzeta} for $\z$, the estimate \eqref{estricciomega} for $\po\z$ and Lemma \ref{lemma:assemblee}, $\z$ satisfies the assumption of Corollary \ref{bis:cor:xx1}. Also, in view of the estimate \eqref{estb} for $b$, the estimate \eqref{estricciomega} for $\po b$ and Lemma \ref{lemma:assemblee}, $\nabb b$ satisfies the assumptions of Corollary \ref{bis:cor:xx1}. Thus, the desired decomposition of Corollary \ref{cor:atlast} for $\z$ and $\nabb b$ follows from Corollary \ref{bis:cor:xx1}. This concludes the proof of Corollary \ref{cor:atlast}.

\subsubsection{Proof of Corollary \ref{cor:atlast1}}\lab{sec:atlast1}

We have:
\be\label{atlast1}
b(.,\o)-b(,.\o')=\int_{[\o,\o']}\po b(.,\o'')d\o''(\o-\o').
\ee
We denote $\po b''=\po b(.,\o'')$. In view of the estimate \eqref{estricciomega} for $\po b$, the estimate \eqref{estpo2b} for $\po^2b$ and Lemma \ref{lemma:assemblee1}, $\po b$ satisfies the assumptions of Proposition \ref{bis:prop:xx1}. Thus, we have the following decomposition for $\po b''$
 \be\lab{atlast2}
\po b''=f_1^j+f_2^j,
\ee
where the scalar $f_1^j$ only depends on $\o'$ and satisfies:
\be\lab{atlast3}
\norm{f_1^j}_{L^\infty}\les \norm{\po b}_{L^\infty}\lesssim \ep
\ee
in view of the estimate \eqref{estricciomega} for $\po b$, and where the scalar $f_2^j$ satisfies:
\be\lab{atlast4}
\norm{f_2^j}_{L^\infty_u\lh{2}}\les 2^{-\frac{j}{4}}\ep.
\ee
Injecting the decomposition \eqref{atlast2} in \eqref{atlast1}, and in view of \eqref{atlast3} \eqref{atlast4}, 
we obtain the desired decomposition for $b(.\o)-b(,.\o')$. This concludes the proof of Corollary \ref{cor:atlast1}.

\section{Additional estimates for $\trc$}\lab{sec:commutatorest}

The goal of this section is to prove Proposition \ref{prop:zz37} and Proposition \ref{prop:zz39}. 

\subsection{Commutator estimates between $P_j$ and $\ddb_L, \ddb_N$}

\begin{proposition}\lab{prop:gowinda11}
Let $F$ as tensor on $\mathcal{M}$. Let  a real number $a$ such that $0<a<\frac{1}{4}$. Then, we have
\be\label{bonobo}
\norm{[\ddb_{bN},P_j]F}_{L^\infty_tL^{\frac{4}{3}}_u\lpt{2}}\les 2^{ja}\norm{\nabla F}_{\tx{\infty}{2}}.
\ee
\end{proposition}

\begin{proposition}\lab{prop:gowinda12}
Let a scalar function $f$ on $\H_u$. Then, we have
\be\lab{zz}
\norm{[bN, P_l]f}_{\li{\infty}{2}}+2^{-j}\norm{\nabb[bN, P_l]f}_{\li{\infty}{2}}\les\ep\no(f),
\ee
and
\be\lab{zz1}
\norm{[nL, P_l]f}_{\li{\infty}{2}}+2^{-j}\norm{\nabb[nL, P_l]f}_{\li{\infty}{2}}\les\ep\no(f).
\ee
\end{proposition}

\begin{proposition}\lab{prop:gowinda13}
Let $f$ a scalar on $\mathcal{M}$. Then, we have
\be\label{lxx4:4}
\norm{[nL,P_q]f}_{\li{\infty}{2}}+\norm{[bN,P_q]f}_{\li{\infty}{2}}\lesssim 2^q\norm{f}_{\li{\infty}{2}}.
\ee
\end{proposition}

The proof of Proposition \ref{prop:gowinda11} is postponed to section \ref{sec:gowinda11}, the proof of Proposition \ref{prop:gowinda12} is postponed to section \ref{sec:gowinda12}, and the proof of Proposition \ref{prop:gowinda13} is postponed to section \ref{sec:gowinda13}.

\subsection{Commutator estimates acting on $\trc$}

\begin{proposition}\lab{prop:gowinda14}
We have the following commutator estimate
\be\lab{zz2}
2^j\norm{[nL,P_j]\trc}_{\tx{1}{2}}+\norm{\nabb[nL,P_j]\trc}_{\tx{1}{2}}\les \ep.
\ee
\end{proposition}

\begin{proposition}\lab{prop:gowinda15}
We have
\be\lab{zz17}
2^{\frac{j}{2}}\norm{[bN,P_j]\trc}_{\li{\infty}{2}}+2^{-\frac{j}{2}}\norm{\nabb[bN,P_j]\trc}_{\li{\infty}{2}}\les\ep.
\ee
and
\be\lab{zz18}
2^{\frac{j}{2}}\norm{[nL,P_j]\trc}_{\li{\infty}{2}}+2^{-\frac{j}{2}}\norm{\nabb[nL,P_j]\trc}_{\li{\infty}{2}}\les\ep.
\ee
\end{proposition}

\begin{proposition}\lab{prop:gowinda16}
We have the following commutator estimate
\be\lab{zz31}
\norm{[\nabb, P_j]\trc}_{\tx{2}{4}}\les\ep.
\ee
\end{proposition}

The proof of Proposition \ref{prop:gowinda14} is postponed to section \ref{sec:gowinda14}, the proof of Proposition \ref{prop:gowinda15} is postponed to section \ref{sec:gowinda15}, and the proof of Proposition \ref{prop:gowinda16} is postponed to section \ref{sec:gowinda16}.

\subsection{Additional estimates for $P_j\trc$}

The goal of this section is to prove Proposition \ref{prop:zz37} and Proposition \ref{prop:zz39}. Note that the finite band property for $P_j$ together with the estimate \eqref{esttrc} for $\trc$ yields
\be\lab{tontonkenji}
\norm{P_j\trc}_{\tx{\infty}{2}}\les 2^{-j}\norm{\nabb\trc}_{\tx{\infty}{2}}\les 2^{-j}\ep.
\ee
Also, the boundedness on $\lpt{2}$ of $P_j$ together with the estimate \eqref{esttrc} for $\trc$ yields
\be\lab{tontonkenji1}
\norm{\nabb P_{\leq j}\trc}_{\tx{\infty}{2}}=\norm{(-\lap)^{\frac{1}{2}}P_{\leq j}\trc}_{\tx{\infty}{2}}\les\norm{\nabb\trc}_{\tx{\infty}{2}}\les 2^{-j}\ep.
\ee
In order to prove Proposition \ref{prop:zz37} and Proposition \ref{prop:zz39}, we need in particular to obtain \eqref{tontonkenji} and \eqref{tontonkenji1}, where the norm $\tx{\infty}{2}$ is replaced by its stronger version $\xt{2}{\infty}$. We will need the following lemmas. 
\begin{lemma}\label{lemma:zz41}
Let $h$ a scalar on $\ptu$, and let $F$ a tensor on $\ptu$. Then, we have
\be\lab{zz41}
\norm{[P_{>j}, P_{\leq j}(h)]F}_{\lpt{2}}\les \norm{\nabb h}_{\lpt{2}}\norm{F}_{\lpt{2}}.
\ee
\end{lemma}

\begin{lemma}\label{lemma:zz42}
Let $h$ a scalar on $\ptu$, and let $F$ a tensor on $\ptu$. Then, we have
\be\lab{zz42}
\norm{\nabb[P_j, P_{\leq j}(h)]F}_{\lpt{2}}\les 2^j(\norm{\nabb h}_{\lpt{2}}+\norm{K}_{\lpt{2}}\norm{h}_{\lpt{2}})\norm{F}_{\lpt{2}}.
\ee
\end{lemma}

\begin{lemma}\label{lemma:zz43}
Let $h$ a scalar on $\ptu$, and let $a>0$. Then, we have
\be\lab{zz43}
\norm{[P_{\leq j}, \nabb]h}_{\lpt{2}}\les \norm{K}_{\lpt{2}}(\norm{K}_{\lpt{2}}\norm{h}_{\lpt{2}}+\norm{\La^ah}_{\lpt{2}}).
\ee
\end{lemma}

\begin{lemma}\label{lemma:zz44}
Let $h$ a scalar on $\ptu$, and let $a>0$. Then, we have
\be\lab{zz44}
\norm{\nabb[P_j, \nabb]h}_{\lpt{2}}\les 2^j\norm{K}_{\lpt{2}}(\norm{K}_{\lpt{2}}\norm{h}_{\lpt{2}}+\norm{\La^ah}_{\lpt{2}}).
\ee
\end{lemma}

\begin{lemma}\lab{lemma:yo}
Let $f$ a scalar on $\ptu$. We have
\be\lab{yo}
\norm{\nabb f}_{\lpt{\infty}}\les \norm{\lap f}_{\lpt{2}}+\norm{\nabb\lap f}_{\lpt{2}}^\frac{1}{2}\norm{\nabb f}_{\lpt{2}}^{\frac{1}{2}}+\norm{K}_{\lpt{2}}\norm{\nabb f}_{\lpt{2}}.
\ee
\end{lemma}

In the subsequent sections, we provide a proof of Proposition \ref{prop:zz37} and Proposition \ref{prop:zz39}. The proof of Lemma \ref{lemma:zz41} is postponed to section \ref{sec:gowinda17}, the proof of Lemma \ref{lemma:zz42} is postponed to section \ref{sec:gowinda18}, the proof of Lemma \ref{lemma:zz43} is postponed to section \ref{sec:gowinda19}, the proof of Lemma \ref{lemma:zz44} is postponed to section \ref{sec:gowinda20}, 
and the proof of Lemma \ref{lemma:yo} is postponed to section \ref{sec:gowinda21}.

\subsubsection{Proof of Proposition \ref{prop:zz37}}

Using the estimate \eqref{estimtransport1} for transport equations, we have
\bea\lab{zz45}
\norm{P_j\trc}_{\xt{2}{\infty}}&\les& \norm{nL P_j\trc}_{\xt{2}{1}}\\
\nn&\les& \norm{P_j(nL\trc)}_{\xt{2}{1}}+\norm{[nL, P_j]\trc}_{\xt{2}{1}}\\
\nn&\les& \norm{P_j(nL\trc)}_{\xt{2}{1}}+2^{-j}\ep,
\eea
where we used the estimate \eqref{zz2} in the last inequality. Now, \eqref{zz37} follows from \eqref{zz38} and 
\eqref{zz45}. Thus, it remains to prove \eqref{zz38}.

Next, using the Raychaudhuri equation \eqref{D4trchi}, we have
\bee
\norm{P_j(nL\trc)}_{\xt{2}{1}}&\les& \norm{P_j(n|\hch|^2)}_{\xt{2}{1}}+\normm{P_j\left(n\left(\half (\trc)^2+\db   \trc\right)\right)}_{\xt{1}{2}}\\
&\les& \norm{P_j(n|\hch|^2)}_{\xt{2}{1}}+2^{-j}\normm{\nabb\left(n\left(\half (\trc)^2+\db   \trc\right)\right)}_{\lh{2}},
\eee
where we used the finite band property for $P_j$ in the last inequality. Together with the estimate \eqref{esttrc} for $\trc$, the estimate \eqref{estn} for $n$, and the estimates \eqref{estn} \eqref{estk} for $\db$, we obtain
$$\norm{P_j(nL\trc)}_{\xt{2}{1}}\les \norm{P_j(n|\hch|^2)}_{\xt{2}{1}}+2^{-j}\ep.$$
Thus, it remains to prove
\be\lab{zz46}
\norm{P_j(n|\hch|^2)}_{\xt{2}{1}}\les 2^{-j}\ep.
\ee

We have
$$P_j(n|\hch|^2)=2^{-2j}\lap P_j(n|\hch|^2)=2^{-2j}\divb(\nabb P_j(n|\hch|^2)).$$
Thus, we deduce
\be\lab{zz47}
\norm{P_j(n|\hch|^2)}_{\xt{2}{1}}\les 2^{-2j}\norm{\divb P_j\nabb(n|\hch|^2)}_{\xt{2}{1}}+2^{-2j}\norm{\divb [\nabb, P_j](n|\hch|^2)}_{\xt{1}{2}}.
\ee
Now, in view of \eqref{zz44}, we have for any $a>0$
$$\norm{\divb [\nabb, P_j](n|\hch|^2)}_{\lpt{2}}\les 2^j\norm{K}_{\lpt{2}}(\norm{K}_{\lpt{2}}\norm{n|\hch|^2}_{\lpt{2}}+\norm{\La^a(n|\hch|^2)}_{\lpt{2}}).$$
Taking the $L^1_t$ norm, we obtain
\bea\lab{zz48}
\nn\norm{\divb [\nabb, P_j](n|\hch|^2)}_{\tx{1}{2}}&\les& 2^j\norm{K}_{\lh{2}}(\norm{K}_{\lh{2}}\norm{n}_{L^\infty}\norm{\hch}_{\tx{\infty}{4}}^2+\norm{\La^a(n|\hch|^2)}_{\lh{2}})\\
&\les& 2^j\ep(1+\norm{\La^a(n|\hch|^2)}_{\lh{2}}),
\eea
where we used in the last inequality the estimate \eqref{estgauss1} for $K$, the estimate \eqref{estn} for $n$, and the estimate \eqref{esthch} for $\hch$. Now, choosing $0<a<\frac{1}{2}$, the non sharp product estimate \eqref{nonsharpprod2} yields
\bea\lab{zz48bis}
\norm{\La^a(n|\hch|^2)}_{\lh{2}}&\les& \norm{n|\hch|^2}_{\PP^a}\\
\nn&\les& \no(\hch)(\norm{n\hch}_{\lh{2}}+\norm{\nabb (n\hch)}_{\lh{2}})\\
\nn&\les& \ep,
\eea
where we used in the last inequality the estimate \eqref{estn} for $n$, and the estimate \eqref{esthch} for $\hch$. Together with \eqref{zz47} and \eqref{zz48}, we obtain
\be\lab{zz49}
\norm{P_j(n|\hch|^2)}_{\xt{2}{1}}\les 2^{-2j}\norm{\divb P_j\nabb(n|\hch|^2)}_{\xt{2}{1}}+2^{-j}\ep.
\ee

Next, we estimate the right-hand side of \eqref{zz49}. We have
\be\lab{zz49bis}
\norm{|\hch|^2\nabb n}_{\li{\infty}{2}}\les \norm{\nabb n}_{L^\infty}\norm{\hch}^2_{\lh{4}}\les \ep,
\ee
where we used the estimate \eqref{estn} for $n$ and the estimate \eqref{esthch} for $\hch$. Together with \eqref{zz49} and the finite band property for $P_j$, we obtain
\be\lab{zz50}
\norm{P_j(n|\hch|^2)}_{\xt{2}{1}}\les 2^{-2j}\norm{\divb P_j(n\hch\c \nabb\hch)}_{\xt{2}{1}}+2^{-j}\ep.
\ee
We define a scalar $h$ and a vectorfield $F$ by
\be\lab{zz51}
h=|\hch|\textrm{ and }F=n\frac{\hch}{|\hch|}\c\nabb\chi,
\ee
and we decompose
\be\lab{zz52}
n\hch\c\nabb\hch= hF= P_{\leq j}(h)F+P_{>j}(h)F.
\ee
Note in particular in view of the estimate \eqref{esthch} for $\hch$ and the estimate \eqref{estn} for $n$ that we have the following estimate for $h$ and $F$
\be\lab{zz53}
\no(h)+\norm{h}_{\xt{\infty}{2}}\les\ep\textrm{ and }\norm{F}_{\lh{2}}\les\ep.
\ee
We have
\bee
\norm{\divb P_j(P_{>j}(h)F)}_{\xt{2}{1}}&\les&\norm{\divb P_j(P_{>j}(h)F)}_{\tx{1}{2}}\\
&\les&2^j\norm{P_j(P_{>j}(h)F)}_{\tx{1}{2}},
\eee
where we used in the last inequality the finite band property for $P_j$. Together with the dual of the sharp Bernstein inequality for tensors \eqref{eq:strongberntensor}, we obtain
\bee
\norm{\divb P_j(P_{>j}(h)F)}_{\xt{2}{1}}&\les&2^j\normm{(2^j+\norm{K}_{\lpt{2}})\norm{P_{>j}(h)F}_{\lpt{1}}}_{L^1_t}\\
&\les& 2^{2j}\norm{P_{>j}(h)F}_{\lh{1}}+2^j\norm{K}_{\lh{2}}\norm{P_{>j}(h)F}_{\tx{2}{1}}\\
&\les& \Big(2^{2j}\norm{P_{>j}(h)}_{\lh{2}}+2^j\norm{K}_{\lh{2}}\norm{P_{>j}(h)}_{\tx{\infty}{2}}\Big)\norm{F}_{\lh{2}}.
\eee
Together with the finite band property and the boundedness on $L^2(\ptu)$ of $P_j$, we obtain
\bea\lab{zz54}
\nn\norm{\divb P_j(P_{>j}(h)F)}_{\xt{2}{1}}&\les& 2^j\Big(\norm{\nabb h}_{\lh{2}}+\norm{K}_{\lh{2}}\norm{h}_{\tx{\infty}{2}}\Big)\norm{F}_{\lh{2}}\\
&\les& 2^j\ep,
\eea
where we used in the last inequality the estimate \eqref{estgauss1} for $K$ and the estimate \eqref{zz53} for $h$ and $F$. 

Next, we evaluate the first term in the decomposition \eqref{zz52}. We have
\be\lab{zz55}
\norm{\divb P_j(P_{\leq j}(h)F)}_{\xt{2}{1}}\les \norm{P_{\leq j}(h)\divb P_j(F)}_{\xt{2}{1}}+\norm{\divb [P_j,P_{\leq j}(h)]F}_{\tx{1}{2}}.
\ee
In view of \eqref{zz42}, we have
\bea\lab{zz56}
\nn\norm{\divb [P_j,P_{\leq j}(h)]F}_{\tx{1}{2}}&\les& 2^j\normm{(\norm{\nabb h}_{\lpt{2}}+\norm{K}_{\lpt{2}}\norm{h}_{\lpt{2}})\norm{F}_{\lpt{2}}}_{L^1_t}\\
\nn&\les& 2^j(\norm{\nabb h}_{\lh{2}}+\norm{K}_{\lh{2}}\norm{h}_{\tx{\infty}{2}})\norm{F}_{\lh{2}}\\
&\les& 2^j\ep,
\eea
where we used in the last inequality the estimate \eqref{estgauss1} for $K$ and the estimate \eqref{zz53} for $h$ and $F$. Next, we consider the first term in the right-hand side of \eqref{zz55}. We have
\be\lab{zz57}
\norm{P_{\leq j}(h)\divb P_j(F)}_{\xt{2}{1}}\les \norm{h\divb P_j(F)}_{\xt{2}{1}}+\norm{P_{>j}(h)\divb P_j(F)}_{\xt{2}{1}}.
\ee
The first term in the right-hand side of \eqref{zz57} is estimated as follows
\bea\lab{zz58}
\norm{h\divb P_j(F)}_{\xt{2}{1}}&\les & \norm{h}_{\xt{\infty}{2}}\norm{\divb P_j(F)}_{\lh{2}}\\
\nn&\les& 2^j\norm{h}_{\xt{\infty}{2}}\norm{F}_{\lh{2}}\\
\nn&\les& 2^j\ep,
\eea
where we used the finite band property for $P_j$ and the estimate \eqref{zz53} for $h$ and $F$. Next, we 
estimate the second term in \eqref{zz57}. We have
\bea\lab{zz59}
\norm{P_{>j}(h)\divb P_j(F)}_{\lpt{2}}&\les& \norm{P_{>j}(h)}_{\lpt{4}}\norm{\divb P_j(F)}_{\lpt{4}}\\
\nn&\les& \norm{P_{>j}(h)}_{\lpt{4}}\norm{\nabb^2P_j(F)}_{\lpt{2}}^{\frac{1}{2}}\norm{\nabb P_j(F)}_{\lpt{2}}^{\frac{1}{2}}\\
\nn&\les& \norm{P_{>j}(h)}_{\lpt{4}}\norm{\nabb^2P_j(F)}_{\lpt{2}}^{\frac{1}{2}}2^{\frac{j}{2}}\norm{P_j(F)}_{\lpt{2}}^{\frac{1}{2}},
\eea
where we used the Gagliardo-Nirenberg inequality \eqref{eq:GNirenberg} and the finite band property for $P_j$. Using the Bochner identity for tensors \eqref{vbochineq}, we have
\bee
&&\norm{\nabb^2P_j(F)}_{\lpt{2}}\\
&\les& \norm{\lap P_j(F)}_{\lpt{2}}+\norm{K}_{\lpt{2}}\norm{\nabb P_j(F)}_{\lpt{2}}+\norm{K}^2_{\lpt{2}}\norm{P_j(F)}_{\lpt{2}}\\
&\les& (2^{2j}+\norm{K}^2_{\lpt{2}})\norm{F}_{\lpt{2}},
\eee
where we used in the last inequality the finite band property for $P_j$. Together with \eqref{zz59}, we obtain
\bee
&&\nn\norm{P_{>j}(h)\divb P_j(F)}_{\lpt{2}}\\
&\les& \norm{P_{>j}(h)}_{\lpt{4}}2^{\frac{3j}{2}}\norm{F}_{\lpt{2}}+\norm{P_{>j}(h)}_{\lpt{4}}2^{\frac{j}{2}}\norm{K}_{\lpt{2}}\norm{F}_{\lpt{2}}\\
&\les& \left(\sum_{l>j}2^{\frac{l}{2}}\norm{P_lh}_{\lpt{2}}\right)2^{\frac{3j}{2}}\norm{F}_{\lpt{2}}+\norm{h}_{\lpt{4}}2^{\frac{j}{2}}\norm{K}_{\lpt{2}}\norm{F}_{\lpt{2}}\\
&\les& \left(\left(\sum_{l>j}2^{-\frac{l}{2}}\right)\norm{\nabb h}_{\lpt{2}}2^{\frac{3j}{2}}+2^{\frac{j}{2}}\norm{K}_{\lpt{2}}\norm{h}_{\lpt{4}}\right)\norm{F}_{\lpt{2}}\\
&\les& 2^j\left(\norm{\nabb h}_{\lpt{2}}+\norm{K}_{\lpt{2}}\norm{h}_{\lpt{4}}\right)\norm{F}_{\lpt{2}},
\eee
where we used Bernstein, the boundedness on $\lpt{4}$ and the finite band property for $P_l$. Taking the $L^1_t$ norm, we obtain
\bea\lab{zz60}
\nn \norm{P_{>j}(h)\divb P_j(F)}_{\tx{1}{2}}&\les& 2^j\left(\norm{\nabb h}_{\lh{2}}+\norm{K}_{\lh{2}}\norm{h}_{\tx{\infty}{4}}\right)\norm{F}_{\lh{2}}\\
&\les& 2^j\ep,
\eea
where we used in the last inequality the estimate \eqref{estgauss1} for $K$ and the estimate \eqref{zz53} for $h$ and $F$. Now, \eqref{zz57}, \eqref{zz58} and \eqref{zz60} imply
$$\norm{P_{\leq j}(h)\divb P_j(F)}_{\xt{2}{1}}\les 2^j\ep.$$
Together with \eqref{zz55} and \eqref{zz56}, this yields
$$\norm{\divb P_j(P_{\leq j}(h)F)}_{\xt{2}{1}}\les 2^j\ep.$$
Together with \eqref{zz52} and \eqref{zz54}, we obtain
$$\norm{\divb P_j(n\hch\c\nabb\hch)}_{\xt{2}{1}}\les 2^j\ep.$$
Together with \eqref{zz50}, this yields the desired estimate \eqref{zz46}. This concludes the proof of the proposition.

\subsubsection{Proof of Proposition \ref{prop:zz39}}

Using the estimate \eqref{estimtransport1} for transport equations, we have
\bee
\norm{\nabb P_{\leq j}\trc}_{\xt{2}{\infty}}&\les& \norm{\ddb_{nL}\nabb P_j\trc}_{\xt{2}{1}}\\
\nn&\les& \norm{\nabb P_j(nL\trc)}_{\xt{2}{1}}+\norm{[\ddb_{nL}, \nabb]P_{\leq j}\trc}_{\xt{2}{1}}+\norm{\nabb [nL, P_{\leq j}]\trc}_{\xt{2}{1}}\\
\nn&\les& \norm{\nabb P_{\leq j}(nL\trc)}_{\xt{2}{1}}+\norm{n\chi\nabb P_{\leq j}\trc}_{\xt{2}{1}}+2^{-j}\ep\\
\nn&\les& \norm{\nabb P_{\leq j}(nL\trc)}_{\xt{2}{1}}+\norm{n}_{L^\infty}\norm{\chi}_{\xt{\infty}{2}}\norm{\nabb P_{\leq j}\trc}_{\xt{2}{1}}+2^{-j}\ep\\
\nn&\les& \norm{\nabb P_{\leq j}(nL\trc)}_{\xt{2}{1}}+\ep\norm{\nabb P_{\leq j}\trc}_{\xt{2}{\infty}}+2^{-j}\ep,
\eee
where we used the commutator formula \eqref{comm5}, the commutator estimate \eqref{zz2}, and the estimates \eqref{estn} for $n$ and \eqref{esttrc} \eqref{esthch} for $\chi$. Since $\ep>0$ is small, we obtain
\be\lab{zzbis45}
\norm{\nabb P_{\leq j}\trc}_{\xt{2}{\infty}}\les \norm{\nabb P_{\leq j}(nL\trc)}_{\xt{2}{1}}+2^{-j}\ep.
\ee
Now, \eqref{zz39} follows from \eqref{zz40} and 
\eqref{zzbis45}. Thus, it remains to prove \eqref{zz38}.

Next, using the Raychaudhuri equation \eqref{D4trchi}, we have
\bee
\norm{\nabb P_{\leq j}(nL\trc)}_{\xt{2}{1}}&\les& \norm{\nabb P_{\leq j}(n|\hch|^2)}_{\xt{2}{1}}+\normm{\nabb P_{\leq j}\left(n\left(\half (\trc)^2+\db   \trc\right)\right)}_{\xt{1}{2}}\\
&\les& \norm{\nabb P_{\leq j}(n|\hch|^2)}_{\xt{2}{1}}+\normm{\nabb\left(n\left(\half (\trc)^2+\db   \trc\right)\right)}_{\lh{2}},
\eee
where we used the finite band property for $P_j$ in the last inequality. Together with the estimate \eqref{esttrc} for $\trc$, the estimate \eqref{estn} for $n$, and the estimates \eqref{estn} \eqref{estk} for $\db$, we obtain
$$\norm{\nabb P_{\leq j}(nL\trc)}_{\xt{2}{1}}\les \norm{\nabb P_{\leq j}(n|\hch|^2)}_{\xt{2}{1}}+\ep.$$
Thus, it remains to prove
\be\lab{zzbis46}
\norm{\nabb P_{\leq j}(n|\hch|^2)}_{\xt{2}{1}}\les \ep.
\ee

In view of \eqref{zz43}, we have for any $a>0$
$$\norm{[\nabb, P_{\leq j}](n|\hch|^2)}_{\lpt{2}}\les \norm{K}_{\lpt{2}}(\norm{K}_{\lpt{2}}\norm{n|\hch|^2}_{\lpt{2}}+\norm{\La^a(n|\hch|^2)}_{\lpt{2}}).$$
Taking the $L^1_t$ norm, we obtain
\bee
\nn\norm{[\nabb, P_{\leq j}](n|\hch|^2)}_{\tx{1}{2}}&\les& \norm{K}_{\lh{2}}(\norm{K}_{\lh{2}}\norm{n}_{L^\infty}\norm{\hch}_{\tx{\infty}{4}}^2+\norm{\La^a(n|\hch|^2)}_{\lh{2}})\\
&\les& \ep,
\eee
where we used in the last inequality the estimate \eqref{estgauss1} for $K$, the estimate \eqref{estn} for $n$, the estimate \eqref{esthch} for $\hch$ and the estimate \eqref{zz48bis} with the choice $0<a<\frac{1}{2}$. Thus, we obtain
\be\lab{zzbis49}
\norm{\nabb P_{\leq j}(n|\hch|^2)}_{\xt{2}{1}}\les \norm{P_{\leq j}\nabb(n|\hch|^2)}_{\xt{2}{1}}+\ep.
\ee

Next, we estimate the right-hand side of \eqref{zzbis49}. In view of \eqref{zz49bis} and the boundedness on $\lpt{2}$ of $P_{\leq j}$, we have
\be\lab{zzbis50}
\norm{\nabb P_{\leq j}(n|\hch|^2)}_{\xt{2}{1}}\les \norm{P_j(n\hch\c \nabb\hch)}_{\xt{2}{1}}+\ep.
\ee
Now, recall the definition \eqref{zz51} of the scalar $h$ and the vectorfield $F$, the decomposition \eqref{zz52} of $n\hch\c\nabb\hch$ and the estimate \eqref{zz53} for $h$ and $F$. Using Bernstein for $P_{\leq j}$, we have
\bea\lab{zzbis54}
\norm{P_{\leq j}(P_{>j}(h)F)}_{\xt{1}{2}}&\les&2^{\frac{j}{2}}\norm{P_{>j}(h)F}_{\tx{1}{\frac{4}{3}}}\\
\nn&\les& 2^{\frac{j}{2}}\norm{P_{>j}h}_{\tx{2}{4}}\norm{F}_{\lh{2}}\\
\nn&\les& 2^{\frac{j}{2}}\left(\sum_{l>j}2^{\frac{l}{2}}\norm{P_lh}_{\lpt{2}}\right)\norm{F}_{\lh{2}}\\
\nn&\les& 2^{\frac{j}{2}}\left(\sum_{l>j}2^{-\frac{l}{2}}\right)\norm{\nabb h}_{\lh{2}}\norm{F}_{\lh{2}}\\
\nn&\les&\ep,
\eea
where we used Bernstein and the finite band property for $P_l$ and the estimate \eqref{zz53} for $h$ and $F$.

Next, we evaluate the first term in the decomposition \eqref{zz52} of $n\hch\c\nabb\hch$. We have
\be\lab{zzbis55}
\norm{P_{\leq j}(P_{\leq j}(h)F)}_{\xt{2}{1}}\les \norm{P_{\leq j}(h)P_{\leq j}F}_{\xt{2}{1}}+\norm{[P_{\leq j}, P_{\leq j}(h)]F)}_{\xt{1}{2}}.
\ee
Since $[P_{\leq j}, P_{\leq j}(h)]=[P_{>j}, P_{\leq j}(h)]$, we have in view of the commutator estimate \eqref{zz41}
\bea\lab{zzbis56}
\norm{[P_{\leq j}, P_{\leq j}(h)]F)}_{\tx{1}{2}}&=& \norm{[P_{> j}, P_{\leq j}(h)]F)}_{\tx{1}{2}}\\
\nn&\les& \normm{\norm{\nabb h}_{\lpt{2}}\norm{F}_{\lpt{2}}}_{L^1_t}\\
\nn&\les& \norm{\nabb h}_{\lh{2}}\norm{F}_{\lh{2}},
\eea
where we used in the last inequality the estimate \eqref{zz53} for $h$ and $F$. Next, we consider the first term in the right-hand side of \eqref{zzbis55}. We have
\be\lab{zzbis57}
\norm{P_{\leq j}(h)P_{\leq j}(F)}_{\xt{2}{1}}\les \norm{hP_{\leq j}(F)}_{\xt{2}{1}}+\norm{P_{>j}(h)P_{\leq j}(F)}_{\xt{2}{1}}.
\ee
The first term in the right-hand side of \eqref{zzbis57} is estimated as follows
\bea\lab{zzbis58}
\norm{hP_{\leq j}(F)}_{\xt{2}{1}}&\les & \norm{h}_{\xt{\infty}{2}}\norm{P_{\leq j}(F)}_{\lh{2}}\\
\nn&\les& \norm{h}_{\xt{\infty}{2}}\norm{F}_{\lh{2}}\\
\nn&\les& \ep,
\eea
where we used the boundedness on $\lpt{2}$ of $P_{\leq j}$ and the estimate \eqref{zz53} for $h$ and $F$. Next, we estimate the second term in \eqref{zzbis57}. We have
\bee
\norm{P_{>j}(h)P_{\leq j}(F)}_{\lpt{2}}&\les& \norm{P_{>j}(h)}_{\lpt{4}}\norm{P_{\leq j}(F)}_{\lpt{4}}\\
\nn&\les&  \left(\sum_{l>j}2^{\frac{l}{2}}\norm{P_lh}_{\lpt{2}}\right)2^{\frac{j}{2}}\norm{F}_{\lpt{2}}\\
\nn&\les& \left(\sum_{l>j}2^{-\frac{l}{2}}\right)\norm{\nabb h}_{\lpt{2}}2^{\frac{j}{2}}\norm{F}_{\lpt{2}}\\
\nn&\les& \norm{\nabb h}_{\lpt{2}}\norm{F}_{\lpt{2}},
\eee
where we used Bernstein for $P_l$ and $P_j$, and the finite band property for $P_l$. Taking the $L^1_t$ norm, we obtain
\bea\lab{zzbis60}
\nn \norm{P_{>j}(h)P_{\leq j}(F)}_{\tx{1}{2}}&\les& \norm{\nabb h}_{\lh{2}}\norm{F}_{\lh{2}}\\
&\les& \ep,
\eea
where we used in the last inequality the estimate \eqref{zz53} for $h$ and $F$. Now, \eqref{zzbis57}, \eqref{zzbis58} and \eqref{zzbis60} imply
$$\norm{P_{\leq j}(h)P_{\leq j}(F)}_{\xt{2}{1}}\les 2^j\ep.$$
Together with \eqref{zzbis55} and \eqref{zzbis56}, this yields
$$\norm{P_{\leq j}(P_{\leq j}(h)F)}_{\xt{2}{1}}\les 2^j\ep.$$
Together with \eqref{zzbis54} and the decomposition of $n\hch\c\nabb\hch$ given by \eqref{zz52}, we obtain
$$\norm{P_{\leq j}(n\hch\c\nabb\hch)}_{\xt{2}{1}}\les 2^j\ep.$$
Together with \eqref{zzbis50}, this yields the desired estimate \eqref{zzbis46}. This concludes the proof of the proposition.

\appendix

\section{Appendix to section \ref{sec:regxproof}}

\subsection{Proof of Proposition \ref{propK}}\lab{sec:propK}

Recall from the Gauss equation \eqref{gauss} that:
$$K=\half\hch_{AB}\hchb_{AB}-\frac{1}{4} \trc \trchb -\rho.$$
First, remark from \eqref{sobineq1} that:
\be\lab{gau1}
\normm{\half\hch_{AB}\hchb_{AB}-\frac{1}{4} \trc \trchb}_{\tx{\infty}{2}}\lesssim\norm{\chi}^2_{\tx{\infty}{4}}\lesssim\no(\chi)^2\lesssim \ep.
\ee
Furthermore, from the assumptions on the curvature flux \eqref{curvflux} \eqref{curvflux1}, we have:
\be\lab{gau2}
\norm{\rho}_{\lh{2}}\leq \ep.
\ee
\eqref{gau1} and \eqref{gau2} imply \eqref{estgauss1}. 

We now concentrate on \eqref{estgauss2}. We assume:
\begin{equation}\label{ad26}
\sum_{j\geq 0}2^{-j}\norm{P_jK}^2_{\tx{\infty}{2}}+\norm{P_{<0}K}^2_{\tx{\infty}{2}}\lesssim E^2\ep^2,
\end{equation}
where $E$ is a large enough constant. We will then try to improve \eqref{ad26}. Note that \eqref{eq:Bochconseq}, \eqref{La8} and \eqref{ad26} yield for any scalar function $f$ on $\ptu$:
\begin{equation}\label{ad26bis}
\norm{\nabb^2f}^2_{\lpt{2}}\lesssim\norm{\lap f}^2_{\lpt{2}}+(E\ep+E^4\ep^4)\norm{\nabb f}^2_{\lpt{2}}.
\end{equation}

In view of \eqref{gau1}, we just need to bound $\norm{\La^{-\half}\rho}_{\tx{\infty}{2}}$. Note from \eqref{La8} that it suffices to bound:
$$\norm{P_{<0}\rho}^2_{\tx{\infty}{2}}+\sum_{j\geq 0}2^{-j}\norm{P_j\rho}^2_{\tx{\infty}{2}}.$$
The term $\norm{P_{<0}\rho}_{\tx{\infty}{2}}$ is easier to bound, so we concentrate on estimating the sum $\sum_{j\ge 0}2^{-j}\norm{P_j\rho}_{\tx{\infty}{2}}$. We will use the following variant of \eqref{sobineq2} where we 
do not yet use Cauchy-Schwarz in $t$ for the integral containing $\dd_L F$:
\begin{equation}\label{ad27bis}
\norm{F}^2_{\tx{\infty}{2}}\lesssim \ds\int_0^1\norm{\dd_L F}_{\lpt{2}}\norm{F}_{\lpt{2}}dt+\norm{F}_{\lh{2}}^2.
\end{equation}
Using \eqref{ad27bis}, properties (ii) and (iii) of Theorem \ref{thm:LP} for $P_j$, the bound on $\rho$ given by \eqref{curvflux1} and the bound on $n$ given by the bootstrap assumption \eqref{boot1}, we have:
\begin{equation}\label{ad27}
\begin{array}{ll}
&\ds\sum_{j\geq 0}2^{-j}\norm{P_j\rho}^2_{\tx{\infty}{2}}\\
\ds\lesssim &\ds\sum_{j\ge 0} 2^{-j}\left(\int_0^1\norm{P_j\rho}_{\lpt{2}}\norm{\dd_LP_j\rho}_{\lpt{2}}dt+\norm{P_j\rho}_{\lh{2}}^2\right)\\
\ds\lesssim &\ds\sum_{j\ge 0} 2^{-j}\left(\int_0^1\norm{P_j\rho}_{\lpt{2}}\norm{\dd_L P_j\rho}_{\lpt{2}}dt\right)+\sum_{j\ge 0}2^{-j}\norm{\rho}^2_{\lh{2}}\\
\ds\lesssim &\ds\sum_{j\ge 0} 2^{-j}\left(\int_0^1\norm{P_j\rho}_{\lpt{2}}\norm{nLP_j\rho}_{\lpt{2}}dt\right)+\ep^2.
\end{array}
\end{equation}
We have:
\begin{equation}\label{ad28}
nL U(\tau)\rho=U(\tau)nL\rho+V(\tau)
\end{equation}
where $V(\tau)$ is satisfies:
\begin{equation}\label{ad29}
(\partial_{\tau}-\lap)V(\tau)=[nL,\lap]U(\tau),\,V(0)=0.
\end{equation}
Using \eqref{eq:LP} and \eqref{ad28}, we obtain:
\begin{equation}\label{ad30}
nLP_j\rho=P_jnL\rho +\int_0^\infty m_j(\tau)V(\tau) d\tau.
\end{equation}

We now estimate $\norm{P_jnL\rho}_{\lh{2}}$. We may assume the existence of $\widetilde{P_j}$ with the same properties than $P_j$ such that $P_j=\widetilde{P_j}^2$ (see \cite{LP}), and for simplicity we write $P_j=P_j^2$. Also, using the fact that $\La\La^{-1}=I$ and that $\La$ commutes with $P_j$, we obtain:
\begin{equation}\label{ad32}
P_j=\La P_jP_j\La^{-1},
\end{equation}
which together with property (iii) of Theorem \ref{thm:LP} for $P_j$ yields:
\begin{equation}\label{ad33}
\norm{P_jnL\rho}_{\lh{2}}\lesssim \norm{\La P_j(P_j\La^{-1}nL\rho)}_{\lh{2}}\lesssim 2^j\norm{P_j\La^{-1}nL\rho}_{\lh{2}}.
\end{equation}
Using the Bianchi identity \eqref{bianc2}, we have:
\be\lab{ad34}
nL(\rho)=\divb(n\b)-\nabb(n)\b -\frac{n}{2}\hchb\a+n(k_{AN}-2\kepb_A)\b.
\ee
Together with properties \eqref{La3} and \eqref{La5} of $\La$, this yields:
\bea
\norm{\La^{-1}nL(\rho)}_{\lh{2}}\lesssim\norm{n\b}_{\lh{2}}+\normm{\nabb(n)\b -\frac{n}{2}\hchb\a+n(\kep-2\kepb)\b}_{\tx{2}{\frac{4}{3}}}\lab{ad35}\\
\lesssim\ep+\norm{\nabb(n)}_{\tx{\infty}{4}}\norm{\b}_{\lh{2}} +\norm{\hchb}_{\tx{\infty}{4}}\norm{\a}_{\lh{2}}+\norm{\kep-2\kepb}_{\tx{\infty}{4}}\norm{\b}_{\lh{2}}\nn\\
\lesssim \ep(1+\no(\nabb(n))+\no(\hchb)+\no(\kep)+\no(\kepb))\lesssim \ep\nn,
\eea
where we have also used \eqref{sobineq1} to bound the $\tx{\infty}{4}$ norms, \eqref{curvflux1} to estimate $\a, \b$, and the bootstrap assumptions \eqref{boot1}-\eqref{boot5}. Now, \eqref{ad33} and \eqref{ad35} yield:
\begin{equation}\label{ad36}
\sum_{j\geq 0}2^{-2j}\norm{P_jnL\rho}^2_{\lh{2}}\lesssim \sum_{j\geq 0}\norm{P_j\La^{-1}nL\rho}^2_{\lh{2}}\lesssim \norm{\La^{-1}nL\rho}^2_{\lh{2}}\lesssim\ep^2.
\end{equation}

Property (ii) of Theorem \ref{thm:LP}, \eqref{ad27}, \eqref{ad30}, \eqref{gau2}, \eqref{ad33} and \eqref{ad35} imply:
\begin{equation}\label{ad36bis}
\begin{array}{l}
\ds\sum_{j\geq 0}2^{-j}\norm{P_j\rho}^2_{\tx{\infty}{2}}\ds\lesssim\sum_{j\ge 0} 2^{-j}\left(\int_0^1\norm{P_j\rho}_{\lpt{2}}\norm{nLP_j\rho}_{\lpt{2}}dt\right)+\ep^2\\
\ds\lesssim\sum_{j\ge 0} 2^{-j}\norm{P_j\rho}_{\lh{2}}\norm{P_jnL\rho}_{\lh{2}}\\
\ds +\sum_{j\ge 0} 2^{-j}\left(\int_0^1\norm{P_j\rho}_{\lpt{2}}\normm{\int_0^\infty m_j(\tau)V(\tau) d\tau}_{\lpt{2}}dt\right)+\ep^2\\
\ds\lesssim\sum_{j\ge 0}\norm{P_j\rho}^2_{\lh{2}}+\sum_{j\ge 0} 2^{-2j}\norm{P_jnL\rho}^2_{\lh{2}}\\
\ds +\sum_{j\ge 0} 2^{-j}\norm{P_j\rho}_{\tx{\infty}{2}}\normm{\int_0^\infty m_j(\tau)V(\tau) d\tau}_{\tx{1}{2}}+\ep^2\\
\ds\lesssim\left(\sum_{j\ge 0} 2^{-j}\norm{P_j\rho}^2_{\tx{\infty}{2}}\right)^{\frac{1}{2}}\\
\ds\times \left(\sum_{j\ge 0} 2^{-j}\normm{\int_0^\infty m_j(\tau)\norm{V(\tau)}_{\lpt{2}} d\tau}^2_{L^1(0,1)}\right)^{\frac{1}{2}}+\ep^2,
\end{array}
\end{equation}
which yields:
\begin{equation}\label{ad37}
\sum_{j\geq 0}2^{-j}\norm{P_j\rho}^2_{\tx{\infty}{2}}\lesssim\sum_{j\ge 0} 2^{-j}\normm{\int_0^\infty m_j(\tau)\norm{V(\tau)}_{\lpt{2}} d\tau}^2_{L^1(0,1)}+\ep^2.
\end{equation}

In view of \eqref{ad37}, we have to estimate $\norm{V(\tau)}_{\lpt{2}}$. Let $a, p$ real numbers satisfying:
\begin{equation}\label{ad38bis}
\begin{array}{l}
\ds 0<a<\frac{1}{2},\,2<p<+\infty,\textrm{ such that }p<\min\left(\frac{2}{1-a},\frac{8}{3}\right).
\end{array}
\end{equation}
\eqref{eq:l2heat1bis} implies:
\begin{equation}\label{ad46}
\begin{array}{ll}
& \ds\norm{\La^{-a}V(\tau)}^2_{\lpt{2}}+\int_0^\tau\norm{\nabb\La^{-a}V(\tau')}^2_{\lpt{2}}d\tau'\\
\ds\lesssim &\ds\int_0^\tau\int_{\ptu} \La^{-2a}V(\tau')[nL,\lap]U(\tau')\rho\dmt d\tau'.
\end{array}
\end{equation}
Let $p$ defined in \eqref{ad38bis}, and let $p'$ such that $\frac{1}{p}+\frac{1}{p'}=\frac{1}{2}$. 
Using the commutator formula \eqref{comm6}, \eqref{ad28}, and integrating by parts the term $\nabb^2U(\tau)\rho$ yields:
\bea\label{ad47}
&&\int_0^\tau\int_{\ptu} \La^{-2a}V(\tau')[nL,\lap]U(\tau')\rho\dmt d\tau'\\
\nn&\lesssim& (\norm{\nabb(n\chi)}_{\lpt{2}}+\norm{n(2\hch\kepb-\kep\trc-\nabb\trc)}_{\lpt{2}})\int_0^\tau\norm{\nabb U(\tau')}_{\lpt{p}}\\
\nn&&\times\norm{\La^{-2a}V(\tau')}_{\lpt{p'}}d\tau'+\norm{n\chi}_{\lpt{p'}}\int_0^\tau\norm{\nabb U(\tau')}_{\lpt{p}}\norm{\nabb\La^{-2a}V(\tau')}_{\lpt{2}}d\tau'.
\eea
\eqref{eq:GNirenberg}, \eqref{La1}, \eqref{interpolLa} and \eqref{ad26bis} yield:
\begin{equation}\label{ad48}
\begin{array}{l}
\ds\int_0^\tau\norm{\nabb U(\tau')}_{\lpt{p}}\norm{\La^{-2a}V(\tau')}_{\lpt{p'}}d\tau'+\int_0^\tau\norm{\nabb U(\tau')}_{\lpt{p}}\norm{\nabb\La^{-2a}V(\tau')}_{\lpt{2}}d\tau'\\
\ds\lesssim 
\int_0^\tau\norm{\nabb U(\tau')}^{\frac{2}{p}}_{\lpt{2}}\norm{\nabb^2 U(\tau')}^{1-\frac{2}{p}}_{\lpt{2}}\norm{\La^{-a}V(\tau')}^{a}_{\lpt{p'}}\norm{\nabb\La^{-a}V(\tau')}^{1-a}_{\lpt{p'}}d\tau'\\
\ds\lesssim E^{2(1-\frac{2}{p})}\int_0^\tau\norm{\nabb U(\tau')}^{\frac{2}{p}}_{\lpt{2}}\norm{\lap U(\tau')}^{1-\frac{2}{p}}_{\lpt{2}}\norm{\La^{-a}V(\tau')}^{a}_{\lpt{p'}}\norm{\nabb\La^{-a}V(\tau')}^{1-a}_{\lpt{p'}}d\tau'\\
\ds\lesssim \left(E^{4(1-\frac{2}{p})}\int_0^\tau\norm{\nabb U(\tau')}^{2}_{\lpt{2}}d\tau'
+ E^{4(1-\frac{2}{p})}\int_0^\tau \tau'\norm{\lap U(\tau')}^{2}_{\lpt{2}}d\tau'\right)^{\frac{1}{2}}\\
\ds\left(\frac{1}{2}\int_0^\tau\norm{\nabb\La^{-a}V(\tau')}^2_{\lpt{2}}d\tau'
+\int_0^\tau{\tau'}^{-\frac{p-2}{ap}}\norm{\nabb\La^{-a}V(\tau')}^2_{\lpt{2}}d\tau'\right)^{\frac{1}{2}}
\end{array}
\end{equation}
which together with the estimates for the heat flow \eqref{eq:l2heat1}, \eqref{eq:l2heat2} implies:
\begin{equation}\label{ad49}
\begin{array}{l}
\ds\int_0^\tau\norm{\nabb U(\tau')}_{\lpt{p}}\norm{\La^{-2a}V(\tau')}_{\lpt{p'}}d\tau'+\int_0^\tau\norm{\nabb U(\tau')}_{\lpt{p}}\norm{\nabb\La^{-2a}V(\tau')}_{\lpt{2}}d\tau'\\
\ds\lesssim E^{\frac{1}{2}}\norm{\rho}_{\lpt{2}}\left(\int_0^\tau\norm{\nabb\La^{-a}V(\tau')}^2_{\lpt{2}}d\tau'
+\int_0^\tau{\tau'}^{-\frac{p-2}{ap}}\norm{\La^{-a}V(\tau')}^2_{\lpt{2}}d\tau'\right)^{\frac{1}{2}}.
\end{array}
\end{equation}

Finally, the choice of $p$ \eqref{ad38bis}, \eqref{ad46}, \eqref{ad47} and \eqref{ad49} yield:
\begin{equation}\label{ad50}
\begin{array}{l}
\ds\norm{\La^{-a}V(\tau)}^2_{\lpt{2}}+\int_0^\tau\norm{\nabb\La^{-a}V(\tau')}^2_{\lpt{2}}d\tau'\\
\ds\lesssim E(\norm{\nabb(n\chi)}_{\lpt{2}}+\norm{n(2\hch\kepb-\kep\trc-\nabb\trc)}_{\lpt{2}})^2\norm{\rho}^2_{\lpt{2}}.
\end{array}
\end{equation}
Using the interpolation inequality \eqref{interpolLa}, we obtain:
\begin{equation}\label{ad51}
\begin{array}{l}
\ds\int_0^{+\infty}\norm{V(\tau)}^{\frac{2}{a}}_{\lpt{2}}d\tau\lesssim \int_0^\tau \norm{\La^{-a}V(\tau')}^{\frac{2(1-a)}{a}}_{\lpt{2}}\norm{\nabb\La^{-a}V(\tau')}^2_{\lpt{2}}d\tau'
\\
\ds\lesssim E^{\frac{1}{a}}(\norm{\nabb(n\chi)}_{\lpt{2}}+\norm{n(2\hch\kepb-\kep\trc-\nabb\trc)}_{\lpt{2}})^{\frac{2}{a}}\norm{\rho}^{\frac{2}{a}}_{\lpt{2}},
\end{array}
\end{equation}
which together with the bootstrap assumptions \eqref{boot1}-\eqref{boot5} and the estimate on $\rho$ \eqref{gau2} yields:
\begin{equation}\label{ad52}
\begin{array}{l}
\ds\normm{\int_0^{+\infty}m_j(\tau)\norm{V(\tau)}_{\lpt{2}}d\tau}_{L^1(0,1)}\lesssim 2^{ja}\normm{\left(\int_0^{+\infty}\norm{V(\tau)}^{\frac{2}{a}}_{\lpt{2}}d\tau\right)^{\frac{a}{2}}}_{L^1(0,1)}\\
\ds\lesssim 2^{ja}E^{\frac{1}{2}}(\no(\chi)(\no(\nabb n)+\norm{n}_{\lh{\infty}})+\norm{n}_{\lh{\infty}}(\no(\chi)\no(\kep)+\no(\chi))\norm{\rho}_{\lh{2}}\\
\lesssim 2^{ja}E^{\frac{1}{2}}\ep 
\end{array}
\end{equation}
In turn, we obtain together with \eqref{ad37} and the fact that $0<a<\frac{1}{2}$:
\begin{equation}\label{ad53}
\sum_{j\geq 0}2^{-j}\norm{P_j\rho}^2_{\tx{\infty}{2}}\lesssim\sum_{j\geq 0}2^{-j}2^{2ja}E\ep^2+\ep^2\lesssim E\ep^2.
\end{equation}
Using \eqref{gau1}, we obtain for $K$:
\begin{equation}\label{ad54}
\sum_{j\geq 0}2^{-j}\norm{P_jK}^2_{\tx{\infty}{2}}\lesssim E\ep^2,
\end{equation}
which is an improvement of \eqref{ad26}. Thus, we have proved:
\begin{equation}\label{ad55}
\sum_{j\geq 0}2^{-j}\norm{P_jK}^2_{\tx{\infty}{2}}\lesssim \ep^2,
\end{equation}
which together with \eqref{La8} concludes the proof of \eqref{estgauss2}.

\subsection{Proof of Lemma \ref{ineqtraceptu}}\lab{sec:gowinda}

Let $u_0<u_1$. We have
\bee
\int_{P_{t,u_1}}F(u_1, \c)\c\nabb G(u_1, \c)&=&\int_{P_{t,u_0}}F(u_0, \c)\c\nabb G(u_0, \c)+\int_{u_0}^{u_1}\partial_u\left(\int_{\ptu}F\c\nabb G\right) du\\
&\les& \int_{P_{t,u_0}}F(u_0, \c)\nabb G(u_0, \c)+\int_{\mathbb{R}}\left|\partial_u\left(\int_{\ptu}F\c\nabb G\right)\right| du.
\eee
Letting $u_0\rightarrow -\infty$ and taking the supremum in $u_1$, this yields
$$\sup_u\left(\int_{\ptu}F\c\nabb G\right)\les \int_{\mathbb{R}}\left|\partial_u\left(\int_{\ptu}F\c\nabb G\right)\right| du.$$
Together with \eqref{dusit}, this yields
\bee
&&\sup_u\left(\int_{\ptu}F\c\nabb G\right)\\
&\les& \int_u\left|\int_{\ptu}b(\nabb_NF\c\nabb G+F\c\nabb_N\nabb G+\trt F\c\nabb G)\dmt\right|du\\
&\les& \int_u\left|\int_{\ptu}b(\nabb_NF\c\nabb G+F\c \nabb\nabb_NG+F\c [\nabb_N, \nabb]G+\trt F\c\nabb G)\dmt\right|du.
\eee
Decomposing 
$$F\c\nabb\nabb_NG=\divb(F\c \nabb_N G)-\nabb F\c \nabb_NG$$
and integrating by parts the divergence term on $\ptu$, we deduce
\bee
&&\sup_u\left(\int_{\ptu}F\c\nabb G\right)\\
&\les& \int_u\int_{\ptu}\Big|\nabb_NF\c\nabb G-\nabb F\c \nabb_N G-b^{-1}\nabb b\c F\c \nabb_NG+F\c [\nabb_N, \nabb]G+\trt F\c\nabb G\Big| b\dmt du,
\eee
which in view of the coarea formula \eqref{coarea} yields
\bea\lab{themonster}
&&\sup_u\left(\int_{\ptu}F\c\nabb G\right)\\
\nn&\les& \int_{\Sigma_t}\Big|\nabb_NF\c\nabb G-\nabb F\c \nabb_N G-b^{-1}\nabb b\c F\c \nabb_NG+F\c [\nabb_N, \nabb]G+\trt F\c\nabb G\Big| d\Sigma_t\\
\nn&\les& \norm{\nabla F}_{L^2(\Sigma_t)}\norm{\nabla G}_{L^2(\Sigma_t)}+(\norm{b^{-1}\nabb b}_{L^4(\Sigma_t)}+\norm{\trt}_{L^4(\Sigma_t)})\norm{F}_{L^4(\Sigma_t)}\norm{\nabla G}_{L^2(\Sigma_t)}\\
\nn&&+\norm{F}_{L^4(\Sigma_t)}\norm{[\nabb_N, \nabb]G}_{L^{\frac{4}{3}}(\Sigma_t)}.
\eea
The commutator formula \eqref{comm4} implies
\bee
&&\norm{[\nabb_N, \nabb]G}_{L^{\frac{4}{3}}(\Sigma_t)}\\
&\les& (\norm{b^{-1}\nabb b}_{L^4(\Sigma_t)}+\norm{\chi}_{L^4(\Sigma_t)}+\norm{\eta}_{L^4(\Sigma_t)})\norm{\nabla G}_{L^2(\Sigma_t)}\\
&&+(\norm{\rr}_{L^2(\Sigma_t)}+(\norm{\chi}_{L^4(\Sigma_t)}+\norm{\kepb}_{L^4(\Sigma_t)}+\norm{\xib}_{L^4(\Sigma_t)}+\norm{\chb}_{L^4(\Sigma_t)}+\norm{\zeta}_{L^4(\Sigma_t)})^2)\norm{G}_{L^4(\Sigma_t)},
\eee
which together with \eqref{themonster} and the Sobolev embedding \eqref{sobineqsit} on $\Sigma_t$ implies
\bea\lab{themonster1}
&&\sup_u\left(\int_{\ptu}F\c\nabb G\right)\\
\nn&\les& \Big(1+\norm{b^{-1}\nabb b}_{L^4(\Sigma_t)}+\norm{\chi}_{L^4(\Sigma_t)}+\norm{\eta}_{L^4(\Sigma_t)}+(\norm{\rr}_{L^2(\Sigma_t)}+(\norm{\chi}_{L^4(\Sigma_t)}+\norm{\kepb}_{L^4(\Sigma_t)}\\
\nn&&+\norm{\xib}_{L^4(\Sigma_t)}+\norm{\chb}_{L^4(\Sigma_t)}+\norm{\zeta}_{L^4(\Sigma_t)})^2\Big)\norm{F}_{H^1(\Sigma_t)}\norm{G}_{H^1(\Sigma_t)},
\eea
where we used in the last inequality the definition \eqref{def:theta} of $\th$ in the last inequality. Now, in view of the embedding \eqref{sobineq1}, we have
\bea\lab{themonster2}
\nn&&\norm{b^{-1}\nabb b}_{L^4(\Sigma_t)}+\norm{\chi}_{L^4(\Sigma_t)}+\norm{\eta}_{L^4(\Sigma_t)}+\norm{\kepb}_{L^4(\Sigma_t)}+\norm{\xib}_{L^4(\Sigma_t)}+\norm{\chb}_{L^4(\Sigma_t)}+\norm{\zeta}_{L^4(\Sigma_t)}\\
\nn&\les& \norm{b^{-1}\nabb b}_{\tx{\infty}{4}}+\norm{\chi}_{\tx{\infty}{4}}+\norm{\eta}_{\tx{\infty}{4}}+\norm{\kepb}_{\tx{\infty}{4}}+\norm{\xib}_{\tx{\infty}{4}}+\norm{\chb}_{\tx{\infty}{4}}+\norm{\zeta}_{\tx{\infty}{4}}\\
\nn&\les& \no(b^{-1}\nabb b)+\no(\chi)+\no(\eta)+\no(\kep)+\no(\xib)+\no(\chb)+\no(\zeta)\\
&\les& 1,
\eea
where we used in the last inequality the bootstrap assumptions \eqref{boot1}-\eqref{boot6} for $b$, $\chi$, $\eta$, $\kep$, $\xib$, $\chb$ and $\zeta$. Finally, \eqref{themonster1}, \eqref{themonster2} and the assumption \eqref{curvflux1} on $\rr$ yield
$$\sup_u\left(\int_{\ptu}F\c\nabb G\right)\les\norm{F}_{H^1(\Sigma_t)}\norm{G}_{H^1(\Sigma_t)},$$
which is the desired estimate. This concludes the proof of Lemma \ref{ineqtraceptu}.

\subsection{Proof of Lemma \ref{lemma:lbt8}}\lab{sec:gowinda1} 

Using the formula \eqref{eq:commLdcal-1} for the commutator $[{}^*\mathcal{D}^{-1}_1,\ddb_{nL}]$, we have:
$$[{}^*\mathcal{D}^{-1}_1,\ddb_{nL}](\bb)={}^*\mathcal{D}^{-1}_1[{}^*\mathcal{D}_1,\ddb_{nL}]({}^*\mathcal{D}^{-1}_1(\bb))$$
which in view of Lemma \ref{lemma:lbt6} yields:
\bea\lab{minou10}
&&\norm{[{}^*\mathcal{D}^{-1}_1,\ddb_{nL}](\bb)}_{\tx{2}{3}}+\norm{[{}^*\mathcal{D}^{-1}_1,\ddb_{nL}](\bb)}_{\tx{1}{4}}\\
\nn&\les& \norm{[{}^*\mathcal{D}_1,\ddb_{nL}]({}^*\mathcal{D}^{-1}_1(\bb))}_{\tx{2}{\frac{4}{3}}}+\norm{[{}^*\mathcal{D}_1,\ddb_{nL}]({}^*\mathcal{D}^{-1}_1(\bb))}_{\tx{1}{\frac{3}{2}}}.
\eea
Now, from the commutator formula \eqref{comm5} and the fact that ${}^*\mathcal{D}^{-1}_1(\bb)$ is a scalar, we have 
$$[{}^*\mathcal{D}_1,\ddb_{nL}]({}^*\mathcal{D}^{-1}_1(\bb))=n\chi\c\nabb({}^*\mathcal{D}^{-1}_1(\bb))$$
which together with \eqref{minou10} implies:
\bee
&&\nn\norm{[{}^*\mathcal{D}^{-1}_1,\ddb_{nL}](\bb)}_{\tx{2}{3}}+\norm{[{}^*\mathcal{D}^{-1}_1,\ddb_{nL}](\bb)}_{\tx{1}{4}}\\
\nn&\les& \norm{n\chi\c\nabb({}^*\mathcal{D}^{-1}_1(\bb))}_{\tx{2}{\frac{4}{3}}}+\norm{n\chi\c\nabb({}^*\mathcal{D}^{-1}_1(\bb))}_{\tx{1}{\frac{3}{2}}}\\
\nn&\les&\norm{n}_{L^\infty}(\norm{\chi}_{\tx{\infty}{4}}+\norm{\chi}_{\tx{2}{6}})\norm{\nabb({}^*\mathcal{D}^{-1}_1(\bb)}_{\lh{2}}\\
\nn&\les&\no(\chi)\norm{\bb}_{\lh{2}}\\
&\les&D\ep^2\\
&\les& \ep,
\eee
where we used the bootstrap assumption \eqref{boot1} for $n$, the bootstrap assumptions \eqref{boot4} \eqref{boot5} for $\chi$, the curvature bound \eqref{curvflux1} for $\bb$, and the estimate \eqref{eq:estimdcal-1} for ${}^*\mathcal{D}^{-1}_1$. This concludes the proof of Lemma \ref{lemma:lbt8}.

\section{Appendix to section \ref{sec:secondderlb}}

\subsection{Proof of Lemma \ref{lemma:lbz2}}\lab{sec:lbz1}

We decompose $\norm{Pj(H\c F)}_{\lh{2}}$ using the property \eqref{eq:partition} of the geometric Littlewood-Paley projections:
\be\lab{lbz11}
\norm{P_j(H\c F)}_{\lh{2}}\les \norm{P_j(H\c P_{<0}F)}_{\lh{2}}+\sum_{l\geq 0}\norm{P_j(H\c P_lF)}_{\lh{2}}.
\ee
We focus on the second term in the right-hand side of \eqref{lbz11}, the other being easier to handle. We start with the case $l\leq j$. Using the assumption \eqref{lbz3} for $F$, the Sobolev inequality \eqref{sobineq1} and the weak Bernstein inequality iv) of Theorem \ref{thm:LP}, we have:
\bee
\norm{P_j(H\c P_lF)}_{\lh{2}}&\les& \norm{H\c P_lF}_{\lh{2}}\\
&\les& \norm{H}_{\tx{\infty}{4}}\norm{P_lF}_{\tx{2}{4}}\\
&\les& \no(H)2^{\frac{l}{2}}\norm{P_lF}_{\lh{2}}\\
&\les& \no(H)(2^lC_1+2^{\frac{l}{2}}C_2),
\eee
 which yields:
\be\lab{lbz12}
\sum_{l\leq j}\norm{P_j(H\c P_lF)}_{\lh{2}}\les \sum_{l\leq j}\no(H)(2^lC_1+2^{\frac{l}{2}}C_2)\les \no(H)(2^jC_1+2^{\frac{j}{2}}C_2).
\ee 

We now focus on the case $l>j$. We further decompose:
\be\lab{lbz13}
\norm{P_j(H\c P_lF)}_{\lh{2}}\les \norm{P_j(P_{\leq l}H\c P_lF)}_{\lh{2}}+\norm{P_j(P_{>l}H\c P_lF)}_{\lh{2}}.
\ee 
We evaluate first the second term in the right-hand side of \eqref{lbz13}. Using the dual of the sharp Bernstein inequality \eqref{eq:strongbernscalarbis} for scalars,  
we obtain:
\bea
\lab{lbz15}\norm{P_j(P_{>l}H\c P_lF)}_{\lh{2}}&\les& 2^j\norm{P_{>l}H\c P_lF}_{\tx{2}{1}}\\
\nn&\les& \norm{P_{>l}H}_{\tx{\infty}{2}}\norm{P_lF}_{\lh{2}}\\
\nn&\les& 2^{-\frac{l}{2}}\no(H)(C_1+2^{-\frac{l}{2}}C_2),
\eea
where we used the assumption \eqref{lbz3} for $F$ and the estimate \eqref{lbz14} for $H$. We now consider the first term in the right-hand side of \eqref{lbz13}. Using \eqref{lbz14bis} with $p=\frac{4}{3}$, 
 the dual of the sharp Bernstein inequality \eqref{eq:strongbernscalarbis} for scalars and \eqref{lbz14}, we obtain:
\bea\lab{lbz16}
&&\norm{P_j(P_{\leq l}H\c P_lF)}_{\lh{2}}\\
\nn&=& 2^{-2l}\norm{P_j(P_{\leq l}H\c\lap P_lF)}_{\lh{2}}\\
\nn&\les & 2^{-2l}\norm{P_j(\divb(P_{\leq l}H\c\nabb P_lF))}_{\lh{2}}+2^{-2l}\norm{P_j(\nabb P_{\leq l}H\c\nabb P_lF)}_{\lh{2}}\\
\nn&\les & 2^{-2l}2^{\frac{3}{2}j}\norm{P_{\leq l}H\c\nabb P_lF}_{\tx{2}{\frac{4}{3}}}+2^{-2l+j}\norm{\nabb P_{\leq l}H\c\nabb P_lF}_{\tx{2}{1}}\\
\nn&\les & 2^{-2l+\frac{3}{2}j}\norm{H}_{\tx{\infty}{4}}\norm{\nabb P_lF}_{\lh{2}}+2^{-2l+j}\norm{\nabb P_{\leq l}H}_{\tx{\infty}{2}}\norm{\nabb P_lF}_{\lh{2}}\\
\nn&\les& (2^{\frac{3}{2}j-l}+2^{j-\frac{1}{2}l})\no(H)(C_1+2^{-\frac{l}{2}}C_2).
\eea

\subsection{Proof of Lemma \ref{lemma:lbz1}}\lab{sec:lbz0}

We decompose $\norm{Pj(hf)}_{\tx{p}{2}}$ using the property \eqref{eq:partition} of the geometric Littlewood-Paley projections:
\be\lab{lbz60}
\norm{P_j(hf)}_{\tx{p}{2}}\les \norm{P_j(hP_{<0}f)}_{\tx{p}{2}}+\sum_{l\geq 0}\norm{P_j(hP_lf)}_{\tx{p}{2}}.
\ee
We focus on the second term in the right-hand side of \eqref{lbz60}, the other being easier to handle. 
Using the $L^2$-boundedness of the Littlewood-Paley projection $P_j$, we have:
\be\lab{lbz61}
\norm{P_j(hP_lf)}_{\tx{p}{2}}\les \norm{P_lf}_{\tx{p}{2}}\les \norm{h}_{L^\infty}(2^lC_1+2^{\frac{l}{2}}C_2).
\ee
We now decompose $\norm{P_j(hP_lf)}_{\tx{p}{2}}$ again using the property \eqref{eq:partition} of the geometric Littlewood-Paley projections:
\be\lab{lbz62}
\norm{P_j(hP_lf)}_{\tx{p}{2}}\les \norm{P_j(P_{<0}(h)P_lf)}_{\tx{p}{2}}+\sum_{q\geq 0}\norm{P_j(P_q(h)P_lf)}_{\tx{p}{2}}.
\ee
We focus on the second term in the right-hand side of \eqref{lbz62}, the other being easier to handle. 
We have:
\bea
\lab{lbz63}\norm{P_j(P_q(h)P_lf)}_{\tx{p}{2}}&\les& 2^j\norm{P_q(h)P_lf}_{\tx{p}{1}}\\
\nn&\les & 2^j\norm{P_q(h)}_{\tx{\infty}{2}}\norm{P_lf}_{\tx{p}{2}}\\
\nn&\les & 2^j(2^lC_1+2^{\frac{l}{2}})\norm{P_q(h)}_{\tx{\infty}{2}},
\eea
where we used in the last inequality the assumption \eqref{lbz1} for $f$.
 
We now derive a second estimate. Using the properties of the Littlewood-Paley projection $P_l$, we have:
\bea
\lab{lbz65}
\norm{P_j(P_q(h)P_lf)}_{\tx{p}{2}}&\les& 2^{-2l}\norm{P_j(P_q(h)\lap P_lf)}_{\tx{p}{2}}\\
\nn&\les & 2^{-2l}\norm{P_j(\lap(P_q(h)P_lf))}_{\tx{p}{2}}+2^{-2l}\norm{P_j(\divb(\nabb P_q(h)P_lf))}_{\tx{p}{2}}\\
\nn&&+2^{-2l}\norm{P_j(\lap(P_q(h))P_lf)}_{\tx{p}{2}}\\
\nn&\les & 2^{2j-2l}\norm{P_q(h)P_lf}_{\tx{p}{2}}+2^{j-2l}\norm{\nabb P_q(h)P_lf}_{\tx{p}{2}}\\
\nn&&+2^{j+2q-2l}\norm{P_q(h)P_lf}_{\tx{p}{1}}\\
\nn&\les & 2^{2j-2l}\norm{P_q(h)}_{L^\infty}\norm{P_lf}_{\tx{p}{2}}+2^{j-2l}\norm{\nabb P_q(h)}_{\tx{\infty}{4}}\norm{P_lf}_{\tx{p}{4}}\\
\nn&&+2^{j+2q-2l}\norm{P_q(h)}_{\tx{\infty}{2}}\norm{P_lf}_{\tx{p}{2}}\\
\nn&\les & \bigg(2^{2j-2l+q}\norm{P_q(h)}_{\tx{\infty}{2}}+2^{j-\frac{3l}{2}}\norm{\nabb^2P_q(h)}^{\half}_{\tx{\infty}{4}}\norm{\nabb P_q(h)}^{\half}_{\tx{\infty}{2}}\\
\nn&&+2^{j+2q-2l}\norm{P_q(h)}_{\tx{\infty}{2}}\bigg)\norm{P_lf}_{\tx{p}{2}}\\
\nn&\les & (2^{2j-2l+q}+2^{j-\frac{3l}{2}+\frac{3q}{2}}+2^{j+2q-2l})\norm{P_q(h)}_{\tx{\infty}{2}}(2^l\ep+2^{\frac{l}{2}}\ep\gamma(u)),
\eea 
where we used the dual of the sharp Bernstein inequality \eqref{eq:strongbernscalarbis} and the finite 
band property for $P_j$, the weak Bernstein inequality for $P_l$, the Gagliardo-Nirenberg inequality \eqref{eq:GNirenberg}, the Bochner inequality \eqref{eq:Bochconseqbis} for scalars, the finite band property and the sharp Bernstein inequality \eqref{eq:strongbernscalarbis} for $P_q$, and the assumption \eqref{lbz1} for $f$.

Then, using \eqref{lbz63} when $q>l$, and \eqref{lbz65} when $q\leq l$, we obtain:
$$\sum_{q,l>j}\norm{P_j(P_q(b)P_lf)}_{\tx{p}{2}}\les \norm{h}_{\BB^1}(2^jC_1+2^{\frac{j}{2}}C_2)$$
which together with \eqref{lbz62} yields:
\be\lab{lbz67}
\sum_{l>j}\norm{P_j(hP_lf)}_{\tx{p}{2}}\les \norm{h}_{\BB^1}(2^jC_1+2^{\frac{j}{2}}C_2).
\ee
Finally, using \eqref{lbz61} when $l\leq j$ and \eqref{lbz67} when $l>j$, we obtain:
$$\sum_l\norm{P_j(hP_lf)}_{\tx{p}{2}}\les (\norm{h}_{L^\infty}+\norm{h}_{\BB^1})(2^jC_1+2^{\frac{j}{2}}C_2),$$
which together with \eqref{lbz60} implies:
\be\lab{lbz67:bis}
\norm{P_j(hf)}_{\tx{p}{2}}\les (\norm{h}_{L^\infty}+\norm{h}_{\BB^1})(2^jC_1+2^{\frac{j}{2}}C_2).
\ee
Now, the embedding \eqref{linftybound} applied to $h$ together with \eqref{lbz67:bis} concludes the proof of Lemma \ref{lemma:lbz1}.

\subsection{Proof of Lemma \ref{lemma:lbz3}}\lab{sec:lbz2}

Let $f$ the scalar function on $\H_u$ defined by $f=\mathcal{D}_1(F)$. The assumption \eqref{lbz5} now reads for all $j\geq 0$:
\be\lab{lbz17}
\norm{P_jf}_{\lh{2}}\lesssim 2^jC_1+2^{\frac{j}{2}}C_2
\ee
where $C_1, C_2$ are constants possibly depending on $u$. From the definition of $f$, we have $F=\mathcal{D}_1^{-1}(f)$. We decompose the norm $\norm{P_jF}_{\lh{2}}$ using the property \eqref{eq:partition} of the geometric Littlewood-Paley projections:
\be\lab{lbz18}
\norm{P_jF}_{\lh{2}}\lesssim \norm{P_j\mathcal{D}^{-1}P_{<0}(f)}_{\lh{2}}+\sum_{q\geq 0}\norm{P_j\mathcal{D}^{-1}P_q(f)}_{\lh{2}}.
\ee
The first term in the right-hand side of \eqref{lbz17} is easier to handle, so we focus on the sum in $q$.  We have:
\be\lab{lbz19}
\norm{P_j\mathcal{D}_1^{-1}P_q(f)}_{\lh{2}}\lesssim 2^{-j}\norm{\nabb\mathcal{D}_1^{-1}}_{\mathcal{L}(\lpt{2})}\norm{P_q(f)}_{\lh{2}}\lesssim 2^{-j}(2^qC_1+2^{\frac{q}{2}}C_2),
\ee
where we used the finite band properties of the Littlewood-Paley projection $P_j$, the estimate \eqref{eq:estimdcal-1} for $\mathcal{D}_1^{-1}$ and \eqref{lbz17}. 
We now derive a second estimate. Using the properties of the Littlewood-Paley projection $P_q$ 
and the identity \eqref{eq:dcalident} for $\mathcal{D}_1$, we have:
\bea\lab{lbz20}
\norm{P_j\mathcal{D}_1^{-1}P_q(f)}_{\lh{2}}&\lesssim& 2^{-2q}\norm{P_j\mathcal{D}_1^{-1}\lap P_q(f)}_{\lpt{2}}\\
\nn&\lesssim& 2^{-2q}\norm{P_j{}^*\mathcal{D}_1 P_q(f)}_{\lpt{2}}\\
\nn&\lesssim& 2^{j-2q}\norm{P_q(f)}_{\lpt{2}}\\
\nn&\lesssim& 2^{j-2q}(2^qC_1+2^{\frac{q}{2}}C_2),
\eea
where we used the finite band properties of the Littlewood-Paley projection $P_j$ and \eqref{lbz17}. 

We now use \eqref{lbz19} for $q\leq j$ and \eqref{lbz20} for $q>j$ to obtain:
\bee
\sum_{q\geq 0}\norm{P_j\mathcal{D}_1^{-1}P_q(f)}_{\lh{2}}&\lesssim &\sum_{q\leq l}2^{-j}(2^qC_1+2^{\frac{q}{2}}C_2)+\sum_{q>l}2^{j-2q}(2^qC_1+2^{\frac{q}{2}}C_2)\\
&\lesssim & C_1+2^{-\frac{j}{2}}C_2
\eee
which together with \eqref{lbz18} concludes the proof of Lemma \ref{lemma:lbz3}

\subsection{Proof of Lemma \ref{lemma:lbz4}}\lab{sec:lbz3}

First, from the finite band property of the Littlewood-Paley projection $P_j$, we have:
\be\lab{lbz21}
\norm{\nabb P_jF}_{\tx{\infty}{2}}\lesssim 2^j\norm{P_jF}_{\tx{\infty}{2}}
\ee
so that we only need to estimate the first term in the left-hand side of \eqref{lbz14}.

Using \eqref{ad27bis}, properties (ii) and (iii) of Theorem \ref{thm:LP} for $P_j$, and the $L^\infty$ bound on $n$ given by \eqref{estn}, we have:
\begin{displaymath}
\begin{array}{ll}
&\ds\sum_{j\geq 0}2^{j}\norm{P_jF}^2_{\tx{\infty}{2}}\\
\ds\lesssim &\ds\sum_{j\ge 0} 2^{j}\left(\int_0^1\norm{P_jF}_{\lpt{2}}\norm{\ddb_LP_jF}_{\lpt{2}}dt+\norm{P_jF}_{\lh{2}}^2\right)\\
\ds\lesssim &\ds\sum_{j\ge 0} 2^{j}\left(\int_0^1\norm{P_jF}_{\lpt{2}}\norm{\ddb_{nL} P_jF}_{\lpt{2}}dt\right)+\sum_{j\ge 0}2^{j}\norm{P_jF}^2_{\lh{2}}\\
\ds\lesssim &\ds\sum_{j\ge 0} 2^{j}\norm{P_jF}_{\lh{2}}\norm{P_j\ddb_{nL}F}_{\lh{2}}+\ds\sum_{j\ge 0}2^j\norm{P_jF}_{\tx{\infty}{2}}\norm{[P_j,\ddb_{nL}]F}_{\tx{1}{2}}+\no(F)^2\\
\ds\lesssim &\ds\left(\sum_{j\ge 0}2^{j}\norm{P_jF}^2_{\tx{\infty}{2}}\right)^{\frac{1}{2}}\left(\sum_{j\ge 0}2^{j}\norm{[P_j,\ddb_{nL}]F}^2_{\tx{1}{2}}\right)^{\frac{1}{2}}+\no(F)^2
\end{array}
\end{displaymath}
which yields:
\begin{equation}\label{lbz22}
\sum_{j\geq 0}2^{j}\norm{P_jF}^2_{\tx{\infty}{2}}\lesssim\sum_{j\ge 0}2^{j}\norm{[P_j,\ddb_{nL}]F}^2_{\tx{1}{2}}+\no(F)^2.
\end{equation}
Now, the commutator estimate \eqref{supercommut2} and \eqref{lbz22} yield
$$\sum_{j\geq 0}2^{j}\norm{P_jF}^2_{\tx{\infty}{2}}\lesssim\sum_{j\ge 0}2^j2^{-2j_+}\no(F)^2+\no(F)^2\lesssim \no(F)^2$$
which together with \eqref{lbz21} concludes the proof of Lemma \ref{lemma:lbz4}.

\subsection{Proof of Lemma \ref{lemma:lbz5}}\lab{sec:lbz4}

By duality, it suffices to prove for any scalar function $f$ on $\ptu$, for any $2\leq p<+\infty$ and for all $j\geq 0$ the following inequality:
\be\lab{lbz14ter}
\norm{\nabb P_jf}_{\lpt{p}}\les 2^{2(1-\frac{1}{p})j}\norm{f}_{\lpt{2}}.
\ee
Now, using the Gagliardo-Nirenberg inequality \eqref{eq:GNirenberg}, the Bochner inequality for scalar functions \eqref{eq:Bochconseqbis}, and the property iii) of Theorem \ref{thm:LP} for Littlewood-Paley projections, we have:
\bee
\norm{\nabb P_jf}_{\lpt{p}}&\les& \norm{\nabb^2P_jf}^{1-\frac{2}{p}}_{\lpt{2}}\norm{\nabb P_jf}^{\frac{2}{p}}_{\lpt{2}}\\
&\les& (\norm{\lap P_jf}_{\lpt{2}}+\norm{\nabb P_jf}_{\lpt{2}})^{1-\frac{2}{p}}\norm{\nabb P_jf}^{\frac{2}{p}}_{\lpt{2}}\\
&\les& 2^{2j(1-\frac{1}{p})}\norm{f}_{\lpt{2}},
\eee
which is \eqref{lbz14ter}. This concludes the proof of Lemma \ref{lemma:lbz5}.

\subsection{Proof of Lemma \ref{lemma:lbtinit}}\lab{sec:lbtinit}

We first decompose $f$ using the property \eqref{eq:partition} for the Littlewood-Paley projections $P_l$. We have:
\be\lab{spi1}
f=\sum_lf_l
\ee
where $f_l$ is the solution of the following transport equation:
\be\lab{spi2}
nL(f_l)=0,\,f_l=P_lf_0\textrm{ on }\pou.
\ee

Using the $L^2$ boundedness of $P_j$, the equation \eqref{spi2}, and the estimate \eqref{estimtransport1} for transport equations applied to $f_l$, we have:
\be\lab{spi3}
\norm{P_jf_l}_{\tx{\infty}{2}}\les \norm{f_l}_{\tx{\infty}{2}}\les \norm{P_lf_0}_{L^2(\pou)}\les C2^{\frac{l}{2}}.
\ee

Next, we derive a second estimate for $\norm{P_jf_l}_{\tx{\infty}{2}}$. We define $v_l$ as 
$$v_l=-2^{-2l}\lap f_l+f_l$$
which in view of \eqref{spi2} satisfies the following transport equation:
\be\lab{spi4}
nL(v_l)=-2^{-2l}[nL,\lap]f_l,\, v_l=0\textrm{ on }\pou.
\ee
The definition of $v_l$ yields:
$$P_j(f_l)=2^{-2l}P_j(\lap f_l)+P_j(v_l)$$
which together with the finite band property for $P_j$ implies:
\bea
\lab{spi5}\norm{P_j(f_l)}_{\tx{\infty}{2}}&\les& 2^{2j-2l}\norm{f_l}_{\tx{\infty}{2}}+\norm{P_j(v_l)}_{\tx{\infty}{2}}\\
\nn&\les& 2^{2j-\frac{3l}{2}}C+2^{-2l}\normm{P_j\left(\int_0^t([nL,\lap]f)\right)}_{\tx{\infty}{2}},
\eea
where we used the estimate \eqref{spi3} for $f_l$, and the transport equation \eqref{spi4} for $v_l$ in the last inequality. Next, we estimate the second term in the right-hand side of \eqref{spi5}. The commutator formula \eqref{comm6} implies:
\bee
\normm{P_j\left(\int_0^t([nL,\lap]f)\right)}_{\tx{\infty}{2}}&\les& \normm{P_j\left(\int_0^t\divb(n\chi\c\nabb f_l)\right)}_{\tx{\infty}{2}}\\
\nn&&+\normm{P_j\left(\int_0^t(n(\nabb\chi+\chi\c(\kepb+n^{-1}\nabb n)\nabb f_l)\right)}_{\tx{\infty}{2}}
\eee
which together with Lemma \ref{lemma:lbt5} and the dual of the sharp Bernstein inequality 
\eqref{eq:strongbernscalarbis} for $P_j$ yields:
\bea
\lab{spi6}&&\normm{P_j\left(\int_0^t([nL,\lap]f)\right)}_{\tx{\infty}{2}}\\
\nn&\les& 2^j\norm{n\chi\c\nabb f_l}_{\xt{2}{1}}+2^j\normm{\int_0^t(n(\nabb\chi+\chi\c(\kepb+n^{-1}\nabb n)\nabb f_l)}_{\tx{\infty}{1}}\\
\nn&\les& 2^j\norm{n}_{L^\infty}\norm{\chi}_{\xt{\infty}{2}}\norm{\nabb f_l}_{\lh{2}}+2^j\norm{n(\nabb\chi+\chi\c(\kepb+n^{-1}\nabb n)f_l}_{\lh{1}}\\
\nn&\les& 2^j\ep\norm{\nabb f_l}_{\lh{2}}+2^j\norm{n}_{L^\infty}(\norm{\nabb\chi}_{\lh{2}}+\no(\chi)(\no(\kepb)+\no(n^{-1}\nabb n))\norm{f_l}_{\lh{2}}\\
\nn&\les& 2^j\ep\norm{\nabb f_l}_{\lh{2}},
\eea
where we used the estimate \eqref{estimtransport1} for transport equations, and the estimates \eqref{estn}-\eqref{esthch} for $n, \chi$, and $\kepb$.

In view of \eqref{spi6}, we need to estimate $\norm{\nabb f_l}_{\lh{2}}$. In view of the transport equation \eqref{spi2} satisfied by $f_l$, we have:
$$nL(\nabb f_l)=[nL,\nabb]f_l,\,\nabb f_l=\nabb P_lf_0\textrm{ on }\pou.$$
Together with the estimate \eqref{estimtransport1} and the commutator formula \eqref{comm5}, this  yields:
\bee
\norm{\nabb f_l}_{\lh{2}}&\les& \norm{\nabb P_lf_0}_{L^2(\pou)}+\norm{[nL,\nabb]f_l}_{\xt{2}{1}}\\
&\les& 2^l\norm{P_lf_0}_{L^2(\pou)}+\norm{n\chi\c\nabb f_l}_{\xt{2}{1}}\\
&\les& 2^{\frac{3l}{2}}C+\norm{n}_{L^\infty}\norm{\chi}_{\xt{\infty}{2}}\norm{\nabb f_l}_{\lh{2}}\\
&\les& 2^{\frac{3l}{2}}C+\ep\norm{\nabb f_l}_{\lh{2}},
\eee
where we used the finite band property for $P_l$, the assumption on $f_0$, and the estimates \eqref{estn}-\eqref{esthch} for $n$ and $\chi$. Since $\ep$ is small, we obtain:
\be\lab{spi7}
\norm{\nabb f_l}_{\lh{2}}\les 2^{\frac{3l}{2}}C.
\ee

Finally, \eqref{spi6} and \eqref{spi7} imply:
$$\normm{P_j\left(\int_0^t([nL,\lap]f)\right)}_{\tx{\infty}{2}}\les 2^{j+\frac{3l}{2}}C\ep,$$
which together with \eqref{spi5} yields:
\be\lab{spi8}
\norm{P_j(f_l)}_{\tx{\infty}{2}}\les 2^{2j-\frac{3l}{2}}C+2^{j-\frac{l}{2}}C\ep.
\ee
Now, using \eqref{spi1}, and summing \eqref{spi3} for $l\leq j$ and \eqref{spi8} for $l>j$, we obtain:
$$\norm{P_jf}_{\tx{\infty}{2}}\les \sum_l\norm{P_j(f_l)}_{\tx{\infty}{2}}\les C2^{\frac{j}{2}}$$
which yields the conclusion of Lemma \ref{lemma:lbtinit}.

\subsection{Proof of Lemma \ref{lemma:lbt1}}\lab{sec:lbt1}

We decompose $\norm{P_j\left(\int_0^t(f\mu_1)d\tau\right)}_{\lh{2}}$ using the property \eqref{eq:partition} of the geometric Littlewood-Paley projections:
\bea\lab{lbt13}
&&\normm{P_j\left(\int_0^t(f\mu_1)d\tau\right)}_{\lh{2}}\\
\nn&\les& \normm{P_j\left(\int_0^t(fP_{<0}\mu_1)d\tau\right)}_{\lh{2}}+\sum_{l\geq 0}\normm{P_j\left(\int_0^t(fP_l\mu_1)d\tau\right)}_{\lh{2}}.
\eea
We focus on the second term in the right-hand side of \eqref{lbt13}, the other being easier to handle.  Using the $L^2$ boundedness of the Littlewood-Paley projection $P_j$ and the estimate for transport equations \eqref{estimtransport1}, we have:
\bea
\lab{lbt14}\normm{P_j\left(\int_0^t(fP_l\mu_1)d\tau\right)}_{\lh{2}}&\les &\normm{\int_0^t(fP_l\mu_1)d\tau}_{\lh{2}}\\
\nn&\les& \norm{fP_l\mu_1}_{\xt{2}{1}}\\
\nn&\les& \norm{f}_{\xt{\infty}{2}}\norm{P_l\mu_1}_{\lh{2}}\\
\nn&\les& D\ep^22^l+D\ep^22^{\frac{l}{2}}\gamma(u),
\eea
where we used the estimate \eqref{lbt1} for $\mu_1$ and the assumption of Lemma \ref{lemma:lbt1} for $f$.

We now make another decomposition using the property \eqref{eq:partition} of the geometric Littlewood-Paley projections:
\bea\lab{lbt15}
&&\normm{P_j\left(\int_0^t(fP_l\mu_1)d\tau\right)}_{\lh{2}}\\
\nn&\les& \normm{P_j\left(\int_0^t(P_{<0}(f)P_l\mu_1)d\tau\right)}_{\lh{2}}+\sum_{q\geq 0}\normm{P_j\left(\int_0^t(P_q(f)P_l\mu_1)d\tau\right)}_{\lh{2}}.
\eea
We focus on the second term in the right-hand side of \eqref{lbt15}, the other being easier to handle. Using the dual of the sharp Bernstein inequality \eqref{eq:strongbernscalarbis} and the estimate for transport equations \eqref{estimtransport1}, we have:
\bea
\lab{lbt16}\normm{P_j\left(\int_0^t(P_q(f)P_l\mu_1)d\tau\right)}_{\lh{2}}&\les & 2^j\normm{\int_0^t(P_q(f)P_l\mu_1)d\tau}_{\tx{2}{1}}\\
\nn&\les& 2^j\norm{P_q(f)P_l\mu_1}_{\lh{1}}\\
\nn&\les & 2^j\norm{P_qf}_{\lh{2}}\norm{P_lf}_{\lh{2}}\\
\nn&\les & 2^j(D\ep 2^l+D\ep 2^{\frac{l}{2}}\gamma(u))\norm{P_qf}_{\lh{2}}
\eea
where we used the estimate \eqref{lbt1} for $\mu_1$ in the last inequality.

We now derive a second estimate. Using the property of the Littlewood-Paley projection $P_l$, we have:
\bee
&&\normm{P_j\left(\int_0^t(P_q(f)P_l\mu_1)d\tau\right)}_{\lh{2}}\\
\nn&\les&  2^{-2l}\normm{P_j\left(\int_0^t(P_q(f)\lap P_l\mu_1)d\tau\right)}_{\lh{2}}\\
\nn&\les& 2^{-2l}\normm{P_j\left(\int_0^t\lap(P_q(f)P_l\mu_1)d\tau\right)}_{\lh{2}}+2^{-2l}\normm{P_j\left(\int_0^t\divb(\nabb(P_q(f))P_l\mu_1)d\tau\right)}_{\lh{2}}\\
\nn&&+2^{-2l}\normm{P_j\left(\int_0^t(\lap(P_q(f)) P_l\mu_1)d\tau\right)}_{\lh{2}}\\
\eee
which together with Lemma \ref{lemma:lbt4}, Lemma \ref{lemma:lbt5}, the dual of the 
sharp Bernstein inequality \eqref{eq:strongbernscalarbis} for $P_j$, and 
the estimate for transport equations \eqref{estimtransport1} implies:
\bea
&&\lab{lbt17}\normm{P_j\left(\int_0^t(P_q(f)P_l\mu_1)d\tau\right)}_{\lh{2}}\\
\nn&\les& 2^{2j-2l}\norm{P_q(f)P_l\mu_1}_{\xt{2}{1}}+2^{\frac{3j}{2}-2l}\norm{\nabb(P_q(f))P_l\mu_1}_{\xt{\frac{4}{3}}{1}}+2^{j-2l}\norm{\lap(P_q(f)) P_l\mu_1}_{\lh{1}}\\
\nn&\les& 2^{2j-2l}\norm{P_q(f)}_{\tx{2}{4}}\norm{P_l\mu_1}_{\tx{2}{4}}+2^{\frac{3j}{2}-2l}\norm{\nabb(P_q(f))}_{\lh{2}}\norm{P_l\mu_1}_{\tx{2}{4}}\\
\nn&&+2^{j-2l}\norm{\lap(P_q(f))}_{\lh{2}}\norm{P_l\mu_1}_{\lh{2}}\\
\nn&\les& 2^{2j+\frac{q}{2}-\frac{3l}{2}}\norm{P_q(f)}_{\lh{2}}\norm{P_l\mu_1}_{\lh{2}}+2^{\frac{3j}{2}+q-\frac{3l}{2}}\norm{P_q(f)}_{\lh{2}}\norm{P_l\mu_1}_{\lh{2}}\\
\nn&&+2^{j+2q-2l}\norm{P_q(f)}_{\lh{2}}\norm{P_l\mu_1}_{\lh{2}}\\
\nn&\les& (2^{2j+\frac{q}{2}-\frac{3l}{2}}+2^{\frac{3j}{2}+q-\frac{3l}{2}}+2^{j+2q-2l})(D\ep 2^l+D\ep 2^{\frac{l}{2}}\gamma(u))\norm{P_q(f)}_{\lh{2}},
\eea
where we have used the weak Bernstein inequality for $P_q$ and $P_l$, and the 
estimate \eqref{lbt1} for $\mu_1$. 

Then, using \eqref{lbt16} for $q>l$ and \eqref{lbt17} 
for $q\leq l$, we obtain:
\bee
&&\sum_{q\geq 0}\normm{P_j\left(\int_0^t(P_q(f)P_l\mu_1)d\tau\right)}_{\lh{2}}\\
\nn&\les&  (2^{2j-2l}+2^{\frac{3j}{2}-\frac{5l}{2}})(D\ep 2^l+D\ep 2^{\frac{l}{2}}\gamma(u))\left(\sum_{q\geq 0}2^q\norm{P_q(f)}_{\lh{2}}\right)\\
\nn&&+2^j(D\ep +D\ep 2^{-\frac{l}{2}}\gamma(u))\left(\sum_{q\geq 0}2^{-|q-l|}2^q\norm{P_q(f)}_{\lh{2}}\right)\\
\nn&\les&  (2^{2j-2l}+2^{\frac{3j}{2}-\frac{5l}{2}})(D\ep 2^l+D\ep 2^{\frac{l}{2}}\gamma(u))\ep\\
\nn&&+2^j(D\ep +D\ep 2^{-\frac{l}{2}}\gamma(u))\left(\sum_{q\geq 0}2^{-|q-l|}2^q\norm{P_q(f)}_{\lh{2}}\right),
\eee
where we used the bound on $\norm{\nabb f}_{\BB^0}$ given by the assumptions of Lemma \ref{lemma:lbt1} in the last inequality. Together with \eqref{lbt15}, we obtain:
\bea
\lab{lbt18}\normm{P_j\left(\int_0^t(fP_l\mu_1)d\tau\right)}_{\lh{2}}&\les&  (2^{2j-2l}+2^{\frac{3j}{2}-\frac{5l}{2}})(D\ep 2^l+D\ep 2^{\frac{l}{2}}\gamma(u))\ep\\
\nn&&+2^j(D\ep +D\ep 2^{-\frac{l}{2}}\gamma(u))\left(\sum_{q\geq 0}2^{-|q-l|}2^q\norm{P_q(f)}_{\lh{2}}\right).
\eea
Finally, using \eqref{lbt13}, \eqref{lbt14} for $l\leq j$ and \eqref{lbt18} for $l>j$, we get:
\bea
\lab{lbt19}\normm{P_j\left(\int_0^t(f\mu_1)d\tau\right)}_{\lh{2}}&\les& (D\ep 2^j+D\ep 2^{\frac{j}{2}}\gamma(u))\ep\\
\nn&&+(D\ep 2^j +D\ep 2^{\frac{j}{2}}\gamma(u))\left(\sum_{l, q\geq 0}2^{-|q-l|}2^q\norm{P_q(f)}_{\lh{2}}\right)\\
\nn&\les& (D\ep 2^j+D\ep 2^{\frac{j}{2}}\gamma(u))(\ep+\norm{\nabb f}_{\BB^0})\\
\nn&\les& (D\ep^2 2^j+D\ep^2 2^{\frac{j}{2}}\gamma(u)),
\eea
where we used the bound on $\norm{\nabb f}_{\BB^0}$ given by the assumptions of Lemma \ref{lemma:lbt1} in the last inequality. This concludes the proof of Lemma \ref{lemma:lbt1}. 

\subsection{Proof of Lemma \ref{lemma:lbt2}}\lab{sec:lbt2}

We have:
\bea
&&\lab{lbt22}\normm{P_j\left(\int_0^t(F\c\nabb\ddb_{\lb}(\z))d\tau\right)}_{\lh{2}}\\
\nn&\les & \normm{P_j\left(\int_0^t\divb(F\c \ddb_{\lb}(\z))d\tau\right)}_{\lh{2}}+\normm{P_j\left(\int_0^t(\nabb F\c \ddb_{\lb}(\z))d\tau\right)}_{\lh{2}}\\
\nn &\les & \normm{P_j\left(\int_0^t\divb(F\c \ddb_{\lb}(\z))d\tau\right)}_{\lh{2}}+\normm{P_j\left(\int_0^t(\ddb_{nL}(P)\c \ddb_{\lb}(\z))d\tau\right)}_{\lh{2}}\\
\nn&&+\normm{P_j\left(\int_0^t(E\c \ddb_{\lb}(\z))d\tau\right)}_{\lh{2}}\\
\nn &\les & \normm{P_j\left(\int_0^t\divb(F\c \ddb_{\lb}(\z))d\tau\right)}_{\lh{2}}+\norm{P_j(P\c \ddb_{\lb}(\z))}_{\lh{2}}\\
\nn&& +\normm{P_j\left(\int_0^t(P\c \ddb_{nL}\ddb_{\lb}(\z))d\tau\right)}_{\lh{2}}
+\normm{P_j\left(\int_0^t(E\c \ddb_{\lb}(\z))d\tau\right)}_{\lh{2}},
\eea
where we used the assumption of Lemma \ref{lemma:lbt2} $\nabb F=\ddb_{nL}(P)+E$, and then  where we integrated by part in $t$. Since $\norm{E}_{\PP^0}\les\ep$ and $\ddb_{\lb}(\z)$ satisfies \eqref{lbz10}, the fourth term in the right-hand side of \eqref{lbt22} is estimated using Lemma \ref{lemma:lbt3}:
\be\lab{lbt20bis}
\normm{P_j\left(\int_0^t(E\c \ddb_{\lb}(\z))d\tau\right)}_{\lh{2}}\les D\ep^22^j+D\ep^2 2^{\frac{j}{2}}\gamma(u).
\ee
Next, we estimate the first, the second and the third term in the right-hand side of \eqref{lbt22}.

\subsubsection{Estimate of the first term in the right-hand side of \eqref{lbt22}}

We decompose \\$\norm{P_j\left(\int_0^t\divb(F\c\ddb_{\lb}(\z))d\tau\right)}_{\lh{2}}$ using the property \eqref{eq:partition} of the geometric Littlewood-Paley projections:
\bea
\lab{lbt20}\normm{P_j\left(\int_0^t\divb(F\c\ddb_{\lb}(\z))d\tau\right)}_{\lh{2}}&\les& \normm{P_j\left(\int_0^t\divb(F\c P_{<0}\ddb_{\lb}(\z))d\tau\right)}_{\lh{2}}\\
\nn&&+\sum_{l\geq 0}\normm{P_j\left(\int_0^t\divb(F\c P_l\ddb_{\lb}(\z))d\tau\right)}_{\lh{2}}.
\eea
We focus on the second term in the right-hand side of \eqref{lbt20}, the other being easier to handle.  Using the $L^2$ boundedness of the Littlewood-Paley projection $P_j$, the weak Bernstein inequality for $P_j$, and the estimate for transport equations \eqref{estimtransport1}, we have:
\bea
&&\lab{lbt21}\normm{P_j\left(\int_0^t\divb(F\c P_l\ddb_{\lb}(\z))d\tau\right)}_{\lh{2}}\\
\nn&\les&\normm{P_j\left(\int_0^t(\nabb F\c P_l\ddb_{\lb}(\z))d\tau\right)}_{\lh{2}}+\normm{P_j\left(\int_0^t(F\c \nabb P_l\ddb_{\lb}(\z))d\tau\right)}_{\lh{2}}\\
\nn&\les &2^{\frac{j}{3}}\normm{\int_0^t(\nabb F\c P_l\ddb_{\lb}(\z))d\tau}_{\tx{2}{\frac{3}{2}}}+\normm{\int_0^t(F\c\nabb P_l\ddb_{\lb}(\z))d\tau}_{\lh{2}}\\
\nn&\les& 2^{\frac{j}{3}}\norm{\nabb F\c P_l\ddb_{\lb}(\z)}_{\xt{\frac{3}{2}}{1}}+\norm{F\c\nabb P_l\ddb_{\lb}(\z)}_{\xt{2}{1}}\\
\nn&\les& 2^{\frac{j}{3}}\norm{\nabb F}_{\lh{2}}\norm{P_l\ddb_{\lb}(\z)}_{\tx{2}{6}}+\norm{F}_{\xt{\infty}{2}}\norm{\nabb P_l\ddb_{\lb}(\z)}_{\lh{2}}\\
\nn&\les& 2^{\frac{j}{3}+\frac{2l}{3}}(D\ep^2+D\ep^22^{-\frac{l}{2}}\gamma(u))+D\ep^22^l+D\ep^22^{\frac{l}{2}}\gamma(u),
\eea
where we used the finite band property for $P_l$, the weak Bernstein inequality for $P_l$, the estimate \eqref{lbz10} for $\ddb_{\lb}(\z)$ and the assumption of Lemma \ref{lemma:lbt2} for $F$.

We will need another estimate for $\norm{P_j\left(\int_0^t\divb(F\c P_l\ddb_{\lb}(\z))d\tau\right)}_{\lh{2}}$. 
We decompose $\norm{P_j\left(\int_0^t\divb(F\c P_l\ddb_{\lb}(\z))d\tau\right)}_{\lh{2}}$ using the property \eqref{eq:partition} of the geometric Littlewood-Paley projections:
\bea
\lab{lbt23}\normm{P_j\left(\int_0^t\divb(F\c P_l\ddb_{\lb}(\z))d\tau\right)}_{\lh{2}}&\les & \normm{P_j\left(\int_0^t\divb(P_{\leq l}(F)\c P_l\ddb_{\lb}(\z))d\tau\right)}_{\lh{2}}\\
\nn&&+\sum_{q> l}\normm{P_j\left(\int_0^t\divb(P_q(F)\c P_l\ddb_{\lb}(\z))d\tau\right)}_{\lh{2}}.
\eea
We first estimate the second term in the right-hand side of \eqref{lbt23}. Using Lemma \ref{lemma:lbt5} with $p=\frac{4}{3}$, we have:
\bee
\normm{P_j\left(\int_0^t\divb(P_q(F)\c P_l\ddb_{\lb}(\z))d\tau\right)}_{\lh{2}}&\les& 2^{\frac{3j}{2}}\norm{P_q(F)\c P_l\ddb_{\lb}(\z)}_{\tx{\frac{4}{3}}{1}}\\
\nn&\les& 2^{\frac{3j}{2}}\norm{P_q(F)}_{\lh{2}}\norm{P_l\ddb_{\lb}(\z)}_{\tx{2}{4}}\\
\nn&\les& 2^{\frac{3j}{2}}2^{-q}\norm{\nabb F}_{\lh{2}}2^{\frac{l}{2}}\norm{P_l\ddb_{\lb}(\z)}_{\lh{2}}\\
\nn&\les& 2^{\frac{3j}{2}-q+\frac{l}{2}}\ep(D\ep+D\ep 2^{-\frac{l}{2}}\gamma(u)),
\eee  
where we used the finite band property for $P_q$, the weak Bernstein inequality for $P_l$, the assumption on $F$ and the estimate \eqref{lbz10} for $\ddb_{\lb}(\z)$. This yields the following estimate 
for the the second term in the right-hand side of \eqref{lbt23}:
\be\lab{lbt24}
\sum_{q>l}\normm{P_j\left(\int_0^t\divb(P_q(F)\c P_l\ddb_{\lb}(\z))d\tau\right)}_{\lh{2}}\les  2^{\frac{3j}{2}-\frac{l}{2}}\ep(D\ep+D\ep 2^{-\frac{l}{2}}\gamma(u)).
\ee  

We now estimate the first term in the right-hand side of \eqref{lbt23}. Using the property of the Littlewood-Paley projection $P_l$, we have:
\bee
&&\normm{P_j\left(\int_0^t\divb(P_{\leq l}(F)\c P_l\ddb_{\lb}(\z))d\tau\right)}_{\lh{2}}\\
&\les& 2^{-2l}\normm{P_j\left(\int_0^t\divb(P_{\leq l}(F)\c\lap P_l\ddb_{\lb}(\z))d\tau\right)}_{\lh{2}}\\
&\les& 2^{-2l}\normm{P_j\left(\int_0^t\divb\divb(P_{\leq l}(F)\c\nabb P_l\ddb_{\lb}(\z))d\tau\right)}_{\lh{2}}\\
&&+2^{-2l}\normm{P_j\left(\int_0^t\divb(\nabb P_{\leq l}(F)\c\nabb P_l\ddb_{\lb}(\z))d\tau\right)}_{\lh{2}}
\eee  
which together with Lemma \ref{lemma:lbt4} and Lemma \ref{lemma:lbt5} with $p=\frac{4}{3}$ yields:
\bea
\lab{lbt25}&&\normm{P_j\left(\int_0^t\divb(P_{\leq l}(F)\c P_l\ddb_{\lb}(\z))d\tau\right)}_{\lh{2}}\\
\nn&\les& 2^{2j-2l}\norm{P_{\leq l}(F)\c\nabb P_l\ddb_{\lb}(\z)}_{\xt{2}{1}}+2^{\frac{3j}{2}-2l}\norm{\nabb P_{\leq l}(F)\c\nabb P_l\ddb_{\lb}(\z)}_{\xt{\frac{4}{3}}{1}}\\
\nn&\les& 2^{2j-2l}\norm{P_{\leq l}(F)\c\nabb P_l\ddb_{\lb}(\z)}_{\xt{2}{1}}+2^{\frac{3j}{2}-2l}\norm{\nabb P_{\leq l}(F)\c\nabb P_l\ddb_{\lb}(\z)}_{\xt{\frac{4}{3}}{1}}.
\eea  
Using the fact that $P_{\leq l}(F)=F-P_{>l}(F)$, we estimate the first term in the right-hand side of \eqref{lbt25} as follows:
\bea\lab{lbt26}
&&\norm{P_{\leq l}(F)\c\nabb P_l\ddb_{\lb}(\z)}_{\xt{2}{1}}\\
\nn&\les& \norm{F\c\nabb P_l\ddb_{\lb}(\z)}_{\xt{2}{1}}+\sum_{q>l}\norm{P_q(F)\c\nabb P_l\ddb_{\lb}(\z)}_{\xt{2}{1}}\\
\nn&\les &\norm{F}_{\xt{\infty}{2}}\norm{\nabb P_l\ddb_{\lb}(\z)}_{\lh{2}}+\sum_{q>l}\normm{\norm{P_q(F)}_{\lpt{4}}\norm{\nabb P_l\ddb_{\lb}(\z)}_{\lpt{4}}}_{L^1_t}\\
\nn&\les &\ep2^l(D\ep+2^{\frac{l}{2}}D\ep\gamma(u))+\sum_{q>l}\normm{\norm{P_q(F)}_{\lpt{4}}\norm{\nabb P_l\ddb_{\lb}(\z)}_{\lpt{4}}}_{L^1_t},
\eea
where we used the finite band property for $P_l$, the assumption on $F$ and the estimate \eqref{lbz10} for $\ddb_{\lb}(\z)$. We consider the second term in the right-hand side of \eqref{lbt26}. The Gagliardo-Nirenberg inequality \eqref{eq:GNirenberg}, the Bochner inequality for tensors \eqref{vbochineq}, and the weak Bernstein inequality for $P_l$ yield:
\bea\lab{lbt27}
&&\norm{\nabb P_l\ddb_{\lb}(\z)}_{\lpt{4}}\\
\nn&\les & \norm{\nabb^2P_l\ddb_{\lb}(\z)}_{\lpt{2}}^{\frac{1}{2}}\norm{\nabb P_l\ddb_{\lb}(\z)}_{\lpt{2}}^{\frac{1}{2}}\\
\nn&\les& (\norm{\lap P_l\ddb_{\lb}(\z)}_{\lpt{2}}+\norm{K}_{\lpt{2}}\norm{\nabb P_l\ddb_{\lb}(\z)}_{\lpt{2}}\\
\nn&&+\norm{K}^{\frac{3}{2}}_{\lpt{2}}\norm{\nabb P_l\ddb_{\lb}(\z)}^{\half}_{\lpt{2}}\norm{P_l\ddb_{\lb}(\z)}^{\half}_{\lpt{2}})^{\frac{1}{2}}2^{\frac{l}{2}}\norm{P_l\ddb_{\lb}(\z)}_{\lpt{2}}^{\frac{1}{2}}\\
\nn&\les& (2^{2l}+2^{\frac{l}{2}}\norm{K}^{\frac{3}{2}}_{\lpt{2}})^{\frac{1}{2}}2^{\frac{l}{2}}\norm{P_l\ddb_{\lb}(\z)}_{\lpt{2}}.
\eea
Now, \eqref{lbt27}, the weak Bernstein inequality for $P_q$, the finite band property for $P_q$, and Lemma \ref{lemma:lbz4} imply:
\bea
&&\lab{lbt28}\normm{\norm{P_q(F)}_{\lpt{4}}\norm{\nabb P_l\ddb_{\lb}(\z)}_{\lpt{4}}}_{L^1_t}\\
\nn&\les &\norm{P_q(F)}_{\tx{2}{4}}2^{\frac{3l}{2}}\norm{P_l\ddb_{\lb}(\z)}_{\lh{2}}+\norm{P_q(F)}_{\tx{8}{4}}2^{\frac{3l}{4}}\norm{K}^{\frac{3}{4}}_{\lh{2}}\norm{P_l\ddb_{\lb}(\z)}_{\lh{2}}\\
\nn&\les & 2^{-\frac{q}{2}+\frac{3l}{2}}\norm{\nabb F}_{\lh{2}}(D\ep+D\ep 2^{-\frac{l}{2}}\gamma(u))+2^{-\frac{q}{8}+\frac{3l}{4}}\no(F)\ep^{\frac{3}{4}}(D\ep+D\ep 2^{-\frac{l}{2}}\gamma(u))\\
\nn&\les&(2^{-\frac{q}{2}+\frac{3l}{2}}+2^{\frac{3l}{4}-\frac{q}{8}})(D\ep^2+D\ep^2 2^{-\frac{l}{2}}\gamma(u))
\eea
where we used the bound \eqref{estgauss1} for $K$, the assumptions on $F$ and the estimate \eqref{lbz10} for $\ddb_{\lb}(\z)$. \eqref{lbt26} and \eqref{lbt28} yield:
\bea
\nn\norm{P_{\leq l}(F)\c\nabb P_l\ddb_{\lb}(\z)}_{\xt{2}{1}}&\les &\left(2^l+\sum_{q>l}(2^{-\frac{q}{2}+\frac{3l}{2}}+2^{\frac{3l}{4}-\frac{q}{8}})\right)(D\ep^2+D\ep^2 2^{-\frac{l}{2}}\gamma(u))\\
\lab{lbt29}&\les & D\ep^22^l+D\ep^2 2^{\frac{l}{2}}\gamma(u).
\eea

Next, we estimate the second term in the right-hand side of \eqref{lbt25}:
\bea
\lab{lbt30}
\norm{\nabb P_{\leq l}(F)\c\nabb P_l\ddb_{\lb}(\z)}_{\xt{\frac{4}{3}}{1}}&\les&\norm{\nabb P_{\leq l}(F)}_{\tx{2}{4}}\norm{\nabb P_l\ddb_{\lb}(\z)}_{\lh{2}}\\
\nn &\les&\left(\sum_{q\leq l}\norm{\nabb P_{q}(F)}_{\tx{2}{4}}\right)2^l(D\ep+D\ep 2^{-\frac{l}{2}}\gamma(u)),
\eea  
where we used the finite band property of $P_l$ and the  estimate \eqref{lbz10} for $\ddb_{\lb}(\z)$ in the last inequality. We estimate $\norm{\nabb P_{q}(F)}_{\lpt{4}}$ using the Gagliardo-Nirenberg inequality \eqref{eq:GNirenberg}, the Bochner inequality for tensors \eqref{vbochineq}, and the weak Bernstein inequality for $P_q$:
\bee
\norm{\nabb P_{q}(F)}_{\lpt{4}}&\les& \norm{\nabb^2P_{q}(F)}^{\half}_{\lpt{2}}\norm{\nabb P_{q}(F)}^{\half}_{\lpt{2}}\\
\nn&\les& (\norm{\lap P_{q}(F)}_{\lpt{2}}+\norm{K}_{\lpt{2}}\norm{\nabb P_{q}(F)}_{\lpt{2}}\\
\nn&&+\norm{K}^2_{\lpt{2}}\norm{P_{q}(F)}_{\lpt{2}})^{\half}2^{\frac{q}{2}}\norm{P_{q}(F)}^{\half}_{\lpt{2}}\\
\nn&\les& (2^{2q}+\norm{K}^2_{\lpt{2}})^{\half}2^{\frac{q}{2}}\norm{P_{q}(F)}_{\lpt{2}}
\eee  
which together with the finite band property for $P_q$, and Lemma \ref{lemma:lbz4} implies:
\bea
\lab{lbt31}
\norm{\nabb P_{q}(F)}_{\tx{2}{4}}&\les& 2^{\frac{3q}{2}}\norm{P_{q}(F)}_{\lh{2}}+\norm{K}_{\lpt{2}}2^{\frac{q}{2}}\norm{P_{q}(F)}_{\tx{\infty}{2}}\\
\nn&\les& 2^{\frac{q}{2}}\norm{\nabb F}_{\lh{2}}+\ep \no(F)\\
\nn&\les& 2^{\frac{q}{2}}\ep
\eea 
where we used the bound \eqref{estgauss1} for $K$ and the assumptions on $F$. \eqref{lbt30} and \eqref{lbt31} yield:
\bea
\lab{lbt32}
\norm{\nabb P_{\leq l}(F)\c\nabb P_l\ddb_{\lb}(\z)}_{\xt{\frac{4}{3}}{1}}&\les&\left(\sum_{q\leq l}2^{\frac{q}{2}}\ep\right)2^l(D\ep+D\ep 2^{-\frac{l}{2}}\gamma(u))\\
\nn&\les& 2^{\frac{3l}{2}}(D\ep^2+D\ep^2 2^{-\frac{l}{2}}\gamma(u)).
\eea  

Finally, \eqref{lbt25}, \eqref{lbt29} and \eqref{lbt32} imply:
\be\lab{lbt33}
\normm{P_j\left(\int_0^t\divb(P_{\leq l}(F)\c P_l\ddb_{\lb}(\z))d\tau\right)}_{\lh{2}}
\les (2^{2j-l}+2^{\frac{3j}{2}-\frac{l}{2}})(D\ep^2+D\ep^2 2^{-\frac{l}{2}}\gamma(u)).
\ee  
Now, \eqref{lbt23}, \eqref{lbt24} and \eqref{lbt33} yield:
\be\lab{lbt34}
\normm{P_j\left(\int_0^t\divb(F\c P_l\ddb_{\lb}(\z))d\tau\right)}_{\lh{2}}\les  (2^{2j-l}+2^{\frac{3j}{2}-\frac{l}{2}})(D\ep^2+D\ep^2 2^{-\frac{l}{2}}\gamma(u)).
\ee
Using \eqref{lbt20}, \eqref{lbt21} for $l\leq j$ and \eqref{lbt34} for $l>j$, we obtain:
\be\lab{lbt35}
\normm{P_j\left(\int_0^t\divb(F\c \ddb_{\lb}(\z))d\tau\right)}_{\lh{2}}\les  D\ep^22^j+D\ep^2 2^{\frac{j}{2}}\gamma(u),
\ee
which is the desired estimate of the first term in the right-hand side of \eqref{lbt22}.

\subsubsection{Estimate of the second term in the right-hand side of \eqref{lbt22}}

We decompose $\norm{P_j(P\c \ddb_{\lb}(\z))}_{\lh{2}}$ using the property \eqref{eq:partition} of the geometric Littlewood-Paley projections:
\be\lab{lbt36}
\norm{P_j(P\c \ddb_{\lb}(\z))}_{\lh{2}}\les\norm{P_j(P\c P_{<0}\ddb_{\lb}(\z))}_{\lh{2}}+\sum_{l\geq 0}\norm{P_j(P\c P_l\ddb_{\lb}(\z))}_{\lh{2}}.
\ee
We focus on the second term in the right-hand side of \eqref{lbt36}, the other being easier to handle.  Using the weak Bernstein inequality for $P_j$, we have:
\bea
\lab{lbt37}\norm{P_j(P\c P_l\ddb_{\lb}(\z))}_{\lh{2}}&\les & 2^{\frac{j}{2}}\norm{P\c P_l\ddb_{\lb}(\z)}_{\tx{2}{\frac{4}{3}}}\\
\nn&\les & 2^{\frac{j}{2}}\norm{P}_{\tx{\infty}{4}}\norm{P_l\ddb_{\lb}(\z)}_{\lh{2}}\\
\nn&\les & 2^{\frac{j}{2}}\no(P)(D\ep +D\ep 2^{-\frac{l}{2}}\gamma(u))\\
\nn&\les & 2^{\frac{j}{2}}(D\ep^2 +D\ep^2 2^{-\frac{l}{2}}\gamma(u)),
\eea
where we used the assumption of Lemma \ref{lemma:lbt2} for $P$ and the estimate \eqref{lbz10} for $\ddb_{\lb}(\z)$.

We will need another estimate for $\norm{P_j(P\c P_l\ddb_{\lb}(\z))}_{\lh{2}}$. We decompose $\norm{P_j(P\c P_l\ddb_{\lb}(\z))}_{\lh{2}}$ using the property \eqref{eq:partition} of the geometric Littlewood-Paley projections:
\be\lab{lbt38}
\norm{P_j(P\c P_l\ddb_{\lb}(\z))}_{\lh{2}}\les  \norm{P_j(P_{\leq l}P\c P_l\ddb_{\lb}(\z))}_{\lh{2}}+\sum_{q> l}\norm{P_j(P_q(P)\c P_l\ddb_{\lb}(\z))}_{\lh{2}}.
\ee
We first estimate the second term in the right-hand side of \eqref{lbt38}. Using the dual of the sharp Bernstein inequality \eqref{eq:strongbernscalarbis}, we have:
\bee
\norm{P_j(P_q(P)\c P_l\ddb_{\lb}(\z))}_{\lh{2}}&\les& 2^j \norm{P_q(P)\c P_l\ddb_{\lb}(\z)}_{\tx{2}{1}}\\
\nn &\les & 2^j \norm{P_q(P)}_{\tx{\infty}{2}}\norm{P_l\ddb_{\lb}(\z)}_{\lh{2}}\\
\nn &\les & 2^{j-\frac{q}{2}}\no(P)(D\ep+D\ep 2^{-\frac{l}{2}}\gamma(u))\\
\nn &\les & 2^{j-\frac{q}{2}}(D\ep^2+D\ep^2 2^{-\frac{l}{2}}\gamma(u)),
\eee
where we used Lemma \ref{lemma:lbz4}, the assumption of Lemma \ref{lemma:lbt2} on $P$ and the estimate \eqref{lbz10} for $\ddb_{\lb}(\z)$. This yields the following estimate for the second term in the right-hand side of \eqref{lbt38}:
\be\lab{lbt39}
\sum_{q>l}\norm{P_j(P_q(P)\c P_l\ddb_{\lb}(\z))}_{\lh{2}}\les 2^{j-\frac{l}{2}}(D\ep^2+D\ep^2 2^{-\frac{l}{2}}\gamma(u)).
\ee

We now estimate the first term in the right-hand side of \eqref{lbt38}. Using the property of the Littlewood-Paley projection $P_l$, we have:
\bee
&&\norm{P_j(P_{\leq l}P\c P_l\ddb_{\lb}(\z))}_{\lh{2}}\\
\nn&\les & 2^{-2l}\norm{P_j(P_{\leq l}P\c \lap P_l\ddb_{\lb}(\z))}_{\lh{2}}\\
\nn&\les & 2^{-2l}\norm{P_j\divb(P_{\leq l}P\c \nabb P_l\ddb_{\lb}(\z))}_{\lh{2}}+2^{-2l}\norm{P_j(\nabb P_{\leq l}P\c \nabb P_l\ddb_{\lb}(\z))}_{\lh{2}}\\
\eee
which together with the property \eqref{chat2} of $P_j$ with $p=\frac{4}{3}$ and the dual of the sharp Bernstein inequality \eqref{eq:strongbernscalarbis} yields: 
\bea
&&\lab{lbt40}\norm{P_j(P_{\leq l}P\c P_l\ddb_{\lb}(\z))}_{\lh{2}}\\
\nn&\les & 2^{\frac{3j}{2}-2l}\norm{P_{\leq l}P\c \nabb P_l\ddb_{\lb}(\z)}_{\tx{2}{\frac{4}{3}}}+2^{j-2l}\norm{\nabb P_{\leq l}P\c \nabb P_l\ddb_{\lb}(\z)}_{\tx{2}{1}}\\
\nn&\les & 2^{\frac{3j}{2}-2l}\norm{P_{\leq l}P}_{\tx{\infty}{4}}\norm{\nabb P_l\ddb_{\lb}(\z)}_{\lh{2}}+2^{j-2l}\norm{\nabb P_{\leq l}P}_{\tx{\infty}{2}}\norm{\nabb P_l\ddb_{\lb}(\z)}_{\lh{2}}\\
\nn&\les & 2^{\frac{3j}{2}-2l}\norm{P}_{\tx{\infty}{4}}2^l\norm{P_l\ddb_{\lb}(\z)}_{\lh{2}}+2^{j-2l}\left(\sum_{q\leq l}2^q\norm{P_{q}P}_{\tx{\infty}{2}}\right)2^l\norm{P_l\ddb_{\lb}(\z)}_{\lh{2}}\\
\nn&\les & 2^{\frac{3j}{2}-l}\no(P)(D\ep+D\ep 2^{-\frac{l}{2}}\gamma(u))+2^{j-l}\left(\sum_{q\leq l}2^{\frac{q}{2}}\no(P)\right)(D\ep+D\ep 2^{-\frac{l}{2}}\gamma(u))\\
\nn&\les & (2^{\frac{3j}{2}-l}+2^{j-\frac{l}{2}})(D\ep^2+D\ep^2 2^{-\frac{l}{2}}\gamma(u))
\eea
where we used the finite band property of $P_l$ and $P_q$, the embedding \eqref{sobineq1}, Lemma \ref{lemma:lbz4}, the estimate \eqref{lbz10} for $\ddb_{\lb}(\z)$, and the assumption of Lemma \ref{lemma:lbt2} for $P$.

Finally, \eqref{lbt38}, \eqref{lbt39} and \eqref{lbt40} imply:
\be\lab{lbt41}
\norm{P_j(P\c P_l\ddb_{\lb}(\z))}_{\lh{2}}\les  (2^{\frac{3j}{2}-l}+2^{j-\frac{l}{2}})(D\ep^2+D\ep^2 2^{-\frac{l}{2}}\gamma(u)).
\ee
Using \eqref{lbt36}, \eqref{lbt37} for $l\geq j$, and \eqref{lbt41} for $l>j$, we obtain:
\be\lab{lbt42}
\norm{P_j(P\c \ddb_{\lb}(\z))}_{\lh{2}}\les D\ep^22^{\frac{j}{2}} +D\ep^2 2^{\frac{j}{2}}\gamma(u)),
\ee
which is the desired estimate of the second term in the right-hand side of \eqref{lbt22}.

\subsubsection{Estimate of the third term in the right-hand side of \eqref{lbt22}}

We start by deriving an equation for $\ddb_{nL}\ddb_{\lb}(\z)$. Differentiating the transport equation \eqref{D4eta} satisfied by $\z$ with respect to $\lb$, we obtain:
$$\ddb_{\lb}\ddb_{L}\z_A = -(\kepb_B+\z_B)\ddb_{\lb}(\chi)_{AB}-(\ddb_{\lb}(\kepb)_B+\ddb_{\lb}(\z)_B)\chi_{AB} - \ddb_{\lb}(\b)_A$$
which together with the commutator formula \eqref{comm3} and the Bianchi identity \eqref{bianc1bis}  yields after multiplication by $n$:
\be\lab{lbt43}
\ddb_{nL}\ddb_{\lb}\z=n(\db-\chi\c) \ddb_{\lb}(\z)+B-\nabb(n\r)-(\nabb(n\s))^*
\ee
where the 1-form $B$ is given by:
\bee
B&=&-n(\d+n^{-1}\nabla_Nn)\ddb_L(\z)-2n(\z-\zb)\c\nabb\z+2n(\zb\wedge\z+\in{}^*\s)\c\z\\
\nn&&-n(\z+\kepb)\c\ddb_{\lb}(\chi)-n\ddb_{\lb}(\kepb)\c\chi-2n\hch\c\bb-2n(\d+n^{-1}\nabla_Nn)\b-n\xib\c\a\\
&&-3n(\z\r+{}^*\z\s)+\nabb(n)\r+\nabb(n)\s.
\eee
We estimate the $\tx{2}{\frac{4}{3}}$ norm of $B$. We have:
\bea
\lab{lbt44}\norm{B}_{\tx{2}{\frac{4}{3}}}&\les&\norm{n}_{L^\infty}\bigg(\norm{\d+n^{-1}\nabla_Nn}_{\tx{\infty}{4}}\norm{\ddb_L(\z)}_{\lh{2}}+\norm{\z-\zb}_{\tx{\infty}{4}}\norm{\nabb\z}_{\lh{2}}\\
\nn&&+\norm{\zb\wedge\z+\in{}^*\s}_{\lh{2}}\norm{\z}_{\tx{\infty}{4}}+\norm{\z+\kepb}_{\tx{\infty}{4}}\norm{\ddb_{\lb}(\chi)}_{\lh{2}}\\
\nn&&+\norm{\ddb_{\lb}(\kepb)}_{\lh{2}}\norm{\chi}_{\tx{\infty}{4}}+\norm{\hch}_{\tx{\infty}{4}}\norm{\bb}_{\lh{2}}\\
\nn&&+\norm{\d+n^{-1}\nabla_Nn}_{\tx{\infty}{4}}\norm{\b}_{\lh{2}}+\norm{\xib}_{\tx{\infty}{4}}\norm{\a}_{\lh{2}}\\
\nn&&+\norm{\z}_{\tx{\infty}{4}}(\norm{\r}_{\lh{2}}+\norm{\s}_{\lh{2}})+\norm{n^{-1}\nabb n}_{\tx{\infty}{4}}(\norm{\r}_{\lh{2}}+\norm{\s}_{\lh{2}})\bigg)\\
\nn&\les& \norm{n}_{L^\infty}\bigg(\no(\z)^2+\no(\d)^2+\no(\nabla n)^2+\no(\zb)^2+\no(\kepb)^2+\no(\chi)^2+\no(\xib)^2\\
\nn&&+\norm{\ddb_{\lb}(\chi)}^2_{\lh{2}}+\norm{\ddb_{\lb}(\kepb)}^2_{\lh{2}}+\norm{\s}^2_{\lh{2}}+\norm{\r}^2_{\lh{2}}+\norm{\b}^2_{\lh{2}}\\
\nn&&+\norm{\a}_{\lh{2}}^2+\norm{\bb}^2_{\lh{2}}\bigg)\\
\nn&\les& \ep^2,
\eea
where we used the curvature bound \eqref{curvflux1} for $\a, \b, \r, \s$ and $\bb$, and the estimates \eqref{estn}-\eqref{estzeta} for $n, \d, \kepb, \zb, \chi, \xib$ and $\z$.

We have the following estimate for the third term in the right-hand side of \eqref{lbt22}:
\bea
\lab{lbt45}&&\normm{P_j\left(\int_0^t(P\c\ddb_{nL}\ddb_{\lb}(\z))d\tau\right)}_{\lh{2}}\\
\nn&\les& \normm{P_j\left(\int_0^t(nP\c (\db-\chi\c)\ddb_{\lb}(\z))d\tau\right)}_{\lh{2}}+\normm{P_j\left(\int_0^t(P\c B)d\tau\right)}_{\lh{2}}\\
\nn&&+\normm{P_j\left(\int_0^t(P\c(\nabb(n\r)+(\nabb(n\s))^*)d\tau\right)}_{\lh{2}}
\eea
We estimate the three terms in the right-hand side of \eqref{lbt45} starting with the first one. The non sharp product estimates \eqref{nonsharpprod1} and \eqref{nonsharpprod2} imply:
$$\norm{nP(\db-\c\chi)}_{\PP^0}\les \noo(n)\norm{P(\db-\c\chi)}_{\PP^0}\les \noo(n)\no(P)(\no(\db)+\no(\chi))\les \ep^2$$
which together with Lemma \ref{lemma:lbt3} yields the following estimate for the first term in the right-hand side of \eqref{lbt45}:
\be\lab{lbt46}
\normm{P_j\left(\int_0^t(nP\c (\db-\chi\c)\ddb_{\lb}(\z))d\tau\right)}_{\lh{2}}\les D\ep^22^j+D\ep^22^{\frac{j}{2}}\gamma(u).
\ee
To estimate the second term in the right-hand side of \eqref{lbt45}, we use the dual of the sharp Bernstein inequality \eqref{eq:strongbernscalarbis} and the estimate for transport equations \eqref{estimtransport1}. We have:
\bea
\lab{lbt47}\normm{P_j\left(\int_0^t(P\c B)d\tau\right)}_{\lh{2}}&\les& 2^j\normm{\int_0^t(P\c B)d\tau}_{\tx{2}{1}}\\
\nn&\les&2^j\norm{P\c B}_{\lh{1}}\\
\nn&\les& 2^j\norm{P}_{\tx{2}{4}}\norm{B}_{\tx{2}{\frac{4}{3}}}\\
\nn&\les& 2^j\no(P)\ep^2\\
\nn&\les & 2^j\ep,
\eea
where we used the assumption on $P$ in Lemma \ref{lemma:lbt2}, and the estimate \eqref{lbt44}.

We now focus on estimating the third term in the right-hand side of \eqref{lbt45}. Using the decomposition of $\nabb(n\r)+(\nabb(n\s))^*$ given by Lemma \ref{lemma:lbt7}, we estimate the third term in the right-hand side of \eqref{lbt45} as follows:
\bea
&&\lab{lbt50}\normm{P_j\left(\int_0^t(P\c(\nabb(n\r)+(\nabb(n\s))^*)d\tau\right)}_{\lh{2}}\\
\nn&\les&\normm{P_j\left(\int_0^t(P\c {}^*\mathcal{D}_1\c J\c {}^*\mathcal{D}_1^{-1}(\ddb_{nL}(\bb))d\tau\right)}_{\lh{2}}+\normm{P_j\left(\int_0^t(P\c {}^*\mathcal{D}_1(H))d\tau\right)}_{\lh{2}}.
\eea
Next, we estimate both terms in the right-hand side of \eqref{lbt50} starting with the second one. We  have:
\bee
\normm{P_j\left(\int_0^t(P\c {}^*\mathcal{D}_1(H))d\tau\right)}_{\lh{2}}&\les&
\normm{P_j\left(\int_0^t\divb(P\c H)d\tau\right)}_{\lh{2}}\\
&&+\normm{P_j\left(\int_0^t(\nabb P\c H)d\tau\right)}_{\lh{2}}
\eee
which together with the finite band property for $P_j$, the sharp Bernstein inequality \eqref{eq:strongbernscalarbis}, and the estimate for transport equations \eqref{estimtransport1} yields:
\bea
\nn\normm{P_j\left(\int_0^t(P\c {}^*\mathcal{D}_1(H))d\tau\right)}_{\lh{2}}&\les& 2^j\normm{\int_0^t(P\c H)d\tau}_{\lh{2}}+2^j\normm{\int_0^t(\nabb P\c H)d\tau}_{\tx{2}{1}}\\
\nn&\les& 2^j\norm{P\c H}_{\tx{1}{2}}+2^j\norm{\nabb P\c H}_{\lh{1}}\\
\nn&\les& 2^j\norm{P}_{\tx{2}{6}}\norm{H}_{\tx{2}{3}}+2^j\norm{\nabb P}_{\lh{2}}\norm{H}_{\lh{2}}\\
\lab{lbt51}&\les & 2^j\ep,
\eea
where we used the estimate \eqref{lbt49} for $H$, and the assumption of Lemma \ref{lemma:lbt2} on $P$.

We turn to the first term in the right-hand side of \eqref{lbt50}. We have:
\bea
&&\lab{lbt52}\normm{P_j\left(\int_0^t(P\c {}^*\mathcal{D}_1\c J\c {}^*\mathcal{D}_1^{-1}(\ddb_{nL}(\bb))d\tau\right)}_{\lh{2}}\\
\nn&\les& \normm{P_j\left(\int_0^t(P\c \ddb_{nL}{}^*\mathcal{D}_1\c J\c {}^*\mathcal{D}_1^{-1}(\bb))d\tau\right)}_{\lh{2}}\\
\nn&&+\normm{P_j\left(\int_0^t(P\c [{}^*\mathcal{D}_1,\ddb_{nL}]\c J\c {}^*\mathcal{D}_1^{-1}(\bb))d\tau\right)}_{\lh{2}}\\
\nn&&+\normm{P_j\left(\int_0^t(P\c {}^*\mathcal{D}_1\c J\c [{}^*\mathcal{D}_1^{-1},\ddb_{nL}](\bb))d\tau\right)}_{\lh{2}}\\
\nn&\les& \norm{P_j(P\c {}^*\mathcal{D}_1\c J\c {}^*\mathcal{D}_1^{-1}(\bb))}_{\lh{2}}+\normm{P_j\left(\int_0^t(\ddb_{nL}(P)\c {}^*\mathcal{D}_1\c J\c {}^*\mathcal{D}_1^{-1}(\bb))d\tau\right)}_{\lh{2}}\\
\nn&&+\normm{P_j\left(\int_0^t(P\c [{}^*\mathcal{D}_1,\ddb_{nL}]\c J\c {}^*\mathcal{D}_1^{-1}(\bb))d\tau\right)}_{\lh{2}}\\
\nn&&+\normm{P_j\left(\int_0^t(P\c {}^*\mathcal{D}_1\c J\c [{}^*\mathcal{D}_1^{-1},\ddb_{nL}](\bb))d\tau\right)}_{\lh{2}}.
\eea
Next, we estimate the four terms in the right-hand side of \eqref{lbt52} starting with the first one. 

Using the dual of the sharp Bernstein inequality \eqref{eq:strongbernscalarbis}, we have:
\bea
\lab{lbt53}\norm{P_j(P\c {}^*\mathcal{D}_1\c J\c {}^*\mathcal{D}_1^{-1}(\bb))}_{\lh{2}}&\les& 2^j\norm{P\c {}^*\mathcal{D}_1\c J\c {}^*\mathcal{D}_1^{-1}(\bb)}_{\tx{2}{1}}\\
\nn&\les& 2^j\norm{P}_{\tx{\infty}{2}}\norm{{}^*\mathcal{D}_1\c J\c {}^*\mathcal{D}_1^{-1}(\bb)}_{\lh{2}}\\
\nn&\les& 2^j\no(P)\norm{\bb}_{\lh{2}}\\
\nn&\les& 2^j\ep^2,
\eea
where we used the estimate \eqref{eq:estimdcal-1} for ${}^*\mathcal{D}^{-1}_1$, the assumption of Lemma \ref{lemma:lbt2} for $P$ and the curvature bound \eqref{curvflux1} for $\bb$.

We now consider the second term in the right-hand side of \eqref{lbt52}. Using the dual of the sharp Bernstein inequality \eqref{eq:strongbernscalarbis} and the estimate for transport equations \eqref{estimtransport1}, we have:
\bea\lab{lbt54}
&&\normm{P_j\left(\int_0^t(\ddb_{nL}(P)\c {}^*\mathcal{D}_1\c J\c {}^*\mathcal{D}_1^{-1}(\bb))d\tau\right)}_{\lh{2}}\\
\nn&\les& 2^j\normm{\int_0^t(\ddb_{nL}(P)\c {}^*\mathcal{D}_1\c J\c {}^*\mathcal{D}_1^{-1}(\bb))d\tau}_{\tx{2}{1}}\\
\nn&\les& 2^j\norm{\ddb_{nL}(P)\c {}^*\mathcal{D}_1\c J\c {}^*\mathcal{D}_1^{-1}(\bb)}_{\lh{1}}\\
\nn&\les& 2^j\norm{\ddb_{nL}(P)}_{\lh{2}}\norm{{}^*\mathcal{D}_1\c J\c {}^*\mathcal{D}_1^{-1}(\bb)}_{\lh{2}}\\
\nn&\les& 2^j\no(P)\norm{\bb}_{\lh{2}}\\
\nn&\les& 2^j\ep^2,
\eea
where we used the estimate \eqref{eq:estimdcal-1} for ${}^*\mathcal{D}^{-1}_1$, the assumption of Lemma \ref{lemma:lbt2} for $P$ and the curvature bound \eqref{curvflux1} for $\bb$.

We consider the third term in the right-hand side of \eqref{lbt52}. From the commutator formula \eqref{comm5} and the fact that $\mathcal{D}^{-1}_1(\bb)$ is a scalar, we have 
$$[{}^*\mathcal{D}_1,\ddb_{nL}]\c J\c({}^*\mathcal{D}^{-1}_1(\bb))=n\chi\c\nabb({}^*\mathcal{D}^{-1}_1(\bb))$$
which together with the dual of the sharp Bernstein inequality \eqref{eq:strongbernscalarbis} and the estimate for transport equations \eqref{estimtransport1} yields:
\bea
&&\nn\normm{P_j\left(\int_0^t(P\c [{}^*\mathcal{D}_1,\ddb_{nL}]\c J\c {}^*\mathcal{D}_1^{-1}(\bb))d\tau\right)}_{\lh{2}}\\
\nn&\les& \normm{P_j\left(\int_0^t(P\c n\chi\c\nabb({}^*\mathcal{D}^{-1}_1(\bb)))d\tau\right)}_{\lh{2}}\\
\nn&\les& 2^j\normm{\int_0^t(P\c n\chi\c\nabb({}^*\mathcal{D}^{-1}_1(\bb)))d\tau}_{\tx{2}{1}}\\
\nn&\les& 2^j\norm{P\c n\chi\c\nabb({}^*\mathcal{D}^{-1}_1(\bb))}_{\lh{1}}\\
\nn&\les& 2^j\norm{P}_{\lh{4}}\norm{n}_{L^\infty}\norm{\chi}_{\lh{4}}\norm{\nabb({}^*\mathcal{D}^{-1}_1(\bb))}_{\lh{2}}\\
\nn&\les& 2^j\no(P)\no(\chi)\norm{\bb}_{\lh{2}}\\
\lab{lbt55}&\les& 2^j\ep^2,
\eea
where we used the estimate \eqref{eq:estimdcal-1} for ${}^*\mathcal{D}^{-1}_1$, the assumption of Lemma \ref{lemma:lbt2} for $P$, the curvature bound \eqref{curvflux1} for $\bb$, and the $L^\infty$ bound for $n$ provided by \eqref{estn}.

We now consider the last term in the right-hand side of \eqref{lbt52}. We have:
\bee
\nn&&\normm{P_j\left(\int_0^t(P\c {}^*\mathcal{D}_1\c J\c [{}^*\mathcal{D}_1^{-1},\ddb_{nL}](\bb))d\tau\right)}_{\lh{2}}\\
&\les& \normm{P_j\left(\int_0^t\divb(P\c J\c [{}^*\mathcal{D}_1^{-1},\ddb_{nL}](\bb))d\tau\right)}_{\lh{2}}+
 \normm{P_j\left(\int_0^t(\nabb P\c J\c [{}^*\mathcal{D}_1^{-1},\ddb_{nL}](\bb))d\tau\right)}_{\lh{2}}
\eee
which together with Lemma \ref{lemma:lbt6}, the dual of the sharp Bernstein inequality \eqref{eq:strongbernscalarbis} and the estimate for transport equations \eqref{estimtransport1} yields:
\bea
\nn&&\normm{P_j\left(\int_0^t(P\c {}^*\mathcal{D}_1\c J\c [{}^*\mathcal{D}_1^{-1},\ddb_{nL}](\bb))d\tau\right)}_{\lh{2}}\\
\nn&\les& 2^j\norm{P\c J\c [{}^*\mathcal{D}_1^{-1},\ddb_{nL}](\bb)}_{\tx{1}{2}}+2^j\norm{\nabb P\c J\c [{}^*\mathcal{D}_1^{-1},\ddb_{nL}](\bb))}_{\lh{1}}\\
\nn&\les& 2^j\norm{P}_{\tx{2}{6}}\norm{[{}^*\mathcal{D}_1^{-1},\ddb_{nL}](\bb)}_{\tx{2}{3}}+2^j\norm{\nabb P}_{\lh{2}}\norm{[{}^*\mathcal{D}_1^{-1},\ddb_{nL}](\bb))}_{\lh{2}}\\
\nn&\les& 2^j\no(P)\ep\\
\lab{lbt56}&\les& 2^j\ep^2,
\eea
where we used Lemma \ref{lemma:lbt8} for the commutator term, and the assumption of Lemma \ref{lemma:lbt2} for $P$.

Finally, \eqref{lbt52}, \eqref{lbt53}, \eqref{lbt54}, \eqref{lbt55} and \eqref{lbt56} imply:
$$\normm{P_j\left(\int_0^t(P\c {}^*\mathcal{D}_1\c J\c {}^*\mathcal{D}_1^{-1}(\ddb_{nL}(\bb))d\tau\right)}_{\lh{2}}\les 2^j\ep.$$
Together with \eqref{lbt50} and \eqref{lbt51}, we obtain:
\be\lab{lbt57}
\normm{P_j\left(\int_0^t(P\c(\nabb(n\r)+(\nabb(n\s))^*)d\tau\right)}_{\lh{2}}\les 2^j\ep.
\ee
Now, \eqref{lbt45}, \eqref{lbt46}, \eqref{lbt47} and \eqref{lbt57} yield:
\be\lab{lbt58}
\normm{P_j\left(\int_0^t(P\c\ddb_{nL}\ddb_{\lb}(\z))d\tau\right)}_{\lh{2}}\les 2^j\ep+D\ep^22^j+D\ep^22^{\frac{j}{2}}\gamma(u),
\ee
which is the desired estimate for the third term in the right-hand side of \eqref{lbt22}.

\subsubsection{End of the proof of Lemma \ref{lemma:lbt2}}

\eqref{lbt22}, \eqref{lbt20bis}, \eqref{lbt35}, \eqref{lbt42} and \eqref{lbt58} imply:
\be\lab{lbt59}
\normm{P_j\left(\int_0^t(F\c\nabb\ddb_{\lb}(\z))d\tau\right)}_{\lh{2}}\les 2^j\ep+D\ep^22^j+D\ep^22^{\frac{j}{2}}\gamma(u),
\ee
which concludes the proof of Lemma \ref{lemma:lbt2}.

\subsection{Proof of Lemma \ref{lemma:lbt3}}\lab{sec:lbt3}

We decompose $\norm{P_j\left(\int_0^t(F\c\ddb{\lb}(\z))d\tau\right)}_{\lh{2}}$ using the property \eqref{eq:partition} of the geometric Littlewood-Paley projections:
\bea
\lab{lbt60}\normm{P_j\left(\int_0^t(F\c\ddb{\lb}(\z))d\tau\right)}_{\tx{\infty}{2}}&\les& \normm{P_j\left(\int_0^t(F\c P_{<0}\ddb{\lb}(\z))d\tau\right)}_{\tx{\infty}{2}}\\
\nn&&+\sum_{l\geq 0}\normm{P_j\left(\int_0^t(F\c P_l\ddb{\lb}(\z))d\tau\right)}_{\tx{\infty}{2}}.
\eea
We focus on the second term in the right-hand side of \eqref{lbt60}, the other being easier to handle.  Using the weak Bernstein inequality for $P_j$ and the estimate for transport equations \eqref{estimtransport1}, we have:
\bea
\lab{lbt61}\normm{P_j\left(\int_0^t(F\c P_l\ddb{\lb}(\z))d\tau\right)}_{\tx{\infty}{2}}&\les &2^{\frac{j}{3}}\normm{\int_0^t(F\c P_l\ddb{\lb}(\z))d\tau}_{\tx{\infty}{\frac{3}{2}}}\\
\nn&\les& 2^{\frac{j}{3}}\norm{F\c P_l\ddb{\lb}(\z)}_{\xt{1}{\frac{3}{2}}}\\
\nn&\les& 2^{\frac{j}{3}}\norm{F}_{\lh{2}}\norm{P_l\ddb{\lb}(\z)}_{\tx{2}{6}}\\
\nn&\les& 2^{\frac{j}{3}}2^{\frac{2l}{3}}\ep\norm{P_l\ddb{\lb}(\z)}_{\lh{2}}\\
\nn&\les& 2^{\frac{j}{3}+\frac{2l}{3}}(D\ep^2+D\ep^22^{-\frac{l}{2}}\gamma(u),
\eea
where we used the weak Bernstein inequality for $P_l$, the estimate \eqref{lbz10} for $\ddb{\lb}(\z)$ and the assumption of Lemma \ref{lemma:lbt3} for $F$.

We now make another decomposition using the property \eqref{eq:partition} of the geometric Littlewood-Paley projections:
\bea
\lab{lbt62}\normm{P_j\left(\int_0^t(F\c P_l\ddb{\lb}(\z))d\tau\right)}_{\tx{\infty}{2}}&\les& \normm{P_j\left(\int_0^t(P_{<0}(F)\c P_l\ddb{\lb}(\z))d\tau\right)}_{\tx{\infty}{2}}\\
\nn&&+\sum_{q\geq 0}\normm{P_j\left(\int_0^t(P_q(F)\c P_l\ddb{\lb}(\z))d\tau\right)}_{\tx{\infty}{2}}.
\eea
We focus on the second term in the right-hand side of \eqref{lbt62}, the other being easier to handle. Using the property of the Littlewood-Paley projection $P_l$, we have:
\bee
&&\normm{P_j\left(\int_0^t(P_q(F)\c P_l\ddb{\lb}(\z))d\tau\right)}_{\tx{\infty}{2}}\\
&\les&2^{-2l}\normm{P_j\left(\int_0^t(P_q(F)\c \lap P_l\ddb{\lb}(\z))d\tau\right)}_{\tx{\infty}{2}}\\
&\les&2^{-2l}\normm{P_j\left(\int_0^t\divb(P_q(F)\c \nabb P_l\ddb{\lb}(\z))d\tau\right)}_{\tx{\infty}{2}}+2^{-2l}\normm{P_j\left(\int_0^t(\nabb P_q(F)\c \nabb P_l\ddb{\lb}(\z))d\tau\right)}_{\tx{\infty}{2}}
\eee
which together with Lemma \ref{lemma:lbt5} with $p=\frac{4}{3}$, the strong Bernstein inequality \eqref{eq:strongbernscalarbis} and the estimate for transport equations \eqref{estimtransport1} yields:
\bea
\lab{lbt63}&&\normm{P_j\left(\int_0^t(P_q(F)\c P_l\ddb{\lb}(\z))d\tau\right)}_{\tx{\infty}{2}}\\
\nn&\les&2^{\frac{3j}{2}-2l}\norm{P_q(F)\c \nabb P_l\ddb{\lb}(\z)}_{\tx{1}{\frac{4}{3}}}+2^{j-2l}\norm{\nabb P_q(F)\c \nabb P_l\ddb{\lb}(\z)}_{\lh{1}}\\
\nn&\les&2^{\frac{3j}{2}-2l}\norm{P_q(F)}_{\tx{2}{4}}\norm{\nabb P_l\ddb{\lb}(\z)}_{\lh{2}}+2^{j-2l}\norm{\nabb P_q(F)}_{\lh{2}}\norm{\nabb P_l\ddb{\lb}(\z)}_{\lh{2}}\\
\nn&\les&(2^{\frac{3j}{2}+\frac{q}{2}-l}+2^{j+q-l})\norm{P_q(F)}_{\lh{2}}\norm{P_l\ddb{\lb}(\z)}_{\lh{2}}\\
\nn&\les&(2^{\frac{3j}{2}+\frac{q}{2}-l}+2^{j+q-l})\norm{P_q(F)}_{\lh{2}}(D\ep+D\ep 2^{-\frac{l}{2}}\gamma(u)),
\eea
where we used the weak Bernstein inequality for $P_q$, the finite band property for $P_q$ and $P_l$, and the estimate \eqref{lbz10} for $\ddb{\lb}(\z)$. Similarly, we may exchange the role of $l$ and $q$ and obtain:
\bea
&&\lab{lbt64}\normm{P_j\left(\int_0^t(P_q(F)\c P_l\ddb{\lb}(\z))d\tau\right)}_{\tx{\infty}{2}}\\
\nn&\les&(2^{\frac{3j}{2}+\frac{l}{2}-q}+2^{j+l-q})\norm{P_q(F)}_{\lh{2}}(D\ep+D\ep 2^{-\frac{l}{2}}\gamma(u)).
\eea

Now, using \eqref{lbt63} for $q\leq l$ and \eqref{lbt64} for $q>l$ and assuming $l>j$ yields:
\bee
&&\sum_{l> j, q\geq 0}\normm{P_j\left(\int_0^t(P_q(F)\c P_l\ddb{\lb}(\z))d\tau\right)}_{\tx{\infty}{2}}\\
&\les&\sum_{l> j, q\geq 0}(2^{j-\frac{|l-q|}{2}}+2^{j-|l-q|})\norm{P_q(F)}_{\lh{2}}(D\ep+D\ep 2^{-\frac{j}{2}}\gamma(u))\\
&\les& \left(\sum_{q\geq 0}\norm{P_q(F)}_{\lh{2}}\right)2^j(D\ep+D\ep 2^{-\frac{j}{2}}\gamma(u))\\
&\les& \norm{F}_{\PP^0}2^j(D\ep+D\ep 2^{-\frac{j}{2}}\gamma(u))\\
&\les& D\ep^22^j+D\ep^22^{\frac{j}{2}}\gamma(u)),
\eee
where we used the definition of $\PP^0$ and the assumption of Lemma \ref{lemma:lbt3} on $F$. Together with \eqref{lbt62}, this yields:
\be\lab{lbt65}
\sum_{l> j}\normm{P_j\left(\int_0^t(F\c P_l\ddb{\lb}(\z))d\tau\right)}_{\tx{\infty}{2}}\les D\ep^22^j+D\ep^22^{\frac{j}{2}}\gamma(u)).
\ee
Finally, using \eqref{lbt60}, \eqref{lbt61} for $l\leq j$, and \eqref{lbt65} for $l>j$, we obtain:
$$\normm{P_j\left(\int_0^t(F\c \ddb{\lb}(\z))d\tau\right)}_{\tx{\infty}{2}}\les D\ep^22^j+D\ep^22^{\frac{j}{2}}\gamma(u)),$$
which concludes the proof of Lemma \ref{lemma:lbt3}.

\subsection{Proof of Lemma \ref{lemma:lbt4}}\lab{sec:lbt4}

By duality, the conclusion of Lemma \ref{lemma:lbt4} is equivalent to the estimate:
\be\lab{matou1}
\normm{\nabb^2\left(\int_0^tP_jfd\tau\right)}_{\xt{2}{\infty}}\les 2^{2j}\norm{f}_{\lh{2}}
\ee
for any scalar function $f$ on $\H_u$ and any $j\geq 0$. Let $w$ the solution of the following 
transport equation:
\be\lab{matou2}
nL(w)=P_jf,\,w=0\textrm{ on }\pou.
\ee
Then, \eqref{matou1} may be rewritten as:
\be\lab{matou3}
\norm{\nabb^2w}_{\xt{2}{\infty}}\les 2^{2j}\norm{f}_{\lh{2}}.
\ee
From now on, we focus on obtaining \eqref{matou3}. We first derive an estimate 
for $\norm{\nabb w}_{L^\infty}$. Differentiating \eqref{matou2} 
with respect to $\nabb$ and using the commutator formula \eqref{comm5}, we obtain:
$$\ddb_{nL}(\nabb w)=n\chi\nabb w+\nabb P_jf,\,\nabb w=0\textrm{ on }\pou$$
which together with the estimate for transport equations \eqref{estimtransport1} implies:
$$\norm{\nabb w}_{L^{\infty}}\les \norm{n\chi}_{\xt{\infty}{2}}\norm{\nabb w}_{L^\infty}+\norm{\nabb P_jf}_{\tx{1}{\infty}}.$$
Using the $L^\infty$ bound for $n$ given by \eqref{estn} and the trace bound for $\chi$ given by \eqref{esttrc} \eqref{esthch}, we get:
\be\lab{matou4}
\norm{\nabb w}_{L^{\infty}}\les \norm{\nabb P_jf}_{\tx{1}{\infty}}.
\ee

In view of \eqref{matou4}, we need to estimate $ \norm{\nabb P_jf}_{\tx{1}{\infty}}$. Using the $L^\infty$ bound \eqref{linftynormtensor} for tensors on $\ptu$ with the choice $p=2$, we have:
\bea
\nn\norm{\nabb P_jf}_{L^\infty(\ptu)}&\les&\norm{\nabb^3P_jf}^{\frac{1}{2}}_{\lpt{2}}\norm{\nabb P_jf}^{\frac{1}{2}}_{\lpt{2}}+\norm{\nabb^2P_jf}_{\lpt{2}}\\
\nn&\les&\norm{\nabb^3P_jf}^{\frac{1}{2}}_{\lpt{2}}2^{\frac{j}{2}}\norm{P_jf}^{\frac{1}{2}}_{\lpt{2}}+\norm{\lap P_jf}_{\lpt{2}}+\norm{\nabb P_jf}_{\lpt{2}}\\
\lab{matou5}&\les& 2^{\frac{j}{2}}\norm{\nabb^3P_jf}^{\frac{1}{2}}_{\lpt{2}}\norm{f}^{\half}_{\lpt{2}}+2^{2j}\norm{f}_{\lpt{2}}
\eea
where we used the Bochner inequality \eqref{eq:Bochconseqbis}, and the $L^2$ boundedness and the finite band property of $P_j$. In view of \eqref{matou5}, we need to estimate $\norm{\nabb^3P_jf}_{\lpt{2}}$. Using the Bochner inequality for tensors \eqref{vbochineq}, we have:
\bea\lab{matou6}
&&\norm{\nabb^3P_jf}_{\lpt{2}}\\
\nn&\les&  \norm{\lap\nabb P_jf}_{\lpt{2}}+\norm{K}_{\lpt{2}}\norm{\nabb^2 P_jf}_{\lpt{2}}+\norm{K}^2_{\lpt{2}}\norm{\nabb P_jf}_{\lpt{2}}\\
\nn&\les&\norm{[\lap,\nabb] P_jf}_{\lpt{2}}+\norm{\nabb\lap P_jf}_{\lpt{2}}\\
\nn&&+\norm{K}_{\lpt{2}}(\norm{\lap P_jf}_{\lpt{2}}+\norm{\nabb P_jf}_{\lpt{2}})+\norm{K}^2_{\lpt{2}}2^j\norm{P_jf}_{\lpt{2}}\\
\nn&\les&\norm{[\lap,\nabb] P_jf}_{\lpt{2}}+2^{3j}\norm{f}_{\lpt{2}}+2^{2j}\norm{K}_{\lpt{2}}\norm{f}_{\lpt{2}}\\
\nn&&+2^j\norm{K}^2_{\lpt{2}}\norm{f}_{\lpt{2}}
\eea
where we used the Bochner inequality \eqref{eq:Bochconseqbis}, and the $L^2$ boundedness and the finite band property of $P_j$. Now, for any scalar function $f$ on $\ptu$, there holds the following commutator formula:
\be\lab{commutptu}
[\nabb,\lap]f=K\nabb f
\ee
which together with \eqref{matou6} yields:
\bea
\nn\norm{\nabb^3P_jf}_{\lpt{2}}&\les&  \norm{K\nabb P_jf}_{\lpt{2}}+2^{3j}\norm{f}_{\lpt{2}}+2^{2j}\norm{K}_{\lpt{2}}\norm{f}_{\lpt{2}}\\
\nn&&+2^j\norm{K}^2_{\lpt{2}}\norm{f}_{\lpt{2}}\\
\nn&\les&  \norm{K}_{\lpt{2}}\norm{\nabb P_jf}_{\lpt{\infty}}+2^{3j}\norm{f}_{\lpt{2}}+2^{2j}\norm{K}_{\lpt{2}}\norm{f}_{\lpt{2}}\\
\lab{matou7}&&+2^j\norm{K}^2_{\lpt{2}}\norm{f}_{\lpt{2}}.
\eea
Now, \eqref{matou5} and \eqref{matou7} imply:
\bee
\norm{\nabb P_jf}_{L^\infty(\ptu)}&\les& 2^{\frac{j}{2}}\norm{K}^{\frac{1}{2}}_{\lpt{2}}\norm{\nabb P_jf}^{\frac{1}{2}}_{L^\infty(\ptu)}\norm{f}^{\frac{1}{2}}_{\lpt{2}}\\
&&+2^{2j}(1+\norm{K}_{\lpt{2}})\norm{f}_{\lpt{2}}
\eee
which yields:
\be\lab{matou8}
\norm{\nabb P_jf}_{L^\infty(\ptu)}\les 2^{2j}(1+\norm{K}_{\lpt{2}})\norm{f}_{\lpt{2}}.
\ee
Integrating \eqref{matou8} and using the bound \eqref{estgauss1} for $K$, we obtain:
\be\lab{matou9}
\norm{\nabb P_jf}_{\tx{1}{\infty}}\les 2^{2j}(1+\norm{K}_{\lh{2}})\norm{f}_{\lh{2}}\les 2^{2j}\norm{f}_{\lh{2}}.
\ee

Next, we come back to $w$. \eqref{matou4} and \eqref{matou8} yield:
\be\lab{matou10}
\norm{\nabb w}_{L^\infty}\les 2^{2j}\norm{f}_{\lh{2}}.
\ee
Differentiating \eqref{matou2} with respect to $\nabb^2$ and using twice the commutator formula \eqref{comm5}, we obtain:
$$\ddb_{nL}(\nabb^2w)=-2n\chi\nabb^2 w+(2n\chi\kepb+\nabb(n\chi)-n\b)\nabb w+\nabb^2P_jf,\,\nabb^2w=0\textrm{ on }\pou$$
which together with the estimate for transport equations \eqref{estimtransport1} implies:
\bee
\norm{\nabb^2w}_{\xt{2}{\infty}}&\les& \norm{n\chi\nabb^2w}_{\xt{2}{1}}+\norm{(2n\chi\kepb+\nabb(n\chi)-n\b)\nabb w}_{\lh{2}}+\norm{\nabb^2P_jf}_{\lh{2}}\\
&\les & \norm{n}_{L^\infty}\norm{\chi}_{\xt{\infty}{2}}\norm{\nabb^2w}_{\lh{2}}+\norm{n}_{L^\infty}\norm{\nabb w}_{L^\infty}(\no(\chi)^2+\no(\kepb)^2\\
&&+\no(\nabb n)^2+\norm{\b}_{\lh{2}})+\norm{\lap P_jf}_{\lh{2}},
\eee
where we used the Bochner inequality \eqref{eq:Bochconseqbis} in the last inequality. 
Now, using \eqref{matou10}, the $L^2$ boundedness and the finite band property of $P_j$, 
the bound \eqref{curvflux1} for $\b$, and the estimates \eqref{estn}-\eqref{esthch} 
for $n$, $\chi$ and $\kepb$, we obtain:
$$\norm{\nabb^2w}_{\xt{2}{\infty}}\les\ep\norm{\nabb^2w}_{\lh{2}}+2^{2j}\norm{f}_{\lh{2}}.$$
This yields \eqref{matou3} which concludes the proof of Lemma \ref{lemma:lbt4}.

\subsection{Proof of Lemma \ref{lemma:lbt5}}\lab{sec:lbt5}

By duality, the conclusion of Lemma \ref{lemma:lbt5} is equivalent to the estimate:
\be\lab{matou11}
\normm{\nabb\left(\int_0^tP_jfd\tau\right)}_{\xt{p}{\infty}}\les 2^{2j(1-\frac{1}{p})}\norm{f}_{\tx{1}{2}}
\ee
for any scalar function $f$ on $\H_u$, any $1<p\leq 2$ and any $j\geq 0$. Consider again $w$ the solution of the transport equation \eqref{matou2}. Then, \eqref{matou11} may be rewritten as:
\be\lab{matou12}
\norm{\nabb w}_{\xt{p}{\infty}}\les 2^{2j(1-\frac{1}{p})}\norm{f}_{\tx{1}{2}}.
\ee
From now on, we focus on obtaining \eqref{matou12}. Differentiating \eqref{matou2} 
with respect to $\nabb$ and using the commutator formula \eqref{comm5}, we obtain:
$$\ddb_{nL}(\nabb w)=n\chi\nabb w+\nabb P_jf,\,\nabb w=0\textrm{ on }\pou$$
which together with the estimate for transport equations \eqref{estimtransport1} implies:
$$\norm{\nabb w}_{\xt{p}{\infty}}\les \norm{n\chi}_{\xt{\infty}{2}}\norm{\nabb w}_{\xt{p}{2}}+\norm{\nabb P_jf}_{\tx{1}{p}}.$$
Using the $L^\infty$ bound for $n$ given by \eqref{estn} and the trace bound for $\chi$ given by \eqref{esttrc} \eqref{esthch}, we get:
\be\lab{matou13}
\norm{\nabb w}_{\xt{p}{\infty}}\les \norm{\nabb P_jf}_{\tx{1}{p}}.
\ee

In view of \eqref{matou13}, we need to estimate $\norm{\nabb P_jf}_{\tx{1}{p}}$. The Gagliardo-Nirenberg inequality \eqref{eq:GNirenberg} yields:
\bea
\lab{matou14}\norm{\nabb P_jf}_{\lpt{p}}&\les& \norm{\nabb^2P_jf}^{1-\frac{2}{p}}_{\lpt{2}}\norm{\nabb P_jf}^{\frac{2}{p}}_{\lpt{2}}\\
\nn&\les & 2^{\frac{2j}{p}}(\norm{\lap P_jf}_{\lpt{2}}+\norm{\nabb P_jf}_{\lpt{2}})^{1-\frac{2}{p}}\norm{P_jf}^{\frac{2}{p}}_{\lpt{2}}\\
\nn&\les & 2^{2j(1-\frac{1}{p})}\norm{f}_{\lpt{2}}
\eea
where we used the Bochner inequality \eqref{eq:Bochconseqbis}, and the $L^2$ boundedness and the finite band properties of $P_j$. Integrating \eqref{matou14}, we obtain:
$$\norm{\nabb P_jf}_{\tx{1}{p}}\les 2^{2j(1-\frac{1}{p})}\norm{f}_{\tx{1}{2}}$$
which together with \eqref{matou13} yields \eqref{matou12}. This concludes the proof of Lemma \ref{lemma:lbt5}.

\subsection{Proof of Lemma \ref{lemma:lbt7}}\lab{sec:lbt7}

Recall that $J$ denotes the involution $(\r,\s)\rightarrow (-\r,\s)$. Then, $\nabb(n\r)+(\nabb(n\s))^*$ may be rewritten as:
$$\nabb(n\r)+(\nabb(n\s))^*={}^*\mathcal{D}_1\c J(n\r,n\s).$$
Now, in view of the Bianchi identity \eqref{bianc6}, we have:
$$(n\r,n\s)={}^*\mathcal{D}_1^{-1}\bigg(\ddb_{nL}(\bb)-\nabb(n)\r+\nabb(n)\s-2n\hchb\c\b-n\db\bb+3n(\zb\r-{}^*\zb\s)\bigg)$$
which yields:
\be\lab{lbt48}
\nabb(n\r)+(\nabb(n\s))^*={}^*\mathcal{D}_1\c J\c {}^*\mathcal{D}_1^{-1}(\ddb_{nL}(\bb))+{}^*\mathcal{D}_1(H)
\ee
where $H$ is given by:
$$H=J\c {}^*\mathcal{D}_1^{-1}\bigg(-\nabb(n)\r+\nabb(n)\s-2n\hchb\c\b-n\db\bb+3n(\zb\r-{}^*\zb\s)\bigg).$$
Now, using Lemma \ref{lemma:lbt6} with $p=\frac{4}{3}$, $q=3$, we obtain for $H$ the following estimate:
\bea
\nn\norm{H}_{\tx{2}{3}}&\les&\normm{J\c {}^*\mathcal{D}_1^{-1}\bigg(-\nabb(n)\r+\nabb(n)\s-2n\hchb\c\b-n\db\bb+3n(\zb\r-{}^*\zb\s)\bigg)}_{\tx{2}{3}}\\
\nn&\les&\normm{-\nabb(n)\r+\nabb(n)\s-2n\hchb\c\b-n\db\bb+3n(\zb\r-{}^*\zb\s)}_{\tx{2}{\frac{4}{3}}}\\
\nn&\les&\norm{\nabb n}_{\tx{\infty}{4}}\norm{\r}_{\lh{2}}+\norm{\nabb n}_{\tx{\infty}{4}}\norm{\s}_{\lh{2}}+\norm{n\hchb}_{\tx{\infty}{4}}\norm{\b}_{\lh{2}}\\
\nn&&+\norm{n\db}_{\tx{\infty}{4}}\norm{\bb}_{\lh{2}}+\norm{n\zb}_{\tx{\infty}{4}}\norm{\r}_{\lh{2}}+\norm{n\zb}_{\tx{\infty}{4}}\norm{\s}_{\lh{2}}\\
\lab{lbt49}&\les&\ep^2
\eea
where we used the curvature bound \eqref{curvflux1} for $\b, \r, \s$ and $\bb$, and the estimates \eqref{estn}-\eqref{estzeta} for $n, \db, \hchb$ and $\z$. \eqref{lbt48} and \eqref{lbt49} give the conclusion of the proof of Lemma \ref{lemma:lbt7}.

\section{Appendix to section \ref{sec:firstderivomega}}

\subsection{Proof of Lemma \ref{lemma:imj}}\lab{sec:imj}

Recall the transport equation \eqref{popo4} satisfied by $\Pi(\po\chi)$. We have
\bee
\ddb_L(\Pi(\po\chi))_{AB}&=&-\nabb_{\po N}\chi_{AB}-(\po\chi)_{AC}\chi_{CB}-\chi_{AC}(\po\chi)_{CB}-\db\po\chi_{AB}\\
\nn&&+\kepb_A\chi_{\po NB}+\kepb_B\chi_{A\po N}+(\po N)_A\chi_{CB}\kep_C+(\po N)_B\chi_{AC}\kepb_C\\
\nn&&-(2\kep_{\po N}-n^{-1}\nabla_{\po N}n)\chi_{AB}+(\po N)_C(\in_{AC}{}^*\b_B+\in_{BC}{}^*\b_A).
\eee
Differentiating with respect to $\ddb_{\lb}$, this yields, schematically
\bea\lab{imj3}
&&\ddb_L(\ddb_{\lb}(\Pi(\po\chi)))+[\ddb_{\lb},\ddb_L](\Pi(\po\chi))\\
\nn&=&-\ddb_{\lb}\Pi(\po\chi)\c\chi-\chi\c\ddb_{\lb}(\Pi(\po\chi))-\db\ddb_{\lb}(\Pi(\po\chi))+\po N\ddb_{\lb}\b+\nabb(F_3)+F_4,
\eea
where the tensors $F_3$ and $F_4$ are given schematically by
$$F_3=\po N\ddb_{\lb}\chi,$$
and
\bee
F_4&=& \nabb(\po N)\ddb_{\lb}\chi-\nabb_{\ddb_{\lb}\po N}\chi-(\po N)[\ddb_{\lb},\nabb]\chi-\po\chi\c\ddb_{\lb}\chi-\ddb_{\lb}\chi\c\po\chi-\lb(\db)\po\chi\\
\nn&&+\ddb_{\lb}(\kepb)\chi\po N+\kepb\ddb_{\lb}(\chi)\po N+\kepb\chi\ddb_{\lb}\po N-(2\kep_{\po N}-n^{-1}\nabla_{\po N}n)\ddb_{\lb}\chi\\
&&-(2\ddb_{\lb}(\kep)\po N+2\kep\ddb_{\lb}\po N-\ddb_{\lb}(n^{-1}\nabla_{\po N}n))\chi+\ddb_{\lb}\po N\b.
\eee
$F_3$ satisfies the following estimate
\be\lab{imj4}
\norm{F_3}_{\li{\infty}{2}}\les \norm{\po N}_{L^\infty}\norm{\ddb_{\lb}\chi}_{\li{\infty}{2}}\les \ep,
\ee
where we used in the last inequality the estimate \eqref{estNomega} for $\po N$ and the estimates \eqref{esttrc} \eqref{esthch} for $\chi$. Also, $F_4$ satisfies the following estimate
\bea\lab{imj5}
&&\norm{F_4}_{\tx{2}{1}}\\
\nn&\les& \norm{\dd\po N}_{\tx{\infty}{2}}\norm{\dd\chi}_{\li{\infty}{2}}+\norm{\po N}_{L^\infty}\norm{[\ddb_{\lb},\nabb]\chi}_{\tx{2}{1}}+\norm{\po\chi}_{\tx{\infty}{2}}\norm{\ddb_{\lb}\chi}_{\li{\infty}{2}}\\
\nn&&+\norm{\lb(\db)}_{\li{\infty}{2}}\norm{\po\chi}_{\tx{\infty}{2}}+\norm{\po N}_{L^\infty}\Big(\norm{\ddb_{\lb}(\kepb)}_{\li{\infty}{2}}\norm{\chi}_{\tx{\infty}{2}}\\
\nn&&+\norm{\kepb}_{\tx{\infty}{2}}\norm{\ddb_{\lb}(\chi)}_{\li{\infty}{2}}+\norm{n^{-1}\nabla n}_{\tx{\infty}{2}}\norm{\ddb_{\lb}\chi}_{\li{\infty}{2}}\Big)\\
\nn&&+\norm{\kepb}_{\tx{\infty}{4}}\norm{\chi}_{\tx{\infty}{4}}\norm{\ddb_{\lb}\po N}_{\li{\infty}{2}}+\norm{\ddb_{\lb}(n^{-1}\nabla_{\po N}n)}_{\li{\infty}{2}}\norm{\chi}_{\tx{\infty}{2}}\\
\nn&&+\norm{\ddb_{\lb}\po N}_{\tx{\infty}{2}}\norm{\b}_{\li{\infty}{2}}\\
\nn&\les& \ep+\norm{[\ddb_{\lb},\nabb]\chi}_{\tx{2}{1}},
\eea
where we used in the last inequality the estimates \eqref{estn}-\eqref{esthch} for $\chi$, $\db$, $\kep$, and $n$, the assumption \eqref{curvflux1} for $\b$, the estimate \eqref{estNomega} for $\po N$, and the estimate \eqref{estricciomega} for $\po N$ and $\po\chi$. Now, in view of the commutator formula \eqref{comm2}, we have
\bee
&&\norm{[\ddb_{\lb},\nabb]\chi}_{\tx{2}{1}}\\
&\les& (\norm{\chb}_{\tx{\infty}{2}}+\norm{\xib}_{\tx{\infty}{2}}+\norm{b^{-1}\nabb b}_{\tx{\infty}{2}})\norm{\dd\chi}_{\li{\infty}{2}}\\
&&+(\norm{\chi}_{\tx{\infty}{4}}(\norm{\kepb}_{\tx{\infty}{4}}+\norm{\xib}_{\tx{\infty}{4}})+\norm{\chb}_{\tx{\infty}{4}}\norm{\z}_{\tx{\infty}{4}})(\norm{\b}_{\li{\infty}{2}}+\norm{\bb}_{\li{\infty}{2}})\\
&\les&\ep,
\eee
where we used in the last inequality the estimates  \eqref{estn}-\eqref{esthch} for $\chi$, $\chb$, $\xib$, $b$, and $\z$, and the assumption \eqref{curvflux1} for $\b$ and $\bb$. Injecting in \eqref{imj5}, we obtain
\be\lab{imj6}
\norm{F_4}_{\tx{2}{1}}\les \ep.
\ee

Next, we estimate the commutator term in the right-hand side of \eqref{imj3}. In view of the commutator formula \eqref{comm3}, we have
\be\lab{imj7}
[\ddb_{\lb},\ddb_L](\Pi(\po\chi))=-\db\ddb_{\lb}(\Pi(\po\chi))+\nabb(F_5)+F_6,
\ee
where the tensors $F_5$ and $F_6$ are given schematically by
$$F_5=(\z-\zb)\Pi(\po\chi),$$
and
$$F_6= (\d + n^{-1} \nab_Nn)\ddb_L(\Pi(\po\chi))+(\nabb\z-\nabb\zb)\Pi(\po\chi)+ (\zb \z+\s)(\Pi(\po\chi)).$$ 
$F_5$ satisfies the following estimate
\be\lab{imj8}
\norm{F_5}_{\li{\infty}{2}}\les (\norm{\z}_{\xt{\infty}{2}}+\norm{\zb}_{\xt{\infty}{2}})\norm{\po\chi}_{\xt{2}{\infty}}\les\ep,
\ee
where we used in the last inequality the estimate \eqref{estk} for $\zb$, the estimate \eqref{estzeta} for $\z$, and the estimate \eqref{estricciomega} for $\po\chi$. Next, we estimate $F_6$. We have
\bea\lab{imj9}
&&\norm{F_6}_{\tx{2}{1}}\\
\nn&\les&  (\norm{\d}_{\tx{\infty}{2}} + \norm{n^{-1} \nab_Nn}_{\tx{\infty}{2}})\norm{\ddb_L(\Pi(\po\chi))}_{\li{\infty}{2}}+(\norm{\nabb\z}_{\li{\infty}{2}}\\
\nn&&+\norm{\nabb\zb}_{\li{\infty}{2}}+\norm{\z}_{\li{\infty}{4}}\norm{\zb}_{\li{\infty}{4}}+\norm{\s}_{\li{\infty}{2}})\norm{\po\chi}_{\tx{\infty}{2}}\\
\nn&\les&\ep+\norm{\ddb_L(\Pi(\po\chi))}_{\li{\infty}{2}},
\eea
where we used in the last inequality the estimate \eqref{estn} for $n$, the estimate \eqref{estk} for $\d$ and $\zb$, the estimate \eqref{estzeta} for $\z$, the assumption \eqref{curvflux1} for $\s$ and the estimate \eqref{estricciomega} for $\po\chi$. Now, the estimate \eqref{po7} for $\ddb_L(\Pi(\po\chi))$ together with the estimate \eqref{estricciomega} for $\po\chi$ implies
$$\norm{\ddb_L(\Pi(\po\chi))}_{\li{\infty}{2}}\les\ep.$$
Injecting in \eqref{imj9}, we obtain
\be\lab{imj10}
\norm{F_6}_{\tx{2}{1}}\les\ep.
\ee

Next, we evaluate the term involving $\ddb_{\lb}\b$ in the right-hand side of \eqref{imj3}. In view of the bianchi identity \eqref{bianc1bis}, we have
\be\lab{imj11}
\po N\ddb_{\lb}\b= \nabb(F_7)+F_8,
\ee
where the tensors $F_7$ and $F_8$ are given schematically by
$$F_7=\po N\r+\po N\s,$$
and
$$F_8=\nabb\po N(\r+\s)+\po N\Big(\hch\bb+(\d+n^{-1}\nabla_Nn)\b+\xib\c\a+\z\r+{}^*\z\s\Big).$$
$F_7$ satisfies the following estimate
\be\lab{imj12}
\norm{F_7}_{\li{\infty}{2}}\les \norm{\po N}_{L^\infty}(\norm{\r}_{\li{\infty}{2}}+\norm{\s}_{\li{\infty}{2}})\les\ep.
\ee
Next, we estimate $F_8$. We have
\bea\lab{imj13}
\norm{F_8}_{\tx{2}{1}}&\les& \Big(\norm{\nabb\po N}_{\tx{\infty}{2}}+\norm{\po N}_{L^\infty}(\norm{\hch}_{\tx{\infty}{2}}+\norm{\d}_{\tx{\infty}{2}}\\
\nn&&+\norm{n^{-1}\nabla_N n}_{\tx{\infty}{2}}+\norm{\xib}_{\tx{\infty}{2}}+\norm{\z}_{\tx{\infty}{2}})\Big)\norm{(\a, \b, \r, \s, \bb)}_{\li{\infty}{2}}\\
\nn&\les&\ep,
\eea
where we used in the last inequality the estimates \eqref{estn}-\eqref{estzeta} for $\hch$, $\d$, $n$, $\xib$ and $\z$, the assumption \eqref{curvflux1} on $(\a, \b, \r, \s, \bb)$, and the estimates \eqref{estNomega} \eqref{estricciomega} for $\po N$. 

Finally, in view of \eqref{imj3}, \eqref{imj7} and \eqref{imj11}, we obtain 
$$\ddb_L(\ddb_{\lb}(\Pi(\po\chi)))=-\ddb_{\lb}\Pi(\po\chi)\c\chi-\chi\c\ddb_{\lb}(\Pi(\po\chi))+\nabb(F_1)+F_2,$$
where the tensors $F_1$ and $F_2$ are given by
$$F_1=F_3+F_5+F_7\textrm{ and }F_2=F_4+F_6+F_8.$$
In view of \eqref{imj4}, \eqref{imj8} and \eqref{imj12}, we have
$$\norm{F_1}_{\li{\infty}{2}}\les \ep.$$
Also, in view of \eqref{imj6}, \eqref{imj10} and \eqref{imj13}, we have
$$\norm{F_2}_{\tx{2}{1}}\les\ep.$$
This concludes the proof of the lemma.

\subsection{Proof of Lemma \ref{lemma:poo3}}\lab{sec:proofpoo3}

Applying the estimate \eqref{estimtransport1} for transport equations to the transport equation \eqref{poo35} for $M$, we obtain:
\bee
\norm{M}_{L^\infty}&\les& \norm{\ga}_{\pou}+\norm{M\c\chi}_{\xt{\infty}{1}}\\
\nn&\les& 1+\norm{M}_{L^\infty}\norm{\chi}_{\xt{\infty}{2}}\\
&\les& 1+\ep\norm{M}_{L^\infty}
\eee
where we used the estimates \eqref{esttrc} \eqref{esthch} for $\chi$ in the last inequality. This yields:
\be\lab{poo47}
\norm{M}_{L^\infty}\les 1.
\ee
Now, since $\ddb_L\ga=0$, we may rewrite the transport equation \eqref{poo35} for $M$ as:
$$\ddb_L(M-\gamma)_{AB}=M_{AC}\chi_{CB},\,(M-\ga)_{AB}=0\textrm{ on }\pou,$$
Together with the estimate \eqref{estimtransport1} for transport equations, the estimates \eqref{esttrc} \eqref{esthch} for $\chi$, and the estimate \eqref{poo47}, this implies:
\bea
\lab{poo48}\norm{M-\ga}_{L^\infty}&\les& \norm{M\c\chi}_{\xt{\infty}{1}}\\
\nn&\les& \norm{M}_{L^\infty}\norm{\chi}_{\xt{\infty}{2}}\\
\nn&\les& \ep.
\eea

Next, we estimate $\nabb M$. We rewrite the transport equation \eqref{poo35} for $M$ as:
$$\ddb_{nL}M=n\chi\c M,\,M_{AB}=\ga_{AB}\textrm{ on }\pou.$$
Differentiating with respect to $\nabb$ and using the commutator formula \eqref{comm5}, we obtain:
\bee
\ddb_{nL}(\nabb M)&=&[\ddb_{nL},\nabb]M+\nabb\ddb_{nL}M\\
&=&n\chi\c\nabb M+(n\chi\c\kep +n\b+\nabb(n\chi))\c M
\eee
Together with the decomposition \eqref{estfluxcorres1:0} for $n\b$ and the decomposition \eqref{impbes22} for $\nabb(n\chi)$, we obtain:
\be\lab{poo49}
\ddb_{nL}(\nabb M)=n\chi\c\nabb M+(n\chi\c\kepb+\ddb_{nL}(P)+E)\c M,
\ee
where $P$ and $E$ satisfy:
\be\lab{poo50}
\no(P)+\norm{E}_{\PP^0}\les \ep.
\ee
\eqref{poo49} and the sharp trace theorem for transport equations \eqref{eq:mainleprop1t} imply:
\bea\lab{poo51}
\|\nabb M\|_{{\BB}^0}&\lesssim&\big(\no(n\chi)+\|n\chi\|_{L_{x'}^\infty L_t^2}\big)\|\nabb M\|_{{\PP}^0}\\
\nn&&+\big(\no(M)+\|M\|_{L_{x'}^\infty L_t^2}\big)\cdot(\no(P)+\|E\|_{{\PP}^0}+\|n\chi\c\kepb\|_{{\PP}^0})\\
\nn&\lesssim&\ep+\ep\|\nabb M\|_{{\PP}^0},
\eea
where we used the estimates \eqref{estn} \eqref{esttrc} \eqref{esthch} for $n$ and $\chi$, the estimate \eqref{poo47} for $M$, the estimate \eqref{poo50} for $P$ and $E$,  and the estimate:
$$\|n\chi\c\kepb\|_{{\PP}^0}\les\noo(n)\no(\chi)\no(\kepb)\les \ep,$$
which follows form the non sharp product estimates \eqref{nonsharpprod1} \eqref{nonsharpprod2} and the estimates \eqref{estn}-\eqref{esthch} for $n$, $\kepb$ and $\chi$. 

Finally, \eqref{poo51} yields:
$$\|\nabb M\|_{{\BB}^0}\les\ep$$
which together with \eqref{poo48} yields the conclusion of Lemma \ref{lemma:poo3}.

\subsection{Proof of Lemma \ref{lemma:poo4}}\lab{sec:proofpoo4}

\eqref{poo36bis} follows immediately from the sharp Bernstein inequality for tensors \eqref{eq:strongberntensor}. Then, \eqref{poo37} follows immediately from \eqref{poo36bis} by taking the dual.  
 This concludes the proof of Lemma \ref{lemma:poo4}.

\subsection{Proof of Lemma \ref{lemma:poo5}}\lab{sec:proofpoo5}

It suffices to prove the dual inequality. Let $H$ the solution of the following transport equation:
\be\lab{poo52}
\ddb_{nL}(H)=P_jF,\, H=0\textrm{ on }\pou.
\ee
Then, the conclusion of Lemma \ref{lemma:poo5} is equivalent to:
\be\lab{poo53}
\norm{\nabb H}_{\tx{\infty}{2}}\les 2^j\norm{F}_{\tx{p}{2}},
\ee
for any $1<p\leq 2$.

From now on, we focus on proving \eqref{poo53}. Note first from the estimate on transport equations \eqref{estimtransport1} and the transport equation \eqref{poo52} satisfied by $H$:
\be\lab{poo54}
\norm{H}_{L^\infty}\les \norm{P_jF}_{\tx{1}{\infty}}\les  2^j\norm{F}_{\tx{p}{2}},
\ee
where we used in the last inequality the sharp Bernstein inequality for tensors \eqref{poo36bis}.

Next, we differentiate the transport equation \eqref{poo52} for $H$ with respect to $\nabb$:
\bee
\ddb_{nL}(\nabb H)&=&[\ddb_{nL},\nabb]H+\nabb\ddb_{nL}(H)\\
\nn&=& n\chi\c\nabb H+(n\chi\c\kepb+n\b)\c H+\nabb P_jF,
\eee
where we used in the last equality the commutator formula \eqref{comm5}. Together with the estimate for transport equations \eqref{estimtransport1}, this yields:
\bea\lab{poo55}
&&\norm{\nabb H}_{\tx{\infty}{2}}\\
\nn&\les & \norm{n\chi\c\nabb H+(n\chi\c\kepb+n\b)\c H+\nabb P_jF}_{\xt{2}{1}}\\
\nn&\les & \norm{n}_{L^\infty}\big(\norm{\chi}_{\xt{\infty}{2}}\norm{\nabb H}_{\lh{2}}+(\no(\chi)\no(\kepb)+\norm{\b}_{\lh{2}})\norm{H}_{L^\infty}\big)+\norm{\nabb P_jF}_{\tx{1}{2}}\\
\nn&\les& \ep\norm{\nabb H}_{\lh{2}}+2^j\norm{F}_{\tx{p}{2}},
\eea
where we used in the last inequality the estimates \eqref{estn}-\eqref{esthch} for $n, \kepb$ and $\chi$, the estimate \eqref{poo54} for $H$, and the finite band property for $P_j$.

Finally, \eqref{poo55} yields \eqref{poo53} which concludes the proof of Lemma \ref{lemma:poo5}

\subsection{Proof of Lemma \ref{lemma:poo6}}\lab{sec:proofpoo6}

Using the product estimate \eqref{eq:secondbilBesov:bis}, we have:
\bea
\lab{poo56}\norm{\nabb(M^{-1})}_{\BB^0}&=& \norm{M^{-1}(\nabb M)M^{-1}}_{\BB^0}\\
\nn&\les&(\norm{\nabb(M^{-1})}_{\tx{\infty}{2}}+\norm{M^{-1}}_{L^\infty})^2\norm{\nabb M}_{\BB^0}\\
\nn&\les&(\norm{\nabb M}_{\tx{\infty}{2}}\norm{M^{-1}}^2_{L^\infty}+\norm{M^{-1}}_{L^\infty})^2\norm{\nabb M}_{\BB^0}\\
\nn&\les & \ep
\eea
where we used in the last inequality the fact that $\norm{M-\ga}_{L^\infty}+\norm{\nabb M}_{\BB^0}\les \ep$ from the assumptions of Lemma \ref{lemma:poo6}. Then, in view of \eqref{poo56}, Lemma \ref{lemma:poo6} is an immediate consequence of the following slightly more general lemma.  

\begin{lemma}\lab{lemma:grr}
Let $F$ a $\ptu$-tangent tensor and $2<p\leq +\infty$ such that for all $j\geq 0$:
$$\norm{P_jF}_{\tx{p}{2}}\les 2^j\ep+2^{\frac{j}{2}}\ep\gamma(u).$$
Also, let $H$ a $\ptu$-tangent tensor such that for any $2\leq r<+\infty$, we have
$$\norm{H}_{\tx{r}{\infty}}+\norm{\nabb H}_{L^r_tB^0_{2,1}(\ptu)}\les 1.$$
Then, we have for any $2\leq q<p$ and all $j\geq 0$:
$$\norm{P_j(HF)}_{\tx{q}{2}}\les 2^j\ep+2^{\frac{j}{2}}\ep\gamma(u).$$
\end{lemma}

We conclude this section with the proof of Lemma \ref{lemma:grr}. Using the property \eqref{eq:partition} of the Littlewood-Paley projections, we have:
\be\lab{poo57}
\norm{P_j(HF)}_{\tx{q}{2}}\les \sum_l\norm{P_j(HP_lF)}_{\tx{q}{2}}.
\ee
We estimate the right-hand side of \eqref{poo57}. Using the $L^2$ boundedness of $P_j$, the assumption $\norm{H}_{\tx{r}{\infty}}\les 1$ on $H$ with $r$ large enough, and the assumption for $F$:
\be\lab{poo58}
\norm{P_j(HP_lF)}_{\tx{q}{2}}\les \norm{H}_{\tx{r}{\infty}}\norm{P_lF}_{\tx{p}{2}}\les 2^l\ep+2^{\frac{l}{2}}\ep\gamma(u).
\ee

We will need a second estimate for the right-hand side of \eqref{poo57}. Using the property of $P_l$, we have:
\bea
\lab{poo59}&&\norm{P_j(HP_lF)}_{\tx{q}{2}}\\
\nn&=& 2^{-2l}\norm{P_j(H\lap P_lF)}_{\tx{q}{2}}\\
\nn&\les&2^{-2l}\norm{P_j\divb(H\nabb P_lF)}_{\tx{q}{2}}+2^{-2l}\norm{P_j(\nabb H\nabb P_lF)}_{\tx{q}{2}}\\
\nn&\les&2^{-2l}\norm{P_j\lap(HP_lF)}_{\tx{q}{2}}+2^{-2l}\norm{P_j\divb(\nabb HP_lF)}_{\tx{q}{2}}+2^{-2l}\norm{P_j(\nabb H\nabb P_lF)}_{\tx{q}{2}}.
\eea
Next, we  bound the three terms in the right-hand side of \eqref{poo59} starting with the first one. 
Using the finite band property for $P_j$, we have:
\bea
\lab{poo60}\norm{P_j\lap(HP_lF)}_{\tx{q}{2}}&\les &2^{2j}\norm{HP_lF}_{\tx{q}{2}}\\
\nn&\les &2^{2j}\norm{H}_{\tx{r}{\infty}}\norm{P_lF}_{\tx{p}{2}}\\
\nn&\les &2^{2j}(2^l\ep+2^{\frac{l}{2}}\ep\gamma(u)),
\eea
where we used the assumption $\norm{H}_{\tx{r}{\infty}}\les 1$ on $H$, the fact that we may choose $r$ large enough, and the assumption for $F$.

Next, we estimate the second term in the right-hand side of \eqref{poo59}. Using the property \eqref{eq:partition} of the Littlewood-Paley projections, we have:
\be\lab{poo61}
\norm{P_j\divb(\nabb(H)P_lF)}_{\tx{q}{2}}\les \sum_m\norm{P_j\divb(P_m(\nabb H)P_lF)}_{\tx{q}{2}}.
\ee
Using the finite band property for $P_j$, and the weak Bernstein inequality for $P_l$ and $P_m$, we have:
\bea
\lab{poo62}\norm{P_j\divb(P_m(\nabb H)P_lF)}_{\tx{q}{2}}&\les&2^j\norm{P_m(\nabb H)P_lF}_{\tx{q}{2}}\\
\nn&\les& 2^j\norm{P_m(\nabb H)}_{\tx{r}{4}}\norm{P_lF}_{\tx{p}{4}}\\
\nn&\les& 2^{j+\frac{m}{2}+\frac{l}{2}}\norm{P_m(\nabb H)}_{\tx{r}{2}}\norm{P_lF}_{\tx{p}{2}}\\
\nn&\les& 2^{j+\frac{m}{2}+\frac{l}{2}}\norm{P_m(\nabb H)}_{\tx{r}{2}}(2^l\ep+2^{\frac{l}{2}}\ep\gamma(u)),
\eea
where we used the assumption for $F$. 

We will need another estimate for the right-hand side of \eqref{poo61}. First, let $q_+$ such that $q<q_+<p_-<p$. Using the finite band property for $P_j$, we have for $r$ large enough:
\bea
\lab{poo63}\norm{P_j\divb(P_m(\nabb H)P_lF)}_{\tx{q_+}{2}}&\les&2^j\norm{P_m(\nabb H)P_lF}_{\tx{q_+}{2}}\\
\nn&\les& 2^j\norm{P_m(\nabb H)}_{\tx{r}{2}}\norm{P_lF}_{\tx{p_-}{\infty}}\\
\nn&\les& 2^{j+l}\norm{P_m(\nabb H)}_{\tx{r}{2}}\norm{P_lF}_{\tx{p}{2}},
\eea
where we used in the last inequality the sharp Bernstein inequality for tensors \eqref{poo36bis}. Also, using the properties of $P_m$, we have:
\bee
&&\norm{P_j\divb(P_m(\nabb H)P_lF)}_{\lpt{2}}\\
&=&2^{-2m}\norm{P_j\divb(\lap(P_m(\nabb H))P_lF)}_{\lpt{2}}\\
&\les &2^{-2m}\norm{P_j\divb\divb(\nabb(P_m(\nabb H))P_lF)}_{\lpt{2}}+2^{-2m}\norm{P_j\divb(\nabb(P_m(\nabb H))\nabb P_lF)}_{\lpt{2}}.
\eee
Together with the finite band property for $P_j$, this yields:
\bee
&&\norm{P_j\divb(P_m(\nabb H)P_lF)}_{\lpt{2}}\\
\nn&\les &2^{-2m}\norm{P_j\divb\divb}_{\mathcal{L}(\lpt{2})}\norm{\nabb(P_m(\nabb H))P_lF}_{\lpt{2}}+2^{j-2m}\norm{\nabb(P_m(\nabb H))\nabb P_lF}_{\lpt{2}}\\
\nn&\les &2^{-2m}\norm{\nabb^2P_j}_{\mathcal{L}(\lpt{2})}\norm{\nabb(P_m(\nabb H))}_{\lpt{2}}\norm{P_lF}_{\lpt{\infty}}\\
\nn&&+2^{j-2m}\norm{\nabb(P_m(\nabb H))}_{\lpt{4}}\norm{\nabb P_lF}_{\lpt{4}}\\
\nn&\les &2^{-m}\norm{\nabb^2P_j}_{\mathcal{L}(\lpt{2})}\norm{P_m(\nabb H)}_{\lpt{2}}\norm{P_lF}_{\lpt{\infty}}\\
\nn&&+2^{j-2m}\norm{\nabb^2(P_m(\nabb H))}^{\half}_{\lpt{2}}\norm{\nabb(P_m(\nabb H))}^{\half}_{\lpt{2}}\norm{\nabb^2P_lF}^{\half}_{\lpt{2}}\norm{\nabb P_lF}^{\half}_{\lpt{2}},
\eee
where we used in the last inequality the Gagliardo-Nirenberg inequality \eqref{eq:GNirenberg} and the finite band property for $P_l$. Finally, using the Bochner inequality for tensors \eqref{vbochineq}, the sharp Bernstein inequality \eqref{eq:strongberntensor} for tensors, and the fact that $p>2$, we obtain for $r$ large enough:
$$\norm{P_j\divb(P_m(\nabb H)P_lF)}_{\tx{1}{2}}\les (2^{2j+l-m}+2^{j-\frac{m}{2}+\frac{3l}{2}})\norm{P_m(\nabb H)}_{\tx{r}{2}}\norm{P_lF}_{\tx{p}{2}},$$
which in the case $j<l<m$ implies:
\be\lab{poo64}
\norm{P_j\divb(P_m(\nabb H)P_lF)}_{\tx{1}{2}}\les 2^{j-\frac{m}{2}+\frac{3l}{2}}\norm{P_m(\nabb H)}_{\tx{r}{2}}\norm{P_lF}_{\tx{p}{2}}.
\ee
Since $1<q<q_+$, we may interpolate \eqref{poo63} and \eqref{poo64}. We obtain for $j<l<m$:
\bea
\lab{poo65}\norm{P_j\divb(P_m(\nabb H)P_lF)}_{\tx{q}{2}}&\les& 2^{j+l-|m-l|_+}\norm{P_m(\nabb H)}_{\tx{r}{2}}\norm{P_lF}_{\tx{p}{2}}\\
\nn&\les& 2^{j+l-(m-l)_+}\norm{P_m(\nabb H)}_{\tx{r}{2}}(2^l\ep+2^{\frac{l}{2}}\ep\gamma(u)),
\eea
where we used the assumption for $F$. Now, using \eqref{poo61}, \eqref{poo62} for $m\leq l$ and \eqref{poo65} for $j<l<m$ yield for any $j<l$:
\be\lab{poo66}
\norm{P_j\divb((\nabb H)P_lF)}_{\tx{q}{2}}\les \sum_m2^{j+l-|m-l|_+}\norm{P_m(\nabb H)}_{\tx{r}{2}}(2^l\ep+2^{\frac{l}{2}}\ep\gamma(u)).
\ee

The the third term in the right-hand side of \eqref{poo59} satisfies for $r$ large enough the following estimate:
\be\lab{poo67}
\norm{P_j(\nabb H \nabb P_lF)}_{\tx{q}{2}}\les \sum_m2^{j+l-|m-l|_+}\norm{P_m(\nabb H)}_{\tx{r}{2}}(2^l\ep+2^{\frac{l}{2}}\ep\gamma(u)).
\ee
The proof of \eqref{poo67} is similar to the proof of \eqref{poo66}, so we skip it. 

\eqref{poo59}, \eqref{poo60}, \eqref{poo66} and \eqref{poo67} yield for any $j<l$:
\bea
\lab{poo68}\norm{P_j(HP_lF)}_{\tx{q}{2}}&\les& 2^{2j-2l}(2^l\ep+2^{\frac{l}{2}}\ep\gamma(u))\\
\nn&&+\sum_m2^{j-l-|m-l|_+}\norm{P_m(\nabb H)}_{\tx{r}{2}}(2^l\ep+2^{\frac{l}{2}}\ep\gamma(u)),
\eea
where $r$ is large enough. Finally, summing \eqref{poo57} for $l\leq j$ and \eqref{poo68} for $l>j$ implies for $r$ large enough:
$$\sum_l\norm{P_j(HP_lF)}_{\tx{q}{2}}\les (1+\norm{\nabb H}_{L^r_tB^0_{2,1}(\ptu)})(2^j\ep+2^{\frac{j}{2}}\ep\gamma(u)),$$
which together with the bound \eqref{poo56} for $H$ and the inequality \eqref{poo57} yields:
$$\norm{P_j(HF)}_{\tx{q}{2}}\les 2^j\ep+2^{\frac{j}{2}}\ep\gamma(u).$$
This concludes the proof of Lemma \ref{lemma:grr}.

\section{Proof of Lemma \ref{lemma:po41}}\lab{sec:gowinda2}

Using the property \eqref{eq:partition} of the Littlewood-Paley projections, we have:
\be\lab{po43}
\norm{F\c H}_{L^r_tB^0_{2,1}(\ptu)}\les \sum_{j,q,l}\norm{P_jP_q(F)\c P_l(H)}_{\tx{r}{2}}.
\ee
Note first that \eqref{po41} is symmetric with respect to $F$ and $H$. Thus, we may assume for instance $l\leq q$ in \eqref{po43}. We will estimate the right-hand side \eqref{po43} in the two cases $q\leq j$ and $q>j$ starting with the first one. We have:
\bee
&&\norm{P_jP_q(F)\c P_l(H)}_{\lpt{2}}\\
&\les& 2^{-2j}\norm{P_j\lap P_q(F)\c P_l(H)}_{\lpt{2}}\\
&\les & 2^{-2j}\norm{P_j\divb(\nabb P_q(F)\c P_l(H))}_{\lpt{2}}+2^{-2j}\norm{P_j\divb(P_q(F)\c \nabb P_l(H))}_{\lpt{2}},
\eee
which yields:
\bea
\lab{po44}&&\norm{P_jP_q(F)\c P_l(H)}_{\lpt{2}}\\
\nn&\les & 2^{-2j}\norm{P_j\divb}_{\mathcal{L}(\lpt{p},\lpt{2})}\bigg(\norm{\nabb P_q(F)\c P_l(H))}_{\lpt{p}}+\norm{P_q(F)\c \nabb P_l(H))}_{\lpt{p}}\bigg)\\
\nn&\les & 2^{-2j}\norm{P_j\divb}_{\mathcal{L}(\lpt{p},\lpt{2})}\bigg(\norm{\nabb P_q(F)}_{\lpt{2}}\norm{P_l(H))}_{\lpt{\frac{2p}{2-p}}}\\
\nn&&+\norm{P_q(F)}_{\lpt{\frac{2p}{2-p}}}\norm{\nabb P_l(H))}_{\lpt{2}}\bigg)\\
\nn&\les & 2^{-2j}(2^{(2-\frac{2}{p})l+q}+2^{(2-\frac{2}{p})q+l})\norm{P_j\divb}_{\mathcal{L}(\lpt{p},\lpt{2})}\norm{P_q(F)}_{\lpt{2}}\norm{P_l(H))}_{\lpt{2}},
\eea
where $1<p<2$ will be chosen later, and where we used the finite band property for $P_l$ and $P_q$, and the weak Bernstein inequality for $P_l$ and $P_q$. In view of \eqref{po44} we need to evaluate $\norm{P_j\divb}_{\mathcal{L}(\lpt{p},\lpt{2})}$. Let $p'$ the conjugate exponent of $p$, i.e. $\frac{1}{p}+\frac{1}{p'}=1$. Using the Gagliardo-Nirenberg inequality \eqref{eq:GNirenberg}, we have:
\bea
\lab{po45}\norm{\nabb P_jF}_{\lpt{p'}}&\les& \norm{\nabb P_jF}^{\frac{2}{p'}}_{\lpt{2}}\norm{\nabb^2 P_jF}^{1-\frac{2}{p'}}_{\lpt{2}}\\
\nn&\les& 2^{\frac{2j}{p'}}\norm{P_jF}^{\frac{2}{p'}}_{\lpt{2}}\big(\norm{\lap P_jF}_{\lpt{2}}+\norm{K}_{\lpt{2}}\norm{\nabb P_jF}_{\lpt{2}}\\
\nn&&+\norm{K}^2_{\lpt{2}}\norm{P_jF}_{\lpt{2}}\big)^{1-\frac{2}{p'}}\\
\nn&\les& (1+\norm{K}^{2-\frac{4}{p'}}_{\lpt{2}})2^{\frac{2j}{p}}\norm{P_jF}_{\lpt{2}},
\eea
where we used the weak Bernstein inequality and the Bochner inequality for tensors \eqref{vbochineq}. 
In view of \eqref{po45}, we have:
$$\norm{P_j\divb}_{\mathcal{L}(\lpt{p},\lpt{2})}=\norm{\nabb P_j}_{\mathcal{L}(\lpt{2},\lpt{p'})}\les (1+\norm{K}^{2-\frac{4}{p'}}_{\lpt{2}})2^{\frac{2j}{p}},$$
which together with \eqref{po44} implies:
\bee
&&\norm{P_jP_q(F)\c P_l(H)}_{\lpt{2}}\\
&\les&  2^{-(2-\frac{2}{p})j}(2^{(2-\frac{2}{p})l+q}+2^{(2-\frac{2}{p})q+l})(1+\norm{K}^{2-\frac{4}{p'}}_{\lpt{2}})\norm{P_q(F)}_{\lpt{2}}\norm{P_l(H))}_{\lpt{2}}.
\eee
We fix $p=\frac{2r}{r+1}$ which satisfies $1<p<2$. Using the estimate \eqref{estgauss1} for $K$, and the fact that $l\leq q\leq j$, we obtain:
\be\lab{po46}
\norm{P_jP_q(F)\c P_l(H)}_{\tx{r}{2}}\les  2^{-(\frac{1}{2}-\frac{1}{r})|j-q|-(\frac{1}{2}-\frac{1}{r})|j-l|}2^{\frac{q}{2}}\norm{P_q(F)}_{\tx{\infty}{2}}2^{\frac{l}{2}}\norm{P_l(H))}_{\tx{\infty}{2}}.
\ee

Next, we consider the case $q>j$. Using the weak Bernstein inequality for $P_j$ and $P_l$, we have:
\bee
\norm{P_jP_q(F)\c P_l(H)}_{\lpt{2}}&\les& 2^{\frac{j}{3}}\norm{P_q(F)\c P_l(H)}_{\lpt{\frac{3}{2}}}\\
\nn&\les& 2^{\frac{j}{3}}\norm{P_q(F)}_{\lpt{2}}\norm{P_l(H)}_{\lpt{6}}\\
\nn&\les& 2^{\frac{j}{3}+\frac{2l}{3}}\norm{P_q(F)}_{\lpt{2}}\norm{P_l(H)}_{\lpt{2}}\\
\nn&\les& 2^{-\frac{|j-q|}{6}-\frac{|l-j|}{6}}2^{\frac{q}{2}}\norm{P_q(F)}_{\lpt{2}}2^{\frac{l}{2}}\norm{P_l(H)}_{\lpt{2}}
\eee
where we used the fact that $q>j$ and $q\geq l$. This yields:
\be\lab{po47}
\norm{P_jP_q(F)\c P_l(H)}_{\tx{\infty}{2}} \les 2^{-\frac{|j-q|}{6}-\frac{|l-j|}{6}}2^{\frac{q}{2}}\norm{P_q(F)}_{\tx{\infty}{2}}2^{\frac{l}{2}}\norm{P_l(H)}_{\tx{\infty}{2}}.
\ee
Recall from \eqref{lbz14} that:
\be\lab{po48}
\sum_q2^q\norm{P_q(F)}^2_{\tx{\infty}{2}}\les \no(F)^2\textrm{ and }\sum_l2^l\norm{P_l(H)}^2_{\tx{\infty}{2}}\les \no(H)^2.
\ee
\eqref{po43}, \eqref{po46}, \eqref{po47} and \eqref{po48} imply \eqref{po41} which concludes the proof of Lemma \ref{lemma:po41}.

\subsection{Proof of Lemma \ref{lemma:po1}}\lab{sec:gowinda3}

Since $H=(\r, \s, \b, \bb)$, Lemma \ref{lemma:po5} yields:
\be\lab{po68}
\norm{P_lH}_{\tx{\infty}{2}}\les \ep2^{\frac{l}{2}}.
\ee
We estimate the quantity $\norm{P_j\mathcal{D}_2^{-1}(bF\c P_lH)}_{\tx{\infty}{4_-}}$. Using the weak Bernstein inequality and the finite band property for $P_j$, we have:
\bea
\lab{po69}\norm{P_j\mathcal{D}_2^{-1}(bF\c P_lH)}_{\tx{\infty}{4_-}}&\les& 2^{-j(\half)_+}\norm{\nabb\mathcal{D}_2^{-1}(bF\c P_lH)}_{\tx{\infty}{2}}\\
\nn&\les& 2^{-j(\half)_+}\norm{bF\c P_lH}_{\tx{\infty}{2}}\\
\nn&\les& 2^{-j(\half)_+}\norm{b}_{L^\infty}\norm{F}_{L^\infty}\norm{P_lH}_{\tx{\infty}{2}}\\
\nn&\les& 2^{-j(\half)_++\frac{l}{2}}\norm{F}_{L^\infty}\ep,
\eea
where we used the estimate \eqref{eq:estimdcal-1} for $\mathcal{D}_2^{-1}$, the estimate \eqref{estb} for $b$ and the estimate \eqref{po68} for $H$.

We derive a second estimate for $\norm{P_j\mathcal{D}_2^{-1}(bF\c P_lH)}_{\tx{\infty}{4_-}}$. We have:
\bee
&&\norm{P_j\mathcal{D}_2^{-1}(bF\c P_lH)}_{\tx{\infty}{4_-}}\\
&= & 2^{-2l}\norm{P_j\mathcal{D}_2^{-1}(bF\c \lap P_lH)}_{\tx{\infty}{4_-}}\\
&\les& 2^{-2l}\norm{P_j\mathcal{D}_2^{-1}\divb(bF\c \nabb P_lH)}_{\tx{\infty}{4_-}}+2^{-2l}\norm{P_j\mathcal{D}_2^{-1}(\nabb(bF)\c \nabb P_lH)}_{\tx{\infty}{4_-}}
\eee
which together with the weak Bernstein inequality for $P_j$ yields:
\bea
\lab{po70}&&\norm{P_j\mathcal{D}_2^{-1}(bF\c P_lH)}_{\tx{\infty}{4_-}}\\
\nn&\les& 2^{j(\half)_--2l}\norm{\mathcal{D}_2^{-1}\divb(bF\c \nabb P_lH)}_{\tx{\infty}{2}}+2^{j(\half)_--2l}\norm{\mathcal{D}_2^{-1}(\nabb(bF)\c \nabb P_lH)}_{\tx{\infty}{2}}\\
\nn&\les& 2^{j(\half)_--2l}\norm{bF\c P_lH}_{\tx{\infty}{2}}+2^{-j(\half)_+-2l}\norm{\nabb(bF)\c P_lH}_{\tx{\infty}{1_+}}\\
\nn&\les& 2^{j(\half)_--2l}\norm{b}_{L^\infty}\norm{F}_{L^\infty}\norm{\nabb P_lH}_{\tx{\infty}{2}}+2^{-j(\half)_+-2l}\norm{\nabb(bF)}_{\tx{\infty}{2}}\norm{\nabb P_lH}_{\tx{\infty}{2_+}}\\
\nn&\les& 2^{j(\half)_--l(\frac{1}{2})_-}(\norm{F}_{L^\infty}+\norm{\nabb F}_{\tx{\infty}{2}})\ep,
\eea
where we used the estimate \eqref{eq:estimdcal-1}  and Remark \ref{rem:Dcont} for $\mathcal{D}_2^{-1}$, the estimate \eqref{estb} for $b$ and the estimate \eqref{po68} for $H$. In the last inequality, note that as soon as $4_-$ is fixed in the $\tx{\infty}{4_-}$ norm, then $(\half)_-$ is fixed in $j(\half)_-$. Let us fix $j(\half)_-=j(\half-a)$ for some $a>0$, then we may choose $l(\half)_-=l(\half-\frac{a}{2})$ in order to obtain:
$$2^{j(\half)_--l(\half)_-}=2^{j(\half-a)-l(\half-\frac{a}{2})},$$
which together with \eqref{po70} yields:
\be\lab{po71bis}
\norm{P_j\mathcal{D}_2^{-1}(bF\c P_lH)}_{\tx{\infty}{4_-}}\les 2^{j(\half-a)-l(\half-\frac{a}{2})}(\norm{F}_{L^\infty}+\norm{\nabb F}_{\tx{\infty}{2}})\ep.
\ee

Summing on $j$ and $l$ and taking \eqref{po69} for $l\leq j$ and \eqref{po71bis} for $l>j$, we obtain:
\bee
\norm{P_j\mathcal{D}_2^{-1}(bF\c P_lH)}_{\tx{\infty}{4_-}}&\les& \sum_{j,l}\norm{P_j\mathcal{D}_2^{-1}(bF\c P_lH)}_{\tx{\infty}{4_-}}\\
&\les&\left(\sum_{l\leq j}2^{-j(\half)_++\frac{l}{2}}+\sum_{l>j}2^{j(\half-a)-l(\half-\frac{a}{2})}\right)(\norm{F}_{L^\infty}+\norm{\nabb F}_{\tx{\infty}{2}})\ep\\
&\les&(\norm{F}_{L^\infty}+\norm{\nabb F}_{\tx{\infty}{2}})\ep
\eee
which yields the conclusion of the Lemma. 

\subsection{Proof of Lemma \ref{lemma:po2}}\lab{sec:gowinda4}

Since $h=(\r, \s)$, Lemma \ref{lemma:po5} yields:
\be\lab{po71}
\norm{P_lh}_{\tx{\infty}{2}}\les \ep2^{\frac{l}{2}}.
\ee
We have:
\be\lab{po72}
\norm{\mathcal{D}_2^{-1}b^{-1}\mathcal{D}_2^{-1}(b\nabb h)}_{\tx{p}{4_-}}\les \norm{\mathcal{D}_2^{-1}b^{-1}\mathcal{D}_2^{-1}((\nabb b)h)}_{\tx{p}{4_-}}+\norm{\mathcal{D}_2^{-1}b^{-1}\mathcal{D}_2^{-1}(\nabb(bh)}_{\tx{p}{4_-}}.
\ee
Lemma \ref{lemma:po3} applied to the firs term in the right-hand side of \eqref{po72} with $F=b^{-1}\nabb b$ and $H=h=(\r, \s)$ yields:
\be\lab{po73}
 \norm{\mathcal{D}_2^{-1}b^{-1}\mathcal{D}_2^{-1}((\nabb b)h)}_{\tx{p}{4_-}}\les \no(\nabb b)\ep\les \ep,
 \ee
where we used the estimate \eqref{estb} for $b$.

Next, we evaluate $\norm{P_l(bP_jh)}_{\tx{\infty}{2}}$. Using the $L^2$ boundedness of $P_l$, we have:
\bea
\lab{po74}\norm{P_l(bP_jh)}_{\tx{\infty}{2}}&\les& \norm{bP_jh}_{\tx{\infty}{2}}\\
\nn&\les& \norm{b}_{L^\infty}\norm{P_jh}_{\tx{\infty}{2}}\\
\nn&\les& 2^{\frac{j}{2}},
\eea
where we used the estimate \eqref{estb} for $b$ and the estimate \eqref{po71} for $h$. 

We derive a second estimate for $\norm{P_l(bP_jh)}_{\tx{\infty}{2}}$. We have:
\bee
\norm{P_l(bP_jh)}_{\tx{\infty}{2}}&\les& 2^{-2j}\norm{P_l(b\lap P_jh)}_{\tx{\infty}{2}}\\
&\les& 2^{-2j}\norm{P_l\divb(b\nabb P_jh)}_{\tx{\infty}{2}}+2^{-2j}\norm{P_l(\nabb b\nabb P_jh)}_{\tx{\infty}{2}}\\
\eee
which together with the finite band property and the weak Bernstein inequality for $P_l$ yields:
\bea
\lab{po75}\norm{P_l(bP_jh)}_{\tx{\infty}{2}}&\les& 2^{l-2j}\norm{b\nabb P_jh}_{\tx{\infty}{2}}+2^{\frac{l}{2}-2j}\norm{\nabb b\nabb P_jh}_{\tx{\infty}{\frac{4}{3}}}\\
\nn&\les& 2^{l-2j}\norm{b}_{L^\infty}\norm{\nabb P_jh}_{\tx{\infty}{2}}+2^{\frac{l}{2}-2j}\norm{\nabb b}_{\tx{\infty}{4}}\norm{\nabb P_jh}_{\tx{\infty}{2}}\\
\nn&\les& 2^{l-j}\norm{P_jh}_{\tx{\infty}{2}}+2^{\frac{l}{2}-j}\no(\nabb b)\norm{P_jh}_{\tx{\infty}{2}}\\
\nn&\les&2^{l-\frac{j}{2}}\ep,
\eea
where we used the finite band property for $P_l$, the estimate \eqref{estb} for $b$ and the estimate \eqref{po71} for $h$. Finally, \eqref{po74} for $j\leq l$ and \eqref{po75} for $j>l$ yield:
\be\lab{po76}
\norm{P_l(bh)}_{\tx{\infty}{2}}\les \sum_j\norm{P_l(bP_jh)}_{\tx{\infty}{2}}\les 2^{\frac{l}{2}}\ep.
\ee

In view of \eqref{po72}, we need to evaluate $\norm{\mathcal{D}_2^{-1}b^{-1}\mathcal{D}_2^{-1}(\nabb(bh)}_{\tx{p}{4_-}}$. Note first that we have the following commutator formula:
$$\mathcal{D}_2^{-1}\nabb-\nabb\mathcal{D}_1^{-1}=\mathcal{D}_2^{-1}K\mathcal{D}_1^{-1}$$
which yields:
\be\lab{po77}
\norm{\mathcal{D}_2^{-1}b^{-1}\mathcal{D}_2^{-1}(\nabb(bh)}_{\tx{p}{4_-}}\les \norm{\mathcal{D}_2^{-1}b^{-1}\nabb\mathcal{D}_1^{-1}(bh)}_{\tx{p}{4_-}}+\norm{\mathcal{D}_2^{-1}b^{-1}\mathcal{D}_2^{-1}K\mathcal{D}_1^{-1}(bh)}_{\tx{p}{4_-}}.
\ee
We first evaluate the first term in the right-hand side of \eqref{po77}. Using the weak Bernstein inequality for $P_j$, we have:
\bea
\lab{po78}&&\norm{P_j\mathcal{D}_2^{-1}b^{-1}\nabb\mathcal{D}_1^{-1}P_l(bh)}_{\tx{\infty}{4_-}}\\
\nn&\les& 2^{(\half)_-j}\norm{\mathcal{D}_2^{-1}b^{-1}\nabb\mathcal{D}_1^{-1}P_l(bh)}_{\tx{\infty}{2}}\\
\nn&\les& 2^{(\half)_-j}\norm{\mathcal{D}_2^{-1}\nabb(b^{-1}\mathcal{D}_1^{-1}P_l(bh))}_{\tx{\infty}{2}}
+2^{(\half)_-j}\norm{\mathcal{D}_2^{-1}\nabb(b^{-1})\mathcal{D}_1^{-1}P_l(bh)}_{\tx{\infty}{2}}\\
\nn&\les& 2^{(\half)_-j}\norm{b^{-1}}_{L^\infty}\norm{\mathcal{D}_1^{-1}P_l(bh)}_{\tx{\infty}{2}}
+2^{(\half)_-j}\norm{\nabb(b^{-1})\mathcal{D}_1^{-1}P_l(bh)}_{\tx{\infty}{\frac{4}{3}}}\\
\nn&\les& 2^{(\half)_-j-\frac{l}{2}}\ep+2^{(\half)_-j}\norm{\nabb(b^{-1})}_{\tx{\infty}{4}}\norm{\mathcal{D}_1^{-1}P_l(bh)}_{\tx{\infty}{2}}\\
\nn&\les& 2^{(\half)_-j-\frac{l}{2}}\ep,
\eea
where we used the estimate \eqref{eq:estimdcal-1} and the Remark \ref{rem:Dcont} for $\mathcal{D}_2^{-1}$, the estimate \eqref{eq:estimdcal-1} for $\mathcal{D}_1^{-1}$, the finite band property for $P_l$, the estimate \eqref{estb} for $b$, and the estimate \eqref{po76} for $bh$.

We derive a second estimate for $\norm{P_j\mathcal{D}_2^{-1}b^{-1}\nabb\mathcal{D}_1^{-1}P_l(bh)}_{\tx{\infty}{4_-}}$. Using the weak Bernstein inequality and the finite band property for $P_j$, we have:
\bea
\lab{po78bis}&&\norm{P_j\mathcal{D}_2^{-1}b^{-1}\nabb\mathcal{D}_1^{-1}P_l(bh)}_{\tx{\infty}{4_-}}\\
\nn&\les& 2^{(\half)_-j}\norm{P_j\mathcal{D}_2^{-1}b^{-1}\nabb\mathcal{D}_1^{-1}P_l(bh)}_{\tx{\infty}{2}}\\
\nn&\les& 2^{-(\half)_+j}\norm{\nabb\mathcal{D}_2^{-1}b^{-1}\nabb\mathcal{D}_1^{-1}P_l(bh)}_{\tx{\infty}{2}}\\
\nn&\les& 2^{-(\half)_+j}\norm{b^{-1}}_{L^\infty}\norm{\nabb\mathcal{D}_1^{-1}P_l(bh)}_{\tx{\infty}{2}}\\
\nn&\les& 2^{-(\half)_+j}\norm{P_l(bh)}_{\tx{\infty}{2}}\\
\nn&\les& 2^{-(\half)_+j+\frac{l}{2}}\ep,
\eea
where we used the estimate \eqref{eq:estimdcal-1} for $\mathcal{D}_2^{-1}$, the estimate \eqref{eq:estimdcal-1} for $\mathcal{D}_1^{-1}$, the estimate \eqref{estb} for $b$, and the estimate \eqref{po76} for $bh$.
Summing \eqref{po78} for $j\leq l$ and \eqref{po78bis} for $j>l$ yields:
\be\lab{po78ter}
\norm{\mathcal{D}_2^{-1}b^{-1}\nabb\mathcal{D}_1^{-1}(bh)}_{\tx{p}{4_-}}\les \sum_{j\leq l}2^{j(\half)_--\frac{l}{2}}\ep+\sum_{j>l}2^{-j(\half)_++\frac{l}{2}}\ep\les \ep.
\ee

Next, we evaluate the second term in the right-hand side of \eqref{po77}. Using Remark \ref{rem:Dcont} for $\mathcal{D}_2^{-1}$, we have:
\bea
\lab{po79}\norm{\mathcal{D}_2^{-1}b^{-1}\mathcal{D}_2^{-1}K\mathcal{D}_1^{-1}(bh)}_{\tx{p}{4_-}}&\les& \norm{b^{-1}}_{L^\infty}\norm{\mathcal{D}_2^{-1}K\mathcal{D}_1^{-1}(bh)}_{\tx{p}{2}}\\
\nn&\les& \norm{\mathcal{D}_2^{-1}K\mathcal{D}_1^{-1}(bh)}_{\tx{p}{2}}.
\eea
Using Remark \ref{rem:Dcont} for $\mathcal{D}_2^{-1}$, we have:
\bea
\lab{po80}\norm{\mathcal{D}_2^{-1}P_j(K)\mathcal{D}_1^{-1}P_l(bh)}_{\tx{p}{2}}&\les& \norm{P_j(K)\mathcal{D}_1^{-1}P_l(bh)}_{\tx{p}{1_+}}\\
\nn&\les& \norm{P_j(K)}_{\tx{p}{2_+}}\norm{\mathcal{D}_1^{-1}P_l(bh)}_{\tx{\infty}{2}}\\
\nn&\les& 2^{j(\half)_-}\norm{K}^{\frac{2}{p}}_{\lh{2}}\norm{\La^{-\half}K}^{1-\frac{2}{p}}_{\tx{\infty}{2}}2^{-l}\norm{P_l(bh)}_{\tx{\infty}{2}}\\
\nn&\les& 2^{j(\half)_--\frac{l}{2}}\ep,
\eea
where we used the weak Bernstein inequality for $P_j$, that finite band property for $P_l$, the estimates \eqref{estgauss1} and \eqref{estgauss2} for $K$, and the estimate \eqref{po76} for $bh$.

We derive another estimate for $\norm{\mathcal{D}_2^{-1}P_j(K)\mathcal{D}_1^{-1}P_l(bh)}_{\tx{p}{2}}$. We have:
\bee
&&\norm{\mathcal{D}_2^{-1}P_j(K)\mathcal{D}_1^{-1}P_l(bh)}_{\tx{p}{2}}\\
&=&2^{-2j}\norm{\mathcal{D}_2^{-1}\lap P_j(K)\mathcal{D}_1^{-1}P_l(bh)}_{\tx{p}{2}}\\
&\les& 2^{-2j}\norm{\mathcal{D}_2^{-1}\divb(\nabb P_j(K)\mathcal{D}_1^{-1}P_l(bh))}_{\tx{p}{2}}+
2^{-2j}\norm{\mathcal{D}_2^{-1}\nabb P_j(K)\nabb\mathcal{D}_1^{-1}P_l(bh)}_{\tx{p}{2}}
\eee
which together with the estimate \eqref{eq:estimdcal-1} and the Remark \ref{rem:Dcont} for $\mathcal{D}_2^{-1}$ yields:
\bea
\lab{po81}&&\norm{\mathcal{D}_2^{-1}P_j(K)\mathcal{D}_1^{-1}P_l(bh)}_{\tx{p}{2}}\\
\nn&\les& 2^{-2j}\norm{\nabb P_j(K)\mathcal{D}_1^{-1}P_l(bh)}_{\tx{p}{2}}+
2^{-2j}\norm{\nabb P_j(K)\nabb\mathcal{D}_1^{-1}P_l(bh)}_{\tx{p}{1_+}}\\
\nn&\les& 2^{-2j}\norm{\nabb P_j(K)}_{\tx{p}{2_+}}\norm{\nabb\mathcal{D}_1^{-1}P_l(bh)}_{\tx{\infty}{2}}\\
\nn&\les& 2^{-2j}\norm{\nabb^2P_j(K)}^{1-\frac{2}{2_+}}_{\tx{p}{2}}\norm{\nabb P_j(K)}^{\frac{2}{2_+}}_{\tx{p}{2}}\norm{P_l(bh)}_{\tx{\infty}{2}}\\
\nn&\les& 2^{-j_-+\frac{l}{2}}\norm{P_j(K)}_{\tx{p}{2}}\\
\nn&\les&2^{-j(\half)_++\frac{l}{2}}\norm{K}^{\frac{2}{p}}_{\lh{2}}\norm{\La^{-\half}K}^{1-\frac{2}{p}}_{\tx{\infty}{2}}\ep\\
\nn&\les& 2^{-j(\half)_++\frac{l}{2}}\ep,
\eea
where we used the finite band property for $P_j$, the estimate \eqref{eq:estimdcal-1} for $\mathcal{D}_1^{-1}$, the Bochner inequality for scalars \eqref{eq:Bochconseqbis}, the estimates \eqref{estgauss1} and \eqref{estgauss2} for $K$, and the estimate \eqref{po76} for $bh$. Using \eqref{po79}, and summing \eqref{po80} for $j\leq l$ and \eqref{po81} for $j>l$ yields:
\be\lab{po82}
\norm{\mathcal{D}_2^{-1}b^{-1}\mathcal{D}_2^{-1}K\mathcal{D}_1^{-1}(bh)}_{\tx{p}{4_-}}\les \sum_{j\leq l}2^{j(\half)_--\frac{l}{2}}\ep+\sum_{j>l}2^{-j(\half)_++\frac{l}{2}}\ep\les \ep.
\ee

Finally, \eqref{po72}, \eqref{po73}, \eqref{po77}, \eqref{po78ter} and \eqref{po82} yield the conclusion of Lemma \ref{lemma:po2}.

\subsection{Proof of Lemma \ref{lemma:po3}}\lab{sec:gowinda5}

Since $H=(\r, \s, \b, \bb)$ and $\no(G)\les\ep$, the curvature estimate \eqref{curvflux1} and the finite band property for $P_l$ yield:
\be\lab{po83bis}
\norm{P_lH}_{\lh{2}}\les \ep\textrm{ and }\norm{P_lG}_{\lh{2}}\les 2^{-l}\ep.
\ee
while Lemma \ref{lemma:po5} and Lemma \ref{lemma:lbz4} yield:
\be\lab{po83}
\norm{P_lH}_{\tx{\infty}{2}}\les 2^{\frac{l}{2}}\ep\textrm{ and }\norm{P_lG}_{\tx{\infty}{2}}\les 2^{-\frac{l}{2}}\ep.
\ee

Using Remark \ref{rem:Dcont} for $\mathcal{D}_2^{-1}$, we have:
\bea
\lab{po85}&&\norm{\mathcal{D}_2^{-1}b(\mathcal{D}_2^{-1}(F\c H))}_{\tx{p}{4_-}}+
\norm{\mathcal{D}_2^{-1}b(\mathcal{D}_2^{-1}(F\c \nabb G))}_{\tx{p}{4_-}}\\
\nn&\les& \norm{b(\mathcal{D}_2^{-1}(F\c H))}_{\tx{p}{2}}+
\norm{b(\mathcal{D}_2^{-1}(F\c \nabb G))}_{\tx{p}{2}}\\
\nn&\les& \norm{\mathcal{D}_2^{-1}(F\c H)}_{\tx{p}{2}}+
\norm{\mathcal{D}_2^{-1}(F\c \nabb G)}_{\tx{p}{2}},
\eea
where we used the estimate \eqref{estb} for $b$ in the last inequality.

Next, we estimate the two terms in the right-hand side of \eqref{po85}. Using Remark \ref{rem:Dcont} for $\mathcal{D}_2^{-1}$, we have:
\bea
\lab{po86}&&\norm{\mathcal{D}_2^{-1}(P_q(F)\c P_lH)}_{\tx{p}{2}}+\norm{\mathcal{D}_2^{-1}(P_q(F)\c \nabb P_lG)}_{\tx{p}{2}}\\
\nn&\les& \norm{P_q(F)\c P_lH}_{\tx{p}{1_+}}+\norm{P_q(F)\c \nabb P_lG}_{\tx{p}{1_+}}\\
\nn&\les& \norm{P_q(F)}_{\tx{\infty}{2}}\norm{P_lH}_{\tx{p}{2_+}}+\norm{P_q(F)}_{\tx{\infty}{2}}\norm{\nabb P_lG}_{\tx{p}{2_+}}\\
\nn&\les& 2^{-\frac{q}{2}}\no(F)\bigg(2^{0_+l}\norm{P_lH}_{\tx{p}{2}}+\normm{\norm{\nabb^2P_lG}^{1-\frac{2}{2_+}}_{\lpt{2}}\norm{\nabb P_lG}^{\frac{2}{2_+}}_{\lpt{2}}}_{L^p(0,1)}\bigg)\\
\nn&\les& 2^{-\frac{q}{2}}\no(F)\Bigg(2^{0_+l}\norm{P_lH}^{1-\frac{2}{p}}_{\tx{\infty}{2}}\norm{P_lH}^{\frac{2}{p}}_{\lh{2}}+\Bigg\|(\norm{\lap P_lG}_{\lpt{2}}\\
\nn&&+\norm{K}_{\lpt{2}}\norm{\nabb P_lG}_{\lpt{2}}+\norm{K}^2_{\lpt{2}}\norm{P_lG}_{\tx{p}{2}})^{1-\frac{2}{2_+}}2^l\norm{P_lG}^{\frac{2}{2_+}}_{\lpt{2}}\Bigg\|_{L^p(0,1)}\Bigg)\\
\nn&\les& 2^{l(\frac{3}{2}-\frac{2}{2_+}-\frac{1}{p})-\frac{q}{2}}\no(F)(1+\norm{K}^{2(1-\frac{2}{2_+})}_{\lh{2}})\ep\\
\nn&\les& 2^{l(\half)_--\frac{q}{2}}\no(F)\ep,
\eea
where we used Lemma \ref{lemma:lbz4} for $\norm{P_q(F)}_{\tx{\infty}{2}}$, the Bochner inequality for tensors \eqref{vbochineq}, \eqref{po83bis} and \eqref{po83} for $G$ and $H$, and the estimate \eqref{estgauss1} for $K$. We also used the fact that once $p<+\infty$ is fixed, we may choose $2_+>2$ such that $1-\frac{2}{2_+}-\frac{1}{p}<0$.

We derive a second estimate for $\norm{\mathcal{D}_2^{-1}(P_q(F)\c P_lH)}_{\tx{p}{2}}$ and $\norm{\mathcal{D}_2^{-1}(P_q(F)\c \nabb P_lG)}_{\tx{p}{2}}$. We have:
\bee
&&\norm{\mathcal{D}_2^{-1}(P_q(F)\c P_lH)}_{\tx{p}{2}}+\norm{\mathcal{D}_2^{-1}(P_q(F)\c \nabb P_lG)}_{\tx{p}{2}}\\
\nn&\les& 2^{-2l}\norm{\mathcal{D}_2^{-1}(P_q(F)\c\lap P_lH)}_{\tx{p}{2}}+\norm{\mathcal{D}_2^{-1}\divb(P_q(F)\c  P_lG)}_{\tx{p}{2}}\\
\nn&&+\norm{\mathcal{D}_2^{-1}(\nabb P_q(F)\c  P_lG)}_{\tx{p}{2}}\\
\nn&\les& 2^{-2l}\norm{\mathcal{D}_2^{-1}\divb(P_q(F)\c\nabb P_lH)}_{\tx{p}{2}}+2^{-2l}\norm{\mathcal{D}_2^{-1}(\nabb P_q(F)\c\nabb P_lH)}_{\tx{p}{2}}\\
&&+\norm{\mathcal{D}_2^{-1}\divb(P_q(F)\c  P_lG)}_{\tx{p}{2}}+\norm{\mathcal{D}_2^{-1}(\nabb P_q(F)\c  P_lG)}_{\tx{p}{2}}\\
\eee
which together with the estimate \eqref{eq:estimdcal-1} and Remark \ref{rem:Dcont} for $\mathcal{D}_2^{-1}$ implies:
\bea
\lab{po87}&&\norm{\mathcal{D}_2^{-1}(P_q(F)\c P_lH)}_{\tx{p}{2}}+\norm{\mathcal{D}_2^{-1}(P_q(F)\c \nabb P_lG)}_{\tx{p}{2}}\\
\nn&\les& 2^{-2l}\norm{P_q(F)\c\nabb P_lH}_{\tx{p}{2}}+2^{-2l}\norm{\nabb P_q(F)\c\nabb P_lH}_{\tx{p}{1_+}}\\
\nn&&+\norm{P_q(F)\c P_lG}_{\tx{p}{2}}+\norm{\nabb P_q(F)\c  P_lG)}_{\tx{p}{1_+}}\\
\nn&\les& 2^{-2l}\norm{\nabb P_q(F)}_{\tx{p}{\infty}}\norm{\nabb P_lH}_{\tx{\infty}{2}}+2^{-2l}\norm{\nabb P_q(F)}_{\tx{p}{2_+}}\norm{\nabb P_lH}_{\tx{\infty}{2}}\\
\nn&&+\norm{P_q(F)}_{\tx{p}{\infty}}\norm{P_lG}_{\tx{\infty}{2}}+\norm{\nabb P_q(F)}_{\tx{p}{2_+}}\norm{P_lG}_{\tx{\infty}{2}}\\
\nn&\les& 2^{-\frac{l}{2}}\normm{\norm{\nabb^2P_q(F)}_{\lpt{2}}^{1-\frac{2}{2_+}}\norm{\nabb P_q(F)}_{\lpt{2}}^{\frac{2}{2_+}}}_{L^p(0,1)}\ep\\
\nn&\les& 2^{-\frac{l}{2}}\Bigg\|(\norm{\lap P_qF}_{\lpt{2}}+\norm{K}_{\lpt{2}}\norm{\nabb P_qF}_{\lpt{2}}+\norm{K}^2_{\lpt{2}}\norm{P_qF}_{\tx{p}{2}})^{1-\frac{2}{2_+}}\\
\nn&&\times 2^q\norm{P_qF}^{\frac{2}{2_+}}_{\lpt{2}}\Bigg\|_{L^p(0,1)}\ep\\
\nn&\les& 2^{q(\half)_--\frac{l}{2}}\no(F)\ep,
\eea
where we used Lemma \ref{lemma:lbz4} for $\norm{P_q(F)}_{\tx{\infty}{2}}$, the finite band property for $P_q$ and $P_l$, the Bochner inequality for tensors \eqref{vbochineq}, \eqref{po83bis} and \eqref{po83} for $G$ and $H$, and the estimate \eqref{estgauss1} for $K$. 

Finally, summing \eqref{po86} for $l\leq q$ and \eqref{po87} for $l>q$ implies:
$$\norm{\mathcal{D}_2^{-1}(F\c H)}_{\tx{p}{2}}+\norm{\mathcal{D}_2^{-1}(F\c \nabb G)}_{\tx{p}{2}}\les\no(F)\ep$$
which together with \eqref{po85} yields the conclusion of Lemma \ref{lemma:po3}.

\subsection{Proof of Lemma \ref{lemma:po4}}\lab{sec:gowinda6}

The analog of Lemma \ref{lemma:lbt6} for $\mathcal{D}_2^{-1}$ implies:
\bee
\norm{\mathcal{D}_2^{-1}(FGH)}_{\tx{\infty}{4_-}}&\les& \norm{FGH}_{\tx{\infty}{\frac{4}{3}}}\\
&\les& \norm{F}_{\tx{\infty}{4}}\norm{G}_{\tx{\infty}{4}}\norm{H}_{\tx{\infty}{4}}\\ 
&\les&\no(F)\no(G)\no(H), 
\eee
which concludes the proof of Lemma \ref{lemma:po4}.

\subsection{Proof of Lemma \ref{lemma:po5}}\lab{sec:gowinda7}

Note first from the curvature bound \eqref{curvflux1} for $\b, \bb, \r, \s$ that $H$ satisfies the following estimate:
\be\lab{po64}
\norm{H}_{\lh{2}}\leq\ep.
\ee
The proof follows the same strategy as the one of Proposition \ref{propK}. However, one has to be more careful since $\b$ and $\bb$ are tensors unlike $K$. In particular, using the estimate \eqref{ad27bis}, 
the $L^2$ boundedness of $P_j$, and  the estimate \eqref{po64} for $H$, we obtain:
\bea\lab{po65}
&&\norm{P_jH}^2_{\tx{\infty}{2}}\\
\nn&\les& \left(\int_0^1\norm{P_jH}_{\lpt{2}}\norm{\ddb_{nL}P_jH}_{\lpt{2}}dt\right)+\norm{P_jH}^2_{\lh{2}}\\
\nn&\les& \norm{P_jH}_{\lh{2}}\norm{P_j\ddb_{nL}H}_{\lh{2}}+\left(\int_0^1\norm{P_jH}_{\lpt{2}}\norm{[\ddb_{nL},P_j]H}_{\lpt{2}}dt\right)+\ep^2\\
\nn&\les& \ep\norm{P_j\ddb_{nL}H}_{\lh{2}}+\left(\int_0^1\norm{P_jH}_{\lpt{2}}\norm{[\ddb_{nL},P_j]H}_{\lpt{2}}dt\right)+\ep^2.
\eea
Now, the Bianchi identities \eqref{bianc1}, \eqref{bianc2}, \eqref{bianc4} and \eqref{bianc6} for $\ddb_L(\b), L(\rho), L(\s)$ and $\ddb_L(\bb)$ have the following structure:
$$\ddb_LH=(\divb(\a), \divb(\b), \curlb(\b), \nabb\r, \nabb\s)+F\c(\a, \b, \r, \s, \bb)$$
where in view of the estimates \eqref{estn}-\eqref{estzeta}, $F$ satisfies $\no(F)\les\ep$. Thus, using the finite band property and the weak Bernstein inequality for $P_j$, we obtain for $\ddb_LH$ the following estimate:
\bea
\lab{po66}\norm{P_j\ddb_LH}_{\lh{2}}&\les& 2^j\norm{(\a, \b, \r, \s, \bb)}_{\lh{2}}+2^{\frac{j}{2}}\norm{F\c(\a, \b, \r, \s, \bb)}_{\tx{2}{\frac{4}{3}}}\\
\nn&\les& 2^j\norm{(\a, \b, \r, \s, \bb)}_{\lh{2}}+2^{\frac{j}{2}}\norm{F\c(\a, \b, \r, \s, \bb)}_{\tx{2}{\frac{4}{3}}}\\
\nn&\les& 2^j\ep+2^{\frac{j}{2}}\no(F)\ep\\
\nn&\les&2^j\ep,
\eea
where we used the curvature bound \eqref{curvflux1} for $\a, \b, \r, \s$ and $\bb$. \eqref{po65} and \eqref{po66} imply:
$$\norm{P_jH}^2_{\tx{\infty}{2}}\les
\left(\int_0^1\norm{P_jH}_{\lpt{2}}\norm{[\ddb_{nL},P_j]H}_{\lpt{2}}dt\right)+2^j\ep^2,$$
which yields:
\be\lab{po67}
\norm{P_jH}_{\tx{\infty}{2}}\les\norm{[\ddb_{nL},P_j]H}_{\tx{1}{2}}+2^{\frac{j}{2}}\ep.
\ee

We now evaluate the right-hand side of \eqref{po67}. Again, let us say that the difference with the proof of Proposition \ref{propK} is the fact that $H$ is a tensor unlike $K$. Using the definition \eqref{eq:LP} of $P_j$, we have:
\begin{equation}\label{bis:ad30}
[\ddb_{nL},P_j]H=\int_0^\infty m_j(\tau)V(\tau) d\tau,
\end{equation}
where $V(\tau)$ satisfies:
\be\lab{bis:ad29}
(\partial_{\tau}-\lap)V(\tau)=[\ddb_{nL},\lap]U(\tau)H,\,V(0)=0.
\ee
\eqref{bis:ad30} yields:
\begin{equation}\label{bis:ad36bis}
\norm{[\ddb_{nL},P_j]H}_{\tx{1}{2}}\les \normm{\int_0^\infty m_j(\tau)\norm{V(\tau)}_{\lpt{2}} d\tau}^2_{L^1(0,1)}.
\end{equation}

In view of \eqref{po67} and \eqref{bis:ad36bis}, we have to estimate $\norm{V(\tau)}_{\lpt{2}}$. Let $a, p$ real numbers satisfying:
\begin{equation}\label{bis:ad38bis}
\begin{array}{l}
\ds 0<a<\frac{1}{2},\,2<p<+\infty,\textrm{ such that }p<\frac{4}{2-a}.
\end{array}
\end{equation}
The energy estimate \eqref{eq:l2heat1bis} implies:
\begin{equation}\label{bis:ad46}
\begin{array}{ll}
& \ds\norm{\La^{-a}V(\tau)}^2_{\lpt{2}}+\int_0^\tau\norm{\nabb\La^{-a}V(\tau')}^2_{\lpt{2}}d\tau'\\
\ds\lesssim &\ds\int_0^\tau\int_{\ptu} \La^{-2a}V(\tau')[\ddb{nL},\lap]U(\tau')H\dmt d\tau'.
\end{array}
\end{equation}

We need to estimate the commutator term $[\ddb{nL},\lap]U$. Using twice the commutator formula \eqref{comm5}, we have:
\be\lab{midi1}
[\ddb{nL},\lap]U=F\nabb^2U+G\nabb U+\nabb(GU)
\ee
where the tensors $F$ and $G$ are given by $F=n\chi$ and $G=n\chi\kep+n{}^*\b$. Using the curvature bound \eqref{curvflux1} for $\b$ and the bound \eqref{estn}-\eqref{esthch} for $n, \kep$ and $\chi$, we obtain the following bound for $F$ and $G$:
\be\lab{midi2}
\norm{\nabb F}_{\lh{2}}+\norm{G}_{\lh{2}}\les \ep
\ee

Let $p$ defined in \eqref{bis:ad38bis}, and let $p'$ such that $\frac{1}{p}+\frac{1}{p'}=\frac{1}{2}$. 
Using the commutator formula \eqref{midi1}, and integrating by parts the terms $\nabb^2U(\tau)H$ and $\nabb(GU)$ yields:
\begin{equation}\label{bis:ad47}
\begin{array}{l}
\ds\int_0^\tau\int_{\ptu} \La^{-2a}V(\tau')[\ddb_{nL},\lap]U(\tau')H\dmt d\tau'\\
\ds\lesssim (\norm{\nabb F}_{\lpt{2}}+\norm{G}_{\lpt{2}})\int_0^\tau\norm{\nabb U(\tau')}_{\lpt{p}}\norm{\La^{-2a}V(\tau')}_{\lpt{p'}}d\tau'\\
+\ds \norm{F}_{\lpt{p'}}\int_0^\tau\norm{\nabb U(\tau')}_{\lpt{p}}\norm{\nabb\La^{-2a}V(\tau')}_{\lpt{2}}d\tau'\\
+\ds \norm{G}_{\lpt{2}}\int_0^\tau\norm{U(\tau')}_{\lpt{\infty}}\norm{\nabb\La^{-2a}V(\tau')}_{\lpt{2}}d\tau'\\
\ds\lesssim (\norm{\nabb F}_{\lpt{2}}+\norm{G}_{\lpt{2}})\int_0^\tau\norm{\nabb U(\tau')}_{\lpt{p}}\norm{\nabb\La^{-2a}V(\tau')}_{\lpt{2}}d\tau'
\end{array}
\end{equation}
where we used the Sobolev embeddings \eqref{eq:GNirenberg} and \eqref{sobinftyptu} in the last inequality. The Gagliardo-Nirenberg inequality \eqref{eq:GNirenberg}, the properties \eqref{La1} and \eqref{interpolLa} of $\Lambda$, and the Bochner inequality \eqref{vbochineq} for tensors yield:
\bea\label{bis:ad48}
&&\int_0^\tau\norm{\nabb U(\tau')}_{\lpt{p}}\norm{\nabb\La^{-2a}V(\tau')}_{\lpt{2}}d\tau'\\
\nn&\lesssim& 
\int_0^\tau\norm{\nabb U(\tau')}^{\frac{2}{p}}_{\lpt{2}}\norm{\nabb^2 U(\tau')}^{1-\frac{2}{p}}_{\lpt{2}}\norm{\La^{-a}V(\tau')}^{a}_{\lpt{p'}}\norm{\nabb\La^{-a}V(\tau')}^{1-a}_{\lpt{p'}}d\tau'\\
\nn&\lesssim& \int_0^\tau\norm{\nabb U(\tau')}^{\frac{2}{p}}_{\lpt{2}}\big(\norm{\lap U(\tau')}_{\lpt{2}}+\norm{K}_{\lpt{2}}\norm{\nabb U(\tau')}_{\lpt{2}}\\
\nn&& +\norm{K}^2_{\lpt{2}}\norm{U(\tau')}_{\lpt{2}}\big)^{1-\frac{2}{p}}\norm{\La^{-a}V(\tau')}^{a}_{\lpt{p'}}\norm{\nabb\La^{-a}V(\tau')}^{1-a}_{\lpt{p'}}d\tau'\\
\nn&\lesssim& \Bigg(\int_0^\tau\norm{\nabb U(\tau')}^{2}_{\lpt{2}}d\tau'
+ \int_0^\tau \tau'\norm{\lap U(\tau')}^{2}_{\lpt{2}}d\tau'\\
\nn&&+\norm{K}_{\lpt{2}}^2\int_0^\tau \tau'\norm{\nabb U(\tau')}^{2}_{\lpt{2}}d\tau'\Bigg)^{\frac{1}{2}}\\
\nn&&\times\left(\frac{1}{2}\int_0^\tau\norm{\nabb\La^{-a}V(\tau')}^2_{\lpt{2}}d\tau'
+\int_0^\tau{\tau'}^{-\frac{2(p-2)}{ap}}\norm{\La^{-a}V(\tau')}^2_{\lpt{2}}d\tau'\right)^{\frac{1}{2}}
\eea
which together with the estimates for the heat flow \eqref{eq:l2heat1}, \eqref{eq:l2heat2} and \eqref{eq:l2lambda1}, implies:
\begin{equation}\label{bis:ad49}
\begin{array}{l}
\ds\int_0^\tau\norm{\nabb U(\tau')}_{\lpt{p}}\norm{\nabb\La^{-2a}V(\tau')}_{\lpt{2}}d\tau'\\
\ds\lesssim (\norm{H}_{\lpt{2}}+\norm{K}_{\lpt{2}}\norm{\La^{-1}H}_{\lpt{2}})\\
\ds\left(\int_0^\tau\norm{\nabb\La^{-a}V(\tau')}^2_{\lpt{2}}d\tau'+\int_0^\tau{\tau'}^{-\frac{2(p-2)}{ap}}\norm{\nabb\La^{-a}V(\tau')}^2_{\lpt{2}}d\tau'\right)^{\frac{1}{2}}.
\end{array}
\end{equation}

Finally, the choice of $p$ \eqref{bis:ad38bis}, \eqref{bis:ad46}, \eqref{bis:ad47} and \eqref{bis:ad49} implies:
\begin{equation}\label{bis:ad50}
\begin{array}{l}
\ds\norm{\La^{-a}V(\tau)}^2_{\lpt{2}}+\int_0^\tau\norm{\nabb\La^{-a}V(\tau')}^2_{\lpt{2}}d\tau'\\
\ds\lesssim (\norm{\nabb F}_{\lpt{2}}+\norm{G}_{\lpt{2}})(\norm{H}_{\lpt{2}}+\norm{K}_{\lpt{2}}\norm{\La^{-1}H}_{\lpt{2}}).
\end{array}
\end{equation}
Using the interpolation inequality \eqref{interpolLa}, we obtain:
\begin{equation}\label{bis:ad51}
\begin{array}{l}
\ds\int_0^{+\infty}\norm{V(\tau)}^{\frac{2}{a}}_{\lpt{2}}d\tau\lesssim \int_0^\tau \norm{\La^{-a}V(\tau')}^{\frac{2(1-a)}{a}}_{\lpt{2}}\norm{\nabb\La^{-a}V(\tau')}^2_{\lpt{2}}d\tau'
\\
\ds\lesssim (\norm{\nabb F}_{\lpt{2}}+\norm{G}_{\lpt{2}})(\norm{H}_{\lpt{2}}+\norm{K}_{\lpt{2}}\norm{\La^{-1}H}_{\lpt{2}}),
\end{array}
\end{equation}
which together with the estimate \eqref{po64} for $H$ and the estimate \eqref{midi2} for $F$ and $G$ yields:
\begin{equation}\label{bis:ad52}
\begin{array}{l}
\ds\normm{\int_0^{+\infty}m_j(\tau)\norm{V(\tau)}_{\lpt{2}}d\tau}_{L^1(0,1)}\lesssim 2^{ja}\normm{\left(\int_0^{+\infty}\norm{V(\tau)}^{\frac{2}{a}}_{\lpt{2}}d\tau\right)^{\frac{a}{2}}}_{L^1(0,1)}\\
\ds\lesssim 2^{ja}(\norm{\nabb F}_{\lh{2}}+\norm{G}_{\lh{2}})(\norm{H}_{\lh{2}}+\norm{K}_{\lh{2}}\norm{\La^{-1}H}_{\tx{\infty}{2}})\\
\lesssim 2^{ja}\ep(1+\norm{\La^{-1}H}_{\tx{\infty}{2}}).
\end{array}
\end{equation}
Now, \eqref{po67}, \eqref{bis:ad36bis} and \eqref{bis:ad52} imply:
\bea
\lab{bis:ad53}\norm{P_jH}_{\tx{\infty}{2}}&\les& 2^{\frac{j}{2}}\ep+2^{ja}\ep(1+\norm{\La^{-1}H}_{\tx{\infty}{2}})\\
\nn&\les& 2^{\frac{j}{2}}\ep+2^{\frac{j}{2}}\ep\norm{\La^{-1}H}_{\tx{\infty}{2}},
\eea
where we used the choice of $a$ \eqref{bis:ad38bis} in the last inequality. Finally, from the properties of $\La$ and $P_j$, we have:
\bee
\norm{\La^{-1}H}_{\tx{\infty}{2}}&\les& \sum_j\norm{P_j\La^{-1}H}_{\tx{\infty}{2}}\\
&\les& \sum_j2^{-j}\norm{P_jH}_{\tx{\infty}{2}}\\
&\les& \sup_j\sum_j2^{-\frac{j}{2}}\norm{P_jH}_{\tx{\infty}{2}}
\eee
which together with \eqref{bis:ad53} implies:
$$\norm{P_jH}_{\tx{\infty}{2}}\les 2^{\frac{j}{2}}\ep.$$
This concludes the proof of Lemma \ref{lemma:po5}.

\subsection{Proof of Lemma \ref{lemma:ad2}}\lab{sec:ad2}

Using the $L^\infty$ estimate \eqref{prop:linftysitptu}, we have:
\bea\lab{ad9}
\nn\norm{Q_{>1}N}_{L^\infty}&\les& \norm{Q_{>1}N}_{\lsit{\infty}{2}}+\norm{\nabla Q_{>1}N}_{\lsit{\infty}{2}}+\norm{\nabb\nabla Q_{>1}N}_{\lsit{\infty}{2}}\\
&\les& \norm{\nabla N}_{\lsit{\infty}{2}}+\norm{\nabb\nabla N}_{\lsit{\infty}{2}}+\norm{\nabb\nabla Q_{\leq 1}N}_{\lsit{\infty}{2}},
\eea
where we used in the last inequality the finite band property for $Q_{>1}$, and the decomposition $N=Q_{\leq 1}(N)+Q_{>1}(N)$.

We now evaluate the various terms in the right-hand side of \eqref{ad9}. Since $N=\half(L-\lb)$, the Ricci equation \eqref{ricciform} imply:
\be\lab{ad10}
\nabla_AN=\th_{AB}e_B,\,\nabla_NN=-b^{-1}\nabb b.
\ee
\eqref{ad10} implies:
\bea
\lab{ad11}&&\norm{\nabla N}_{\lsit{\infty}{2}}+\norm{\nabb\nabla N}_{\lsit{\infty}{2}}\\
\nn&\les& \norm{\th}_{\lsit{\infty}{2}}+\norm{\nabb\th}_{\lsit{\infty}{2}}+\norm{b^{1}\nabb b}_{\lsit{\infty}{2}}+\norm{b^{-1}\nabb b}^2_{\lsit{\infty}{4}}+\norm{b^{-1}\nabb^2 b}_{\lsit{\infty}{2}}.
\eea
Furthermore, the Bochner inequality \eqref{prop:bochsit} and the finite band property for $Q_{\leq 1}$ imply:
\bea
\lab{ad12}\norm{\nabb\nabla Q_{\leq 1}N}_{\lsit{\infty}{2}}&\les&\norm{\nabla^2Q_{\leq 1}N}_{\lsit{\infty}{2}}\\
\nn&\les& \norm{\Delta Q_{\leq 1}N}_{\lsit{\infty}{2}}\\
\nn&\les& \norm{\nabla N}_{\lsit{\infty}{2}}.
\eea
Now, \eqref{ad9}, \eqref{ad11} and \eqref{ad12} yield:
\bea
\lab{ad12bis}\norm{Q_{>1}N}_{L^\infty}&\les& \norm{\th}_{\lsit{\infty}{2}}+\norm{\nabb\th}_{\lsit{\infty}{2}}+\norm{b^{1}\nabb b}_{\lsit{\infty}{2}}\\
\nn&&+\norm{b^{-1}\nabb b}^2_{\lsit{\infty}{4}}+\norm{b^{-1}\nabb^2 b}_{\lsit{\infty}{2}}.
\eea

In view of \eqref{ad12bis}, we need to estimate $b$ and $\th$ on $\Sit$. So far, we only proved regularity 
estimates on $\H_u$. In order to transfer them to $\Sit$, we consider the structure equation for the foliation generated by $u$ on $\Sit$ (see \cite{ChKl} p. 56):
\be\lab{struct}
\left\{\begin{array}{l}
b^{-1}\lap b=-\nabla_N\trt-|\th|^2+R_{NN},\\[1mm]
\nabb^B\hth_{AB}=\frac{1}{2}\nabb_A\trt+R_{NA}.
\end{array}\right.
\ee
Recall from the definition of $\th$ \eqref{def:theta} that $\trt$ is given by:
$$\trt=\trc-\d$$
where we used the fact that the time foliation is maximal \eqref{maxfoliation}. In view of the estimate \eqref{estk} for $\d$ and the estimate \eqref{esttrc} for $\trc$, we obtain:
\be\lab{ad13}
\norm{\trt}_{L^\infty_u\lpt{2}}+\norm{\nabb\trt}_{L^\infty_u\lpt{2}}+\norm{\nabla_N\trt}_{L^\infty_u\lpt{2}}\les\ep.
\ee
Furthermore, using the definition of $\th$ \eqref{def:theta} and the Sobolev embedding \eqref{sobineq1}, we have:
\be\lab{ad14}
\norm{|\th|^2}_{L^\infty_u\lpt{2}}\les \norm{\th}^2_{L^\infty_u\lpt{4}}\les \no(\th)^2\les\no(\chi)^2+\no(\eta)^2\les \ep^2,
\ee
where we used the estimate \eqref{estk} for $\eta$ and the estimates \eqref{esthch} \eqref{esttrc} for $\chi$. Also, using the Sobolev embedding \eqref{sobineq2}, we have:
\be\lab{ad15}
\norm{\nabb b}_{L^\infty_u\lpt{2}}+\norm{\nabb b}_{L^\infty_u\lpt{4}}\les \no(\nabb b)\les \ep,
\ee
where we used the estimate \eqref{estb} for $b$.

Next, we estimate $\norm{\nabb^2b}_{\lsit{\infty}{2}}$. In view of the Bochner inequality \eqref{eq:Bochconseqbis}, we have:
\bea
\lab{ad16}\norm{\nabb^2b}_{\lsit{\infty}{2}}&\les& \norm{\lap b}_{\lsit{\infty}{2}}+\norm{\nabb b}_{\lsit{\infty}{2}}\\
\nn&\les& \norm{b}_{L^\infty}(\norm{\nabla_N\trt}_{\lsit{\infty}{2}}+\norm{|\th|^2}_{\lsit{\infty}{2}}+\norm{R_{NN}}_{\lsit{\infty}{2}})\\
\nn&&+\norm{\nabb b}_{\lsit{\infty}{2}}\\
\nn&\les&\ep,
\eea
where we used in the last inequality the curvature bound \eqref{curvflux1} for $R_{NN}$, the estimate \eqref{estb} for $b$, and the estimates \eqref{ad13}-\eqref{ad15} for $\th$ and $b$.

Next, we estimate $\norm{\hth}_{L^\infty_u\lpt{2}}$. in view of the Hodge estimate \eqref{eq:estimdcal-1}, we have:
\bea
\lab{ad17}\norm{\nabb\hth}_{\lsit{\infty}{2}}&\les& \norm{\nabb\trt}_{\lsit{\infty}{2}}+\norm{R_{AN}}_{\lsit{\infty}{2}}\\
\nn&\les&\ep,
\eea
where we used in the last inequality the curvature bound \eqref{curvflux1} for $R_{AN}$ and the estimates \eqref{ad13} and \eqref{ad15} for $\th$ and $b$.

Finally, \eqref{ad12bis}, \eqref{ad13}, \eqref{ad15}, \eqref{ad16} and \eqref{ad17} yield \eqref{ad2}. This concludes the proof of Lemma \ref{lemma:ad2}.

\subsection{Proof of Lemma \ref{lemma:ad3}}\lab{sec:ad3}

We estimate the following quantity:
\bea
\lab{ad18}&&\norm{\nabb Q_{\leq 1}(N')}_{\lsit{\infty}{2}}+\norm{\nabla\nabb Q_{\leq 1}(N')}_{\lsit{\infty}{2}}\\
\nn&\les& \norm{\nabla Q_{\leq 1}(N')}_{\lsit{\infty}{2}}+\norm{\nabb\nabla Q_{\leq 1}(N')}_{\lsit{\infty}{2}}+
\norm{[\nabb,\nabb_N]Q_{\leq 1}(N')}_{\lsit{\infty}{2}}\\
\nn&\les& \norm{\nabla N'}_{\lsit{\infty}{2}}+\norm{\nabla^2 Q_{\leq 1}(N')}_{\lsit{\infty}{2}}+
\norm{[\nabb,\nabb_N]Q_{\leq 1}(N')}_{\lsit{\infty}{2}}\\
\nn&\les& \norm{\nabla N'}_{\lsit{\infty}{2}}+\norm{\Delta Q_{\leq 1}(N')}_{\lsit{\infty}{2}}+
\norm{[\nabb,\nabb_N]Q_{\leq 1}(N')}_{\lsit{\infty}{2}}\\
\nn&\les& \norm{\nabla N'}_{\lsit{\infty}{2}}+\norm{[\nabb,\nabb_N]Q_{\leq 1}(N')}_{\lsit{\infty}{2}},
\eea
where we used several times the finite band property for $Q_{\leq 1}$ and the Bochner inequality \eqref{prop:bochsit}. Now, for any tensor $F$, the following estimate is a immediate consequence of the proof of \eqref{prop:linftysitptu} (see \eqref{eul3}):
\be\lab{ad19}
\norm{[\nabb,\nabla_N]F}_{\lsit{\infty}{2}}\les \ep\norm{\nabb\nabla F}_{\lsit{\infty}{2}}+\ep\norm{F}_{L^\infty}
\ee
Using \eqref{ad19} with $F=Q_{\leq 1}(N')$ yields:
\bee
\norm{[\nabb,\nabla_N]Q_{\leq 1}(N')}_{\lsit{\infty}{2}}&\les& \ep\norm{\nabb\nabla Q_{\leq 1}(N')}_{\lsit{\infty}{2}}+\ep\norm{Q_{\leq 1}(N')}_{L^\infty}\\
&\les& \ep\norm{\nabla^2Q_{\leq 1}(N')}_{\lsit{\infty}{2}}+\ep\norm{N'}_{L^\infty}\\
&\les& \ep\norm{\Delta Q_{\leq 1}(N')}_{\lsit{\infty}{2}}+\ep\norm{N'}_{L^\infty}\\
&\les& \ep\norm{\nabla N'}_{\lsit{\infty}{2}}+\ep\norm{N'}_{L^\infty},
\eee
where we used the $L^\infty$ boundedness of $Q_{\leq 1}$, the Bochner inequality \eqref{prop:bochsit}, and the finite band property for $Q_{\leq 1}$.

Note from the proof of Lemma \ref{lemma:ad2} (see \eqref{ad11} the following estimate:
$$\norm{\nabla N'}_{\lsit{\infty}{2}}\les \ep.$$
Together with \eqref{ad18} and \eqref{ad19}, this implies:
\bea\lab{ad20}
\nn\norm{\nabb Q_{\leq 1}(N')}_{\lsit{\infty}{2}}+\norm{\nabla\nabb Q_{\leq 1}(N')}_{\lsit{\infty}{2}}&\les& \norm{\nabla N'}_{\lsit{\infty}{2}}+\ep\norm{N'}_{L^\infty}\\
&\les&\ep.
\eea

We will prove for any tensor vectorfield $F$ the following non sharp estimate:
\be\lab{ad21}
\norm{F}_{\BB^0}\les \norm{F}_{\lsit{\infty}{2}}+\norm{\nabla F}_{\lsit{\infty}{2}}.
\ee
\eqref{ad20} and \eqref{ad21} immediately yield \eqref{ad3}. 

In order to conclude the proof of Lemma \ref{lemma:ad3}, it remains to prove \eqref{ad21}. We estimate 
$\norm{P_jF}^2_{\tx{\infty}{2}}$. In view of \eqref{hehehehe2}, we have:
\bea\lab{ad22}
&&\norm{P_jF}^2_{L^\infty_{t,u}\lpt{2}}\\
\nn&\les& \left(\int_u\norm{P_jF}_{\lpt{2}}\norm{\ddb_{bN}P_jF}_{\lpt{2}}du\right)+\norm{\trt}_{L^\infty_u\lpt{4}}\norm{P_jF}_{\lsit{\infty}{2}}\norm{P_jF}_{L^2_u\lpt{4}}\\
\nn&\les& \norm{P_jF}_{\lsit{\infty}{2}}\norm{P_j\ddb_{bN}F}_{\lsit{\infty}{2}}+\left(\int_u\norm{P_jF}_{\lpt{2}}\norm{[\ddb_{bN},P_j]F}_{\lpt{2}}du\right)\\
\nn&&+2^{\frac{j}{2}}\ep\norm{P_jF}^2_{\lsit{\infty}{2}}\\
\nn&\les& 2^{-j}\norm{\nabla F}^2_{\lsit{\infty}{2}}+\norm{P_jF}_{L^4_u\lpt{2}}\norm{[\ddb_{bN},P_j]F}_{L^{\frac{4}{3}}_u\lpt{2}}\\
\nn&\les& 2^{-j}\norm{\nabla F}^2_{\lsit{\infty}{2}}+\norm{P_jF}_{L^\infty_u\lpt{2}}^{\half}\norm{P_jF}_{\lsit{\infty}{2}}^{\half}\norm{[\ddb_{bN},P_j]F}_{L^{\frac{4}{3}}_u\lpt{2}}\\
\nn&\les& 2^{-j}\norm{\nabla F}^2_{\lsit{\infty}{2}}+2^{-\frac{j}{2}}\norm{P_jF}_{L^\infty_u\lpt{2}}^{\half}\norm{\nabb F}_{\lsit{\infty}{2}}^{\half}\norm{[\ddb_{bN},P_j]F}_{L^{\frac{4}{3}}_u\lpt{2}},
\eea
where we used the estimate \eqref{ad13} for $\trt$, the Bernstein inequality for $P_j$, and the finite band property and the $L^2$ boundedness of $P_j$. \eqref{ad22} implies:
\be\lab{ad23}
\norm{P_jF}_{\tx{\infty}{2}}\les 2^{-\frac{j}{2}}\norm{\nabla F}_{\lsit{\infty}{2}}
+2^{-\frac{j}{4}}\norm{[\ddb_{bN},P_j]F}_{L^{\frac{4}{3}}_u\lpt{2}}.
\ee
Now, \eqref{ad23} and the commutator estimate \eqref{bonobo} imply:
$$\norm{P_jF}_{\tx{\infty}{2}}\les (2^{-\frac{j}{2}}+2^{-\frac{j}{4}+a})\norm{\nabla F}_{\tx{\infty}{2}},$$
where $0<a<\frac{1}{4}$. In view of the definition of $\BB^0$, this yields \eqref{ad21}. This concludes the proof of Lemma \ref{lemma:ad3}.

\subsection{Proof of Lemma \ref{lemma:ad4}}\lab{sec:ad4}

In view of the Ricci equations \eqref{ricciform}, we have:
\bea
\lab{aad38}\norm{\dd_L(N')}_{\lh{2}}&\les& \norm{\chi'}_{\lh{2}}+\norm{\chb'}_{\lh{2}}+\norm{\kep'}_{\lh{2}}+\norm{\d'}_{\lh{2}}\\
\nn&&+\norm{\z'}_{\lh{2}}+\norm{n^{-1}\nabla_{N'}n}_{\lh{2}}+\norm{\xib'}_{\lh{2}},
\eea
where $\chi', \chb', \d', \z', \xib'$ are the Ricci coefficients associated to $u(.,\o')$. We only estimate $\z'$ since it is the worst term in \eqref{aad38}. In view of the computations \eqref{du} and \eqref{dusit}, we have for any scalar function $f$:
$$\lb\left(\int_{\ptu}f\dmt\right)=\int_{\ptu}(\lb(f)+\trchb f)\dmt.$$
Together with the coarea formula \eqref{coarea} and the fact that $\z'$ vanishes at infinity, we obtain:
\bea
\lab{aad39}\norm{\z'}^2_{\lh{2}}&\les& \int_u\int_{\H_u}(\dd_{\lb}(\z')\c\z'+|\z'|^2(\trchb+b^{-1}\dd_{\lb}(b)))d\H_udu\\
\nn&\les& \left|\int_u\int_{\H_u}\dd_{\lb}(\z')\c\z'd\H_udu\right|+\norm{\z'}_{L^4(\mathcal{M})}^2(\norm{\trchb}_{L^2(\mathcal{M})}+\norm{b^{-1}\dd_{\lb}(b)}_{L^2(\mathcal{M})})\\
\nn&\les& \left|\int_u\int_{\H_u}\dd_{\lb}(\z')\c\z'd\H_udu\right|+\ep,
\eea
where we used the estimates \eqref{estb}-\eqref{estzeta} for $\z', b$ and $\trchb$.

Next, we estimate the right-hand side of \eqref{aad39}. Decomposing $\lb$ on the frame $L', \lb', e_1', e_2'$, we have:
\bea
\lab{aad40}&&\left|\int_u\int_{\H_u}\dd_{\lb}(\z')\c \z'd\H_udu\right|\\
\nn&\les& \left|\int_u\int_{\H_u}g(\lb,\lb')\dd_{L'}(\z')\c \z'd\H_udu\right|+\left|\int_u\int_{\H_u}g(\lb,e_A')\dd_{e_A'}(\z')\c \z'd\H_udu\right|\\
\nn&&+\left|\int_u\int_{\H_u}g(\lb,L')\dd_{\lb'}(\z')\c \z'd\H_udu\right|\\
\nn&\les& (\norm{\ddb_{L'}(\z')}_{L^2(\mathcal{M})}+\norm{\nabb\z'}_{L^2(\mathcal{M})})\norm{\z'}_{L^2(\mathcal{M})}+\left|\int_u\int_{\H_u}g(\lb,L')\dd_{\lb'}(\z')\c \z'd\H_udu\right|\\
\nn&\les&\ep^2+\left|\int_u\int_{\H_u}g(\lb,L')\dd_{\lb'}(\z')\c \z'd\H_udu\right|,
\eea
where we used the estimate \eqref{estzeta} for $\z'$ in the last inequality. 

Now, we estimate the right-hand side of \eqref{aad40}. Using the Littlewood-Paley decomposition, we have:
\bea
\lab{aad41}&&\left|\int_u\int_{\H_u}(P'_j\dd_{\lb'}(\z')\c P'_j(g(\lb,L')\z')d\H_udu\right|\\
\nn&\les& \norm{P'_j\dd_{\lb'}(\z')}_{L^2(\mathcal{M})}\norm{P'_j(g(\lb,L')\z')}_{L^2(\mathcal{M})}\\
\nn&\les& \ep 2^{-j}\norm{\nabb'(g(\lb,L')\z')}_{L^2(\mathcal{M})}\\
\nn&\les&\ep 2^{-j}(\norm{\nabb'\z'}_{L^2(\mathcal{M})}+\norm{\dd_{e'_A}(\lb)\z')}_{L^2(\mathcal{M})}+\norm{\dd_{e'_A}(L')\z')}_{L^2(\mathcal{M})})\\
\nn&\les&\ep 2^{-j}(\norm{\nabb'\z'}_{L^2(\mathcal{M})}+\norm{\z'}_{L^4(\mathcal{M})}(\norm{\dd\lb}_{L^4(\mathcal{M})}+\norm{\dd L'}_{L^4(\mathcal{M})}))\\
\nn&\les& \ep^2 2^{-j}
\eea
where we used the estimate \eqref{estlbzeta} for $\dd_{\lb'}(\z')$, the finite band property for $P'_j$, the estimate \eqref{estzeta} for $\z'$, and the Ricci equations \eqref{ricciform} together with the estimates \eqref{estn}-\eqref{estzeta} of the Ricci coefficients to estimate $\dd \lb$ and $\dd L'$.

Finally, summing with respect to $j$ in \eqref{aad41}, together with \eqref{aad39} and \eqref{aad40} yields:
$$\norm{\z'}_{\lh{2}}\les \ep.$$
The estimates of the other Ricci coefficients in the right-hand side of \eqref{aad38} are easier, and we obtain in the end:
$$\norm{\dd_L(N')}_{\lh{2}}\les\ep,$$
which concludes the proof of Lemma \ref{lemma:ad4}.

\section{Appendix to section \ref{sec:depnormonomega}}

\subsection{Proof of Lemma \ref{lemma:xx2}}\lab{sec:xx2}

We have constructed a global coordinate system on $\Sigma_t$ in section \ref{sec:globalcoordsit}. We will need another global coordinate system. Let $\o\in\S$. Let $\Phi_{t,\o}:\Sigma_t\rightarrow \R^3$ defined by:
\begin{equation}\label{gl1}
\Phi_{t,\o}(t,x):=u(t,x,\o)\o+\po u(t,x,\o).
\end{equation}
Then we claim that $\Phi_{t,\o}$ is a global $C^1$ diffeomorphism from $\Sigma_t$ to $\R^3$ and therefore provides a global coordinate system on $\Sigma_t$. The proof has been done in \cite{param1} for the particular case $t=0$ of a global coordinate system on $\Sigma_0$. The proof for $\Sigma_t$ is completely analogous an we refer the interested reader to Proposition 2.9 in \cite{param1}. The proof also provides the following bound for $d\Phi_{t,\o}^{-1}$:
\be\label{lxx2:1}
\norm{d\Phi^{-1}_{t,\o}}_{L^\infty}\lesssim \ep.
\ee
Recall from \eqref{po5} that we have $\norm{\dd_L(\po N)}_{\xt{\infty}{2}}\lesssim \ep$. This yields:
$$\norm{\dd_L\left(g(\po N,\po N)-I\right)}_{\xt{\infty}{2}}\lesssim\ep$$
which together with the estimate for transport equation \eqref{estimtransport1} and the corresponding estimate at initial time (see \cite{param1}) yields:
\be\label{lxx2:1ter}
\norm{g(\po N,\po N)-I}_{L^\infty}\lesssim\ep.
\ee
Consider the global coordinate system on $\Sigma_t$ provided by $\Phi^{-1}_{t,\o}(u,y')$. Then, for any scalar function $f$ on $\Sigma_t$, one easily derives the following formulas:
\be\label{lxx2:1bis}
\frac{\partial f}{\partial u}=g\left(N+O(\ep),\nabla f\right)\textrm{ and }\frac{\partial f}{\partial y'}=g\left(\po N+O(\ep),\nabla f\right),
\ee
where we used the fact that $g(N,\po N)=0$, $\nabla u(t,x,\o)=b^{-1}N$, $\nabla \po u(t,x,\o)=-b^{-2}\po b N+b^{-1}\po N$, $\norm{b-1}_{L^\infty}\lesssim\ep$, $\norm{\po b}_{L^\infty}\lesssim\ep$ and \eqref{lxx2:1ter}.

Finally, $u$ being fixed, $\Phi^{-1}_{t,\o}(u,y')$ provided a coordinate system on $\ptu$ such that the following estimate holds for the induced metric $\gamma$ in the coordinate system:
\be\lab{lxx2:2}
|\ga_{AB}(p)\xi^A\xi^B-|\xi|^2|\lesssim\ep |\xi|^2, \qquad \mbox{uniformly for  all }
\,\, p\in \R^2.
\ee

We evaluate $\norm{F}_{L^2(\H_{u})}$. Using the global coordinate system on $\ptu$ provided by $\Phi^{-1}_{t,\o}(u,y')$, we have:
\bea\lab{lxx2:3}
\norm{F}^2_{L^2(\H_{u})}&=&\int_0^1\int |F(\Phi^{-1}_{t,\o'}(u,y'))|^2\sqrt{\gamma}dy'dt\\
\nn&\lesssim &\int_0^1\int |F(\Phi^{-1}_{t,u,\o'}(y'))|^2dy'dt,
\eea
where we used \eqref{lxx2:2} in the last inequality. Let $(t,x_t)$ a point on $\Sigma_t$. Let $0\leq \sigma\leq 1$ parametrize the arc on $\S$ joining $\o$ and $\o'$, and let $\o_\sigma\in \S$ corresponding to $\sigma$. Let $u_\sigma=u(t,x_t,\o_\sigma)$ and $\po u_\sigma=\po u(t,x_t,\o_\sigma)$. Let $\rho$ a positive smooth bounded function on $\R$ vanishing in the neighborhood of 0. We consider the following integral:
\be\lab{lxx2:4}
I(\sigma)=\int_0^1\int |F(\Phi^{-1}_{t,,\o_\sigma}(u=u_\sigma,y'))|^2\rho(\po u-\po u_\sigma)dy'dt.
\ee
We have:
\be\lab{lxx2:5}
I(0)=\int_0^1\int |F(\Phi^{-1}_{t,\o'}(u=u_0,y'))|^2dy'dt\gtrsim \norm{F\sqrt{\rho(\po u-\po u_\sigma)}}^2_{L^2(\H_{u'=u_0})}
\ee
where we used \eqref{lxx2:2} and \eqref{lxx2:3}, and where $u'=u(.,.,\o')$. We also have:
\bea\lab{lxx2:6}
I(1)&=&\int_0^1\int |F(\Phi^{-1}_{t,\o}(u=u_1,y'))|^2\rho(\po u-\po u_\sigma)dy'dt\\
\nn&\lesssim & \int_0^1\int |F(\Phi^{-1}_{t,\o}(u=u_1,y'))|^2\sqrt{\gamma}dy'dt\\
\nn&\lesssim & \norm{F}^2_{\li{\infty}{2}},
\eea
using again \eqref{lxx2:2} and \eqref{lxx2:3}.

Next, we evaluate $\frac{dI}{d\sigma}$:
\bea
\nn\frac{dI}{d\sigma}&=&2\int_0^1\int DF(\Phi^{-1}_{t,\o_\sigma}(u_\sigma,y'))\frac{d}{d\sigma}\Big[\Phi^{-1}_{t,\o_\sigma}(u_\sigma,y')\Big]F(\Phi^{-1}_{t,\o_\sigma}(u_\sigma,y'))\rho(\po u-\po u_\sigma)dy'dt\\
\label{lxx2:7}&&-\int_0^1\int |F(\Phi^{-1}_{t,,\o_\sigma}(u=u_\sigma,y'))|^2\po^2u_\sigma\rho'(\po u-\po u_\sigma)dy'dt.
\eea
Now, we have
$$\frac{d}{d\sigma}\Big[\Phi^{-1}_{t,\o_\sigma}(u_\sigma,y')\Big]=d\Phi^{-1}\left(\frac{d}{d\sigma}\Big[\Phi_{t,\o_\sigma}\Big]\circ\Phi^{-1}\right)(u_\sigma,y)+d\Phi^{-1}_{t,\o_\sigma}\left(\frac{du_\sigma}{d\sigma},y'\right),$$
which yields:
\bea
\nn\normm{\frac{d}{d\sigma}\Big[\Phi^{-1}_{t,\o_\sigma}(u_\sigma,.)\Big]}_{L^\infty}&\leq& \norm{d\Phi^{-1}}_{L^\infty}\left(\normm{\frac{d}{d\sigma}\Big[\Phi_{t,\o_\sigma}\Big]}_{L^\infty}+\normm{\frac{du_\sigma}{d\sigma}}_{L^\infty}\right)\\
\nn&\lesssim& \norm{\partial^2_\o u}_{L^\infty}+\norm{\partial_\o u}_{L^\infty}\\
\label{lxx2:7bis}&\lesssim& 1
\eea
where we used \eqref{lxx2:1}. Also, differentiating twice the Eikonal equation with respect to $\o$, we obtain:
$$L(\po^2u)=-b^{-1}g(\po N,\po N).$$
Since $\norm{\po N}_{L^\infty}\lesssim 1$, the use of the estimate for transport equations \eqref{estimtransport1} together with a corresponding estimate at initial time (see \cite{param1}) yields:
$$\norm{\po^2u}_{L^\infty}\lesssim 1.$$
Together with \eqref{lxx2:7} and \eqref{lxx2:7bis}, we obtain:
\bea\label{lxx2:8}
\left|\frac{dI}{d\sigma}\right|&\lesssim& \int_0^1\int (|DF(\Phi^{-1}_{t,\o_\sigma}(u_\sigma,y'))||F(\Phi^{-1}_{t,\o_\sigma}(u_\sigma,y'))|\rho(\po u-\po u_\sigma)\\
\nn&&+|F(\Phi^{-1}_{t,\o_\sigma}(u_\sigma,y'))|^2)\rho'(\po u-\po u_\sigma))dy'dt.
\eea

In view of \eqref{lxx2:5}, \eqref{lxx2:6} and \eqref{lxx2:8}, we obtain:
\bea
&&\nn\norm{F\sqrt{\rho(\po u-\po u_0)}}^2_{L^2(\H_{u'=u_0})}\\
\nn&\lesssim& \norm{F}^2_{\li{\infty}{2}}+\int_0^1\int_0^1\int (|DF(\Phi^{-1}_{t,\o_\sigma}(u_\sigma,y'))||F(\Phi^{-1}_{t,\o_\sigma}(u_\sigma,y'))|\rho(\po u-\po u_\sigma)\\
\label{lxx2:9}&&+|F(\Phi^{-1}_{t,\o_\sigma}(u_\sigma,y'))|^2\rho'(\po u-\po u_\sigma))dy'dtd\sigma.
\eea
Next, we consider the change of variables $(\sigma,y')\rightarrow (u,z')$ where:
$$u=u(t,x,\o),\,y'=\po u(t,x,\o_\sigma)\textrm{ and }z'=\po u(t,x,\o).$$
Given $(t,x)\in \Sigma_t$, there is only one $\sigma(t,x)$ such that $u(t,x,\o_\sigma)=u_\sigma$. $\sigma(t,x)$ is given implicitly by the following equation:
$$u(t,x,\o_\sigma)=u(t,x_0,\omega)$$
which after differentiation provides the formula:
\be\label{lxx2:10}
\nabla\sigma(t,x)=\frac{\nabla u(t,x)}{\partial_\o u(t,x,\o_\sigma)-\po u(t,x,\o_\sigma)}.
\ee
Also, we have:
\bea\label{lxx2:11}
\nabla u(t,x,\o_\sigma)=b_\sigma^{-1}N_\sigma\textrm{ and }\nabla \po u(t,x,\o_\sigma)=-b_\sigma^{-2}\po b_\sigma N_\sigma+b_\sigma^{-1}\po N_\sigma,
\eea
with the notation $N=N(t,x,\o)$, $b=b(t,x,\o)$, $N_\sigma=N(t,x,\o_\sigma)$, and $b_\sigma=b(t,x,\o_\sigma)$. In view of \eqref{lxx2:1bis}, the Jacobian $J$ of the change of variable $(\sigma,y')\rightarrow (u,z')$ in $\Sigma_t$ is the $3\times 3$ matrix given by:
\begin{displaymath}
J=\left(\begin{array}{ll}
g(N+O(\ep),\nabla\sigma(t,x)) & g(N+O(\ep),\nabla\po u(t,x,\o_\sigma))\\
g(\po N+O(\ep),\nabla\sigma(t,x)) & g(\po N+O(\ep),\nabla\po u(t,x,\o_\sigma))\\
\end{array}\right).
\end{displaymath}
Together with \eqref{lxx2:10} and \eqref{lxx2:11}, this yields for the determinant $|J|$:
\bea\label{lxx2:12}
|J|&=&\frac{b^{-3}_\sigma}{\partial_\o u(t,x,\o_\sigma)-\po u(t,x,\o_\sigma)}\\
&&\times\left|\begin{array}{ll}
g(N+O(\ep),N_\sigma) & g(N+O(\ep),-b_\sigma^{-1}\po b_\sigma N_\sigma+\po N_\sigma)\\
\nn g(\po N+O(\ep),N_\sigma) & g(\po N,-b_\sigma^{-1}\po b_\sigma N_\sigma+\po N_\sigma)\\
\end{array}\right|.
\eea
Now, recall that:
$$\norm{b-1}_{L^\infty}\lesssim 1,\, \norm{\po b}_{L^\infty}\lesssim\ep\textrm{ and }\norm{\po N}_{L^\infty}\lesssim 1$$
which together with \eqref{lxx2:12} yields:
\be\label{lxx2:13}
|J|\lesssim\frac{1}{\partial_\o u(t,x,\o_\sigma)-\po u(0,x_0,\o_\sigma)}.
\ee
Now, recall that $\rho$ vanishes in the neighborhood of 0, which together with \eqref{lxx2:13} implies:
\be\label{lxx2:14}
|\rho(\po u-\po u_\sigma)|J||+|\rho'(\po u-\po u_\sigma)|J||\lesssim 1.
\ee

Next, we consider the range of $u(t,x,\o)$ in the domain of the integral in the right-hand side of \eqref{lxx2:9}. We have:
\bea\label{lxx2:14bis}
|u(t,x,\o)-u_1|&\lesssim& |u(t,x,\o)-u(t,x,\o_\sigma)|+|u(t,x,\o_\sigma)-u_\sigma|+|u_\sigma-u_1|\\
\nn&\lesssim & |u(t,x,\o)-u(t,x,\o_\sigma)|+|u(t,x_t,\o_\sigma-u(t,x_t,\o)|\\
\nn&\lesssim & \norm{\po u}_{L^\infty}|\o_\sigma-\o|\\
\nn&\lesssim & |\o-\o'|,
\eea
where we used the fact that $u(t,x,\o_\sigma)=u_\sigma$, $u_\sigma=u(t,x_t,\sigma)$, $u_1=u(t,x_t,\o)$ and $\norm{\po u}_{L^\infty}\lesssim 1$.

In view of \eqref{lxx2:14} and \eqref{lxx2:14bis}, the change of variables $(\sigma,y')\rightarrow (u,z')$ in \eqref{lxx2:9} yields:
\bea\label{lxx2:15}
&&\norm{F\sqrt{\rho(\po u-\po u_0)}}^2_{L^2(\H_{u'=u_0})}\\
\nn&\lesssim& \norm{F}^2_{\li{\infty}{2}}+\int_0^1\int_{u_1}^{u_1+|\o-\o'|}\int |DF(\Phi^{-1}_{t,\o}(u,y'))||F(\Phi^{-1}_{t,\o}(u,y'))|dy'dtdu\\
\nn&\lesssim& \norm{F}^2_{\li{\infty}{2}}+\int_0^1\int_{u_1}^{u_1+|\o-\o'|}\int |DF(\Phi^{-1}_{t,\o}(u,y'))||F(\Phi^{-1}_{t,\o}(u,y'))|dy'dtdu\\
\nn&\lesssim& \norm{F}^2_{\li{\infty}{2}}+\sup_u\left(\int_{u}^{u+|\o-\o'|}\int_0^1\int |DF(\Phi^{-1}_{t,\o}(u,y'))||F(\Phi^{-1}_{t,\o}(u,y'))|dy'dtdu\right)\\
\nn&\lesssim& \norm{F}^2_{\li{\infty}{2}}+|\o-\o'|^{\frac{1}{2}}\sup_u\left(\int_0^1\int |F(\Phi^{-1}_{t,\o}(u,y'))|^2dy'dt\right)^{\frac{1}{2}}\\
\nn&&\times\sup_u\left(\int_{u}^{u+|\o-\o'|}\int_0^1\int |DF(\Phi^{-1}_{t,\o}(u,y'))|^2dy'dtdu\right)^{\frac{1}{2}}.
\eea
Now, we have:
$$\norm{F}^2_{\lh{2}} = \int_0^1\int |F(\Phi^{-1}_{t,\o}(u,y'))|^2\sqrt{\gamma}dy'dt\gtrsim \int_0^1\int |F(\Phi^{-1}_{t,\o}(u,y'))|^2dy'dt$$
where we used \eqref{lxx2:2}. Together with \eqref{lxx2:15}, this yields:
\bea\label{lxx2:16}
&&\norm{F\sqrt{\rho(\po u-\po u_0)}}^2_{L^2(\H_{u'=u_0})}\\
\nn&\lesssim& \norm{F}^2_{\li{\infty}{2}}
+|\o-\o'|^{\frac{1}{2}}\norm{F}_{L^\infty_uL^2(\H_u)}\left(\sup_u\left(\int_u^{u+|\o-\o'|}\norm{\dd F}^2_{\lh{2}}d\tau\right)^{\frac{1}{2}}\right).
\eea
Now, \eqref{lxx2:16} holds regardless of the choice $(t,x_t)$ on $P_{t,u_0}$. Also, $\po u$ as a map from $\ptu$ to the tangent space $T_\o\S$ is a $C^1$ diffeomorphism (see \cite{param1} Proposition 2.8 for a completely analogous proof in the case $t=0$ of $P_{0,u}$). Thus, we may choose $\rho$, and two points $(t,x^1_t)$ and $(t,x^2_t)$ on $P_{t,u_0}$ sufficiently far from each other such that for all $(t,x)\in \ptu$, we have:
$$\rho(\po u(t,x,\o')-\po u(t,x^1_t,\o'))+\rho(\po u(t,x,\o')-\po u(t,x^2_t,\o'))\geq 1.$$
Together with \eqref{lxx2:16}, we obtain:
$$\norm{F}^2_{L^2(\H_{u'=u_0})}\lesssim \norm{F}^2_{\li{\infty}{2}}+|\o-\o'|^{\frac{1}{2}}\norm{F}_{L^\infty_uL^2(\H_u)}\left(\sup_u\left(\int_u^{u+|\o-\o'|}\norm{\dd F}^2_{\lh{2}}d\tau\right)^{\frac{1}{2}}\right).$$
Taking the supremum over $u_0$ implies:
$$\norm{F}^2_{L^\infty_{u'}L^2(\H_{u'})}\lesssim \norm{F}^2_{\li{\infty}{2}}
+|\o-\o'|^{\frac{1}{2}}\norm{F}_{L^\infty_uL^2(\H_u)}\left(\sup_u\left(\int_u^{u+|\o-\o'|}\norm{\dd F}^2_{\lh{2}}d\tau\right)^{\frac{1}{2}}\right),$$
which concludes the proof of Lemma \ref{lemma:xx2}.

\begin{remark}
The change of variables $(\sigma,y')\rightarrow (u,z')$ in \eqref{lxx2:9} is singular at $(t,x)=(t,x_t)$ in 
view of the determinant of the corresponding Jacobian \eqref{lxx2:13}. This is also the case in the flat case where $u(t,x,\o)=t+x\cdot\o$ and where $\ptu$ are parallel planes in $\R^3$ orthogonal to $\o$. In this case, the corresponding change of variables corresponds to a change of variable in the plane of $\R^3$ spanned by $\o$ and $\o'$ passing through $x_t$ from polar coordinates with origin at $x_t$ to cartesian coordinates. This explains why the singularity at $(t,x_t)$ in the change of variables $(\sigma,y')\rightarrow (u,z')$ in \eqref{lxx2:9} is natural. Fortunately, one has the freedom to chose the point $(t,x_t)$ around which we rotate the surfaces 
$\ptu$ which allows us to tackle this issue by considering successively two point in $\Sigma_t$ $(t,x^1_t)$ and $(t,x^2_t)$ as preformed in the end of the above proof.
\end{remark}

\subsection{Proof of Lemma \ref{lemma:xx3}}\lab{sec:xx3}

Let us apply Lemma \ref{lemma:xx2} with $F=P_lf$ where $f$ is a scalar function. Then:
\bea
\label{lxx3:1}&&\norm{P_lf}_{L^\infty_{u'}L^2(\H_{u'})}\\
\nn&\lesssim& \norm{P_lf}_{\li{\infty}{2}}
+|\o-\o'|^{\frac{1}{4}}\norm{P_lf}^{\frac{1}{2}}_{L^\infty_uL^2(\H_u)}\left(\sup_u\left(\int_u^{u+|\o-\o'|}\norm{\dd P_lf}^2_{\lh{2}}d\tau\right)^{\frac{1}{2}}\right)^{\frac{1}{2}}\\
\nn&\lesssim& 2^{-l}\norm{\nabb f}_{\li{\infty}{2}}
+2^{-\frac{l}{2}}|\o-\o'|^{\frac{1}{2}}\norm{\nabb f}^{\frac{1}{2}}_{L^\infty_uL^2(\H_u)}\norm{\dd P_lf}_{\li{\infty}{2}}^{\frac{1}{2}}
\eea
where we used the finite band property for $P_l$ in the last inequality. 

In order to prove Lemma \ref{lemma:xx3}, it is enough in view of \eqref{lxx3:1} to prove:
\be\label{lxx3:2}
\norm{\dd P_lf}_{\li{\infty}{2}}\lesssim \norm{\dd f}_{\li{\infty}{2}}.
\ee
Furthermore, \eqref{lxx3:2} for $\dd=\nabb$ follows from the properties of $P_l$, so we may focus on the case of $L$ and $\lb$, or even $L$ and $N$. Also, the case of $L$ being easier, we focus on the case of $N$. Thus, the proof of Lemma \ref{lemma:xx3} reduces to:
\be\label{lxx3:3}
\norm{N(P_lf)}_{\li{\infty}{2}}\lesssim \norm{\dd f}_{\li{\infty}{2}}.
\ee

Since $\norm{b-1}_{L^\infty}\lesssim\ep$, we have:
\bee
\norm{N(P_lf)}_{\li{\infty}{2}}&\lesssim& \norm{bN(P_lf)}_{\li{\infty}{2}}\\
&\lesssim& \norm{P_l(bN(f))}_{\li{\infty}{2}}+\norm{[bN,P_l]f}_{\li{\infty}{2}}\\
&\lesssim& \norm{\dd f}_{\li{\infty}{2}}+\norm{[bN,P_l]f}_{\li{\infty}{2}},
\eee
where we used the $\lpt{2}$ boundedness of $P_l$ in the last inequality. Together with the commutator 
estimate \eqref{zz}, we obtain the desired estimate \eqref{lxx3:3}. This concludes the proof of Lemma \ref{lemma:xx3}.

\subsection{Proof of Lemma \ref{lemma:xx4}}\lab{sec:xx4}

Let us apply Lemma \ref{lemma:xx2} with $F=P_{\leq l}f$ where $f$ is a scalar function. Then:
\bea
\label{lxx4:1}&&\norm{P_{\leq l}f}_{L^\infty_{u'}L^2(\H_{u'})}\\
\nn&\lesssim& \norm{P_{\leq l}f}_{\li{\infty}{2}}
+|\o-\o'|^{\frac{1}{4}}\norm{P_{\leq l}f}^{\frac{1}{2}}_{L^\infty_uL^2(\H_u)}\left(\sup_u\left(\int_u^{u+|\o-\o'|}\norm{\dd P_{\leq l}f}^2_{\lh{2}}d\tau\right)^{\frac{1}{2}}\right)^{\frac{1}{2}}\\
\nn&\lesssim& \norm{f}_{\li{\infty}{2}}
+|\o-\o'|^{\frac{1}{2}}\norm{f}^{\frac{1}{2}}_{L^\infty_uL^2(\H_u)}\left(\sup_u\left(\int_u^{u+|\o-\o'|}\norm{\dd P_{\leq l}f}^2_{\lh{2}}d\tau\right)^{\frac{1}{2}}\right)^{\frac{1}{2}}
\eea
where we used the finite band property for $P_{\leq l}$ in the last inequality. Now, we have:
\bea\label{lxx4:2}
\norm{\dd P_{\leq l}f}_{\lh{2}}&\les& \norm{\nabb P_{\leq l}f}_{\lh{2}}+\norm{L(P_{\leq l}f)}_{\lh{2}}+\norm{N(P_{\leq l}f)}_{\lh{2}}\\
\nn&\les& 2^l\norm{f}_{\li{\infty}{2}}+\norm{nL(P_{\leq l}f)}_{\lh{2}}+\norm{bN(P_{\leq l}f)}_{\lh{2}},
\eea
where we used the finite band property for $P_l$ in the last inequality. Also:
\bee
&&\norm{nL(P_{\leq l}f)}_{\lh{2}}+\norm{bN(P_{\leq l}f)}_{\lh{2}}\\
\nn&\les& \norm{P_{\leq l}(nL(f))}_{\lh{2}}+\norm{P_{\leq l}(bN(f))}_{\lh{2}}+\norm{[nL,P_{\leq l}]f}_{\lh{2}}+\norm{[bN,P_{\leq l}]f}_{\lh{2}}\\
\nn&\les& \sum_{q\leq l}\left(\norm{P_q(nL(f))}_{\lh{2}}+\norm{P_q(bN(f))}_{\lh{2}}+\norm{[nL,P_q]f}_{\lh{2}}+\norm{[bN,P_q]f}_{\lh{2}}\right),
\eee
which together with the commutator estimate \eqref{lxx4:4} implies
\bea\label{lxx4:3}
&&\norm{nL(P_{\leq l}f)}_{\lh{2}}+\norm{bN(P_{\leq l}f)}_{\lh{2}}\\
\nn&\les& \sum_{q\leq l}\left(\norm{P_q(nL(f))}_{\lh{2}}+\norm{P_q(bN(f))}_{\lh{2}}+2^q\norm{f}_{\li{\infty}{2}}\right)\\
\nn&\les& \sum_{q\leq l}\left(\norm{P_q(nL(f))}_{\lh{2}}+\norm{P_q(bN(f))}_{\lh{2}}\right)+2^l\norm{f}_{\li{\infty}{2}}.
\eea
Finally, in view of \eqref{lxx4:1}-\eqref{lxx4:3}, we obtain the conclusion of Lemma \ref{lemma:xx4}.

\subsection{Proof of Lemma \ref{lemma:xx5}}\lab{sec:xx5}

Lemma \ref{lemma:xx8} with the choice $p=2$ yields the following estimate for any $a>\half$:
\bea
\label{lxx5:0}\norm{f}_{L^\infty_{u'}L^2(\H_{u'})}&\leq &\sum_j\norm{P_jf}_{L^\infty_{u'}L^2(\H_{u'})}\\
\nn&\leq &\sum_j \norm{P_jf}^{\frac{1}{2}}_{L^\infty_uL^2(\H_u)}\norm{\nabb P_jf}^{\frac{1}{2}}_{L^\infty_uL^2(\H_u)}\\
\nn&\leq &\sum_j 2^j\norm{P_jf}_{L^\infty_uL^2(\H_u)}\\
\nn&\leq &\left(\sum_j 2^{-j(1-a)}\right) \norm{\Lambda^af}_{\li{\infty}{2}}\\
\nn&\leq &\norm{\Lambda^af}_{\li{\infty}{2}}
\eea
where we used the finite band property for $P_j$.

Next, we evaluate $[\po, P_{\leq l}]f$. We have:
$$\po U(\tau)f=U(\tau)\po f+W(\tau)$$
where $W(\tau)$ satisfies:
\begin{equation}\label{lxx5:1}
(\partial_{\tau}-\lap)W(\tau)=[\po,\lap]U(\tau)f,\,W(0)=0.
\end{equation}
Using the definition of $P_q$ \eqref{eq:LP} and \eqref{lxx5:1}, we obtain:
$$[\po, P_q]f=\int_0^\infty m_q(\tau)W(\tau) d\tau.$$
Together with \eqref{lxx5:0}, this yields for any $a>\half$:
\bea
\label{lxx5:2}\norm{[\po,P_{\leq l}]f}_{L^\infty_{u'}L^2(\H_{u'})}&\les& \normm{\sum_{q\leq l}[\po,P_q]f}_{L^\infty_{u'}L^2(\H_{u'})}\\
\nn &\les &\normm{\int_0^\infty \left|\left(\sum_{q\leq l}m_q(\tau)\right)\right|\norm{\Lambda^aW(\tau)}_{\lh{2}} d\tau}_{L^\infty_u}\\
\nn &\les &\left(\int_0^\infty \left|\left(\sum_{q\leq l}m_q(\tau)\right)\right|d\tau\right)\normm{\sup_{\tau}\norm{\Lambda^aW(\tau)}_{\lh{2}}}_{L^\infty_u}\\
\nn &\les &\normm{\sup_{\tau}\norm{\Lambda^aW(\tau)}_{\lh{2}}}_{L^\infty_u}.
\eea

Let 
$$\half<a<1.$$
The energy estimate \eqref{eq:l2heat1bis} implies:
\begin{displaymath}
\begin{array}{ll}
& \ds\norm{\La^aW(\tau)}^2_{\lh{2}}+\int_0^\tau\norm{\nabb\La^aW(\tau')}^2_{\lh{2}}d\tau'\\
\ds\lesssim &\ds\int_0^\tau\int_0^1\int_{\ptu} \La^{2a}W(\tau')[\po,\lap]U(\tau')f\dmt dtd\tau'\\
\ds\lesssim &\ds\int_0^\tau \norm{\La^{1+a}W(\tau')}_{\lh{2}}\norm{\Lambda^{-1+a}[\po,\lap]U(\tau')f}_{\lh{2}} d\tau'\\
\ds\lesssim &\ds\int_0^\tau \norm{\nabb\La^aW(\tau')}_{\lh{2}}\norm{\Lambda^{-1+a}[\po,\lap]U(\tau')f}_{\lh{2}} d\tau',
\end{array}
\end{displaymath}
where we used the property \eqref{La1} for $\Lambda$. Thus, we obtain:
\begin{equation}\label{lxx5:3}
\ds\norm{\La^aW(\tau)}^2_{\lh{2}}+\int_0^\tau\norm{\nabb\La^aW(\tau')}^2_{\lh{2}}d\tau'\lesssim \ds\int_0^\tau\norm{\Lambda^{-1+a}[\po,\lap]U(\tau')f}^2_{\lh{2}} d\tau'.
\end{equation}

The following formula has been established in \cite{param1}:
\begin{equation}\label{lxx5:4}
\begin{array}{ll}
\ds [\po,\lap]U(\tau)f= & \ds -2\nabb_{\po N}\nabla_N U(\tau)f+2\th(\po N,\nabb U(\tau)f) -\trt\nabb_{\po N}U(\tau)f\\
& \ds -\po\trt\nabla_N U(\tau)f,
\end{array}
\end{equation}
where $\th=\chi+\eta$ has been defined in \eqref{def:theta}. The estimates \eqref{estk}-\eqref{esthch} for $\th$, \eqref{estNomega} for $\po N$ and \eqref{estricciomega} for $\po\trt$ together with the Gagliardo-Nirenberg inequality \eqref{eq:GNirenberg}, the fact that $a<1$, the estimate \eqref{La6} for $\Lambda^{-1+a}$ and \eqref{lxx5:4} imply:
\begin{displaymath}
\begin{array}{lll}
\ds\norm{\Lambda^{-1+a}[\po,\lap]U(\tau)f}_{\lh{2}} & \lesssim & \norm{[\po,\lap]U(\tau)f}_{\tx{2}{2_-}}\\
& \lesssim & \ds\norm{\nabb\nabla_N U(\tau)f}_{\lh{2}}\norm{\po N}_{L^\infty}\\
& & \ds +\norm{\th}_{\tx{\infty}{4}}\norm{\po N}_{L^\infty}\norm{\nabb U(\tau)f}_{\tx{2}{4}}\\
& & \ds +\norm{\po\trt}_{\tx{\infty}{2}}\norm{\nabb\nabla_N U(\tau)f}_{\lh{2}}\\
& \ds\lesssim & \ds\norm{\nabb^2 U(\tau)f}_{\lh{2}}+\norm{\nabb\nabla_N U(\tau)f}_{\lh{2}}.
\end{array}
\end{displaymath}
Together with the Bochner inequality for scalars \eqref{eq:Bochconseqbis}, the estimates for $b$ \eqref{estb} and the definition of $V$ \eqref{lxx3:6} yield:
\begin{equation}\label{lxx5:6}
\norm{\Lambda^{-1+a}[\po,\lap]U(\tau)f}_{\lh{2}}\lesssim  \ds\norm{\lap U(\tau)f}_{\lh{2}}+\norm{\nabb U(\tau)\nabla_{bN} f}_{\lh{2}}+\norm{\nabb V(\tau)}_{\lh{2}}.
\end{equation}
Using the Heat flow estimate \eqref{eq:l2heatnab}, we have:
\begin{equation}\label{lxx5:7}
\norm{\nabb U(\tau)f}^2_{\lh{2}}+\int_0^\infty\norm{\lap U(\tau)f}^2_{\lh{2}}d\tau\lesssim\norm{\nabb f}^2_{\lh{2}}.
\end{equation}
Using the Heat flow estimate \eqref{eq:l2heat1}, we have:
\begin{equation}\label{lxx5:8}
\norm{U(\tau)\nabla_N f}^2_{\lh{2}}+\int_0^\infty\norm{\nabb U(\tau)\nabla_N f}^2_{\lh{2}}d\tau\lesssim\norm{\nabla_N f}^2_{\lh{2}}.
\end{equation}
The estimate \eqref{lxx3:11bis} for $\nabb V$, \eqref{lxx5:6}, \eqref{lxx5:7} and \eqref{lxx5:8} imply:
$$\ds\int_0^\tau\norm{\Lambda^{-1+a}[\po,\lap]U(\tau')f}^2_{\lh{2}}d\tau'\lesssim\norm{\dd f}^2_{\li{\infty}{2}},$$
which together with \eqref{lxx5:3} yields:
\begin{equation}\label{lxx5:9}
\ds\norm{\La^aW(\tau)}^2_{\lh{2}}+\int_0^\tau\norm{\nabb\La^aW(\tau')}^2_{\lh{2}}d\tau'\lesssim \norm{\dd f}^2_{\li{\infty}{2}}.
\end{equation}
Since \eqref{lxx5:9} holds for any $\half<a<1$, we obtain together with \eqref{lxx5:2}
$$\norm{[\po,P_{\leq l}]f}_{L^\infty_{u'}L^2(\H_{u'})}\les \norm{\dd f}_{\li{\infty}{2}}.$$
This concludes the proof of Lemma \ref{lemma:xx5}.

\subsection{Proof of Lemma \ref{lemma:xx6}}\lab{sec:xx6}

We have:
\bee
&&\norm{\dd_TQ_{\leq 1}(N)}_{\lsit{\infty}{2}}+\norm{\nabla\dd_TQ_{\leq 1}(N)}_{\lsit{\infty}{2}}\\
&\les& \norm{\dd_{nT}Q_{\leq 1}(N)}_{\lsit{\infty}{2}}+\norm{\nabla\dd_{nT}Q_{\leq 1}(N)}_{\lsit{\infty}{2}}+\norm{\nabla n}_{\lsit{\infty}{6}}\norm{\dd_TQ_{\leq 1}(N)}_{\lsit{\infty}{3}}\\
&\les& \norm{\dd_{nT}Q_{\leq 1}(N)}_{\lsit{\infty}{2}}+\norm{\nabla\dd_{nT}Q_{\leq 1}(N)}_{\lsit{\infty}{2}},
\eee
where we used in the last inequality the estimates \eqref{estn} and \eqref{lapn3} for $n$, and the Sobolev embedding \eqref{sobineqsit} on $\Sigma_t$. This yields:
\bea\label{lxx6:0}
&&\norm{\dd_TQ_{\leq 1}(N)}_{\lsit{\infty}{2}}+\norm{\nabla\dd_TQ_{\leq 1}(N)}_{\lsit{\infty}{2}}\\
\nn&\les& \norm{Q_{\leq 1}(\dd_{nT}N)}_{\lsit{\infty}{2}}+\norm{\nabla Q_{\leq 1}(\dd_{nT}N)}_{\lsit{\infty}{2}}
+ \norm{[\dd_{nT},Q_{\leq 1}](N)}_{\lsit{\infty}{2}}\\
\nn&&+\norm{\nabla[\dd_{nT},Q_{\leq 1}](N)}_{\lsit{\infty}{2}}\\
\nn&\les& \norm{\dd_{nT}N}_{\lsit{\infty}{2}}+ \norm{[\dd_{nT},Q_{\leq 1}](N)}_{\lsit{\infty}{2}}+\norm{\nabla[\dd_{nT},Q_{\leq 1}](N)}_{\lsit{\infty}{2}}.
\eea
Now, we have in view of the Ricci equations \eqref{ricciform}, we have:
$$\dd_TN=n^{-1}\nabla_N n T+(\z_A-n^{-1}\nabb_An)e_A$$
which together with the estimates \eqref{estn} for $n$ and \eqref{estzeta} for $\z$ yields:
$$\norm{\dd_{nT}N}_{\lsit{\infty}{2}}\les \norm{\nabla n}_{\lsit{\infty}{2}}+\norm{\z}_{\lsit{\infty}{2}}\les \ep.$$
Together with \eqref{lxx6:0}, this yields:
\bea\label{lxx6:0bis}
&&\norm{\dd_TQ_{\leq 1}(N)}_{\lsit{\infty}{2}}+\norm{\nabla\dd_TQ_{\leq 1}(N)}_{\lsit{\infty}{2}}\\
\nn&\les&  \norm{[\dd_{nT},Q_{\leq 1}](N)}_{\lsit{\infty}{2}}+\norm{\nabla[\dd_{nT},Q_{\leq 1}](N)}_{\lsit{\infty}{2}}+\ep.
\eea

Next, we estimate the commutator terms in the right-hand side of \eqref{lxx6:0bis}. 
Using the definition of $Q_j$, we have:
\begin{equation}\label{lxx6:1}
[\dd_{nT}, Q_j]N=\int_0^\infty m_j(\tau)Z(\tau) d\tau,
\end{equation}
where $Z(\tau)$ satisfies:
\be\lab{lxx6:2}
(\partial_{\tau}-\Delta)Z(\tau)=[\dd_{nT},\Delta]Y(\tau)N,\,Z(0)=0,
\ee
with $Y(\tau)N$ the solution of:
$$(\partial_{\tau}-\Delta)Y(\tau)N=0,\, Y(0)N=N.$$
In view of \eqref{lxx6:1}, we have:
\bea
&&\label{lxx6:3} \norm{[\dd_{nT},Q_{\leq 1}](N)}_{\lsit{\infty}{2}}+\norm{\nabla[\dd_{nT},Q_{\leq 1}](N)}_{\lsit{\infty}{2}}\\
\nn &\les& \normm{\sup_\tau\norm{Z(\tau)}_{L^2(\Sit)}}_{L^\infty_t}+\normm{\int_0^\infty\norm{\nabla Z(\tau)}^2_{L^2(\Sit)}d\tau}^{\half}_{L^\infty_t}.
\eea 

Our next goal is to evaluate the right-hand side of \eqref{lxx6:3}. Multiplying \eqref{lxx6:2} with $Z(\tau)$ and integrating on $\Sigma_t$ and with respect to $\tau$ yields:
\be\label{lxx6:4}
\norm{Z(\tau)}^2_{L^2(\Sigma_t)}+\int_0^\tau\norm{\nabla Z(\tau')}^2_{L^2(\Sigma_t)}d\tau'\les \int_0^\tau\int_{\Sigma_t}Z(\tau')[\dd_{nT},\Delta]Y(\tau')Nd\Sigma_t d\tau'.
\ee
In view of the commutator formula \eqref{commsitter}, we have:
\bee
[\dd_{nT},\Delta]Y(\tau)N &=& nk\nabla^2Y(\tau)N+(n\rr + k\nabla n+n\nabla k)\nabla Y(\tau)N\\
&&+(\nabla k\nabla n+k\Delta n)Y(\tau)N+\nabla(nRY(\tau)N)\nn.
\eee
Integrating by parts, this yields:
\bee
&&\int_0^\tau\int_{\Sigma_t}Z(\tau')[\dd_{nT},\Delta]Y(\tau')Nd\Sigma_t d\tau'\\
&= & \int_0^\tau\int_{\Sigma_t}nk\nabla Z(\tau')\nabla Y(\tau')Nd\Sigma_t d\tau'+\int_0^\tau\int_{\Sigma_t}Z(\tau')(n\rr + k\nabla n+n\nabla k)\nabla Y(\tau')Nd\Sigma_t d\tau'\\
&&+\int_0^\tau\int_{\Sigma_t}Z(\tau')(\nabla k\nabla n+k\Delta n)Y(\tau')Nd\Sigma_t d\tau'+\int_0^\tau\int_{\Sigma_t}n\rr\nabla Z(\tau')Y(\tau')Nd\Sigma_t d\tau'\\
&\les& \int_0^\tau\norm{\nabla Z(\tau')}_{L^2(\Sigma_t)}(\norm{\nabla Y(\tau')}_{L^3(\Sigma_t)}\norm{nk}_{L^6(\Sigma_t)}+\norm{Y(\tau')}_{L^\infty(\Sigma_t)}\norm{n\rr}_{L^2(\Sigma_t)})\\
&&+\norm{Z(\tau')}_{L^6(\Sigma_t)}\Big(\norm{\nabla Y(\tau')}_{L^3(\Sigma_t)}(\norm{n\rr}_{L^2(\Sigma_t)}+\norm{n\nabla k}_{L^2(\Sigma_t)}+\norm{k\nabla n}_{L^2(\Sigma_t)})\\
&&+\norm{Y(\tau')}_{L^6(\Sigma_t)}(\norm{\nabla k\nabla n}_{L^{\frac{3}{2}}(\Sigma_t)}+\norm{k\Delta n}_{L^{\frac{3}{2}}(\Sigma_t)})\Big)d\tau'\\
&\les& \ep\int_0^\tau\norm{\nabla Z(\tau')}_{L^2(\Sigma_t)}(\norm{\nabla Y(\tau')}_{L^3(\Sigma_t)}+\norm{Y(\tau')}_{L^\infty(\Sigma_t)})\\
&&+\norm{Z(\tau')}_{L^6(\Sigma_t)}(\norm{\nabla Y(\tau')}_{L^3(\Sigma_t)}+\norm{Y(\tau')}_{L^6(\Sigma_t)})d\tau'
\eee
where we used in the last inequality the estimates \eqref{estn} and \eqref{lapn3} for $n$, the estimate \eqref{eqksit2} for $k$, the curvature bound \eqref{curvflux1} for $\rr$, and the Sobolev embedding \eqref{sobineqsit} on $\Sigma_t$. Together with the Sobolev embedding \eqref{sobineqsit}, the $L^\infty$ estimate  \eqref{sobinftysitbis}, and the Bochner inequality \eqref{prop:bochsit} on $\Sigma_t$, we obtain:
$$\int_0^\tau\int_{\Sigma_t}Z(\tau')[\dd_{nT},\Delta]Y(\tau')Nd\Sigma_t d\tau'\les \ep\int_0^\tau\norm{\nabla Z(\tau')}_{L^2(\Sigma_t)}(\norm{\Delta Y(\tau')}_{L^2(\Sigma_t)}+\norm{\nabla Y(\tau')}_{L^2(\Sigma_t)})d\tau',$$
which together with \eqref{lxx6:4} yields:
\be\label{lxx6:5}
\norm{Z(\tau)}^2_{L^2(\Sigma_t)}+\int_0^\tau\norm{\nabla Z(\tau')}^2_{L^2(\Sigma_t)}d\tau'\les \int_0^\tau(\norm{\Delta Y(\tau')}^2_{L^2(\Sigma_t)}+\norm{\nabla Y(\tau')}^2_{L^2(\Sigma_t)})d\tau'.
\ee
Now, usual Heat flow estimates for $Y(\tau)N$ yield:
\be\label{lxx6:6}
\int_0^\tau(\norm{\Delta Y(\tau')}^2_{L^2(\Sigma_t)}+\norm{\nabla Y(\tau')}^2_{L^2(\Sigma_t)})d\tau'
\les \norm{\dd N}_{L^2(\Sigma_t)}\les \ep,
\ee
where we used in the last inequality the Ricci equations \eqref{ricciform} to compute $\dd N$ in function of the ricci coefficients, and the estimates \eqref{estn}-\eqref{estzeta} to estimate the ricci coefficient in $\tx{\infty}{2}$ which embeds in $\lsit{\infty}{2}$. Finally, \eqref{lxx6:5} and \eqref{lxx6:6} yield:
$$\norm{Z(\tau)}^2_{L^2(\Sigma_t)}+\int_0^\tau\norm{\nabla Z(\tau')}^2_{L^2(\Sigma_t)}d\tau'\les \ep,$$
which together with \eqref{lxx6:3} implies:
$$\norm{[\dd_{nT},Q_{\leq 1}](N)}_{\lsit{\infty}{2}}+\norm{\nabla[\dd_{nT},Q_{\leq 1}](N)}_{\lsit{\infty}{2}}\les\ep.$$
In view of \eqref{lxx6:0bis}, this yields
$$\norm{\dd_TQ_{\leq 1}(N)}_{\lsit{\infty}{2}}+\norm{\nabla\dd_TQ_{\leq 1}(N)}_{\lsit{\infty}{2}}\les\ep,$$
which concludes the proof of Lemma \ref{lemma:xx6}.

\subsection{Proof of Lemma \ref{lemma:xx7}}\lab{sec:xx7}

Let $\delta^j_l=1$ if $j=l$ and 0 otherwise. Our goal is evaluate the $L^\infty$ norm of 
$$g(Q_{\leq 1}(N_j), Q_{\leq 1}(N_l))-\delta^j_l.$$ 
The $L^\infty$ estimate \eqref{sobinftysitbis} on $\Sigma_t$ together with the Bochner inequality \eqref{prop:bochsit} on $\Sigma_t$ yields:
\bea
\label{lxx7:1}&&\norm{g(Q_{\leq 1}(N_j), Q_{\leq 1}(N_l))-\delta^j_l}_{L^\infty}\\
\nn&\les& \norm{g(Q_{\leq 1}(N_j), Q_{\leq 1}(N_l))-\delta^j_l}_{\lsit{\infty}{2}}+\norm{\nabla^2 (g(Q_{\leq 1}(N_j), Q_{\leq 1}(N_l)))}_{\lsit{\infty}{2}}\\
\nn&\les& \norm{g(Q_{\leq 1}(N_j), Q_{\leq 1}(N_l))-\delta^j_l}_{\lsit{\infty}{2}}+\norm{\Delta Q_{\leq 1}(N_j), Q_{\leq 1}(N_l)))}_{\lsit{\infty}{2}}\\
\nn&\les& \norm{g(Q_{\leq 1}(N_j), Q_{\leq 1}(N_l))-\delta^j_l}_{\lsit{\infty}{2}}+\norm{\dd N_j}_{\lsit{\infty}{2}}\norm{\dd N_l}_{\lsit{\infty}{2}}\\
\nn&&+\norm{N_j}_{L^\infty}\norm{\dd N_l}_{\lsit{\infty}{2}}+\norm{\dd N_j}_{\lsit{\infty}{2}}\norm{N_l}_{L^\infty}\\
\nn&\les& \norm{g(Q_{\leq 1}(N_j), Q_{\leq 1}(N_l))-\delta^j_l}_{\lsit{\infty}{2}}+\ep,
\eea
where we have used the finite band property for $Q_{\leq 1}$, the boundedness of $Q_{\leq 1}$ on $L^\infty(\Sigma_t)$, the Ricci equations \eqref{ricciform} to compute $\dd N$ in function of the ricci coefficients, and the estimates \eqref{estn}-\eqref{estzeta} to estimate the ricci coefficient in $\tx{\infty}{2}$ which embeds in $\lsit{\infty}{2}$. Now, we have:
\bee
&&\norm{g(Q_{\leq 1}(N_j), Q_{\leq 1}(N_l))-\delta^j_l}_{\lsit{\infty}{2}}\\
&\les & \norm{g(N_j,N_l)-\delta^j_l}_{\lsit{\infty}{2}}+\norm{g(Q_{\leq 1}(N_j), Q_{>1}(N_l))}_{\lsit{\infty}{2}}\\
&&+\norm{g(Q_{> 1}(N_j), Q_{\leq 1}(N_l))}_{\lsit{\infty}{2}}+\norm{g(Q_{> 1}(N_j), Q_{>1}(N_l))}_{\lsit{\infty}{2}}\\
&\les & \norm{g(N_j,N_l)-\delta^j_l}_{\lsit{\infty}{2}}+\norm{N_j}_{L^\infty}\norm{\dd N_l}_{\lsit{\infty}{2}}
+\norm{\dd N_j}_{\lsit{\infty}{2}}\norm{N_l}_{L^\infty}\\
&\les &  \norm{g(N_j,N_l)-\delta^j_l}_{\lsit{\infty}{2}}+\ep,
\eee
where we have used the finite band property for $Q_{> 1}$, the boundedness of $Q_{\leq 1}$ on $L^\infty(\Sigma_t)$, the Ricci equations \eqref{ricciform} to compute $\dd N$ in function of the ricci coefficients, and the estimates \eqref{estn}-\eqref{estzeta} to estimate the ricci coefficient in $\tx{\infty}{2}$ which embeds in $\lsit{\infty}{2}$. Together with \eqref{lxx7:1}, this yields:
\be\label{lxx7:2}\norm{g(Q_{\leq 1}(N_j), Q_{\leq 1}(N_l))-\delta^j_l}_{L^\infty}\les  \norm{g(N_j,N_l)-\delta^j_l}_{\lsit{\infty}{2}}+\ep.
\ee
Next, we have:
\bee
&&\norm{g(N_j,N_l)-\delta^j_l}_{\lsit{\infty}{2}}\\
&\les& \norm{g(N_j,N_l)-\delta^j_l}_{L^2(\Sigma_0)}+\norm{\dd_T g(N_j,N_l)}_{L^2(\mathcal{M})}\\
&&\norm{g(N_j,N_l)-\delta^j_l}_{L^2(\Sigma_0)}+\norm{N_j}_{L^\infty}\norm{\dd N_l}_{\lsit{\infty}{2}}
+\norm{\dd N_j}_{\lsit{\infty}{2}}\norm{N_l}_{L^\infty}\\
&\les&\ep,
\eee
where we have used in the last inequality the estimate on $g(N_j,N_l)$ on $\Sigma_0$ derived in \cite{param1}, the Ricci equations \eqref{ricciform} to compute $\dd N$ in function of the ricci coefficients, and the estimates \eqref{estn}-\eqref{estzeta} to estimate the ricci coefficient in $\tx{\infty}{2}$ which embeds in $\lsit{\infty}{2}$. Together with \eqref{lxx7:2}, we obtain:
$$\norm{g(Q_{\leq 1}(N_j), Q_{\leq 1}(N_l))-\delta^j_l}_{L^\infty}\les \ep.$$
This proves that $Q_{\leq 1}(N_1), Q_{\leq 1}(N_2)$ and $Q_{\leq 1}(N_3)$ form a basis of the tangent space of $\Sit$. This concludes the proof of Lemma \ref{lemma:xx7}.

\subsection{Proof of Lemma \ref{lemma:xx8}}\lab{sec:xx8}

Let $(\varphi, \psi)$ the spherical coordinates on $\S$ such that $\psi$ measures the angle in the plane spanned by $\o, \o'$, and $\varphi$ measures the angle with the axis $\o\wedge \o'$. Then, we have in particular: 
\be\label{lxx8:1}
\partial_\varphi\o\cdot(\o-\o')=0.
\ee
Now, we claim that we have the analog estimate:
\be\label{lxx8:2}
|g(\partial_\varphi N,N-N')|\lesssim |\o-\o'|(\ep+|\o-\o'|).
\ee
Indeed, we have:
$$g(\partial_\varphi N,N-N')=\int_{[\o,\o']}g(\partial_\varphi N,\partial_\psi N'')d\o''(\o'-\o),$$
where $\partial_\psi N''=\partial_\psi N(.,\o'')$. This yields:
$$|g(\partial_\varphi N,N-N')|\lesssim |\o-\o'|\sup_{\o''\in [\o,\o']}|g(\partial_\varphi N,\partial_\psi N'')|,$$
and \eqref{lxx8:2} now follows from:
\be\label{lxx8:3}
\sup_{\o''\in [\o,\o']}|g(\partial_\varphi N,\partial_\psi N'')|\les \ep+|\o-\o'|.
\ee
Now, let $\o_1\in\S$ defined as:
$$\o_1=\frac{\o-\o'}{|\o-\o'|}.$$
Arguing as in the proof of \eqref{ad1}, we have:
\be\label{lxx8:4}
\norm{g(\partial_\psi N'',N_1)-1}_{L^\infty}\les \ep+|\o-\o'|.
\ee
The choice of $\o_1$ and the fact that $\varphi$ measures the angle with the axis $\o\wedge \o'$ implies 
$$\partial_\varphi\o\cdot\o_1=0.$$
Arguing again as in the proof of \eqref{ad1}, we obtain:
$$\norm{g(\partial_\varphi N,N_1)}_{L^\infty}\les \ep+|\o-\o'|,$$
which together with  \eqref{lxx8:4} yields \eqref{lxx8:3}. This concludes the proof of \eqref{lxx8:2}.

Now, we consider the coordinate system on $\H_{u'}$ consisting of the functions $t, u$ and $\partial_\varphi u$, where $u=u(t,x,\o)$ and $\partial_\varphi u=\partial_\varphi u(t,x,\o)$. The fact that it is indeed a coordinate system on $\H_{u'}$ follows from the fact that $(u,\partial_\varphi u)$ is a coordinate system on $P_{t,u'}$. The later claim follows from the invertibility of the corresponding matrix of the metric coefficients in the coordinate system $(u, \partial_\varphi u)$ which we check now. Using the fact that 
$g(N,\partial_\varphi N)=0$, we easily compute the following identities for the coordinate system $(u, \partial_\varphi u)$ on $P_{t,u'}$:
$$\frac{\partial}{\partial u}=\frac{b}{1-g(N,N')^2}(N-g(N,N')N')+\frac{1}{g(\partial_\varphi N,\partial_\varphi N)}\left(\partial_\varphi b+\frac{bg(N,N')g(\partial_\varphi N,N')}{1-g(N,N')^2}\right)\partial_\varphi N.$$
and: 
$$\frac{\partial}{\partial \partial_\varphi u}=\frac{b}{g(\partial_\varphi N,\partial_\varphi N)}\partial_\varphi N.$$
Let $\gamma'$ denote the induced metric on $P_{t,u'}$. The previous identities yield the 
corresponding coefficients for $\gamma'$ in the coordinate system $(u, \partial_\varphi u)$:
\be\label{lxx8:5}
\gamma'\left(\frac{\partial}{\partial u},\frac{\partial}{\partial u}\right)= \frac{b^2}{1-g(N,N')^2}\left(1-\frac{2g(N,N')g(N',\partial_\varphi N)\partial_\varphi b}{g(\partial_\varphi N, \partial_\varphi N)b}-\frac{2g(N,N')^2g(\partial_\varphi N, N')^2}{1-g(N,N')^2}\right),
\ee
\be\label{lxx8:6}
\gamma'\left(\frac{\partial}{\partial \partial_\varphi u},\frac{\partial}{\partial u}\right)= \frac{b\partial_\varphi b}{g(\partial_\varphi N, \partial_\varphi N)},
\ee
and
\be\label{lxx8:7}
\gamma'\left(\frac{\partial}{\partial \partial_\varphi u},\frac{\partial}{\partial \partial_\varphi u}\right)= b^2.
\ee
Note that we have:
$$1-g(N,N')^2=(1-g(N,N'))(1+g(N,N'))=\frac{g(N-N',N-N')}{2}\frac{g(N+N',N+N')}{2}$$
which together with \eqref{ad1}  and the fact that $\norm{\po N}_{L^\infty}\les 1$ yields:
\be\label{lxx8:8}
1-g(N,N')^2\sim |\o-\o'|^2.
\ee
Now, since $\norm{g(\partial_\varphi N, \partial_\varphi N)-1}_{L^\infty}\les \ep$, $\norm{b-1}_{L^\infty}\les \ep$, $\norm{\po b}_{L^\infty}\les \ep$, and in view of \eqref{lxx8:2} and \eqref{lxx8:5}-\eqref{lxx8:8}, we have:
$$\gamma'\left(\frac{\partial}{\partial u},\frac{\partial}{\partial u}\right)\sim \frac{1}{|\o-\o'|^2},\, \gamma'\left(\frac{\partial}{\partial \partial_\varphi u},\frac{\partial}{\partial u}\right)=O(\ep),\, \gamma'\left(\frac{\partial}{\partial \partial_\varphi u},\frac{\partial}{\partial \partial_\varphi u}\right)=1+O(\ep).$$
This yields the following estimate for the determinant $|\gamma'|$:
\be\label{lxx8:9}
|\gamma'|\sim \frac{1}{|\o-\o'|^2}.
\ee

Since $|\gamma'|\neq 0$ in view of \eqref{lxx8:9},  $(u,\partial_\varphi u)$ is a coordinate system on $P_{t,u'}$. Note also that this coordinate system is global. Indeed, $t, u$ and $\partial_\varphi u$ are defined everywhere on $\mathcal{M}$, and thus everywhere on $\H_{u'}$, so we only need to show that $(t, u, \partial_\varphi u)$ is one-to-one on $\H_{u'}$. $t$ and $u$ being fixed, this is equivalent to check the injectivity of $\partial_\varphi u$ on $\ptu\cap P_{t,u'}$. Next, we check the injectivity of $\partial_\varphi u$ on $\ptu\cap P_{t,u'}$. Let $\ell$ a curve in $\ptu\cap P_{t,u'}$ parametrized by arc length. We have:
\bea\label{lxx8:10}
\frac{d}{d\sigma}[\partial_\varphi u(\ell(\sigma),\o)]&=&g(\nabla \partial_\varphi u, \dot{\ell})\\
\nn &=& g(-b^{-2}\partial_\varphi bN+b^{-1}\partial_\varphi N, \dot{\ell})\\
\nn &=& b^{-1}g(\partial_\varphi N, \dot{\ell}),
\eea
where we used in the last equality the fact that $\ell$ is a curve in $\ptu\cap P_{t,u'}$ which yields:
$$g(\ell,N)=g(\ell,N')=0.$$
Note that this implies the fact that: 
$$N,\,\, \frac{N'-g(N,N')N}{\sqrt{1-g(N,N')^2}},\,\, \dot{\ell}$$
forms an orthonormal basis of $\Sigma_t$. Now, we have:
$$g(\partial_\varphi N,N)=0$$
and:
\bee
\left|g\left(\partial_\varphi N, \frac{N'-g(N,N')N}{\sqrt{1-g(N,N')^2}}\right)\right| & = & \left|\frac{g(N,N')g(\partial_\varphi N,N'-N)}{\sqrt{1-g(N,N')^2}}\right|\\
& \les & \ep+|\o-\o'|,
\eee
where we used \eqref{lxx8:2} and \eqref{lxx8:8} in the last inequality. Since $g(\partial_\varphi N,\partial_\varphi N)=1+O(\ep)$, and since 
$$N,\,\, \frac{N'-g(N,N')N}{\sqrt{1-g(N,N')^2}},\,\, \dot{\ell}$$
forms an orthonormal basis of $\Sigma_t$, we deduce:
$$g(\partial_\varphi N, \dot{\ell})\neq 0$$
which together with \eqref{lxx8:10} and the fact that $b\sim 1$ yields:
$$\frac{d}{d\sigma}[\partial_\varphi u(\ell(\sigma),\o)]\neq 0\textrm{ for all }\sigma.$$
In particular, $\partial_\varphi u$ is one-to-one along $\ell$ which implies that $\partial_\varphi u$ is one-to-one on any connex component of $\ptu\cap P_{t,u'}$. 

Thus, to conclude that $\partial_\varphi u$ is one-to-one on $\ptu\cap P_{t,u'}$, is suffices to show that $\ptu\cap P_{t,u'}$ is connex. Assume for some $0\leq t_0\leq 1$ that $P_{t_0,u}\cap P_{t_0,u'}$ is connex. Note that on $\H_{u'}$, we have:
$$\left|\frac{\partial u}{\partial t}\right|=|g(L,L')|=1-g(N,N')=\half g(N-N',N-N')\gtrsim |\o-\o'|^2$$
where we used \eqref{ad1} in the last inequality. Thus, we have:
$$\frac{\partial u}{\partial t}\neq 0$$
and the implicit function theorem implies that in a neighborhood of $t=t_0$ of size depending only on $|\o-\o'|^2$ (but not on $t_0$), $P_{t,u}\cap P_{t,u'}$ is the image of $P_{t_0,u}\cap P_{t_0,u'}$ by a smooth map. Thus $P_{t,u}\cap P_{t,u'}$ is connex for $t$ in a neighborhood of $t=t_0$ of size depending only on $|\o-\o'|^2$. Therefore, if $P_{0,u}\cap P_{0,u'}$ is connex, applying the implicit function theorem successively $O(|\o-\o'|^{-2})$, we obtain that $\ptu\cap P_{t,u'}$ is connex for all $0\leq t\leq 1$. Now,  $P_{0,u}\cap P_{0,u'}$ is connex as an easy consequence of the construction in \cite{param1} on the initial slice $t=0$. Therefore, $\ptu\cap P_{t,u'}$ is connex for all $0\leq t\leq 1$.

Finally, we have obtained the fact that $(t,u, \partial_\varphi u)$ is a global coordinate system on $\H_{u'}$. Now, we use it to estimate the norm of a scalar $f$ in $L^p(\H_{u'})$ for $2\leq p<+\infty$. Let $u_0$ a real number. We have:
\bea
\label{lxx8:11}\norm{f}^p_{L^p(\H_{u'=u_0})}&=&\int_0^1\int  |f|^p\sqrt{|\gamma'|}d\partial_\varphi u du dt\\
\nn&\lesssim & \frac{1}{|\o-\o'|}\left(\int_0^1\int |f|^pd\partial_\varphi u du dt\right),
\eea
where we used \eqref{lxx8:9} in the last inequality. Note that we have on $u'=u_0$ the estimate:
$$|u-u_0|= |u-u'|\leq \norm{\po u}_{L^\infty}|\o-\o'|\les |\o-\o'|$$
which together with \eqref{lxx8:11} yields:
\be\label{lxx8:12}
\norm{f}^p_{L^p(\H_{u'=u_0})}\lesssim  \frac{1}{|\o-\o'|}\left(\int_0^1\int _{u_0-|\o-\o'|}^{u_0+|\o-\o'|}\int |f|^pd\partial_\varphi u du dt\right).
\ee
Next, recall the global coordinate system $\Phi_{t,\o}$ on $\Sigma_t$ introduced in \eqref{gl1}. Since $\po u=(\partial_\varphi u, \partial_\psi u)$, we have in view of \eqref{lxx8:12}:
\be\label{lxx8:13}
\norm{f}^p_{L^p(\H_{u'=u_0})}\lesssim  \frac{1}{|\o-\o'|}\left(\int_0^1\int _{u_0-|\o-\o'|}^{u_0+|\o-\o'|}\int_{y_1} \sup_{y_2}|f(\Phi^{-1}_{t,\o}(u,y_1,y_2))|^pdy_1 du dt\right).
\ee
From a standard estimate in $\R^2$, we have:
\bea\lab{tolbiac}
&&\int_{y_1} \sup_{y_2}|f(\Phi^{-1}_{t,\o}(u,y_1,y_2))|^pdy_1\\
\nn&\les& \left(\int_{y}|f(\Phi^{-1}_{t,\o}(u,y_1,y_2))|^{2(p-1)}dy_1dy_2\right)^{\half}\left(\int_{y}|\partial_{y_2}f(\Phi^{-1}_{t,\o}(u,y_1,y_2))|^2dy_1dy_2\right)^{\half}\\
\nn&\les& \left(\int_{\ptu}|f|^{2(p-1)}\dmt\right)^{\half}\left(\int_{\ptu}|\nabb f|^2\dmt\right)^{\half}
\eea
where we used the estimate \eqref{lxx2:2} for the coefficients of the induced metric $\gamma$ on $\ptu$ in the global coordinate system $\Phi^{-1}_{t,\o}(u,y_1,y_2)$. Together with \eqref{lxx8:13}, this yields:
\bee
\norm{f}^p_{L^p(\H_{u'=u_0})}&\lesssim&  \frac{1}{|\o-\o'|}\left(\int_0^1\int _{u_0-|\o-\o'|}^{u_0+|\o-\o'|}
\left(\int_{\ptu}|f|^{2(p-1)}\dmt\right)^{\half}\left(\int_{\ptu}|\nabb f|^2\dmt\right)^{\half}du dt\right)\\
&\lesssim&  \frac{1}{|\o-\o'|}\left(\int _{u_0-|\o-\o'|}^{u_0+|\o-\o'|}\norm{f}^{p-1}_{\lh{2(p-1)}}\norm{\nabb f}_{\lh{2}}du\right)\\
&\lesssim&  \norm{f}_{\li{\infty}{2(p-1)}}^{p-1}\norm{\nabb f}_{\li{\infty}{2}}.
\eee
Since this holds for any real number $u_0$, we take the supremum which yields:
$$\norm{f}^p_{L^{\infty}_{u'}L^p(\H_{u'})}\lesssim   \norm{f}_{\li{\infty}{p-1}}^{p-1}\norm{\nabb f}_{\li{\infty}{2}}.$$
Finally, let $F$ a tensor. Applying the previous inequality to $f=|F|$, we obtain
$$\norm{F}^p_{L^{\infty}_{u'}L^p(\H_{u'})}\lesssim   \norm{F}_{\li{\infty}{p-1}}^{p-1}\norm{\nabb F}_{\li{\infty}{2}}.$$
This concludes the proof of Lemma \ref{lemma:xx8}.

\subsection{Proof of Lemma \ref{lemma:nicosi}}\lab{sec:gowinda10}

Let $2\leq r<+\infty$. Then, we have
$$\norm{FH}_{\tx{r}{\infty}}\lesssim \norm{F}_{\tx{2r}{\infty}}\norm{H}_{\tx{2r}{\infty}}\lesssim\ep.$$
Thus, it suffices to bound $F\nabb H$ and $F\nabb H$ in $L^r_tB^0_{2,1}(\ptu)$. These terms are treated exactly in the same way, so we focus on $F\nabb H$. We have
\be\lab{nicosi}
\norm{P_j(F\nabb H)}_{\tx{r}{2}}\lesssim\sum_l\norm{P_j(FP_l(\nabb H)}_{\tx{r}{2}}.
\ee

Next, we estimate the right-hand side of \eqref{nicosi}. Using the finite band property for $P_j$, we have
\bea\lab{nicosi1}
&&\norm{P_j(FP_l(\nabb H)}_{\tx{r}{2}}\\
\nn&\les& 2^{-j}\norm{\nabb(FP_l(\nabb H))}_{\tx{r}{2}}\\
\nn&\les& 2^{-j}\norm{\nabb(F)P_l(\nabb H)}_{\tx{r}{2}}+2^{-j}\norm{F\nabb P_l(\nabb H))}_{\tx{r}{2}}\\
\nn&\les& 2^{-j}\norm{\nabb(F)}_{\tx{2r}{2}}\norm{P_l(\nabb H)}_{\tx{2r}{\infty}}+2^{-j}\norm{F}_{\tx{2r}{\infty}}\norm{\nabb P_l(\nabb H))}_{\tx{2r}{2}}\\
\nn&\les& 2^{-j+l}\norm{P_l(\nabb H)}_{\tx{3r}{2}},
\eea
where we used in the last inequality the finite band property and the sharp Bernstein inequality for tensors \eqref{eq:strongberntensor} for $P_l$, and the assumptions on $F$. Also, we have
\bee
\norm{P_j(FP_l(\nabb H)}_{\tx{r}{2}}&\les& 2^{-2l}\norm{P_j(F\lap P_l(\nabb H)}_{\tx{r}{2}}\\
&\les& 2^{-2l}\norm{P_j\divb(F\nabb P_l(\nabb H)}_{\tx{r}{2}}+2^{-2l}\norm{P_j(\nabb F\nabb P_l(\nabb H)}_{\tx{r}{2}}\\
&\les& 2^{j-2l}\norm{F\nabb P_l(\nabb H)}_{\tx{r}{2}}+2^{j-2l}\norm{\nabb F\nabb P_l(\nabb H)}_{\tx{2r}{1}},\\
\eee
where we used in the last inequality the finite band property and the dual of the sharp Bernstein inequality for tensors \eqref{eq:strongberntensor} for $P_j$. We obtain
\bea\lab{nicosi2}
&&\norm{P_j(FP_l(\nabb H)}_{\tx{r}{2}}\\
\nn&\les& 2^{j-2l}\norm{F}_{\tx{2r}{\infty}}\norm{\nabb P_l(\nabb H)}_{\tx{2r}{2}}+2^{j-2l}\norm{\nabb F}_{\tx{4r}{2}}\norm{\nabb P_l(\nabb H)}_{\tx{4r}{2}}\\
\nn&\les& 2^{j-l}\ep\norm{\nabb P_l(\nabb H)}_{\tx{4r}{2}},
\eea
where we used in the last inequality  the finite band property for $P_l$, and the assumptions on $F$. Finally, using  \eqref{nicosi1} for $l\leq j$ and \eqref{nicosi2} for $l>j$, we obtain
$$\norm{P_j(FP_l(\nabb H)}_{\tx{r}{2}}\les 2^{-|j-l|}\ep\norm{\nabb P_l(\nabb H)}_{\tx{4r}{2}},$$
which together with \eqref{nicosi} and the assumption on $H$ implies
$$\norm{F\nabb H}_{L^r_tB^0_{2,1}(\ptu)}\lesssim\ep.$$
This concludes the proof of the lemma.

\subsection{Proof of Lemma \ref{bis:lemma:xx5}}\lab{bis:sec:xx5}

Note that it suffices to prove for any $l\geq 0$ the estimate
\be\lab{gog}
\norm{[\po, P_l]f}_{L^\infty_{u'}L^2(\H_{u'})}\les 2^{\frac{l}{2}}\ep,
\ee
provided $f$ satisfies the assumptions of Lemma \ref{bis:lemma:xx5}. Let $W(\tau)$ solution of 
\begin{equation}\label{gog1}
(\partial_{\tau}-\lap)W(\tau)=[\po,\lap]U(\tau)f,\,W(0)=0.
\end{equation}
Then, we have
\begin{equation}\label{gog2}
[\po, P_l]f=\int_0^\infty m_l(\tau)W(\tau) d\tau.
\end{equation}
Assume that we have the following decomposition for $W$:
\be\lab{gog3}
W=W_1+W_2,
\ee
where $W_1$ and $W_2$ satisfy respectively 
\be\lab{gog3bis}
\sup_\tau\norm{W_1(\tau)}_{\lh{2}}+\norm{\nabb W_1(\c)}_{L^2_\tau\lh{2}}\les\ep,
\ee
and 
\be\lab{gog3ter}
\norm{W_2(\tau)}_{\lh{2}}+\sqrt{\tau}\norm{\nabb W_2(\c)}_{L^2_\tau\lh{2}}\les\ep.
\ee
Then, \eqref{gog3}, \eqref{gog3bis}, \eqref{gog3ter} together with \eqref{gog2} yields:
\bea\lab{gog3:4}
\norm{[\po, P_l]f}_{\li{\infty}{2}}&\les& \sup_u\int_0^\infty m_l(\tau)\norm{W(\tau)}_{\lh{2}} d\tau\\
\nn&\les& \sup_u\int_0^\infty m_l(\tau)\norm{W_1(\tau)}_{\lh{2}} d\tau+\sup_u\int_0^\infty m_l(\tau)\norm{W_2(\tau)}_{\lh{2}} d\tau\\
\nn&\les& \ep\int_0^\infty m_l(\tau)d\tau\\
\nn&\les& \ep,
\eea
and
\bea\lab{gog3:5}
&&\norm{\nabb[\po, P_l]f}_{\li{\infty}{2}}\\
\nn&\les& \sup_u\int_0^\infty m_l(\tau)\norm{\nabb W(\tau)}_{\lh{2}} d\tau\\
\nn&\les& \sup_u\int_0^\infty m_l(\tau)\norm{\nabb W_1(\tau)}_{\lh{2}} d\tau+\sup_u\int_0^\infty m_l(\tau)\norm{\nabb W_2(\tau)}_{\lh{2}} d\tau\\
\nn&\les& \ep\left(\int_0^\infty m_l(\tau)^2d\tau\right)^{\frac{1}{2}}+\ep\int_0^\infty m_l(\tau)\tau^{\frac{1}{2}}d\tau\\
\nn&\les& 2^l\ep.
\eea
\eqref{gog3:4} and \eqref{gog3:5} together with Lemma \ref{lemma:xx8} yield
$$\norm{[\po, P_l]f}_{L^\infty_{u'}L^2(\H_{u'})}\les \norm{[\po, P_l]f}_{\li{\infty}{2}}^{\frac{1}{2}}\norm{\nabb[\po, P_l]f}_{\li{\infty}{2}}^{\frac{1}{2}}\les 2^{\frac{l}{2}}\ep,$$
which is \eqref{gog}. Thus it remains to prove \eqref{gog3} \eqref{gog3bis} \eqref{gog3ter}. 

We first precise our choice for $W_1$ and $W_2$. Let $h$ a scalar on function on $\Sigma_t$. Then, we have the following commutator formula
\begin{equation}\label{gog5}
[\po,\lap]h= -2\nabb_{\po N}\nabla_Nh+2\th(\po N,\nabb h) -\trt\nabb_{\po N}h -\po\trt\nabla_Nh.
\end{equation}
\eqref{gog5} is in the spirit of section \ref{sec:commutomega}. We refer to   section 5.1.1 of \cite{param1} for a proof. 
We have
\be\lab{gog5bis}
\nabla_{bN}U(\tau)f=U(\tau)\nabla_{bN}f+V(\tau),
\ee
where $V(\tau)$ is the solution of
\begin{equation}\label{gog5ter}
(\partial_{\tau}-\lap)V(\tau)=[\nabla_{bN},\lap]U(\tau)f,\,V(0)=0.
\end{equation}
In view of \eqref{gog5} and \eqref{gog5bis}, we deduce
\bea\lab{gog5:4}
\nn[\po,\lap]U(\tau)f&=& -2\nabb_{\po N}(b^{-1}U(\tau)\nabla_{bN}f)-2\nabb_{\po N}(b^{-1}V(\tau))+2\th(\po N,\nabb U(\tau)f)\\
&&-\trt\nabb_{\po N}U(\tau)f -b^{-1}\po\trt U(\tau)\nabla_{bN}f -b^{-1}\po\trt V(\tau).
\eea
We choose $W_1$ and $W_2$ solution of the following equations
\bea\label{gog5:5}
\nn(\partial_{\tau}-\lap)W_1(\tau)&=&-2\divb(\po N b^{-1}U(\tau)\nabla_{bN}f)-2\nabb_{\po N}(b^{-1}V(\tau))\\
\nn&&+2\th(\po N,\nabb U(\tau)f) -\trt\nabb_{\po N}U(\tau)f-b^{-1}\po\trt V(\tau),\\
W_1(0)&=&0,
\eea
and 
\bea\label{gog5:6}
(\partial_{\tau}-\lap)W_2(\tau)&=& b^{-1}(2\divb(\po N)-\po\trt) U(\tau)\nabla_{bN}f, \\
\nn W_2(0)&=&0.
\eea
In view of \eqref{gog1}, \eqref{gog5:4}, \eqref{gog5:5} and \eqref{gog5:6}, we have \eqref{gog3}. Thus, it remains to prove the estimate \eqref{gog3bis} for $W_1$ and the estimate \eqref{gog3ter} for $W_2$. We start with the estimate \eqref{gog3bis}. The energy estimate \eqref{heatF2} implies:
\begin{equation}\label{gog4}
\norm{W_1(\tau)}^2_{\lh{2}}+\int_0^\tau\norm{\nabb W_1(\tau')}^2_{\lh{2}}d\tau'=\int_0^\tau\int_{\H_u} W_1(\tau')(\partial_{\tau}-\lap)W_1(\tau')\dmt dt d\tau'.
\end{equation}
In view of \eqref{gog5:5}, we obtain after integration by parts
\bee
&&\int_0^\tau\int_{\H_u} W_1(\tau')(\partial_{\tau}-\lap)W_1(\tau')f\dmt dt d\tau'\\
&\les& \norm{b^{-1}\po N}_{L^\infty}\int_0^\tau\norm{\nabb W_1(\tau')}_{\lh{2}}\norm{U(\tau')\nabla_{bN}f}_{\lh{2}}d\tau'\\
&&+\norm{b^{-1}}_{L^\infty}(\norm{\po N}_{L^\infty}+\norm{\divb(\po N)}_{\tx{\infty}{2}}+\norm{\po\trt}_{\tx{\infty}{2}})\\
&&\times\int_0^\tau\norm{\nabb W_1(\tau')}_{\lh{2}}\norm{\nabb V(\tau')}_{\lh{2}}d\tau'\\
&&+ \norm{\po N}_{L^\infty}\norm{\th}_{\tx{\infty}{4}}\int_0^\tau\norm{\nabb W_1(\tau')}_{\lh{2}}\norm{\nabb U(\tau')f}_{\lh{2}}d\tau'\\
&\les& \int_0^\tau\norm{\nabb W_1(\tau')}_{\lh{2}}(\norm{U(\tau')\nabla_{bN}f}_{\lh{2}}+\norm{\nabb V(\tau')}_{\lh{2}}+\norm{\nabb U(\tau')f}_{\lh{2}})d\tau', 
\eee
where we used in the last inequality the fact that $\th=\chi+\eta$ (see \eqref{def:theta}), the estimates \eqref{esttrc} \eqref{esthch} for $\chi$, the estimate \eqref{estk} for $k$, the estimate \eqref{estb} for $b$, the estimates \eqref{estNomega} \eqref{estricciomega} for $\po N$, and the estimate \eqref{estricciomega} for $\po\chi$. Together with \eqref{gog4}, we deduce
\bea\label{gog6}
&&\norm{W_1(\tau)}^2_{\lh{2}}+\int_0^\tau\norm{\nabb W_1(\tau')}^2_{\lh{2}}d\tau'\\
\nn&\lesssim &  \int_0^\tau(\norm{U(\tau')\nabla_{bN}f}^2_{\lh{2}}+\norm{\nabb V(\tau')}^2_{\lh{2}}+\norm{\nabb U(\tau')f}^2_{\lh{2}})d\tau'.
\eea
Next, we evaluate the right-hand side of \eqref{gog6}. The heat flow estimate \eqref{eq:l2heat1} yields
\be\lab{gog7}
\int_0^\tau\norm{\nabb U(\tau')}^2_{\lh{2}}d\tau'\les\norm{f}^2_{\li{\infty}{2}}\les\ep^2,
\ee
where we used in the last inequality the assumptions on $f$. A heat flow estimate yields
\be\lab{gog11}
\int_0^\tau\norm{U(\tau')\nabla_{bN}f}^2_{\lh{2}}d\tau'\les \norm{\La^{-1}(\nabla_{bN}f)}_{\li{\infty}{2}}^2\les\ep^2,
\ee
where we used in the last inequality the assumptions on $f$. Also, as a consequence of the estimate \eqref{lxx3:11bis} which will be proved later, we have 
\be\lab{gog12}
 \int_0^\tau\norm{\nabb V(\tau')}^2_{\lh{2}}d\tau'\les\no(f)^2\les\ep^2,
\ee
where we used in the last inequality the assumptions on $f$. Finally, \eqref{gog6}, \eqref{gog7}, \eqref{gog11} and \eqref{gog12} imply the desired estimate \eqref{gog3bis}. 

It remains to prove the estimate \eqref{gog3ter}. Using \eqref{gog5:6} together with Duhamel's formula, we have
\be\lab{goggog}
W_2(\tau)=\int_0^\tau U(\tau-\s)\left[b^{-1}(2\divb(\po N)-\po\trt) U(\s)\nabla_{bN}f\right]d\s.
\ee
Using the Gagliardo-Nirenberg inequality \eqref{eq:GNirenberg} and the heat flow estimate \eqref{eq:l2heat2}, we have for any scalar $h$ and any $2\leq p<+\infty$
\be\lab{gog19}
\norm{U(\tau)h}_{\lpt{p}}\les \frac{1}{\tau^{\frac{1}{2}-\frac{1}{p}}}\norm{h}_{\lpt{2}}.
\ee
In view of the formula \eqref{goggog} of $W_2$, and using the dual of \eqref{gog19}, we have:
\bea\lab{gog20}
\nn\norm{W_2(\tau)}_{\lh{2}}&\les&  \int_0^\tau \norm{U(\tau-\s)\left[b^{-1}(2\divb(\po N)-\po\trt) U(\s)\nabla_{bN}f\right]}_{\lh{2}}d\s\\
\nn&\les&  \int_0^\tau \frac{1}{(\tau-\s)^{\frac{1}{4}}}\norm{b^{-1}(2\divb(\po N)-\po\trt) U(\s)\nabla_{bN}f}_{\tx{2}{\frac{4}{3}}}d\s\\
\nn&\les&  \int_0^\tau \frac{1}{(\tau-\s)^{\frac{1}{4}}}\norm{b^{-1}(2\divb(\po N)-\po\trt)}_{\tx{\infty}{2}}\norm{U(\s)\nabla_{bN}f}_{\tx{2}{4}}d\s\\
&\les& \ep\int_0^\tau \frac{1}{(\tau-\s)^{\frac{1}{4}}}\frac{1}{\s^{\frac{1}{4}}}\normm{U\left(\frac{\s}{2}\right)\nabla_{bN}f}_{\lh{2}}d\s,
\eea
where we used in the last inequality the fact that 
$$U(\s)=U\left(\frac{\s}{2}\right)U\left(\frac{\s}{2}\right),$$
\eqref{gog19} with $p=4$, the fact that $\th=\chi+\eta$ (see \eqref{def:theta}), the estimates \eqref{esttrc} \eqref{esthch} for $\chi$, the estimate \eqref{estk} for $k$, the estimate \eqref{estb} for $b$, the estimates \eqref{estNomega} \eqref{estricciomega} for $\po N$, and the estimate \eqref{estricciomega} for $\po\chi$. The heat flow estimate \eqref{eq:l2lambda1} and \eqref{gog20} yield
\bea\lab{gog22}
\norm{W_2(\tau)}_{\lh{2}}&\les& \ep\left(\int_0^\tau \frac{1}{(\tau-\s)^{\frac{1}{4}}}\frac{1}{\s^{\frac{3}{4}}}d\s\right)\norm{\La^{-1}\nabla_{bN}f}_{\lh{2}}\\
\nn&\les& \ep,
\eea
where we used in the last inequality the assumptions for $f$. Next, we estimate $\nabb W_2$. Using the fact that 
$$U(\tau-\s)=U\left(\frac{\tau}{2}-\frac{\s}{2}\right)U\left(\frac{\tau}{2}-\frac{\s}{2}\right),$$
we obtain
\bee
\norm{\nabb W_2(\tau)}_{\lh{2}}&\les&  \int_0^\tau \norm{\nabb U(\tau-\s)\left[b^{-1}(2\divb(\po N)-\po\trt) U(\s)\nabla_{bN}f\right]}_{\lh{2}}d\s\\
&\les&  \int_0^\tau  \frac{1}{\sqrt{\tau-\s}}\normm{U\left(\frac{\tau}{2}-\frac{\s}{2}\right)\left[b^{-1}(2\divb(\po N)-\po\trt) U(\s)\nabla_{bN}f\right]}_{\lh{2}}d\s,
\eee
where we used in the last inequality the estimate \eqref{eq:l2heat2} for the heat flow. Then, arguing as for the proof of \eqref{gog22}, and noticing that we have:
$$\int_0^\tau \frac{1}{(\tau-\s)^{\frac{3}{4}}}\frac{1}{\s^{\frac{3}{4}}}d\s\les \tau^{-\frac{1}{2}},$$
we obtain:
\be\lab{gog23}
\norm{\nabb W_2(\tau)}_{\lh{2}}\les \tau^{-\frac{1}{2}}\ep.
\ee
Finally, \eqref{gog22} and \eqref{gog23} imply the desired estimate \eqref{gog3ter}. This concludes the proof of Lemma \ref{bis:lemma:xx5}. 

\subsection{Proof of Lemma \ref{lemma:assemblee}}\lab{sec:assemblee}

We start with the estimate for $\nabb b$. We have
\be\lab{jlevy}
\ddb_{bN}\nabb b=\nabb (\nabla_{bN}(b))+[\ddb_{bN},\nabb]b,
\ee
with 
$$h_1=\nabla_{bN}(b)\textrm{ and }H_2=[\ddb_{bN},\nabb]b.$$
In view of the commutator formula \eqref{comm7}, we have
\bea\lab{jlevy1}
\nn\norm{h_1}_{\li{\infty}{2}}+\norm{H_2}_{\tx{2}{\frac{4}{3}}}&\les& \norm{\dd b}_{\li{\infty}{2}}+\norm{b(\chi+k)\nabb b}_{\tx{2}{\frac{4}{3}}}\\
\nn&\les & \norm{\dd b}_{\li{\infty}{2}}+\norm{b}_{L^\infty}(\norm{\chi}_{\tx{\infty}{4}}+\norm{k}_{\tx{\infty}{4}})\norm{\nabb b}_{\lh{2}}\\
&\les& \ep,
\eea
where we used in the last inequality the estimate \eqref{estb} for $b$, the estimate \eqref{estk} for $k$ and the estimates \eqref{esttrc} \eqref{esthch} for $\chi$.  

Next, we consider the estimate for $\z$. In view of the identity \eqref{etaa}, we have
\be\lab{jlevy2}
\ddb_{bN}\z= \nabb h_3+H_4,
\ee
with 
$$h_3=b^{-1}\nabla_{bN}(b)=b^{-1}h_1\textrm{ and }H_2=b^{-1}[\ddb_{bN},\nabb]b+\ddb_{bN}\kep=b^{-1}H_2+\ddb_{bN}\kep.$$
We have
\bea\lab{jlevy3}
\nn\norm{h_3}_{\li{\infty}{2}}+\norm{H_4}_{\tx{2}{\frac{4}{3}}}&\les& \norm{b^{-1}h_1}_{\li{\infty}{2}}+\norm{b^{-1}H_2}_{\tx{2}{\frac{4}{3}}}+\norm{\ddb_{bN}\kep}_{\tx{2}{\frac{4}{3}}}\\
\nn&\les & \norm{b^{-1}}_{L^\infty}(\norm{h_1}_{\li{\infty}{2}}+\norm{H_2}_{\tx{2}{\frac{4}{3}}})+\norm{b}_{L^\infty}\norm{\ddb_N\z}_{\lh{2}}\\
&\les& \ep,
\eea
where we used in the last inequality the estimate \eqref{estk} for $\kep$ and the estimate \eqref{jlevy1} for $h_1$ and $H_2$. Finally, \eqref{jlevy}-\eqref{jlevy3} yields the desired decompositions. This concludes the proof of Lemma \ref{lemma:assemblee}.

\subsection{Proof of Lemma \ref{lemma:assemblee1}}\lab{sec:assemblee1}

Recall the transport equation \eqref{dw} for $\po b$
$$L(\po b)=-b\z_{\po N}-\po(b)\db-\kepb_{\po N}b.$$
We differentiate with respect to $\lb$. This yields
\bee
L(\lb\po b)+[\lb, L]\po b&=&-b\ddb_{\lb}\z_{\po N}-\lb(b)\z_{\po N}-b\z_{\ddb_{\lb}\po N}-\lb(\po(b))\db-\po(b)\lb(\db)\\
&&-\ddb_{\lb}\kepb_{\po N}b-\kepb_{\ddb_{\lb}\po N}b-\kepb_{\po N}\lb(b).
\eee
Together with the commutator formula \eqref{comm3}, we obtain
\be\lab{jlevy7}
L(\lb\po b)=-b\ddb_{\lb}\z_{\po N}+f,
\ee
where the scalar $f$ is given by
\bee
f &=&-\lb(b)\z_{\po N}-b\z_{\ddb_{\lb}\po N}-\po(b)\lb(\db)-\ddb_{\lb}\kepb_{\po N}b-\kepb_{\ddb_{\lb}\po N}b-\kepb_{\po N}\lb(b)\\
&&- (\d + n^{-1} \nab_Nn)L(\po b)-2(\z-\zb)\c \nabb\po b.
\eee
$f$ satisfies the following estimate
\bea\lab{jlevy8}
&&\norm{f}_{\li{\infty}{2}}\\
\nn&\les& \Big(\norm{\lb(b)}_{\xt{2}{\infty}}+\norm{\dd\po N}_{\xt{2}{\infty}}+\norm{L(\po b)}_{\xt{2}{\infty}}+\norm{\nabb\po b}_{\xt{2}{\infty}}\Big)\\
\nn&&\times\Big(1+\norm{\z}_{\xt{\infty}{2}}+\norm{\kepb}_{\xt{\infty}{2}}+\norm{\d}_{\xt{\infty}{2}}+\norm{n^{-1} \nab_Nn}_{L^\infty}+\norm{\zb}_{\xt{\infty}{2}}\Big)\\
\nn&&\times\Big(1+\norm{\po N}_{L^\infty}+\norm{b}_{L^\infty}+\norm{\po(b)}_{L^\infty}\Big)+\norm{\ddb_{\lb}\kepb}_{\li{\infty}{2}}\norm{\po N}_{L^\infty}\norm{b}_{L^\infty}\\
\nn&\les& \ep,
\eea
where we used in the last inequality the estimates \eqref{estn}-\eqref{estb} for $n$, $\kepb$, $\d$, 
$\zb$ and $b$, the estimate \eqref{estzeta} for $\z$, the estimate \eqref{estNomega} for $\po N$, and the estimate \eqref{estricciomega} for $\po N$ and $\po b$. 

In view of the identity \eqref{etaa}, we have
\bea\lab{jlevy9}
b\ddb_{\lb}\z_{\po N}&=&b\ddb_{\lb}(b^{-1}\nabb b+\kep)_{\po N}\\
\nn&=& (\ddb_{\lb}\nabb b)_{\po N}-b^{-2}\lb(b)\nabb_{\po N}b+\ddb_{\lb}\kep_{\po N}\\
\nn&=& \divb(\lb(b)\po N)+f_1,
\eea
where the scalar $f_1$ is given by
$$f_1=([\ddb_{\lb},\nabb]b)_{\po N}-\lb(b)\divb(\po N)-b^{-2}\lb(b)\nabb_{\po N}b+\ddb_{\lb}\kep_{\po N}.$$
In view of the definition of $f_1$, we have
\bee
\norm{f_1}_{\tx{2}{\frac{4}{3}}}&\les& \norm{[\ddb_{\lb},\nabb]b}_{\tx{2}{\frac{4}{3}}}\norm{\po N}_{L^\infty}+\norm{\lb(b)}_{\tx{\infty}{4}}\norm{\divb(\po N)}_{\tx{\infty}{2}}\\
&&+\norm{b^{-2}\po N}_{L^\infty}\norm{\lb(b)}_{\tx{\infty}{2}}\norm{\nabb b}_{\tx{\infty}{4}}+\norm{\ddb_{\lb}\kep}_{\li{\infty}{2}}\norm{\po N}_{L^\infty}\\
&\les& \norm{[\ddb_{\lb},\nabb]b}_{\tx{2}{\frac{4}{3}}}+\ep,
\eee
where we used in the last inequality the estimate \eqref{estb} for $b$, the estimate \eqref{estk} for $\kep$ and the estimates \eqref{estNomega} \eqref{estricciomega} for $\po N$. Together with the commutator formula \eqref{comm2}, we deduce
\bea\lab{jlevy10}
\norm{f_1}_{\tx{2}{\frac{4}{3}}}&\les& \norm{(\chb, \xib, b^{-1}\nabb b)}_{\tx{\infty}{2}}\norm{b}_{\tx{\infty}{4}}+\ep\\
\nn&\les&\ep,
\eea
where we used in the last inequality the estimates \eqref{estn}-\eqref{estzeta} for $b, \chb$ and $\xib$. 

In view of the transport equation \eqref{jlevy7} and the estimate for its initial data, we have
\be\lab{jlevy7bis}
nL(\lb\po b)= -\divb(n\lb(b)\po N)+f_2,
\ee
where $f_2$ is given by
$$f_2=\lb(b)\nabb_{\po N}n-nf_1+nf.$$
In view of the definition of $f_2$, we have
\bea\lab{jlevy7ter}
\norm{f_2}_{\tx{2}{\frac{4}{3}}}&\les& \norm{\lb(b)\nabb_{\po N}n}_{\tx{2}{\frac{4}{3}}}+\norm{nf_1}_{\tx{2}{\frac{4}{3}}}+\norm{nf}_{\tx{2}{\frac{4}{3}}}\\
\nn&\les& \norm{\lb(b)}_{\tx{\infty}{2}}\norm{\nabb n}_{\tx{2}{4}}\norm{\po N}_{L^\infty}+\norm{n}_{L^\infty}(\norm{f_1}_{\tx{2}{\frac{4}{3}}}+\norm{f}_{\tx{2}{\frac{4}{3}}})\\
\nn&\les& \ep,
\eea
where we used in the last inequality the estimate \eqref{estb} for $b$, the estimate \eqref{estn} for $n$, the estimate \eqref{estNomega} for $\po N$, the estimate \eqref{jlevy8} for $f$ and the estimate \eqref{jlevy10} for $f_1$. In view of the transport equation \eqref{jlevy7bis} and the estimate for its initial data, we have
\be\lab{jlevy11}
\norm{\La^{-1}(b\lb\po b)}_{\tx{\infty}{2}}\les \ep+\normm{\La^{-1}\left(b\int_0^t\divb(\lb(b)\po N)\right)}_{\tx{\infty}{2}}+\normm{\La^{-1}\left(b\int_0^tf_2\right)}_{\tx{\infty}{2}}.
\ee
Using the estimate \eqref{La5} for $\La^{-1}$ and the estimate \eqref{estimtransport1} for transport equations, we have
\be\lab{jlevy13}
\normm{\La^{-1}\left(b\int_0^tf_2\right)}_{\tx{\infty}{2}}\les\normm{b\int_0^tf_2}_{\tx{\infty}{\frac{4}{3}}}\les \norm{b}_{L^\infty}\norm{f_2}_{\tx{1}{\frac{4}{3}}}\les\ep,
\ee
where we used \eqref{jlevy7ter} and the estimate \eqref{estb} for $b$ in the last inequality. Finally, we define
$$w=\int_0^t\divb(\lb(b)\po N)$$
and the tensor $W$ solution to the following transport equation
$$\ddb_{nL}W-n\chi\c W=\lb(b)\po N, \, W=0\textrm{ on }P_{0,u}.$$
Then, Lemma \ref{lemma:commutdivb} implies
\bea\lab{jlevy14}
\norm{\divb(W)-w}_{\tx{\infty}{\frac{4}{3}}}&\les& \norm{\lb(b)\po N}_{\tx{1}{4}}\\
\nn&\les& \norm{\lb{b}}_{\tx{\infty}{4}}\norm{\po N}_{L^\infty}\\
\nn&\les& \ep,
\eea
where we used in the last inequality the estimate \eqref{estb} for $b$ and the estimate \eqref{estNomega} for $\po N$. Also, in view of the transport equation satisfied by $W$, the estimate \eqref{estimtransport1} for transport equations yields
\bee
\norm{W}_{\xt{\infty}{2}}&\les& \norm{n\chi\c W}_{\xt{2}{1}}+\norm{\lb(b)\po N}_{\xt{2}{1}}\\
&\les& \norm{n}_{L^\infty}\norm{\chi}_{\xt{\infty}{2}}\norm{W}_{\lh{2}}+\norm{\lb(b)}_{\tx{\infty}{2}}\norm{\po N}_{L^\infty}\\
&\les& \ep\norm{W}_{\tx{\infty}{2}}+\ep,
\eee
where we used in the last inequality the estimate \eqref{estn} for $n$, the estimates \eqref{esttrc} \eqref{esthch} for $\chi$, the estimate \eqref{estb} for $b$ and the estimate \eqref{estNomega} for $\po N$. We deduce
\be\lab{jlevy15}
\norm{W}_{\tx{\infty}{2}}\les\ep.
\ee
Using the estimates \eqref{La3} and \eqref{La5} for $\La^{-1}$, we have
\bee
\norm{\La^{-1}(bw)}_{\tx{\infty}{2}}&\les& \norm{\La^{-1}(b(w-\divb(W)))}_{\tx{\infty}{2}}+\norm{\La^{-1}(b\divb(W))}_{\tx{\infty}{2}}\\
&\les& \norm{b(w-\divb(W))}_{\tx{\infty}{\frac{4}{3}}}+\norm{\La^{-1}(\nabb(b)\c W)}_{\tx{\infty}{2}}+\norm{\La^{-1}\divb(bW)}_{\tx{\infty}{2}}\\
&\les& \norm{b}_{L^\infty}\norm{w-\divb(W)}_{\tx{\infty}{\frac{4}{3}}}+\norm{\nabb(b)\c W}_{\tx{\infty}{\frac{4}{3}}}+\norm{bW}_{\tx{\infty}{2}}\\
&\les& \norm{b}_{L^\infty}\norm{w-\divb(W)}_{\tx{\infty}{\frac{4}{3}}}+\norm{\nabb b}_{\tx{\infty}{4}}\norm{W}_{\tx{\infty}{2}}+\norm{b}_{L^\infty}\norm{W}_{\tx{\infty}{2}}\\
&\les& \ep,
\eee
where we used in the last inequality the estimate \eqref{estb} for $b$ and the estimates \eqref{jlevy14} and \eqref{jlevy15}. In view of the definition of $w$, and together with \eqref{jlevy11} and \eqref{jlevy13}, we finally obtain
\be\lab{jlevy16}
\norm{\La^{-1}(b\lb\po b)}_{\tx{\infty}{2}}\les \ep.
\ee
On the other hand, we have
\bea\lab{jlevy17}
\norm{\La^{-1}(bL\po b)}_{\tx{\infty}{2}}&\les& \norm{bL\po b}_{\tx{\infty}{2}}\\
\nn&\les& \norm{b}_{L^\infty}\norm{L\po b}_{\tx{\infty}{2}}\\
\nn&\les& \ep,
\eea
where we used in the last inequality the estimate \eqref{estb} for $b$ and the estimate \eqref{po12} for $\po b$. Recall that 
$$N=\frac{1}{2}(L-\lb),$$
which together with \eqref{jlevy16} and \eqref{jlevy17} implies
$$\norm{\La^{-1}(bN\po b)}_{\tx{\infty}{2}}\les \ep.$$
This concludes the proof the lemma \ref{lemma:assemblee1}.

\section{Appendix to section \ref{sec:commutatorest}}

\subsection{Proof of Proposition \ref{prop:gowinda11}}\lab{sec:gowinda11}

Using the definition \eqref{eq:LP} of $P_j$, we have:
\begin{equation}\label{ad24}
[\ddb_{bN},P_j]F=\int_0^\infty m_j(\tau)V(\tau) d\tau,
\end{equation}
where $V(\tau)$ satisfies:
\be\lab{ad25}
(\partial_{\tau}-\lap)V(\tau)=[\ddb_{bN},\lap]U(\tau)F,\,V(0)=0.
\ee
\eqref{ad24} yields:
\begin{equation}\label{aad26}
\norm{[\ddb_{bN},P_j]F}_{L^{\frac{4}{3}}_u\lpt{2}}\les \normm{\int_0^\infty m_j(\tau)\norm{V(\tau)}_{\lpt{2}} d\tau}_{L^{\frac{4}{3}}_u}.
\end{equation}

In view of \eqref{bonobo} and \eqref{aad26}, we have to estimate $\norm{V(\tau)}_{\lpt{2}}$. Let $a, p$ real numbers satisfying:
\begin{equation}\label{aad27}
\begin{array}{l}
\ds 0<a<\frac{1}{4},\,2<p<+\infty,\textrm{ such that }p<\min\left(\frac{8}{3},\frac{4}{2-a}\right).
\end{array}
\end{equation}
The energy estimate \eqref{eq:l2heat1bis} implies:
\begin{equation}\label{aad28}
\begin{array}{ll}
& \ds\norm{\La^{-a}V(\tau)}^2_{\lpt{2}}+\int_0^\tau\norm{\nabb\La^{-a}V(\tau')}^2_{\lpt{2}}d\tau'\\
\ds\lesssim &\ds\int_0^\tau\int_{\ptu} \La^{-2a}V(\tau')[\ddb{bN},\lap]U(\tau')F\dmt d\tau'.
\end{array}
\end{equation}

We need to estimate the commutator term $[\ddb{bN},\lap]U$. Using the definition of $\th$ \eqref{def:theta}, we may rewrite the commutator formula \eqref{comm7} 
for any m-covariant tensor $\Pi_{\und{A}}$ tangent to $\ptu$ as:
\bea
\nabb_B \nabb_{bN} \Pi_{\und{A}} - \nabb_{bN}\nabb_B \Pi_{\und{A}} &=&
b\th_{BC} \nabb_C \Pi_{\und{A}} 
\label{aad29}\\ 
\nn&+& b\sum_i (-\th_{BC}b^{-1}\nabb_Cb+\th_{BC}b^{-1}\nabb_Cb-k_{AB}k_{CN}+k_{BC}k_{AN}\\
&&-\half\in_{A_i C}{}^*(\b_B+\bb_B)) \Pi_{A_1..\Check{C}..A_m}.\nn
\eea
Using twice the commutator formula \eqref{aad29}, we have:
\be\lab{aad30}
[\ddb{bN},\lap]U=H\nabb^2U+G\nabb U+\nabb(GU)
\ee
where the tensors $H$ and $G$ are given by $H=b\th$ and $G=\th\c\nabb b+k\c k+b{}^*(\b+\bb)$. Using the curvature bound \eqref{curvflux1} for $\b, \bb$, the $L^\infty$ bound \eqref{estb} for $b$, the estimate 
\eqref{eqksit3} for $k$ on $\Sit$, and the bounds \eqref{ad13}-\eqref{ad15} for $b$ and $\th$ on $\Sit$, we obtain the following bound for $H$ and $G$:
\bea
\lab{aad31}\norm{\nabb H}_{\lsit{\infty}{2}}+\norm{G}_{\lsit{\infty}{2}}&\les& \norm{b}_{L^\infty}\norm{\nabb\th}_{\lsit{\infty}{2}}+\norm{\th}_{\lsit{\infty}{4}}\norm{\nabb b}_{\lsit{\infty}{4}}\\
\nn&&+\norm{k}_{\lsit{\infty}{4}}^2+\norm{b}_{L^\infty}(\norm{\b}_{\lsit{\infty}{2}}+\norm{\bb}_{\lsit{\infty}{2}})\\
\nn&\les&\ep.
\eea
Notice that the structure \eqref{aad30} \eqref{aad31} is completely analogous to \eqref{midi1} \eqref{midi2}. Therefore, proceeding as in \eqref{bis:ad47}, we obtain:
\begin{equation}\label{aad32}
\begin{array}{l}
\ds\int_0^\tau\int_{\ptu} \La^{-2a}V(\tau')[\ddb_{nL},\lap]U(\tau')F\dmt d\tau'\\
\ds\lesssim (\norm{\nabb H}_{\lpt{2}}+\norm{G}_{\lpt{2}})\int_0^\tau\norm{\nabb U(\tau')}_{\lpt{p}}\norm{\nabb\La^{-2a}V(\tau')}_{\lpt{2}}d\tau'.
\end{array}
\end{equation}
The Gagliardo-Nirenberg inequality \eqref{eq:GNirenberg}, the properties \eqref{La1} and \eqref{interpolLa} of $\Lambda$, and the Bochner inequality \eqref{vbochineq} for tensors yield:
\begin{equation}\label{aad33}
\begin{array}{l}
\ds\int_0^\tau\norm{\nabb U(\tau')}_{\lpt{p}}\norm{\nabb\La^{-2a}V(\tau')}_{\lpt{2}}d\tau'\\
\ds\lesssim 
\int_0^\tau\norm{\nabb U(\tau')}^{\frac{2}{p}}_{\lpt{2}}\norm{\nabb^2 U(\tau')}^{1-\frac{2}{p}}_{\lpt{2}}\norm{\La^{-a}V(\tau')}^{a}_{\lpt{p'}}\norm{\nabb\La^{-a}V(\tau')}^{1-a}_{\lpt{p'}}d\tau'\\
\ds\lesssim \int_0^\tau\norm{\nabb U(\tau')}^{\frac{2}{p}}_{\lpt{2}}\big(\norm{\lap U(\tau')}_{\lpt{2}}+\norm{K}_{\lpt{2}}\norm{\nabb U(\tau')}_{\lpt{2}}\\
\ds +\norm{K}^2_{\lpt{2}}\norm{U(\tau')}_{\lpt{2}}\big)^{1-\frac{2}{p}}\norm{\La^{-a}V(\tau')}^{a}_{\lpt{p'}}\norm{\nabb\La^{-a}V(\tau')}^{1-a}_{\lpt{p'}}d\tau'\\
\ds\lesssim \left(\left(1+\norm{K}_{\lpt{2}}^{2(1-\frac{2}{p})}\right)\int_0^\tau\norm{\nabb U(\tau')}^{2}_{\lpt{2}}d\tau'
+ \int_0^\tau \tau'\norm{\lap U(\tau')}^{2}_{\lpt{2}}d\tau'\right)^{\frac{1}{2}}\\
\ds\left(\frac{1}{2}\int_0^\tau\norm{\nabb\La^{-a}V(\tau')}^2_{\lpt{2}}d\tau'
+\int_0^\tau{\tau'}^{-\frac{2(p-2)}{ap}}\norm{\La^{-a}V(\tau')}^2_{\lpt{2}}d\tau'\right)^{\frac{1}{2}}
\end{array}
\end{equation}
which together with the estimates for the heat flow \eqref{eq:l2heat1} and \eqref{eq:l2heat2} implies:
\bea\label{aad34}
&&\ds\int_0^\tau\norm{\nabb U(\tau')}_{\lpt{p}}\norm{\nabb\La^{-2a}V(\tau')}_{\lpt{2}}d\tau'\\
\nn&\lesssim&\left(1+\norm{K}_{\lpt{2}}^{2(1-\frac{2}{p})}\right)\norm{F}_{\lpt{2}}\\
\nn&&\times\left(\int_0^\tau\norm{\nabb\La^{-a}V(\tau')}^2_{\lpt{2}}d\tau'+\int_0^\tau{\tau'}^{-\frac{2(p-2)}{ap}}\norm{\nabb\La^{-a}V(\tau')}^2_{\lpt{2}}d\tau'\right)^{\frac{1}{2}}.
\eea

Finally, the choice of $p$ \eqref{ad38bis}, \eqref{bis:ad46}, \eqref{bis:ad47} and \eqref{bis:ad49} implies:
\begin{equation}\label{aad35}
\begin{array}{l}
\ds\norm{\La^{-a}V(\tau)}^2_{\lpt{2}}+\int_0^\tau\norm{\nabb\La^{-a}V(\tau')}^2_{\lpt{2}}d\tau'\\
\ds\lesssim (\norm{\nabb H}_{\lpt{2}}+\norm{G}_{\lpt{2}})\left(1+\norm{K}_{\lpt{2}}^{2(1-\frac{2}{p})}\right)\norm{F}_{L^\infty_u\lpt{2}}.
\end{array}
\end{equation}
Using the interpolation inequality \eqref{interpolLa}, we obtain:
\begin{equation}\label{aad36}
\begin{array}{l}
\ds\int_0^{+\infty}\norm{V(\tau)}^{\frac{2}{a}}_{\lpt{2}}d\tau\lesssim \int_0^\tau \norm{\La^{-a}V(\tau')}^{\frac{2(1-a)}{a}}_{\lpt{2}}\norm{\nabb\La^{-a}V(\tau')}^2_{\lpt{2}}d\tau'
\\
\ds\lesssim (\norm{\nabb H}_{\lpt{2}}+\norm{G}_{\lpt{2}})\left(1+\norm{K}_{\lpt{2}}^{2(1-\frac{2}{p})}\right)\norm{F}_{L^\infty_u\lpt{2}}.
\end{array}
\end{equation}
The estimate \eqref{aad31} for $H$ and $G$ and the choice \eqref{aad27} for $p$,  yields:
\begin{equation}\label{aad37}
\begin{array}{l}
\ds\normm{\int_0^{+\infty}m_j(\tau)\norm{V(\tau)}_{\lpt{2}}d\tau}_{L^{\frac{4}{3}}_u}\lesssim 2^{ja}\normm{\left(\int_0^{+\infty}\norm{V(\tau)}^{\frac{2}{a}}_{\lpt{2}}d\tau\right)^{\frac{a}{2}}}_{L^{\frac{4}{3}}_u}\\
\ds\lesssim 2^{ja}(\norm{\nabb H}_{\lsit{\infty}{2}}+\norm{G}_{\lsit{\infty}{2}})\left(1+\norm{K}_{\lsit{\infty}{2}}^{2(1-\frac{2}{p})}\right)\norm{\nabla F}_{L^\infty_u\lpt{2}}\\
\lesssim 2^{ja}\norm{\nabla F}_{L^\infty_u\lpt{2}}.
\end{array}
\end{equation}
\eqref{aad37} and \eqref{aad26} yield
$$\norm{[\ddb_{bN},P_j]F}_{L^{\frac{4}{3}}_u\lpt{2}}\les 2^{ja}\norm{\nabla F}_{L^infty_u\lpt{2}}.$$
Taking the supremum in $t$ yields the desired estimate \eqref{bonobo}. This concludes the proof of the proposition.

\subsection{Proof of Proposition \ref{prop:gowinda12}}\lab{sec:gowinda12}

The proof of the estimate \eqref{zz1} being similar and slightly easier than the proof of \eqref{zz}, we focus on \eqref{zz}. In view of \eqref{ad24} \eqref{ad25}, we have:
\begin{equation}\label{lxx3:5}
[bN,P_l]f=\int_0^\infty m_l(\tau)V(\tau) d\tau,
\end{equation}
where $V(\tau)$ satisfies:
\be\lab{lxx3:6}
(\partial_{\tau}-\lap)V(\tau)=[bN,\lap]U(\tau)f,\,V(0)=0.
\ee
Assume that $V$ satisfies for all $\tau$
\be\label{lxx3:11bis}
\norm{V(\tau)}^2_{\lh{2}}+\int_0^\tau\norm{\nabb V(\tau')}^2_{\lh{2}}d\tau'\les \ep\no(f).
\ee
Then, in view of \eqref{lxx3:5}, we obtain
\bee
&&\norm{[bN, P_l]f}_{\lh{2}}+2^{-l}\norm{\nabb[bN, P_l]f}_{\lh{2}}\\
&\les& \int_0^\infty m_l(\tau)\norm{V(\tau)}_{\lh{2}} d\tau+2^{-l}\int_0^\infty m_l(\tau)\norm{\nabb V(\tau)}_{\lh{2}} d\tau\\
&\les& \ep\no(f)\left(\int_0^\infty m_l(\tau) d\tau+2^{-l}\left(\int_0^\infty m^2_l(\tau)d\tau\right)\right)\\
&\les& \ep\no(f),
\eee
which after taking the supremum in $u$ yields \eqref{zz}. Thus, it remains to prove \eqref{lxx3:11bis}. 

The energy estimate \eqref{heatF2} implies after integration along null geodesics:
\begin{equation}\label{lxx3:8}
 \ds\norm{V(\tau)}^2_{\lh{2}}+\int_0^\tau\norm{\nabb V(\tau')}^2_{\lh{2}}d\tau'\lesssim\ds\int_0^\tau\int_{\H_u} V(\tau')[bN,\lap]U(\tau')f\dmt d\tau'.
\end{equation}
We need to estimate the commutator term $[bN,\lap]U$. Using twice the commutator formula \eqref{comm7} together with the fact that $U(\tau)f$ is a scalar function, we have:
\be\lab{lxx3:9}
[bN,\lap]U=H\nabb^2U+G\nabb U
\ee
where the tensors $H$ and $G$ are given by $H=b(\chi+k)$ and $G=b\nabb\chi+b\nabb k+(\chi+k)\nabb b+\chi(\kep+\xib)+\chb\z+b{}^*(\b+\bb)$. Using the curvature estimate \eqref{curvflux1}, and the estimates \eqref{estn}-\eqref{estzeta} for $k, b, \chi, \z, \xib$ and $\chb$, we obtain the following bound for $H$ and $G$:
\bea
\nn\no(H)+\norm{G}_{\li{\infty}{2}}&\les& \norm{b}_{L^\infty}(\no(\chi)+\no(k))+\norm{(L,\nabb)b}_{\tx{2}{4}}(\norm{\chi}_{\tx{\infty}{4}}+\norm{k}_{\tx{\infty}{4}})\\
\nn&&+\norm{\chi}_{\tx{\infty}{4}}(\norm{\kep}_{\tx{\infty}{4}}+\norm{\xib}_{\tx{\infty}{4}})+\norm{\chb}_{\tx{\infty}{4}}\norm{\z}_{\tx{\infty}{4}}\\
\nn&&+\norm{\b}_{\li{\infty}{2}}+\norm{\bb}_{\li{\infty}{2}}\\
\lab{lxx3:10}&\les&\ep.
\eea
Using \eqref{lxx3:10}, we obtain:
\bee
&&\int_0^\tau\int_{\H_u} V(\tau')[bN,\lap]U(\tau')f\dmt d\tau'\\
\nn&\lesssim& \int_0^\tau \norm{H}_{\tx{\infty}{4}}\norm{\nabb^2U(\tau')}_{\lh{2}}\norm{V(\tau')}_{\tx{2}{4}}d\tau'\\
\nn&&+\int_0^1\int_0^\tau \norm{G}_{\lpt{2}}\norm{\nabb U(\tau')}_{\lpt{4}}\norm{V(\tau')}_{\lpt{4}}d\tau'dt\\
\nn&\lesssim& \ep\int_0^\tau \norm{\nabb^2U(\tau')}_{\lh{2}}\norm{\nabb V(\tau')}_{\lh{2}}d\tau'+\frac{\norm{G}^2_{\lh{2}}}{\ep}\normm{\int_0^\tau \norm{\nabb U(\tau')}^2_{\lpt{4}}d\tau'}_{L^\infty_t}\\
&&+\ep\int_0^\tau \norm{\nabb V(\tau')}^2_{\lh{2}}d\tau'\\
\nn&\lesssim& \ep\int_0^\tau \norm{\nabb^2U(\tau')}_{\lh{2}}\norm{\nabb V(\tau')}_{\lh{2}}d\tau'+\ep\normm{\int_0^\tau \norm{\nabb U(\tau')}^2_{\lpt{4}}d\tau'}_{L^\infty_t}\\
&&+\ep\int_0^\tau \norm{\nabb V(\tau')}^2_{\lh{2}}d\tau'
\eee
which together with \eqref{lxx3:8} implies:
\bea\label{lxx3:11}
&& \ds\norm{V(\tau)}^2_{\lh{2}}+\int_0^\tau\norm{\nabb V(\tau')}^2_{\lh{2}}d\tau' \\
\nn &\lesssim& \ep\int_0^\tau \norm{\nabb^2U(\tau')}^2_{\lh{2}}d\tau'+\ep\normm{\int_0^\tau \norm{\nabb U(\tau')}^2_{\lpt{4}}d\tau'}_{L^\infty_t}\\
\nn &\lesssim& \ep\int_0^\tau \norm{\lap U(\tau')}^2_{\lh{2}}d\tau'+\ep\normm{\int_0^\tau \norm{\nabb U(\tau')}^2_{\lpt{4}}d\tau'}_{L^\infty_t},
\eea
where we used the Bochner inequality for scalars \eqref{eq:Bochconseqbis} in the last inequality. Now, the energy estimates 
\eqref{eq:l2heat1} and \eqref{eq:l2heatnab} yield:
\bee
\int_0^\tau \norm{\lap U(\tau')}^2_{\lh{2}}d\tau'+\ep\normm{\int_0^\tau \norm{\nabb U(\tau')}^2_{\lpt{4}}d\tau'}_{L^\infty_t}\lesssim\norm{\nabb f}^2_{\lh{2}}+\norm{f}_{\tx{\infty}{4}}^2\lesssim \no(f),
\eee
Together with \eqref{lxx3:11}, we obtain
$$\norm{V(\tau)}^2_{\lh{2}}+\int_0^\tau\norm{\nabb V(\tau')}^2_{\lh{2}}d\tau' \lesssim\ep\no(f),$$
which is the desired estimate \eqref{lxx3:11bis}. This concludes the proof of the proposition. 

\subsection{Proof of Proposition \ref{prop:gowinda13}}\lab{sec:gowinda13}

The estimate of the first term in the right-hand side of \eqref{lxx4:4} being similar and slightly easier, we focus on the estimate of the second term involving $[bN,P_q]f$. In view of \eqref{lxx3:5} and \eqref{lxx3:6}, we have:
\begin{equation}\label{lxx4:5}
\norm{[bN,P_q]f}_{\li{\infty}{2}}\les \normm{\int_0^\infty m_q(\tau)\norm{V(\tau)}_{\lh{2}} d\tau}_{L^\infty_u},
\end{equation}
where $V(\tau)$ satisfies:
\be\lab{lxx4:6}
(\partial_{\tau}-\lap)V(\tau)=[bN,\lap]U(\tau)f,\,V(0)=0.
\ee

In view of \eqref{lxx4:5}, we have to estimate $\norm{V(\tau)}_{\lh{2}}$. Let $a, \delta$ real numbers satisfying:
\begin{equation}\label{lxx4:7}
\begin{array}{l}
\ds \half<a<1,\textrm{ and }0<\delta<a-\frac{1}{2}.
\end{array}
\end{equation}
The energy estimate \eqref{eq:l2heat1bis} implies:
\begin{equation}\label{lxx4:8}
\begin{array}{ll}
& \ds\norm{\La^{-a}V(\tau)}^2_{\lh{2}}+\int_0^\tau\norm{\nabb\La^{-a}V(\tau')}^2_{\lh{2}}d\tau'\\
\ds\lesssim &\ds\int_0^\tau\int_0^1\int_{\ptu} \La^{-2a}V(\tau')[bN,\lap]U(\tau')f\dmt dt d\tau'.
\end{array}
\end{equation}

As in \eqref{lxx3:9}, we need to estimate the commutator term $[bN,\lap]U$. Using twice the commutator formula \eqref{comm7} together with the fact that $U(\tau)f$ is a scalar function, we have:
\be\lab{lxx4:9}
[bN,\lap]U=H\nabb^2 U+\nabb(H\nabb U)+G\nabb U
\ee
where the tensors $H$ and $G$ are given by $H=b(\chi+k)$ and $G=(\chi+k)\nabb b+\chi(\kep+\xib)+\chb\z+b{}^*(\b+\bb)$. Using Lemma \ref{lemma:lbz4}, Lemma \ref{lemma:po5}, and the estimates \eqref{estn}-\eqref{estzeta} for $k, b, \chi, \z, \xib$ and $\chb$, we obtain the following bound for $H$ and $G$:
\bea\lab{lxx4:10}
&&\sup_j\left(2^{\frac{j}{2}}\norm{P_jH}_{\tx{\infty}{2}}+2^{-\frac{j}{2}}\norm{P_jG}_{\tx{\infty}{2}}\right)\\
\nn&\les& \no(b(\chi+k))+\norm{\nabb b}_{\tx{2}{4}}(\norm{\chi}_{\tx{\infty}{4}}+\norm{k}_{\tx{\infty}{4}})+\norm{\chi}_{\tx{\infty}{4}}(\norm{\kep}_{\tx{\infty}{4}}+\norm{\xib}_{\tx{\infty}{4}})\\
\nn&&+\norm{\chb}_{\tx{\infty}{4}}\norm{\z}_{\tx{\infty}{4}}+\ep\\
\nn&\les&\ep.
\eea

Using the property of $P_j$, and in view of \eqref{lxx4:9} and \eqref{lxx4:10}, we have:
\bea
&&\nn\int_0^\tau\int_0^1\int_{\ptu} \La^{-2a}V(\tau')[bN,\lap]U(\tau')f\dmt dt d\tau'\\
\nn&=&\sum_j\int_0^\tau\int_0^1\int_{\ptu} (P_j(H)P_j(\nabb(\nabb U\La^{-2a}V(\tau')))+P_j(H)P_j(\nabb U\nabb \La^{-2a}V(\tau')))\dmt dt d\tau'\\
\nn&&+\int_0^\tau\int_0^1\int_{\ptu}P_j(G) P_j(\nabb U\La^{-2a}V(\tau'))\dmt dt d\tau'\\
\nn&\les& \sum_j 2^{-\frac{j}{2}}\ep\int_0^\tau(\norm{P_j(\nabb(\nabb U\La^{-2a}V(\tau')))}_{\lh{2}}+\norm{P_j(\nabb U\nabb \La^{-2a}V(\tau'))}_{\lh{2}})d\tau'\\
\lab{lxx4:11}&&+\sum_j 2^{\frac{j}{2}}\ep\int_0^\tau\norm{P_j(\nabb U\La^{-2a}V(\tau'))}_{\lh{2}}d\tau'.
\eea

In order to estimate the right-hand side of \eqref{lxx4:11}, we derive three product estimates. Let $h_1, h_2$ two scalar functions. Let $\delta>0$ a small constant to be chosen later on. Using the finite band property for $P_j$, the weak Bernstein inequality, the Gagliardo Nirenberg inequality \eqref{eq:GNirenberg}, and the Bochner inequality for scalars \eqref{eq:Bochconseqbis}, we obtain: 
\bea
\lab{lxx4:12}&&\norm{P_j\nabb((\nabb h_1)h_2)}_{\lpt{2}}\\
\nn&\les& \norm{P_j\nabb((\nabb h_1)h_2)}_{\lpt{2}}^{\frac{1+\d}{2}}\norm{P_j\nabb((\nabb h_1)h_2)}_{\lpt{2}}^{\frac{1-\delta}{2}}\\
\nn&\les&(2^{j0_+}\norm{(\nabb^2h_1)h_2)}_{\lpt{2_-}}+2^{j0_+}\norm{(\nabb h_1)(\nabb h_2)}_{\lpt{2_-}})^{\frac{1+\d}{2}}(2^j\norm{(\nabb h_1)h_2)}_{\lpt{2}})^{\frac{1-\delta}{2}}\\
\nn&\les&2^{j(\half-\frac{\delta}{4})}(\norm{\nabb^2h_1}_{\lpt{2}}\norm{\nabb h_2}_{\lpt{2}})^{\frac{1+\d}{2}}(\norm{\nabb h_1}_{\lpt{2_+}}\norm{\nabb h_2}_{\lpt{2}})^{\frac{1-\delta}{2}}\\
\nn&\les&2^{j(\half-\frac{\delta}{4})}\norm{\lap h_1}_{\lpt{2}}^{\half+\delta}\norm{\nabb h_1}_{\lpt{2}}^{\half-\delta}\norm{\nabb h_2}_{\lpt{2}}.
\eea
Also, the weak Bernstein inequality, the Gagliardo Nirenberg inequality \eqref{eq:GNirenberg}, 
and the Bochner inequality for scalars \eqref{eq:Bochconseqbis} yields:
\bea
\lab{lxx4:13}\norm{P_j((\nabb h_1)(\nabb h_2))}_{\lpt{2}}&\les& 2^{j(\half-\delta)}\norm{(\nabb h_1)(\nabb h_2)}_{\lpt{(\frac{4}{3})_+}}\\
\nn&\les& 2^{j(\half-\delta)}\norm{\nabb h_1}_{\lpt{\frac{4}{1-2\delta}}}\norm{\nabb h_2}_{\lpt{2}}\\
\nn&\les& 2^{j(\half-\delta)}\norm{\nabb^2h_1}_{\lpt{q}}^{\half+\delta}\norm{\nabb h_1}_{\lpt{q}}^{\half-\delta}\norm{\nabb h_2}_{\lpt{2}}\\
\nn&\les& 2^{j(\half-\delta)}\norm{\lap h_1}_{\lpt{q}}^{\half+\delta}\norm{\nabb h_1}_{\lpt{q}}^{\half-\delta}\norm{\nabb h_2}_{\lpt{2}}.
\eea
Finally, we have:
\be\lab{lxx4:14}
\norm{P_j((\nabb h_1)h_2)}_{\lpt{2}}\les \sum_{l,q}\norm{P_j(\nabb(P_l(h_1))P_q(h_2))}_{\lpt{2}}.
\ee
If $j\geq \max(l,q)$, we obtain using the finite band property for $P_j, P_l$ and $P_q$, the strong Bernstein inequality \eqref{eq:strongbernscalarbis} for $P_q$, the Gagliardo Nirenberg inequality \eqref{eq:GNirenberg}, and the Bochner inequality for scalars \eqref{eq:Bochconseqbis}:
\bea\lab{lxx4:15}
&&\norm{P_j(\nabb(P_l(h_1))P_q(h_2))}_{\lpt{2}}\\
\nn&\les & 2^{-j}\norm{\nabb^2(P_l(h_1))P_q(h_2))}_{\lpt{2}}+2^{-j}\norm{\nabb(P_l(h_1))\nabb(P_q(h_2))}_{\lpt{2}}\\
\nn&\les & 2^{-j}\norm{\nabb^2(P_l(h_1))}_{\lpt{2}}\norm{P_q(h_2))}_{\lpt{\infty}}+2^{-j}\norm{\nabb(P_l(h_1))}_{\lpt{6}}\norm{\nabb(P_q(h_2))}_{\lpt{3}}\\
\nn&\les &(2^{-j+q+2l}+2^{-j+\frac{4q}{3}+\frac{5l}{3}})\norm{P_l(h_1)}_{\lpt{2}}\norm{P_q(h_2)}_{\lpt{2}}\\
\nn&\les &(2^{-j+l(\half-\delta)}+2^{-j+\frac{q}{3}+l(\frac{1}{6}-\delta)})\norm{\lap h_1}^{\frac{1}{2}+\delta}_{\lpt{2}}\norm{\nabb h_1}^{\frac{1}{2}-\delta}_{\lpt{2}}\norm{\nabb h_2}_{\lpt{2}}\\
\nn&\les & 2^{-j(\half+\delta)}\norm{\lap h_1}^{\frac{1}{2}+\delta}_{\lpt{2}}\norm{\nabb h_1}^{\frac{1}{2}-\delta}_{\lpt{2}}\norm{\nabb h_2}_{\lpt{2}}.
\eea
Next, if $l\geq \max(j,q)$, we obtain using the finite band property for $P_j$ and $P_l$, the strong Bernstein inequality \eqref{eq:strongbernscalarbis} for $P_q$, the Gagliardo Nirenberg inequality \eqref{eq:GNirenberg}, and the Bochner inequality for scalars \eqref{eq:Bochconseqbis}:
\bea\lab{lxx4:16}
\norm{P_j(\nabb(P_l(h_1))P_q(h_2))}_{\lpt{2}}&\les & \norm{\nabb(P_l(h_1))P_q(h_2)}_{\lpt{2}}\\
\nn&\les & \norm{\nabb(P_l(h_1))}_{\lpt{2}}\norm{P_q(h_2)}_{\lpt{\infty}}\\
\nn&\les & 2^{-l(\half+\delta)}\norm{\lap h_1}^{\frac{1}{2}+\delta}_{\lpt{2}}\norm{\nabb h_1}^{\frac{1}{2}-\delta}_{\lpt{2}}\norm{\nabb h_2}_{\lpt{2}}.
\eea
Finally,  if $q\geq \max(j,l)$, we obtain using the finite band property for $P_j, P_l$ and $P_q$, the weak Bernstein inequality for $P_q$, the Gagliardo Nirenberg inequality \eqref{eq:GNirenberg}, and the Bochner inequality for scalars \eqref{eq:Bochconseqbis} :
\bee
\norm{P_j(\nabb(P_l(h_1))P_q(h_2))}_{\lpt{2}}&\les & \norm{\nabb(P_l(h_1))P_q(h_2)}_{\lpt{2}}\\
\nn&\les & \norm{\nabb(P_l(h_1))}_{\lpt{6}}\norm{P_q(h_2)}_{\lpt{3}}\\
\nn&\les & 2^{\frac{q}{3}+\frac{5l}{3}}\norm{P_l(h_1)}_{\lpt{2}}\norm{P_q(h_2)}_{\lpt{2}}\\
\nn&\les & 2^{-\frac{2q}{3}+l(\frac{1}{6}-\delta)}\norm{\lap h_1}^{\frac{1}{2}+\delta}_{\lpt{2}}\norm{\nabb h_1}^{\frac{1}{2}-\delta}_{\lpt{2}}\norm{\nabb h_2}_{\lpt{2}}\\
\nn&\les & 2^{-q(\frac{1}{2}+\delta)}\norm{\lap h_1}^{\frac{1}{2}+\delta}_{\lpt{2}}\norm{\nabb h_1}^{\frac{1}{2}-\delta}_{\lpt{2}}\norm{\nabb h_2}_{\lpt{2}},
\eee
which together with \eqref{lxx4:14}-\eqref{lxx4:16} yields:
\be\lab{lxx4:17}
\norm{P_j((\nabb h_1)h_2)}_{\lpt{2}}\les 2^{-j(\half+\delta)}\norm{\lap h_1}^{\frac{1}{2}+\delta}_{\lpt{2}}\norm{\nabb h_1}^{\frac{1}{2}-\delta}_{\lpt{2}}\norm{\nabb h_2}_{\lpt{2}}.
\ee

Now, we use \eqref{lxx4:12} \eqref{lxx4:13} and \eqref{lxx4:17} with $h_1=U$ and $h_2=\Lambda^{-2a}V$ to estimate respectively the first, second and third term in the right-hand side of \eqref{lxx4:11}. We obtain:
\bea
\lab{lxx4:18}&&\int_0^\tau\int_0^1\int_{\ptu} \La^{-2a}V(\tau')[bN,\lap]U(\tau')f\dmt dt d\tau'\\
\nn&\les& \ep\left(\sum_j 2^{-\frac{j\delta}{4}}\right)\int_0^\tau\norm{\lap U(\tau')}^{\frac{1}{2}+\delta}_{\lh{2}}\norm{\nabb U(\tau')}^{\frac{1}{2}-\delta}_{\lh{2}}\norm{\nabb \Lambda^{-2a}V(\tau')}_{\lh{2}}d\tau'\\
\nn&\les & \ep\int_0^\tau\norm{\lap U(\tau')}^{\frac{1}{2}+\delta}_{\lh{2}}\norm{\nabb U(\tau')}^{\frac{1}{2}-\delta}_{\lh{2}}\norm{\nabb \Lambda^{-a}V(\tau')}^{1-a}_{\lh{2}}\norm{\Lambda^{-a}V(\tau')}^a_{\lh{2}}d\tau',
\eea
where we used the interpolation estimate \eqref{interpolLa}, \eqref{La1}, and the fact that $\delta>0$ in the last inequality. Next, \eqref{lxx4:7}, \eqref{lxx4:8} and \eqref{lxx4:18} yield:
\bea
\label{lxx4:19}
&&\ds\norm{\La^{-a}V(\tau)}^2_{\lh{2}}+\int_0^\tau\norm{\nabb\La^{-a}V(\tau')}^2_{\lh{2}}d\tau'\\
\nn&\lesssim& \ep\int_0^\tau(\tau')\norm{\lap U(\tau')}^2_{\lh{2}}d\tau'+\ep\int_0^\tau\norm{\nabb U(\tau')}^2_{\lh{2}}d\tau'\\
\nn&\les& \ep\norm{f}^2_{\lh{2}},
\eea
where we used the heat flow estimates \eqref{eq:l2heat1} and \eqref{eq:l2heat2} in the last inequality.

Using the interpolation inequality \eqref{interpolLa} and \eqref{lxx4:19}, we obtain:
\bee
\ds\int_0^{+\infty}\norm{V(\tau)}^{\frac{2}{a}}_{\lpt{2}}d\tau&\lesssim& \int_0^\tau \norm{\La^{-a}V(\tau')}^{\frac{2(1-a)}{a}}_{\lpt{2}}\norm{\nabb\La^{-a}V(\tau')}^2_{\lpt{2}}d\tau'
\\
\nn&\lesssim&\ep\norm{f}^{\frac{2}{a}}_{\li{\infty}{2}}.
\eee
Together with \eqref{lxx4:5}, we obtain:
\bea\label{lxx4:20}
\norm{[bN,P_q]f}_{\li{\infty}{2}}&\les& \normm{\int_0^\infty m_q(\tau)\norm{V(\tau)}_{\lh{2}} d\tau}_{L^\infty_u}\\
\nn&\lesssim& 2^{ja}\normm{\left(\int_0^{+\infty}\norm{V(\tau)}^{\frac{2}{a}}_{\lh{2}}d\tau\right)^{\frac{a}{2}}}_{L^{\infty}_u}\\
\nn&\lesssim& 2^{ja}\norm{f}_{\li{\infty}{2}}.
\eea
Since $a<1$ in view of \eqref{lxx4:7}, \eqref{lxx4:20} yields \eqref{lxx4:4}. This concludes the proof of the proposition.

\subsection{Proof of Proposition \ref{prop:gowinda14}}\lab{sec:gowinda14}

In view of the analog of \eqref{ad24} \eqref{ad25}, we have:
\begin{equation}\label{zz3}
[nL,P_j]\trc=\int_0^\infty m_j(\tau)V(\tau) d\tau,
\end{equation}
where $V(\tau)$ satisfies:
\be\lab{zz4}
(\partial_{\tau}-\lap)V(\tau)=[nL,\lap]U(\tau)\trc,\,V(0)=0.
\ee

Assume that $U(\tau)\trc$ satisfies the following estimates
\be\lab{zz5}
\norm{\chi\nabb^2U(\tau)\trc}_{L^\infty_uL^1_tL^2_\tau \lpt{2}}\les \ep,
\ee
and
\be\lab{zz6}
\norm{\nabb U(\tau)\trc}_{L^\infty_uL^2_tL^2_\tau \lpt{\infty}}\les\ep.
\ee
Then, in view of the commutator estimate \eqref{comm6}, we have
\bea\lab{zz7}
&&\norm{[nL,\lap]U(\tau)\trc}_{L^\infty_uL^1_tL^2_\tau \lpt{2}}\\
\nn&\les& \norm{n}_{L^\infty}\Big(\norm{\chi\nabb^2U(\tau)\trc}_{L^\infty_uL^1_tL^2_\tau \lpt{2}}+(\norm{\chi}_{\tx{\infty}{4}}\norm{\kepb}_{\tx{\infty}{4}}\\
\nn&&+\norm{n^{-1}\nabb n}_{\tx{\infty}{2}}\norm{\trc}_{L^\infty}+\norm{\nabb\trc}_{\tx{\infty}{2}})\norm{\nabb U(\tau)\trc}_{L^\infty_uL^1_tL^2_\tau \lpt{\infty}}\Big)\\
\nn&\les& \ep,
\eea
where we used in the last inequality the estimate \eqref{estn} for $n$, the estimates \eqref{esttrc} \eqref{esthch} for $\chi$, the estimates \eqref{estn} \eqref{estk} for $\kepb$, and the estimates \eqref{zz5} and \eqref{zz6}. The energy estimate \eqref{heatF1} implies
$$\norm{\nabb V(\tau)}^2_{\lpt{2}}+\int_0^\tau\norm{\lap V(\tau')}^2_{\lpt{2}}d\tau'\lesssim \int_0^\tau\norm{[nL,\lap]U(\tau')\trc}^2_{\lpt{2}}d\tau'.$$
Taking the $L^\infty_uL^1_t$ norm, and using the estimate \eqref{zz7}, we obtain
$$\norm{\nabb V}_{L^\infty_uL^1_tL^2_\tau \lpt{2}}\les \ep,$$
which together with \eqref{zz3} yields the second part of the estimate \eqref{zz2}
\be\lab{zz8}
\norm{\nabb[nL,P_j]\trc}_{\tx{1}{2}}\les \ep.
\ee
Also, the energy estimate \eqref{heatF2} implies
\be\lab{zz9}
\norm{V(\tau)}^2_{\lpt{2}}+\int_0^\tau\norm{\nabb V(\tau')}^2_{\lpt{2}}d\tau'\lesssim\int_0^\tau\norm{V(\tau')}_{\lpt{2}}\norm{[nL,\lap]U(\tau')\trc}_{\lpt{2}} d\tau'.
\ee
Let
$$Y(\tau)=\int_0^\tau\norm{V(\tau')}_{\lpt{2}}\norm{[nL,\lap]U(\tau')\trc}_{\lpt{2}} d\tau'.$$
Then, \eqref{zz9} yields
$$Y'(\tau)\les \sqrt{Y(\tau)}\norm{[nL,\lap]U(\tau)\trc}_{\lpt{2}}.$$
Integrating in $\tau$ and using $Y(0)=0$, we obtain
$$\norm{V(\tau)}_{\lpt{2}}\les \left(\int_0^\tau\norm{[nL,\lap]U(\tau')\trc}_{\lpt{2}}d\tau'\right)\les \tau\norm{[nL,\lap]U(\cdot)\trc}_{L^2_\tau\lpt{2}}.$$
Together with \eqref{zz3}, this implies
\bee
\norm{[nL,P_j]\trc}_{\lpt{2}}&\les& \int_0^\infty m_j(\tau)\tau\norm{[nL,\lap]U(\tau)\trc}_{L^2_\tau\lpt{2}}d\tau\\
&\les& 2^{-j}\norm{[nL,\lap]U(\cdot)\trc}_{L^2_\tau\lpt{2}}.
\eee
Taking the $L^\infty_uL^1_t$ norm, and using the estimate \eqref{zz7}, we obtain the first part of the estimate \eqref{zz2}
\be\lab{zz10}
\norm{[nL,P_j]\trc}_{\tx{1}{2}}\les 2^{-j}\ep.
\ee
Finally, \eqref{zz8} and \eqref{zz10} yield the desired estimate \eqref{zz2}. Thus, it remains to prove the estimates \eqref{zz5} and \eqref{zz6}. 

We start with the proof of \eqref{zz6}. We have
\be\lab{zz11}
\norm{\nabb U(\tau)\trc}_{\lpt{\infty}}\les \sum_{j, l}\norm{P_j\nabb U(\tau)P_l\trc}_{\lpt{\infty}}.
\ee
We first consider the case $j<l$. Using the sharp Bernstein inequality for tensors \eqref{eq:strongberntensor} and the finite band property for $P_j$, we have
\bee
\norm{P_j\nabb U(\tau)P_l\trc}_{\lpt{\infty}}&\les& 2^j(1+\norm{K}_{\lpt{2}})\|P_j\nabb U(\tau)P_l\trc\|_{\lpt{2}}\\
&\les& 2^{2j}(1+\norm{K}_{\lpt{2}})\|U(\tau)P_l\trc\|_{\lpt{2}}.
\eee
Taking the $L^2_tL^2_\tau$ norm, we obtain
\bee
\norm{P_j\nabb U(\c)P_l\trc}_{L^2_tL^2_\tau\lpt{\infty}}&\les& 2^{2j}(1+\norm{K}_{\li{\infty}{2}})\|U(\c)P_l\trc\|_{L^\infty_tL^2_\tau\lpt{2}}\\
&\les& 2^{2j}\|\Lambda^{-1}P_l\trc\|_{\tx{\infty}{2}},
\eee
where we used in the last inequality the estimate \eqref{estgauss1} for $K$ and a heat flow estimate for $U(\tau)\trc$. Together with the finite band property for $P_j$ and the assumption $j<l$, we obtain
\be\lab{zz12}
\norm{P_j\nabb U(\c)P_l\trc}_{L^2_tL^2_\tau\lpt{\infty}}\les 2^{-2|l-j|}(2^l\| P_l\trc\|_{\tx{\infty}{2}}).
\ee
Next, we consider the case $l\geq j$. Using the sharp Bernstein inequality for tensors \eqref{eq:strongberntensor} and the finite band property for $P_j$, we have
\bea\lab{zz13:0}
\norm{P_j\nabb U(\tau)P_l\trc}_{\lpt{\infty}}&\les& 2^j(1+\norm{K}^{\frac{1}{2}}_{\lpt{2}})\|P_j\nabb U(\tau)P_l\trc\|_{\lpt{2}}\\
\nn&\les& (1+\norm{K}^{\frac{1}{2}}_{\lpt{2}})\|\nabb^2U(\tau)P_l\trc\|_{\lpt{2}}\\
\nn&\les& (1+\norm{K}^{\frac{1}{2}}_{\lpt{2}})\|\lap U(\tau)P_l\trc\|_{\lpt{2}},
\eea
where we used in the last inequality the Bochner inequality for scalars \eqref{eq:Bochconseqbis}. 
Also, using the sharp Bernstein inequality for tensors \eqref{eq:strongberntensor} and the finite band property for $P_j$, we have
\bea\lab{zz13}
&&\norm{P_j\nabb U(\tau)P_l\trc}_{\lpt{\infty}}\\
\nn&\les& 2^j(1+\norm{K}^{\frac{1}{2}}_{\lpt{2}})\|P_j\nabb U(\tau)P_l\trc\|_{\lpt{2}}\\
\nn&\les& 2^{-j}(1+\norm{K}^{\frac{1}{2}}_{\lpt{2}})\|P_j\lap\nabb U(\tau)P_l\trc\|_{\lpt{2}}\\
\nn&\les& 2^{-j}(1+\norm{K}^{\frac{1}{2}}_{\lpt{2}})(\|\nabb\lap U(\tau)P_l\trc\|_{\lpt{2}}+\|P_j([\lap,\nabb] U(\tau)P_l\trc)\|_{\lpt{2}}).
\eea
Using the commutator formula \eqref{commutptu}, the Bernstein inequality for $P_j$, the Gagliardo-Nirenberg inequality \eqref{eq:GNirenberg}, and the Bochner inequality for scalars \eqref{eq:Bochconseqbis}, we obtain
\bee
\|P_j(K\nabb U(\tau)P_l\trc)\|_{\lpt{2}}&\les& 2^{\frac{j}{2}}\norm{K\nabb U(\tau)P_l\trc}_{\lpt{\frac{4}{3}}}\\
&\les& 2^{\frac{j}{2}}\norm{K}_{\lpt{2}}\norm{\nabb U(\tau)P_l\trc}_{\lpt{4}}\\
&\les& 2^{\frac{j}{2}}\norm{K}_{\lpt{2}}\norm{\nabb^2U(\tau)P_l\trc}_{\lpt{2}}^{\frac{1}{2}}\norm{\nabb U(\tau)P_l\trc}_{\lpt{2}}^{\frac{1}{2}}\\
&\les& 2^{\frac{j}{2}}\norm{K}_{\lpt{2}}\norm{\lap U(\tau)P_l\trc}_{\lpt{2}}^{\frac{1}{2}}\norm{\nabb U(\tau)P_l\trc}_{\lpt{2}}^{\frac{1}{2}}.
\eee
Together with \eqref{zz13}, this yields
\bee
\norm{P_j\nabb U(\tau)P_l\trc}_{\lpt{\infty}}&\les& 2^{-j}(1+\norm{K}^{\frac{1}{2}}_{\lpt{2}})\Big(\|\nabb\lap U(\tau)P_l\trc\|_{\lpt{2}}\\
\nn&&+2^{\frac{j}{2}}\norm{K}_{\lpt{2}}\norm{\lap U(\tau)P_l\trc}_{\lpt{2}}^{\frac{1}{2}}\norm{\nabb U(\tau)P_l\trc}_{\lpt{2}}^{\frac{1}{2}}\Big).
\eee
Interpolating with \eqref{zz13:0}, we deduce
\bee
\norm{P_j\nabb U(\tau)P_l\trc}_{\lpt{\infty}}&\les& 2^{-\frac{j}{2}}(1+\norm{K}^{\frac{1}{2}}_{\lpt{2}})\|\lap U(\tau)P_l\trc\|_{\lpt{2}}^{\frac{1}{2}}\Big(\|\nabb\lap U(\tau)P_l\trc\|_{\lpt{2}}\\
\nn&&+2^{\frac{j}{2}}\norm{K}_{\lpt{2}}\norm{\lap U(\tau)P_l\trc}_{\lpt{2}}^{\frac{1}{2}}\norm{\nabb U(\tau)P_l\trc}_{\lpt{2}}^{\frac{1}{2}}\Big)^{\frac{1}{2}}.
\eee
Taking the $L^2_tL^2_\tau$ norm, we obtain
\bee
&&\norm{P_j\nabb U(\c)P_l\trc}_{L^2_tL^2_\tau\lpt{\infty}}\\
&\les& 2^{-\frac{j}{2}}(1+\norm{K}^{\frac{1}{2}}_{\li{\infty}{2}})\|\lap U(\tau)P_l\trc\|_{L^\infty_tL^2_\tau\lpt{2}}^{\frac{1}{2}}\Big(\|\nabb\lap U(\c)P_l\trc\|_{L^\infty_tL^2_\tau\lpt{2}}\\
\nn&&+2^{\frac{j}{2}}\norm{K}_{\li{\infty}{2}}\norm{\lap U(\c)P_l\trc}_{L^\infty_tL^2_\tau\lpt{2}}^{\frac{1}{2}}\norm{\nabb U(\tau)P_l\trc}_{L^\infty_tL^2_\tau\lpt{2}}^{\frac{1}{2}}\Big)^{\frac{1}{2}}\\
&\les& 2^{-\frac{j}{2}}\norm{\nabb P_l\trc}^{\frac{1}{2}}_{\tx{\infty}{2}}\Big(\norm{\lap P_l\trc}_{\li{\infty}{2}}+2^{\frac{j}{2}}\norm{\nabb P_l\trc}^{\frac{1}{2}}_{\tx{\infty}{2}}\norm{P_l\trc}^{\frac{1}{2}}_{\tx{\infty}{2}}\Big)^{\frac{1}{2}}, 
\eee
where we used in the last inequality the estimate \eqref{estgauss1} for $K$ and a heat flow estimate for $U(\tau)\trc$.
Together with the finite band property for $P_j$ and the assumption $l\leq j$, we obtain
\be\lab{zz14}
\norm{P_j\nabb U(\c)P_l\trc}_{L^2_tL^2_\tau\lpt{\infty}}\les 2^{-\frac{|l-j|}{4}}(2^l\| P_l\trc\|_{\tx{\infty}{2}}).
\ee
Finally, \eqref{zz11}, \eqref{zz12} for $l>j$ and \eqref{zz14} for $l\leq j$ yield
\be\lab{zz15}
\norm{\nabb U(\tau)\trc}_{L^2_tL^2_\tau\lpt{\infty}}\les \sum_{j, l}2^{-\frac{|l-j|}{4}}(2^l\| P_l\trc\|_{\tx{\infty}{2}})\les \norm{\trc}_{\BB^1},
\ee
where the Besov space $\BB^1$ has been defined in \eqref{eq:Besovnormt}. Now, in view of the estimate \eqref{linftybound}, and the estimates \eqref{esttrc} \eqref{esttrcbesov} for $\trc$, we have
$$\norm{\trc}_{\BB^1}\les \norm{\trc}_{\tx{\infty}{2}}+\norm{\nabb\trc}_{\BB^0}\les\ep.$$
Together with \eqref{zz15}, this implies \eqref{zz6}.

Next, we prove \eqref{zz5}. Recall the Bochner identity for scalars on $\ptu$ which is a 2-surface. For any scalar $f$ on $\ptu$, we have
$$\lap(|\nabb f|^2)=\nabb(\lap f)\c\nabb f+K|\nabb f|^2+|\nabb^2f|^2.$$
Choosing $f=U(\tau)\trc$, multiplying by $|\chi|^2$ and integrating over $\ptu$ yields
\bee
&&\int_{\ptu} |\chi|^2\lap(|\nabb U(\tau)\trc|^2)\\
\nn&=&\int_{\ptu} |\chi|^2\nabb(\lap U(\tau)\trc)\c\nabb U(\tau)+\int_{\ptu} K|\chi|^2|\nabb U(\tau)\trc|^2+\int_{\ptu} |\chi|^2|\nabb^2U(\tau)\trc|^2,
\eee
which implies after integration by parts
\bee
&&\norm{\chi\nabb^2U(\tau)\trc}^2_{\lpt{2}}\\
&= & \norm{\chi\lap U(\tau)\trc}^2_{\lpt{2}}-\int_{\ptu} K|\chi|^2|\nabb U(\tau)\trc|^2+\int_{\ptu} \chi\c\nabb\chi \lap U(\tau)\trc\c\nabb U(\tau)\trc\\
&&-\int_{\ptu} \chi\c\nabb\chi \nabb^2U(\tau)\trc\c\nabb U(\tau)\trc.
\eee
We deduce 
\bee
\norm{\chi\nabb^2U(\tau)\trc}^2_{\lpt{2}}&\les & \norm{\chi\lap U(\tau)\trc}^2_{\lpt{2}}+\norm{K}_{\lpt{2}}\norm{\chi}_{\tx{\infty}{4}}^2\norm{\nabb U(\tau)\trc}_{\lpt{\infty}}^2\\
&&+\norm{\chi\nabb^2U(\tau)\trc}_{\lpt{2}}\norm{\nabb\chi}_{\lpt{2}}\norm{\nabb U(\tau)\trc}_{\lpt{\infty}},
\eee
which yields
\bee
\norm{\chi\nabb^2U(\tau)\trc}_{\lpt{2}}&\les & \norm{\chi\lap U(\tau)\trc}_{\lpt{2}}+\norm{K}^{\frac{1}{2}}_{\lpt{2}}\norm{\nabb U(\tau)\trc}_{\lpt{\infty}}\\
&&+\norm{\nabb\chi}_{\lpt{2}}\norm{\nabb U(\tau)\trc}_{\lpt{\infty}},
\eee
where we used in the last inequality the estimates \eqref{esttrc} \eqref{esthch} for $\chi$. Taking the $L^1_tL^2_\tau$ norm, we obtain
\bea\lab{zz16}
&&\norm{\chi\nabb^2U(\tau)\trc}_{L^1_tL^2_\tau\lpt{2}}\\
\nn&\les & \norm{\chi\lap U(\tau)\trc}_{L^1_tL^2_\tau\lpt{2}}+\norm{K}^{\frac{1}{2}}_{\li{\infty}{2}}\norm{\nabb U(\tau)\trc}_{L^2_tL^2_\tau\lpt{\infty}}\\
\nn&&+\norm{\nabb\chi}_{\li{\infty}{2}}\norm{\nabb U(\tau)\trc}_{L^2_tL^2_\tau\lpt{\infty}}\\
\nn&\les& \norm{\chi\lap U(\tau)\trc}_{L^1_tL^2_\tau\lpt{2}}+\ep,
\eea
where we used in the last inequality the estimates \eqref{esttrc} \eqref{esthch} for $\chi$, the estimate \eqref{estgauss1} for $K$, and the estimate \eqref{zz6} for $\nabb U(\tau)\trc$. Next, we estimate the right-hand side of \eqref{zz16}. We multiply the heat equation satisfied by $U(\tau)\trc$ by $|\chi|^2\lap U(\tau)\trc$ and we integrate over $\ptu$. We obtain
$$\frac{1}{2}\frac{d}{d\tau}\norm{\chi\nabb U(\tau)\trc}_{\lpt{2}}^2+\norm{\chi\lap U(\tau)\trc}_{\lpt{2}}^2=\int_{\ptu}\chi\c\nabb\chi\c\nabb U(\tau)\trc  U(\tau)\trc\dmt.$$
This yields
\bee
&&\norm{\chi\lap U(\tau)\trc}_{L^2_\tau\lpt{2}}^2\\
&\les& \norm{\chi\nabb\trc}_{\lpt{2}}^2+\norm{\nabb\chi}_{\lpt{2}}\norm{\chi}_{\tx{\infty}{4}}\norm{\nabb U(\c)\trc}_{L^2_\tau\lpt{4}}\norm{U(\c)\trc}_{L^2_\tau \lpt{\infty}}\\
&\les& \norm{\chi\nabb\trc}_{\lpt{2}}^2+\norm{\nabb\chi}_{\lpt{2}}\norm{\nabb^2U(\c)\trc}^{\frac{1}{2}}_{L^2_\tau\lpt{2}}\norm{\nabb U(\c)\trc}_{L^2_\tau\lpt{2}}^{\frac{1}{2}}\norm{\trc}_{L^\infty},
\eee
where we used in the last inequality the estimates \eqref{esttrc} \eqref{esthch} for $\chi$, the Gagliardo-Nirenberg inequality \eqref{eq:GNirenberg} and the fact that the heat flow $U(\tau)$ is bounded on $\lpt{\infty}$ (see for example \cite{LP} for a proof). Using the the Bochner inequality for scalars \eqref{eq:Bochconseqbis} and heat flow estimates for $U(\tau)\trc$, we obtain
\bee
&&\norm{\chi\lap U(\tau)\trc}_{L^2_\tau\lpt{2}}^2\\
&\les& \norm{\chi\nabb\trc}_{\lpt{2}}^2+\norm{\nabb\chi}_{\lpt{2}}\norm{\nabb\trc}^{\frac{1}{2}}_{\tx{\infty}{2}}\norm{\trc}_{\tx{\infty}{2}}^{\frac{1}{2}}\norm{\trc}_{L^\infty}\\
&\les& \norm{\chi\nabb\trc}_{\lpt{2}}^2+\norm{\nabb\chi}_{\lpt{2}}\ep,
\eee
where we used in the last inequality the estimate \eqref{esttrc} for $\trc$. Integrating in time, this yields
\bee
\norm{\chi\lap U(\tau)\trc}_{L^2_tL^2_\tau\lpt{2}}&\les& \norm{\chi\nabb\trc}_{\li{\infty}{2}}+\norm{\nabb\chi}_{\li{\infty}{2}}+\ep\\
&\les&\ep,
\eee
where we used in the last inequality the estimates \eqref{esttrc} \eqref{esthch} for $\chi$. Together with \eqref{zz16}, we finally obtain 
$$\norm{\chi\nabb^2U(\tau)\trc}_{L^1_tL^2_\tau\lpt{2}}\les\ep.$$
Taking the supremum in $u$ yields \eqref{zz5}. This concludes the proof of the proposition.

\subsection{Proof of Proposition \ref{prop:gowinda15}}\lab{sec:gowinda15}

The proof of the estimate \eqref{zz18} being similar and slightly easier than the proof of \eqref{zz17}, we focus on \eqref{zz17}. In view of \eqref{ad24} \eqref{ad25}, we have:
\begin{equation}\label{zz19}
[bN,P_j]\trc=\int_0^\infty m_j(\tau)V(\tau) d\tau,
\end{equation}
where $V(\tau)$ satisfies:
\be\lab{zz20}
(\partial_{\tau}-\lap)V(\tau)=[bN,\lap]U(\tau)\trc,\,V(0)=0.
\ee
Assume that $V$ satisfies for all $\tau$
\be\label{zz21}
\norm{V(\tau)}_{\lh{2}}\les \ep\tau^{\frac{1}{4}},
\ee
and
\be\label{zz22}
\norm{\La^{\frac{1}{2}}V(\tau)}^2_{\lh{2}}+\int_0^\tau\norm{\nabb\La^{\frac{1}{2}} V(\tau')}^2_{\lh{2}}d\tau'\les \ep^2.
\ee
Then, first note in view of the interpolation inequality \eqref{interpolLa}, that
$$\norm{\nabb V(\tau)}_{\lpt{2}}\les \norm{\La^{\frac{1}{2}}V(\tau)}_{\lpt{2}}^{\frac{1}{2}}\norm{\nabb\La^{\frac{1}{2}}V(\tau)}_{\lpt{2}}^{\frac{1}{2}}$$
which together with \eqref{zz22} implies
\be\lab{zz23}
\norm{\nabb V(\c)}_{L^4_\tau\lh{2}}\les\norm{\La^{\frac{1}{2}} V(\c)}_{L^\infty_\tau\lh{2}}^{\frac{1}{2}}\norm{\nabb\La^{\frac{1}{2}} V(\c)}_{L^2_\tau\lh{2}}^{\frac{1}{2}}\les \ep.
\ee
Then, in view of \eqref{zz19}, \eqref{zz21} and \eqref{zz23}, we obtain
\bee
&&2^{\frac{j}{2}}\norm{[bN, P_j]\trc}_{\lh{2}}+2^{-\frac{j}{2}}\norm{\nabb[bN, P_j]\trc}_{\lh{2}}\\
&\les& 2^{\frac{j}{2}}\int_0^\infty m_j(\tau)\norm{V(\tau)}_{\lh{2}} d\tau+2^{-\frac{j}{2}}\int_0^\infty m_j(\tau)\norm{\nabb V(\tau)}_{\lh{2}} d\tau\\
&\les& 2^{\frac{j}{2}}\ep\left(\int_0^\infty m_j(\tau)\tau^{\frac{1}{4}} d\tau\right)+2^{-\frac{j}{2}}\ep\left(\int_0^\infty m^{\frac{4}{3}}_j(\tau)d\tau\right)^{\frac{3}{4}}\\
&\les& \ep,
\eee
which after taking the supremum in $u$ yields \eqref{zz17}. Thus, it remains to prove \eqref{zz21} and \eqref{zz22}. 

We start with the proof of \eqref{zz21}. The energy estimate \eqref{heatF2} implies
\begin{equation}\label{zz24}
 \ds\norm{V(\tau)}^2_{\lpt{2}}+\int_0^\tau\norm{\nabb V(\tau')}^2_{\lpt{2}}d\tau'\lesssim\ds\int_0^\tau\int_{\ptu} V(\tau')[bN,\lap]U(\tau')\trc\dmt.
\end{equation}
We need to estimate the commutator term $[bN,\lap]U$. Recall from \eqref{lxx3:9} and \eqref{lxx3:10} that we have
\be\lab{zz25}
[bN,\lap]U=H\nabb^2U+G\nabb U
\ee
where the tensors $H$ and $G$ satisfy
\be\lab{zz26}
\no(H)+\norm{G}_{\li{\infty}{2}}\les\ep.
\ee
In view of \eqref{zz25}, and integrating by parts the term $\nabb^2U$, we obtain:
\bee
&&\int_0^\tau\int_{\H_u} V(\tau')[bN,\lap]U(\tau')f\dmt d\tau'\\
\nn&\lesssim& \int_0^\tau \norm{H}_{\tx{\infty}{4}}\norm{\nabb U(\tau')}_{\lpt{4}}\norm{\nabb V(\tau')}_{\lpt{2}}d\tau'\\
\nn&&+\int_0^\tau (\norm{\nabb H}_{\lpt{2}}+\norm{G}_{\lpt{2}})\norm{\nabb U(\tau')}_{\lpt{4}}\norm{V(\tau')}_{\lpt{4}}d\tau'.
\eee
Together with \eqref{zz24} and the Gagliardo-Nirenberg inequality \eqref{eq:GNirenberg}, this yields
\bea\lab{zz27}
&&\norm{V(\tau)}^2_{\lpt{2}}+\int_0^\tau\norm{\nabb V(\tau')}^2_{\lpt{2}}d\tau'\\
\nn&\lesssim& (\norm{H}_{\tx{\infty}{4}}^2+\norm{\nabb H}_{\lpt{2}}^2+\norm{G}_{\lpt{2}}^2)\int_0^\tau\norm{\nabb^2 U(\tau')}_{\lpt{2}}\norm{\nabb U(\tau')}_{\lpt{2}}d\tau'\\
\nn&\lesssim& (\ep^2+\norm{\nabb H}_{\lpt{2}}^2+\norm{G}_{\lpt{2}}^2)\int_0^\tau\norm{\lap U(\tau')}_{\lpt{2}}\norm{\nabb U(\tau')}_{\lpt{2}}d\tau',
\eea
where we used in the last inequality the estimate \eqref{zz26} and the Bochner inequality for scalars \eqref{eq:Bochconseqbis}. Now, the heat flow estimate \eqref{eq:l2heatnab} yield:
\be\lab{zz28}
\norm{\nabb U(\tau)}_{\lpt{2}}^2+\int_0^\tau \norm{\lap U(\tau')}^2_{\lpt{2}}d\tau'\les \norm{\nabb\trc}_{\tx{\infty}{2}}^2\les\ep^2,
\ee
where we used in the last inequality the estimate \eqref{esttrc} for $\trc$. Together with \eqref{zz27}, we obtain
$$
\norm{V(\tau)}^2_{\lpt{2}}+\int_0^\tau\norm{\nabb V(\tau')}^2_{\lpt{2}}d\tau'\lesssim\ep^2\tau^{\frac{1}{2}}(\ep+\norm{\nabb H}_{\lpt{2}}^2+\norm{G}_{\lpt{2}}^2).$$
Integrating in time, and using the estimate \eqref{zz26} yields \eqref{zz21}.

Next, we prove \eqref{zz22}. The energy estimate \eqref{eq:l2heat1bis} implies:
\bee
&&\norm{\La^\frac{1}{2}V(\tau)}^2_{\lpt{2}}+\int_0^\tau\norm{\nabb\La^\frac{1}{2}V(\tau')}^2_{\lpt{2}}d\tau'\\
\nn&= &\ds\int_0^\tau\int_{\ptu} \La V(\tau')[bN,\lap]U(\tau')\trc\dmt d\tau'\\
\nn&\les& \int_0^\tau\norm{\La^{\frac{3}{2}}V(\tau')}_{\lpt{2}}\norm{\La^{-\frac{1}{2}}([bN,\lap]U(\tau')\trc)}_{\lpt{2}}d\tau'\\
\nn&\les& \int_0^\tau\norm{\nabb\La^{\frac{1}{2}}V(\tau')}_{\lpt{2}}\norm{[bN,\lap]U(\tau')\trc}_{\lpt{\frac{3}{2}}}d\tau',
\eee
where we used \eqref{La6} in the last inequality. This yields
\be\label{zz29}
\norm{\La^\frac{1}{2}V(\tau)}^2_{\lpt{2}}+\int_0^\tau\norm{\nabb\La^\frac{1}{2}V(\tau')}^2_{\lpt{2}}d\tau'\les \int_0^\tau\norm{[bN,\lap]U(\tau')\trc}_{\lpt{\frac{3}{2}}}^2d\tau'.
\ee
In view of \eqref{zz25}, we have
\bee
&&\norm{[bN,\lap]U(\tau')\trc}_{\lpt{\frac{3}{2}}}\\
&\les& \norm{H}_{\lpt{6}}\norm{\nabb^2U(\tau')\trc}_{\lpt{2}}+\norm{G}_{\lpt{2}}\norm{\nabb U(\tau')\trc}_{\lpt{6}}\\
&\les& \norm{H}_{\lpt{6}}\norm{\nabb^2U(\tau')\trc}_{\lpt{2}}+\norm{G}_{\lpt{2}}\norm{\nabb^2U(\tau')\trc}_{\lpt{2}}^{\frac{2}{3}}\norm{\nabb U(\tau')\trc}_{\lpt{2}}^{\frac{1}{3}}\\
&\les& \norm{H}_{\lpt{6}}\norm{\lap U(\tau')\trc}_{\lpt{2}}+\norm{G}_{\lpt{2}}\norm{\lap U(\tau')\trc}_{\lpt{2}}^{\frac{2}{3}}\norm{\nabb U(\tau')\trc}_{\lpt{2}}^{\frac{1}{3}},
\eee
where we used the Gagliardo-Nirenberg inequality \eqref{eq:GNirenberg} and the Bochner inequality for scalars \eqref{eq:Bochconseqbis}. Taking the $L^2_\tau$ norm and using \eqref{zz28} implies
$$\norm{[bN,\lap]U(\tau')\trc}_{L^2_\tau\lpt{\frac{3}{2}}}\les \ep(\norm{H}_{\lpt{6}}+\norm{G}_{\lpt{2}}).$$
Now, taking the $L^2_t$ norm and using \eqref{zz26} yields
\be\lab{zz30}
\norm{[bN,\lap]U(\tau')\trc}_{L^2_tL^2_\tau\lpt{\frac{3}{2}}}\les \ep.
\ee
Finally, integrating \eqref{zz29} in $t$, and injecting \eqref{zz30}, we obtain \eqref{zz22}. This concludes the proof of the proposition.

\subsection{Proof of Proposition \ref{prop:gowinda16}}\lab{sec:gowinda16}

We have:
\begin{equation}\label{zz32}
[\nabb,P_j]\trc=\int_0^\infty m_j(\tau)V(\tau) d\tau,
\end{equation}
where $V(\tau)$ satisfies:
\be\lab{zz33}
(\partial_{\tau}-\lap)V(\tau)=[\nabb,\lap]U(\tau)\trc,\,V(0)=0.
\ee
Assume that $V$ satisfies for all $\tau$ 
\be\label{zz34}
\norm{\La^\frac{3}{4}V(\tau)}_{\lh{2}}\les \ep. 
\ee
Then, using the Bernstein inequality for $P_j$, we have
\bee
\norm{V(\tau)}_{\tx{2}{4}}&\les& \sum_j\norm{P_jV(\tau)}_{\tx{2}{4}}\\
&\les& \sum_j 2^{\frac{j}{2}}\norm{P_jV(\tau)}_{\lh{2}}\\
&\les& \left(\sum_j 2^{-\frac{j}{4}}\right)\norm{\La^\frac{3}{4}V(\tau)}_{\lh{2}}\\
&\les&\ep,
\eee
where we used in the last inequality \eqref{zz34}. Together with \eqref{zz32}, we obtain
$$\norm{[\nabb, P_j]\trc}_{\tx{2}{4}}\les \int_0^\infty m_j(\tau)\norm{V(\tau)}_{\tx{2}{4}} d\tau\les\ep,$$
which is the desired estimate \eqref{zz31}. Thus, it remains to prove \eqref{zz34}. 

The energy estimate \eqref{eq:l2heat1bis} implies:
\bee
&&\norm{\La^\frac{3}{4} V(\tau)}^2_{\lpt{2}}+\int_0^\tau\norm{\nabb\La^\frac{3}{4} V(\tau')}^2_{\lpt{2}}d\tau'\\
\nn&= &\ds\int_0^\tau\int_{\ptu} \La^\frac{3}{2} V(\tau')[\nabb,\lap]U(\tau')\trc\dmt d\tau'\\
\nn&\les& \int_0^\tau\norm{\La^{1+\frac{3}{4}}V(\tau')}_{\lpt{2}}\norm{\La^{-\frac{1}{4}}([\nabb,\lap]U(\tau')\trc)}_{\lpt{2}}d\tau'\\
\nn&\les& \int_0^\tau\norm{\nabb\La^{\frac{3}{4}}V(\tau')}_{\lpt{2}}\norm{[\nabb,\lap]U(\tau')\trc}_{\lpt{\frac{5}{3}}}d\tau',
\eee
where we used \eqref{La6} in the last inequality. This yields
\be\label{zz35}
\norm{\La^\frac{3}{4}V(\tau)}^2_{\lpt{2}}+\int_0^\tau\norm{\nabb\La^\frac{3}{4}V(\tau')}^2_{\lpt{2}}d\tau'\les \int_0^\tau\norm{[\nabb,\lap]U(\tau')\trc}_{\lpt{\frac{5}{3}}}^2d\tau'.
\ee
Now, in view of the commutator formula \eqref{commutptu}, we have
\bee
\norm{[\nabb,\lap]U(\tau')\trc}_{\lpt{\frac{5}{3}}}&\les& \norm{K\nabb U(\tau')\trc}_{\lpt{\frac{5}{3}}}\\
&\les& \norm{K}_{\lpt{2}}\norm{\nabb U(\tau')\trc}_{\lpt{10}}\\
&\les& \norm{K}_{\lpt{2}}\norm{\nabb^2U(\tau')\trc}_{\lpt{2}}^{\frac{4}{5}}\norm{\nabb U(\tau')\trc}_{\lpt{2}}^{\frac{1}{5}}\\
&\les& \norm{K}_{\lpt{2}}\norm{\lap U(\tau')\trc}_{\lpt{2}}^{\frac{4}{5}}\norm{\nabb U(\tau')\trc}_{\lpt{2}}^{\frac{1}{5}},
\eee
where we used the Gagliardo-Nirenberg inequality \eqref{eq:GNirenberg} and the Bochner inequality for scalars \eqref{eq:Bochconseqbis}. Taking the $L^2_\tau$ norm and using \eqref{zz28} implies
$$\norm{[\nabb,\lap]U(\tau')\trc}_{L^2_\tau\lpt{\frac{5}{3}}}\les \ep\norm{K}_{\lpt{2}}.$$
Now, taking the $L^2_t$ norm and using the estimate \eqref{estgauss1} for $K$ yields
\be\lab{zz36}
\norm{[\nabb,\lap]U(\tau')\trc}_{L^2_tL^2_\tau\lpt{\frac{5}{3}}}\les \ep.
\ee
Finally, integrating \eqref{zz35} in $t$, and injecting \eqref{zz36}, we obtain \eqref{zz34}. This concludes the proof of the proposition.

\subsection{Proof of Lemma \ref{lemma:zz41}}\lab{sec:gowinda17}

We have:
\begin{equation}\label{zz61}
[P_{>j}, P_{\leq j}(h)]F=\int_0^\infty m_{>j}(\tau)V(\tau) d\tau,
\end{equation}
where $V(\tau)$ satisfies:
\be\lab{zz62}
(\partial_{\tau}-\lap)V(\tau)=\lap P_{\leq j}(h) U(\tau)F+\nabb P_{\leq j}(h)\c\nabb U(\tau)F,\,V(0)=0.
\ee
Assume that $V$ satisfies for all $\tau$ 
\be\label{zz63}
\norm{V(\tau)}_{\lpt{2}}\les (1+2^j\sqrt{\tau}+2^{\frac{3j}{2}}\tau^{\frac{3}{4}})\norm{\nabb h}_{\lpt{2}}\norm{F}_{\lpt{2}}. 
\ee
Then, \eqref{zz61} and \eqref{zz63} imply
\bee
\norm{[P_{>j}, P_{\leq j}(h)]F}_{\lpt{2}}&\les& \int_0^\infty m_{>j}(\tau)\norm{V(\tau)}_{\lpt{2}} d\tau\\
&\les& \left(\int_0^\infty m_{>j}(\tau)(1+2^j\sqrt{\tau}+2^{\frac{3j}{2}}\tau^{\frac{3}{4}}) d\tau\right)\norm{\nabb h}_{\lpt{2}}\norm{F}_{\lpt{2}}\\
&\les& \norm{\nabb h}_{\lpt{2}}\norm{F}_{\lpt{2}}
\eee
which is the desired estimate \eqref{zz41}. Thus, it remains to prove \eqref{zz63}. 

The energy estimate \eqref{heatF2} implies
\bee
&& \norm{V(\tau)}^2_{\lpt{2}}+\int_0^\tau\norm{\nabb V(\tau')}^2_{\lpt{2}}d\tau'\\
 &\lesssim&\ds\int_0^\tau\int_{\ptu} V(\tau')\Big(\lap P_{\leq j}(h) U(\tau)F+\nabb P_{\leq j}(h)\c\nabb U(\tau)F\Big)\dmt\\
&\lesssim&\int_0^\tau \norm{\lap P_{\leq j}(h)}_{\lpt{4}}\norm{U}_{\lpt{4}}\norm{V}_{\lpt{2}}+ \norm{\nabb P_{\leq j}(h)}_{\lpt{4}}\norm{\nabb U}_{\lpt{2}}\norm{V}_{\lpt{4}}\\
&\les& \int_0^\tau \norm{\nabb\lap P_{\leq j}(h)}_{\lpt{2}}^{\frac{1}{2}}\norm{\lap P_{\leq j}(h)}_{\lpt{2}}^{\frac{1}{2}}\norm{\nabb U}_{\lpt{2}}^\frac{1}{2}\norm{U}_{\lpt{2}}^\frac{1}{2}\norm{V}_{\lpt{2}}\\
&&+\int_0^\tau \norm{\nabb^2P_{\leq j}(h)}_{\lpt{2}}^\frac{1}{2}\norm{\nabb P_{\leq j}(h)}_{\lpt{2}}^\frac{1}{2}\norm{\nabb U}_{\lpt{2}}\norm{\nabb V}_{\lpt{2}}^\frac{1}{2}\norm{V}_{\lpt{2}}^\frac{1}{2}
\eee
where we used in the last inequality the Gagliardo-Nirenberg inequality \eqref{eq:GNirenberg}. Together with the Bochner inequality for scalars \eqref{eq:Bochconseqbis} and the finite band property for $P_j$, we obtain
\bee
&& \norm{V(\tau)}^2_{\lpt{2}}+\int_0^\tau\norm{\nabb V(\tau')}^2_{\lpt{2}}d\tau'\\
&\les& 2^{\frac{3j}{2}}\int_0^\tau \norm{\nabb h}_{\lpt{2}}\norm{\nabb U}_{\lpt{2}}^\frac{1}{2}\norm{U}_{\lpt{2}}^\frac{1}{2}\norm{V}_{\lpt{2}}\\
&&+\int_0^\tau \norm{\lap P_{\leq j}(h)}_{\lpt{2}}^\frac{1}{2}\norm{\nabb h}_{\lpt{2}}^\frac{1}{2}\norm{\nabb U}_{\lpt{2}}\norm{\nabb V}_{\lpt{2}}^\frac{1}{2}\norm{V}_{\lpt{2}}^\frac{1}{2}\\
&\les& 2^{\frac{3j}{2}}\int_0^\tau \norm{\nabb h}_{\lpt{2}}\norm{\nabb U}_{\lpt{2}}^\frac{1}{2}\norm{U}_{\lpt{2}}^\frac{1}{2}\norm{V}_{\lpt{2}}\\
&&+2^{\frac{j}{2}}\int_0^\tau \norm{\nabb h}_{\lpt{2}}\norm{\nabb U}_{\lpt{2}}\norm{\nabb V}_{\lpt{2}}^\frac{1}{2}\norm{V}_{\lpt{2}}^\frac{1}{2}\\
&\les & 2^{\frac{3j}{2}}\int_0^\tau \norm{\nabb h}_{\lpt{2}}\norm{\nabb U}_{\lpt{2}}^\frac{1}{2}\norm{U}_{\lpt{2}}^\frac{1}{2}\norm{V}_{\lpt{2}} \\
&&+2^j\int_0^\tau \norm{\nabb h}_{\lpt{2}}\norm{\nabb U}_{\lpt{2}}\norm{V}_{\lpt{2}}\\
&&+\int_0^\tau \norm{\nabb h}_{\lpt{2}}\norm{\nabb U}_{\lpt{2}}\norm{\nabb V}_{\lpt{2}}
\eee
This yields
\bee
\norm{V(\tau)}^2_{\lpt{2}}&\les & \norm{\nabb h}_{\lpt{2}}^2\int_0^\tau \norm{\nabb U}_{\lpt{2}}^2\\
&&+\int_0^\tau \norm{\nabb h}_{\lpt{2}}\Big(2^{\frac{3j}{2}}\norm{\nabb U}_{\lpt{2}}^\frac{1}{2}\norm{U}_{\lpt{2}}^\frac{1}{2}+2^j\norm{\nabb U}_{\lpt{2}}\Big)\norm{V}_{\lpt{2}}
\eee
which together with the heat flow estimate \eqref{eq:l2heat1} and the fact that $U(0)=F$ implies
\bee
\norm{V(\tau)}^2_{\lpt{2}}&\les & \norm{\nabb h}_{\lpt{2}}^2\norm{F}_{\lpt{2}}^2\\
&&+\int_0^\tau \norm{\nabb h}_{\lpt{2}}\Big(2^{\frac{3j}{2}}\norm{\nabb U}_{\lpt{2}}^\frac{1}{2}\norm{U}_{\lpt{2}}^\frac{1}{2}+2^j\norm{\nabb U}_{\lpt{2}}\Big)\norm{V}_{\lpt{2}}.
\eee
Integrating this differential inequality, we obtain
\bea\lab{zz64}
\norm{V(\tau)}^2_{\lpt{2}}&\les & \norm{\nabb h}_{\lpt{2}}^2\norm{F}_{\lpt{2}}^2\\
\nn&&+2^{2j}\norm{\nabb h}^2_{\lpt{2}}\Bigg(\int_0^\tau \Big(2^{\frac{j}{2}}\norm{\nabb U}_{\lpt{2}}^\frac{1}{2}\norm{U}_{\lpt{2}}^\frac{1}{2}+\norm{\nabb U}_{\lpt{2}}\Big)\Bigg)^2\\
\nn&\les & \norm{\nabb h}_{\lpt{2}}^2\norm{F}_{\lpt{2}}^2\\
\nn&&+2^{2j}\norm{\nabb h}^2_{\lpt{2}}\tau\Bigg(\int_0^\tau \Big(2^j\norm{\nabb U}_{\lpt{2}}\norm{U}_{\lpt{2}}+\norm{\nabb U}_{\lpt{2}}^2\Big)\Bigg).
\eea
Now, the heat flow estimate \eqref{eq:l2heat1} and the fact that $U(0)=F$ implies
\bee
&&\int_0^\tau \Big(2^j\norm{\nabb U}_{\lpt{2}}\norm{U}_{\lpt{2}}+\norm{\nabb U}_{\lpt{2}}^2\Big)\\
&\les& 2^j\sqrt{\tau}\sup_\tau\norm{U(\tau)}_{\lpt{2}}\left(\int_0^\tau\norm{\nabb U}_{\lpt{2}}^2\right)+\norm{F}^2_{\lpt{2}}\\
&\les& (1+2^j\sqrt{\tau})\norm{F}^2_{\lpt{2}}
\eee
which together with \eqref{zz64} yields the desired estimate \eqref{zz63}. This concludes the proof of the lemma. 

\subsection{Proof of Lemma \ref{lemma:zz42}}\lab{sec:gowinda18}

Let $V(\tau)$ defined in \eqref{zz62}. Assume that $V$ satisfies for all $\tau$ 
\be\label{zz65}
\norm{\nabb V(\tau)}_{\lpt{2}}\les 2^j((1+2^{\frac{j}{2}}\tau^{\frac{1}{4}})\norm{\nabb h}_{\lpt{2}}+\norm{K}_{\lpt{2}}\norm{h}_{\lpt{2}})\norm{F}_{\lpt{2}}. 
\ee
Then, \eqref{zz61} and \eqref{zz65} imply
\bee
&&\norm{\nabb [P_{>j}, P_{\leq j}(h)]F}_{\lpt{2}}\\
&\les& \int_0^\infty m_{>j}(\tau)\norm{\nabb V(\tau)}_{\lpt{2}} d\tau\\
&\les& 2^j\left(\int_0^\infty m_{>j}(\tau)((1+2^{\frac{j}{2}}\tau^{\frac{1}{4}})\norm{\nabb h}_{\lpt{2}}+\norm{K}_{\lpt{2}}\norm{h}_{\lpt{2}}) d\tau\right)\norm{F}_{\lpt{2}}\\
&\les& 2^j(\norm{\nabb h}_{\lpt{2}}+\norm{K}_{\lpt{2}}\norm{h}_{\lpt{2}})\norm{F}_{\lpt{2}}
\eee
which is the desired estimate \eqref{zz42}. Thus, it remains to prove \eqref{zz65}. 

The energy estimate \eqref{heatF1} implies
\bee
&& \norm{\nabb V(\tau)}^2_{\lpt{2}}+\int_0^\tau\norm{\lap V(\tau')}^2_{\lpt{2}}d\tau'\\
 &\lesssim&\ds\int_0^\tau\int_{\ptu} \lap V(\tau')\Big(\lap P_{\leq j}(h) U(\tau)F+\nabb P_{\leq j}(h)\c\nabb U(\tau)F\Big)\dmt d\tau.
 \eee
 This yields
 \bee
&& \norm{\nabb V(\tau)}^2_{\lpt{2}}\\ 
 &\lesssim&\int_0^\tau\Big(\norm{\lap P_{\leq j}(h)}^2_{\lpt{4}} \norm{U(\tau)F}^2_{\lpt{4}}+\norm{\nabb P_{\leq j}(h)}_{\lpt{\infty}}^2\norm{\nabb U(\tau)F}^2_{\lpt{2}}\Big)\\
&\lesssim&\int_0^\tau \Big(\norm{\nabb\lap P_{\leq j}(h)}_{\lpt{2}}\norm{\lap P_{\leq j}(h)}_{\lpt{2}}\norm{\nabb U}_{\lpt{2}}\norm{U}_{\lpt{2}}\\
&&+\norm{\nabb P_{\leq j}(h)}^2_{\lpt{\infty}}\norm{\nabb U(\tau)F}^2_{\lpt{2}}\Big)
\eee
where we used in the last inequality the Gagliardo-Nirenberg inequality \eqref{eq:GNirenberg}. Together with the Bochner inequality for scalars \eqref{eq:Bochconseqbis} and the finite band property for $P_j$, we obtain
\bee
 \norm{\nabb V(\tau)}^2_{\lpt{2}}&\lesssim&\int_0^\tau \Big(2^{3j}\norm{\nabb h}_{\lpt{2}}^2\norm{\nabb U}_{\lpt{2}}\norm{U}_{\lpt{2}}\\
&&+\norm{\nabb P_{\leq j}(h)}^2_{\lpt{\infty}}\norm{\nabb U(\tau)F}^2_{\lpt{2}}\Big).
\eee
Together with the heat flow estimate \eqref{eq:l2heat1} and the fact that $U(0)=F$, this yields
\bea\lab{zz66}
 \norm{\nabb V(\tau)}^2_{\lpt{2}}&\lesssim& 2^{3j}\norm{\nabb h}_{\lpt{2}}^2\sqrt{\tau}\sup_\tau\norm{U}_{\lpt{2}}\left(\int_0^\tau\norm{\nabb U}^2_{\lpt{2}}\right)^\frac{1}{2}\\ 
 \nn&&+\norm{\nabb P_{\leq j}(h)}^2_{\lpt{\infty}}\norm{F}^2_{\lpt{2}}\\
 \nn&\les& \Big(2^{3j}\sqrt{\tau}\norm{\nabb h}_{\lpt{2}}^2+\norm{\nabb P_{\leq j}(h)}^2_{\lpt{\infty}}\Big)\norm{F}^2_{\lpt{2}}.
\eea
Now, using \eqref{yo} with the choice $f=P_{\leq j}(h)$ yields
\bee
\norm{\nabb P_{\leq j}(h)}^2_{\lpt{\infty}}&\les& \norm{\lap P_{\leq j}(h)}_{\lpt{2}}+\norm{\nabb\lap P_{\leq j}(h)}_{\lpt{2}}^\frac{1}{2}\norm{\nabb P_{\leq j}(h)}_{\lpt{2}}^{\frac{1}{2}}\\
&&+\norm{K}_{\lpt{2}}\norm{\nabb P_{\leq j}(h)}_{\lpt{2}}\\
&\les& 2^j\Big(\norm{\nabb h}_{\lpt{2}}+\norm{K}_{\lpt{2}}\norm{h}_{\lpt{2}}\Big),
\eee
where we used in the last inequality the finite band property for $P_j$. Together with \eqref{zz66}, this yields the desired estimate \eqref{zz65}. This concludes the proof of the lemma.

\subsection{Proof of Lemma \ref{lemma:zz43}}\lab{sec:gowinda19}

We have:
\begin{equation}\label{zz67}
[\nabb,P_{\leq j}]h=\int_0^\infty m_{\leq j}(\tau)V(\tau) d\tau,
\end{equation}
where $V(\tau)$ satisfies:
\be\lab{zz68}
(\partial_{\tau}-\lap)V(\tau)=[\nabb,\lap]U(\tau)h,\,V(0)=0.
\ee
Assume that $V$ satisfies for all $\tau$ and for all $a>0$
\be\lab{zz69}
\norm{V(\tau)}_{\lpt{2}}\les \norm{K}_{\lpt{2}}(\norm{K}_{\lpt{2}}\norm{h}_{\lpt{2}}+\norm{\La^ah}_{\lpt{2}}).
\ee
Then, in view of \eqref{zz67}, we obtain for all $a>0$
\bee
\norm{[P_{\leq j}, \nabb]h}_{\lpt{2}}&\les& \int_0^{+\infty}m_{\leq j}(\tau)\norm{V(\tau)}_{\lpt{2}}d\tau\\
&\les& \left(\int_0^{+\infty}m_{\leq j}(\tau)d\tau\right)\norm{K}_{\lpt{2}}(\norm{K}_{\lpt{2}}\norm{h}_{\lpt{2}}+\norm{\La^ah}_{\lpt{2}})\\
&\les& \norm{K}_{\lpt{2}}(\norm{K}_{\lpt{2}}\norm{h}_{\lpt{2}}+\norm{\La^ah}_{\lpt{2}})
\eee
which is the desired estimate \eqref{zz43}. Thus, it remains to prove \eqref{zz69}. 

The energy estimate \eqref{heatF2}, together with the commutator formula \eqref{commutptu}, implies
\bee
&&\norm{V(\tau)}^2_{\lpt{2}}+\int_0^\tau\norm{\nabb V(\tau')}^2_{\lpt{2}}d\tau'\\
&\lesssim&\ds\int_0^\tau\norm{V(\tau')}_{\lpt{2}}\norm{K}_{\lpt{2}}\norm{\nabb U(\tau')h}_{\lpt{\infty}}d\tau'.
\eee
Integrating this differential inequality, we obtain
\bea\lab{zz70}
\norm{V(\tau)}^2_{\lpt{2}}+\int_0^\tau\norm{\nabb V(\tau')}^2_{\lpt{2}}&\les& \norm{K}_{\lpt{2}}^2\left(\int_0^\tau\norm{\nabb U(\tau')h}_{\lpt{\infty}}d\tau'\right)^2\\
\nn&\les& \norm{K}_{\lpt{2}}^2\int_0^\tau{\tau'}^{1-\delta}\norm{\nabb U(\tau')h}_{\lpt{\infty}}^2d\tau',
\eea
where $0<\delta<1$ will be chosen later. In view of the estimate \eqref{yo}, we have
\bea\lab{zz71}
&&\int_0^\tau{\tau'}^{1-\delta}\norm{\nabb U(\tau')h}_{\lpt{\infty}}^2d\tau'\\
\nn&\les& \int_0^\tau{\tau'}^{1-\delta}(\norm{\lap U(\tau')h}^2_{\lpt{2}}+\norm{\nabb\lap U(\tau')h}_{\lpt{2}}\norm{\nabb U(\tau')h}_{\lpt{2}}\\
\nn&&+\norm{K}_{\lpt{2}}^2\norm{\nabb U(\tau')h}_{\lpt{2}}^2)d\tau'\\
\nn&\les& \int_0^\tau{\tau'}^{1-\delta}\norm{\lap U(\tau')h}^2_{\lpt{2}}d\tau'+\int_0^\tau{\tau'}^{2-2\delta}\norm{\nabb\lap U(\tau')h}_{\lpt{2}}^2d\tau'\\
\nn&&+(1+\norm{K}_{\lpt{2}}^2)\int_0^\tau\norm{\nabb U(\tau')h}_{\lpt{2}}^2d\tau'\\
\nn&\les& \int_0^\tau{\tau'}^{1-\delta}\norm{\lap U(\tau')h}^2_{\lpt{2}}d\tau'+\int_0^\tau{\tau'}^{2-2\delta}\norm{\lap^{\frac{3}{2}} U(\tau')h}_{\lpt{2}}^2d\tau'\\
\nn&&+(1+\norm{K}_{\lpt{2}}^2)\norm{h}_{\lpt{2}}^2,
\eea
where we used in the last inequality the heat flow estimate \eqref{eq:l2heat1}. 

Next, we estimate the two first terms in the right-hand side of \eqref{zz71}. We have
\bee
&&\left(\int_0^\tau{\tau'}^{1-\delta}\norm{\lap U(\tau')h}^2_{\lpt{2}}d\tau'\right)^\frac{1}{2}\\
&\les& \sum_{j\geq 0}\left(\int_0^\tau{\tau'}^{1-\delta}\norm{\lap P_j U(\tau')h}^2_{\lpt{2}}d\tau'\right)^\frac{1}{2}\\
&\les& \sum_{j\geq 0}\left(\int_0^\tau\tau'\norm{\lap P_j U(\tau')h}^2_{\lpt{2}}d\tau'\right)^\frac{1-\delta}{2}\left(\int_0^\tau\norm{\lap P_j U(\tau')h}^2_{\lpt{2}}d\tau'\right)^\frac{\delta}{2}\\
&\les& \sum_{j\geq 0}\norm{P_jh}^{1-\delta}_{\lpt{2}}\norm{\nabb P_jh}^\delta_{\lpt{2}},
\eee
where we used in the last inequality the heat flow estimates \eqref{eq:l2heatnab} and \eqref{eq:l2heat2}. Together with the finite band property for $P_j$, we obtain
\bea\lab{zz72}
\left(\int_0^\tau{\tau'}^{1-\delta}\norm{\lap U(\tau')h}^2_{\lpt{2}}d\tau'\right)^\frac{1}{2}&\les& \sum_{j\geq 0}2^{\delta j}\norm{P_jh}_{\lpt{2}}\\
\nn&\les& \left(\sum_{j\geq 0}2^{-\delta j}\right)\norm{\La^{2\delta}h}_{\lpt{2}}\\
\nn&\les& \norm{\La^{2\delta}h}_{\lpt{2}}.
\eea
Also, we have
\bee
&&\left(\int_0^\tau{\tau'}^{2-2\delta}\norm{\lap^\frac{3}{2} U(\tau')h}^2_{\lpt{2}}d\tau'\right)^\frac{1}{2}\\
&\les& \sum_{j\geq 0}\left(\int_0^\tau{\tau'}^{2-2\delta}\norm{\lap^{\frac{3}{2}} P_j U(\tau')h}^2_{\lpt{2}}d\tau'\right)^\frac{1}{2}\\
&\les& \sum_{j\geq 0}\left(\int_0^\tau(\tau')^2\norm{\nabb\lap P_j U(\tau')h}^2_{\lpt{2}}d\tau'\right)^\frac{1-\delta}{2}\left(\int_0^\tau\norm{\nabb\lap P_j U(\tau')h}^2_{\lpt{2}}d\tau'\right)^\frac{\delta}{2}\\
&\les& \sum_{j\geq 0}\norm{P_jh}^{1-\delta}_{\lpt{2}}\norm{\lap P_jh}^\delta_{\lpt{2}},
\eee
where we used in the last inequality heat flow estimates. Together with the finite band property for $P_j$, we obtain
\bea\lab{zz73}
\left(\int_0^\tau{\tau'}^{2-2\delta}\norm{\lap^\frac{3}{2} U(\tau')h}^2_{\lpt{2}}d\tau'\right)^\frac{1}{2}&\les& \sum_{j\geq 0}2^{2\delta j}\norm{P_jh}_{\lpt{2}}\\
\nn&\les& \left(\sum_{j\geq 0}2^{-\delta j}\right)\norm{\La^{3\delta}h}_{\lpt{2}}\\
\nn&\les& \norm{\La^{3\delta}h}_{\lpt{2}}.
\eea
Finally, \eqref{zz71}, \eqref{zz72} and \eqref{zz73} imply for all $0<\delta<1$
\be\lab{zz73bis}
\int_0^\tau{\tau'}^{1-\delta}\norm{\nabb U(\tau')h}_{\lpt{\infty}}^2d\tau'\les \norm{\La^{3\delta}h}_{\lpt{2}}^2+\norm{K}_{\lpt{2}}^2\norm{h}_{\lpt{2}}^2.
\ee
Injecting \eqref{zz73bis} in \eqref{zz70}, we obtain
\be\lab{zz74}
\norm{V(\tau)}_{\lpt{2}}\les \norm{K}_{\lpt{2}}(\norm{K}_{\lpt{2}}\norm{h}_{\lpt{2}}+\norm{\La^{3\delta}h}_{\lpt{2}}).
\ee
Choosing $\delta=\frac{a}{3}$ in \eqref{zz74} yields the desired estimate \eqref{zz69}. This concludes the proof of the lemma. 

\subsection{Proof of Lemma \ref{lemma:zz44}}\lab{sec:gowinda20}

We have:
\begin{equation}\label{zz75}
[\nabb,P_j]h=\int_0^\infty m_j(\tau)V(\tau) d\tau,
\end{equation}
where $V(\tau)$ satisfies:
$$(\partial_{\tau}-\lap)V(\tau)=[\nabb,\lap]U(\tau)h,\,V(0)=0.$$
Assume that $V$ satisfies for all $a>0$
\be\lab{zz76}
\left(\int_0^{+\infty}\norm{\nabb V(\tau)}^2_{\lpt{2}}d\tau\right)^\frac{1}{2}\les \norm{K}_{\lpt{2}}(\norm{K}_{\lpt{2}}\norm{h}_{\lpt{2}}+\norm{\La^ah}_{\lpt{2}}).
\ee
Then, in view of \eqref{zz75}, we obtain for all $a>0$
\bee
\norm{\nabb [P_j, \nabb]h}_{\lpt{2}}&\les& \int_0^{+\infty}m_j(\tau)\norm{\nabb V(\tau)}_{\lpt{2}}d\tau\\
&\les& \left(\int_0^{+\infty}m_j(\tau)^2d\tau\right)^\frac{1}{2}\norm{K}_{\lpt{2}}(\norm{K}_{\lpt{2}}\norm{h}_{\lpt{2}}+\norm{\La^ah}_{\lpt{2}})\\
&\les& 2^j\norm{K}_{\lpt{2}}(\norm{K}_{\lpt{2}}\norm{h}_{\lpt{2}}+\norm{\La^ah}_{\lpt{2}})
\eee
which is the desired estimate \eqref{zz44}. Thus, it remains to prove \eqref{zz76}. 

Injecting \eqref{zz73bis} in \eqref{zz70}, we obtain
\be\lab{zz74bis}
\int_0^{+\infty}\norm{\nabb V(\tau)}^2_{\lpt{2}}d\tau\les \norm{K}^2_{\lpt{2}}(\norm{K}^2_{\lpt{2}}\norm{h}^2_{\lpt{2}}+\norm{\La^{3\delta}h}_{\lpt{2}}^2).
\ee
Choosing $\delta=\frac{a}{3}$ in \eqref{zz74bis} yields the desired estimate \eqref{zz76}. This concludes the proof of the lemma. 

\subsection{Proof of Lemma \ref{lemma:yo}}\lab{sec:gowinda21}

We have in view of \eqref{linftynormtensor}
\be\lab{yo1}
\norm{\nabb f}_{\lpt{\infty}}\les\norm{\nabb^3f}^{\frac{1}{2}}_{\lpt{2}}\norm{\nabb f}^{\frac{1}{2}}_{\lpt{2}}+\norm{\nabb^2 f}_{\lpt{2}}.
\ee
Now, using the Bochner inequality for tensors \eqref{vbochineq}, we have
\bee
\norm{\nabb^3f}_{\lpt{2}}&\lesssim& \norm{\lap\nabb f}_{\lpt{2}}+\norm{K}_{\lpt{2}}\norm{\nabb^2f}_{\lpt{2}}+\norm{K}^2_{\lpt{2}}\norm{\nabb f}_{\lpt{2}}\\
&\lesssim& \norm{\nabb\lap f}_{\lpt{2}}+\norm{[\nabb, \lap]f}_{\lpt{2}}+\norm{K}_{\lpt{2}}\norm{\nabb^2f}_{\lpt{2}}\\
&&\nn +\norm{K}^2_{\lpt{2}}\norm{\nabb f}_{\lpt{2}}.
\eee
In view of the commutator formula \eqref{commutptu}, we obtain
\bee
\norm{\nabb^3f}_{\lpt{2}}&\lesssim& \norm{\nabb\lap f}_{\lpt{2}}+\norm{K\nabb f}_{\lpt{2}}+\norm{K}_{\lpt{2}}\norm{\nabb^2f}_{\lpt{2}}\\
&&\nn +\norm{K}^2_{\lpt{2}}\norm{\nabb f}_{\lpt{2}}\\
&\lesssim& \norm{\nabb\lap f}_{\lpt{2}}+\norm{K}_{\lpt{2}}\norm{\nabb f}_{\lpt{\infty}}+\norm{K}_{\lpt{2}}\norm{\nabb^2f}_{\lpt{2}}\\
&&\nn +\norm{K}^2_{\lpt{2}}\norm{\nabb f}_{\lpt{2}}
\eee
which together with \eqref{yo1} yields
\bee
\norm{\nabb f}_{\lpt{\infty}}&\les&\norm{\nabb\lap f}^{\frac{1}{2}}_{\lpt{2}}\norm{\nabb f}^{\frac{1}{2}}_{\lpt{2}}+\norm{K}_{\lpt{2}}^\frac{1}{2}\norm{\nabb f}^{\frac{1}{2}}_{\lpt{\infty}}\norm{\nabb f}^{\frac{1}{2}}_{\lpt{2}}\\
&&+\norm{\nabb^2 f}_{\lpt{2}}+\norm{K}_{\lpt{2}}\norm{\nabb f}_{\lpt{2}}.
\eee
We deduce
\bee
\norm{\nabb f}_{\lpt{\infty}}&\les&\norm{\nabb\lap f}^{\frac{1}{2}}_{\lpt{2}}\norm{\nabb f}^{\frac{1}{2}}_{\lpt{2}}+\norm{\nabb^2 f}_{\lpt{2}}+\norm{K}_{\lpt{2}}\norm{\nabb f}_{\lpt{2}},
\eee
which together with the Bochner inequality for scalars \eqref{eq:Bochconseqbis} yields \eqref{yo}. This concludes the proof of the lemma. 


\end{document}